\documentclass{amsart}
\usepackage{amsmath,amssymb,amsthm,mathrsfs,fullpage,graphicx,xcolor,mathpazo,enumerate}
\usepackage{hyperref}
\hypersetup{
    colorlinks=true,
    linkcolor=blue,
    filecolor=magenta,      
    urlcolor=cyan,
}
\usepackage{subfigure}
\newcommand{\ee}{\mathrm{e}}
\newcommand{\dd}{\mathrm{d}}
\newcommand{\ii}{\mathrm{i}}
\newcommand{\KK}{\mathbb{K}}
\newcommand{\EE}{\mathbb{E}}
\newcommand{\gc}{\mathrm{gc}}
\newcommand{\bfA}{\mathbf{A}}

\newcommand{\bfH}{\mathbf{H}}

\newcommand{\bfK}{\mathbf{K}}

\newcommand{\bfM}{\mathbf{M}}

\newcommand{\bfO}{\mathbf{O}}

\newcommand{\bfT}{\mathbf{T}}

\newcommand{\bfY}{\mathbf{Y}}

\newcommand{\Id}{\mathbb{I}}
\newtheorem{lemma}{Lemma}[section]
\newtheorem{theorem}{Theorem}[section]
\newtheorem*{theorem*}{Theorem}
\definecolor{armygreen}{rgb}{0.29, 0.33, 0.13}
\newtheorem{rhp}{Riemann--Hilbert Problem}[section]
\newtheorem{proposition}{Proposition}[section]

\newenvironment{remark}{$\triangleleft$ \textbf{Remark:}}{$\triangleright$\newline}

\theoremstyle{definition}
\newtheorem{definition}{Definition}[section]
\newtheorem{assumption}{Assumption}[section]
\numberwithin{equation}{section}
\numberwithin{figure}{section}
\begin{document}
\title{Universality Near the Gradient Catastrophe Point in the Semiclassical Sine-Gordon Equation}
\author{Bing-Ying Lu}
\address{Department of Mathematics, University of Michigan, East Hall, 530 Church St., Ann Arbor, MI 48109\footnote{Current affiliation:  Institute of Mathematics, Academia Sinica, Taipei, Taiwan}}
\email{bylu@gate.math.sinica.tw}
\author{Peter D. Miller}
\address{Department of Mathematics, University of Michigan, East Hall, 530 Church St., Ann Arbor, MI 48109}
\email{millerpd@umich.edu}
\date{\today}
\maketitle

\begin{abstract}

We study the semiclassical limit of the sine-Gordon (sG) equation with below threshold pure impulse initial data of Klaus--Shaw type. The Whitham averaged approximation of this system exhibits a gradient catastrophe in finite time. In accordance with a conjecture of Dubrovin, Grava and Klein, we found that in a $\mathcal{O}(\epsilon^{4/5})$ neighborhood near the gradient catastrophe point, the asymptotics of the sG solution are universally described by the Painlev\'e I tritronqu\'ee solution. A linear map can be explicitly made from the tritronqu\'ee solution to this neighborhood. Under this map: away from the tritronqu\'ee poles, the first correction of sG is universally given by the real part of the Hamiltonian of the tritronqu\'ee solution; localized defects appear at locations mapped from the poles of tritronqu\'ee solution; the defects are proved universally to be a two parameter family of special localized solutions on a periodic background for the sG equation. We are able to characterize the solution in detail. Our approach is the rigorous steepest descent method for matrix Riemann--Hilbert problems, substantially generalizing \cite{BertolaTovbis2014} to establish universality beyond the context of solutions of a single equation.

\end{abstract}

\tableofcontents

\section{Introduction}

This paper concerns the behavior of solutions of the sine-Gordon equation in the form
\begin{equation}
	\epsilon^2 u_{tt}-\epsilon^2 u_{xx}+\sin u=0,\quad x\in\mathbb{R},\quad t\ge 0,
\label{eq:SemiClassicalsG}
\end{equation}
where $\epsilon\ll 1$ is a small positive parameter.  In this form, the sine-Gordon equation arises in the theory of crystal dislocations \cite{FrenkelKontorova1939}, superconducting Josephson junctions \cite{ScottChuReible1976}, vibrations of DNA molecules \cite{Yakushevich2004}, and quantum field theory \cite{Coleman1975}; see also the reviews \cite{BaroneEspositoMageeScott1971} and \cite{Cuevas-MaraverKevrekidisWilliams2014}.  More precisely, we are interested in a specific sequence of solutions $u=u_N(x,t)$, $N\in\mathbb{Z}_{>0}$, of \eqref{eq:SemiClassicalsG} for a specific sequence of parameters $\epsilon=\epsilon_N\to 0$ as  $N\to\infty$.  The sequence of solutions, called a \emph{fluxon condensate} in \cite{BuckinghamMiller2013}, is associated in a well-defined manner with initial data for \eqref{eq:SemiClassicalsG} of the form
\begin{equation}
	u(x,0)=0, \quad \epsilon u_t(x,0)=G(x),
\label{eq:InitialConditionFG}
\end{equation}
where $G(\cdot)$ is a given function independent of $\epsilon$.  Setting $u(x,0)=0$ makes the initial condition \eqref{eq:InitialConditionFG} of \emph{pure impulse} type; more generally one might set $u(x,0)=F(x)$ for a given function $F(\cdot)$ independent of $\epsilon$, but the class of fluxon condensates associated with pure impulse initial conditions is already extremely rich and makes available certain useful analytical shortcuts.  The association of the fluxon condensate $\{u_N(x,t)\}_{N=1}^\infty$ with the initial condition \eqref{eq:InitialConditionFG} is such that $u_N(x,t)$ satisfies \eqref{eq:InitialConditionFG} with $\epsilon=\epsilon_N$ up to terms of order $\mathcal{O}(\epsilon)$, see Proposition~\ref{prop:IC-approximation} below.  
Hence the fluxon condensate may be viewed as an approximate solution of the Cauchy problem \eqref{eq:SemiClassicalsG}--\eqref{eq:InitialConditionFG}.  There are functions $G(\cdot)$  for which the approximation is known to be exact; see \eqref{eq:sech-IC} below.

To set the stage for our study, suppose we firstly examine the solution of the Cauchy problem \eqref{eq:SemiClassicalsG}--\eqref{eq:InitialConditionFG} for small time, taking $t=\epsilon T$ for $T$ bounded.  The initial-value problem can then be rewritten in the form
\begin{equation}
	u_{TT}+\sin(u)=\epsilon^2u_{xx},\quad
	u(x,0)=0,\quad u_T(x,0)=G(x).
\label{eq:PerturbedPendulum}
\end{equation}
This scaling suggests neglecting $\epsilon^2u_{xx}$ as a small perturbation, in which case the sine-Gordon equation reduces to an independent ordinary differential equation for each value of $x\in\mathbb{R}$, namely that describing the simple pendulum.  It is well-known that the unperturbed problem ($\epsilon=0$ in \eqref{eq:PerturbedPendulum}) is of a different character depending on the value of the total energy $\mathcal{E}(x):=\tfrac{1}{2}u_T(x,T)^2-\cos(u(x,T))\ge -1$, which is independent of $T$ for $\epsilon=0$. If $\mathcal{E}(x)<1$ (the \emph{librational case}) then the pendulum swings back and forth, while if $\mathcal{E}(x)>1$ (the \emph{rotational case}) then the pendulum rotates around its pivot point.  The borderline case $\mathcal{E}(x)=1$ characterizes the separatrix in the phase portrait and the corresponding motions are homoclinic orbits representing the nonlinear saturation of the linearized instability of the unstable vertical equilibrium configuration of the pendulum.  It turns out that when $\epsilon>0$ is small, the solutions of the sine-Gordon equation exhibit a similar dichotomy at least for a certain range of $T$, and interestingly the behavior can be different for different values of $x$ because it is possible for smooth impulse profiles $G(x)$ to give rise to librational motion for some $x$ and rotational motion for other $x$.

In the papers \cite{BuckinghamMiller2012,BuckinghamMiller2013}, the fluxon condensate was studied in detail in the case that $G(\cdot)$ is a real-valued function decaying rapidly to zero for large $|x|$ and satisfying the inequality $\|G\|_\infty>2$ (among other properties, some essential and some technical).  In this situation, the impulse profile $G$ is such that the pendulum energy $\mathcal{E}(x)$ is above the threshold for rotation for some values of $x$ and below the threshold for others.  Hence the impulse profile initiates two distinct types of wave motion for small time:  near values of $x$ where $|G(x)|>2$ one observes the generation of modulated \emph{rotational waves} (in which case the value of $u$ increases without bound, like the pivot angle of a rotating pendulum) while near values of $x$ where $|G(x)|<2$ one observes instead the generation of modulated \emph{librational waves} (in which case $u$ is locally periodic).  In both cases, the waves are \emph{superluminal}, having a local phase velocity exceeding the (unit) characteristic speed for \eqref{eq:SemiClassicalsG} in absolute value.  It was further shown in \cite{BuckinghamMiller2012} that near transitional values of $x$ where the initial data describes a transversal crossing of the pendulum separatrix with varying $x$, a certain universal (i.e., largely independent of initial conditions) wave pattern emerges in the semiclassical limit $\epsilon\to 0$.

This paper concerns the less-energetic case, in which $\|G\|_\infty<2$ or equivalently $\mathcal{E}(x)<1$ holds for all $x\in\mathbb{R}$ at the initial time.  Thus the system is initially globally below threshold for rotation in the sense of the unperturbed simple pendulum problem, and one expects modulated superluminal librational waves for all $x\in\mathbb{R}$ when $t$ is sufficiently small.   At a formal level one may appeal to Whitham modulation theory \cite{Whitham1965a,Whitham1965b} to describe the modulation, and it is well-known that the Whitham modulation equations relevant for librational, superluminal periodic solutions of \eqref{eq:SemiClassicalsG} constitute a $2\times 2$ quasilinear system of elliptic type.  Therefore one needs to assume analyticity of the initial data $G(\cdot)$ to ensure local existence of a solution, and then the solution is not generally global.  The generic breakdown mechanism is the finite-time formation of a singularity, a \emph{gradient catastrophe} of elliptic umbilic type \cite{DubrovinGravaKlein2009} occurring at a point $(x_\mathrm{gc},t_\mathrm{gc})$.
After the gradient catastrophe occurs, the formal Whitham modulation theory is no longer valid and one expects the solution of the Cauchy problem \eqref{eq:SemiClassicalsG}--\eqref{eq:InitialConditionFG} to exhibit more complicated behavior, locally near $x=x_\mathrm{gc}$, at least.  

The main thrust of this paper is a study of the asymptotic properties of fluxon condensates for the Cauchy problem \eqref{eq:SemiClassicalsG}--\eqref{eq:InitialConditionFG} with analytic impulse profile $G(\cdot)$ globally below threshold for rotation, when the space-time coordinates $(x,t)$ are suitably scaled with $\epsilon$ to be close to a point of generic gradient catastrophe for the Whitham modulation equations.  We show that the fluxon condensate locally generates a different kind of universal wave pattern than had been seen in \cite{BuckinghamMiller2012}, although it has more in common with a conjecture of Dubrovin et al. \cite{DubrovinGravaKlein2009} that was proved in the setting of a different equation in \cite{BertolaTovbis2014} and extended at the formal level to other systems in \cite{DubrovinGravaKleinMoro2015}.  Our results therefore give a second example of the type of universality first predicted in \cite{DubrovinGravaKlein2009}, and hence they provide the first rigorous evidence that universality extends beyond independence on initial data to independence on the equation of motion, as was formally argued in \cite{DubrovinGravaKleinMoro2015}.

\subsection{Fluxon condensates below the threshold for rotation}
We begin by describing the fluxon condensate associated to the Cauchy problem \eqref{eq:SemiClassicalsG}--\eqref{eq:InitialConditionFG} in more detail.  The sine-Gordon equation in the form \eqref{eq:SemiClassicalsG} is well-known to be a completely integrable partial differential equation in the sense that it is the compatibility condition for the Lax pair 
\begin{equation}
	\begin{split}
		& 4\ii \epsilon \mathbf{v}_x=
			\begin{bmatrix}
			\displaystyle 4E(w)-\frac{\ii}{\sqrt{-w}}(1-\cos(u)) & \displaystyle \frac{\ii}{\sqrt{-w}}\sin(u)-\ii\epsilon(u_x+u_t)\\
			\displaystyle \frac{\ii}{\sqrt{-w}}\sin(u)+\ii\epsilon(u_x+u_t) & \displaystyle -4E(w)+\frac{\ii}{\sqrt{-w}}(1-\cos(u))\end{bmatrix} \mathbf{v}, \\
		& 4\ii \epsilon \mathbf{v}_t=
		 \begin{bmatrix}
			\displaystyle 4D(w)+\frac{\ii}{\sqrt{-w}}(1-\cos(u)) & \displaystyle -\frac{\ii}{\sqrt{-w}}\sin(u)-\ii\epsilon(u_x+u_t)\\
			\displaystyle -\frac{\ii}{\sqrt{-w}}\sin(u)+\ii\epsilon(u_x+u_t) & \displaystyle -4D(w)-\frac{\ii}{\sqrt{-w}}(1-\cos(u))
		\end{bmatrix} \mathbf{v},
	\end{split}
	\label{eq:wLaxPair}
\end{equation}
where
\begin{equation}
	E(w):=\frac{\ii}{4}\left[\sqrt{-w}+\frac{1}{\sqrt{-w}}\right]\quad\text{and}\quad
	D(w):=\frac{\ii}{4}\left[\sqrt{-w}-\frac{1}{\sqrt{-w}}\right],\quad |\arg(-w)|<\pi.
\label{eq:E-and-D-define}
\end{equation}
Here $w\in\mathbb{C}\setminus\mathbb{R}_+$ is a spectral parameter.
This representation is due to Faddeev, Takhtajan, and Zakharov \cite{FaddeevTakhtajanZakharov1974}.  For further information about the use of this Lax pair in studying the Cauchy problem \eqref{eq:SemiClassicalsG}--\eqref{eq:InitialConditionFG}, see \cite{Kaup1975} and \cite[Appendix A]{BuckinghamMiller2008}.  An equivalent form of the sine-Gordon equation written in characteristic or ``light-cone'' coordinates and important in differential geometry \cite{Bour1862} was earlier given a Lax pair formulation by Ablowitz, Kaup, Newell, and Segur \cite{AblowitzKaupNewellSegur1973}, but this is not as relevant for the Cauchy problem \eqref{eq:SemiClassicalsG}--\eqref{eq:InitialConditionFG} which is formulated instead in ``laboratory'' coordinates.  Note that if $u(x,t)$ and $u_t(x,t)$ are substituted from the pure-impulse initial condition \eqref{eq:InitialConditionFG} into the first equation of the Lax pair \eqref{eq:wLaxPair} for $t=0$, one sees that
\begin{equation}
\epsilon \mathbf{v}_x=\begin{bmatrix}
-\ii \lambda & \psi(x) \\ -\psi(x)^* & \ii \lambda
\end{bmatrix} \mathbf{v}, \qquad \psi(x):=-\frac{1}{4}G(x), \quad \lambda:=E(w),\quad t=0,
\label{eq:Zakharov-Shabat}
\end{equation}
and therefore for initial conditions of the form \eqref{eq:InitialConditionFG} the direct scattering problem is reduced to one of Zakharov-Shabat type \cite{ZakharovShabat1971} for a real-valued potential $\psi:\mathbb{R}\to\mathbb{R}$.  The first step in the solution of the Cauchy problem \eqref{eq:SemiClassicalsG}--\eqref{eq:InitialConditionFG} is to compute relevant scattering data for the simplified direct scattering problem \eqref{eq:Zakharov-Shabat}.

We take on the following basic assumption on $G$.
\begin{assumption}
\label{assumption:G}
The impulse profile $G:\mathbb{R}\to\mathbb{R}$ is Schwartz class, even (i.e., $G(-x)=G(x)$), and monotone increasing for $x>0$.  
\end{assumption}
Each such function $G$ takes a negative minimum value precisely at $x=0$.  Given an impulse profile satisfying Assumption~\ref{assumption:G}, we define a related function on an interval of the positive imaginary axis as follows:
\begin{equation}
\Psi(\lambda):=\frac{1}{4} \int_{x_-(\lambda)}^{x_+(\lambda)} \sqrt{G(s)^2+ 16 \lambda^2}\,\dd s, \quad 0<-\ii\lambda<\max_{x\in\mathbb{R}}\left(-\frac{1}{4}G(x)\right),
\label{eq:Psi_Initial_Condition}
\end{equation}
where $x_-(\lambda)<x_+(\lambda)$ are the two opposite real roots of $G(s)^2+16 \lambda^2$. The function $\Psi(\lambda)$ appears naturally (playing the role of a \emph{phase integral}) in the analysis of  \eqref{eq:Zakharov-Shabat} for small $\epsilon$ via the WKB method.  Indeed, if $\epsilon$ is taken from the discrete sequence
\begin{equation}
\epsilon=\epsilon_N:=-\frac{1}{4\pi N}\int_\mathbb{R}G(x)\,\dd x>0, \quad N\in\mathbb{Z}_{>0},
\label{eq:epsilon-N}
\end{equation}
then the WKB method predicts that the reflection coefficient defined for $\lambda\in\mathbb{R}$ from \eqref{eq:Zakharov-Shabat} is negligible, and that the eigenvalues\footnote{For the type of initial data under consideration it is known that the eigenvalues are purely imaginary numbers in complex-conjugate pairs lying in the indicated imaginary interval of definition of $\Psi(\lambda)$ and its Schwarz reflection, see \cite{KlausShaw2002}.  Moreover, all complex eigenvalues are simple, and for the values of $\epsilon=\epsilon_N$ defined in \eqref{eq:epsilon-N} $\lambda=0$ is not an eigenvalue (spectral singularity) and there are exactly $2N$ eigenvalues in purely imaginary conjugate pairs.} in the positive imaginary interval are well-approximated by points $\{\lambda_k\}_{k=0}^{N-1}$ in the same interval defined by the Bohr-Sommerfeld quantization rule:
\begin{equation} \label{eq:BohrSommerfeld}
\Psi(\lambda_k)=\pi \epsilon \left(k+\frac{1}{2} \right), \quad k=0,1,2,\dots, N-1,\quad\epsilon=\epsilon_N.
\end{equation}
We refer to the purely positive imaginary points $\{\lambda_k\}_{k=0}^{N-1}$ defined from a suitable impulse profile $G(\cdot)$ and a value of $N\in\mathbb{Z}>0$ as \emph{approximate eigenvalues} for \eqref{eq:Zakharov-Shabat} in the upper half-plane.  Whenever $\lambda$ is a positive imaginary eigenvalue 
of \eqref{eq:Zakharov-Shabat}, the unique solutions $\mathbf{v}=\mathbf{v}_\pm(x;\lambda)$ determined by exponentially decaying asymptotics as $x\to\pm\infty$ via 
\begin{equation}
\begin{split}
\mathbf{v}_-(x;\lambda)&=\ee^{-\ii\lambda x/\epsilon}\begin{bmatrix}1+o(1) \\ o(1)\end{bmatrix},\quad x\to -\infty\\
\mathbf{v}_+(x;\lambda)&=\ee^{\ii\lambda x/\epsilon}\begin{bmatrix}o(1)\\1+o(1)\end{bmatrix},\quad x\to +\infty
\end{split}
\end{equation}
are necessarily proportional, i.e., there is a constant $\gamma\neq 0$ associated with the eigenvalue $\lambda$ such that $\mathbf{v}_-(x;\lambda)=\gamma\mathbf{v}_+(x;\lambda)$.  For real and even $G(\cdot)$ it is easy to see that $\gamma=\pm 1$, and WKB theory predicts a correlation with the index $k$ in \eqref{eq:BohrSommerfeld} in the form $\gamma=\gamma_k:=(-1)^{k+1}$ for the approximate eigenvalue $\lambda=\lambda_k$, $k=0,\dots,N-1$.  This is sufficient information to specify the fluxon condensate for $G$ as follows.
\begin{definition}[Fluxon condensate associated with $G$]
Let $G:\mathbb{R}\to\mathbb{R}$ be a function satisfying Assumption~\ref{assumption:G}.  The \emph{fluxon condensate} for $G$ is the sequence of functions $\{u_N(x,t)\}_{N=1}^\infty$ such that $u(x,t)=u_N(x,t)$ is the exact solution of the sine-Gordon equation in the form \eqref{eq:SemiClassicalsG} with $\epsilon=\epsilon_N$ given by \eqref{eq:epsilon-N} that is a reflectionless potential (i.e. pure multi-soliton solution) constructed from the discrete eigenvalues $w_k$ and $w_k^*$ that are preimages under $\lambda=E(w)$ of the approximate eigenvalues $\lambda_k$ defined by \eqref{eq:BohrSommerfeld} and the connection coefficients $\gamma_k:=(-1)^{k+1}$, $k=0,\dots,N-1$.
\label{def:fluxon-condensate}
\end{definition}

Definition~\ref{def:fluxon-condensate} will be clarified further in Section~\ref{sec:RHP} below, where we properly define $u_N(x,t)$ in terms of the solution of a (discrete) Riemann--Hilbert problem.  If $G(0)=\min_{x\in\mathbb{R}}G(x)<-2$, then some of the preimages $w_k$ lie on the unit circle in conjugate pairs and others lie on the negative real axis in pairs symmetric with respect to reflection through the unit circle.  This means that the fluxon condensate contains both breathers (corresponding to the conjugate pairs) and counterpropagating kinks/antikinks (corresponding to the real pairs symmetric in reflection through the circle).  This is the case that is primarily studied in \cite{BuckinghamMiller2012,BuckinghamMiller2013}.  However, in this work we assume instead throughout that $-2<G(0)=\min_{x\in\mathbb{R}}G(x)<0$; in addition to ensuring that the unperturbed simple pendulum problem in \eqref{eq:PerturbedPendulum} is globally (in $x$, at $t=0$) below the threshold for rotation this means that the fluxon condensate is a nonlinear superposition of breathers only.  Indeed, the preimages under $\lambda=E(w)$ of the approximate eigenvalues with their complex conjugates fill out as $N$ increases an arc of the unit circle passing through $w=1$ and having endpoints $w=\ee^{\pm\ii\mu}$ for some $\mu\in (0,\pi)$.

To study the fluxon condensate $\{u_N(x,t)\}_{N=1}^\infty$ from the point of view of asymptotic analysis as $N\to\infty$ equivalent to the semiclassical limit $\epsilon\to 0$, we will require some specific properties of the phase integral function $\Psi$ defined in \eqref{eq:Psi_Initial_Condition}.
\begin{assumption}
The function $\Psi$ is a strictly monotone decreasing real-valued function of $v=-\ii\lambda\in (0,-\tfrac{1}{4}G(0))$ that admits analytic continuation to an open neighborhood of its initial domain $\lambda\in\mathbb{C}: 0<-\ii\lambda<-\tfrac{1}{4}G(0)$ which satisfies $\Psi(-\tfrac{1}{4}\ii G(0))=0$,  $-\ii\Psi'(-\tfrac{1}{4}\ii G(0))>0$, $\Psi(0)=-\tfrac{1}{4}\int_\mathbb{R}G(x)\,\dd x>0$, and $-\ii\Psi'(0)>0$, and such that $\Psi(\lambda)-\Psi(0)-\lambda\Psi'(0)$ is the germ of an even analytic function at $\lambda=0$.
\label{assumption:Psi}
\end{assumption}
In \cite[Section 1.1]{BuckinghamMiller2013} one can find conditions on $G$, beyond Assumption~\ref{assumption:G} and in particular including real-analyticity for all $x\in\mathbb{R}$, sufficient to guarantee that Assumption~\ref{assumption:Psi} holds on some small neighborhood of the initial domain for $\Psi$.  This in turn is enough to prove the following.
\begin{proposition}
\label{prop:IC-approximation}
Let $G$ be an impulse profile $G$ satisfying Assumption~\ref{assumption:G} and $-2<G(0)<0$, for which the phase integral $\Psi$ given by \eqref{eq:Psi_Initial_Condition} satisfies Assumption~\ref{assumption:Psi}. If $\{u_N(x,t)\}_{N=1}^\infty$ denotes the fluxon condensate for $G$, then
\begin{equation}
u_N(x,0)=\mathcal{O}(\epsilon)\quad\text{and}\quad\epsilon \frac{\partial u_N}{\partial t}(x,0)=G(x)+\mathcal{O}(\epsilon),\quad\epsilon=\epsilon_N,\quad N\to\infty
\label{eq:t-zero-formulae}
\end{equation}
where the error terms are uniform for bounded $x$.
\end{proposition}
This is the ``below threshold'' analogue of \cite[Corollary 1.1]{BuckinghamMiller2013}, but its proof is simpler, lacking complications arising from the presence of kink/antikink components of the fluxon condensate which lead to nonuniformity of the limit near the transitional $x$-values studied in \cite{BuckinghamMiller2012}.  Proposition~\ref{prop:IC-approximation} is a kind of justification for replacing the solution of the Cauchy problem \eqref{eq:SemiClassicalsG}--\eqref{eq:InitialConditionFG} with the fluxon condensate $\{u_N(x,t)\}_{N=1}^\infty$, which is easier to analyze in the semiclassical limit because it is a reflectionless (pure soliton) potential for all $N$.  The step of replacing the solution of a Cauchy problem with a reflectionless approximation goes back to the analysis of Lax and Levermore \cite{LaxLevermore1979} of the small dispersion Korteweg-de Vries equation, and has since been used on many different integrable equations \cite{BuckinghamJenkinsMiller2017,BuckinghamMiller2012,BuckinghamMiller2013,ErcolaniJinLevermoreMacEvoy2003,JinLevermoreMcLaughlin1999,OrangeBook,MillerXu2011}.  In situations for which the dispersionless or Whitham modulational system is of hyperbolic type it becomes possible to propagate the accuracy at $t=0$ afforded by results such as Proposition~\ref{prop:IC-approximation} to positive time $t$.

The same assumptions needed to prove Proposition~\ref{prop:IC-approximation} allow the fluxon condensate to be studied for nonzero $t$ sufficiently small (but independent of $\epsilon$), with the result being that the elliptic Whitham modulation equations correctly describe the modulated librational waves that are generated over the whole $x$-axis.  The more precise statement is essentially \cite[Theorem 1.1]{BuckinghamMiller2013} in which the domain $S_\mathsf{L}$ is enlarged to a strip around the $x$-axis in the $(x,t)$-plane.  However, the fact that the Whitham equations have a smooth solution on this domain $S_\mathsf{L}$ means that this result on its own does not begin to capture the phenomenon of gradient catastrophe or the ensuing dynamics.  In order to extend the result to values of $t$ sufficiently large to approach a gradient catastrophe point, it is necessary to make stronger assumptions.  In particular, we will need to require the domain of analyticity in Assumption~\ref{assumption:Psi} to be sufficiently large because as $t$ increases certain contours in the complex $w$-plane (to be explained fully in Section~\ref{sec:steepest-descent}) begin to move away from their initial positions, and the domain of analyticity of $\Psi$, when pulled back to the $w$-plane under $\lambda=E(w)$, should be large enough to accommodate this motion.  In practice we just assume in this paper that any singularities of $\Psi$ are sufficiently distant as to provide no obstruction to our analysis.  Note that there exist impulse profiles $G$ satisfying Assumption~\ref{assumption:G} for which $\Psi$ admits analytic continuation as an entire function; in particular for any $A>0$,
\begin{equation}
G(x)=-4A\,\mathrm{sech}(x)\quad\implies\quad\Psi(\lambda)=\ii\pi\lambda+\pi A.
\label{eq:sech-IC}
\end{equation}
We also note that, as a consequence of a calculation of Satsuma and Yajima \cite{SatsumaYajima1974},  the error terms in \eqref{eq:t-zero-formulae} vanish identically for the fluxon condensate associated with the impulse profile in \eqref{eq:sech-IC}, i.e., the fluxon condensate provides the exact solution of the Cauchy problem \eqref{eq:SemiClassicalsG}--\eqref{eq:InitialConditionFG} when $\epsilon=\epsilon_N$.
In addition to sufficient analyticity of $\Psi$ given the point $(x,t)$ of interest, we require certain other technical properties that are difficult to explain at this juncture, but for which we provide a concrete definition in Section~\ref{sec:modulated-librational-wave-region} below.
\begin{proposition}
\label{prop:BeforeCatastrophe}
Suppose that $G$ is an impulse profile satisfying Assumption~\ref{assumption:G} and $-2<G(0)<0$, and that a point $(x_0,t_0)\in\mathbb{R}^2$ is given.  Suppose also that the associated phase integral $\Psi$ satisfies Assumption~\ref{assumption:Psi} on a large enough domain of the $\lambda$-plane given $(x_0,t_0)$ and 
that $(x_0,t_0)$ belongs to the modulated librational wave region (see Definition~\ref{def:MLW} below).  Then there exists a neighborhood $N\subset\mathbb{R}^2$ of $(x_0,t_0)$ and well-defined differentiable functions $n_\mathrm{p}:N\to (-1,1)$, $\mathcal{E}:N\to (-1,1)$, and $\Phi:N\to\mathbb{R}$ such that
\begin{itemize}
\item $n_\mathrm{p}(x,t)$ and $\mathcal{E}(x,t)$ satisfy the elliptic Whitham modulation equations in the form
\begin{equation}
\frac{\partial}{\partial t}\begin{bmatrix}n_\mathrm{p}\\\mathcal{E}\end{bmatrix}+
\frac{1}{\mathcal{N}}\begin{bmatrix}n_\mathrm{p}[JJ''+(J')^2] & -(1-n_\mathrm{p}^2)^2J'J''\\
JJ' &  n_\mathrm{p}[JJ''+(J')^2]\end{bmatrix}\frac{\partial}{\partial x}\begin{bmatrix}n_\mathrm{p}\\\mathcal{E}\end{bmatrix}=\mathbf{0}
\label{eq:Whitham-intro}
\end{equation}
with $\mathcal{N}:=n_\mathrm{p}^2JJ''+(J')^2$ and
\begin{equation}
J=J(\mathcal{E}):=\frac{8}{\pi}\left(\EE(m)+(m-1)\KK(m)\right),\quad m:=\frac{1}{2}(1+\mathcal{E}),
\label{eq:J-func-define}
\end{equation}
\item
the partial derivatives of $\Phi(x,t)$ are related to $n_\mathrm{p}(x,t)$ and $\mathcal{E}(x,t)$ by
\begin{equation}
\frac{\partial\Phi}{\partial x}=k(x,t)\quad\text{and}\quad\frac{\partial\Phi}{\partial t}=-\omega(x,t)
\label{eq:Phi-partials}
\end{equation}
where with $m(x,t):=\tfrac{1}{2}(1+\mathcal{E}(x,t))\in (0,1)$,
\begin{equation}
\omega(x,t):=-\frac{\pi}{2\KK(m(x,t))}\frac{1}{\sqrt{1-n_\mathrm{p}(x,t)^2}}\quad\text{and}\quad k(x,t):=\omega(x,t)n_\mathrm{p}(x,t),
\label{eq:omega-k}
\end{equation}
\item
the fluxon condensate obeys the following asymptotic formul\ae\ in the limit $N\to\infty$ with $\epsilon=\epsilon_N$ given by \eqref{eq:epsilon-N}:
\begin{equation}
\begin{split}
\cos(\tfrac{1}{2}u_N(x,t))&=\mathrm{dn}\left(\frac{2\Phi(x,t)\KK(m(x,t))}{\pi\epsilon};m(x,t)\right)+\mathcal{O}(\epsilon)\\
\sin(\tfrac{1}{2}u_N(x,t))&=-\sqrt{m(x,t)}\,\mathrm{sn}\left(\frac{2\Phi(x,t)\KK(m(x,t))}{\pi\epsilon};m(x,t)\right)+\mathcal{O}(\epsilon),
\end{split}
\label{eq:cos-sin-before}
\end{equation}
in which the error terms are uniform over $(x,t)\in N$.
\end{itemize}
\end{proposition}

Here $\mathbb{K}(m)$ and $\EE(m)$ denote the complete elliptic integrals of the first and second kinds respectively:
\begin{equation}
	\mathbb{K}(m):=\int_0^1 \frac{\dd s}{\sqrt{(1-s^2)(1-ms^2)}},\quad \EE(m):=\int_0^1 \sqrt{\frac{1-ms^2}{1-s^2}} \dd s,\quad
0<m<1,
\label{eq:KE}
\end{equation}
\cite[Ch.\@ 19]{NIST:DLMF}, while $\mathrm{sn}(z;m)$ and $\mathrm{dn}(z;m)$ (and in Theorem~\ref{thm:AwayFromPoles} below, $\mathrm{cn}(z;m)$) denote Jacobi elliptic functions, which are described in detail also in \cite[Ch.\@ 22]{NIST:DLMF} in terms of a different notation for the parameter, using $k\in (0,1)$ instead of $m=k^2\in (0,1)$.  In particular, the leading terms in \eqref{eq:cos-sin-before} are $2\pi\epsilon$-periodic in $\Phi(x,t)$ (and for $\cos(\tfrac{1}{2}u_N(x,t))$ the fundamental period is half as big), so that $\Phi(x,t)$ has the interpretation of a phase variable.  
Note also that in \eqref{eq:Phi-partials}, $k(x,t)$ has the interpretation of a local wavenumber, $\omega(x,t)$ has the interpretation of a local frequency, and then from \eqref{eq:omega-k}, $n_\mathrm{p}(x,t)$ has the interpretation of a local reciprocal phase velocity $k(x,t)/\omega(x,t)$.  These observations imply that  the leading terms in the formul\ae\ \eqref{eq:cos-sin-before} describe
a slowly-modulated and rapidly oscillating superluminal librational wave of local energy $\mathcal{E}(x,t)$; indeed this is the hypothesis on which the formal Whitham modulation theory leading to the elliptic system \eqref{eq:Whitham-intro} is based.

If $t_0>0$ is sufficiently small, then $(x_0,t_0)$ is automatically in the modulated librational wave region, and the proof of Proposition~\ref{prop:BeforeCatastrophe} follows along the lines of that of \cite[Theorem 1.1]{BuckinghamMiller2013} but again with some simplifications due to the given lower bound on $G(0)$.  In this situation we also have the initial conditions 
\begin{equation}
n_\mathrm{p}(x,0)=0 \quad\text{and}\quad \mathcal{E}(x,0)=\tfrac{1}{2}G(x)^2-1
\label{eq:Whitham-ICs}
\end{equation}
and also $\Phi(x,0)=0$.  For larger $t_0$ the proof relies essentially on all of the hypotheses, and it also uses a modification of the type of analysis described in \cite{BuckinghamMiller2013}; it will be given in Section~\ref{sec:BeforeCatastropheProof} below.

\subsection{Main results and discussion}

The boundary of the modulated librational wave region may contain one or more points of simple gradient catastrophe, a notion that will be properly defined in Section~\ref{sec:steepest-descent} below.  Given the even symmetry of impulse profiles $G(\cdot)$ satisfying Assumption~\ref{assumption:G}, the Whitham equations \eqref{eq:Whitham-intro} and initial conditions \eqref{eq:Whitham-ICs} imply that $n_\mathrm{p}(x,t)$ and $\mathcal{E}(x,t)$ are respectively odd and even functions of $x$ for each $t$ for which they are defined.  Therefore, the gradient catastrophe points are symmetric about $x=0$. Working from known examples and numerical calculations we assume for simplicity that there is a simple gradient catastrophe point on the $t$-axis, at a point $(x,t)=(0,t_\mathrm{gc})$, $t_\mathrm{gc}>0$, and that the domain $(x,t)\in (-\delta,\delta)\times [0,t_\mathrm{gc})$ is contained in the modulated librational wave region for sufficiently small $\delta>0$.  Our main results concern the asymptotic behavior of the fluxon condensate in different scaling limits in which $(x,t)\to (0,t_\mathrm{gc})$ at a suitable rate while $N\to\infty$.  As the catastrophe in question is one in which derivatives of $n_\mathrm{p}(x,t)$ and $\mathcal{E}(x,t)$ blow up as $(x,t)\to (0,t_\mathrm{gc})$ while their values have definite limits, the following values
are well-defined:
\begin{equation}
m_\mathrm{gc}:=m(0,t_\mathrm{gc}),\quad \omega_\mathrm{gc}:=\omega(0,t_\mathrm{gc}),\quad \Phi_\mathrm{gc}:=\Phi(0,t_\mathrm{gc}).
\label{eq:Whitham-values-at-gc}
\end{equation}
Note also that, according to \eqref{eq:Phi-partials}, $\Phi_t(0,t_\mathrm{gc})=-\omega(0,t_\mathrm{gc})=-\omega_\mathrm{gc}$ and that
$\Phi_x(0,t_\mathrm{gc})=k(0,t_\mathrm{gc})=0$ because $n_\mathrm{p}(0,t)=0$ for $0<t<t_\mathrm{gc}$ by odd symmetry and hence also $k(0,t)=0$.

Near the catastrophe point $(0,t_\mathrm{gc})$, the $\mathcal{O}(\epsilon)$ error estimates in \eqref{eq:cos-sin-before} are not uniform. Our first main result shows that they are larger, of size $\epsilon^{1/5}$, and that the leading term of the error can be expressed in terms of the \emph{real tritronqu\'ee solution} $y(\tau)$ of the first Painlev\'e equation
\begin{equation}
\frac{\dd^2y}{\dd\tau^2}=6y^2+\tau.
\label{eq:PI-intro}
\end{equation}
By definition \cite{Kapaev2004}, $y(\tau)$ is the unique solution of \eqref{eq:PI-intro} with the property that 
\begin{equation}
y(\tau)=-\left(-\frac{1}{6}\tau\right)^{1/2}(1+o(1)),\quad\tau\to\infty,\quad |\arg(-\tau)|\le \frac{4}{5}\pi-\delta
\label{eq:y-asymp}
\end{equation}
holds for every $\delta>0$.  The paper \cite{DubrovinGravaKlein2009} raised the conjecture that $y(\tau)$ is analytic in the sector $|\arg(-\tau)|<\frac{4}{5}\pi$ without the asymptotic condition $\tau\to\infty$; this conjecture was subsequently proven by Costin, Huang, and Tanveer \cite{CostinHuangTanveer2014}.  The Painlev\'e-I \emph{Hamiltonian} associated with $y(\tau)$ is the related function
\begin{equation}
h(\tau):=-\frac{1}{2}y'(\tau)^2+2y(\tau)^3+\tau y(\tau).
\label{eq:Hamiltonian-intro}
\end{equation}
Obviously, $h(\tau)$ is analytic in the same sector of the complex plane where $y(\tau)$ is, and whereas one can show that all poles of $y(\tau)$ in the complementary sector are double, all corresponding poles of $h(\tau)$ are simple, with residue $-1$.
Also, both $y(\tau)$ and its Hamiltonian $h(\tau)$ are Schwarz-symmetric functions:  $y(\tau^*)=y(\tau)^*$ and $h(\tau^*)=h(\tau)^*$.

\begin{theorem}[First correction near the gradient catastrophe point]
Suppose that $G$ is an impulse profile satisfying Assumption~\ref{assumption:G} and $-2<G(0)<0$ for which the phase integral $\Psi$ satisfies Assumption~\ref{assumption:Psi} on a sufficiently large domain of the $\lambda$-plane.  Let $(0,t_\mathrm{gc})$ be a point of simple gradient catastrophe, on the boundary of the modulated librational wave domain $(x,t)\in(-\delta,\delta)\times [0,t_\mathrm{gc})$ at which the limits \eqref{eq:Whitham-values-at-gc} are defined.
Then there is a positive number $\sigma>0$ depending on $G(\cdot)$ and related constants defined by
\begin{equation}
M:=\frac{2}{\sigma}\left(\frac{m_\mathrm{gc}}{1-m_\mathrm{gc}}\right)^{1/4}>0\quad\text{and}\quad
a:=-\frac{(m_\mathrm{gc}(1-m_\mathrm{gc}))^{1/4}}{2\sigma}<0
\end{equation}
as well as $b:=-a\rho(m_\mathrm{gc})>0$ in which $\rho(m_\mathrm{gc})>0$ is defined by
\begin{equation}
\rho(m)=\frac{\EE(m)}{\KK(m)\sqrt{m(1-m)}}-\sqrt{\frac{1-m}{m}},\quad 0<m<1,
\label{eq:rho-func-define}
\end{equation}
such that the fluxon condensate $u_N(x,t)$ associated with $G$ obeys the following asymptotic formul\ae\ in the limit $N\to\infty$ with $\epsilon=\epsilon_N$ given by \eqref{eq:epsilon-N}:
\begin{equation}
\begin{split}
\cos(\tfrac{1}{2}u_N(x,t))&=\dot{C}(t)
-M\epsilon^{1/5}\mathrm{Re}\left\{h\left(\frac{\ii ax+b(t-t_\mathrm{gc})}{\epsilon^{4/5}}\right)\right\}
\mathrm{cn}\left(\frac{2(\Phi_\mathrm{gc}-\omega_\mathrm{gc}(t-t_\mathrm{gc}))\KK(m_\mathrm{gc})}{\pi\epsilon};m_\mathrm{gc}\right)\dot{S}(t)\\
&\qquad\qquad\qquad{}+\mathcal{O}(\epsilon^{2/5})\\
\sin(\tfrac{1}{2}u_N(x,t))&=\dot{S}(t)
+M\epsilon^{1/5}\mathrm{Re}\left\{h\left(\frac{\ii ax+b(t-t_\mathrm{gc})}{\epsilon^{4/5}}\right)\right\}
\mathrm{cn}\left(\frac{2(\Phi_\mathrm{gc}-\omega_\mathrm{gc}(t-t_\mathrm{gc}))\KK(m_\mathrm{gc})}{\pi\epsilon};m_\mathrm{gc}\right)\dot{C}(t)\\
&\qquad\qquad\qquad{}+\mathcal{O}(\epsilon^{2/5})
\end{split}
\label{eq:AwayFromPolesSolution}
\end{equation}
where $h(\tau)$ is the Hamiltonian \eqref{eq:Hamiltonian-intro} for the real tritronqu\'ee solution $y(\tau)$ of the Painlev\'e-I equation \eqref{eq:PI-intro}, where
\begin{equation}
\dot{C}(t):=\mathrm{dn}\left(\frac{2(\Phi_\mathrm{gc}-\omega_\mathrm{gc}(t-t_\mathrm{gc}))\KK(m_\mathrm{gc})}{\pi\epsilon};m_\mathrm{gc}\right)
\label{eq:C-dot-define}
\end{equation}
and
\begin{equation}
\dot{S}(t):=-\sqrt{m_\mathrm{gc}}\,\mathrm{sn}\left(\frac{2(\Phi_\mathrm{gc}-\omega_\mathrm{gc}(t-t_\mathrm{gc}))\KK(m_\mathrm{gc})}{\pi\epsilon};m_\mathrm{gc}\right),
\label{eq:S-dot-define}
\end{equation}
and where the error terms are uniform for
$x^2+(t-t_\mathrm{gc})^2=\mathcal{O}(\epsilon^{8/5})$ and $y(\epsilon^{-4/5}(\ii ax+b(t-t_\mathrm{gc})))$ bounded.
\label{thm:AwayFromPoles}
\end{theorem}
\begin{remark}
The constant $\sigma>0$ is not easy to describe in terms of the ingredients mentioned so far, but it is well-defined and is given by \eqref{eq:SIGMA} below.
\end{remark}
\indent \begin{remark}
Comparing \eqref{eq:rho-func-define} with \eqref{eq:J-func-define} shows that $\pi J(\mathcal{E})=8\KK(m)\sqrt{m(1-m)}\rho(m)$, i.e., the same linear combination of complete elliptic integrals apparently occurs in two different but related settings.  See the right-hand panel of Figure~\ref{fig:Minus-b-Over-a} below for a plot of the function $\rho(m_\mathrm{gc})$ as a function of $m_\mathrm{gc}$.
\end{remark}

Some basic observations related to Theorem~\ref{thm:AwayFromPoles} are the following.
First, the leading terms correspond to the leading terms in Proposition~\ref{prop:BeforeCatastrophe} except that 
the elliptic modulus $m(x,t)$ has been ``frozen'' at the gradient catastrophe point $(x,t)=(0,t_\mathrm{gc})$, and that
the phase $\Phi(x,t)$ has been expanded through the linear terms in its two-variable Taylor expansion about the gradient catastrophe point (where $k(0,t_\mathrm{gc})=0$ has been used).
Also, the leading terms $\dot{C}(t)$ and $\dot{S}(t)$ are independent of $x$ and periodic in $t$ with a period of $2\pi\epsilon/\omega_\mathrm{gc}$, and there is an exact $x$-independent solution $u(t)$ of the sine-Gordon equation \eqref{eq:SemiClassicalsG} (hence the simple pendulum equation, really) $\epsilon^2u''+\sin(u)=0$ such that $\dot{C}(t)=\cos(\tfrac{1}{2}u(t))$ and $\dot{S}(t)=\sin(\tfrac{1}{2}u(t))$.  
Next we consider the perturbing terms proportional to $\epsilon^{1/5}$.  One consequence of the form of these terms is that 
the asymptotic formul\ae\ \eqref{eq:AwayFromPolesSolution} are consistent with the Pythagorean identity 
\begin{equation}
\cos(\tfrac{1}{2}u_N(x,t))^2+\sin(\tfrac{1}{2}u_N(x,t))^2=1 
\end{equation}
through terms of order $\mathcal{O}(\epsilon^{2/5})$.
Also, since $h(\tau^*)=h(\tau)^*$, the fact that it is the real part $\mathrm{Re}\{h(\tau)\}$ that appears in the correction terms is significant because it is consistent with $\cos(\tfrac{1}{2}u_N(x,t))$ and $\sin(\tfrac{1}{2}u_N(x,t))$ being even functions of $x$, as can be shown from the precise form of Definition~\ref{def:fluxon-condensate} to be explained in Section~\ref{sec:RHP} below (cf., \eqref{eq:condensate-even}).
Note furthermore that 
the factor of $\mathrm{Re}\{h(\tau)\}$ with $\epsilon^{4/5}\tau=\ii ax+b(t-t_\mathrm{gc})$ in the sub-leading terms varies slowly compared to the period of the leading terms, giving the asymptotic formul\ae\ in \eqref{eq:AwayFromPolesSolution} an inherently multiscale structure.  Since $h(\tau)$ has simple poles (these are excluded by the boundedness condition on $y(\tau)$) there is a curve passing through the image in the $(x,t)$-plane of each pole, along which $\mathrm{Re}\{h(\tau)\}=0$ and hence the sub-leading terms vanish identically.  It appears that each such curve closes on itself to form a loop, and there is an additional unbounded component of the level curve that does not meet any poles in the finite $\tau$-plane.  See Figure~\ref{fig:RealTTHamiltonian}.
\begin{figure}[h]
\begin{center}
\includegraphics[height=0.4\linewidth]{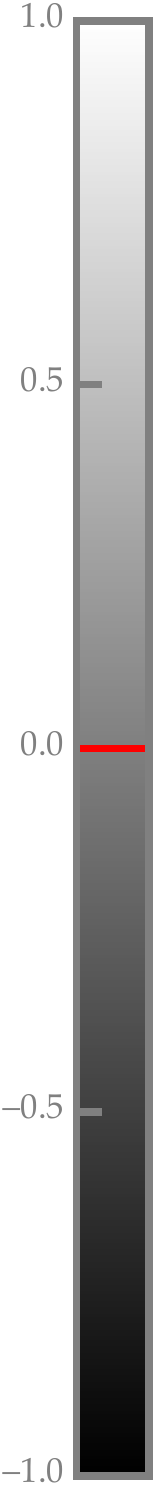}\hspace{0.05\linewidth}%
\includegraphics[height=0.4\linewidth]{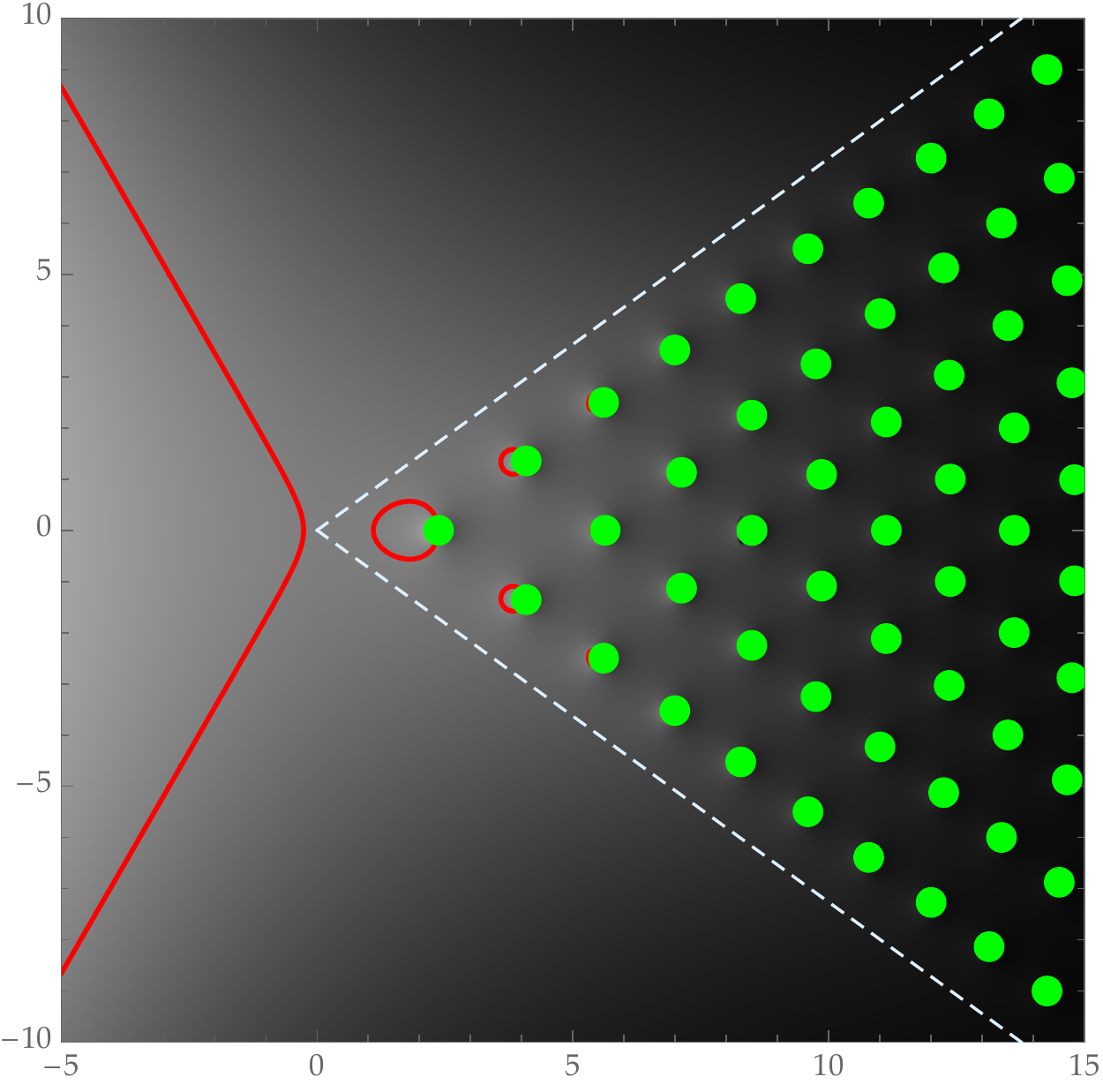}\hspace{0.05\linewidth}%
\includegraphics[height=0.4\linewidth]{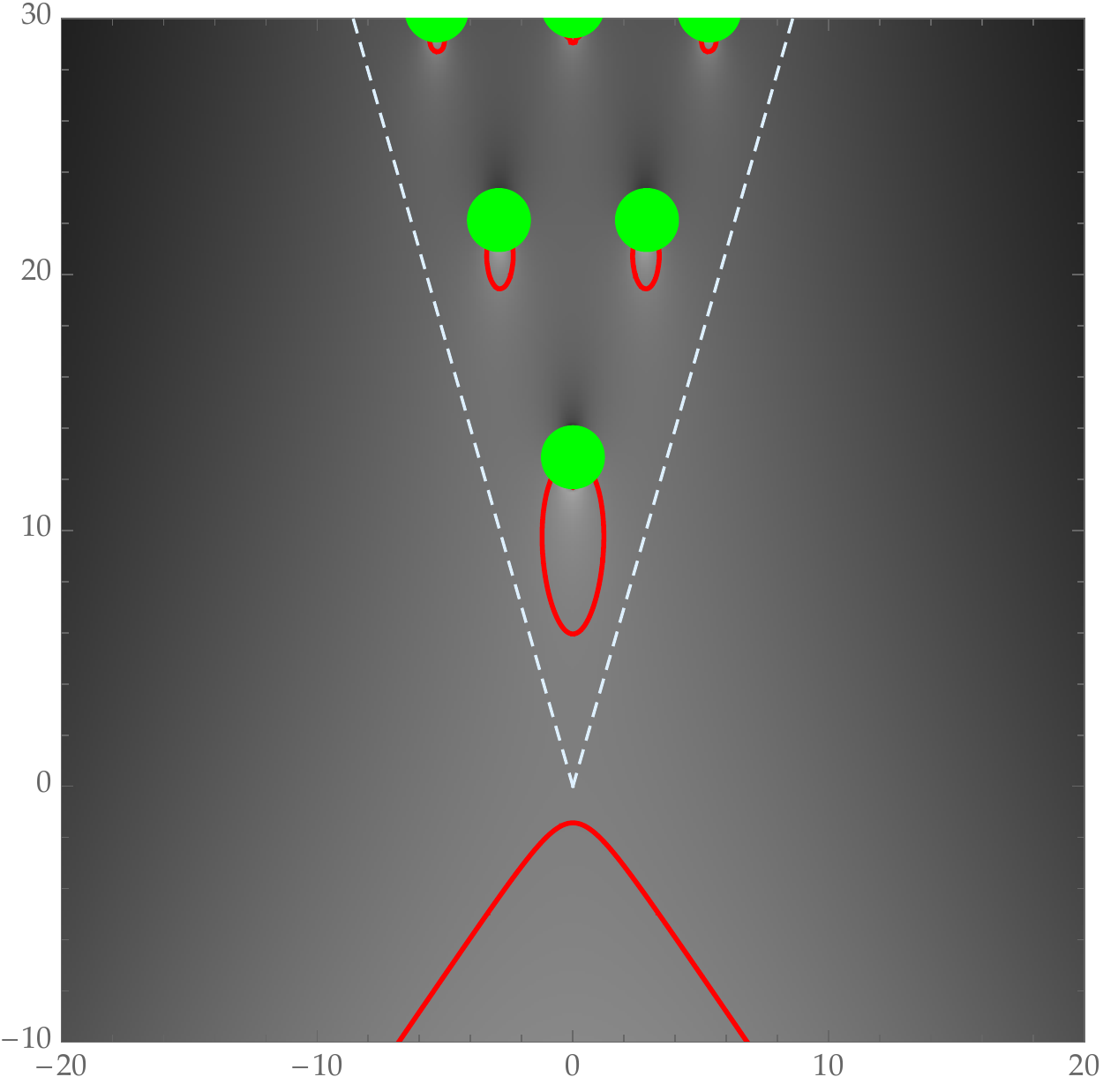}
\end{center}
\caption{Plots of $\tanh(\tfrac{1}{10}\mathrm{Re}\{h\})$ in the $\tau$-plane (left panel) and the $(\tilde{x},\tilde{t})=(\epsilon^{-4/5}x,\epsilon^{-4/5}(t-t_\mathrm{gc}))$-plane where $t_\mathrm{gc}$ and the scale factors $a<0$ and $b>0$ correspond to the initial condition $G(x)=-\mathrm{sech}(x)$ (right panel).  The zero level curve is shown in red; this is where the first correction terms in \eqref{eq:AwayFromPolesSolution} vanish.  Also, the boundary of the pole-confining sector is shown with dashed lines, and a green disk is centered on each pole.  The data for these plots was furnished by M. Fasondini, B. Fornberg, and J.\@ A.\@ C.\@ Weideman, using the numerical methodology for the Painlev\'e equations described in \cite{FornbergWeideman2011}.}
\label{fig:RealTTHamiltonian}
\end{figure}
Finally, if as a relatively slowly-varying function $\mathrm{Re}\{h(\tau)\}$ is treated as a constant, the sub-leading terms amount to a very special exact solution of the linearization of the sine-Gordon equation  about the leading terms.  In other words, starting from $\epsilon^2u''+\sin(u)=0$ and setting $P=\tfrac{1}{2}\epsilon u'$, $\dot{C}=\cos(\tfrac{1}{2}u)$, $\dot{S}=\sin(\tfrac{1}{2}u)$, we deduce the exact first-order system 
\begin{equation}
\begin{split}
\epsilon P'&=-\dot{C}\dot{S}\\
\epsilon\dot{C}'&=-\dot{S}P\\
\epsilon\dot{S}'&=\dot{C}P.
\end{split}
\end{equation}
Linearizing this system about the solution $u(t)$ by replacing $(P,\dot{C},\dot{S})$ by $(P+P_1, \dot{C}+\dot{C}_1,\dot{S}+\dot{S}_1)$ and retaining the terms linear in the perturbation $(P_1,\dot{C}_1,\dot{S}_1)$  yields the linear system 
\begin{equation}
\begin{split}
\epsilon P_1'&=-\dot{C}\dot{S}_1-\dot{C}_1\dot{S}\\
\epsilon\dot{C}_1'&=-\dot{S}P_1-\dot{S}_1P\\
\epsilon\dot{S}_1'&=\dot{C}P_1+\dot{C}_1P.
\end{split}
\end{equation}
It is straightforward to check that if $\kappa$ is any constant, a particular solution of the latter system is proportional to the $t$-derivative of the unperturbed solution:  $(P_1,\dot{C}_1,\dot{S}_1)=(\kappa P',\kappa\dot{C}',\kappa\dot{S}')$.  We then observe that upon replacing $\kappa$ with the ``constant'' (slowly-varying function)
\begin{equation}
\kappa=\frac{\pi M\mathrm{Re}\{h(\tau)\}}{2\omega_\mathrm{gc}\sqrt{m_\mathrm{gc}}\KK(m_\mathrm{gc})}\epsilon^{6/5}, 
\end{equation}
then $\kappa\dot{C}'$ and $\kappa\dot{S}'$ are precisely the explicit sub-leading terms on the two lines of \eqref{eq:AwayFromPolesSolution}.

The fact that $M>0$ in Theorem~\ref{thm:AwayFromPoles} amounts to an indirect proof of an identity involving Riemann theta functions of genus $1$ that we have not been able to prove by other means; see the remark at the end of Appendix~\ref{sec:proof-lemma:C-real} for details.  

In some way, Theorem~\ref{thm:AwayFromPoles} asserts that the dominant effect near the gradient catastrophe point is a weak modulation of the uniformly spatially constant and time-periodic librating wave of period proportional to $\epsilon$ described by the unperturbed terms $(\dot{C}(t),\dot{S}(t))$.  This modulation takes place on space and time scales centered at the catastrophe point and proportional to $\epsilon^{4/5}$, and hence is slowly-varying compared with the period of the unperturbed background.  To illustrate this two-scale phenomenon, we first plot the fluxon condensate for the impulse profile $G(x)=-\mathrm{sech}(x)$, for which $\epsilon_N=(4N)^{-1}$, showing how $\cos(u_N(x,t))=\cos^2(\tfrac{1}{2}u_N(x,t))-\sin^2(\tfrac{1}{2}u_N(x,t))$ behaves as a function of $(x,t)$ for two different values of $N$.  For this particular impulse profile, $\epsilon$-independent numerical calculations described in Section~\ref{sec:boundary} below predict that the gradient catastrophe point occurs at $(x,t)=(0,t_\mathrm{gc})$ for $t_\mathrm{gc}\approx 1.609104$.  See Figure~\ref{fig:G(x)=-sech(x)--cos(u_8,u_16)}.
\begin{figure}[h]
\begin{center}
\includegraphics[height=0.4\linewidth]{fig/Legend-SMALL.pdf}\hspace{0.05\linewidth}
\includegraphics[height=0.4\linewidth]{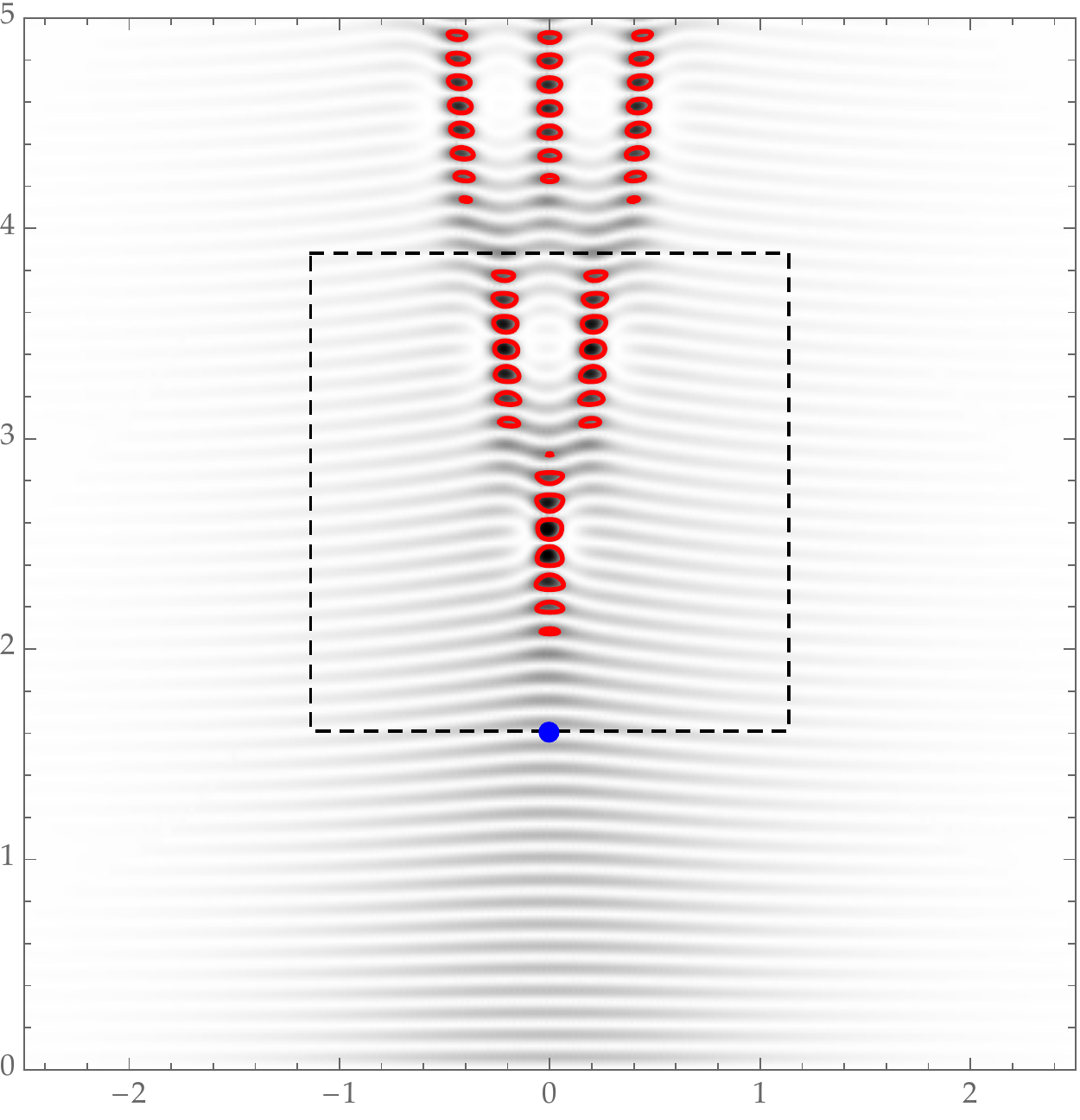}\hspace{0.05\linewidth}
\includegraphics[height=0.4\linewidth]{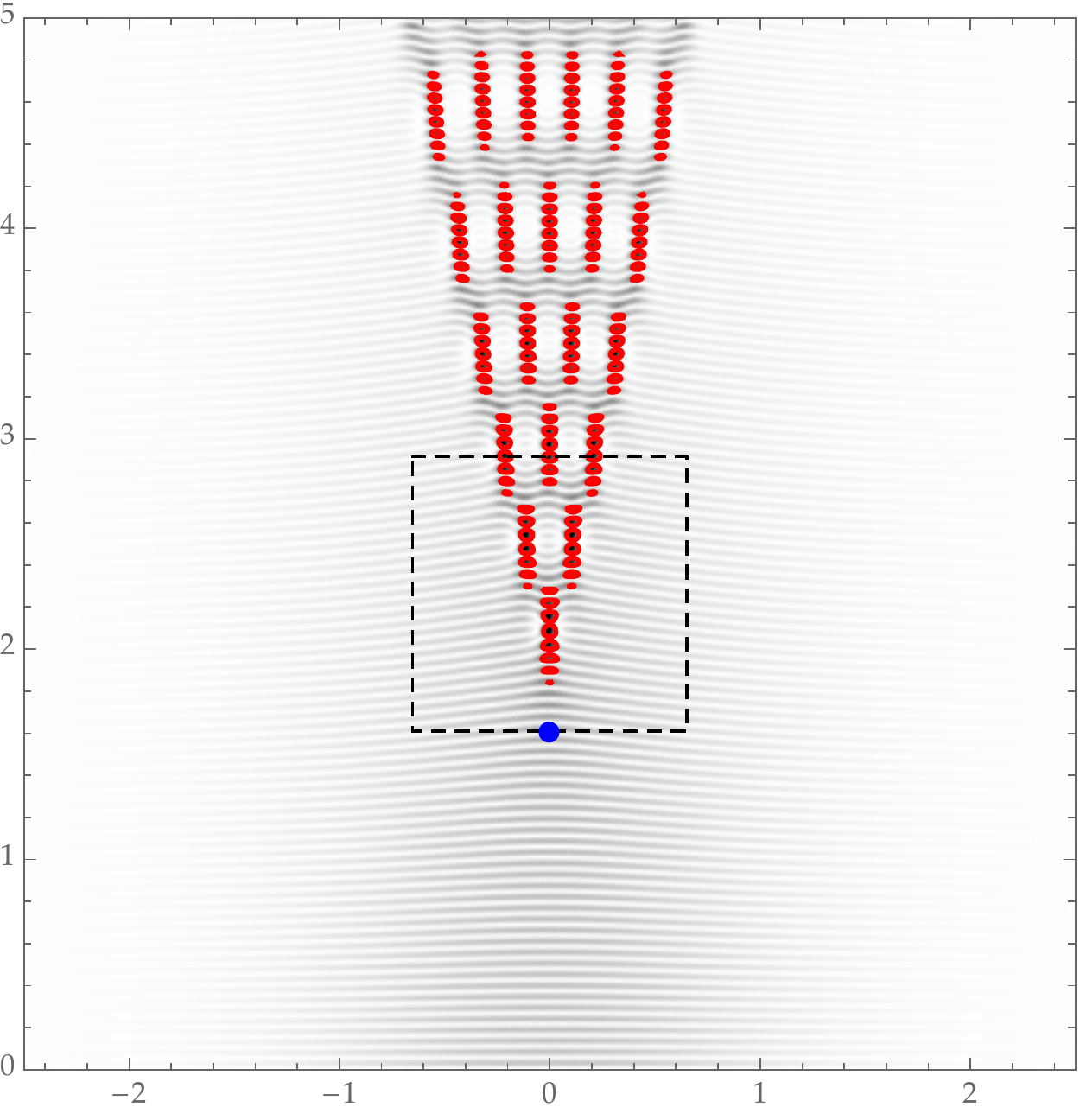}
\end{center}
\caption{Plots of $\cos(u_N(x,t))$ over the $(x,t)$-plane for the fluxon condensate for $G(x)=-\mathrm{sech}(x)$ with $N=8$ (left) and $N=16$ (right).  The blue dot is the theoretical location of the gradient catastrophe point $(x,t)=(0,t_\mathrm{gc})$, where $t_\mathrm{gc}\approx 1.609104$.  The dashed region refers to the zoomed-in plots shown in Figure~\ref{fig:zoom-1} below.}
\label{fig:G(x)=-sech(x)--cos(u_8,u_16)}
\end{figure}
Next, we zoom in on the gradient catastrophe point by introducing new coordinates $(\tilde{x},\tilde{t})=(\epsilon_N^{-4/5}x,\epsilon_N^{-4/5}(t-t_\mathrm{gc}))$ and plotting the same functions in the new coordinates.  (These are the same coordinates used in the right-hand panel of Figure~\ref{fig:RealTTHamiltonian}.) See Figure~\ref{fig:zoom-1}.
\begin{figure}[h]
\begin{center}
\includegraphics[height=0.4\linewidth]{fig/Legend-SMALL.pdf}\hspace{0.05\linewidth}%
\includegraphics[height=0.4\linewidth]{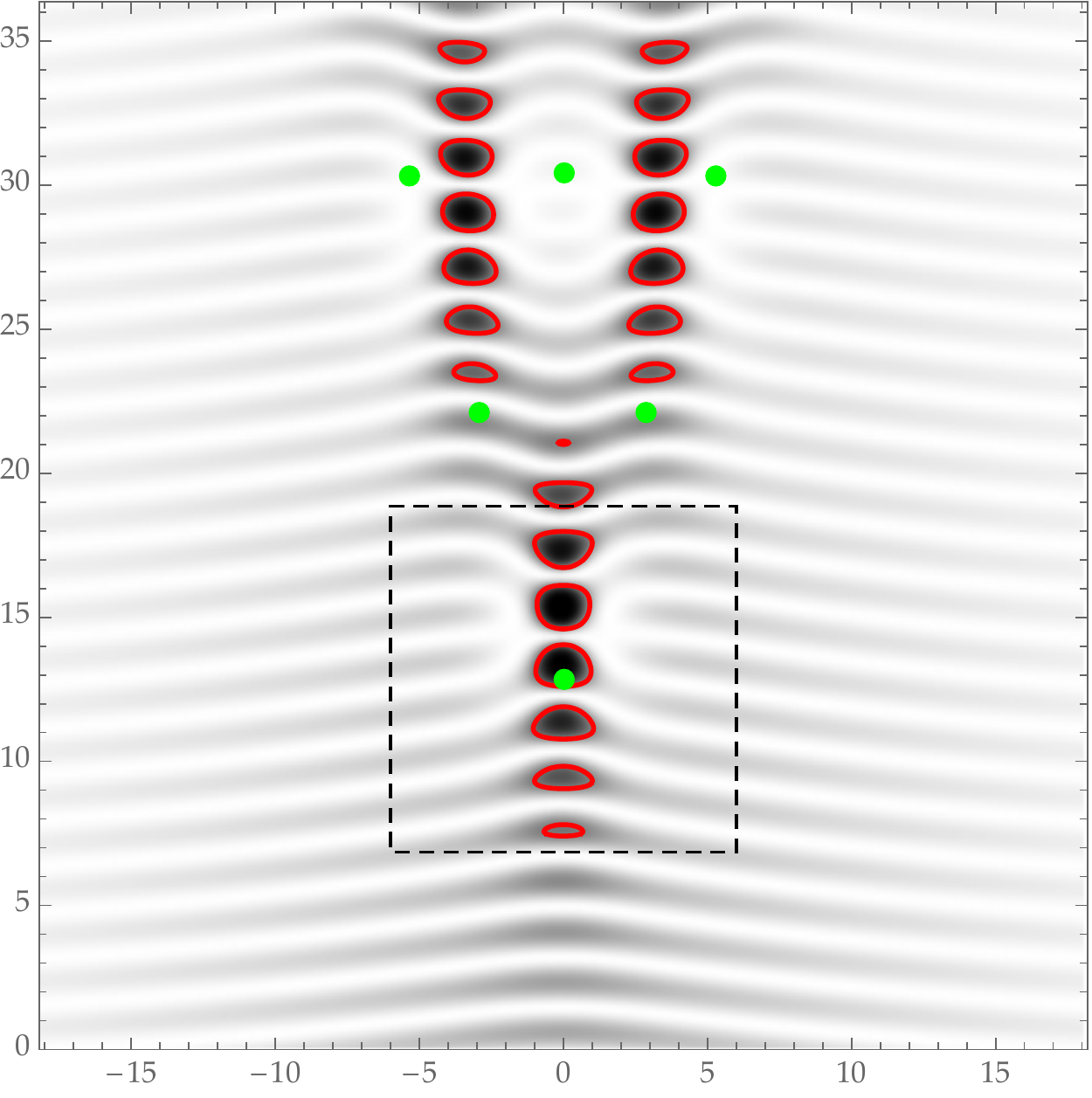}\hspace{0.05\linewidth}%
\includegraphics[height=0.4\linewidth]{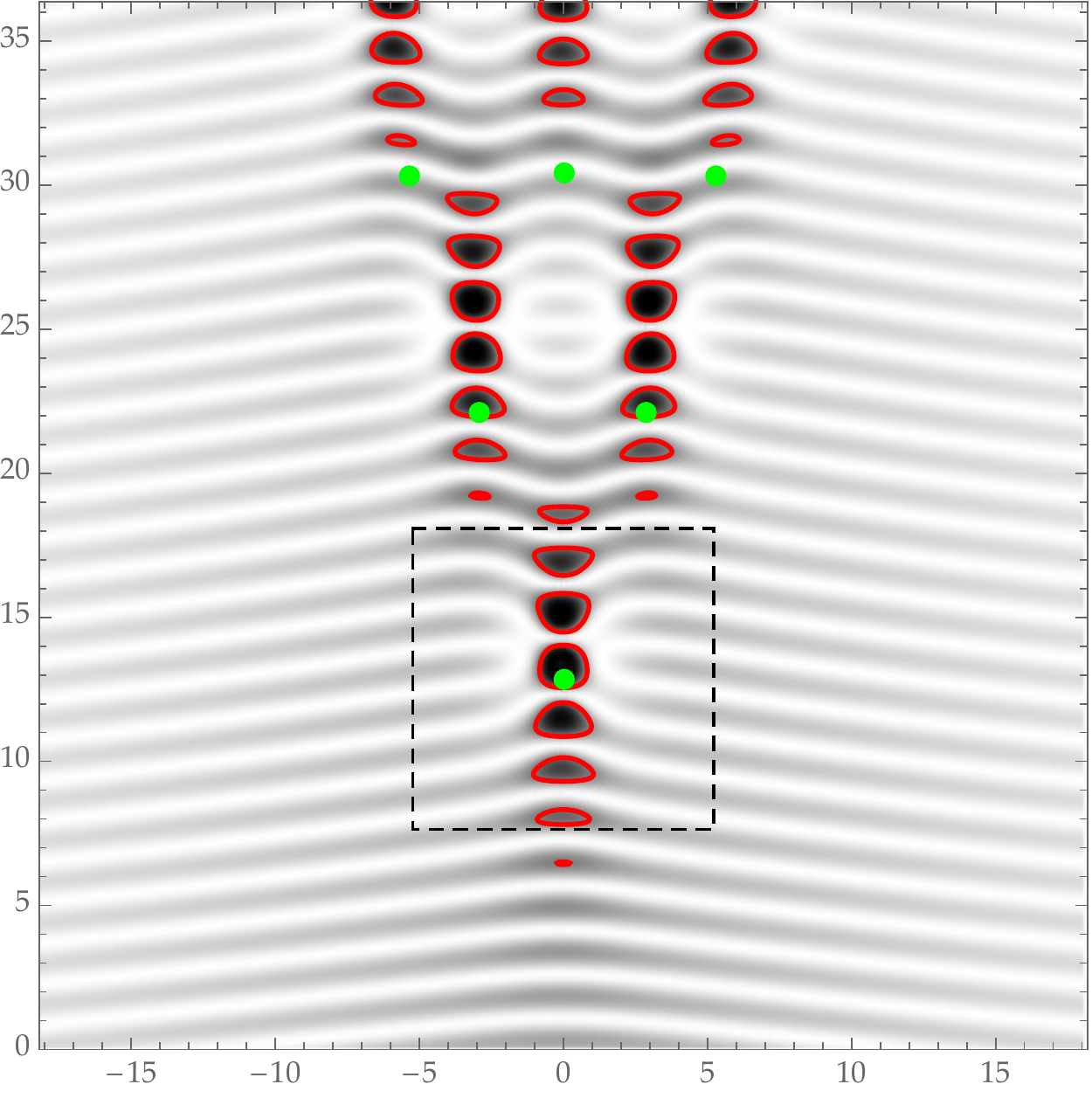}
\end{center}
\caption{Plots of $\cos(u_N(\epsilon_N^{4/5}\tilde{x},t_\mathrm{gc}+\epsilon_N^{4/5}\tilde{t}))$ over the $(\tilde{x},\tilde{t})$-plane for the fluxon condensate for $G(x)=-\mathrm{sech}(x)$ with $N=8$ (left) and $N=16$ (right).  The plot domain is the image in the $(\tilde{x},\tilde{t})$-plane of the dashed squares shown in the corresponding panels of Figure~\ref{fig:G(x)=-sech(x)--cos(u_8,u_16)}.  The green dots are the preimages under $\tau=\ii a\tilde{x}+ b\tilde{t}$ of the six poles of $h(\cdot)$ closest to the origin.  The values of $a$ and $b$ are computed as in the statement of Theorem \ref{thm:AwayFromPoles} with \ $\sigma>0$ specified by \eqref{eq:SIGMA}, and the pole data for $h(\cdot)$ was provided by B. Fornberg and J. A. C. Weideman using the method described in \cite{FornbergWeideman2011}. The dashed region refers to the zoomed-in plots shown in Figure~\ref{fig:zoom-2} below.}
\label{fig:zoom-1}
\end{figure}
The plots in Figure~\ref{fig:zoom-1} suggest a limiting alignment of the visible ``defects'' in the otherwise uniform wavetrain that forms the background in the plots (an ideal uniform wave of reciprocal phase velocity $n_\mathrm{p}=0$ would show up as a pattern of horizontal stripes).  In fact, Theorem~\ref{thm:AwayFromPoles} predicts that the defects must converge, after suitable mapping to the complex coordinate $\tau=\ii a\tilde{x}+b\tilde{t}$, to the poles of the real tritronqu\'ee solution $y(\tau)$ as $N\to\infty$.  

Our next result concerns the defects themselves.  Now Theorem~\ref{thm:AwayFromPoles} does not describe the fluxon condensate near the preimage under $\tau$ of any pole of $y(\tau)$, but if we nonetheless examine the behavior of the approximation on the right-hand sides of \eqref{eq:AwayFromPolesSolution} near such a point, the fact that $h(\tau)$ has simple poles only suggests that the constant $\sigma>0$ cancels out of the leading error terms, since it appears homogeneously in the constants $M>0$, $a<0$, and $b>0$.  As $\sigma$ is the only source of dependence in these terms on the initial data $G(\cdot)$ rather than on quantities evident from the solution in the neighborhood of the gradient catastrophe point only, we may expect that when Theorem~\ref{thm:AwayFromPoles} fails to describe the defects attracted to preimages of poles of the tritronqu\'ee solution, whatever approximation takes over may have a universal character.
To formulate our description of the defects, which indeed verifies this conjecture, we must first describe a particular two-parameter family of exact solutions $U(X,T;m,\Omega)$ of the sine-Gordon equation in the unscaled form $U_{TT}-U_{XX}+\sin(U)=0$, with parameters $0<m<1$ and $\Omega\in\mathbb{R}\pmod{2\pi}$.  
These solutions are constructed as part of the proof of Theorem~\ref{thm:NearThePoles} below, and they can be defined in terms of Jacobian elliptic functions and the complete elliptic integrals $\KK(m)$ and $\EE(m)$ as follows.
\begin{equation}
\begin{bmatrix}
\cos(\tfrac{1}{2}U(X,T;m,\Omega))\\
\sin(\tfrac{1}{2}U(X,T;m,\Omega))\end{bmatrix}=\mathbf{R}(X,T;m,\Omega)
\begin{bmatrix}\dot{C}(T;m,\Omega)\\
\dot{S}(T;m,\Omega)
\end{bmatrix}
\label{eq:ddotCS-compact}
\end{equation}
where
\begin{equation}
\dot{C}(T;m,\Omega):=\mathrm{dn}\quad\text{and}\quad\dot{S}(T;m,\Omega):=-\sqrt{m}\,\mathrm{sn},
\end{equation}
and $\mathbf{R}(X,T;m,\Omega)$ is an orthogonal (rotation) matrix with elements
\begin{equation}
R_{11}(X,T;m,\Omega)=R_{22}(X,T;m,\Omega)=\frac{r(X,T;m,\Omega)^2-q(X,T;m,\Omega)^2}{q(X,T;m,\Omega)^2+r(X,T;m,\Omega)^2}
\end{equation}
and
\begin{equation}
R_{21}(X,T;m,\Omega)=-R_{12}(X,T;m,\Omega)=\frac{2q(X,T;m,\Omega)r(X,T;m,\Omega)}{q(X,T;m,\Omega)^2+r(X,T;m,\Omega)^2},
\end{equation}

in which
\begin{equation}
\begin{split}
q(X,T;m,\Omega)&:=-\frac{\mathrm{sn}\,\mathrm{dn}+\left((1-m)p-\sqrt{m(1-m)}\rho(m)T\right)\mathrm{cn}}{2\sqrt{1-m}}\\
r(X,T;m,\Omega)&:=\frac{m-\mathrm{dn}^2-m(1-m)X^2-\left((1-m)p-\sqrt{m(1-m)}\rho(m)T\right)^2}{4\sqrt{m(1-m)}},
\end{split}
\label{eq:q-and-r-define}
\end{equation}
where $\rho(m)$ is defined by \eqref{eq:rho-func-define} and the following abbreviated notation is used:
\begin{equation}
\mathrm{p}=\mathrm{p}\left(\frac{2\KK(m)\Omega}{\pi}+T;m\right),\quad\text{and}\quad \mathrm{xn}=\mathrm{xn}\left(\frac{2\KK(m)\Omega}{\pi}+T;m\right),\quad \mathrm{x}=\mathrm{c},\mathrm{s},\mathrm{d},
\label{eq:abbreviated-notation-1}
\end{equation}
and where $\mathrm{p}(w;m)$ is a periodic (but not elliptic) function of $w$ defined by
\begin{equation}
\mathrm{p}(w;m):=\int_0^{w+\KK(m)}\left(\left\langle\frac{1}{\mathrm{dn}(\cdot;m)^2}\right\rangle-\frac{1}{\mathrm{dn}(\zeta;m)^2}\right)\,\dd\zeta,
\label{eq:abbreviated-notation-2}
\end{equation}
in which $\langle f(\cdot)\rangle$ denotes the average of a periodic function $f:\mathbb{R}\to\mathbb{R}$,
in this case given explicitly by
\begin{equation}
\left\langle\frac{1}{\mathrm{dn}(\cdot;m)^2}\right\rangle = \frac{\EE(m)}{(1-m)\KK(m)}=1+\sqrt{\frac{m}{1-m}}\rho(m).
\label{eq:abbreviated-notation-3}
\end{equation}
Note that $\dot{C}$ and $\dot{S}$ are essentially the same quantities defined in \eqref{eq:C-dot-define}--\eqref{eq:S-dot-define}, now written in terms of different variables, namely $T$ and the phase $\Omega$.  It is easy to see that $q(X,T;m,\Omega)=\mathcal{O}(T)$, while
$r(X,T;m,\Omega)=-\tfrac{1}{4}\sqrt{m(1-m)}(X^2+\rho(m)^2T^2) + \mathcal{O}(T)$ as $X^2+T^2\to\infty$; this in turn implies that for each $m\in (0,1)$, $\mathbf{R}(X,T;m,\Omega)=\mathbb{I}+\mathcal{O}(T/(X^2+T^2))$ in the same limit.  Therefore, 
these exact solutions take the form of a space-time localized defect of a spatially constant time-periodic background corresponding to a solution of the simple pendulum ordinary differential equation with elliptic modulus $m$.  Despite the fact that the formul\ae\ are complicated, it is easy to plot the defect solutions.  See Appendix~\ref{sec:Catalog} for several such plots displaying how the defect solutions vary with the parameters $m$ and $\Omega$.  
The defect solutions exhibit the following features:
\begin{itemize}
\item The main effect of the parameter $\Omega$ is to position the defect temporally relative to the time-periodic background.  The functions $\cos(U(X,T;m,\Omega))$ and $\sin(U(X,T;m,\Omega))$ are periodic functions of $\Omega$; it is easy to see that $\cos(U(X,T;m,\Omega+\pi))=\cos(U(X,T;m,\Omega))$ and $\sin(U(X,T;m,\Omega+\pi)=-\sin(U(X,T;m,\Omega))$.
\item The temporal localization of the defect depends strongly on the value of $m\in (0,1)$.  For small values of $m$ the defect has a very long duration in $T$, while for larger $m$ the duration is shorter.
\item The spatial localization of the defect seems to be relatively insensitive to the value of $m$.
\end{itemize}
The construction of the solution $U(X,T;m,\Omega)$ of $U_{TT}-U_{XX}+\sin(U)=0$ given in the proof of the following theorem shows that it is obtained from the spatially-constant and time-periodic background solution determined from the leading terms in Theorem~\ref{thm:AwayFromPoles} via a kind of Darboux transformation of the Lax eigenfunctions.  Hence the defects we are considering can also be viewed as \emph{rogue waves on an elliptic function background}.  Similar solutions have been constructed by direct methods for other equations, see for example \cite{ChenPelinovsky2018b,ChenPelinovsky2018a,ChenPelinovsky2019}.
\begin{theorem}[Limiting solution near each tritronqu\'ee pole]
Let $\tau_\mathrm{p}$ be a pole of the real tritronqu\'ee solution $y(\tau)$ of the Painlev\'e-I equation.
Define coordinates $X=\epsilon^{-1}(x-x_\mathrm{p})$ and $T=\epsilon^{-1}(t-t_\mathrm{p})$, where the $\epsilon$-dependent point $(x_\mathrm{p},t_\mathrm{p})$ is determined from $\ii a x_\mathrm{p}+(t_\mathrm{p}-t_\mathrm{gc})b=\epsilon^{4/5}\tau_\mathrm{p}$.  Then, under the same assumptions as in Theorem \ref{thm:AwayFromPoles},
\begin{equation}
\begin{split}
\cos(\tfrac{1}{2}u_N(x_\mathrm{p}+\epsilon X,t_\mathrm{p}+\epsilon T))&=\cos(\tfrac{1}{2}U(X,T;m_\mathrm{gc},\Omega_{\mathrm{p},N}))+\mathcal{O}(\epsilon^{1/5}),
\\
\sin(\tfrac{1}{2}u_N(x_\mathrm{p}+\epsilon X,t_\mathrm{p}+\epsilon T))&=\sin(\tfrac{1}{2}U(X,T;m_\mathrm{gc},\Omega_{\mathrm{p},N}))+\mathcal{O}(\epsilon^{1/5}),
\end{split}
\end{equation}
where $m_\gc=m(0,t_\gc)$, $\Omega_{\mathrm{p},N}:=\epsilon^{-1}(\Phi_\mathrm{gc}-(t_\mathrm{p}-t_\mathrm{gc})\omega_\mathrm{gc})$, $\epsilon=\epsilon_N$, and where the error terms are uniform for bounded $(X,T)$.
\label{thm:NearThePoles}
\end{theorem}
In terms of visual confirmation of this result, we may zoom in further by blowing up the dashed squares in the plots in Figure~\ref{fig:zoom-1}.  Blowing up by a factor of $\epsilon^{-1/5}$ a neighborhood of the image in the $(\tilde{x},\tilde{t})$-plane of the nearest pole of $y(\tau)$ to the origin (say) brings us to the corresponding $(X,T)$-plane of Theorem~\ref{thm:NearThePoles}.  Plots of $\cos(u_N)$ in this ``doubly-zoomed'' frame of reference are shown in Figure~\ref{fig:zoom-2}.
\begin{figure}[h]
\begin{center}
\includegraphics[height=0.4\linewidth]{fig/Legend-SMALL.pdf}\hspace{0.05\linewidth}%
\includegraphics[height=0.4\linewidth]{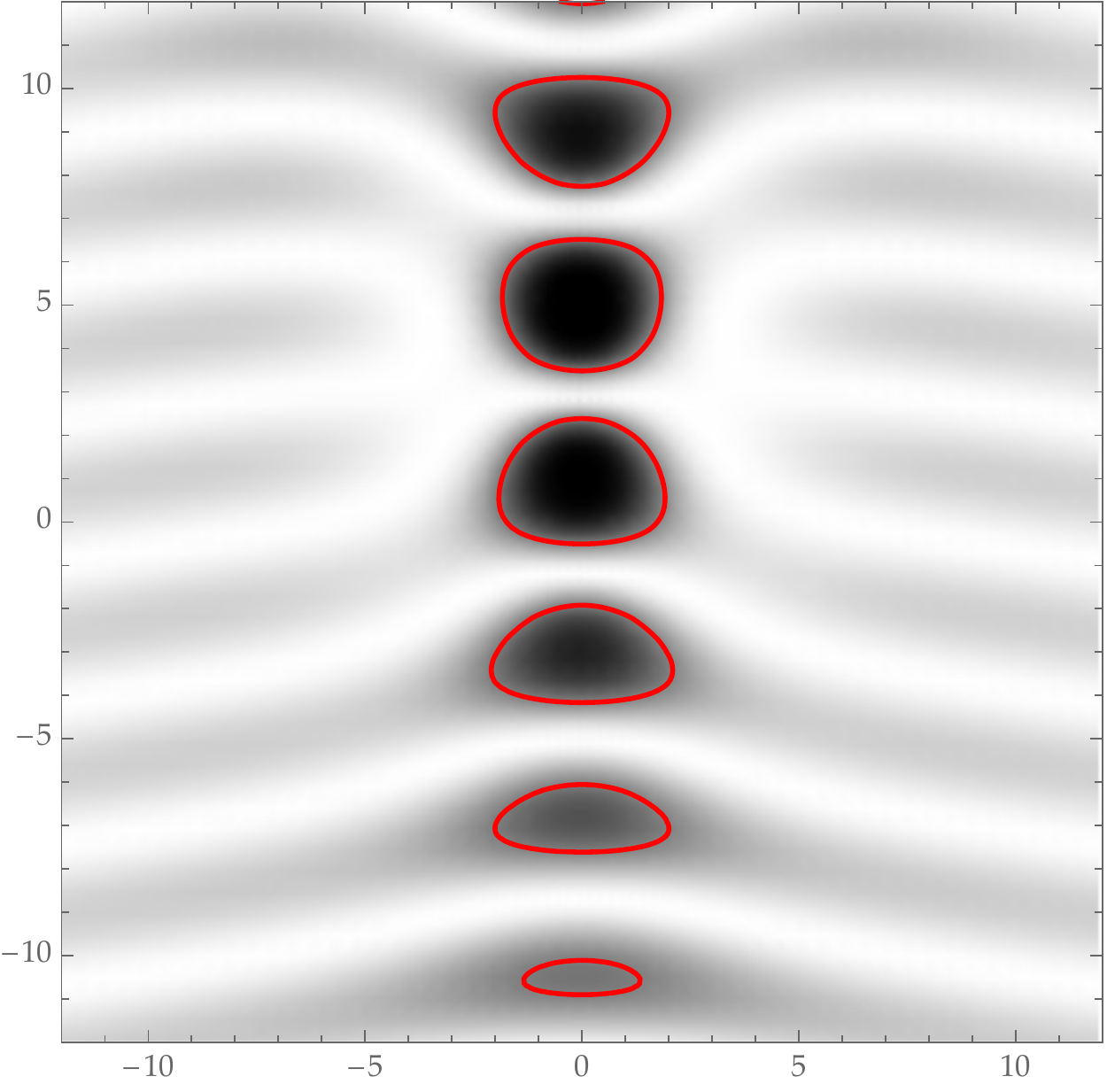}\hspace{0.05\linewidth}%
\includegraphics[height=0.4\linewidth]{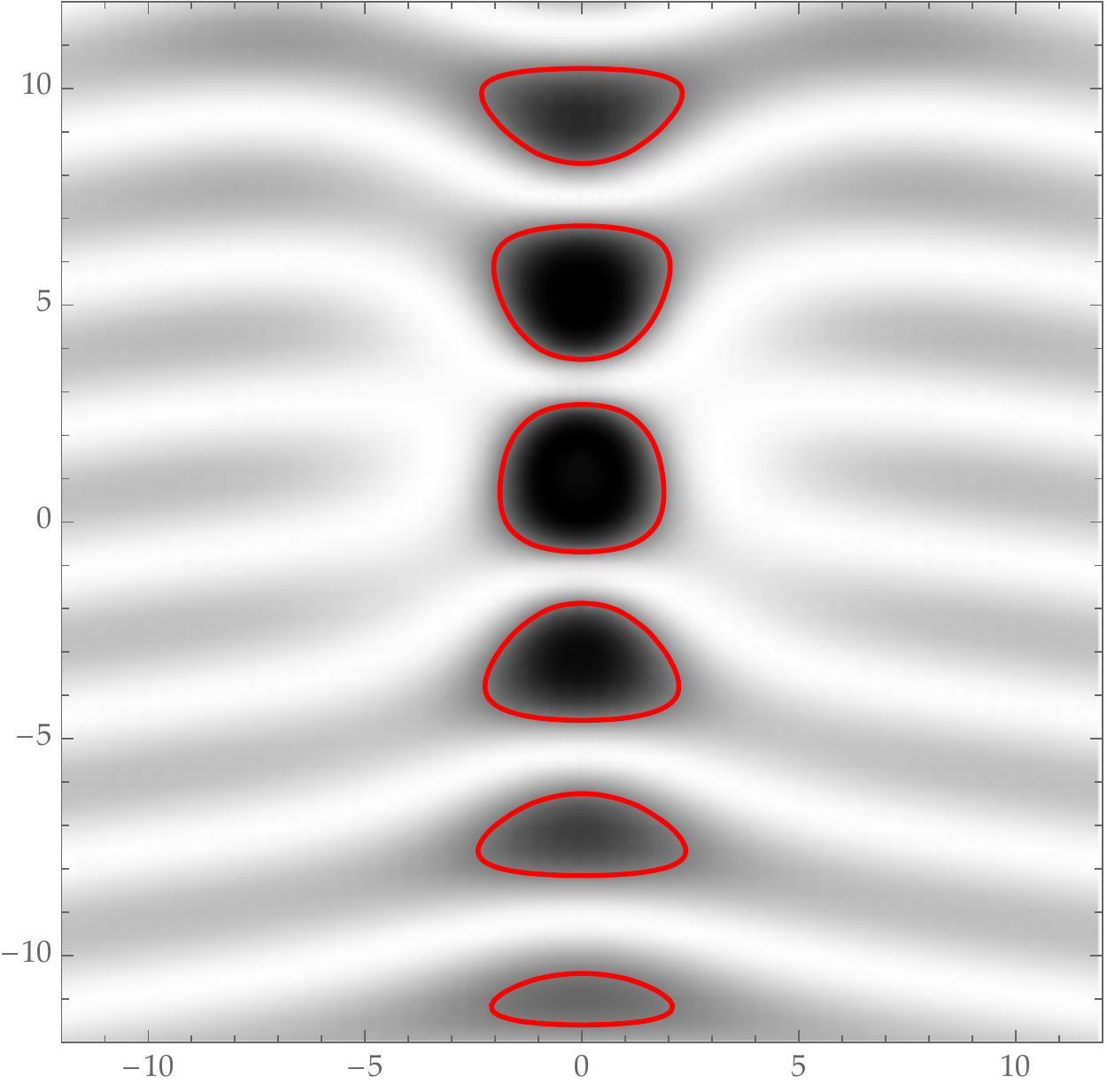}
\end{center}
\caption{Plots of $\cos(u_N(\epsilon_NX,t_\mathrm{gc}+\tau_1b^{-1}\epsilon_N^{4/5}+\epsilon_NT))$ over the $(X,T)$-plane for the fluxon condensate for $G(x)=-\mathrm{sech}(x)$ with $N=8$ (left) and $N=16$ (right).  Here $\tau_1\approx 2.375$ is the (real, positive) closest pole of $h(\tau)$ to the origin and the scaling factor $b$ is given by $b=(m_\gc(1-m_\gc))^{1/4}\rho(m_\gc)/(2\sigma)$ where $\rho(m_\gc)$ is given by \eqref{eq:rho-func-define} and $\sigma$ is defined by \eqref{eq:SIGMA} below.
}
\label{fig:zoom-2}
\end{figure}
The two panels in Figure~\ref{fig:zoom-2} show similar defects.  Moreover, extracting from the $\epsilon$-independent numerics described in Sections~\ref{sec:modulated-librational-wave-region} and \ref{sec:boundary} below the approximate value of $m_\mathrm{gc}\approx 0.416708$, we may compare with plots of the exact solution $U(X,T;m_\mathrm{gc},\Omega_{\mathrm{p},N})$ for different values\footnote{Determining precisely which values of $\Omega_{\mathrm{p},N}$ correspond to the plots in Figure \ref{fig:zoom-2} requires knowing the value of $\Phi_\mathrm{gc}=\Phi(0,t_\mathrm{gc})$ which is global information that we did not compute numerically.} of $\Omega_{\mathrm{p},N}$; see Figure~\ref{fig:zoom-2-compare}.
\begin{figure}[h]
\begin{center}
\includegraphics[height=0.24\linewidth]{fig/Legend-SMALL.pdf}%
\includegraphics[height=0.24\linewidth]{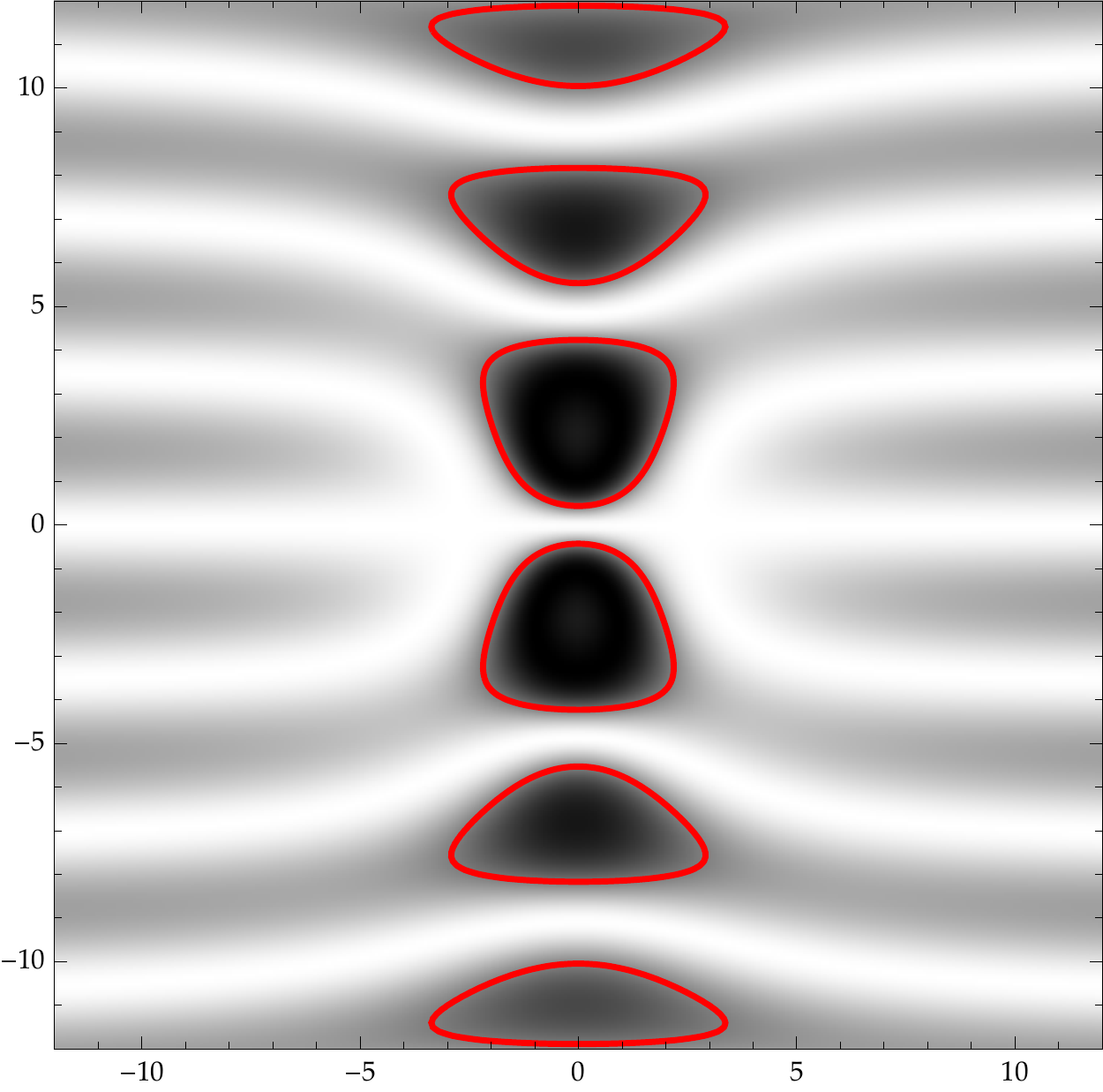}%
\includegraphics[height=0.24\linewidth]{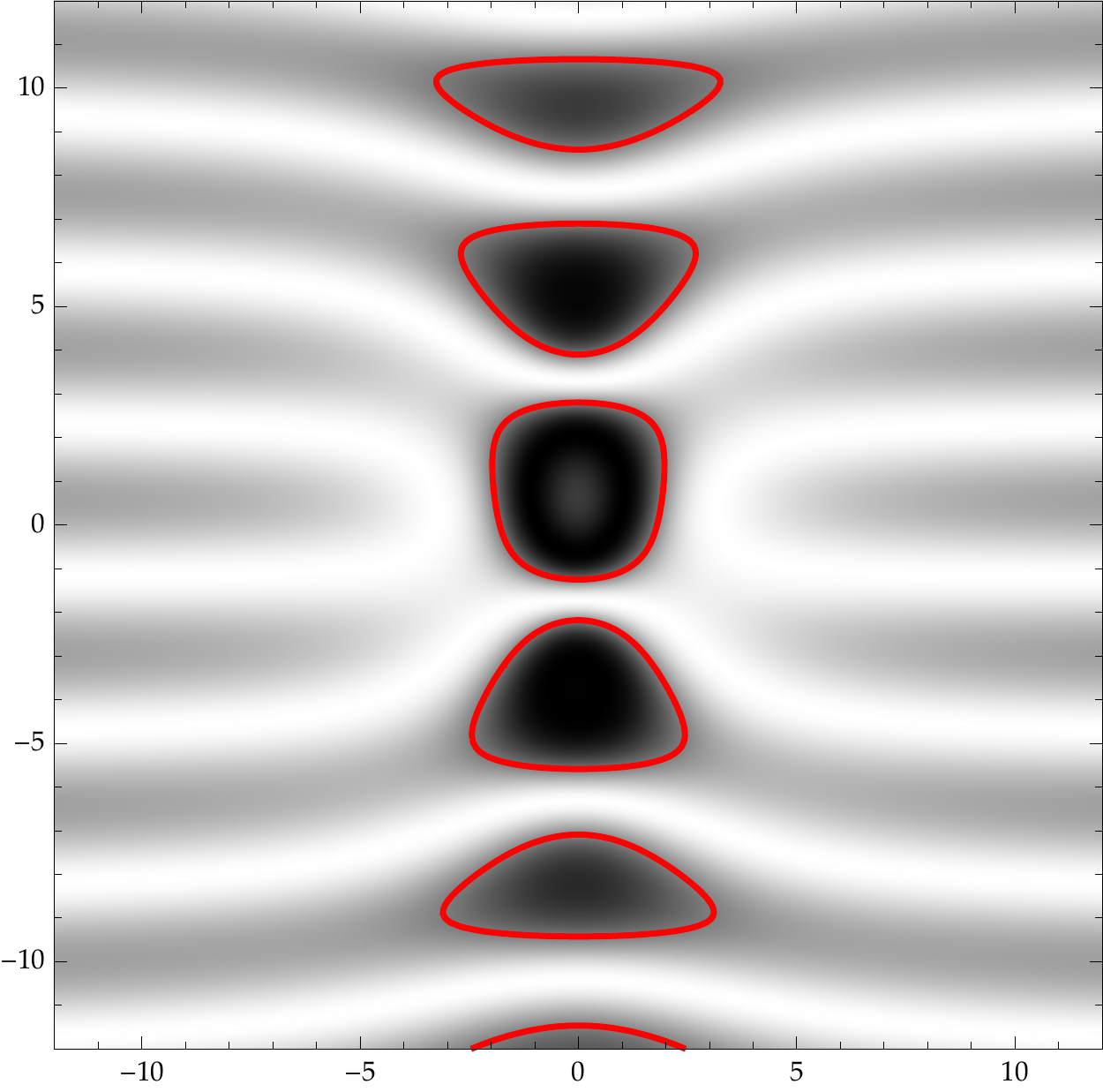}%
\includegraphics[height=0.24\linewidth]{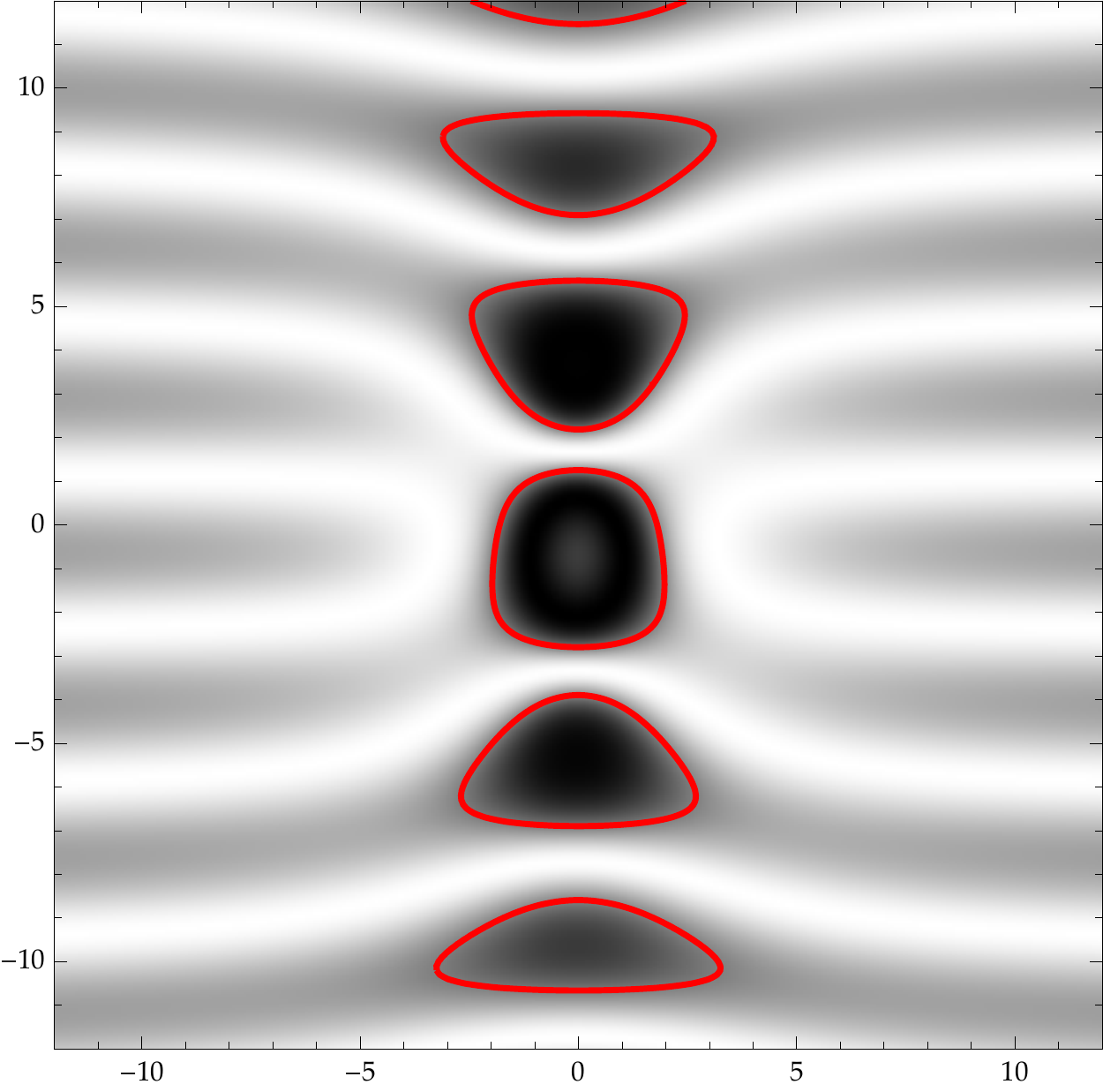}%
\includegraphics[height=0.24\linewidth]{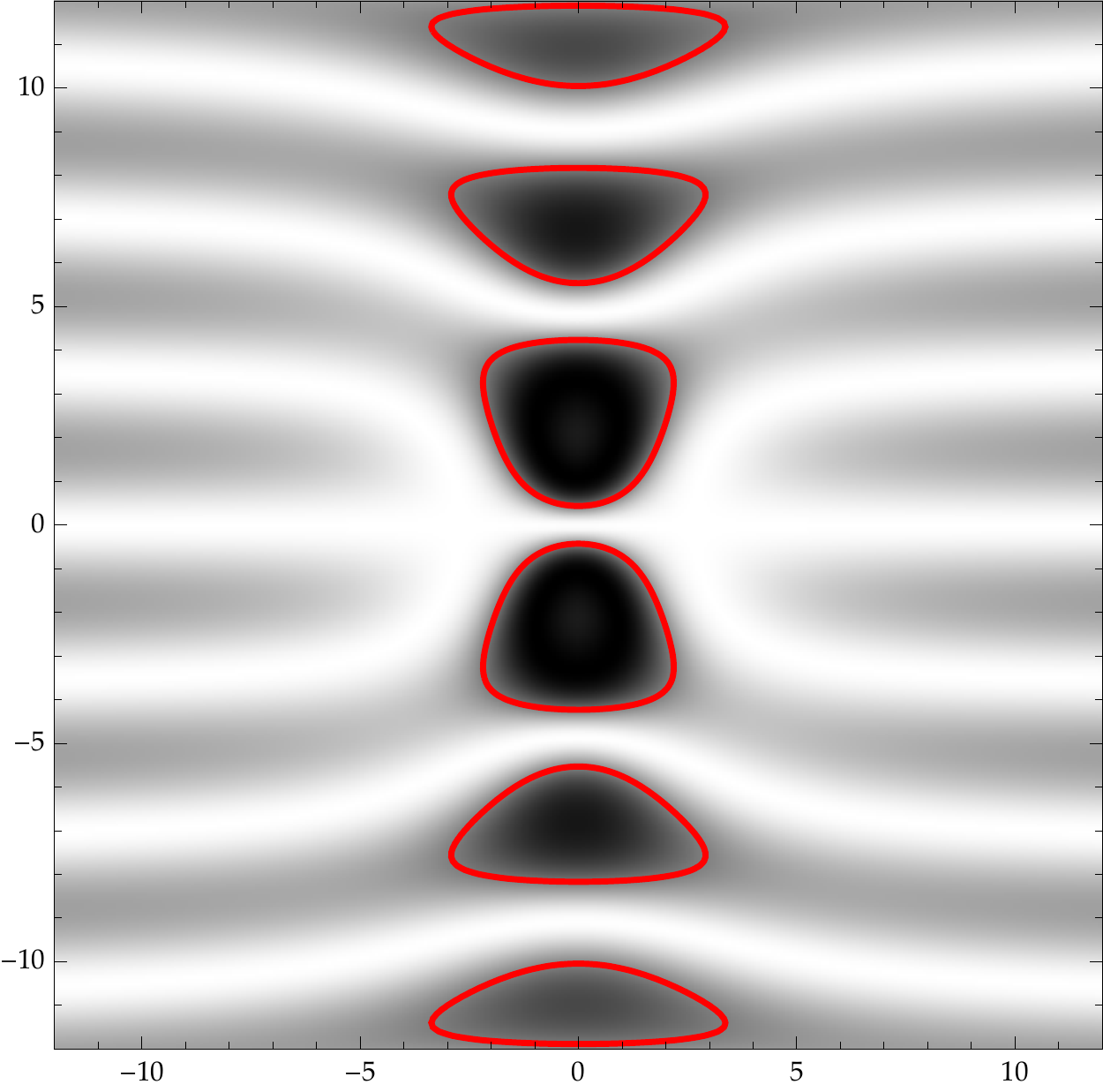}
\end{center}
\caption{Plots of $\cos(U(X,T;0.416708,\Omega))$ over the $(X,T)$-plane for $\Omega=0,\tfrac{1}{3}\pi,\tfrac{2}{3}\pi,\pi$, left-to-right.  Compare with Figure~\ref{fig:zoom-2}.}
\label{fig:zoom-2-compare}
\end{figure}

According to Theorem~\ref{thm:NearThePoles}, similar agreement with the same row of plots would be expected were one to suitably blow up neighborhoods of any of the other defects visible in the plots of $u_N(x,t)$ shown in Figures~\ref{fig:G(x)=-sech(x)--cos(u_8,u_16)}--\ref{fig:zoom-1}.  Indeed, given the impulse profile $G$ defining the fluxon condensate, the value $m=m_\mathrm{gc}$ becomes fixed, so the only thing that can differ in the limiting exact solution defect for each tritronqu\'ee pole is the value of the phase parameter $\Omega=\Omega_{\mathrm{p},N}$, which is asymptotically large, proportional to $N$.  The elliptic modulus parameter $m=m_\mathrm{gc}$ would however be different for different impulse profiles $G$ (also for different catastrophe points for the same impulse profile $G$, should any others exist).  

Finally, we observe that although our results concern fluxon condensates $\{u_N(x,t)\}_{N=1}^\infty$ that, while systematically constructed for general impulse profiles $G$ are only related to the solution of the Cauchy initial value problem \eqref{eq:SemiClassicalsG}--\eqref{eq:InitialConditionFG} by the asymptotic statement\footnote{Recall that the Satsuma-Yajima impulse profiles of the form \eqref{eq:sech-IC} are notable exceptions.  For these profiles, there is no approximation at all when $\epsilon=\epsilon_N$, i.e., the fluxon condensate exactly matches the given Cauchy data at $t=0$.} in Proposition~\ref{prop:IC-approximation}, direct numerical simulations of the solution of the Cauchy problem strongly suggest that similar results hold true in that setting as well.  See \cite{Lu2018} for several such numerical simulations.  

\subsection{Context and relation of results to other works}
The results described above may be considered as an analogue for the semiclassical sine-Gordon equation \eqref{eq:SemiClassicalsG} of corresponding results obtained by Bertola and Tovbis for the semiclassical focusing nonlinear Schr\"odinger equation 
\begin{equation}
\ii \epsilon q_t + \epsilon^2 q_{xx}+2|q|^2q=0.  
\label{eq:NLS}
\end{equation}
Concretely, one may compare \cite[Theorem 5.4]{BertolaTovbis2014} with Theorem~\ref{thm:AwayFromPoles} and \cite[Theorem 6.7]{BertolaTovbis2014} with Theorem~\ref{thm:NearThePoles}.  In particular, for $(x,t)$ near a gradient catastrophe point for the dispersionless limit of \eqref{eq:NLS}, a phase correction proportional to $\epsilon^{1/5}\mathrm{Re}\{\kappa h(\tau)\}$ modifies a finite-amplitude leading pre-breaking approximation of $q$ ($\kappa$ is a complex constant, and $\tau:\mathbb{R}^2\to\mathbb{C}$ is a real-linear map), provided that $\tau$ is bounded away from the poles of the tritronqu\'ee Hamiltonian\footnote{Bertola and Tovbis phrase their result in terms of one of the other four tritronqu\'ee solutions of a Painlev\'e-I equation written with a different normalization than \eqref{eq:PI-intro}.} $h$.  When $\tau$ is near a pole of $h$, there is a new leading term in the approximation of $q$, namely the famous Peregrine breather solution of \eqref{eq:NLS} with amplitude peak at $(X,T)=(0,0)$ in coordinates $X=(x-x_\mathrm{p})/\epsilon$ and $T=(t-t_\mathrm{p})/\epsilon$.  The latter is really a one-parameter family of solutions, parametrized by the finite limiting value of the pre-breaking approximation of the amplitude $|q(x,t)|$ as $(x,t)$ approaches the gradient catastrophe point.  The Peregrine solution is a model for rogue waves, and it exhibits a peak in amplitude at $(X,T)=(0,0)$ exactly three times that of the background value to which it decays as $(X,T)\to\infty$ in all directions.  Notably, the approximation of $q(x,t)$ near every pole of $h$ is given by exactly the same Peregrine solution once the amplitude at the catastrophe point is fixed.  By contrast, one sees from Theorem~\ref{thm:NearThePoles} that in the sine-Gordon problem the limiting solution is generally different for every complex-conjugate pair of poles of $h$ and for every $N$, with the difference entering via the phase parameter $\Omega=\Omega_{\mathrm{p},N}$.  This is a reflection of the fact that the background wave at the catastrophe point is a more complicated type of solution for sine-Gordon (genus $1$, built from elliptic functions) than for the nonlinear Schr\"odinger equation (genus $0$, built from elementary functions).  

The work of Bertola and Tovbis was motivated in part by a universality conjecture formulated by Dubrovin, Grava, and Klein \cite{DubrovinGravaKlein2009} predicting that the behavior of solutions of \eqref{eq:NLS} near a generic gradient catastrophe point of the dispersionless approximation should be independent of initial conditions.  Later, the same authors with Moro \cite{DubrovinGravaKleinMoro2015} extended the notion of universality to the setting of general dispersive perturbations of general elliptic $2\times 2$ quasilinear systems assumed without loss of generality to be in diagonal (Riemann-invariant) form.  Using formal Hamiltonian perturbation theory and the assumption of a solution of the unperturbed elliptic system exhibiting a generic (elliptic-umbilic) gradient catastrophe, the authors of \cite{DubrovinGravaKleinMoro2015} argued that the first correction term induced by the dispersion near the catastrophe point for the leading term should be proportional to the tritronqu\'ee solution $y(\tau)$ of the Painlev\'e-I equation \eqref{eq:PI-intro} in suitable local coordinates.  This universality prediction is therefore stronger, as it asserts that the same asymptotic behavior occurs regardless of both initial conditions \emph{and} the equation of motion.  

Our results, which connect multiscale asymptotics near a catastrophe point of the elliptic Whitham modulation equations \eqref{eq:Whitham-intro} --- a different elliptic system than the dispersionless form of \eqref{eq:NLS} --- with the tritronqu\'ee solution $y(\tau)$, therefore add rigorous evidence to the broader universality conjectures of \cite{DubrovinGravaKleinMoro2015}.  We hesitate to say that our results can be \emph{directly} compared with \cite[Conjecture 4.4]{DubrovinGravaKleinMoro2015} only because our starting point is the sine-Gordon equation \eqref{eq:SemiClassicalsG} itself, from which the corresponding elliptic quasilinear Whitham system \eqref{eq:Whitham-intro} is obtained only after a essential process of averaging over rapid oscillations.  Hence it is not clear to us how to express the exact sine-Gordon equation as a dispersive correction of its Whitham system.  The situation is different for the focusing nonlinear Schr\"odinger equation \eqref{eq:NLS}, which after a change of variables (the Madelung transform $q\mapsto (\rho,\mu):=(|q|^2,\epsilon\mathrm{Im}\{q_x\})$) takes precisely the form of a dispersive correction of an elliptic quasilinear system, \emph{without any averaging}.  

\subsection{Outline of the paper and discussion of techniques}
We begin in Section~\ref{sec:RHP} by making Definition~\ref{def:fluxon-condensate} more precise via the formulation of a Riemann--Hilbert problem given the phase integral $\Psi$ associated with an initial impulse profile $G$ as in \eqref{eq:Psi_Initial_Condition}.  We also introduce some basic deformations of this Riemann--Hilbert problem, in particular removing many pole singularities in favor of jumps along suitable contours.  Then, in Section~\ref{sec:steepest-descent} we use an appropriate $g$-function to stabilize the problem, which is then converted to a small-norm problem in the limit $\epsilon\to 0$ by comparison with a suitable parametrix.  This analysis is valid for $(x,t)$ in the modulated librational wave region, a notion that we define precisely.  We also define in Section~\ref{sec:steepest-descent} the notion of a simple gradient catastrophe point.  

The rest of the paper is concerned with the proofs of Theorem~\ref{thm:AwayFromPoles} (in Section~\ref{sec:AwayFromPoles}) and of Theorem~\ref{thm:NearThePoles} (in Section~\ref{sec:NearThePoles}).  While there are similarities between the multiscale steepest descent analysis in \cite{BertolaTovbis2014} and our proofs, we experience several new complications related to the fact that the background solution that is perturbed near the catastrophe point is associated with an elliptic curve (genus $1$) rather than a Riemann sphere (genus $0$).  We also take a different approach at several key points.  For instance, the gradient catastrophe of the Whitham system is mirrored in a kind of singularity in the $g$-function.  In \cite{BertolaTovbis2014} the singularity of the $g$-function is regularized by working in new local coordinates valid near the catastrophe point, however in our approach we simply modify the $g$-function in a way that unfolds the singularity, so that the modified $g$-function retains all of the desirable properties of the original for $(x,t)$ near $(0,t_\mathrm{gc})$ but has no singularity there at all.  Another difference appears in the way that a local parametrix based on the Painlev\'e-I tritronqu\'ee solution is modified near the poles of the latter.  In \cite{BertolaTovbis2014} the local parametrix is replaced with another one via a connection with a quartic oscillator equation, whereas in our approach a straightforward Schlesinger/Darboux transformation involving left multiplication by a linear function (the corresponding matrix factor in \cite{BertolaTovbis2014} appears to be multi-valued) solves this problem effectively; see Lemma \ref{lemma:Ttilde}.

The appendix of the paper contains proofs of some more technical lemmas that require details of function theory on elliptic curves, as well as a catalog of plots of the defect solutions.  

\subsection{Notation}
Throughout the paper, we use $\sigma_j$, $j=1,2,3$, to denote the Pauli matrices as:
\begin{equation}
\sigma_1:=\begin{bmatrix}0&1\\ 1&0\end{bmatrix},\quad\sigma_2:=\begin{bmatrix}0&-\ii\\ \ii&0\end{bmatrix},\quad\sigma_3:=\begin{bmatrix}1&0\\0&-1\end{bmatrix},
\label{eq:Pauli}
\end{equation}
and we define $\sigma_\pm$ as
\begin{equation}
	\sigma_+:=\begin{bmatrix}0&1\\0&0\end{bmatrix},\quad\sigma_-:=\begin{bmatrix}0&0\\1&0\end{bmatrix}.
\label{eq:sigmapm}
\end{equation}

Except for  these five matrices and the identity matrix $\Id$, all other matrices are denoted as bold capital letters such as $\bfA$.

\subsection{Acknowledgments}

The authors gratefully acknowledge helpful discussions with Marco Bertola, Tamara Grava, Liming Ling, among others. We also thank Marco Fasondini, J. A. C. Weideman and Bengt Fornberg for the numerical data of the Painlev\'e I tritronqu\'ee solution, and Robert Buckingham for the code that produces fluxon condensates for the sine-Gordon equation with $\mathrm{sech}$ initial data. Both authors were supported by the National Science Foundation on grants DMS-1206131 and DMS-1513054.  The second author was additionally supported by the same sponsor on grant number DMS-1812625.

\section{Riemann--Hilbert Problems for Fluxon Condensates}
\label{sec:RHP}
\subsection{A discrete Riemann--Hilbert problem for fluxon condensates}
Recalling the functions $E(w)$ and $D(w)$ defined by \eqref{eq:E-and-D-define},
let $Q(w)$ be defined as
\begin{equation}
Q(w)=Q(w;x,t):=E(w)x+D(w)t,\quad |\arg(-w)|<\pi.
\label{eq:Q-Definition}
\end{equation}
Next, recall the positive imaginary numbers $\lambda_0,\dots,\lambda_{N-1}$ determined by the Bohr-Sommerfeld quantization rule \eqref{eq:BohrSommerfeld}, and 
define the Blaschke product
\begin{equation}
\Pi_N(w):=\displaystyle{\prod_{k=0}^{N-1}\frac{E(w)+\lambda_{k}}{E(w)-\lambda_{k}}},\quad|\arg(-w)|<\pi.
\end{equation}
The approximate eigenvalues lie in the imaginary interval between $\lambda=0$ and $\lambda= -\tfrac{1}{4}\ii G(0)$.  Assuming $-2<G(0)<0$, each approximate eigenvalue is the image under $E$ of exactly two distinct complex-conjugate points on the unit circle, each of which is a potential singularity of $\Pi_N$.  Also, $E(w)$ takes no negative imaginary values, so $\Pi_N(w)$ has poles on the unit circle (at the preimages under $E$ of the approximate eigenvalues) but no zeros in the indicated domain.  It can further be shown under the indicated condition on $G(0)$ that all poles are simple.  In summary, $\Pi_N(w)$ has exactly $2N$ simple poles in complex-conjugate pairs on the unit circle in the $w$-plane, in particular they are confined to the arc of the unit circle joining $w=\ee^{\ii\mu}$ with $w=\ee^{-\ii\mu}$ via $w=1$, where $0<\mu<\pi$ and $E(\ee^{\pm\ii\mu})=-\tfrac{1}{4}\ii G(0)$ is a point on the positive imaginary axis lying below the critical value $\tfrac{1}{2}\ii$ of $E(\cdot)$.  We refer to the indicated arc of the unit circle as $P_\infty$ and to the finite set of poles of $\Pi_N(w)$ as $P_N\subset P_\infty$.  $\Pi_N(w)$ is analytic and nonvanishing for $w\in\mathbb{C}\setminus(P_N\cup\mathbb{R}_+)$.

The fluxon condensate associated with an impulse profile $G$ via the approximate eigenvalues $\lambda_0,\dots,\lambda_N$ is then encoded in the following Riemann--Hilbert problem (cf., \cite[Riemann--Hilbert Problem 2.1]{BuckinghamMiller2013}).
\begin{rhp}[Riemann--Hilbert problem for librational/breather fluxon condensates]
Let $N\in\mathbb{Z}_{>0}$ be given and let $\lambda_0,\dots,\lambda_N$ be the approximate eigenvalues associated with an impulse profile $G(\cdot)$ satisfying Assumption~\ref{assumption:G} and $-2<G(0)<0$.
Find a $2 \times 2$ matrix function $\bfH(w)=\bfH_N(w;x,t)$ that satisfies the following conditions:
\noindent
\begin{itemize}
\item[]\textbf{Analyticity:}
$\bfH(w)$ is analytic for $w\in\mathbb{C}\setminus(P_N\cup \mathbb{R}_+)$ and continuous up to $\mathbb{R}_+$ from both half-planes (so in particular $\mathbf{H}(0)$ is well-defined).
\item[]\textbf{Jump Condition:} The boundary values $\mathbf{H}_\pm(w)$ taken for $w>0$ from $\pm\mathrm{Im}\{w\}>0$ are related by the jump condition 
\begin{equation}
\bfH_+(w)=\sigma_2\bfH_-(w)\sigma_2,\quad w>0.
\end{equation}
Note that in particular the well-defined matrix $\mathbf{H}(0)$ satisfies $\mathbf{H}(0)=\sigma_2\mathbf{H}(0)\sigma_2$.
\item[]\textbf{Singularities:} Each of the points of $P_N$ is a simple pole of $\bfH(w)$. If $y\in P_N$ with $E(y)=\lambda_{k}$ for $k=0,\dots,N-1$ (for each $k$ the points $y$ form a complex-conjugate pair on the unit circle), then
\begin{equation}
\underset{w=y}{\mathrm{Res}}\bfH(w)=\lim_{w\to y}\bfH(w)\begin{bmatrix}
	0&0\\(-1)^{k+1}\underset{w=y}{\mathrm{Res}}\,\ee^{2\ii Q(w;x,t)/\epsilon}\Pi_N(w)&0
\end{bmatrix},\quad \epsilon=\epsilon_N.
\label{eq:basicRHP-singularities}
\end{equation}
\item[]\textbf{Normalization:} The following normalization condition holds:
\begin{equation}
	\lim_{w\to\infty}\bfH(w)=\Id,
\end{equation}
the limit being uniform with respect to direction, including parallel to $\mathbb{R}_+$.
\end{itemize}
\label{rhp:basicRHP}
\end{rhp}
One way to solve this Riemann--Hilbert problem is to make a suitable ansatz that is rational in $\sqrt{-w}$ and consistent with the jump and normalization conditions, with simple poles in $P_N$, and then enforce on the ansatz the conditions \eqref{eq:basicRHP-singularities}.  This practical approach results in a linear algebraic system of dimension proportional to $N$; however, it is not immediately clear whether the determinant of the system could possibly vanish for some or even all $(x,t)\in\mathbb{R}^2$.  A less practical approach that nonetheless establishes unique solvability for all $(x,t)\in\mathbb{R}^2$ is to appeal to the vanishing lemma of Zhou \cite{Zhou1989}.  To verify the conditions of the vanishing lemma one must first unfold the $w$-plane to the upper half-plane by the substitution $\mathbf{F}(\ii\sqrt{-w})=\mathbf{H}(w)$ and then fill in the lower half-plane by defining $\mathbf{F}(z)=\sigma_2\mathbf{F}(-z)\sigma_2$.  The resulting matrix $\mathbf{F}(z)$ has no jump along the real axis but has twice as many poles as $\mathbf{H}(w)$ had.  However the residue conditions inherited by $\mathbf{F}(z)$ have the necessary Schwarz symmetry with respect to reflection in the real axis to allow the vanishing lemma to be applied once the poles are removed by local disk substitutions.  From the uniqueness of the solution of Riemann--Hilbert Problem~\ref{rhp:basicRHP} it follows that $\mathbf{H}(w)$ satisfies the Schwarz symmetry condition
\begin{equation}
\mathbf{H}(w^*)=\mathbf{H}(w)^*,\quad (x,t)\in\mathbb{R}^2,\quad N\in\mathbb{Z}_{>0},
\label{eq:Schwarz-H}
\end{equation}
i.e., all four matrix elements of $\mathbf{H}(w)$ are Schwarz-symmetric functions of $w$.

It follows from the global existence of the solution for all $(x,t)\in\mathbb{R}^2$, via a dressing argument (see \cite[Proposition 2.1]{BuckinghamMiller2013}), that the function $u=u_N(x,t)$ defined modulo $4\pi\mathbb{Z}$ by 
\begin{equation}
\cos\left(\frac{1}{2}u_N(x,t)\right)=H_{N,11}(0;x,t)\quad\text{and}\quad
\sin\left(\frac{1}{2}u_N(x,t)\right)=H_{N,21}(0;x,t)
\label{eq:cos-sin-H0}
\end{equation}
is a global solution of the sine-Gordon equation in the form \eqref{eq:SemiClassicalsG} for $\epsilon=\epsilon_N$.  This is the precise meaning of the heuristic notion of fluxon condensates given in Definition~\ref{def:fluxon-condensate}, under the additional assumption $-2<G(0)<0$ which guarantees that the condensate consists of breathers only.  By following the approach given in \cite[the ``Aside'' beginning on p.\@ 970]{LyngMiller2007}, one can show that 
\begin{equation}
u_N(-x,t)=u_N(x,t),\quad \text{modulo $4\pi\mathbb{Z}$}.
\label{eq:condensate-even}
\end{equation}

If one would like to analyze the solution of Riemann--Hilbert Problem~\ref{rhp:basicRHP} in the limit $N\to\infty$ for general $(x,t)\in\mathbb{R}^2$, it turns out to be useful to first make certain substitutions rational in $\sqrt{-w}$ depending on the coordinates $(x,t)$; in particular one can select a subset $\Delta\subset P_N$ and aim to reverse the triangularity of the residue matrices in \eqref{eq:basicRHP-singularities} for poles $y\in\Delta$ while preserving the triangularity for poles $y\in\nabla:=P_N\setminus\Delta$.  This is potentially useful because when the triangularity is reversed, the sign of the exponent $2\ii Q(w;x,t)/\epsilon$ changes as well, so exponential growth can be converted into exponential decay.  However, it turns out that under the condition $-2<G(0)<0$ that we assume for the rest of this paper, with $x\ge 0$ and $t>0$ as is sufficient given \eqref{eq:condensate-even}, the configuration in which all residue matrices are lower triangular as written already in \eqref{eq:basicRHP-singularities} suffices.  Hence we will take $\Delta=\emptyset$ and $\nabla=P_N$ for readers familiar with the notation in \cite{BuckinghamMiller2013}.  Equivalently, we will simply take Riemann--Hilbert Problem~\ref{rhp:basicRHP} without modification as the starting point for all of our analysis.

\subsection{Interpolation of residues and removal of poles}
Let $\theta_0(w)$ denote the composition of $\Psi(\lambda)$ with $\lambda=E(w)$:
\begin{equation}
	\theta_0(w):=\Psi(E(w)).
\label{eq:theta0_Definition}
\end{equation}
Under Assumption~\ref{assumption:Psi} this function is analytic in the $\mathbb{C}_\pm$ parts of a sufficiently large domain containing the arc $P_\infty$, but it has a jump discontinuity across $\mathbb{R}_+$ near $w=1$ inherited from $E(\cdot)$.  We suppose that domains $\Omega_\pm$, $\Omega_-=\Omega_+^*$ such as are illustrated in Figure~\ref{fig:OmegaPlusMinus} are contained within this domain of analyticity.
\begin{figure}[h]
\begin{center}
\includegraphics{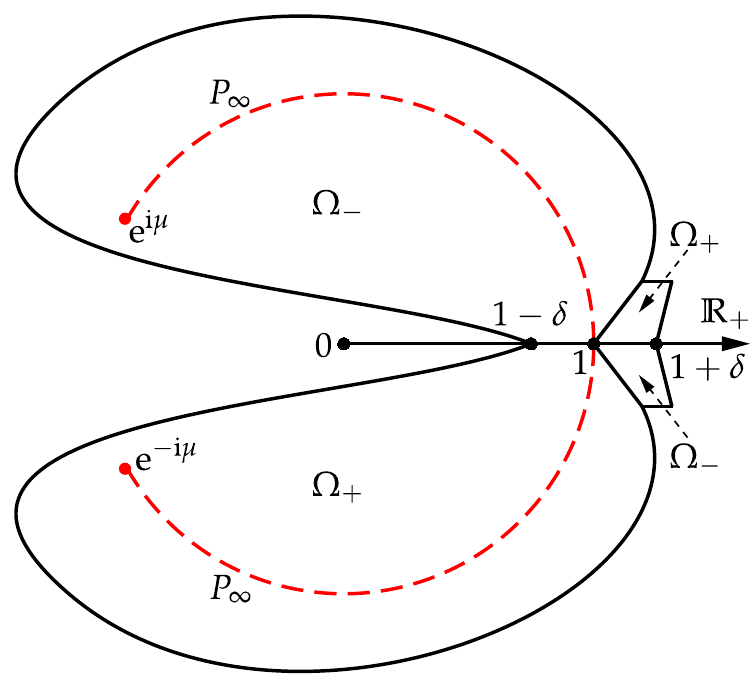}%
\includegraphics{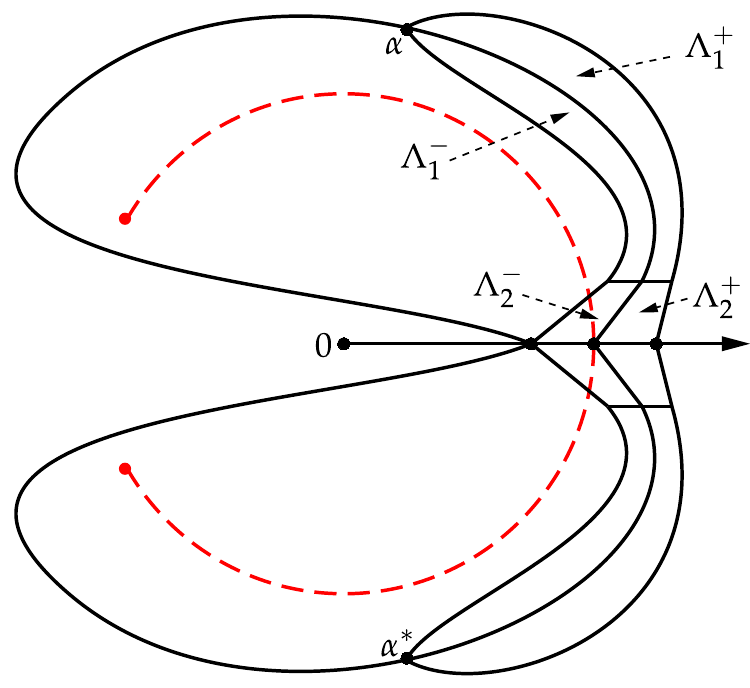}
\end{center}
\caption{Left panel:  the domains $\Omega_+$ and $\Omega_-=\Omega_+^*$ and their relation with the arc $P_\infty$ in which the poles of the solution $\mathbf{H}(w)$ of Riemann--Hilbert Problem~\ref{rhp:basicRHP} are located.  Here $\delta>0$ is a sufficiently small positive number.  Note that $\Omega_-$ (respectively, $\Omega_+$) contains the part of $P_\infty$ in the upper (respectively, lower) half-plane.  Right panel:  the lens domains $\Lambda_1^\pm$ and $\Lambda_2^\pm$ in $\mathbb{C}_+$ and their images under Schwarz reflection in $\mathbb{C}_-$.  Note that $\Lambda_2^+=\Omega_+\cap\mathbb{C}_+$ while $\Lambda_2^-$ necessarily contains points on both sides of $P_\infty$.  Also, $\Lambda_1^\pm$ lie outside of the unit circle.}
\label{fig:OmegaPlusMinus}
\end{figure}
Consider the following definition
(cf., \cite[Eqn. (3.6), lines 1 and 3]{BuckinghamMiller2013}) relative to the domains $\Omega_\pm$ 
\begin{equation}
	\bfM(w):=\begin{cases}
		\bfH(w)\begin{bmatrix}
		1&0\\
		\mp\ii\Pi_N(w)\ee^{[2\ii Q(w;x,t)\pm\ii\theta_0(w)]/\epsilon}&1
		\end{bmatrix}, & w\in\Omega_\pm,\quad \epsilon=\epsilon_N,\\
		\bfH(w), &w\in\mathbb{C}\backslash(\overline{\Omega}\cup\mathbb{R}_+),\quad\Omega:=\Omega_+\cup\Omega_-.
	\end{cases}
	\label{eq:H-to-M}
\end{equation}
It is a consequence of the Bohr-Sommerfeld quantization rule \eqref{eq:BohrSommerfeld} defining the locations of the poles of $\Pi_N$ that $\mathbf{M}(w)=\mathbf{M}_N(w;x,t)$ has only removable singularities at these poles, and hence can be considered to be a matrix-valued analytic function of $w\in\mathbb{C}\setminus(\partial\Omega_+\cup\partial\Omega_-\cup\mathbb{R}_+)$, i.e., $\mathbf{M}(w)$ is analytic in the complement of the solid black contour illustrated in the left-hand panel of Figure~\ref{fig:OmegaPlusMinus}.  Moreover, the matrix function $\mathbf{M}(w)$ inherits the Schwarz symmetry of $\mathbf{H}(w)$ in the form \eqref{eq:Schwarz-H}.

We take the two arcs of $P_\infty$ in the upper and lower half-planes to be oriented toward  $w=1$, and define the analytic function
\begin{equation}
L(w):=\frac{\sqrt{-w}}{\pi}\int_{P_\infty}\frac{\theta_0(y)\,\dd y}{\sqrt{-y}(y-w)},\quad w\in\mathbb{C}\setminus(P_\infty\cup\mathbb{R}_+).
\label{eq:L-function-define}
\end{equation}
For $w\in P_\infty$ with $w\neq 1$, the average of the distinct boundary values taken by $L(w)$ at $w$ is denoted
\begin{equation}
\overline{L}(w):=\frac{1}{2}\left(L_+(w)+L_-(w)\right),\quad w\in P_\infty,\quad w\neq 1.
\label{eq:Lbar-function-define}
\end{equation}
Since the boundary values of $L(w)$ on $P_\infty$ are related by
\begin{equation}
L_+(w)-L_-(w)=2\ii\theta_0(w),\quad w\in P_\infty,
\label{eq:L-function-difference}
\end{equation}
(according to the Plemelj formula) and $2\ii\theta_0(w)$ is analytic on $P_\infty$ with $w\neq 1$, it follows that $\overline{L}(w)$ is analytic where defined as well.  Related functions $Y_N(w)$ and $T_N(w)$ are then given by the following definitions:
\begin{equation}
Y_N(w):=\Pi_N(w)\ee^{-L(w)/\epsilon}\quad\text{and}\quad T_N(w):=2\Pi_N(w)\cos(\epsilon^{-1}\theta_0(w))\ee^{-\overline{L}(w)/\epsilon},\quad\epsilon=\epsilon_N,
\label{eq:Y-and-T-define}
\end{equation}
and these functions are analytic wherever all factors on the right-hand side are defined in each case.  Actually, the Bohr-Sommerfeld quantization rule \eqref{eq:BohrSommerfeld} implies that the poles of $\Pi_N(w)$ are all removable singularities for $T_N(w)$.  Now, taking the curve $(\partial\Omega_-\setminus\partial\Omega_+)\cap\mathbb{C}_+$ to be oriented in the direction away from the point $w=1-\delta$, the jump condition across this arc satisfied by the matrix $\mathbf{M}(w)$ can be written in terms of $Y_N(w)$ as
\begin{equation}
\mathbf{M}_+(w)=\mathbf{M}_-(w)\begin{bmatrix}1 & 0\\-\ii Y_N(w)\ee^{k(w)/\epsilon} & 1\end{bmatrix},\quad w\in (\partial\Omega_-\setminus\partial\Omega_+)\cap\mathbb{C}_+,\quad\epsilon=\epsilon_N,
\label{eq:Jump-M-away}
\end{equation}
in which we have introduced an exponent function $k(w)$ defined by
\begin{equation}
k(w):=2\ii Q(w)+L(w)-\ii\theta_0(w).
\label{eq:k-function-define}
\end{equation}
The jump of $\mathbf{M}(w)$ across $(\partial\Omega_+\setminus\partial\Omega_-)\cap\mathbb{C}_+$ can also be written in terms of $Y_N(w)$ in a similar way.
On the other hand, letting the arc $\partial\Omega_+\cap\partial\Omega_-\cap\mathbb{C}_+$ be oriented toward $w=1$, with the help of \eqref{eq:Lbar-function-define}--\eqref{eq:L-function-difference} (and the fact that the indicated contour lies to the left of $P_\infty$) we obtain the jump condition
\begin{equation}
\mathbf{M}_+(w)=\mathbf{M}_-(w)\begin{bmatrix}1 & 0\\-\ii T_N(w)\ee^{k(w)/\epsilon} & 1\end{bmatrix},\quad w\in\partial\Omega_+\cap\partial\Omega_-\cap\mathbb{C}_+,\quad\epsilon=\epsilon_N.
\label{eq:Jump-M-near}
\end{equation}

The appearance of $Y_N(w)$ and $T_N(w)$ in the jump conditions essentially packages some factors proportional to $\Pi_N(w)$ that can be effectively analyzed for large $N$ by interpreting $\Pi_N(w)$ as the exponential of a Riemann sum.  The details of this analysis are not important here, and they can be found in \cite{BaikKriecherbauerMcLaughlinMiller2007}.  However, we will make use of the following simplified version of \cite[Proposition 3.1]{BuckinghamMiller2013}.
\begin{proposition}
Under Assumption~\ref{assumption:Psi}, $Y_N(w)$ is analytic for $w\in\mathbb{C}\setminus(P_\infty\cup\mathbb{R}_+$) and if $A_1$ is an annulus centered at $w=1$ with sufficiently small outer radius and arbitrarily small inner radius, then $T_N(w)$ is analytic for $w\in A_1\setminus\mathbb{R}$.  Furthermore, for $\epsilon=\epsilon_N$, $Y_N(w)=1+\mathcal{O}(\epsilon)$ holds uniformly for $w\in\mathbb{C}\setminus\mathbb{R}_+$ bounded away from $P_\infty$ and $T_N(w)=1+\mathcal{O}(\epsilon)$ holds uniformly for $w\in A_1\setminus\mathbb{R}$ as $N\to\infty$.
\label{prop:T-and-Y}
\end{proposition}
See Figure~\ref{fig:YN-TN-Regions} for an illustration.
\begin{figure}[h]
\begin{center}
\includegraphics{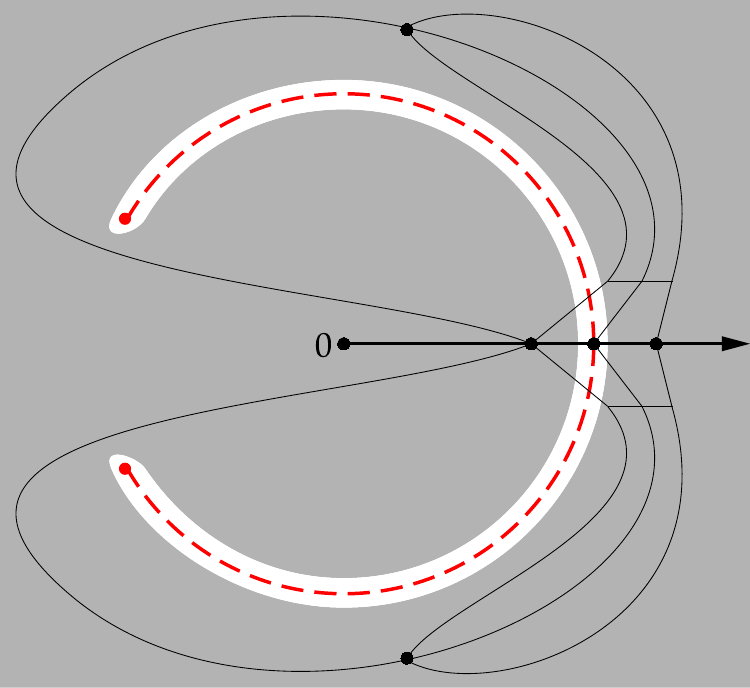}%
\includegraphics{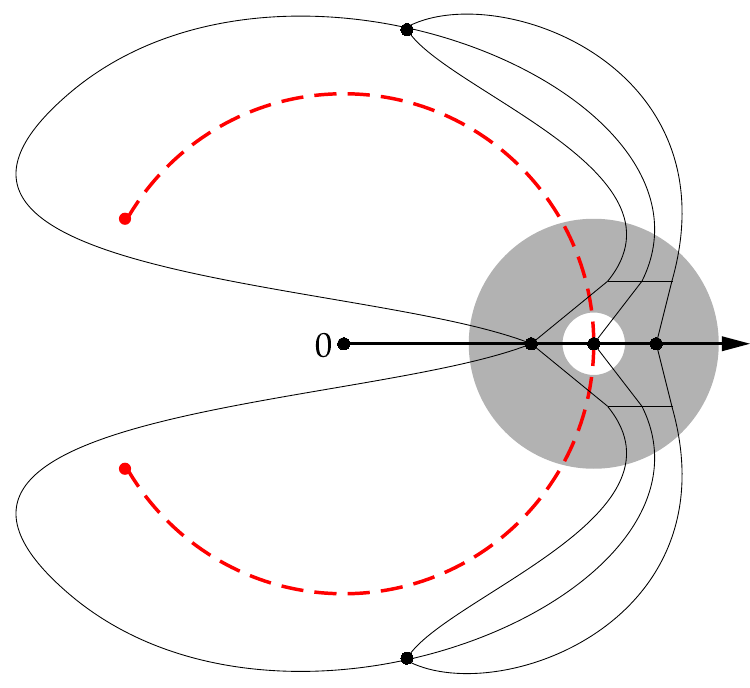}
\end{center}
\caption{Left:  $Y_N(w)$ is analytic and $Y_N(w)=1+\mathcal{O}(\epsilon)$ holds uniformly in the shaded region.  Right:  $T_N(w)$ is analytic and $T_N(w)=1+\mathcal{O}(\epsilon)$ holds uniformly in the shaded region.}
\label{fig:YN-TN-Regions}
\end{figure}

\subsection{Opening lenses on the jump contour}
The jump matrix in \eqref{eq:Jump-M-away} admits the factorization
\begin{multline}
\begin{bmatrix}1 & 0\\-\ii Y_N(w)\ee^{k(w)/\epsilon} & 1\end{bmatrix}\\{}=
Y_N(w)^{-\sigma_3/2}\begin{bmatrix}1 & \ii\ee^{-k(w)/\epsilon}\\0 & 1\end{bmatrix}\begin{bmatrix}0 & -\ii\ee^{-k(w)/\epsilon} \\ -\ii\ee^{k(w)/\epsilon} & 0\end{bmatrix}\begin{bmatrix}1 & \ii\ee^{-k(w)/\epsilon}\\0 & 1\end{bmatrix}Y_N(w)^{\sigma_3/2},
\label{eq:Jump-M-away-factorization}
\end{multline}
in which, according to Proposition~\ref{prop:T-and-Y}, the diagonal factors $Y_N(w)^{\pm\sigma_3/2}$ involving square roots of $Y_N(w)$ can be interpreted as $\mathbb{I}+\mathcal{O}(\epsilon)$ for $w\in (\partial\Omega_-\setminus\partial\Omega_+)\cap\mathbb{C}_+$ as the latter is bounded away from $P_\infty$.
Likewise, the jump matrix in \eqref{eq:Jump-M-near} admits the factorization
\begin{multline}
\begin{bmatrix}1 & 0\\-\ii T_N(w)\ee^{k(w)/\epsilon} & 1\end{bmatrix}\\{}=
T_N(w)^{-\sigma_3/2}\begin{bmatrix}1 & \ii\ee^{-k(w)/\epsilon}\\0 & 1\end{bmatrix}\begin{bmatrix}0 & -\ii\ee^{-k(w)/\epsilon} \\ -\ii\ee^{k(w)/\epsilon} & 0\end{bmatrix}\begin{bmatrix}1 & \ii\ee^{-k(w)/\epsilon}\\0 & 1\end{bmatrix}T_N(w)^{\sigma_3/2},
\label{eq:Jump-M-near-factorization}
\end{multline}
in which Proposition~\ref{prop:T-and-Y} allows us to interpret $T_N(w)^{\pm\sigma_3/2}=\mathbb{I}+\mathcal{O}(\epsilon)$ for $w\in\partial\Omega_+\cap\partial\Omega_-\cap\mathbb{C}_+$.  Based on these factorizations, we pick a point $\alpha\in(\partial\Omega_-\setminus\partial\Omega_+)\cap\mathbb{C}_+$ and ``open up lenses'' by making further substitutions in the domains $\Lambda_1^\pm$ and $\Lambda_2^\pm$ shown in the right-hand panel of Figure~\ref{fig:OmegaPlusMinus} (and their complex-conjugates) to separate the factors in \eqref{eq:Jump-M-away-factorization}--\eqref{eq:Jump-M-near-factorization}.
Specifically, we define $\mathbf{N}(w)=\mathbf{N}_N(w;x,t)$ in terms of $\mathbf{M}(w)$ by setting
\begin{equation}
\mathbf{N}(w):=\mathbf{M}(w)Y_N(w)^{-\sigma_3/2}\begin{bmatrix}1 & \mp\ii\ee^{-k(w)/\epsilon}\\0 & 1\end{bmatrix},\quad w\in \Lambda_1^\pm,
\label{eq:lens-1}
\end{equation}
\begin{equation}
\mathbf{N}(w):=\mathbf{M}(w)T_N(w)^{-\sigma_3/2}\begin{bmatrix}1 & -\ii\ee^{-k(w)/\epsilon}\\0 & 1\end{bmatrix},\quad w\in\Lambda_2^+,\quad\text{and}
\label{eq:lens-2-plus}
\end{equation}
\begin{equation}
\mathbf{N}(w):=\mathbf{M}(w)T_N(w)^{-\sigma_3/2}\begin{bmatrix}1 & \ii\ee^{-\hat{k}(w)/\epsilon}\\0&1
\end{bmatrix},\quad w\in\Lambda_2^-,
\label{eq:lens-2-minus}
\end{equation}
in which $\hat{k}(w)$ denotes the analytic function in $\Lambda_2^-$ that agrees with $k(w)$ outside of the unit circle.  Thus:
\begin{equation}
\hat{k}(w):=\begin{cases} k(w),&\quad w\in\Lambda_2^-,\quad |w|>1\\
k(w)+2\ii\theta_0(w),&\quad w\in\Lambda_2^-,\quad |w|<1
\end{cases}
\label{eq:tilde-k-define}
\end{equation}
(one sees from \eqref{eq:L-function-difference} and \eqref{eq:k-function-define} that the two formul\ae\ agree upon taking boundary values on the unit circle).  Then we set
\begin{equation}
\mathbf{N}(w):=\mathbf{M}(w), \quad\text{elsewhere in $\mathbb{C}_+$,}
\label{eq:M-to-N-elsewhere}
\end{equation}
and then we make corresponding substitutions in $\mathbb{C}_-$ to ensure that Schwarz reflection symmetry is preserved in the form $\mathbf{N}(w^*)=\mathbf{N}(w)^*$.  The matrix function $\mathbf{N}(w)$ so-defined is analytic for $w$ in the complement of the jump contour shown in Figure~\ref{fig:JumpContourN}.  Note that at this juncture, the value of $\alpha$ and other details of the jump contour are not meant to be fully specified; a more precise description that in particular determines the value of $\alpha$ will be given in Section~\ref{sec:steepest-descent} below.
\begin{figure}[h]
\begin{center}
\includegraphics{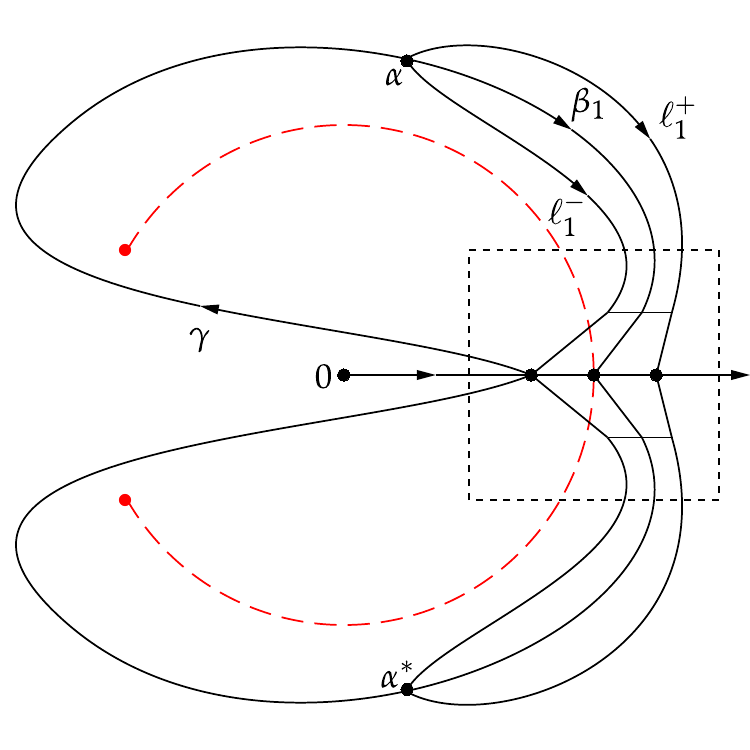}%
\includegraphics{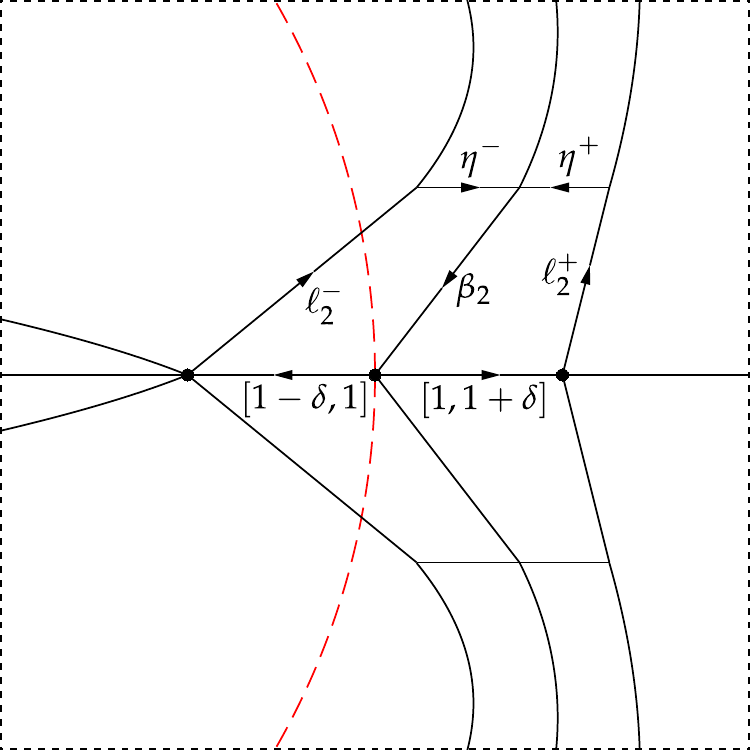}
\end{center}
\caption{Left panel:  the jump contour for $\mathbf{N}(w)$.  Right panel:  a blow up of the jump contour near $w=1$.}
\label{fig:JumpContourN}
\end{figure}

The jump conditions satisfied by $\mathbf{N}(w)$ on the contour arcs that meet at $w=\alpha$ read as follows:
\begin{equation}
\mathbf{N}_+(w)=\mathbf{N}_-(w)\begin{bmatrix}1 & 0\\-\ii Y_N(w)\ee^{k(w)/\epsilon} & 1\end{bmatrix},\quad w\in\gamma,
\end{equation}
\begin{equation}
\mathbf{N}_+(w)=\mathbf{N}_-(w)\begin{bmatrix}1 & \ii\ee^{-k(w)/\epsilon}\\0 & 1\end{bmatrix}Y_N(w)^{\sigma_3/2},\quad w\in\ell_1^+,
\end{equation}
\begin{equation}
\mathbf{N}_+(w)=\mathbf{N}_-(w)Y_N(w)^{-\sigma_3/2}\begin{bmatrix}1 & \ii\ee^{-k(w)/\epsilon}\\0 & 1\end{bmatrix},\quad w\in\ell_1^-,
\end{equation}
\begin{equation}
\mathbf{N}_+(w)=\mathbf{N}_-(w)\begin{bmatrix}0& -\ii\ee^{-k(w)/\epsilon}\\-\ii\ee^{k(w)/\epsilon} & 0\end{bmatrix},\quad w\in \beta_1.
\end{equation}
The other jump conditions satisfied by $\mathbf{N}(w)$ in the upper half-plane are as follows:
\begin{equation}
\mathbf{N}_+(w)=\mathbf{N}_-(w)\begin{bmatrix}1&-\ii\ee^{-\hat{k}(w)/\epsilon}\\0&1\end{bmatrix}T_N(w)^{\sigma_3/2},\quad w\in\ell_2^-,
\end{equation}
\begin{equation}
\mathbf{N}_+(w)=\mathbf{N}_-(w)\begin{bmatrix}0 & -\ii\ee^{-k(w)/\epsilon}\\
-\ii\ee^{k(w)/\epsilon} & 0\end{bmatrix},\quad w\in\beta_2,
\end{equation}
\begin{equation}
\mathbf{N}_+(w)=\mathbf{N}_-(w)\begin{bmatrix}1&0\\-\ii Y_N(w)\ee^{k(w)/\epsilon}\ee^{2\ii\theta_0(w)/\epsilon} & 1\end{bmatrix}T_N(w)^{-\sigma_3/2}\begin{bmatrix}1 & -\ii\ee^{-k(w)/\epsilon}\\0 & 1\end{bmatrix},\quad w\in\ell_2^+,
\end{equation}
\begin{multline}
\mathbf{N}_+(w)=\mathbf{N}_-(w)\begin{bmatrix}T_N(w)^{1/2}Y_N(w)^{-1/2} & \ii(T_N(w)^{1/2}Y_N(w)^{-1/2}-T_N(w)^{-1/2}Y_N(w)^{1/2})\ee^{-k(w)/\epsilon}\\0 & T_N(w)^{-1/2}Y_N(w)^{1/2}\end{bmatrix},\\ w\in \eta^-,
\end{multline}
and
\begin{equation}
\mathbf{N}_+(w)=\mathbf{N}_-(w)\begin{bmatrix}T_N(w)^{1/2}Y_N(w)^{-1/2} & 0\\-\ii T_N(w)^{-1/2}Y_N(w)^{1/2}\ee^{k(w)/\epsilon}\ee^{2\ii\theta_0(w)/\epsilon} & T_N(w)^{-1/2}Y_N(w)^{1/2}\end{bmatrix},\quad w\in \eta^+.
\end{equation}
In computing the latter two jump matrices, we made use of the fact that $T_N(w)=Y_N(w)(1+\ee^{2\ii\theta_0(w)/\epsilon})$ holds for $\epsilon=\epsilon_N$ and $w\in\eta^\pm$ because the latter two contours lie outside of the unit circle where $\overline{L}(w)=L(w)-\ii\theta_0(w)$ holds according to \eqref{eq:Lbar-function-define}--\eqref{eq:L-function-difference}.

\section{Pre-Catastrophe Analysis}
\label{sec:steepest-descent}

\subsection{Introduction of $g$-function}
\label{sec:intro-of-g}
For convenience, let $\beta$ denote the union $\beta=\beta_1\cup\beta_2$ of arcs in the upper half-plane, and let $\Sigma$ denote the jump contour for $\mathbf{N}(w)$ as illustrated in the left-hand panel of Figure~\ref{fig:JumpContourN}.  Suppose that $g(w)$ is a $\epsilon$-independent scalar function analytic for $w\in\mathbb{C}\setminus (\beta\cup\beta^*\cup\mathbb{R}_+)$ and continuous up to $\beta\cup\beta^*\cup\mathbb{R}_+$, enjoying Schwarz symmetry in the form
$g(w^*)=g(w)^*$, and satisfying the jump conditions 
\begin{itemize}
\item $g_+(w)+g_-(w)=0$ for $w>0$.  
\item $g_+(w)+g_-(w)=k(w)-\ii\Phi$ for $w\in\beta$, where $\Phi$ is a real constant (with respect to $w$, but it will necessarily depend on the parameters $(x,t)$).  
\end{itemize}
Set
\begin{equation}
\phi(w):=k(w)-2g(w),
\label{eq:phi-define}
\end{equation}
a function that we assume is analytic on some neighborhood of the jump contour $\Sigma$,  except on the cuts of $g$ and $L$ (this is an assumption about $\theta_0$ only, see also Assumption \ref{assumption:Psi}).  We set 
\begin{equation}
\mathbf{O}(w):=\mathbf{N}(w)\ee^{-g(w)\sigma_3/\epsilon}.
\label{eq:N-to-O}
\end{equation}
Then $\mathbf{O}(w)$ is analytic exactly where $\mathbf{N}(w)$ is, i.e., for $w\in\mathbb{C}\setminus\Sigma$, and the jump of $g$ on $\beta\cup\beta^*\cup\mathbb{R}_+$ will alter the jump conditions for $\mathbf{N}(w)$.  Assuming also that $g(w)\to 0$ as $w\to\infty$, we therefore see that $\mathbf{O}(w)=\mathbf{O}(w^*)^*$ is the solution of the following Riemann--Hilbert problem. 
\begin{rhp}[Riemann--Hilbert problem for $\bfO$] Seek $\bfO:\mathbb{C}\setminus\Sigma\to\mathbb{C}^{2\times 2}$ 
with the following properties:
\begin{itemize}
	\item[]\textbf{Analyticity:} 
		$\mathbf{O}(w)$ is analytic for $w\in\mathbb{C}\setminus\Sigma$
		and continuous up to $\Sigma$ from both sides.
	\item[]\textbf{Jump Conditions:}  Recalling that the subscript ``$+$'' (resp.\@ ``$-$'') indicates a boundary value taken on an oriented arc from the left (resp.\@ right), the boundary values of $\mathbf{O}(w)$ are related across the arcs of $\Sigma$ as follows:
\begin{equation}
			\bfO_+(w)=\bfO_-(w)\begin{bmatrix}
				1 & \ii\ee^{-\phi(w)/\epsilon} \\ 0&1
			\end{bmatrix}Y_N(w)^{\sigma_3/2}, \quad w\in \ell_1^+, 
\label{eq:Ojump-First}
\end{equation}
\begin{equation}
			\bfO_+(w)=\bfO_-(w)Y_N(w)^{-\sigma_3/2}\begin{bmatrix}
				1 & \ii\ee^{-\phi(w)/\epsilon} \\ 0&1
			\end{bmatrix},\quad w\in \ell_1^-,
\end{equation}
\begin{equation}
			\bfO_+(w)=\bfO_-(w)\begin{bmatrix}
				0&-\ii\ee^{-\ii\Phi/\epsilon}\\-\ii\ee^{\ii\Phi/\epsilon}&0
			\end{bmatrix},\quad w\in\beta,
\label{eq:Ojump-beta}
\end{equation}
\begin{equation}
			\bfO_+(w)=\bfO_-(w)\begin{bmatrix}
				1&0\\-\ii Y_N(w)\ee^{\phi(w)/\epsilon}&1
			\end{bmatrix},\quad w\in\gamma,
\label{eq:Ojump-gamma}
\end{equation}
\begin{equation}
			\mathbf{O}_+(w)=\mathbf{O}_-(w)\begin{bmatrix}
				1&-\ii\ee^{-\hat{\phi}(w)/\epsilon}\\0&1
			\end{bmatrix}T_N(w)^{\sigma_3/2},\quad w\in\ell_2^-,
\end{equation}
where in analogy with \eqref{eq:tilde-k-define},
\begin{equation}
\hat{\phi}(w):=\begin{cases}\phi(w),&\quad w\in\ell_2^-,\quad |w|>1\\
\phi(w)+2\ii\theta_0(w),&\quad w\in\ell_2^-,\quad |w|<1,
\end{cases}
\label{eq:PreCatastrophe-tilde-phi-define}
\end{equation}
\begin{equation}
			\mathbf{O}_+(w)=\mathbf{O}_-(w)\begin{bmatrix}
				1&0\\-\ii Y_N(w)\ee^{\phi(w)/\epsilon}\ee^{2\ii\theta_0(w)/\epsilon} & 1
			\end{bmatrix}T_N(w)^{-\sigma_3/2}\begin{bmatrix}
				1 & -\ii\ee^{-\phi(w)/\epsilon}\\0 & 1
			\end{bmatrix},\quad w\in\ell_2^+,
\end{equation}
\begin{multline}
			\mathbf{O}_+(w)=\mathbf{O}_-(w)\cdot \\
			{}\cdot \begin{bmatrix}
				T_N(w)^{1/2}Y_N(w)^{-1/2} & \ii(T_N(w)^{1/2}Y_N(w)^{-1/2}-T_N(w)^{-1/2}Y_N(w)^{1/2})\ee^{-\phi(w)/\epsilon}\\0 & T_N(w)^{-1/2}Y_N(w)^{1/2}
				\end{bmatrix},\\  w\in \eta^-,
\end{multline}
\begin{equation}
			\mathbf{O}_+(w)=\mathbf{O}_-(w)\begin{bmatrix}
				T_N(w)^{1/2}Y_N(w)^{-1/2} & 0\\-\ii T_N(w)^{-1/2}Y_N(w)^{1/2}\ee^{\phi(w)/\epsilon}\ee^{2\ii\theta_0(w)/\epsilon} & T_N(w)^{-1/2}Y_N(w)^{1/2}\end{bmatrix},\quad w\in \eta^+,
\label{eq:Ojump-mid}
\end{equation}
\begin{equation}
			\bfO_+(w)=\sigma_2\bfO_-(w)\sigma_2,\quad w\in\mathbb{R}_+\setminus[1-\delta,1+\delta],
\end{equation}
\begin{multline}
				\mathbf{O}_+(w)=\sigma_2\mathbf{O}_-(w)\sigma_2\cdot\ee^{-g_-(w)\sigma_3/\epsilon}\begin{bmatrix}1&0\\\ii\ee^{-\hat{k}_-(w)/\epsilon} & 1\end{bmatrix} T_{N-}(w)^{-\sigma_3/2}\begin{bmatrix}1 & \ii\Pi_{N-}(w)\ee^{[2\ii Q_-(w)-\ii\theta_{0-}(w)]/\epsilon}\\0&1\end{bmatrix}\cdot\\
	{}\cdot 
				\begin{bmatrix}1&0\\-\ii\Pi_{N-}(w)^*\ee^{[-2\ii Q_-(w)^*+\ii\theta_{0-}(w)^*]/\epsilon} & 1\end{bmatrix}
			T_{N-}(w)^{*-\sigma_3/2}\begin{bmatrix}1&-\ii\ee^{-\hat{k}_-(w)^*/\epsilon}\\0&1\end{bmatrix}\ee^{-g_-(w)^*\sigma_3/\epsilon}, \\ w\in(1-\delta,1),
\label{eq:Ojump-on-(1-delta,1)}
\end{multline}
\begin{multline}
			\bfO_+(w)=\sigma_2\bfO_-(w)\sigma_2\cdot\ee^{-g_+(w)^*\sigma_3/\epsilon}\begin{bmatrix}
				1&0\\ \ii\ee^{k_+(w)^*/\epsilon}&1
			\end{bmatrix}{T_{N+}(w)^*}^{-\sigma_3/2}
			\begin{bmatrix}
				1&\ii\Pi_{N+}(w)^* \ee^{[-2 \ii Q_+(w)^*-\ii\theta_{0+}(w)^*]/\epsilon}
				\\ 0&1
			\end{bmatrix}\cdot\\ 
			{}\cdot \begin{bmatrix}
				1&0\\-\ii\Pi_{N+}(w)\ee^{[2\ii Q_+(w)+\ii\theta_{0+}(w)]/\epsilon}&1
			\end{bmatrix}T_{N+}(w)^{-\sigma_3/2}\begin{bmatrix}
				1&-\ii\ee^{k_+(w)/\epsilon}\\0&1
			\end{bmatrix}\ee^{-g_+(w)\sigma_3/\epsilon},\\ w\in (1,1+\delta).
\label{eq:Ojump-Last}
\end{multline}
The jump conditions on the arcs of $\Sigma$ in the open lower half-plane are consistent with \eqref{eq:Ojump-First}--\eqref{eq:Ojump-mid} under Schwarz symmetry:  $\mathbf{O}(w^*)=\mathbf{O}(w)^*$.		
	\item[]\textbf{Normalization:}
		\begin{equation}
			\lim_{w\to\infty} \mathbf{O}(w)=\mathbb{I}.
		\end{equation}
\end{itemize}
\label{rhp:O}
\end{rhp}

\subsection{Expression for $g$ and determination of $\alpha$}
\label{sec:g-pre-breaking}
We now show that the function $g$ (and the value of $\alpha$) are determined systematically by the basic conditions listed at the beginning of Section~\ref{sec:intro-of-g}, which essentially constitute a scalar Riemann--Hilbert problem.  It turns out that the analytical behavior of the function $\phi(w)$ related to the solution $g(w)$ of this problem by \eqref{eq:phi-define} influences strongly the nature of the jump conditions in Riemann--Hilbert Problem~\ref{rhp:O}.  In particular, for certain $(x,t)\in\mathbb{R}^2$, namely those points in the modulated librational wave region to be defined properly below, we can reduce Riemann--Hilbert Problem~\ref{rhp:O} to a small-norm problem by comparison with a suitable parametrix.  For other values of $(x,t)\in\mathbb{R}^2$ such a simplification is not possible without generalizing the conditions satisfied by $g$ in a substantial way, and this dichotomy is responsible for the observed phase transition in solutions of the sine-Gordon equation near a gradient catastrophe point.  We will illustrate this transition phenomenon clearly in Section~\ref{sec:boundary} below.

In light of \eqref{eq:condensate-even}, let us assume that $x>0$ and $t>0$.  It turns out that to deal with the conditions on $g$ imposed at the beginning of Section~\ref{sec:intro-of-g} it is more convenient to first seek $g'(w)$ and later integrate to obtain $g(w)$.
Indeed, differentiation with respect to $w$ annihilates the constant $\ii\Phi$, so the sum of the boundary values of $g'(w)$ on the jump contour $\beta\cup\beta^*\cup\mathbb{R}_+$ are explicitly given (by $0$ on $\mathbb{R}_+$, by $k'(w)$ on $\beta$, and by $k'(w)=k'(w^*)^*$ on $\beta^*$).  Moreover, $g'(w)$ must be integrable at $w=\infty$.  These conditions imply that the function $g'(w)$ necessarily has the form
\begin{equation}
g'(w)=\frac{R(w)}{2\pi\ii \sqrt{-w}}\int_{\beta\cup\beta^*}\frac{k'(\xi)\sqrt{-\xi}}{R_+(\xi)(\xi-w)}\,\dd \xi,
\label{eq:gprime-1}
\end{equation}
where $R(w)^2:=R(w;\alpha,\alpha^*)^2=(w-\alpha)(w-\alpha^*)$ and $R(w)$ is analytic for $w\in\mathbb{C}\setminus\beta$ and $R(w)=w+\mathcal{O}(1)$ as $w\to\infty$.
Indeed this form automatically guarantees that $g'_+(w)+g'_-(w)=0$ for $w\in\mathbb{R}_+$ and that $g'_+(w)+g'_-(w)=k'(w)$ for $w\in\beta\cup\beta^*$.  To enforce integrability of $g'(w)$ at $w=\infty$ given that $R(w)\sim w$ as $w\to\infty$, we need to insist further that
\begin{equation}
\int_{\beta\cup\beta^*}\frac{k'(\xi)\sqrt{-\xi}}{R_+(\xi)}\,\dd\xi=0.
\label{eq:M-basic}
\end{equation}
Using contour deformation arguments along with \eqref{eq:k-function-define}, the definition of $Q(w)$ \eqref{eq:Q-Definition}, and the identity \eqref{eq:L-function-difference}, this condition can eventually be written in the form
\begin{equation}
M:=\frac{4}{\pi}\int_{\tilde{\gamma}\cup\tilde{\gamma}^*}\frac{\theta_0'(\xi)\sqrt{-\xi}}{R(\xi)}\,\dd\xi + x+t+\frac{x-t}{\sqrt{\alpha\alpha^*}}=0,
\label{eq:M}
\end{equation}
where $\tilde{\gamma}$ is an oriented arc from $\ee^{\ii\mu}$ to $\alpha$ in $\mathbb{C}\setminus(\beta\cup\beta^*\cup\mathbb{R}_+)$.  Indeed, the left-hand side of \eqref{eq:M-basic} is just $-\tfrac{1}{4}\ii\pi M$.  Using similar deformation ideas, one can show that for $w$ near $\beta$ (in particular $|w|>1$),
\begin{equation}
\phi'(w):= k'(w)-2g'(w)=R(w)H(w),\quad H(w):=-\frac{1}{4\sqrt{-w}}\left[\frac{x-t}{w\sqrt{\alpha\alpha^*}}-\frac{2}{\pi}\int_{L}\frac{\theta_0'(\xi)\sqrt{-\xi}}{R(\xi)(\xi-w)}\,\dd\xi\right],
\label{eq:RH}
\end{equation}
where $L$ is a contour such as is shown in the left-hand panel of Figure~\ref{fig:L-Lbeta} that encloses $w$ with $|w|>1$.  
\begin{figure}[h]
\begin{center}
\includegraphics{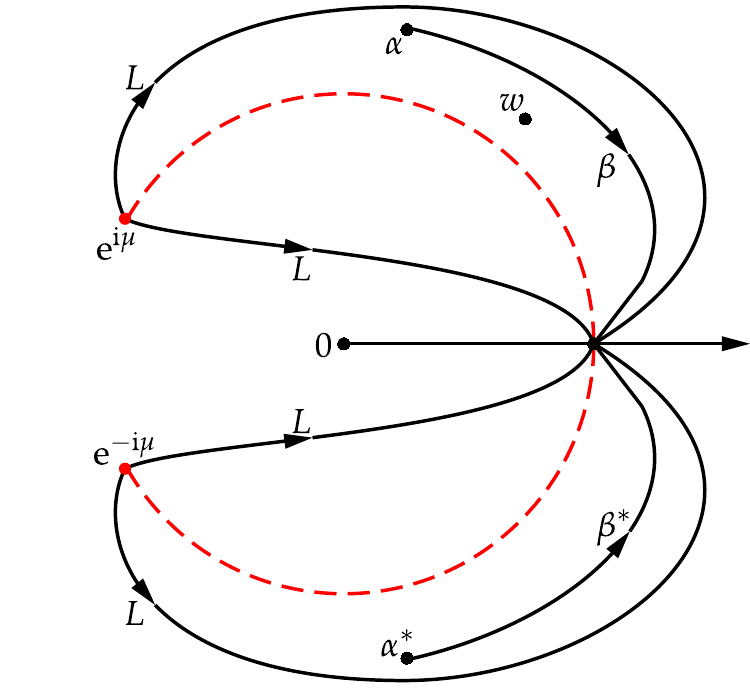}%
\includegraphics{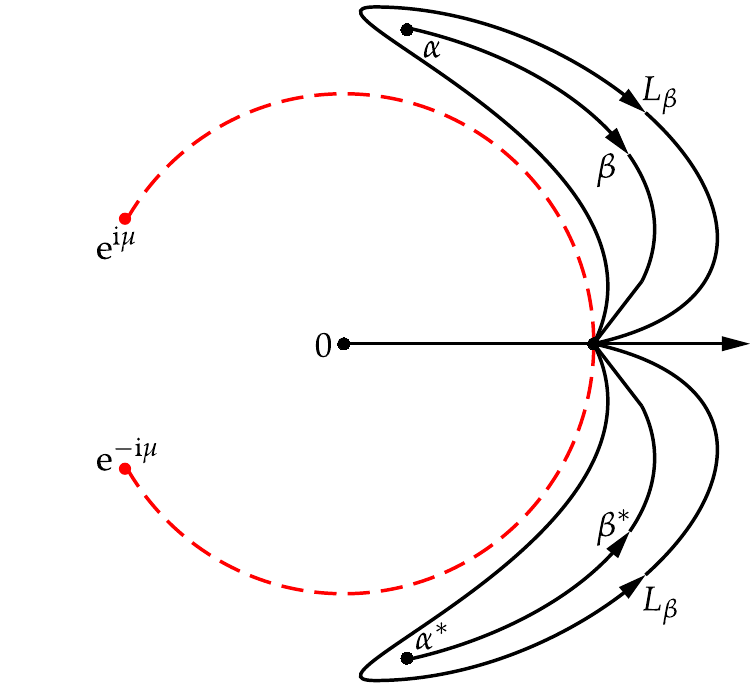}
\end{center}
\caption{Left:  the contour $L$ encloses $P_\infty$, $\beta\cup\beta^*$, and the pole $w$ with $|w|>1$ of the integrand for $H(w)$.  Right:  the contour $L_\beta$ encloses $\beta\cup\beta^*$.}
\label{fig:L-Lbeta}
\end{figure}
When $\alpha^*$ is the complex-conjugate of $\alpha$ and $x$ and $t$ are real, $H(w)$ is a Schwarz-symmetric function analytic except on the contour $L$ of integration and the positive real axis.  With $g'(w)$ given by \eqref{eq:gprime-1}, $g(w)$ is obtained by integration from $w=0$:
\begin{equation}
g(w)=\int_0^w g'(\xi)\,\dd\xi
\label{eq:g-integral}
\end{equation}
which in view of the identity $g'_+(w)+g'_-(w)=0$ for $w>0$ guarantees that also $g_+(w)+g_-(w)=0$ for $0<w<1$.  To extend this latter identity to $w>1$ and simultaneously guarantee that the constant $\Phi$ is real (when $\alpha^*$ is the conjugate of $\alpha$ and $x$ and $t$ are real) we insist that 
\begin{equation}
\int_{L_\beta}g'(w)\,\dd w=0
\label{eq:gprime-on-Lbeta}
\end{equation}
where $L_\beta$ is the contour shown in the right-hand panel of Figure~\ref{fig:L-Lbeta}.
Since $k(w)$ is analytic on $\beta\cup\beta^*$, it is the same to insist that
\begin{equation}
\int_{L_\beta}\phi'(w)\,\dd w=0.
\label{eq:Cbetaphiprime}
\end{equation}
Using \eqref{eq:RH}, this condition can be written in the form
\begin{equation}
I:=\frac{1}{2}\int_{\beta}R_+(\xi)H(\xi)\,\dd\xi +\frac{1}{2}\int_{\beta^*}R_-(\xi)H(\xi)\,\dd\xi = 0.
\label{eq:I}
\end{equation}
Indeed, the left-hand side of \eqref{eq:Cbetaphiprime} is just $4I$.  While the condition \eqref{eq:I} ensures that $g_+(w)+g_-(w)=0$ holds for $w>1$, it also guarantees that $\Phi$ is real when $\alpha^*$ is the conjugate of $\alpha$ and $x$ and $t$ are real.  We omit the argument but it involves checking the value of $\ii\Phi=k(w)-g_+(w)-g_-(w)$ on $\beta$ by evaluating it in the limit along the indicated arc as $w\to 1$ and using \eqref{eq:I} along with known properties of $k(w)$. It is easy to confirm that the condition \eqref{eq:gprime-on-Lbeta} enforced by \eqref{eq:Cbetaphiprime} or equivalently \eqref{eq:I} also guarantees that $g(\infty)=0$ (by splitting the integral of $g'(w)$ along the negative real axis into two equal parts that are closed at infinity in the upper and lower half-planes respectively, and using $g'_+(w)+g'_-(w)=0$ for $w>0$).

The conditions \eqref{eq:M} and \eqref{eq:I} determine $\alpha$ and $\alpha^*$ in terms of $(x,t)\in\mathbb{R}^2$.  In more detail, following \cite[Section 4.2]{BuckinghamMiller2013} one first shows that in the limit $t\downarrow 0$ for $x\ge 0$ fixed, the conditions are satisfied by
$\alpha=\alpha(x,0)=\ee^{\ii\eta_\alpha(x)}$ and $\alpha^*=\alpha^*(x,0)=\ee^{-\ii\eta_\alpha(x)}$ where $\cos(\eta_\alpha(x))=1-\tfrac{1}{2}G(x)^2$ and $\sin(\eta_\alpha(x))>0$ (recall that $G(\cdot)$ is the initial impulse profile, see \eqref{eq:InitialConditionFG}).  By an implicit function theorem argument, the solution can be continued to nearby $(x,t)\in\mathbb{R}^2$, yielding $\alpha=\alpha(x,t)$ and $\alpha^*=\alpha^*(x,t)$ and hence also $g(w)=g(w;x,t)$ and $H(w)=H(w;x,t)$, provided that $\alpha\neq\ee^{\ii\mu}$ and $H(\alpha(x,t);x,t)\neq 0$ (cf., \cite[Proposition 4.6]{BuckinghamMiller2013} or \cite[Proposition 3.1.5]{Lu2018}).  This construction implies that $\alpha^*(x,t)=\alpha(x,t)^*$ for all $(x,t)\in\mathbb{R}^2$ to which the solution can be continued.  Note that once $g(w)=g(w;x,t)$ is determined, then so is the real quantity $\Phi=\Phi(x,t)$, which is a real-analytic function of $(x,t)\in\mathbb{R}^2$ provided $H(\alpha(x,t);x,t)\neq 0$.  One can also check that for small $t>0$ one has $|\alpha|=|\alpha^*|>1$ when $x\ge 0$.  

\subsection{The modulated librational wave region}
\label{sec:modulated-librational-wave-region}
The introduction of the function $g(w)=g(w;x,t)$ is the most useful in the setting of Riemann--Hilbert Problem~\ref{rhp:O} if certain additional conditions are satisfied, which we formalize in the following definition, in which the functions $\phi(w)=\phi(w;x,t)$ and $H(w)=H(w;x,t)$ are given by \eqref{eq:phi-define} and \eqref{eq:RH} respectively.
\begin{definition}[The modulated librational wave region]
Suppose that $\alpha=\alpha(x,t)$ and $\alpha^*=\alpha^*(x,t)$ are chosen so that the equations \eqref{eq:M} and \eqref{eq:I} both hold, and let $g(w)=g(w;x,t)$ be given by \eqref{eq:gprime-1} and \eqref{eq:g-integral}.  We say that $(x,t)$ belongs to the modulated librational wave region of $\mathbb{R}^2$ if in addition the following conditions hold.
\begin{enumerate}[(i)]
	\item
		The oriented arc $\beta$ in $\mathbb{C}_+$ connecting the point $w=\alpha(x,t)$ to the point $w=1$, on which the definition of $g(w;x,t)$ depends, can be chosen such that the boundary values of $\phi(w;x,t)$ taken on $\beta$ satisfy $\mathrm{Re}\{\phi_\pm(w;x,t)\}=0$, and so that the angle $\theta_\beta$ between the ray $w>1$ and the tangent line in $\mathbb{C}_+$ to $\beta$ at $w=1$ satisfies $0<\theta_\beta<\tfrac{1}{2}\pi$.
		\label{item:MLW-beta}
	\item
		With $\beta$ chosen as above, $H(w;x,t)$ is bounded away from zero for $w\in\beta$.
		\label{item:MLW-H-not-zero}
	\item
		For sufficiently small $\delta>0$, there is an oriented arc $\gamma$ in $\mathbb{C}_+$ connecting the point $w=1-\delta$  to the point $w=\alpha(x,t)$  on which $\mathrm{Re}\{\phi(w;x,t)\}<0$ holds and such that $\gamma\cup\beta$ forms a loop enclosing the part of $P_\infty$ in the upper half-plane. 
		\label{item:MLW-gamma}
\end{enumerate}
\label{def:MLW}
\end{definition}
For $x> 0$, if $t>0$ is small and hence $|\alpha(x,t)|>1$, then one can show that $(x,t)$ lies in the modulated librational wave region with the arc $\beta$ lying outside the unit circle except at its terminal point $w=1$.  This explains why $\beta$ follows $\gamma$ in the clockwise direction of orientation of the loop referred to in condition (\ref{item:MLW-gamma}); see Figure~\ref{fig:JumpContourN}.  Also, note that local analysis near $w=1$ shows that $\mathrm{Re}\{\phi(w+\ii 0;x,t)\}>0$ holds for $w-1$ sufficiently small and positive, while $\mathrm{Re}\{\phi(w+\ii 0;x,t)\}<0$ holds for $w-1$ sufficiently small and negative.  This immediately proves that if $(x,t)$ lies in the modulated librational wave region, then 
the inequality $\mathrm{Re}\{\phi(w;x,t)\}>0$ holds on both sides of the arc $\beta$, and $\mathrm{Re}\{\phi(w;x,t)\}<0$ holds near $w=1-\delta$.  Also, since $H(\alpha(x,t);x,t)\neq 0$ according to condition (\ref{item:MLW-H-not-zero}), the formula \eqref{eq:RH} shows that in addition to $\beta$, exactly two more zero level curves of $\mathrm{Re}\{\phi(w;x,t)\}$ emanate from $w=\alpha(x,t)$, between which the inequality $\mathrm{Re}\{\phi(w;x,t)\}<0$ holds.  Therefore condition (\ref{item:MLW-gamma}) automatically holds locally near the endpoints of $\gamma$.

Note that the conditions of Definition~\ref{def:MLW} are independent of the semiclassical parameter $\epsilon=\epsilon_N$.  Whether or not a given point $(x,t)\in\mathbb{R}^2$ with $x\ge 0$ belongs to the modulated librational wave region can therefore be detected by sufficiently well-resolved $\epsilon$-independent numerical calculations.  The idea is to first implement the continuation of the solution $(\alpha,\alpha^*)$ of the equations \eqref{eq:M} and \eqref{eq:I} from the known initial conditions at $t=0$ to determine $\alpha=\alpha(x,t)$ with $|\alpha|>1$.  With the help of the formula \eqref{eq:RH} for $\phi'(w;x,t)$ and a suitable numerical implementation of the multivalued function $R(w)$ one then computes the level curve of $\mathrm{Re}\{\phi(w;x,t)\}$ emanating into the upper half-plane from $w=1$ and checks whether it approaches the point $w=\alpha(x,t)$ to desired accuracy.  If so, since $\phi(\alpha(x,t);x,t)=\ii\Phi(x,t)\in\ii\mathbb{R}$, then the latter curve, re-oriented toward $w=1$, becomes the arc $\beta$; with a check of the tangent angle $\theta_\beta$ condition (\ref{item:MLW-beta}) in Definition~\ref{def:MLW} is thus confirmed.    The numerical construction of $\beta$ automatically fails if $H(w;x,t)$ vanishes at any point along $\beta$ except possibly for the endpoint $w=\alpha(x,t)$; however it is straightforward to evaluate $H(\alpha(x,t);x,t)$ numerically and in this way one can also confirm condition (\ref{item:MLW-H-not-zero}) from Definition~\ref{def:MLW}.  It remains to check condition (\ref{item:MLW-gamma}).  This requires determining the region of the $w$-plane on which the inequality $\mathrm{Re}\{\phi(w;x,t)\}<0$ holds.  There are numerous reliable ways to carry out this systematic numerical computation, and plotting the resulting sign chart in the $w$-plane allows one to determine whether or not condition (\ref{item:MLW-gamma}) holds, and therefore also whether or not $(x,t)$ lies within the modulated librational wave region.  See Figures~\ref{fig:w-plane-plots-x0p2}--\ref{fig:w-plane-plots-x0}.
\begin{figure}[h]
\begin{center}
\includegraphics{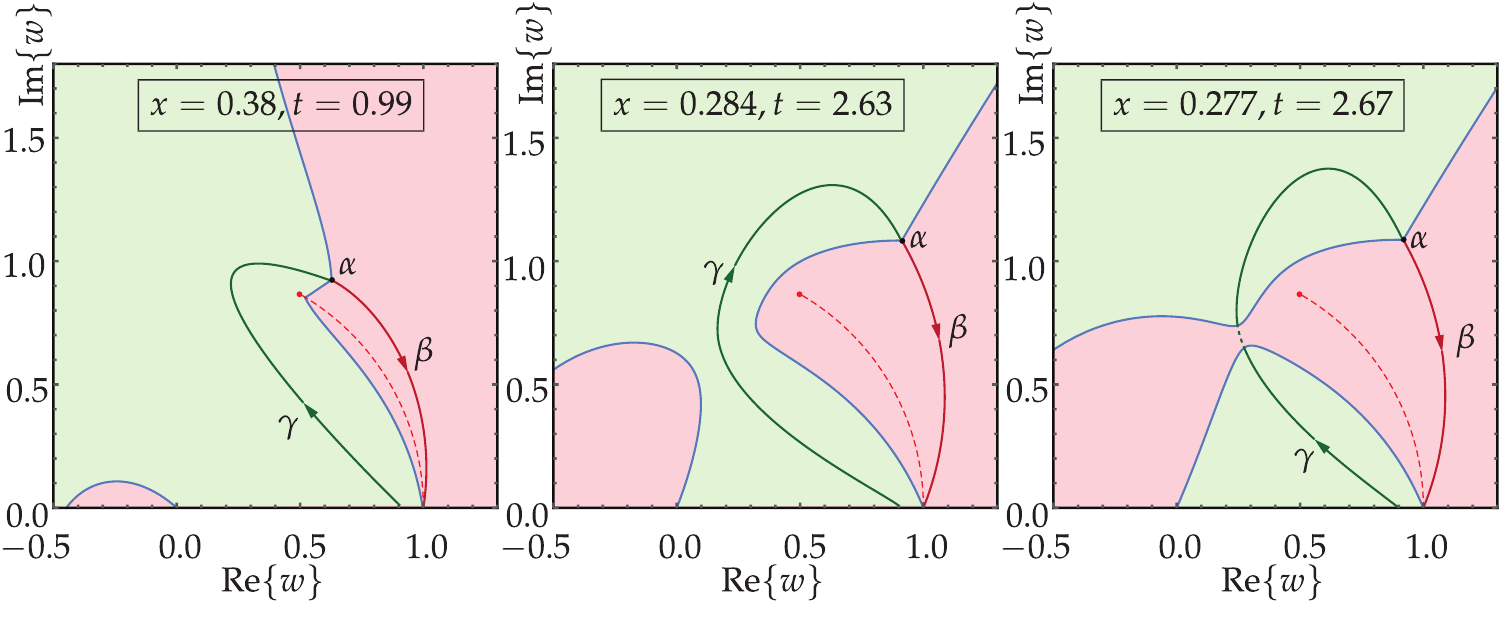}
\end{center}
\caption{Numerically constructed sign charts of $\mathrm{Re}\{\phi(w;x,t)\}$, $w\in\mathbb{C}_+$, for $G(x)=-\mathrm{sech}(x)$ (green for negative, red for positive) and how $\gamma$ can be chosen for $(x,t)$ in the modulated librational wave region.  Left:  $(x,t)=(0.38,0.99)$, in the modulated librational wave region.  Center:  $(x,t)=(0.284,2.63)$, in the modulated librational wave region.  Right:  $(x,t)=(0.277,2.67)$, no longer in the modulated librational wave region (so there can be no $\gamma$ satisfying condition (\ref{item:MLW-gamma}).  
In all plots, the zero level curve of $\mathrm{Re}\{\phi(w;x,t)\}$ includes the arc $\beta$ across which $\mathrm{Re}\{\phi(w;x,t)\}$ is continuous but exhibits a jump discontinuity in its gradient.  Also, the only part of each plot that is artificially superimposed and not numerically calculated from the formula for $g(w;x,t)$ is the green curve representing a choice of $\gamma$ (except in the right-hand panel, where no such choice is possible as suggested by the dotted green arc).}
\label{fig:w-plane-plots-x0p2}
\end{figure}
\begin{figure}[h]
\begin{center}
\includegraphics{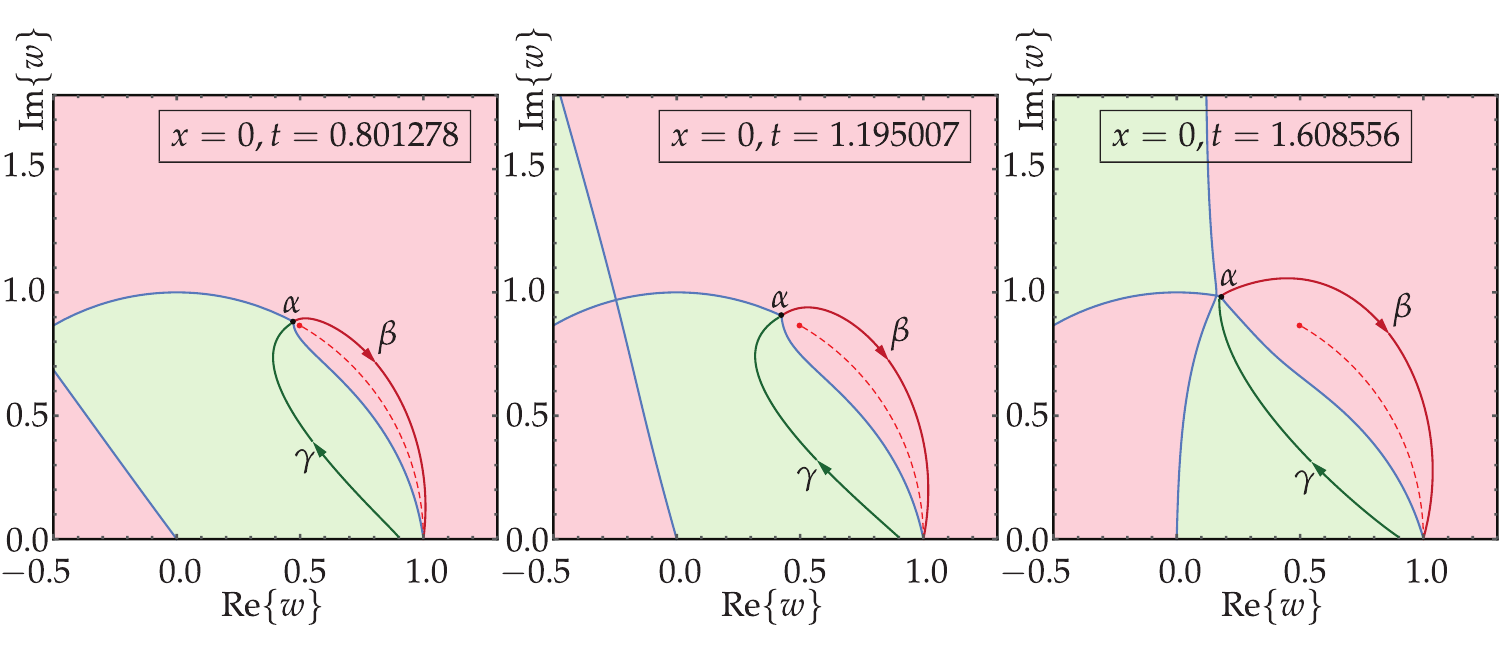}
\end{center}
\caption{As in Figure~\ref{fig:w-plane-plots-x0p2} but for three values of $(x,t)$ with $x=0$ and $t$ increasing toward the gradient catastrophe point $t_\mathrm{gc}\approx 1.609104$ (so for the right-hand plot $t_\mathrm{gc}-t\approx 5\times 10^{-4}$).  In these plots there is a persistent critical point of $\phi(w)$ on the zero level curve of its real part, but it occurs on the branch of the zero level curve emanating from $w=\alpha$ to the left of $\beta$ and hence does not obstruct the placement of the arc $\gamma$ in the green-shaded region.}
\label{fig:w-plane-plots-x0}
\end{figure}
These figures illustrate the fact that, at least for some initial impulse profiles $G(x)$, all points $(x,t)$ with $t>0$ sufficiently small lie in the modulated librational wave region.  Figure~\ref{fig:w-plane-plots-x0} also illustrates the phenomenon that the effect of approaching the gradient catastrophe point is that a critical point of $\phi(w)$ collides with the endpoint $\alpha$ leading to five zero level curves of $\mathrm{Re}\{\phi(w)\}$ emanating from $w=\alpha$ rather than three as is the generic case.  The plots in Figure~\ref{fig:w-plane-plots-x0} also illustrate that when $x=0$ the point $\alpha$ appears to lie exactly on the unit circle along with one of the zero level curves of $\mathrm{Re}\{\phi(w)\}$.  This confinement will be rigorously established in Section~\ref{sec:x=0-Symmetry} below.

\subsection{Proof of Proposition~\ref{prop:BeforeCatastrophe}}
\label{sec:BeforeCatastropheProof}
The statements \eqref{eq:Whitham-intro}--\eqref{eq:omega-k} can be deduced from the equations \eqref{eq:M} and \eqref{eq:I} by cross-differentiation, exactly as in \cite[Proposition 4.2 and Section 4.4]{BuckinghamMiller2013}.  It remains to establish the asymptotic formul\ae\ \eqref{eq:cos-sin-before}.
The proof of these formul\ae\ requires three steps:  showing that certain jump conditions for $\mathbf{O}(w)$ are asymptotically trivial when $\epsilon=\epsilon_N$ is small, construction of a suitable parametrix to deal with the jump conditions for $\mathbf{O}(w)$ that are not asymptotically trivial, and finally the comparison of $\mathbf{O}(w)$ with its parametrix.
\subsubsection{Asymptotically trivial jump conditions for $\mathbf{O}(w)$}
We first show that, since $(x,t)$ lies in the modulated librational wave region, when $\epsilon=\epsilon_N$ is small the jump conditions enumerated in Riemann--Hilbert Problem~\ref{rhp:O} are nearly trivial except when $w\in\beta\cup\beta^*$ and in neighborhoods of $w=\alpha(x,t)$ and $w=\alpha(x,t)^*$.

First note that the sign chart of $\mathrm{Re}\{\phi(w;x,t)\}$ in relation to the arcs $\gamma$ and $\beta$ as implied by $(x,t)$ being in the modulated librational wave region and the fact that the arcs $\ell_1^\pm$ can be taken to lie as close to $\beta_1\subset\beta$ as necessary together imply that the factors $\ee^{\pm\phi(w)/\epsilon}$ appearing in the jump conditions for $w\in\gamma\cup\ell_1^+\cup\ell_1^-$ are exponentially small as $\epsilon\to 0$, uniformly so as long as $w$ is bounded away from $w=\alpha(x,t)$.  Therefore invoking Proposition~\ref{prop:T-and-Y} one sees easily that $\mathbf{O}_+(w)=\mathbf{O}_-(w)(\mathbb{I}+\mathcal{O}(\epsilon))$ holds on these three arcs, uniformly for $w$ bounded away from $w=\alpha(x,t)$ (in fact, the error term is exponentially small for $w\in\gamma$).  

To deal with the jump conditions of $\mathbf{O}(w)$ across the arcs $\ell^\pm_2$ and $\eta^\pm$, we first compute the limiting value of $\phi'(w)=\phi'(w;x,t)$ as $w$ approaches $w=1$ from the upper half-plane to the left of $\beta$ by its orientation.  Using the formula \eqref{eq:RH}, we prepare to take the indicated limit by 
first contracting the contour $L$ pictured in the left-hand panel of Figure~\ref{fig:L-Lbeta} to both sides of $\beta\cup\beta^*$ (these contributions then cancel because $R(\xi)$ changes sign across $\beta\cup\beta^*$) plus a double contribution from the arc $\tilde{\gamma}$ (connecting $w=\ee^{\ii\mu}$ with $w=\alpha(x,t)$) and $\tilde{\gamma}^*$.  This deformation requires taking into account a residue at the point $\xi=w$ in the integral over the outer arc of $L$ in the upper half-plane (because $w$ lies between it and $\beta$ as we prepare to take the limit).  Thus we obtain for such $w$ that
\begin{equation}
\phi'(w)=-\frac{R(w)}{4\sqrt{-w}}\left[\frac{x-t}{w\sqrt{\alpha\alpha^*}}-\frac{4}{\pi}\int_{\tilde{\gamma}\cup\tilde{\gamma}^*}\frac{\theta_0'(\xi)\sqrt{-\xi}}{R(\xi)(\xi-w)}\,\dd\xi\right]-\ii\theta_0'(w).
\end{equation}
Letting $w\to 1$ from the upper half plane to the left of $\beta$ so that $R(w)\to |1-\alpha(x,t)|$ and $\sqrt{-w}\to -\ii$, and (by the chain rule) $\theta_0'(w)\to \tfrac{1}{4}\Psi'(0)$ we see that
\begin{equation}
\phi'(w)\to -\ii\frac{|1-\alpha|}{4}\left[\frac{x-t}{|\alpha|}-\frac{4}{\pi}\int_{\tilde{\gamma}\cup\tilde{\gamma}^*}\frac{\theta_0'(\xi)\sqrt{-\xi}}{R(\xi)(\xi-1)}\,\dd\xi\right] -\frac{1}{4}\ii\Psi'(0).
\end{equation}
Since $\tfrac{1}{4}\ii\Psi'(0)<0$ according to Assumption~\ref{assumption:Psi} and since the quantity in square brackets is purely real, we have found that in the indicated limit, $\phi'(w)\to d$ where $\mathrm{Re}\{d\}>0$.  Property (\ref{item:MLW-beta}) in Definition \ref{def:MLW} then implies that also $\mathrm{Im}\{d\}>0$.
It follows that in the same limit,
\begin{equation}
\phi'(w)+2\ii\theta_0'(w)\to -\ii\frac{|1-\alpha|}{4}\left[\frac{x-t}{|\alpha|}-\frac{4}{\pi}\int_{\tilde{\gamma}\cup\tilde{\gamma}^*}\frac{\theta_0'(\xi)\sqrt{-\xi}}{R(\xi)(\xi-1)}\,\dd\xi\right] +\frac{1}{4}\ii\Psi'(0) = -d^*.
\end{equation}
Since from \eqref{eq:RH} $\phi'(w)$ changes sign across $\beta$, the function $\hat{\phi}'(w)$ (analytic in $\Lambda_2^-$ by the piecewise definition \eqref{eq:PreCatastrophe-tilde-phi-define}) takes the limiting value $-d$ as $w\to 1$ from $\Lambda_2^-$.  
In each case the corresponding limiting values of $\phi(w)$, $\phi(w)+2\ii\theta_0(w)$, and $\hat{\phi}(w)$ are all purely imaginary.  Hence, the above first-order information implies that if the arcs $\ell_2^\pm$ and $\eta^\pm$ (see the right-hand panel of Figure~\ref{fig:JumpContourN}) are taken to be sufficiently short (independent of $\epsilon$), which also requires choosing $\delta>0$ sufficiently small, then as $\epsilon=\epsilon_N\to 0$,
\begin{itemize}
\item $\ee^{-\hat{\phi}(w)/\epsilon}$ is uniformly exponentially small for $w\in\ell_2^-$.
\item $\ee^{-\phi(w)/\epsilon}$ and $\ee^{(\phi(w)+2\ii\theta_0(w))/\epsilon}$ are both uniformly exponentially small for $w\in\ell_2^+$.
\item $|\ee^{-\phi(w)/\epsilon}|\le 1$ holds for $w\in\eta^-$.
\item $\ee^{(\phi(w)+2\ii\theta_0(w))/\epsilon}$ is uniformly exponentially small for $w\in\eta^+$.
\end{itemize}
Therefore, invoking Proposition~\ref{prop:T-and-Y} shows that $\mathbf{O}_+(w)=\mathbf{O}_-(w)(\mathbb{I}+\mathcal{O}(\epsilon))$ holds uniformly for $w\in\ell_2^+\cup\ell_2^-\cup\eta^+\cup\eta^-$.

The jump conditions \eqref{eq:Ojump-on-(1-delta,1)}--\eqref{eq:Ojump-Last} for the real intervals $(1-\delta,1)$ and $(1,1+\delta)$ are also nearly trivial when $\epsilon=\epsilon_N$ is small, in the sense that $\mathbf{O}_+(w)=\sigma_2\mathbf{O}_-(w)\sigma_2(\mathbb{I}+\mathcal{O}(\epsilon))$.  Multiplying out the matrix factors in \eqref{eq:Ojump-on-(1-delta,1)} we arrive at
\begin{equation}
\mathbf{O}_+(w)=\sigma_2\mathbf{O}_-(w)\sigma_2\begin{bmatrix}V_{11}(w) & V_{12}(w)\\V_{21}(w) & V_{22}(W)\end{bmatrix},\quad w\in (1-\delta,1),
\end{equation}
where $V_{11}(w)V_{22}(w)-V_{12}(w)V_{21}(w)=1$ and
\begin{equation}
\begin{split}
V_{11}(w)&:=
|T_{N-}(w)|^{-1}(1+\ee^{\ii[\theta_{0-}(w)^*-\theta_{0-}(w)]/\epsilon}) \\
V_{12}(w)&:=\ii\ee^{-2g_-(w)/\epsilon}|T_{N-}(w)|^{-1}\\
&{}\quad\quad\cdot\left[
\Pi_{N-}(w)T_{N-}(w)^*\ee^{[2\ii Q_-(w)-\ii\theta_{0-}(w)]/\epsilon}-\ee^{-\hat{k}_-(w)^*/\epsilon}(1+\ee^{\ii[\theta_{0-}(w)^*-\theta_{0-}(w)]/\epsilon})\right],
\end{split}
\end{equation}
and $V_{21}(w)=V_{12}(w)^*$ for $1-\delta<w<1$.
In these calculations, we have made use of the following facts:
\begin{itemize}
\item Since $E(w^*)=-E(w)^*$ and since $\Psi(-\lambda^*)=\Psi(\lambda)^*$ it follows that $
\theta_0(w^*)=\theta_0(w)^*$.  This in turn implies via \eqref{eq:L-function-define} that $L(w^*)=L(w)^*$.  As \eqref{eq:L-function-define} also implies that $L_+(w)+L_-(w)=0$ for $w>0$ we find that $\mathrm{Re}\{L_-(w)\}=0$ on $1-\delta<w<1$.  Likewise, since $g(w^*)=g(w)^*$ and $g_+(w)+g_-(w)=0$ for $w>0$ we find that $\mathrm{Re}\{g_-(w)\}=0$ on $1-\delta<w<1$.
\item The boundary values $E_-(w)$, $D_-(w)$, and hence also $Q_-(w)$, taken on $(1-\delta,1)$ from the upper half-plane (as the orientation of the interval is right-to-left) are real-valued.  Moreover, $E_-(w)<0$ holds.  Also, since $E_-(w)$ is real and the approximate eigenvalues $\lambda_k$ are purely imaginary, it follows that $|\Pi_{N-}(w)|=1$ holds for $1-\delta<w<1$.  
\end{itemize}
Now we substitute for $T_N(w)$ from \eqref{eq:Y-and-T-define} and for $\hat{k}(w)$ from \eqref{eq:k-function-define} and \eqref{eq:tilde-k-define}.   Using the relations \eqref{eq:Lbar-function-define} and \eqref{eq:L-function-difference} (in which the subscripts refer to boundary values taken on $P_\infty$) we deduce that (now with subscripts referring to boundary values taken on $(1-\delta,1)$) $\overline{L}_-(w)=L_-(w)+\ii\theta_{0-}(w)$.  Hence 
\begin{equation}
\begin{split}
T_{N-}(w)&=\Pi_{N-}(w)\ee^{-\overline{L}_-(w)/\epsilon}\left(\ee^{\ii\theta_{0-}(w)/\epsilon}+\ee^{-\ii\theta_{0-}(w)/\epsilon}\right)\\
&=\Pi_{N-}(w)\ee^{-L_-(w)/\epsilon}\left(1+\ee^{-2\ii\theta_{0-}(w)/\epsilon}\right),\quad 1-\delta<w<1.
\end{split}
\end{equation}
In particular,
\begin{equation}
|T_{N-}(w)|^{-1}=\left|1+\ee^{-2\ii\theta_{0-}(w)/\epsilon}\right|^{-1},\quad 1-\delta<w<1.
\end{equation}
Furthermore, it follows 
from Assumption~\ref{assumption:Psi} that we can write
\begin{equation}
\theta_{0-}(w)=-\frac{1}{4}\int_\mathbb{R}G(x)\,\dd x + \ii u\lambda +f(\lambda),\quad \lambda=E_-(w)<0,
\end{equation}
where $u>0$ and $f(\lambda)$ denotes the convergent series
\begin{equation}
f(\lambda):=\sum_{n=1}^\infty v_n\lambda^{2n}
\end{equation}
in which $v_n$ are real coefficients.  Note that $f(\lambda)=\mathcal{O}(\lambda^2)$ holds near $\lambda=0$.  In particular, we have
\begin{equation}
\ii[\theta_{0-}(w)^*-\theta_{0-}(w)]=2uE_-(w),\quad E_-(w)<0,\quad 1-\delta<w<1.
\end{equation}
Also, using the fact that $\epsilon=\epsilon_N$ according to \eqref{eq:epsilon-N} we have
\begin{equation}
\ee^{-2\ii\theta_{0-}(w)/\epsilon}=\ee^{2uE_-(w)/\epsilon}\ee^{-2\ii f(E_-(w))/\epsilon} = \ee^{-2 uE_-(w)/\epsilon}(1+\mathcal{O}(\epsilon^{-1}E_-(w)^2)),\quad 1-\delta<w<1
\end{equation}
where we used the fact that $f(E_-(w))$ is real on $(1-\delta,1)$.
From these observations it follows that 
\begin{equation}
V_{11}(w)=\frac{1+\ee^{2u E_-(w)/\epsilon}}{|1+\ee^{2u E_-(w)/\epsilon}\ee^{-2\ii f(E_-(w))}|}=1+\mathcal{O}(\epsilon^{-1}E_-(w)^2\ee^{2u E_-(w)/\epsilon}) = 1+\mathcal{O}(\epsilon)
\end{equation}
because $h(t):=t^2\ee^{2u t}$ is uniformly bounded for $t=E_-(w)/\epsilon<0$.  Also,
\begin{equation}
|V_{12}(w)|=|V_{21}(w)|=|T_{N-}(w)|^{-1}\left|\ee^{2u E_-(w)/\epsilon}\ee^{-\ii\theta_{0-}(w)^*/\epsilon}-\ee^{\ii\theta_{0-}(w)^*/\epsilon}\right|
\end{equation}
Finally, we substitute for $\theta_{0-}(w)^*$ again using the condition \eqref{eq:epsilon-N} to give that the constant terms in both exponents give rise to the same factor of $\pm 1$, and obtain that
\begin{multline}
|V_{12}(w)|=|V_{21}(w)|=|T_{N-}(w)|^{-1}\ee^{u E_-(w)/\epsilon}\left|\ee^{-\ii f(E_-(w))/\epsilon}-\ee^{\ii f(E_-(w))/\epsilon}\right|\\
{}=\mathcal{O}(\epsilon^{-1}E_-(w)^2\ee^{u E_-(w)/\epsilon}) = \mathcal{O}(\epsilon)
\end{multline}
holds uniformly for $1-\delta<w<1$ by almost the same argument.  Since $V_{11}(w)V_{22}(w)-V_{12}(w)V_{21}(w)=1$ it follows that also $V_{22}(w)=1+\mathcal{O}(\epsilon)$.  Likewise, multiplying out the matrix factors in \eqref{eq:Ojump-Last} shows that 
\begin{equation}
\mathbf{O}_+(w)=\sigma_2\mathbf{O}_+(w)\sigma_2\begin{bmatrix}U_{11}(w) & U_{12}(w)\\ U_{21}(w) & U_{22}(w)
\end{bmatrix},\quad 1<w<1+\delta,
\end{equation}
where $U_{11}(w)U_{22}(w)-U_{12}(w)U_{21}(w)=1$ and
\begin{equation}
\begin{split}
U_{11}(w)&:= |T_{N+}(w)|^{-1}(1+\ee^{\ii[\theta_{0+}(w)-\theta_{0+}(w)^*]/\epsilon})\\
U_{12}(w)&:=\ii\ee^{2g_+(w)/\epsilon}|T_{N+}(w)|^{-1}\\
&\quad\quad{}\cdot\left[\Pi_{N+}(w)^*T_{N+}(w)\ee^{[-2\ii Q_+(w)-\ii\theta_{0+}(w)^*]/\epsilon}
-\ee^{-k_+(w)/\epsilon}(1+\ee^{\ii[\theta_{0+}(w)-\theta_{0+}(w)^*]/\epsilon})\right],
\end{split}
\end{equation}
and $U_{21}(w)=U_{12}(w)^*$ for $1<w<1+\delta$.  Similar arguments as in the case of $1-\delta<w<1$ then show that $U_{11}(w)-1$, $U_{22}(w)-1$, $U_{12}(w)$, and $U_{21}(w)$ are all $\mathcal{O}(\epsilon)$ uniformly for $1<w<1+\delta$ under Assumption~\ref{assumption:Psi} and the quantization condition \eqref{eq:epsilon-N} on $\epsilon=\epsilon_N$.  Therefore in both intervals $1-\delta<w<1$ and $1<w<1+\delta$ it holds that $\mathbf{O}_+(w)=\sigma_2\mathbf{O}_-(w)\sigma_2(\mathbb{I}+\mathcal{O}(\epsilon))$.

\subsubsection{Construction of a parametrix for $\mathbf{O}(w)$}
Neglecting terms in the jump conditions for $\mathbf{O}(w)$ that converge to zero with $\epsilon$ leaves only the jump condition \eqref{eq:Ojump-beta} for $w\in\beta$, a corresponding jump condition on the Schwarz reflection $\beta^*$, and the condition $\mathbf{O}_+(w)=\sigma_2\mathbf{O}_-(w)\sigma_2$ for $w>0$.  These are precisely the jump conditions of Riemann--Hilbert Problem~\ref{rhp:OuterParametrixGeneral} for a matrix $\bfY(w;\beta,\nu)$ described in Appendix~\ref{app:theta} in the case of the real parameter $\nu=\epsilon^{-1}\Phi=\epsilon^{-1}\Phi(x,t)$. Hence we define an \emph{outer parametrix} for $\mathbf{O}(w)$ by writing
\begin{equation}
\dot{\mathbf{O}}^\mathrm{out}(w)=\dot{\mathbf{O}}^\mathrm{out}(w;x,t,\epsilon):=\mathbf{Y}(w;\beta,\epsilon^{-1}\Phi(x,t)).
\end{equation}
According to Proposition~\ref{prop:PropertiesOfY}, $\dot{\mathbf{O}}^\mathrm{out}(w)$ has unit determinant and is uniformly bounded and oscillatory with respect to $(x,t)$, except when $w$ lies in neighborhoods of $w=\alpha$ and $w=\alpha^*$.  In such neighborhoods the outer parametrix would be expected to be a poor approximation of $\mathbf{O}(w)$.  

To better approximate $\mathbf{O}(w)$ near $w=\alpha$ and $w=\alpha^*$, let $U$ be a neighborhood of $w=\alpha(x,t)$.  Within $U$, the matrix 
\begin{equation}
\mathbf{P}(w):=\begin{cases}
\mathbf{O}(w)\ee^{-\ii\Phi(x,t)\sigma_3/(2\epsilon)},&\quad w\in U\cap (\Lambda_1^+\cup\Lambda_1^-)\\
\mathbf{O}(w)\ee^{-\ii\Phi(x,t)\sigma_3/(2\epsilon)}Y_N(w)^{-\sigma_3/2},&\quad w\in U\setminus (\Lambda_1^+\cup\Lambda_1^-)
\end{cases}
\end{equation}
satisfies simplified versions of the jump conditions \eqref{eq:Ojump-First}--\eqref{eq:Ojump-gamma} in which $Y_N(w)$ is replaced with $1$, $\Phi(x,t)$ is replaced with $0$, and $\phi(w)$ is replaced with $\phi(w)-\phi(\alpha)$ (note that $\phi(\alpha;x,t)=\ii\Phi(x,t)$).  Since $(x,t)$ lies in the modulated librational wave region, in particular $H(\alpha(x,t);x,t)\neq 0$, so that (from \eqref{eq:RH})
\begin{equation}
	\phi'(w)=k'(w)-2g'(w)=\mathcal{O}((w-\alpha)^{1/2})\quad\text{as $w\to\alpha$,}
	\label{eq:3HalfBehavior}
\end{equation}
which implies that $\phi(w)-\phi(\alpha)$ vanishes at $w=\alpha$ like a branch of $(w-\alpha)^{3/2}$.  In fact, it is easy to see that there exists a conformal mapping $W:U\to\mathbb{C}$ taking $\beta\cap U$ to $\mathbb{R}_-$ and satisfying the equation $W(w)^3=(\phi(w)-\phi(\alpha))^2$.  The jump conditions satisfied by $\mathbf{P}(w)$ for $w\in U$ can thus be written in terms of the exponentials $\ee^{\pm\xi^{3/2}}$ where $\xi:=\epsilon^{-2/3}W(w)$.  Choosing the arcs $\ell^\pm_1$ and $\gamma$ within $U$ to be mapped onto straight rays with $\arg(-\xi)$ equal to $\mp\tfrac{1}{3}\pi$ and $0$ respectively, the jump conditions for $\mathbf{P}(w)$ are expressed in terms of $\xi$ with $\xi\in \epsilon^{-2/3}W(U)$ as follows:
\begin{equation}
\begin{split}
\mathbf{P}_+&=\mathbf{P}_-\begin{bmatrix}1 & 0\\\ii\ee^{-\xi^{3/2}} & 1\end{bmatrix},\quad\arg(\xi)=0,\\
\mathbf{P}_+&=\mathbf{P}_-\begin{bmatrix}1 & \ii\ee^{\xi^{3/2}}\\0 & 1\end{bmatrix},\quad\arg(\xi)=\pm\tfrac{2}{3}\pi,\\
\mathbf{P}_+&=\mathbf{P}_-\begin{bmatrix}0 & -\ii\\-\ii & 0\end{bmatrix},\quad\arg(-\xi)=0.
\end{split}
\label{eq:Airy-jumps}
\end{equation}
Here for convenience the orientation of the contour arc $\gamma$ within $U$ has been reversed so that all four jump rays in the $\xi$-plane are oriented away from the origin.  It is well-known \cite{DeiftZhou1995} that there is a unique matrix function $\mathbf{A}(\xi)$ analytic for $\arg(\xi)\in (0,\tfrac{2}{3}\pi)\cup (-\tfrac{2}{3}\pi,0)\cup(\tfrac{2}{3}\pi,\pi)\cup (-\pi,-\tfrac{2}{3}\pi)$ and continuous up to the four boundary rays of the indicated sectors, such that $\mathbf{A}(\xi)$ satisfies exactly the jump conditions of $\mathbf{P}$ written in \eqref{eq:Airy-jumps} extended to infinite rays with the indicated angles, as well as the normalization condition 
\begin{equation}
\mathbf{A}(\xi)\mathbf{M}\xi^{-\sigma_3/4}=\mathbb{I}+\begin{bmatrix}\mathcal{O}(\xi^{-3}) & \mathcal{O}(\xi^{-1})\\\mathcal{O}(\xi^{-2}) & \mathcal{O}(\xi^{-3})\end{bmatrix},\quad\xi\to\infty,\quad \mathbf{M}:=\frac{1}{\sqrt{2}}\begin{bmatrix}1 & -1\\1 & 1\end{bmatrix}.
\label{eq:Airy-norm}
\end{equation}
The four elements of $\mathbf{A}(\xi)$ can be explicitly written in terms of the classical Airy function $\mathrm{Ai}(\cdot)$ and its derivative \cite[Chapter 9]{NIST:DLMF}.  The inverse of the matrix $\xi^{\sigma_3/4}\mathbf{M}^{-1}$ appears in the above asymptotic expansion because it is an exact solution of the ``twist'' jump condition on the negative real $\xi$-axis.  Taking out a holomorphic left multiplier of $\epsilon^{-\sigma_3/6}$ gives a matrix $W(w)^{\sigma_3/4}\mathbf{M}^{-1}$ with the same property.  Therefore $\dot{\mathbf{O}}^\mathrm{out}(w)\ee^{-\ii\Phi(x,t)\sigma_3/(2\epsilon)}$ and $W(w)^{\sigma_3/4}\mathbf{M}^{-1}$ both have the same domain of analyticity and satisfy the same jump condition for $w\in U$; moreover from Proposition~\ref{prop:PropertiesOfY} it follows that 
\begin{equation}
\mathbf{C}(w):=\dot{\mathbf{O}}^\mathrm{out}(w)\ee^{-\ii\Phi(x,t)\sigma_3/(2\epsilon)}\mathbf{M}W(w)^{-\sigma_3/4},\quad w\in U
\label{eq:CAiry}
\end{equation}
is analytic for $w\in U$, has unit determinant, and is uniformly bounded as $\epsilon\to 0$.  Now we may define an \emph{inner parametrix} for $\mathbf{O}(w)$ valid near $w=\alpha$ by 
\begin{equation}
\dot{\mathbf{O}}^\mathrm{in}(w):=\begin{cases}
\mathbf{C}(w)\epsilon^{\sigma_3/6}\mathbf{A}(\epsilon^{-2/3}W(w))\ee^{\ii\Phi(x,t)\sigma_3/(2\epsilon)},&\quad w\in U\cap (\Lambda_1^+\cup\Lambda_1^-),\\
\mathbf{C}(w)\epsilon^{\sigma_3/6}\mathbf{A}(\epsilon^{-2/3}W(w))\ee^{\ii\Phi(x,t)\sigma_3/(2\epsilon)}Y_N(w)^{\sigma_3/2},&\quad w\in U\setminus(\Lambda_1^+\cup\Lambda_1^-).
\end{cases}
\label{eq:OinAiry}
\end{equation}
Note that $\dot{\mathbf{O}}^\mathrm{in}(w)$ satisfies exactly the same jump conditions within $U$ as does $\mathbf{O}(w)$.

The \emph{global parametrix} for $\mathbf{O}(w)$ is then combines the inner and outer parametrices to approximate $\mathbf{O}(w)$ globally in the complex $w$-plane.  It is defined by the following formula:
\begin{equation}
\dot{\mathbf{O}}(w):=\begin{cases}
\dot{\mathbf{O}}^\mathrm{in}(w),&\quad w\in U\\
\dot{\mathbf{O}}^\mathrm{in}(w^*)^*,&\quad w\in U^*\\
\dot{\mathbf{O}}^\mathrm{out}(w),&\quad w\in\mathbb{C}\setminus (U\cup U^*).
\end{cases}
\label{eq:GlobalParametrix}
\end{equation}

\subsubsection{Error analysis}
The accuracy of approximating $\bfO(w)$ with its global parametrix $\dot{\bfO}(w)$ can be measured with the help of the \emph{error} defined by
\begin{equation}
\mathbf{E}(w):=\mathbf{O}(w)\dot{\mathbf{O}}(w)^{-1}
\label{eq:ErrorMatrix}
\end{equation}
wherever both factors make sense.  In fact, based on the analyticity properties of $\mathbf{O}(w)$ from the conditions of Riemann--Hilbert Problem~\ref{rhp:O} and the definition \eqref{eq:GlobalParametrix} of $\dot{\mathbf{O}}(w)$, one can see that inside the disks $U$ and $U^*$ $\mathbf{E}(w)$ is analytic while outside the disks it is analytic for $w\in\mathbb{C}\setminus (\Sigma\setminus(\beta\cup\beta^*))$.  (There are no jumps across any contours within the disks, nor across $\beta\cup\beta^*$, because the jump conditions of $\dot{\mathbf{O}}(w)$ and $\mathbf{O}(w)$ agree exactly across these arcs.)  There are also jump discontinuities of $\mathbf{E}(w)$ across the boundaries $\partial U$ and $\partial U^*$ of the disks.  For all arcs of $\Sigma\setminus\beta$ outside of $U$ in the open upper half-plane, we have already shown that $\mathbf{O}_+(w)=\mathbf{O}_-(w)(\mathbb{I}+\mathcal{O}(\epsilon))$ holds, so since $\dot{\mathbf{O}}^\mathrm{out}(w)$ is analytic and uniformly bounded on these arcs also as $\epsilon\to 0$ according to Proposition~\ref{prop:PropertiesOfY}, it is easy to check that also $\mathbf{E}_+(w)=\mathbf{E}_-(w)(\mathbb{I}+\mathcal{O}(\epsilon))$ holds.  For $w\in\partial U$ taken with clockwise orientation, it is easy to see that $\mathbf{E}_+(w)=\mathbf{E}_-(w)\dot{\mathbf{O}}^\mathrm{in}(w)\dot{\mathbf{O}}^\mathrm{out}(w)^{-1}$, and from \eqref{eq:CAiry} and \eqref{eq:OinAiry} we get
\begin{equation}
\dot{\mathbf{O}}^\mathrm{in}(w)\dot{\mathbf{O}}^\mathrm{out}(w)^{-1} = 
\begin{cases}
\mathbf{C}(w)\epsilon^{\sigma_3/6}\mathbf{A}(\xi)\mathbf{M}\xi^{-\sigma_3/4}\epsilon^{-\sigma_3/6}\mathbf{C}(w)^{-1},&\quad w\in\partial U\cap (\Lambda_1^+\cup\Lambda_1^-)\\
\mathbf{C}(w)\epsilon^{\sigma_3/6}\mathbf{A}(\xi)\mathbf{M}\xi^{-\sigma_3/4}\epsilon^{-\sigma_3/6}\mathbf{C}(w)^{-1}&\\
\qquad\qquad\qquad{}\cdot\dot{\mathbf{O}}^\mathrm{out}(w)Y_N(w)^{\sigma_3/2}\dot{\mathbf{O}}^\mathrm{out}(w)^{-1},&\quad w\in\partial U\setminus (\Lambda_1^+\cup\Lambda_1^-).
\end{cases}
\label{eq:EjumpAiry}
\end{equation}
Now, since $w\in\partial U$ means that $\xi$ is uniformly proportional to $\epsilon^{-2/3}$, the expansion \eqref{eq:Airy-norm} shows that $\epsilon^{\sigma_3/6}\mathbf{A}(\xi)\mathbf{M}\xi^{-\sigma_3/4}\epsilon^{-\sigma_3/6}=\mathbb{I}+\mathcal{O}(\epsilon)$ holds uniformly for $w\in\partial U$.  Since $\mathbf{C}(w)$ and $\dot{\mathbf{O}}^\mathrm{out}(w)$ have unit determinant and are uniformly bounded for $w\in\partial U$ by Proposition~\ref{prop:PropertiesOfY}, it follows from Proposition~\ref{prop:T-and-Y} that $\mathbf{E}_+(w)=\mathbf{E}_-(w)(\mathbb{I}+\mathcal{O}(\epsilon))$ holds on $\partial U$.  By Schwarz symmetry it follows that similar estimates hold for $w$ in all arcs of the jump contour for $\mathbf{E}(w)$ in the open lower half-plane.  Finally, note that for $w\in\mathbb{R}_+$ we combine the exact jump condition $\dot{\mathbf{O}}^\mathrm{out}_+(w)=\sigma_2\dot{\mathbf{O}}^\mathrm{out}_-(w)\sigma_2$ with the approximate one $\mathbf{O}_+(w)=\sigma_2\mathbf{O}_-(w)\sigma_2(\mathbb{I}+\mathcal{O}(\epsilon))$ where the error term vanishes for $|w-1|>\delta$ to find that $\mathbf{E}_+(w)=\sigma_2\mathbf{E}_-(w)\sigma_2(\mathbb{I}+\mathcal{O}(\epsilon))$ with the same caveat for the error term.   

Noting also that $\mathbf{E}(w)\to\mathbb{I}$ as $w\to\infty$ by the normalization conditions on $\mathbf{O}(w)$ and $\dot{\mathbf{O}}^\mathrm{out}(w)$, one checks that upon carrying $\mathbf{E}(w)$ to the $z=\ii\sqrt{-w}$ plane to deal with the non-standard jump condition for $w>0$ (see the discussion around \eqref{eq:go-to-z-plane} below for more details) one arrives at a small-norm problem for which standard theory guarantees that $\mathbf{E}(w)=\mathbb{I}+\mathcal{O}(\epsilon)$ uniformly for all $w\in\mathbb{C}$.  Now using \eqref{eq:H-to-M}, \eqref{eq:M-to-N-elsewhere}, \eqref{eq:N-to-O}, \eqref{eq:g-integral}, \eqref{eq:GlobalParametrix}, and \eqref{eq:ErrorMatrix}, we have
\begin{equation}
\mathbf{H}(0)=\mathbf{M}(0) = \mathbf{N}(0)=\mathbf{O}(0)\ee^{g(0)\sigma_3/\epsilon} = \mathbf{O}(0)=\mathbf{E}(0)\dot{\mathbf{O}}(0) = \mathbf{E}(0)\dot{\mathbf{O}}^\mathrm{out}(0).
\end{equation}
So, since $\cos(\tfrac{1}{2}u_N(x,t))$ and $\sin(\tfrac{1}{2}u_N(x,t))$ are encoded in the first column of $\mathbf{H}(0)$ according to \eqref{eq:cos-sin-H0} and since $\mathbf{E}(0)=\mathbb{I}+\mathcal{O}(\epsilon)$, the asymptotic formul\ae\ \eqref{eq:cos-sin-before} are confirmed upon using \eqref{eq:general-outer-parametrix-at-origin} from Proposition~\ref{prop:PropertiesOfY}.  This completes the proof of Proposition~\ref{prop:BeforeCatastrophe}.

\subsection{The boundary of the modulated librational wave region and points of gradient catastrophe}
\label{sec:boundary}
The boundary of the modulated librational wave region consists of points $(x,t)$ for which a critical point of $\phi(w)=\phi(w;x,t)$ first appears on the branch of the zero level curve $\mathrm{Re}\{\phi(w)\}=0$ emanating from $w=\alpha$ on the right side of $\beta$.  Indeed, the appearance of this critical point on the zero level signals the closing of the green-shaded channel through which the curve $\gamma$ must pass (see Figures~\ref{fig:w-plane-plots-x0p2}--\ref{fig:w-plane-plots-x0}). 
This gives a criterion for locating points on the boundary of the modulated librational wave region in the $(x,t)$-plane that can be implemented numerically.  We have computed these points for the initial impulse profile $G(x)=-\mathrm{sech}(x)$, and we superimpose the corresponding modulated librational wave region with green shading on density plots of $\cos(u_N(x,t))$ for $N=8$ and $N=16$ in Figure~\ref{fig:Boundary}.
\begin{figure}[h]
\begin{center}
\includegraphics[width=0.4\linewidth]{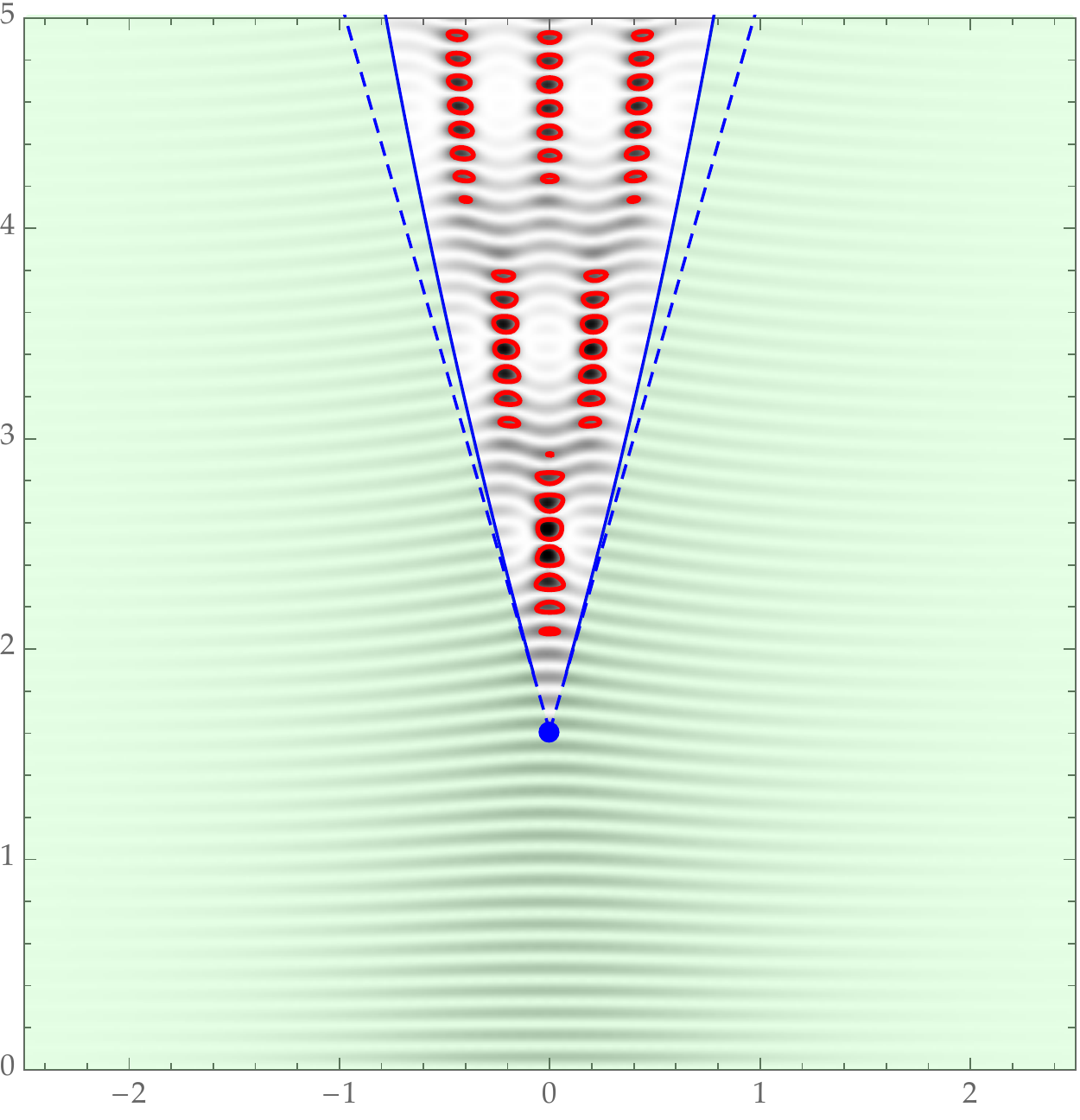}%
\hspace{0.05\linewidth}%
\includegraphics[width=0.4\linewidth]{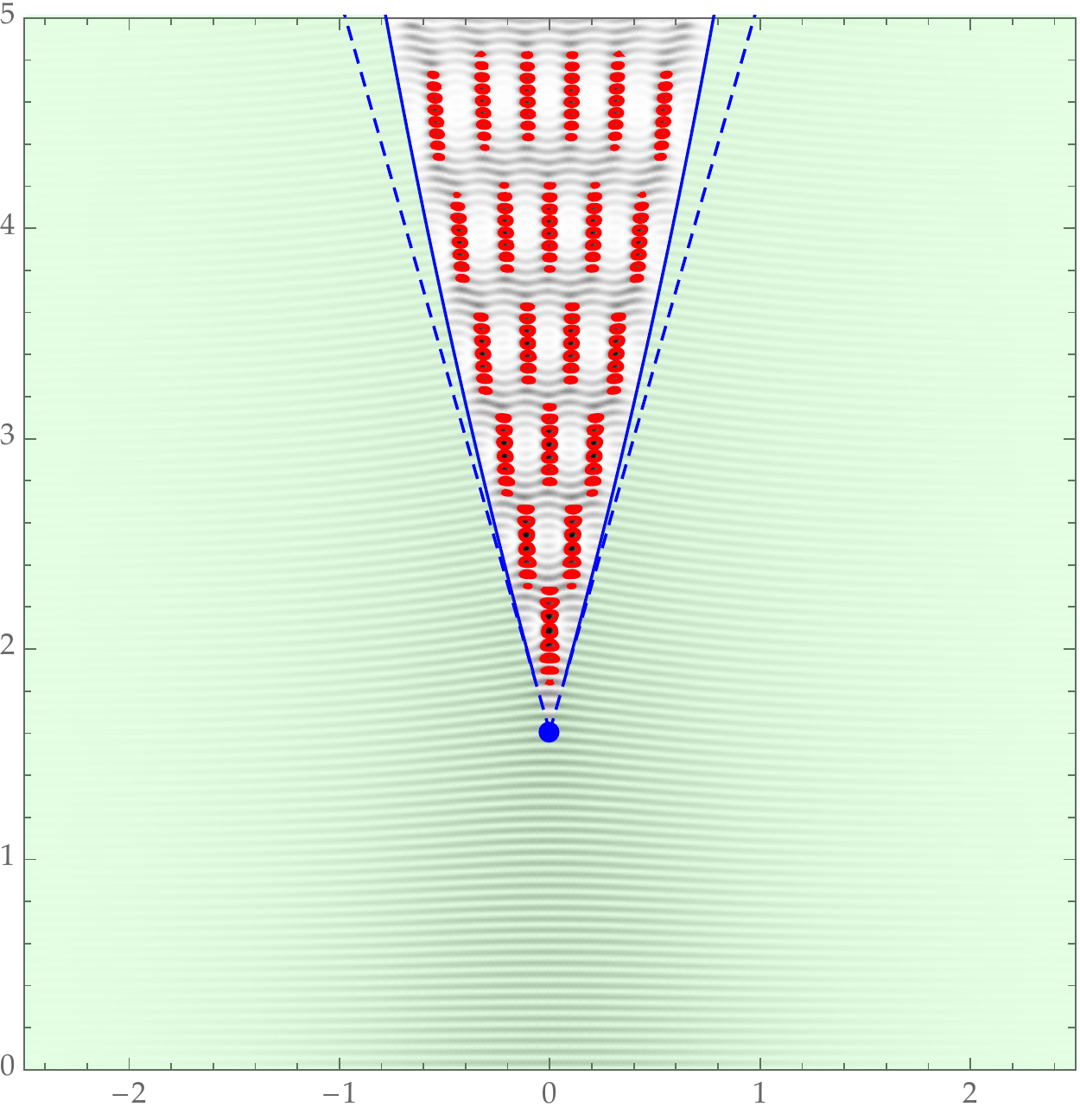}
\end{center}
\caption{Density plots of $\cos(u_N(x,t))$ for the initial impulse profile $G(x)=-\mathrm{sech}(x)$ over the $x$ (horizontal) and $t$ (vertical) plane along with the $N$-independent modulated librational wave region (green shading), its numerically-computed boundary (solid blue curves for $|x|\ge 0.08$), and preimages of $\arg(\tau)=\pm\tfrac{1}{5}\pi$ under the mapping $\tau=\epsilon^{-4/5}(\ii a x + b(t-t_\mathrm{gc}))$ (dashed lines; see \eqref{eq:AwayFromPolesSolution} with $b=-a\rho(m_\gc)$).  Left:  $N=8$.  Right:  $N=16$.  The vertex of the boundary curve is the gradient catastrophe point $(0,t_\mathrm{gc})$ with $t_\mathrm{gc}\approx 1.609104$.  The boundary is difficult to compute reliably for $|x|$ small, but it would be expected to be tangent to the two dashed lines at the catastrophe point.}
\label{fig:Boundary}
\end{figure}

Recall from \eqref{eq:RH} that the critical points of $\phi(w)=\phi(w;x,t)$ other than $w=\alpha$ and $w=\alpha^*$ are the zeros of the function $H(w)=H(w;x,t)$.
Now, the condition (\ref{item:MLW-H-not-zero}) in Definition~\ref{def:MLW} implies that $H(\alpha(x,t);x,t)\neq 0$, as $\alpha(x,t)$ is an endpoint of $\beta$; however it \emph{is} possible that there are points $(x,t)$ on the boundary of the modulated librational wave region at which $H(\alpha(x,t);x,t)=0$.  
\begin{definition}[Gradient catastrophe]
A point $(x,t)$ on the boundary of the modulated librational wave region is called a point of (simple) gradient catastrophe if conditions (\ref{item:MLW-beta}) and (\ref{item:MLW-gamma}) from Definition~\ref{def:MLW} hold but instead of condition (\ref{item:MLW-H-not-zero}) we have
\[
H(\alpha(x,t);x,t)=0\quad\text{and}\quad H'(\alpha(x,t);x,t)\neq 0
\]
and $H(w;x,t)/(w-\alpha(x,t))$ is bounded away from zero for $w\in\beta$.
\label{def:GradientCatastrophe}
\end{definition}
The right-hand plot in Figure~\ref{fig:w-plane-plots-x0} shows a configuration in which a simple zero of $H(w;x,t)$ is very close to $w=\alpha(x,t)$; correspondingly $(x,t)=(0,1.608556)$ is very close to the gradient catastrophe point of $(0,t_\mathrm{gc})$ with $t_\mathrm{gc}\approx 1.609104$ for the initial impulse profile $G(x)=-\mathrm{sech}(x)$.
\subsection{Symmetries for $x=0$}
\label{sec:x=0-Symmetry}
By following the arguments in \cite[Section 4.2]{BuckinghamMiller2013} one shows that in the limit $t\downarrow 0$ with $x>0$, the arc $\beta\cup\beta^*$ becomes a proper sub-arc of $P_\infty$, and the level sets of $\mathrm{Re}\{\phi(w;x,0)\}$ become symmetric with respect to reflection through the unit circle.  In this section, we show that a related symmetry occurs if $x=0$ but $t>0$.  This is important because due to even symmetry in $x$ (see \eqref{eq:condensate-even}) we are restricting our attention in this work to gradient catastrophe points $(x_\mathrm{gc},t_\mathrm{gc})$ with $x_\mathrm{gc}=0$ (it is also possible in principle for simultaneous gradient catastrophe points to occur at opposite nonzero values of $x$, but we do not consider this case here).

We begin with the following elementary observation the proof of which is just an application of the Schwarz reflection principle, using the fact that $\Psi(\lambda)$ is real-valued on the imaginary segment where it is initially defined by \eqref{eq:Psi_Initial_Condition}.
\begin{lemma}
The phase integral $\Psi(\lambda)$, defined in \eqref{eq:Psi_Initial_Condition} on an interval of the imaginary axis and extended to a sufficiently large neighborhood of that interval as an analytic function by Assumption~\ref{assumption:Psi}, satisfies $\Psi(-\lambda^*)=\Psi(\lambda)^*$ on that neighborhood.
\label{lemma:theta0}
\end{lemma}
In particular, the analytic continuation of $\Psi(\lambda)$ along the imaginary axis above the point $-\tfrac{1}{4}\ii G(0)$ is real-valued.

\begin{proposition}
Suppose that $x=0$ and that $\alpha=\ee^{\ii\arg(\alpha)}$, where $\mu\leq\arg(\alpha)<\pi$ (recall that $\ee^{\ii\mu}$ is the endpoint of $P_\infty$ in $\mathbb{C}_+$). If $\beta$ lies in the exterior of the unit circle except at its two endpoints, then the condition $M=0$ (see \eqref{eq:M}) is automatically satisfied.
\label{proposition:MSymmetry}
\end{proposition}

\begin{proof}
Taking $x=0$  and  using $|\alpha|=1$ in  \eqref{eq:M} we get
\begin{equation}
M=\frac{4}{\pi}\int_{\tilde{\gamma}\cup\tilde{\gamma}^*}\frac{\theta'_0(\xi)\sqrt{-\xi}}{R(\xi)}\,\dd\xi = \frac{8}{\pi}\mathrm{Re}\left\{
\int_\gamma\frac{\theta_0'(\xi)\sqrt{-\xi}}{R(\xi)}\,\dd\xi\right\}.
\label{eq:Mx=0}
\end{equation}
Recall that $\tilde{\gamma}$ is an oriented arc from $\ee^{\ii\mu}$ to $\alpha=\ee^{\ii\arg(\alpha)}$; we shall take it to lie along the unit circle, and note that by assumption on the location of $\beta$, $R(\xi)$ is single-valued along $\tilde{\gamma}$.  Now since $\mu\le\arg(\alpha)<\pi$, and hence $\lambda=E(w)$ is a univalent map $\tilde{\gamma}\to\ii\mathbb{R}$, we can instead integrate with respect to $\lambda$, noting that $\theta_0'(\xi)\,\dd\xi = \Psi(\lambda)\,\dd\lambda$.  Therefore, as $R(w)/\sqrt{-w}=-\ii\sqrt{2(1-\cos(\arg(\alpha)))+16\lambda^2}$, where we take the positive square root for $\lambda\in\ii\mathbb{R}$ between $E(\ee^{\ii\mu})$ and $E(\ee^{\ii\arg(\alpha)})$,
\begin{equation}
M=-\frac{8}{\pi}\mathrm{Im}\left\{
\int_{E(\ee^{\ii\mu})}^{E(\ee^{\ii\alpha})}\frac{\Psi'(\lambda)\,\dd\lambda}{\sqrt{2(1-\cos(\arg(\alpha)))+16\lambda^2}}\right\}.
\end{equation}
However, the increment $\Psi'(\lambda)\,\dd\lambda = \dd\Psi(\lambda)$ is real according to Lemma~\ref{lemma:theta0}, so the integral is real, and therefore $M=0$.
\end{proof}
This result is significant in practice because it allows one to reduce the problem of numerical continuation of the solution of the simultaneous nonlinear equations $M=0$ and $I=0$ to the solution of a single real equation $I=0$ to determine $\arg(\alpha)$, under the restriction of $\alpha$ to the unit circle with $P_\infty$ omitted.  This yields a more efficient numerical algorithm for determining $\alpha$ as a function of $t>0$ for $x=0$, which was used to make the plots in Figure~\ref{fig:w-plane-plots-x0}.  Note that in practice the assumption on $\beta$ in the statement of Proposition~\ref{proposition:MSymmetry} can be confirmed after the fact as can also be seen in Figure~\ref{fig:w-plane-plots-x0}.  A similar simplification based on Proposition~\ref{proposition:MSymmetry} allows for the accurate computation of catastrophe points $(0,t_\mathrm{gc})$ and $\theta=\arg(\alpha_\mathrm{gc})$.  Indeed, one can use the equation $H(\ee^{\ii\theta};0,t_\mathrm{gc})=0$ (see Definition~\ref{def:GradientCatastrophe}) to explcitly eliminate $t_\mathrm{gc}$ in favor of $\theta$, after which the calculation of $\theta$ is reduced to a single-variable root search with the equation $I=0$ (see \eqref{eq:I}).  This is how we can accurately compute the value $t_\mathrm{gc}\approx 1.609104$ with $\theta\approx 1.403433$ for the initial impulse profile $G(x)=-\mathrm{sech}(x)$, even though it is difficult to accurately calculate the boundary of the modulated librational wave region near the catastrophe point where $x$, although small, is nonzero.

Now we suppose that for $x=0$ and $t>0$ a point $\alpha=\alpha(0,t)$ with $|\alpha|=1$ has been determined so that the condition \eqref{eq:I} holds, and thus $g(w;x,t)$ is defined, as is the corresponding function $\phi(w;x,t)$.  In this situation, we have the following:
\begin{proposition} \label{proposition:HSymmetry}
$\mathrm{Re}\{\phi(\ee^{\ii\eta};0,t)\}=0$ holds for $\arg(\alpha)<\eta<\pi$.
\end{proposition}

\begin{proof}
First note that since $\phi(\alpha;0,t)=\ii\Phi(0,t)$ is purely imaginary, we have
\begin{equation}
\mathrm{Re}\left\{\phi(\ee^{\ii\eta};0,t)\right\}=\mathrm{Re}\left\{\int_{\alpha}^{\ee^{\ii\eta}}\phi'(w;0,t)\,\dd w\right\}.
\end{equation}
To use \eqref{eq:RH} to represent $\phi'(w;0,t)$ in this situation, it is useful to first rewrite $H(w)=H(w;x,t)$ more generally by taking the point $w$ through the outer arc of the contour $L$ pictured in the left-hand panel of Figure~\ref{fig:L-Lbeta} at the cost of a residue and then merging the two arcs of $L$ between $\xi=\ee^{\ii\mu}$ and $\xi=\alpha$ while taking them to lie on opposite sides of $\beta$ between $\xi=\alpha$ and $\xi=1$ (where the contributions to the integral in $H$ cancel as $R$ changes sign across $\beta$).  In other words, we can write
\begin{equation}
H(w;x,t)=-\frac{1}{4\sqrt{-w}}\left[\frac{x-t}{w\sqrt{\alpha\alpha^*}}+4\ii\frac{\theta_0'(w)\sqrt{-w}}{R(w)} - \frac{4}{\pi}\int_{\tilde{\gamma}\cup\tilde{\gamma}^*}\frac{\theta_0'(\xi)\sqrt{-\xi}}{R(\xi)(\xi-w)}\,\dd\xi\right],
\label{eq:H-residue}
\end{equation}
in which we are assuming that $w\in\mathbb{C}_+$ lies outside of the bounded region enclosed by $P_\infty$, $\tilde{\gamma}$, and $\beta$.  Now consider letting $x\downarrow 0$ for given $t>0$ fixed, so that $\alpha$ tends to $\ee^{\ii\arg(\alpha)}$ with $\mu<\arg(\alpha)<\pi$.  The formula \eqref{eq:H-residue} thus holds for $x=0$, with $t>0$, and for $w$ on the unit circle between $\ee^{\ii\arg(\alpha)}$ and $\ee^{\ii\eta}$, and we may take $\tilde{\gamma}$ to be the oriented arc of the unit circle from $\ee^{\ii\mu}$ to $\alpha=\ee^{\ii\arg(\alpha)}$.  Multiplying by $R(w)$ to obtain $\phi'(w;0,t)$ as in \eqref{eq:RH} we arrive at
\begin{equation}
\mathrm{Re}\left\{\phi(\ee^{\ii\eta};0,t)\right\} = \mathrm{I} + \mathrm{II} + \mathrm{III},
\end{equation}
where
\begin{equation}
\mathrm{I}:=\frac{t}{4}\mathrm{Re}\left\{\int_\alpha^{\ee^{\ii\eta}}\frac{R(w)}{\sqrt{-w}}\frac{\dd w}{w}\right\},
\end{equation}
\begin{equation}
\mathrm{II}:=\mathrm{Im}\left\{\int_\alpha^{\ee^{\ii\eta}}\theta_0'(w)\,\dd w\right\},\quad\text{and}
\end{equation}
\begin{equation}
\mathrm{III}:=\frac{1}{\pi}\mathrm{Re}\left\{\int_\alpha^{\ee^{\ii\eta}}\frac{R(w)}{\sqrt{-w}}\int_{\tilde{\gamma}\cup\tilde{\gamma}^*}\frac{\theta_0'(\xi)\sqrt{-\xi}}{R(\xi)(\xi-w)}\,\dd\xi\,\dd w\right\}.
\end{equation}

As the integration variable $w$ now lies on the unit circle between $\alpha$ and $-1$, one can check that $R(w)/\sqrt{-w}$ is, in contrast to the situation in the proof of Proposition~\ref{proposition:MSymmetry}, purely negative real.  On the other hand, the differential $\dd w/w$ is purely imaginary, and we therefore conclude that $\mathrm{I}=0$.  To deal with $\mathrm{II}$, we may use the univalent map $w\mapsto \lambda:=E(w)$ and $\theta_0'(w)\,\dd w = \Psi'(\lambda)\,\dd\lambda = \dd\Psi(\lambda)$ which is real by Lemma~\ref{lemma:theta0} to conclude that also $\mathrm{II}=0$.
The contour $\tilde{\gamma}\cup\tilde{\gamma}^*$ lies on the unit circle and is therefore invariant under the involution $\xi\mapsto \xi^{-1}$, $\dd\xi\mapsto -\xi^{-2}\,\dd\xi$; on the other hand, by the definition \eqref{eq:E-and-D-define} we have $E(\xi^{-1})=E(\xi)$ for all $\xi$, and differentiation of the identity $\theta_0(\xi):=\Psi(E(\xi))$ therefore shows that $\theta_0'(\xi^{-1})=-\xi^2\theta_0'(\xi)$.  One also checks that since $|\alpha|=1$, $\sqrt{-\xi^{-1}}/R(\xi^{-1})=\sqrt{-\xi}/R(\xi)$.  It follows that the inner integral in $\mathrm{III}$ can be written as
\begin{equation}
\int_{\tilde{\gamma}\cup\tilde{\gamma}^*}\frac{\theta_0'(\xi)\sqrt{-\xi}}{R(\xi)(\xi-w)}\,\dd\xi = 
\int_{\tilde{\gamma}\cup\tilde{\gamma}^*}\frac{\theta_0'(\xi)\sqrt{-\xi}}{R(\xi)(\xi^{-1}-w)}\dd\xi.
\end{equation}
Now we average these two equivalent expressions with the help of the identity
\begin{equation}
\frac{1}{2}\left[\frac{1}{\xi-w}+\frac{1}{\xi^{-1}-w}\right]=\frac{E(w)E'(w)}{E(\xi)^2-E(w)^2}-\frac{1}{2w}.
\label{eq:averaging}
\end{equation}
Taking into account that $M=0$ by Proposition~\ref{proposition:MSymmetry}, where $M$ is written in the form \eqref{eq:Mx=0} for $x=0$, shows that the last term on the right-hand side of \eqref{eq:averaging} makes no contribution, and therefore the inner integral in $\mathrm{III}$ is
\begin{equation}
\int_{\tilde{\gamma}\cup\tilde{\gamma}^*}\frac{\theta_0'(\xi)\sqrt{-\xi}}{R(\xi)(\xi-w)}\,\dd\xi=
E(w)E'(w)\int_{\tilde{\gamma}\cup\tilde{\gamma}^*}\frac{\theta_0'(\xi)\sqrt{-\xi}}{R(\xi)(E(\xi)^2-E(w)^2)}\,\dd\xi.
\label{eq:inner-integral-rewrite}
\end{equation}
Now $E(\xi)$ and $E(w)$ are both purely imaginary because $\xi$ and $w$ lie on the unit circle, and as in the proof of Proposition~\ref{proposition:MSymmetry} the ratio $\sqrt{-\xi}/R(\xi)$ is purely imaginary for $\xi\in\tilde{\gamma}\cup\tilde{\gamma}^*$ while the increment $\theta_0'(\xi)\,\dd\xi=\Psi'(\lambda)\,\dd\lambda=\dd\Psi(\lambda)$ is purely real, so the integral on the right-hand side of \eqref{eq:inner-integral-rewrite} is purely imaginary.  Just as in the analysis of $\mathrm{I}$ above, $R(w)/\sqrt{-w}$ is purely real in the integrand of the outer integral in $\mathrm{III}$, and also $E(w)E'(w)\,\dd w = \tfrac{1}{2}\dd E(w)^2$ is a purely real increment, so we finally conclude that $\mathrm{III}=0$ as well.
\end{proof}

The significance of Proposition~\ref{proposition:HSymmetry} in the context of this paper is that at a simple gradient catastrophe point with $x=0$, there are five zero level curves of $\mathrm{Re}\{\phi(w)\}$ emanating from $w=\alpha=\alpha_\mathrm{gc}=\ee^{\ii\theta}$ separated by equal angles of $2\pi/5$ radians, and exactly one of these is locally an arc of the unit circle with $\arg(w)>\theta$ (see the right-hand panel of Figure~\ref{fig:w-plane-plots-x0} for an illustration).  This information will be used in the proofs of Theorems~\ref{thm:AwayFromPoles} and \ref{thm:NearThePoles} below to concretely determine phase factors that appear in the asymptotic formul\ae\ in the statements of these results.

\section{Proof of Theorem~\ref{thm:AwayFromPoles}}
\label{sec:AwayFromPoles}

In this section we prove the first main result, Theorem \ref{thm:AwayFromPoles}.

\subsection{Simplification of jump conditions near $w=\alpha$}
Recall that when $(x,t)$ is fixed in the modulated librational wave region, the local behavior of the exponent function $\phi(w)$ near $w=\alpha=\alpha(x,t)$ is as specified
in \eqref{eq:3HalfBehavior}, and such behavior leads to the installation of a local parametrix constructed from Airy functions with jump contours forming angles at $w=\alpha$ that are integer multiples of $\tfrac{1}{3}\pi$.  However, when $(x,t)$ approaches the gradient catastrophe, the local behavior of $\phi(w)$ changes and thus a new parametrix is needed.  Indeed, at the gradient catastrophe point $(x,t)=(0,t_\mathrm{gc})$, the angles between the arcs $\gamma$, $\beta_1$, and $\ell^\pm_1$ become different for $\phi(w)-\phi(\alpha)$ to remain real-valued on $\gamma$ and $\ell^\pm_1$ and to take purely imaginary boundary values on $\beta$:  the interior angles at the catastrophe point (where also $\alpha=\alpha_\mathrm{gc}=\ee^{\ii\theta}$) become $\angle(\gamma,\ell^-_1)=\tfrac{2}{5}\pi$, $\angle(\ell^-_1,\beta_1)=\tfrac{1}{5}\pi$, and $\angle(\beta_1,\ell^+_1)=\tfrac{1}{5}\pi$ (the notation $\angle(u,v)$ denotes the angle between tangents at $\alpha_\mathrm{gc}$ of arcs $u$ and $v$ meeting at $w=\alpha_\mathrm{gc}$ taken in counterclockwise order about the vertex). 

To study the Riemann--Hilbert problem for $\mathbf{O}(w)$ when $(x,t)$ is near the catastrophe point, we will take the contours meeting at $w=\alpha\approx\alpha_\mathrm{gc}$ to have the above-indicated tangent angles (where we recall that $\beta_1$ is well-defined curve that meets $w=\alpha$ at a given angle, from which the angles of the other three jump contours are then determined).  The jump contour for $\mathbf{O}(w)$ in a neighborhood $U$ of $w=\alpha$ will therefore be taken to be like that shown in Figure~\ref{fig:O-Jumps-Local}.
\begin{figure}[h]
\begin{center}
\includegraphics{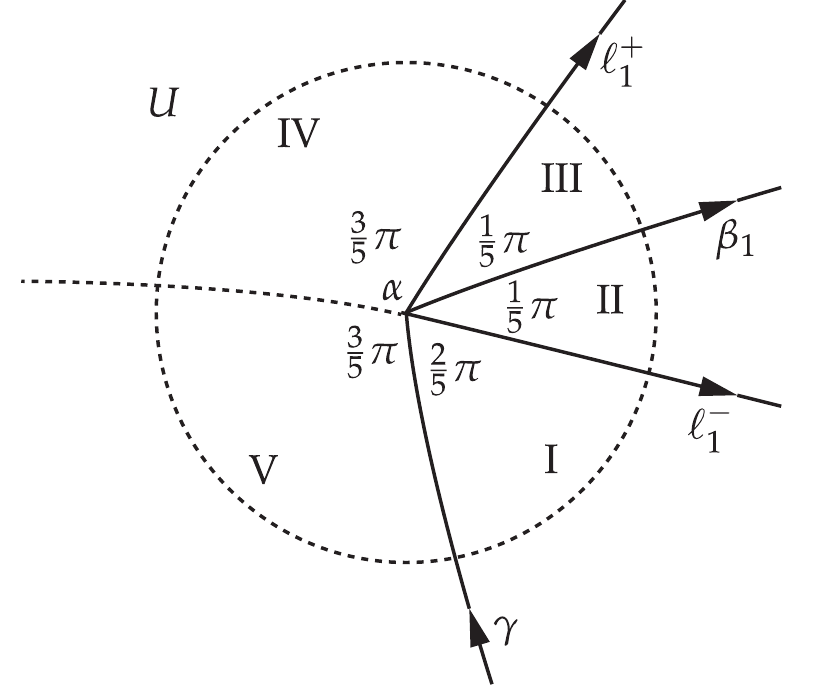}
\end{center}
\caption{The jump contour for $\mathbf{O}(w)$ in a neighborhood $U$ of $w=\alpha$ when $(x,t)$ is near the gradient catastrophe point.  The dotted arc emanating from $w=\alpha$ is arbitrary at the moment, but its tangent line should agree with that of $\ell^-_1$ at $w=\alpha$.  Both of these curves will be tangent to the unit circle when $(x,t)=(0,t_\mathrm{gc})$.  Also indicated are the five sectors $\mathrm{I}$--$\mathrm{V}$ of the neighborhood $U$.  Compare with the right-hand panel of Figure~\ref{fig:w-plane-plots-x0}.}
\label{fig:O-Jumps-Local}
\end{figure}
In order to study the situation when $(x,t)$ is near the gradient catastrophe point, we make a simple substitution in $U$ to simplify and standardize the jump matrices.  We set
\begin{equation}
\widetilde{\mathbf{O}}(w):=\begin{cases}
\mathbf{O}(w)Y_N(w)^{-\sigma_3/2}\ee^{-\ii\Phi\sigma_3/(2\epsilon)}(-\ii\sigma_1),&\quad w\in\mathrm{I}\cap U\\
\mathbf{O}(w)\ee^{-\ii\Phi\sigma_3/(2\epsilon)}(-\ii\sigma_1),&\quad w\in\mathrm{II}\cap U\\
\mathbf{O}(w)\ee^{-\ii\Phi\sigma_3/(2\epsilon)},&\quad w\in\mathrm{III}\cap U\\
\mathbf{O}(w)Y_N(w)^{-\sigma_3/2}\ee^{-\ii\Phi\sigma_3/(2\epsilon)},&\quad w\in\mathrm{IV}\cap U\\
\mathbf{O}(w)Y_N(w)^{-\sigma_3/2}\ee^{-\ii\Phi\sigma_3/(2\epsilon)}(-\ii\sigma_1),&\quad w\in\mathrm{V}\cap U.
\end{cases}
\label{eq:O-Otilde}
\end{equation}
If we define an exponent $\widetilde{\phi}(w)$ as the analytic continuation of $\phi(w)-\ii\Phi$ (which is analytic in $U$ except on $\beta_1$ where its boundary values sum to zero) from the region $\mathrm{I}\cap U$ to the slit domain $U\setminus(\partial\mathrm{IV}\cap\partial\mathrm{V})$:
\begin{equation}
\widetilde{\phi}(w):=\begin{cases}\phi(w)-\ii\Phi,&\quad w\in(\mathrm{I}\cup\mathrm{II}\cup\mathrm{V})\cap U,\\
\ii\Phi-\phi(w),&\quad w\in(\mathrm{III}\cup\mathrm{IV})\cap U,
\end{cases}
\end{equation}
 then $\widetilde{\phi}(w)$ has no jump across $\beta_1\cap U$, and is analytic in $U$ except on $\partial\mathrm{IV}\cap\partial\mathrm{V}$ where its boundary values sum to zero.  Re-orienting the contour $\gamma$ within $U$ for convenience, the jump conditions for $\widetilde{\mathbf{O}}(w)$ within $U$ are as shown in
Figure~\ref{fig:OtildeJumpsNearCatastrophe}.
\begin{figure}[h]
\begin{center}
\includegraphics{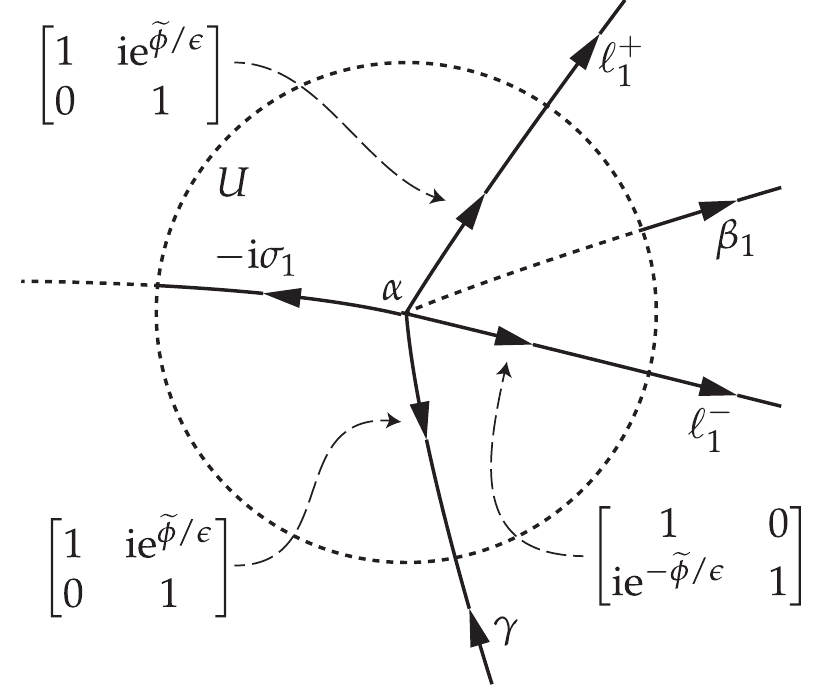}
\end{center}
\caption{The jump conditions satisfied by $\widetilde{\mathbf{O}}(w)$ within $U$.}
\label{fig:OtildeJumpsNearCatastrophe}
\end{figure}

\subsection{Modified $g$-function near the gradient catastrophe}
\label{sec:modified-g}
Recall that the $g$-function was previously obtained by choosing the $(x,t)$-dependence of the endpoints $w=\alpha,\alpha^*$ of $\beta$ in such a way that the conditions $M=I=0$ (see \eqref{eq:M} and \eqref{eq:I} with $H$ defined by \eqref{eq:RH}) are satisfied identically.  This defines $\alpha(x,t)$ and $\alpha^*(x,t)$ as real-analytic functions satisfying $\alpha^*(x,t)=\alpha(x,t)^*$ (i.e., the endpoint $\alpha^*(x,t)$ is the complex conjugate of the other endpoint $\alpha(x,t)$) on some domain of $(x,t)\in\mathbb{R}^2$ having the gradient catastrophe point $(x,t)=(0,t_\mathrm{gc})$ on its boundary.  Recall the notation $\alpha_\mathrm{gc}:=\alpha(0,t_\mathrm{gc})=\ee^{\ii\theta}$.  The Jacobian of the system of equations $M=I=0$ with respect to $(\alpha,\alpha^*)$ is proportional to $H(w;\alpha,\alpha^*,x,t)$ evaluated at $w=\alpha$, and by Definition~\ref{def:GradientCatastrophe} the gradient catastrophe point $(x,t)=(0,t_\mathrm{gc})$ is a point at which this Jacobian vanishes because $H(\alpha_\mathrm{gc};\alpha_\mathrm{gc},\alpha_\mathrm{gc}^*,0,t_\mathrm{gc})=0$.  Thus it is not generally possible to define $\alpha$ and $\alpha^*$ as smooth functions of $(x,t)$ near $(0,t_\mathrm{gc})$ from the conditions $M=I=0$.  Likewise, the $g$-function as obtained in Section~\ref{sec:g-pre-breaking} is generally not well defined on any neighborhood of $(0,t_\mathrm{gc})$.  Since by Definition~\ref{def:GradientCatastrophe} the gradient catastrophe point is simple in the sense that 
\begin{equation}
H'(w;\alpha_\mathrm{gc},\alpha_\mathrm{gc}^*,0,t_\gc)\neq 0\quad\text{for}\quad w=\alpha_\mathrm{gc},
\label{eq:simplicity-condition}
\end{equation}
we will be able to modify the construction in a suitable way.  
Combining \eqref{eq:simplicity-condition} with the analysis in Section~\ref{sec:x=0-Symmetry} then gives the following result:
\begin{lemma}
$\arg(H'(w;\alpha_\mathrm{gc},\alpha_\mathrm{gc}^*,0,t_\mathrm{gc}))=-\tfrac{5}{2}\theta$ when $w=\alpha_\mathrm{gc}=\ee^{\ii\theta}$.
\label{lemma:arg-H'}
\end{lemma}

\subsubsection{A more singular ansatz for $g$}
To get a $g$-function that depends on $(x,t)$ near $(0,t_\mathrm{gc})$ in a real-analytic fashion with the help of condition \eqref{eq:simplicity-condition},
we need to allow 
$g$ to be more singular at $w=\alpha$ and $w=\alpha^*$.  Observe that another solution of the same defining conditions ($k(w)-g_+(w)-g_-(w)=\text{constant}$ on $\beta$ and $g_+(w)+g_-(w)=\text{constant}$ for $w>0$) but with more singular behavior at $w=\alpha$ and $w=\alpha^*$ than \eqref{eq:gprime-1} is
\begin{equation}
g'(w)=\frac{R(w)}{\sqrt{-w}}\left[\frac{n}{w-\alpha}+\frac{n^*}{w-\alpha^*}\right] + g_0'(w)
\end{equation}
where $g_0'(w)$ is the expression on the right-hand side of \eqref{eq:gprime-1}, and where $n$ and $n^*$ are complex parameters to be determined later.  As before, we impose on this formula the condition of integrability at $w=\infty$ which takes the form
\begin{equation}
M:=8n+8n^* + M_0=0
\label{eq:M2}
\end{equation}
where $M_0$ is the expression defined in \eqref{eq:M}.  We then once again obtain $g(w)$ by integration of $g'(w)$ from $w=0$ (see \eqref{eq:g-integral}), and to ensure that $g_+(w)+g_-(w)=0$ for all $w>0$ and that the constant $\Phi$ in the identity $k(w)-g_+(w)-g_-(w)=\ii\Phi$ for $w\in\beta$ is real when $\alpha^*$ is the conjugate of $\alpha$ and $x$ and $t$ are real we impose the condition 
\begin{multline}
I:=-\left[\int_{\beta}\frac{R_+(\xi)\,\dd\xi}{\sqrt{-\xi}(\xi-\alpha)}+\int_{\beta^*}\frac{R_-(\xi)\,\dd\xi}{\sqrt{-\xi}(\xi-\alpha)}\right]n\\
{}-\left[\int_{\beta}\frac{R_+(\xi)\,\dd\xi}{\sqrt{-\xi}(\xi-\alpha^*)}+\int_{\beta^*}\frac{R_-(\xi)\,\dd\xi}{\sqrt{-\xi}(\xi-\alpha^*)}\right]n^* + I_0=0,
\label{eq:I2}
\end{multline}
where $I_0$ is the expression defined in \eqref{eq:I}.

We have thus gone from an unsolvable problem with two equations and two unknowns $(\alpha,\alpha^*)$ to a generalized problem with two equations and four unknowns $(\alpha,\alpha^*,n,n^*)$.  Upon introduction of suitable other constraints this generalized problem will be solvable for $(x,t)$ near $(0,t_\gc)$, but for now we note that the two equations in force (\eqref{eq:M2} and \eqref{eq:I2}) can be used to explicitly eliminate $n$ and $n^*$.  Indeed, these conditions are linear in $n$ and $n^*$ and hence take the form of a $2\times 2$ linear system:
\begin{equation}
\mathbf{L}\begin{bmatrix}n\\n^*\end{bmatrix}=\begin{bmatrix}-M_0\\-I_0\end{bmatrix}.
\label{eq:mmstar}
\end{equation}
The matrix $\mathbf{L}$ depends analytically on $\alpha$ and $\alpha^*$ but not on $(x,t)$.
It is shown in \cite[Lemma 4.1.1]{Lu2018} (see also \eqref{eq:L-detL} below) that $\det(\mathbf{L})\neq 0$ when $\alpha=\alpha_\gc$ and $\alpha^*=\alpha^*_\gc$ form a complex-conjugate pair of distinct complex numbers.  Since $M_0=I_0=0$ when in addition $(x,t)=(0,t_\gc)$, it follows that at the gradient catastrophe point, $n=n_\gc=0$ and $n^*=n_\gc^*=0$.  We have the following result.

\begin{lemma}
The linear system \eqref{eq:mmstar} determines $n$ and $n^*$ as smooth functions of $(\alpha,\alpha^*,x,t)$ near $(\alpha_\gc,\alpha^*_\gc,0,t_\gc)$ with $n(\alpha_\gc,\alpha^*_\gc,0,t_\gc)=n^*(\alpha_\gc,\alpha^*_\gc,0,t_\gc)=0$.  Moreover, the solution has the properties (the strut notation indicates evaluation at $(\alpha,\alpha^*,x,t)=(\alpha_\gc,\alpha^*_\gc,0,t_\gc)$):
\begin{equation}
\left.\frac{\partial n}{\partial\alpha}\right|_\gc=\left.\frac{\partial n}{\partial\alpha^*}\right|_\gc=\left.\frac{\partial n^*}{\partial\alpha}\right|_\gc=\left.\frac{\partial n^*}{\partial \alpha^*}\right|_\gc=0
\label{eq:mmstar-alphaalphastar}
\end{equation}
and
\begin{equation}
\left.\frac{\partial n}{\partial x}\right|_\gc = \left.\frac{\partial n^*}{\partial x}\right|_\gc = -\frac{1}{8}
\label{eq:mmstar-x}
\end{equation}
and
\begin{equation}
\left.\frac{\partial n}{\partial t}\right|_\gc=\frac{\ii}{8}\left[\frac{A_\gc}{B_\gc\sin(\theta)} +\cot(\theta)\right]\quad\text{and}\quad\left.\frac{\partial n^*}{\partial t}\right|_\gc=-\frac{\ii}{8}\left[\frac{A_\gc}{B_\gc\sin(\theta)}+\cot(\theta)\right],
\label{eq:mmstar-t}
\end{equation}
where $A=A(\alpha,\alpha^*)$ and $B=B(\alpha,\alpha^*)$ are defined in \eqref{eq:AB} and are here evaluated at $\alpha=\alpha_\gc=\ee^{\ii\theta}$ and $\alpha^*=\alpha_\gc^*=\ee^{-\ii\theta}$ as the subscript notation indicates.
\label{lemma:mmstar}
\end{lemma}
\begin{proof}
The existence of a unique smooth solution of \eqref{eq:mmstar} follows by explicit computation and $\det(\mathbf{L})\neq 0$ at the catastrophe point.  This also shows that $n(\alpha_\gc,\alpha^*_\gc,0,t_\gc)=n^*(\alpha_\gc,\alpha^*_\gc,0,t_\gc)=0$.  

Letting $n$ and $n^*$ depend smoothly on $(\alpha,\alpha^*,x,t)$ in this way, we differentiate \eqref{eq:mmstar} implicitly with respect to $\alpha$ and evaluate at the gradient catastrophe point to annihilate the terms proportional to $n$ and $n^*$ where the partial derivatives fall on $\mathbf{L}$.  Thus we obtain
\begin{equation}
\mathbf{L}_\gc\begin{bmatrix}\left.\partial_\alpha n\right|_\gc \\\left.\partial_\alpha n^*\right|_\gc\end{bmatrix}=\begin{bmatrix}-\left.\partial_\alpha M_0\right|_\gc\\-\left.\partial_\alpha I_0\right|_\gc\end{bmatrix}.
\end{equation}
According to \cite[Proposition 4.5]{BuckinghamMiller2013}, $\partial_\alpha M_0$ and $\partial_\alpha I_0$ are both proportional to $H(\alpha)$ which vanishes by definition at the gradient catastrophe point.  Hence the right-hand side of the above system vanishes and since $\mathbf{L}$ is invertible at the catastrophe point, it follows that $\partial_\alpha n$ and $\partial_\alpha n^*$ vanish at the gradient catastrophe point.  Similar calculations apply to the derivatives with respect to $\alpha^*$ and thus \eqref{eq:mmstar-alphaalphastar} is proven.

By residue computations at $w=0$ and $w=\infty$ we see that
\begin{equation}
\frac{\partial M_0}{\partial x}=1+\frac{1}{\sqrt{\alpha\alpha^*}}\quad\text{and}\quad
\frac{\partial M_0}{\partial t}=1-\frac{1}{\sqrt{\alpha\alpha^*}}.
\end{equation}
Let
\begin{equation}
A:=\int_{\beta}\frac{\sqrt{-\xi}\,\dd\xi}{R_+(\xi)} +\int_{\beta^*}\frac{\sqrt{-\xi}\,\dd\xi}{R_-(\xi)}\quad\text{and}\quad
B:=\int_{\beta}\frac{\dd\xi}{R_+(\xi)\sqrt{-\xi}} +\int_{\beta^*}\frac{\dd\xi}{R_-(\xi)\sqrt{-\xi}}.
\label{eq:AB}
\end{equation}
When $\alpha$ and $\alpha^*$ form a conjugate pair, both $A$ and $B$ are real-valued, and as a complete elliptic integral of the first kind, $B\neq 0$ .  Then it is straightforward to show that
\begin{equation}
\mathbf{L}=\begin{bmatrix}8 & 8\\A +\alpha^*B & A+\alpha B\end{bmatrix} \quad\text{which implies}\quad
\det(\mathbf{L})=8(\alpha-\alpha^*)B,
\label{eq:L-detL}
\end{equation}
and also that
\begin{equation}
\frac{\partial I_0}{\partial x}=\frac{2A + (\alpha+\alpha^*)B}{8\sqrt{\alpha\alpha^*}}\quad\text{and}\quad
\frac{\partial I_0}{\partial t}=-\frac{\partial I_0}{\partial x}=-\frac{2A+(\alpha+\alpha^*)B}{8\sqrt{\alpha\alpha^*}}.
\end{equation}
Then, since by implicit differentiation and evaluation at the gradient catastrophe point,
\begin{equation}
\mathbf{L}_\gc\begin{bmatrix}\left.\partial_{x,t} n\right|_\gc\\\left.\partial_{x,t} n^*\right|_\gc\end{bmatrix}=\begin{bmatrix}-\left.\partial_{x,t} M_0\right|_\gc\\-\left.\partial_{x,t}I_0\right|_\gc\end{bmatrix},
\end{equation}
the results \eqref{eq:mmstar-x} and \eqref{eq:mmstar-t} follow immediately.
\end{proof}
Going forward, we consider $n$ and $n^*$ to be eliminated by means of Lemma~\ref{lemma:mmstar} and thus the more singular version of $g$ still depends parametrically on $\alpha$ and $\alpha^*$ as well as $(x,t)\in\mathbb{R}^2$.

\subsubsection{Redetermination of $\alpha(x,t)$ and $\alpha^*(x,t)$ via construction of a conformal mapping $W:U\to\mathbb{C}$}
Now we take the function $\widetilde{\phi}(w)$ to be defined in terms of the $g$-function specified as above, for which the endpoints $\alpha$ and $\alpha^*$ are yet to be determined.  We will determine them, along with auxiliary parameters $s$ and $s^*$, as smooth functions of $(x,t)$ near $(0,t_\gc)$ so that the identity 
\begin{equation}
\frac{1}{2}\widetilde{\phi}(w)=\frac{4}{5}W^{5/2} + sW^{1/2}, \quad w\in U\setminus(\partial\mathrm{IV}\cap\partial\mathrm{V}).
\label{eq:conformal-map-identity}
\end{equation}
defines a conformal mapping $W$ on a neighborhood $U$ of $w=\alpha$.
We formalize this statement in a lemma.
\begin{lemma}
There exist well-defined real-analytic functions $\alpha=\alpha(x,t)$, $\alpha^*=\alpha(x,t)^*$, and $s=s(x,t)$, and a conformal mapping $W(\cdot;x,t):U\to\mathbb{C}$ defined for $(x,t)$ in a neighborhood of $(x,t)=(0,t_\mathrm{gc})\in\mathbb{R}^2$, such that the identity \eqref{eq:conformal-map-identity} holds, and such that $\alpha(0,t_\mathrm{gc})=\alpha_\mathrm{gc}=\ee^{\ii\theta}$ and $s(0,t_\mathrm{gc})=0$.
\label{lemma:ConformalMap}
\end{lemma}

\begin{proof}
Since $\widetilde{\phi}(w)$ is related to $\phi(w)-\ii\Phi$ by analytic continuation that simply moves the branch cut, we will have $\widetilde{\phi}'(w)=\widetilde{R}(w)H(w)$ for $w\in U$, where we re-define $H(w)=H(w;\alpha,\alpha^*,x,t)$ as
\begin{equation}
H(w):=-\frac{2}{\sqrt{-w}}\left[\frac{n}{w-\alpha}+\frac{n^*}{w-\alpha^*}\right]+H_0(w),
\end{equation}
with $H_0(w)=H_0(w;\alpha,\alpha^*,x,t)$ being given by the expression on the right-hand side of the second equation in \eqref{eq:RH} and $n$ and $n^*$ depending on $(\alpha,\alpha^*,x,t)$ as explained in Lemma~\ref{lemma:mmstar}, and where $\widetilde{R}(w)=\widetilde{R}(w;\alpha,\alpha^*)$ is another branch of the square root of $(w-\alpha)(w-\alpha^*)$.  Namely, $\widetilde{R}(w)=-w+\mathcal{O}(1)$ as $w\to\infty$ and has its branch cut in $U$ and $U^*$ coinciding locally with the pre-images under $W$ and $W^*$ of the negative real axis.  
Now, since $W$ is yet to be properly defined, we will temporarily make a concrete choice of branch cut for $\widetilde{R}(w)$ as the unique circular arc with endpoints $w=\alpha$ and $w=\alpha^*$ that contains $w=-1$.  Later, once $W$ is properly defined we will deform the branch cut to correspond to the negative real axis as above.  With the temporary choice of cut, we may write $\widetilde{R}(w)=-(w-\alpha)^{1/2}(w-\alpha^*)^{1/2}$ where each square root factor is cut on the union of the half line $w\le -1$ and the sub-arc of the cut for $\widetilde{R}$ connecting the branch point with $w=-1$, and has argument ranging from $-\pi$ to $\pi$ as $w\to\infty$.

The function $H(w)$ has a simple pole at $w=\alpha$.  Its Laurent expansion about $w=\alpha$ is
\begin{equation}
	H(w)=\frac{H^{(-1)}}{w-\alpha} + \sum_{n=0}^\infty H^{(n)}(w-\alpha)^n,
\end{equation}
where the coefficients $H^{(j)}$ are explicit functions of $(\alpha,\alpha^*,x,t)$.  The first few coefficients are:
\begin{equation}
	\begin{split}
		H^{(-1)}=&-\frac{2n}{\sqrt{-\alpha}},\\
		H^{(0)}=&\frac{n}{\alpha\sqrt{-\alpha}}-\frac{2n^*}{(\alpha-\alpha^*)\sqrt{-\alpha}} + H_0(\alpha),\\
		H^{(1)}=&-\frac{3n}{4\alpha^2\sqrt{-\alpha}}+\frac{(3\alpha-\alpha^*)n^*}{(\alpha-\alpha^*)^2\alpha\sqrt{-\alpha}} + H_0'(\alpha).
	\end{split}
\end{equation}
Since $\widetilde{\phi}(\alpha)=0$, by integrating its derivative we find
\begin{equation}
\begin{split}
\frac{1}{2}\widetilde{\phi}(w)&=\frac{1}{2}\int_{\alpha}^w\widetilde{R}(\xi)H(\xi)\,\dd\xi\\
=&-\frac{1}{2}\int_{\alpha}^w(\xi-\alpha^*)^{1/2}H(\xi)(\xi-\alpha)^{1/2}\,\dd\xi\\
=&-\frac{1}{2}\int_\alpha^w(\alpha-\alpha^* + (\xi-\alpha))^{1/2}\left[\frac{H^{(-1)}}{\xi-\alpha}+H^{(0)} + H^{(1)}(\xi-\alpha)+\cdots\right](\xi-\alpha)^{1/2}\,\dd\xi\\
=&-\frac{1}{2}\int_0^{w-\alpha}(\alpha-\alpha^*+z)^{1/2}\left[\frac{H^{(-1)}}{z}+H^{(0)} + H^{(1)}z + \cdots\right]z^{1/2}\,\dd z\\
=&-(\alpha-\alpha^*)^{1/2}\left[H^{(-1)}q +\frac{1}{3}\left(H^{(0)} +\frac{H^{(-1)}}{2(\alpha-\alpha^*)}\right)q^3 \right. \\
& +\left.\frac{1}{5}\left(H^{(1)}+\frac{H^{(0)}}{2(\alpha-\alpha^*)}-\frac{H^{(-1)}}{8(\alpha-\alpha^*)^2}\right)q^5+\cdots\right]\\
=&f_1q+f_3q^3+f_5q^5+\cdots
\end{split}
\label{eq:phi-q-Expansion}
\end{equation}
with coefficients $f_{2k+1}=f_{2k+1}(\alpha,\alpha^*,x,t)$ for $k=0,1,2,\dots$,
where $q:=(w-\alpha)^{1/2}$ with the branch interpreted in relation to the temporary choice of cuts as indicated above.  In other words, for $w\in U$, the function $f(q;\alpha,\alpha^*,x,t):=\tfrac{1}{2}\widetilde{\phi}(w)$ is an odd analytic function of the variable $q$.  Note that the factor $(\alpha-\alpha^*)^{1/2}$ has to be interpreted as the branch $(w-\alpha^*)^{1/2}$ defined relative to the temporary choice of cut as above, evaluated at $w=\alpha$.  Hence when $\alpha^*$ is the conjugate of $\alpha\in\mathbb{C}_+$, $\arg((\alpha-\alpha^*)^{1/2})=\tfrac{1}{4}\pi$.  

Suppose that $(x,t)=(0,t_\mathrm{gc})$ and that $\alpha=\alpha_\gc:=\ee^{\ii\theta}$ and $\alpha^*=\alpha^*_\gc:=\ee^{-\ii\theta}$, which implies by definition of the gradient catastrophe point that $H_0(\alpha_\mathrm{gc})=0$.  
By the condition \eqref{eq:simplicity-condition}
we have $H_0'(\alpha_\gc)\neq 0$ for $(\alpha,\alpha^*,x,t)=(\alpha_\gc,\alpha^*_\gc,0,t_\gc)$, so applying
Lemma~\ref{lemma:mmstar} shows that $H^{(-1)}=H^{(0)}=0$ while $H^{(1)}=H_0'(\alpha_\mathrm{gc})$ at the indicated values of $(\alpha,\alpha^*,x,t)$, and therefore
\begin{equation}
f_1(\alpha_\gc,\alpha^*_\gc,0,t_\gc)=f_3(\alpha_\gc,\alpha^*_\gc,0,t_\gc)=0\quad\text{but}\quad
f_5(\alpha_\gc,\alpha^*_\gc,0,t_\gc)=-\frac{1}{5}H_0'(\alpha_\mathrm{gc})\ee^{\ii\pi/4}\sqrt{2\sin(\theta)}\neq 0.
\label{eq:f135-at-gc}
\end{equation}
Moreover, invoking 
\eqref{eq:simplicity-condition} again and using Lemma~\ref{lemma:arg-H'} shows that
\begin{equation}
\arg(f_5(\alpha_\mathrm{gc},\alpha_\mathrm{gc}^*,0,t_\mathrm{gc}))=-\frac{3}{4}\pi-\frac{5}{2}\theta.
\label{eq:arg-f5-gc}
\end{equation}
The expansion \eqref{eq:phi-q-Expansion} of $\tfrac{1}{2}\widetilde{\phi}(w)$ and the way the coefficients behave at the gradient catastrophe point as indicated in \eqref{eq:f135-at-gc} suggests that the function $\tfrac{1}{2}\widetilde{\phi}(w)$ might be modeled by an expression such as that appearing on the right-hand side of \eqref{eq:conformal-map-identity} in which $w\mapsto W$ is conformal, under the condition that $s=0$ at the catastrophe point.  

It also follows easily from Lemma~\ref{lemma:mmstar} and 
\eqref{eq:simplicity-condition} that
\begin{equation}
\frac{\partial f_1}{\partial\alpha}(\alpha_\mathrm{gc},\alpha^*_\mathrm{gc},0,t_\mathrm{gc})=\frac{\partial f_1}{\partial\alpha^*}(\alpha_\mathrm{gc},\alpha^*_\mathrm{gc},0,t_\mathrm{gc})=0.
\label{eq:f1-partials}
\end{equation}
Furthermore, using in addition differential identities such as those recorded in \cite[Eqns.\@ (4.82)--(4.83)]{BuckinghamMiller2013}, one checks that
\begin{equation}
\frac{\partial f_3}{\partial\alpha}(\alpha_\mathrm{gc},\alpha^*_\mathrm{gc},0,t_\mathrm{gc})=-\frac{1}{2}(\alpha_\mathrm{gc}-\alpha_\mathrm{gc}^*)^{1/2}H_0'(\alpha_\mathrm{gc};\alpha_\mathrm{gc},\alpha^*_\mathrm{gc},0,t_\mathrm{gc})\neq 0\quad\text{but}\quad
\frac{\partial f_3}{\partial\alpha^*}(\alpha_\mathrm{gc},\alpha^*_\mathrm{gc},0,t_\mathrm{gc})=0.
\label{eq:f3-partials}
\end{equation}

We now describe how \eqref{eq:conformal-map-identity} can be solved for $W=W(w;x,t)$ provided $\alpha=\alpha(x,t)$, $\alpha^*=\alpha^*(x,t)$, and $s=s(x,t)$ are suitably chosen for $(x,t)$ near $(0,t_\mathrm{gc})$.  Note that the functions $\alpha(x,t)$ and $\alpha^*(x,t)$ that result will generally differ from the similarly-named functions defined via the equations $M_0=I_0=0$; in particular, the redefined functions will be real-analytic in $(x,t)$ near $(0,t_\mathrm{gc})$.  However we insist that the functions agree at the catastrophe point:  $\alpha(0,t_\mathrm{gc})=\alpha_\mathrm{gc}=\ee^{\ii\theta}$ and $\alpha^*(0,t_\mathrm{gc})=\alpha_\mathrm{gc}^*=\ee^{-\ii\theta}$.
Introducing $Q:=W^{1/2}$, the right-hand side of the  relation \eqref{eq:conformal-map-identity} becomes a polynomial in $Q$, and therefore we may write the desired relation in the form
\begin{equation}
f(q;\alpha,\alpha^*,x,t) = \frac{4}{5}Q^5 + sQ.
\label{eq:Q-eqn}
\end{equation}
First, assuming that $s=0$ and $(\alpha,\alpha^*,x,t)=(\alpha_\mathrm{gc},\alpha_\mathrm{gc}^*,0,t_\mathrm{gc})$, we appeal to \eqref{eq:f135-at-gc} and hence solve this equation for $Q$ as a function of $q$ by taking 
\begin{equation}
Q_\mathrm{gc}(q):=Q(q;0,t_\mathrm{gc})=\left[\frac{5}{4}f_5(\alpha_\mathrm{gc},\alpha_\mathrm{gc}^*,0,t_\mathrm{gc})\right]^{1/5}\left[1+\frac{f_7(\alpha_\mathrm{gc},\alpha_\mathrm{gc}^*,0,t_\mathrm{gc})}{f_5(\alpha_\mathrm{gc},\alpha_\mathrm{gc}^*,0,t_\mathrm{gc})}q^2+\cdots\right]^{1/5}q.
\label{eq:Qgc}
\end{equation}
Here, the first factor is ambiguous up to a phase factor $\ee^{2\pi\ii k/5}$, $k\in\mathbb{Z}\pmod 5$, while the second factor is meant to be an even analytic function of $q$ that equals $1$ when $q=0$.  Using \eqref{eq:arg-f5-gc} shows that 
\begin{equation}
\arg\left(\left[\frac{5}{4}f_5(\alpha_\mathrm{gc},\alpha_\mathrm{gc}^*,0,t_\mathrm{gc})\right]^{1/5}\right)=\frac{2}{5}\pi k-\frac{3}{20}\pi-\frac{1}{2}\theta,\quad k\in\mathbb{Z}\pmod 5.
\label{eq:arg-QgcPrime}
\end{equation}
Using $W=Q^2$ and $q^2=w-\alpha$ for $\alpha=\alpha_\mathrm{gc}$ shows that $W_\mathrm{gc}(w):=W(w;0,t_\mathrm{gc})$ is an odd analytic function of $w-\alpha_\mathrm{gc}$ near $w=\alpha_\mathrm{gc}$ for which
\begin{equation}
W_\mathrm{gc}'(\alpha_\mathrm{gc})=\left[\frac{5}{4}f_5(\alpha_\mathrm{gc},\alpha_\mathrm{gc}^*,0,t_\mathrm{gc})\right]^{2/5}\neq 0,
\end{equation}
and hence $W_\mathrm{gc}(w)$ is locally univalent and $\arg(W_\mathrm{gc}'(\alpha_\mathrm{gc}))=\tfrac{4}{5}\pi k -\tfrac{3}{10}\pi-\theta$ for $k\in\mathbb{Z}\pmod 5$.  We now choose $k=1$, which gives 
\begin{equation}
\arg(W_\mathrm{gc}'(\alpha_\mathrm{gc}))=\frac{1}{2}\pi-\theta.
\label{eq:arg-Wprime-gc}
\end{equation}
With this choice, it is easy to see that $W_\mathrm{gc}(w)$ is a conformal mapping of a sufficiently small neighborhood $U$ of $\alpha_\mathrm{gc}$ under which the preimage of $W\in\mathbb{R}$ is an analytic arc tangent to the unit circle at $w=\alpha_\mathrm{gc}=\ee^{\ii\theta}$.  See Figure~\ref{fig:TangentAngle}.
\begin{figure}[h]
\begin{center}
\includegraphics{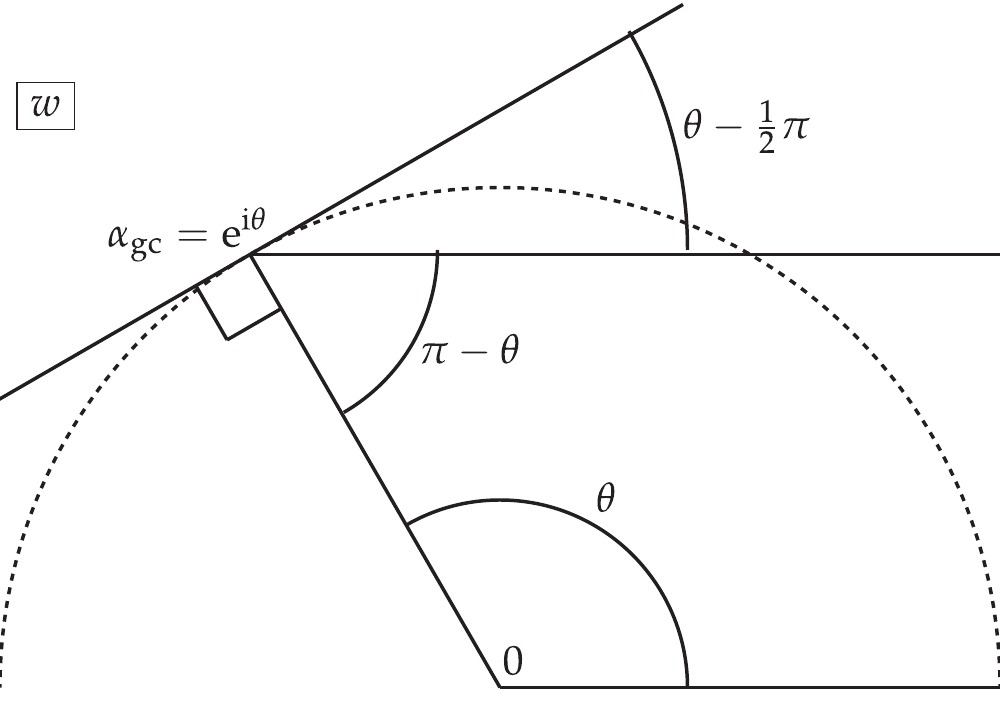}
\end{center}
\caption{Since $\alpha_\mathrm{gc}=\ee^{\ii\theta}$ lies on the unit circle, which also contains $\partial\mathrm{IV}\cap\partial\mathrm{V}\cap U$ and $\ell^-_1\cap U$, $W'_\mathrm{gc}(\alpha_\mathrm{gc})=\ii\ee^{-\ii\theta}|W'_\mathrm{gc}(\alpha_\mathrm{gc})|$.}
\label{fig:TangentAngle}
\end{figure}
In fact, the analysis in Section~\ref{sec:x=0-Symmetry} proves that the local preimage of $\mathbb{R}$ lies exactly on the unit circle.

We will continue the solution $W=W_\mathrm{gc}(w)$ to $(x,t)\neq(0,t_\mathrm{gc})$ by integrating a differential equation defined on a suitable Banach space.  Letting $D_\rho$ denote the open disk given by $|q|<\rho$, the Banach space $\mathcal{B}$ we use is the space of $(\alpha,\alpha^*,s,s^*,Q(\cdot),Q^*(\cdot))$ where $(\alpha,\alpha^*,s,s^*)\in \mathbb{C}^4$ and where for $\rho>0$ sufficiently small, $Q(\cdot)$ and $Q^*(\cdot)$ are odd functions $D_\rho\to\mathbb{C}$ analytic on $D_\rho$ and continuous on $\overline{D_\rho}$.  The norm on this space is taken as
\begin{equation}
\|(\alpha,\alpha^*,s,s^*,Q(\cdot),Q^*(\cdot))\|:= |\alpha|+|\alpha^*| + |s|+|s^*|+\max_{|q|\le\rho}|Q(q)| + \max_{|q|\le\rho}|Q^*(q)|.
\label{eq:BanachSpaceNorm}
\end{equation}
By the maximum modulus principle, the latter maxima can be computed over the circle $|q|=\rho$.
Note that, since $Q_\mathrm{gc}(\cdot)$ and its Schwarz reflection $Q^*_\mathrm{gc}(\cdot)$ are univalent on $D_\rho$ (by choice of $\rho$ sufficiently small), Rouch\'e's theorem guarantees that there exists $\delta>0$ small enough that whenever $Q(\cdot)$ and $Q^*(\cdot)$ correspond to a point in the open ball $B_\delta\subset\mathcal{B}$ of radius $\delta$ about $(\alpha_\mathrm{gc},\alpha^*_\mathrm{gc},0,0,Q_\mathrm{gc}(\cdot),Q^*_\mathrm{gc}(\cdot))\in\mathcal{B}$, then $Q$ and $Q^*$ are also univalent on $D_\rho$.
We formulate a vector field on $B_\delta$ in the direction of an arbitrary linear combination $\partial:=c_x\partial_x+c_t\partial_t$ of coordinate vector fields on $(x,t)\in\mathbb{C}^2$ as follows.  Supposing that $(\alpha,\alpha^*,s,s^*,Q(\cdot),Q^*(\cdot))$ depend on $(x,t)$, we formally differentiate \eqref{eq:Q-eqn} and its complex conjugate $f^*(q;\alpha,\alpha^*,x,t)=\tfrac{4}{5}Q^*(q)^5+s^*Q^*(q)$ (on the left-hand side we mean the complex conjugate of the function $f(q^*;\alpha^*,\alpha,x,t)$ and on the right-hand side we have unknowns $Q^*(q)$ and $s^*$ that will later turn out to be the complex conjugates of $Q(q^*)$ and $s^*$ respectively) in the chosen direction and solve for $\partial Q(\cdot)$ and $\partial Q^*(\cdot)$ to obtain
\begin{equation}
\begin{split}
\partial Q(\cdot)&=
\frac{\partial f(\cdot;\alpha,\alpha^*,x,t)+f_\alpha(\cdot;\alpha,\alpha^*,x,t)\partial\alpha +f_{\alpha^*}(\cdot;\alpha,\alpha^*,x,t)\partial\alpha^*-Q(\cdot)\partial s}{4Q(\cdot)+s}  \\
\partial Q^*(\cdot)&=\frac{\partial f^*(\cdot;\alpha,\alpha^*,x,t)+f^*_\alpha(\cdot;\alpha,\alpha^*,x,t)\partial\alpha +f^*_{\alpha^*}(\cdot;\alpha,\alpha^*,x,t)\partial\alpha^*-Q^*(\cdot)\partial s^*}{4Q^*(\cdot)+s^*}.
\end{split}
\label{eq:QQstar-flow}
\end{equation}
Here in the first term in each numerator, the differential operator $\partial$ acts on the explicit $(x,t)$-dependence in $f$ and $f^*$.  In order for $\partial Q(\cdot)$ and $\partial Q^*(\cdot)$ to be odd analytic functions, it is necessary that the analytic functions in the numerators vanish at each of the zeros of the corresponding denominators, to at least the same order.  At a point of $B_\delta$ we use univalence of $Q(\cdot)$ and $Q^*(\cdot)$ to see that if $s$ (resp.,\@ $s^*$) is small and nonzero then there are exactly four simple roots of $4Q(q)^4+s$, namely $q=Q^{-1}(\pm\Gamma)$ and $q=Q^{-1}(\pm\ii\Gamma)$ where $\Gamma^4=-\tfrac{1}{4}s$ (resp.,\@ there are exactly four simple roots of $4Q^*(q)^4+s^*$, namely $q=Q^{*-1}(\pm\Gamma^*)$ and $q=Q^{*-1}(\pm\ii\Gamma^*)$ where $\Gamma^{*4}=-\tfrac{1}{4}s^*$). Therefore since $f$ and $f^*$ are odd functions of $q$, we need only require that
\begin{equation}
\begin{bmatrix}
\Gamma & 0 & -f_\alpha(Q^{-1}(\Gamma);\diamond)&-f_{\alpha^*}(Q^{-1}(\Gamma);\diamond)\\
\ii\Gamma & 0 & -f_\alpha(Q^{-1}(\ii\Gamma);\diamond)& -f_{\alpha^*}(Q^{-1}(\ii\Gamma);\diamond)\\
0 & \Gamma^* & -f^*_\alpha(Q^{*-1}(\Gamma^*);\diamond)&-f^*_{\alpha^*}(Q^{*-1}(\Gamma^*);\diamond)\\
0 & -\ii\Gamma^* & -f^*_\alpha(Q^{*-1}(-\ii\Gamma^*);\diamond) & -f^*_{\alpha^*}(Q^{*-1}(-\ii\Gamma^*);\diamond)\end{bmatrix}
\begin{bmatrix}\partial s\\
\partial s^*\\
\partial \alpha\\
\partial \alpha^*
\end{bmatrix}
=
\begin{bmatrix}
\partial f(Q^{-1}(\Gamma);\diamond)\\
\partial f(Q^{-1}(\ii\Gamma);\diamond)\\
\partial f^*(Q^{*-1}(\Gamma^*);\diamond)\\
\partial f^*(Q^{*-1}(-\ii\Gamma^*);\diamond)
\end{bmatrix}
\label{eq:system-1}
\end{equation}
where for brevity we are using $\diamond$ to denote the parameter list $\alpha,\alpha^*,x,t$.  We can introduce the contour integral representations (Lagrange-B\"urmann formul\ae):
\begin{equation}
\begin{split}
F(Q^{-1}(\Gamma);\diamond)&=\frac{1}{2\pi\ii}\oint \frac{F(q;\diamond) Q'(q)\,\dd q}{Q(q)-\Gamma}=
\frac{\Gamma}{2\pi\ii}\oint\frac{F(q;\diamond)Q'(q)\,\dd q}{Q(q)^2-\Gamma^2}\\
F(Q^{-1}(\ii\Gamma);\diamond)&=\frac{1}{2\pi\ii}\oint\frac{F(q;\diamond)Q'(q)\,\dd q}{Q(q)-\ii\Gamma} = 
\frac{\ii\Gamma}{2\pi\ii}\oint\frac{F(q;\diamond)Q'(q)\,\dd q}{Q(q)^2+\Gamma^2}\\
F^*(Q^{*-1}(\Gamma^*);\diamond)&=\frac{1}{2\pi\ii}\oint\frac{F^*(q;\diamond)Q^{*\prime}(q)\,\dd q}{Q^*(q)-\Gamma^*}=\frac{\Gamma^*}{2\pi\ii}\oint\frac{F^*(q;\diamond)Q^{*\prime}(q)\,\dd q}{Q^*(q)^2-\Gamma^{*2}}\\
F^*(Q^{*-1}(-\ii\Gamma^*);\diamond)&=\frac{1}{2\pi\ii}\oint\frac{F^*(q;\diamond)Q^{*\prime}(q)\,\dd q}{Q^*(q)+\ii\Gamma^*}=\frac{-\ii\Gamma^*}{2\pi\ii}\oint\frac{F^*(q;\diamond)Q^{*\prime}(q)\,\dd q}{Q^*(q)^2+\Gamma^{*2}},
\end{split}
\end{equation}
where $F$ denotes either $f_\alpha$, $f_{\alpha^*}$, or $\partial f$, and similarly for $F^*$.  Here the contour of integration is the circle $|q|=\rho$ with positive orientation, and the second integral in each case is derived from the first by averaging of each integral with that obtained by the substitution $q\mapsto -q$ using oddness of $F$, $F^*$, $Q$, and $Q^*$ and evenness of $Q'$ and $Q^{*\prime}$.  We then replace the first two rows of \eqref{eq:system-1} with their independent linear combinations having coefficients $(\tfrac{1}{2}\Gamma^{-3},\tfrac{1}{2}\ii\Gamma^{-3})$ and $(\tfrac{1}{2}\Gamma^{-1},-\tfrac{1}{2}\ii\Gamma^{-1})$ respectively, and similarly replace the last two rows of \eqref{eq:system-1} with their independent linear combinations having coefficients $(\tfrac{1}{2}\Gamma^{*-3},-\tfrac{1}{2}\ii\Gamma^{*-3})$ and $(\tfrac{1}{2}\Gamma^{*-1},\tfrac{1}{2}\ii\Gamma^{*-1})$ respectively, which results in the system
\begin{equation}
\begin{bmatrix}
0 & 0 & -4A_0 & -4B_0\\
1 & 0 & -4A_2 & -4B_2\\
0 & 0 & -4A^*_0 & -4 B^*_0\\
0 & 1 & -4A^*_2 & -4 B^*_2\end{bmatrix}\begin{bmatrix}\partial s\\\partial s^*\\\partial\alpha\\\partial\alpha^*\end{bmatrix}=
\begin{bmatrix}4C_0\\
4C_2\\
4C^*_0\\
4C^*_2\end{bmatrix},
\label{eq:system-2}
\end{equation}
where, using $\Gamma^4=-\tfrac{1}{4}s$ and $\Gamma^{*4}=-\tfrac{1}{4}s^*$, for $p=0,2$ we define
\begin{equation}
\begin{split}
A_p(s,\alpha,\alpha^*,x,t)&:=\frac{1}{2\pi\ii}\oint \frac{f_\alpha(q;\alpha,\alpha^*,x,t)Q(q)^pQ'(q)\,\dd q}{4Q(q)^4+s}\\
B_p(s,\alpha,\alpha^*,x,t)&:=\frac{1}{2\pi\ii}\oint \frac{f_{\alpha^*}(q;\alpha,\alpha^*,x,t)Q(q)^pQ'(q)\,\dd q}{4Q(q)^4+s}\\
C_p(s,\alpha,\alpha^*,x,t)&:=\frac{1}{2\pi\ii}\oint \frac{\partial f(q;\alpha,\alpha^*,x,t)Q(q)^pQ'(q)\,\dd q}{4Q(q)^4+s}\\
A_p^*(s^*,\alpha,\alpha^*,x,t)&:=\frac{1}{2\pi\ii}\oint \frac{f^*_\alpha(q;\alpha,\alpha^*,x,t)Q^*(q)^pQ^{*\prime}(q)\,\dd q}{4Q^*(q)^4+s^*}\\
B^*_p(s^*,\alpha,\alpha^*,x,t)&:=\frac{1}{2\pi\ii}\oint \frac{f^*_{\alpha^*}(q;\alpha,\alpha^*,x,t)Q^*(q)^pQ^{*\prime}(q)\,\dd q}{4Q^*(q)^4+s^*}\\
C^*_p(s^*,\alpha,\alpha^*,x,t)&:=\frac{1}{2\pi\ii}\oint \frac{\partial f^*(q;\alpha,\alpha^*,x,t)Q^*(q)^pQ^{*\prime}(q)\,\dd q}{4Q^*(q)^4+s^*}.
\end{split}
\end{equation}
We take the solution of \eqref{eq:system-2} as the finite-dimensional part of our system of differential equations:
\begin{equation}
\begin{split}
\partial s&=4C_2+\frac{(A_2B_0^*-A_0^*B_2)4C_0}{\Delta}+\frac{(A_0B_2-A_2B_0)4C_0^*}{\Delta}\\
\partial s^*&= 4C_2^*+\frac{(A_2^*B_0^*-A_0^*B_2^*)4C_0}{\Delta} +\frac{(A_0B_2^*-A_2^*B_0)4C_0^*}{\Delta}\\
\partial\alpha &=\frac{B_0^*C_0-B_0C_0^*}{\Delta}\\
\partial\alpha^*&=\frac{A_0C_0^*-A_0^*C_0}{\Delta},
\end{split}
\label{eq:finite-dimensional-part}
\end{equation}
where $\Delta$ denotes a constant multiple of the determinant of the coefficient matrix in \eqref{eq:system-2}:
\begin{equation}
\Delta=\Delta(\alpha,\alpha^*,s,s^*,Q(\cdot),Q^*(\cdot),x,t):=A_0^*B_0-A_0B_0^*.
\end{equation}
Adjoining to \eqref{eq:finite-dimensional-part} the two equations \eqref{eq:QQstar-flow} in which \eqref{eq:finite-dimensional-part} has been used to explicitly eliminate $\partial\alpha$, $\partial\alpha^*$, $\partial s$, and $\partial s^*$ on the right-hand side finally defines the vector field on $B_\delta$ corresponding to the directional derivative $\partial$ in the $(x,t)$-space.  

One then needs to show that if $\delta>0$ is sufficiently small, the vector field on $B_\delta$ is Lipschitz with respect to the norm \eqref{eq:BanachSpaceNorm}.  This is fairly straightforward, since by a residue calculation,
\begin{equation}
\Delta(\alpha_\mathrm{gc},\alpha^*_\mathrm{gc},0,0,Q_\mathrm{gc}(\cdot),Q^*_\mathrm{gc}(\cdot),x,t)=-
\left|\frac{1}{4Q'_\mathrm{gc}(0)^3}\frac{\partial f_{3}}{\partial\alpha}(\alpha_\mathrm{gc},\alpha^*_\mathrm{gc},0,t_\mathrm{gc})\right|^2<0
\end{equation}
according to \eqref{eq:f3-partials}, and since all other denominators are bounded away from zero due to integration on the circle $|q|=\rho>0$.  See \cite[Appendix A.4]{BuckinghamMiller2015} for full details of the argument in a similar setting.  Therefore, taking the initial condition at $(x,t)=(0,t_\mathrm{gc})$ to be given by $(\alpha_\mathrm{gc},\alpha_\mathrm{gc}^*,0,0,Q_\mathrm{gc}(\cdot),Q_\mathrm{gc}^*(\cdot))\in B_\delta$, we get a unique local solution to the differential equation.  Finally, we observe that the differential equations corresponding to different choices of the constants in the linear combination $\partial:=c_x\partial_x+c_t\partial_t$ are all compatible (i.e., the corresponding vector fields on $B_\delta$ all commute pairwise), since they were obtained from differentiation of the same identity in different directions.  This implies that the solution $(\alpha(x,t),\alpha^*(x,t),s(x,t),s^*(x,t),Q(\cdot;x,t),Q^*(\cdot;x,t))$ is well-defined and in fact analytic in $(x,t)$ near $(0,t_\mathrm{gc})$, and by uniqueness we learn that for real $(x,t)$ near $(0,t_\mathrm{gc})$, $\alpha^*(x,t)$ and $s^*(x,t)$ are the complex conjugates of $\alpha(x,t)$ and $s(x,t)$, and $Q^*(\cdot;x,t)$ is the Schwarz reflection of $Q(\cdot;x,t)$.  We finish the construction of the conformal mapping by recalling $W=Q^2$ and $q^2=w-\alpha$, which yields $W(w;x,t)$ as an odd analytic function of $w$ near $\alpha=\alpha(x,t)$ as defined near $(x,t)=(0,t_\mathrm{gc})$ as part of the solution of the differential equation.  Since $W'(\alpha(x,t);x,t)\neq 0$ by continuity and the corresponding result $W'_\mathrm{gc}(\alpha_\mathrm{gc})\neq 0$, univalence of $W$ near $w=\alpha(x,t)$ persists in a neighborhood of $(0,t_\mathrm{gc})$.  This completes the construction of the conformal mapping and the proof of Lemma~\ref{lemma:ConformalMap}.
\end{proof}

\subsubsection{Values of the partial derivatives of $s(x,t)$ at the gradient catastrophe}
The existence of the conformal map $W$ and complex quantities $\alpha$, $\alpha^*$, and $s$ as real analytic functions of $(x,t)$ near $(0,t_\gc)$, together with the initial conditions for these quantities at the gradient catastrophe point, is enough to determine all of their derivatives with respect to $(x,t)$ at the catastrophe point, simply by substitution into the defining relation, which we may take as \eqref{eq:Q-eqn}.  This reasoning allows us to prove the following. 

Note that for $\alpha=\alpha_\gc$, \eqref{eq:general-outer-parametrix-at-origin} implies
\begin{equation}
m_\gc=\sin^2(\tfrac{1}{2}\mathrm{arg}(\alpha_\gc))=\sin^2(\tfrac{1}{2}\theta).
\label{eq:mgc-alphagc}
\end{equation}
\begin{lemma}
The real analytic function $s:\mathbb{R}^2\to\mathbb{C}$ satisfies
\begin{equation}
\frac{\partial s}{\partial x}(0,t_\mathrm{gc})=\ii a\quad\text{and}\quad\frac{\partial s}{\partial t}(0,t_\mathrm{gc})=b
\end{equation}
where
\begin{equation}
a=-\frac{(m_\mathrm{gc}(1-m_\mathrm{gc}))^{1/4}}{2|W'_\mathrm{gc}(\alpha_\mathrm{gc})|^{1/2}}<0\quad\text{and}\quad
b=\frac{(m_\mathrm{gc}(1-m_\mathrm{gc}))^{1/4}}{2|W'_\mathrm{gc}(\alpha_\mathrm{gc})|^{1/2}}\rho(m_\mathrm{gc})>0,
\label{eq:a-and-b-formulae}
\end{equation}
in which $\rho(m_\mathrm{gc})>0$ is given explicitly by \eqref{eq:rho-func-define}
in terms of $\KK(m)$ and $\EE(m)$ in \eqref{eq:KE}, the complete elliptic integrals of the first and second kinds, respectively.
\label{lemma:PartialsOfs}
\end{lemma}  
\begin{proof}
Let us assume $Q=Q(q;x,t)$, $s=s(x,t)$, $\alpha=\alpha(x,t)$ and $\alpha^*=\alpha^*(x,t)$ in \eqref{eq:Q-eqn}, and firstly take just the terms linear in $q$ on both sides (this is the same as differentiating in $q$ and setting $q$ to zero).  Since $Q(q;x,t)$ vanishes linearly at $q=0$, we get
\begin{equation}
f_1(\alpha(x,t),\alpha^*(x,t),x,t)=s(x,t)Q'(0;x,t),
\end{equation}
where $f_1$ is defined by \eqref{eq:phi-q-Expansion}.
Differentiation of this identity on $\mathbb{R}^2$ with respect to $x$ gives
\begin{multline}
\frac{\partial f_1}{\partial x}(\alpha(x,t),\alpha^*(x,t),x,t) +
\frac{\partial f_1}{\partial\alpha}(\alpha(x,t),\alpha^*(x,t),x,t)\frac{\partial\alpha}{\partial x}(x,t) + 
\frac{\partial f_1}{\partial\alpha^*}(\alpha(x,t),\alpha^*(x,t),x,t)\frac{\partial\alpha^*}{\partial x}(x,t) \\{}= \frac{\partial s}{\partial x}(x,t)Q'(0;x,t) + s(x,t)\frac{\partial Q'}{\partial x}(0;x,t).
\label{eq:determine-sx}
\end{multline}
Evaluating both sides at $(x,t)=(0,t_\gc)$, using the identities \eqref{eq:f1-partials} and $s(0,t_\mathrm{gc})=0$, gives
\begin{equation}
\frac{\partial s}{\partial x}(0,t_\gc)=\frac{1}{Q_\mathrm{gc}'(0)}\frac{\partial f_1}{\partial x}(\alpha_\mathrm{gc},\alpha^*_\mathrm{gc},0,t_\mathrm{gc}).
\end{equation}
Next, applying Lemma~\ref{lemma:mmstar} to the definition \eqref{eq:phi-q-Expansion} and using \eqref{eq:mgc-alphagc}, we obtain
\begin{equation}
\begin{split}
\frac{\partial f_1}{\partial x}(\alpha_\gc,\alpha^*_\gc,0,t_\gc)&=\frac{2(\alpha_\gc-\alpha^*_\gc)^{1/2}}{\sqrt{-\alpha_\gc}}\frac{\partial n}{\partial x}(\alpha_\gc,\alpha^*_\gc,0,t_\gc)\\ & = -\frac{(\alpha_\gc-\alpha_\gc^*)^{1/2}}{4\sqrt{-\alpha_\gc}}
\\ &=\ee^{-\ii\pi/4}\ee^{-\ii\theta/2}\frac{1}{2}(m_\mathrm{gc}(1-m_\mathrm{gc}))^{1/4},
\end{split}
\label{eq:dalphaf1=0}
\end{equation}
and then we recall that $Q_\mathrm{gc}'(0)^2=W_\mathrm{gc}'(\alpha_\mathrm{gc})$ so that $|Q_\mathrm{gc}'(0)|=|W_\mathrm{gc}'(\alpha_\mathrm{gc})|^{1/2}$, and that $\arg(Q_\mathrm{gc}'(0))=\tfrac{1}{4}\pi-\tfrac{1}{2}\theta$ according to \eqref{eq:Qgc} and 
\eqref{eq:arg-QgcPrime} with index $k=1$.  Therefore
\begin{equation}
\frac{\partial s}{\partial x}(0,t_\mathrm{gc})=-\ii \frac{(m_\mathrm{gc}(1-m_\mathrm{gc}))^{1/4}}{2|W_\mathrm{gc}'(\alpha_\mathrm{gc})|^{1/2}}
\end{equation}
as desired.

Differentiating instead with respect to $t$, we obtain 
\begin{equation}
\frac{\partial s}{\partial t}(0,t_\gc)=\frac{1}{Q_\mathrm{gc}'(0)}\frac{\partial f_1}{\partial t}(\alpha_\gc,\alpha^*_\gc,0,t_\gc),
\end{equation}
in which we again use Lemma~\ref{lemma:mmstar} in \eqref{eq:phi-q-Expansion} to obtain
\begin{equation}
\begin{split}
\frac{\partial f_1}{\partial t}(\alpha_\mathrm{gc},\alpha^*_\mathrm{gc},0,t_\mathrm{gc})&=
\frac{\ii}{4}\frac{(\alpha_\mathrm{gc}-\alpha_\mathrm{gc}^*)^{1/2}}{\sqrt{-\alpha_\mathrm{gc}}}
\left[\frac{A_\mathrm{gc}}{B_\mathrm{gc}\sin(\theta)}+\cot(\theta)\right] \\ &= -\ee^{\ii\pi/4}\ee^{-\ii\theta/2}\frac{1}{2}(m_\mathrm{gc}(1-m_\mathrm{gc}))^{1/4}\left[\frac{A_\mathrm{gc}}{B_\mathrm{gc}\sin(\theta)}+\cot(\theta)\right].
\end{split}
\end{equation}
Now, the value of the quantity in square brackets above can be determined by suitable contour deformations, which will be explained in the remark following the proof of Lemma~\ref{lemma:X-T-identities} in Section~\ref{sec:Darboux} below, with the result that it is precisely the strictly negative quantity $-\rho(m_\mathrm{gc})$ defined in \eqref{eq:rho-func-define}.  Dividing through by $Q'_\mathrm{gc}(0)$ then completes the proof.
\end{proof}
This result shows that locally near $(x,t)=(0,t_\mathrm{gc})$ we may regard the function $s(x,t)$ as a nearly-linear mapping from $\mathbb{R}^2$ to $\mathbb{R}^2\simeq\mathbb{C}$ taking $(0,t_\mathrm{gc})$ to $s=0$ that consists of reflection about the $t$-axis followed by clockwise rotation about $(0,t_\mathrm{gc})$ by $-90^\circ$, followed in turn by independent scalings of the vertical and horizontal coordinates by positive scale factors $-a$ and $b$ respectively.  If the scale factors were to be equal, the mapping would be locally conformal (and in fact anti-holomorphic), but generally these two scalings are not the same, making $(x,t)\mapsto (\mathrm{Re}\{s(x,t)\},\mathrm{Im}\{s(x,t)\})$ a real-differentiable mapping only.  
That $-b/a=\rho(m_\mathrm{gc})\neq 1$ in general is clear from the plots in Figure~\ref{fig:Minus-b-Over-a}.
\begin{figure}[h]
\begin{center}
\includegraphics{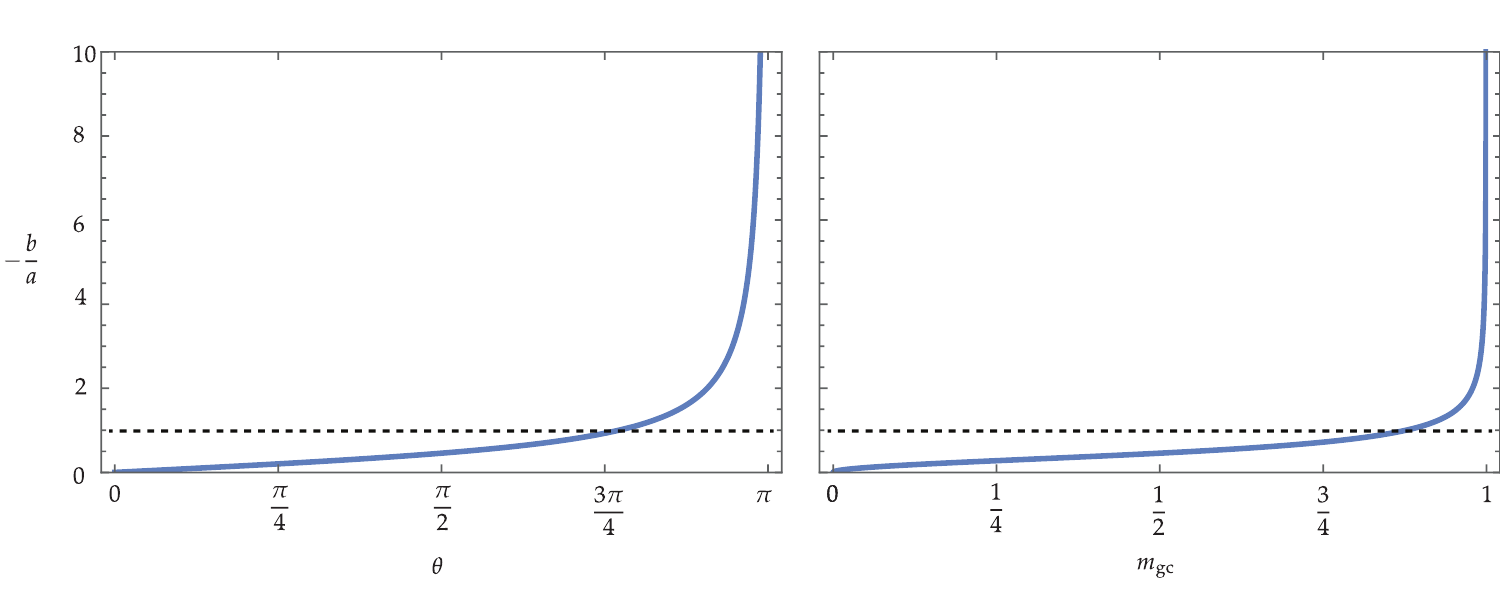}
\end{center}
\caption{The ratio $-b/a=\rho(m_\mathrm{gc})>0$ as a function of $\theta=\arg(\alpha_\mathrm{gc})$ (left) and as a function of $m_\mathrm{gc}=\sin^2(\tfrac{1}{2}\theta)$ (right).  The value $-b/a=1$ is shown with a dotted horizontal line.}
\label{fig:Minus-b-Over-a}
\end{figure}

\subsection{The phase-linearized outer parametrix}
\label{sec:phase-linearized-outer-parametrix}
In order to study the Riemann--Hilbert problem for $\mathbf{O}(w)$ for $(x,t)$ near the gradient catastrophe point $(0,t_\mathrm{gc})$, we will again need an outer parametrix to deal with the jump conditions outside the disks $U$ and $U^*$, but we should keep in mind that the quantities $\alpha(x,t)$ and $\Phi(x,t)$ that parametrize the outer parametrix are now generally different functions of $(x,t)$ since they will be constructed from the modified $g$-function described in Section~\ref{sec:modified-g}.  The modified functions $\alpha(x,t)$ and $\Phi(x,t)$ are related to the modified $g$-function in the same way their original versions were related to the original $g$-function:  $\alpha(x,t)$ is the endpoint of the contour arc $\beta$ and $\Phi(x,t)$ is the real constant (with respect to $w$) value taken by $-\ii (k(w)-g_+(w)-g_-(w))$ along $\beta$.  However, unlike the original definitions of these functions, since they are now related to the modified $g$-function, they are well-defined and smooth functions of $(x,t)$ in a full neighborhood of the gradient catastrophe point. Although a modified $g$-function is involved, the quantities
\begin{equation}
\alpha(0,t_\mathrm{gc}),\quad\alpha(0,t_\mathrm{gc})^*,\quad\Phi(0,t_\mathrm{gc}),\quad
\frac{\partial\Phi}{\partial x}(0,t_\mathrm{gc}),\quad\text{and}\quad\frac{\partial\Phi}{\partial t}(0,t_\mathrm{gc})
\end{equation}
are all equal to their limiting values at $(x,t)=(0,t_\mathrm{gc})$ as computed from the $g$-function appropriate for $(x,t)$ away from the gradient catastrophe point.

Since we are now interested in points $(x,t)$ lying a small distance proportional to $\epsilon^{4/5}$ from $(0,t_\mathrm{gc})$, a further simplification is appropriate.  We do the following for the purpose of defining the outer parametrix: 
\begin{itemize}
\item We \emph{fix the endpoints $\alpha$ and $\alpha^*$} of $\beta$ and $\beta^*$ respectively at their values at the gradient catastrophe point $\alpha(0,t_\mathrm{gc})=\alpha_\mathrm{gc}=\ee^{\ii\theta}$ and $\alpha^*(0,t_\mathrm{gc})=\alpha_\mathrm{gc}^*=\ee^{-\ii\theta}$.  The resulting contours with endpoints $w=1$ and $w=\ee^{\pm\ii\theta}$ will be denoted $\beta_\mathrm{gc}$ and $\beta^*_\mathrm{gc}$ respectively.
\item We \emph{linearize the function} $\Phi(x,t)$ about $(x,t)=(0,t_\mathrm{gc})$, replacing $\Phi(x,t)$ with the linear function 
\begin{equation}
\Phi^l(t):=\Phi_\mathrm{gc}-(t-t_\mathrm{gc})\omega_\mathrm{gc}, 
\label{eq:Phi-l}
\end{equation}
with $\Phi_\mathrm{gc}:=\Phi(0,t_\mathrm{gc})$ and $\omega_\mathrm{gc}:=\omega(0,t_\mathrm{gc})$ where generally $\omega(x,t):=-\partial_t\Phi(x,t)$ (note that $k_\mathrm{gc}=k(0,t_\mathrm{gc})=0$ where $k(x,t):=\partial_x\Phi(x,t)$).  Note that according to \cite[Proposition 4.2]{BuckinghamMiller2013}, 
\begin{equation}
\omega_\mathrm{gc}=-\frac{\pi}{2\KK(m_\mathrm{gc})},
\label{eq:omega0}
\end{equation}
where $\KK(m)$ is given in \eqref{eq:KE}, while $m_\gc$ is defined by \eqref{eq:mgc-alphagc}.
\end{itemize}
We refer to the corresponding outer parametrix determined by these quantities as the \emph{phase-linearized outer parametrix} and denote it by $\dot{\mathbf{O}}^{\mathrm{out,l}}(w)$.  The phase-linearized outer parametrix is given precisely by the formula $\dot{\mathbf{O}}^{\mathrm{out,l}}(w):=\mathbf{Y}(w;\beta_\mathrm{gc},\epsilon^{-1}\Phi^l(t))$ in terms of the solution of Riemann--Hilbert Problem~\ref{rhp:OuterParametrixGeneral} described in some detail in Appendix~\ref{app:theta}.  
Note that $\dot{\mathbf{O}}^\mathrm{out,l}(w)$ depends parametrically on $t$ and $\epsilon$ but not on $x$.  

Since the branch points $\alpha_\mathrm{gc}$ and $\alpha_\mathrm{gc}^*$ for $\dot{\mathbf{O}}^\mathrm{out,l}(w)$ are fixed while as defined in terms of $g$ the arcs $\beta$ and $\beta^*$ depend on $(x,t)$ and have moving endpoints $\alpha(x,t)$ and $\alpha^*(x,t)=\alpha(x,t)^*$ respectively, to compare $\dot{\mathbf{O}}^\mathrm{out,l}(w)$ with $\mathbf{O}(w)$ we will make use of the fact that the jump matrix for both of these matrix functions is analytic on $\beta\cup\beta^*$, so we can move both jump contours a bit and in particular ensure that they coincide outside of the neighborhood $U$ of $w=\alpha(x,t)$ and its Schwarz reflection $U^*$.  We will specify in Section~\ref{sec:inner-parametrix} below a certain $(x,t)$-dependent point we will denote $w_\beta\in\partial U$; then we may and will assume that the $\beta$ arcs for $\mathbf{O}(w)$ and $\dot{\mathbf{O}}^\mathrm{out,l}(w)$ coincide outside of $U\cup U^*$ as a single curve joining $w_\beta$ with $w_\beta^*$ via $w=1$.  Within $U$, the two $\beta$ arcs must be different in general as they have different terminal points; however the terminal point for $\dot{\mathbf{O}}^\mathrm{out,l}(w)$ will be irrelevant while that for $\mathbf{O}(w)$ will be specified precisely in Section~\ref{sec:inner-parametrix} below.

The simplification obtained with the use of the phase-linearized outer parametrix comes at the cost that although both $\dot{\mathbf{O}}^\mathrm{out,l}(w)$ and $\mathbf{O}(w)$ have jumps across the same curve $\beta\cup\beta^*$ outside of $U\cup U^*$, $\dot{\mathbf{O}}^\mathrm{out,l}(w)$ does not exactly satisfy the same jump condition there as does $\mathbf{O}(w)$:  indeed, instead of 
\begin{equation}
\mathbf{O}_+(w)=\mathbf{O}_-(w)\ii\sigma_1\ee^{\ii\Phi(x,t)\sigma_3/\epsilon},\quad w\in\beta,\quad w\not\in U,
\end{equation}
we have
\begin{equation}
\dot{\mathbf{O}}^\mathrm{out,l}_+(w)=\dot{\mathbf{O}}^\mathrm{out,l}_-(w)\ii\sigma_1\ee^{\ii\Phi^l(t)\sigma_3/\epsilon},\quad w\in\beta,\quad w\not\in U,
\end{equation}
with similar Schwarz-symmetric jump conditions holding on $\beta^*$.
However, since by Proposition~\ref{prop:PropertiesOfY} $\dot{\mathbf{O}}^\mathrm{out,l}(w)$ has unit determinant and is uniformly bounded when $w\not\in U\cup U^*$, for such $w$ also along $\beta$ we have the jump condition
\begin{equation}
\begin{split}
\mathbf{O}_+(w)\dot{\mathbf{O}}^\mathrm{out,l}_+(w)^{-1}&=
\mathbf{O}_-(w)\dot{\mathbf{O}}^\mathrm{out,l}_-(w)^{-1}\cdot
\dot{\mathbf{O}}^\mathrm{out,l}_-(w)\ee^{-\ii(\Phi(x,t)-\Phi^l(t))\sigma_3/\epsilon}\dot{\mathbf{O}}^\mathrm{out,l}_-(w)^{-1} \\ &= \mathbf{O}_-(w)\dot{\mathbf{O}}^\mathrm{out,l}_-(w)^{-1}\cdot\left(\mathbb{I} + \mathcal{O}(\epsilon^{-1}(x^2+(t-t_\mathrm{gc})^2))\right),\quad w\in\beta,\quad w\not\in U,
\end{split}
\label{eq:comparison-outside}
\end{equation}
as $\epsilon\to 0$, because the linearization $\Phi^l(t)$ approximates $\Phi(x,t)$ up to terms of second order.  Restriction to $(x,t)$ in a neighborhood of radius $\epsilon^{4/5}$ of the gradient catastrophe point then makes the error term above proportional to $\epsilon^{3/5}\ll 1$.  

\subsection{Use of the modified $g$-function and inner parametrix for $\mathbf{O}(w)$ near $w=\alpha$}
\label{sec:inner-parametrix}

Recall from Lemma~\ref{lemma:ConformalMap} the conformal mapping $w\mapsto W(w;x,t)$ defined on the neighborhood $U$ of $w=\alpha(x,t)$.  We now choose the jump contours for $\widetilde{\mathbf{O}}(w)$ within $U$ to be local pre-images under $W$ of the straight rays $\arg(W)=0,\pm\tfrac{2}{5}\pi$ and $\arg(-W)=0$.  We also assume that the image under $W$ of the arc $\partial\mathrm{II}\cap\partial\mathrm{III}\cap U=\beta\cap U$ lies along the ray $\arg(W)=\tfrac{1}{5}\pi$, and define $w=w_\beta$ as the point at which this arc meets $\partial U$.  Introducing the scaled variables
\[
\xi:=\epsilon^{-2/5}W\quad\text{and}\quad\tau:=\epsilon^{-4/5}s
\]
the jump conditions for $\widetilde{\mathbf{O}}(w)$ within $U$ take the simple form shown in Figure~\ref{fig:PI-Standard-Tritronquee}.
\begin{figure}[h]
\begin{center}
\includegraphics{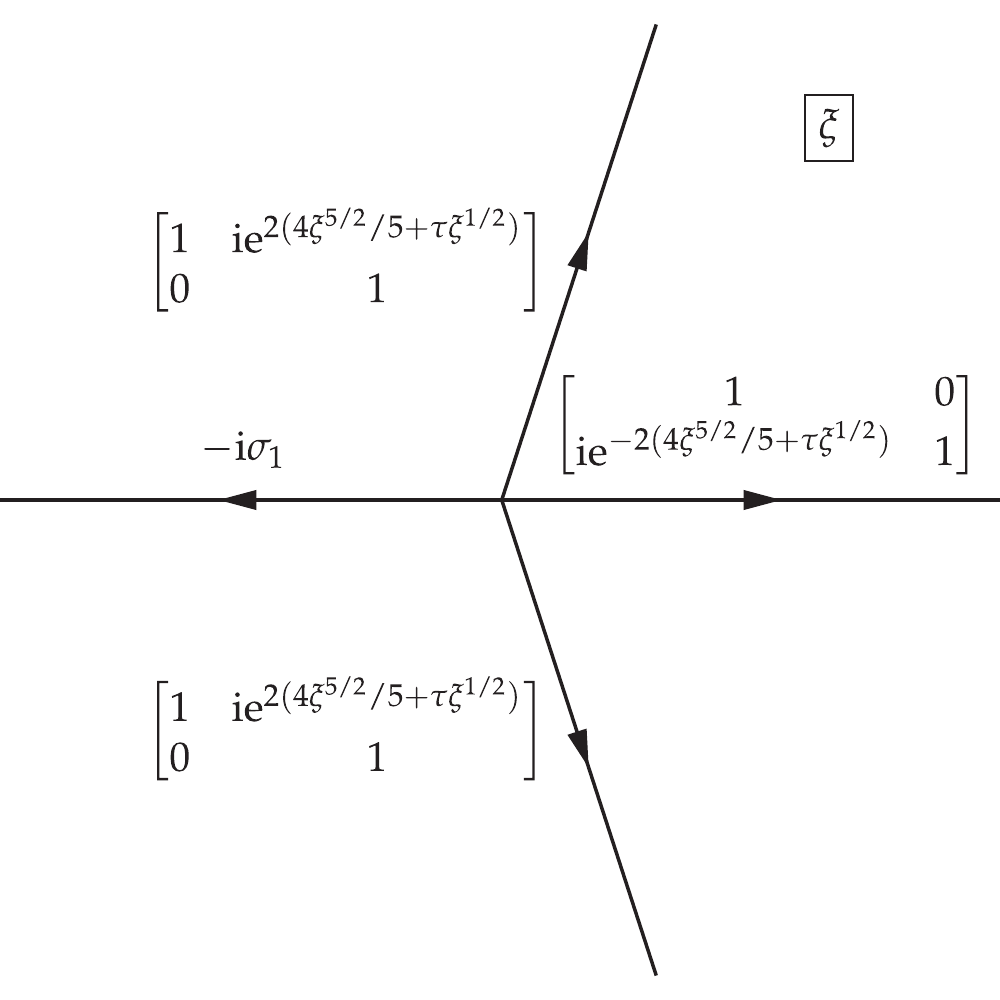}
\end{center}
\caption{The jump conditions satisfied by $\widetilde{\mathbf{O}}(w)$ within $U$ rendered in the $\xi$-plane.  These are also the jump conditions corresponding to the real tritronqu\'ee solution $y(\tau)$ of the Painlev\'e-I equation $y''(\tau)=6y(\tau)^2+\tau$, which is real-valued for real $\tau$ and analytic for $|\arg(-\tau)|<\tfrac{4}{5}\pi$.}
\label{fig:PI-Standard-Tritronquee}
\end{figure}

We will find an exact solution of these jump conditions that matches well onto the phase-linearized outer parametrix $\dot{\mathbf{O}}^\mathrm{out,l}(w)$ for $w\in \partial U$.  To do this, we first express the phase-linearized outer parametrix for $w\in U$ in terms of the conformal coordinate $W(w;x,t)$ \emph{at the gradient catastrophe point} $(0,t_\mathrm{gc})$, i.e., in terms of $W_\mathrm{gc}(w)$.  We apply  to $\dot{\mathbf{O}}^\mathrm{out,l}(w)$ a version of the transformation \eqref{eq:O-Otilde} omitting only the factor $Y_N(w)^{-\sigma_3/2}$ and replacing $\Phi(x,t)$ with $\Phi^l(t)$ to obtain:
\begin{equation}
\widetilde{\mathbf{O}}^\mathrm{out,l}(w):=\begin{cases}
\dot{\mathbf{O}}^\mathrm{out,l}(w)\ee^{-\ii\Phi^l(t)\sigma_3/(2\epsilon)}(-\ii\sigma_1),&\quad w\in\mathrm{I}\cap U\\
\dot{\mathbf{O}}^\mathrm{out,l}(w)\ee^{-\ii\Phi^l(t)\sigma_3/(2\epsilon)}(-\ii\sigma_1),&\quad w\in\mathrm{II}\cap U\\
\dot{\mathbf{O}}^\mathrm{out,l}(w)\ee^{-\ii\Phi^l(t)\sigma_3/(2\epsilon)},&\quad w\in\mathrm{III}\cap U\\
\dot{\mathbf{O}}^\mathrm{out,l}(w)\ee^{-\ii\Phi^l(t)\sigma_3/(2\epsilon)},&\quad w\in\mathrm{IV}\cap U\\
\dot{\mathbf{O}}^\mathrm{out,l}(w)\ee^{-\ii\Phi^l(t)\sigma_3/(2\epsilon)}(-\ii\sigma_1),&\quad w\in\mathrm{V}\cap U.
\end{cases}
\label{eq:Odot-Otilde}
\end{equation}
Within $U$, $\widetilde{\bfO}^{\mathrm{out},l}$ has only the jump $\widetilde{\mathbf{O}}^\mathrm{out,l}_+(w)=\widetilde{\mathbf{O}}^\mathrm{out,l}_-(w)(-\ii\sigma_1)$ along the contour $\partial\mathrm{IV}\cap\partial\mathrm{V}$; so composing with the conformal mapping $W_\mathrm{gc}(w)$ we see that an exact local solution of this jump condition is the matrix 
\begin{equation}
W_\mathrm{gc}(w)^{\sigma_3/4}\mathbf{M}^{-1},\quad\mathbf{M}:=\frac{1}{\sqrt{2}}\begin{bmatrix}1 & -1\\1 & 1\end{bmatrix},
\end{equation}
where the power function is interpreted as the principal branch (cut along the negative real $W$-axis, the image of the jump contour $\partial\mathrm{IV}\cap\partial\mathrm{V}$).  Now define
\begin{equation}
\mathbf{C}(w)=\mathbf{C}(w;\theta,\epsilon^{-1}\Phi^l(t),W_\mathrm{gc}(\cdot)):=\widetilde{\mathbf{O}}^\mathrm{out,l}(w)\mathbf{M}W_\mathrm{gc}(w)^{-\sigma_3/4},\quad w\in U.
\label{eq:CmatrixDefine}
\end{equation}
The parametric dependence on $\theta=\arg(\alpha_\mathrm{gc})$ and $\epsilon^{-1}\Phi^l(t)$ enters via the phase-linearized outer parametrix $\dot{\mathbf{O}}^\mathrm{out,l}(w)$.
The matrix $\mathbf{C}(w)$ is easily seen to be holomorphic within $U$; indeed it has no jump discontinuity along $W<0$ and its singularity at $w=\alpha_\mathrm{gc}=\ee^{\ii\theta}$ is at worst a divergence proportional to $(w-\alpha_\mathrm{gc})^{-1/2}$ hence this isolated singularity is necessarily removable.  Also, although $\dot{\mathbf{O}}^\mathrm{out,l}(w)$ depends on $\epsilon$, it follows from Proposition~\ref{prop:PropertiesOfY} that $\mathbf{C}(w)$ is uniformly bounded on $U$ as $\epsilon\to 0$.

Noting that for $(x,t)=(0,t_\mathrm{gc})+\mathcal{O}(\epsilon^{4/5})$ we have 
$W(w;x,t)=W_\mathrm{gc}(w)(1+\mathcal{O}(\epsilon^{4/5}))$ holding uniformly for $w\in\partial U$, and observing that it is $W(w;x,t)$ that is proportional to the variable $\xi$ in which the jump conditions for $\widetilde{\mathbf{O}}(w)$ take a simple form in $U$, we define an inner parametrix in $U$ as follows.  Since \eqref{eq:CmatrixDefine} can be rearranged to read 
$\widetilde{\mathbf{O}}^\mathrm{out,l}(w)=\mathbf{C}(w)W_\mathrm{gc}(w)^{\sigma_3/4}\mathbf{M}^{-1}$ with $\mathbf{C}(w)$ holomorphic in $U$, we first assume the existence of a solution of the following Riemann--Hilbert problem.
\begin{rhp}[Painlev\'e-I tritronqu\'ee solution]
Given $\tau\in\mathbb{C}$, seek a $2\times 2$ matrix function $\mathbf{T}=\mathbf{T}(\xi;\tau)$ defined in the four sectors $0<\arg(\xi)<\tfrac{2}{5}\pi$, $-\tfrac{2}{5}\pi<\arg(\xi)<0$, $\tfrac{2}{5}\pi<\arg(\xi)<\pi$, and $-\pi<\arg(\xi)<-\tfrac{2}{5}\pi$, that satisfies the following additional conditions:
\begin{itemize}
\item[]\textbf{Analyticity:}  $\mathbf{T}(\xi;\tau)$ is analytic for $\xi$ in each sector of definition, taking continuous boundary values from each sector on the union of rays forming its boundary.
\item[]\textbf{Jump conditions:}  Let each excluded ray be oriented in the direction away from the origin and denote the boundary value taken by $\bfT(\xi;\tau)$ on the ray from the left (resp., right) by $\bfT_+(\xi;\tau)$ (resp., $\bfT_-(\xi;\tau)$). Then the boundary values satisfy $\mathbf{T}_+(\xi;\tau)=\mathbf{T}_-(\xi;\tau)\mathbf{V}_\mathbf{T}(\xi;\tau)$ where the matrix $\mathbf{V}_\mathbf{T}(\xi;\tau)$ is defined on each ray as shown in Figure~\ref{fig:PI-Standard-Tritronquee}.
\item[]\textbf{Normalization:}  The matrix $\mathbf{T}(\xi;\tau)$ satisfies the normalization condition
\begin{equation}
\lim_{\xi\to\infty}\mathbf{T}(\xi;\tau)\mathbf{M}\xi^{-\sigma_3/4}=\mathbb{I}.
\label{eq:T-norm}
\end{equation}
\end{itemize}
\label{rhp:tritronquee}
\end{rhp}
Given a solution $\mathbf{T}(\xi;\tau)$ of Riemann--Hilbert Problem~\ref{rhp:tritronquee}, we build an inner parametrix for $\mathbf{O}(w)$ by defining
\begin{equation}
\dot{\mathbf{O}}^\mathrm{in}(w):=\begin{cases}
\mathbf{C}(w)\epsilon^{\sigma_3/10}\mathbf{T}(\epsilon^{-2/5}W(w;x,t);\epsilon^{-4/5}s(x,t))(\ii\sigma_1)\ee^{\ii\Phi(x,t)\sigma_3/(2\epsilon)}Y_N(w)^{\sigma_3/2},&\quad w\in\mathrm{I}\cap U\\
\mathbf{C}(w)\epsilon^{\sigma_3/10}\mathbf{T}(\epsilon^{-2/5}W(w;x,t);\epsilon^{-4/5}s(x,t))(\ii\sigma_1)\ee^{\ii\Phi(x,t)\sigma_3/(2\epsilon)},&\quad w\in \mathrm{II}\cap U\\
\mathbf{C}(w)\epsilon^{\sigma_3/10}\mathbf{T}(\epsilon^{-2/5}W(w;x,t);\epsilon^{-4/5}s(x,t))\ee^{\ii\Phi(x,t)\sigma_3/(2\epsilon)},&\quad w\in\mathrm{III}\cap U\\
\mathbf{C}(w)\epsilon^{\sigma_3/10}\mathbf{T}(\epsilon^{-2/5}W(w;x,t);\epsilon^{-4/5}s(x,t))\ee^{\ii\Phi(x,t)\sigma_3/(2\epsilon)}Y_N(w)^{\sigma_3/2},&\quad w\in\mathrm{IV}\cap U\\
\mathbf{C}(w)\epsilon^{\sigma_3/10}\mathbf{T}(\epsilon^{-2/5}W(w;x,t);\epsilon^{-4/5}s(x,t))(\ii\sigma_1)\ee^{\ii\Phi(x,t)\sigma_3/(2\epsilon)}Y_N(w)^{\sigma_3/2},&\quad w\in\mathrm{V}\cap U.
\end{cases}
\label{eq:inner-parametrix}
\end{equation}
Since $\mathbf{T}(\xi;\tau)$ satisfies exactly the jump conditions indicated in Figure~\ref{fig:PI-Standard-Tritronquee}, which are the images under the rescaled conformal mapping $\xi=\epsilon^{-2/5}W(w;x,t)$ and the identification $\tau=\epsilon^{-4/5}s(x,t)$ of the jump conditions for $\widetilde{\mathbf{O}}(w)$, upon comparing \eqref{eq:O-Otilde} with \eqref{eq:inner-parametrix} we see that $\dot{\mathbf{O}}^\mathrm{in}(w)$ is an exact local solution of the jump conditions for $\mathbf{O}(w)$ within $U$.  
Also, for $w\in\partial U\cap(\mathrm{II}\cup\mathrm{III})$ it is easy to see that with
\begin{equation}
p(w;x,t):=\frac{W(w;x,t)}{W_\mathrm{gc}(w)},
\label{eq:p-ratio}
\end{equation}
we calculate the mismatch
\begin{multline}
\dot{\mathbf{O}}^\mathrm{in}(w)\dot{\mathbf{O}}^\mathrm{out,l}(w)^{-1}=\\
\mathbf{C}(w)\epsilon^{\sigma_3/10}\mathbf{T}(\xi;\tau)\mathbf{M}\xi^{-\sigma_3/4}\epsilon^{-\sigma_3/10}p(w;x,t)^{\sigma_3/4}\mathbf{C}(w)^{-1},\quad w\in\partial U\cap(\mathrm{II}\cup\mathrm{III}),
\label{eq:mismatch-simple}
\end{multline}
whereas otherwise on the boundary of $U$
\begin{multline}
\dot{\mathbf{O}}^\mathrm{in}(w)\dot{\mathbf{O}}^\mathrm{out,l}(w)^{-1}=\\
\mathbf{C}(w)\epsilon^{\sigma_3/10}\mathbf{T}(\xi;\tau)\mathbf{M}\xi^{-\sigma_3/4}\epsilon^{-\sigma_3/10}p(w;x,t)^{\sigma_3/4}\mathbf{C}(w)^{-1}\cdot\dot{\mathbf{O}}^\mathrm{out,l}(w)Y_N(w)^{\sigma_3/2}\dot{\mathbf{O}}^\mathrm{out,l}(w)^{-1},\\
w\in\partial U\cap(\mathrm{I}\cup\mathrm{IV}\cup\mathrm{V}).
\label{eq:mismatch-Y}
\end{multline}
Since $\dot{\mathbf{O}}^\mathrm{out}(w)$ and its inverse are uniformly bounded independent of $\epsilon$ on $\partial U$, the product of the latter three factors is uniformly $\mathbb{I}+\mathcal{O}(\epsilon)$.  Also note that $p(w;x,t)=1+\mathcal{O}(\epsilon^{4/5})$ holds uniformly on $\partial U$ because $(x,t)=(0,t_\mathrm{gc})+\mathcal{O}(\epsilon^{4/5})$.

\subsection{Characterization of $\mathbf{T}(\xi;\tau)$}
The conditions placed on $\mathbf{T}(\xi;\tau)$ in the formulation of Riemann--Hilbert Problem~\ref{rhp:tritronquee} turn out to relate it to a specific solution of the Painlev\'e-I equation.  Riemann--Hilbert Problem~\ref{rhp:tritronquee} can be recognized in some form in several references, including \cite{Kapaev2004}, \cite{BertolaTovbis2014}, and \cite{BuckinghamMiller2015}.  The latter reference is the closest match, with the main difference being that the spectral variable $\xi$ is replaced throughout with $-\xi$.  However, the description of $\mathbf{T}(\xi;\tau)$ in \cite{BuckinghamMiller2015} contains some inessential errors related to the use of a different matrix $\mathbf{M}$ in the normalization condition \eqref{eq:T-norm}.  For this reason, we take the opportunity to correct the exposition in \cite{BuckinghamMiller2015} and also develop for the reader's convenience the relation between $\mathbf{T}(\xi;\tau)$ and the Painlev\'e-I equation.

\subsubsection{Riemann--Hilbert Problem~\ref{rhp:tritronquee} and the Painlev\'e-I equation}
If $\tau\in\mathbb{C}$ is such that $\mathbf{T}(\xi;\tau)$ exists, then it is unique and satisfies $\det(\mathbf{T}(\xi;\tau))=1$ and $\mathbf{T}(\xi^*;\tau^*)^*=\mathbf{T}(\xi;\tau)$.
Assuming existence, the matrix $\mathbf{L}(\xi;\tau):=\mathbf{T}(\xi;\tau)\ee^{(4\xi^{5/2}/5 + \tau\xi^{1/2})\sigma_3}$ has constant jumps along each ray from which it follows that the matrices 
\begin{equation}
\mathbf{A}(\xi;\tau):=\frac{\partial\mathbf{L}}{\partial\xi}(\xi;\tau)\mathbf{L}(\xi;\tau)^{-1}\quad\text{and}\quad
\mathbf{U}(\xi;\tau):=\frac{\partial\mathbf{L}}{\partial\tau}(\xi;\tau)\mathbf{L}(\xi;\tau)^{-1}
\end{equation}
are entire functions of $\xi$.  They are characterized in terms of the coefficients in the asymptotic expansion of $\mathbf{T}(\xi;\tau)\mathbf{M}\xi^{-\sigma_3/4}$ as $\xi\to\infty$:
\begin{equation}
\mathbf{T}(\xi;\tau)\mathbf{M}\xi^{-\sigma_3/4}\sim\mathbb{I}+\sum_{p=1}^\infty\xi^{-p}\mathbf{T}_p(\tau),\quad\xi\to\infty.
\label{eq:T-expansion}
\end{equation}
One first checks that this expansion (the existence of which given solvability of Riemann--Hilbert Problem~\ref{rhp:tritronquee} follows via standard arguments from exponential convergence of the jump matrices to the identity)  and the definition of $\mathbf{U}(\xi;\tau)$ implies that
\begin{equation}
\mathbf{U}(\xi;\tau)=-\xi\sigma_+-\sigma_--[\mathbf{T}_1(\tau),\sigma_+] + \mathcal{O}(\xi^{-1}),\quad\xi\to\infty,
\end{equation}
where $\sigma_+$ and $\sigma_-$ are given in \eqref{eq:sigmapm}. Therefore by Liouville's theorem $\mathbf{U}(\xi;\tau)$ is linear in $\xi$:
\begin{equation}
\mathbf{U}(\xi;\tau)=-\xi\sigma_+-\sigma_--[\mathbf{T}_1(\tau),\sigma_+].
\end{equation}
Similarly, the definition of $\mathbf{A}(\xi;\tau)$ implies the expansion
\begin{multline}
\mathbf{A}(\xi;\tau)=2\xi\mathbf{U}(\xi;\tau) -2[\mathbf{T}_2(\tau),\sigma_+]+2[\mathbf{T}_1(\tau),\sigma_+]\mathbf{T}_1(\tau)-2[\mathbf{T}_1(\tau),\sigma_-]-\frac{1}{2}\tau\sigma_+ \\{}+ \xi^{-1}\mathbf{A}_1(\tau) +\xi^{-2}\mathbf{A}_2(\tau)+\mathcal{O}(\xi^{-3}),\quad\xi\to\infty
\end{multline}
for some computable coefficients $\mathbf{A}_1(\tau)$ and $\mathbf{A}_2(\tau)$, and therefore by Liouville's theorem again $\mathbf{A}(\xi;\tau)$ is a quadratic polynomial in $\xi$:
\begin{equation}
\mathbf{A}(\xi;\tau)=2\xi\mathbf{U}(\xi;\tau) -2[\mathbf{T}_2(\tau),\sigma_+]+2[\mathbf{T}_1(\tau),\sigma_+]\mathbf{T}_1(\tau)-2[\mathbf{T}_1(\tau),\sigma_-]-\frac{1}{2}\tau\sigma_+
\end{equation}
and the coefficients $\mathbf{A}_1(\tau)$ and $\mathbf{A}_2(\tau)$ both vanish identically.  Using that $A_{1,21}(\tau)$ vanishes (an algebraic identity among the elements of the matrices $\mathbf{T}_p(\tau)$) and also the condition (following from $\det(\mathbf{T}(\xi;\tau))=1$) that $T_{1,11}(\tau)+T_{1,22}(\tau)=0$, the matrices $\mathbf{U}(\xi;\tau)$ and $\mathbf{A}(\xi;\tau)$ can be parametrized by just three functions of $\tau$:
\begin{equation}
\begin{split}
h(\tau)&:=-T_{1,21}(\tau)\\
y(\tau)&:=T_{1,21}(\tau)^2+2T_{1,11}(\tau)\\
z(\tau)&:=2T_{2,21}(\tau)-2T_{1,12}(\tau)+2T_{1,21}(\tau)^3+6T_{1,11}(\tau)T_{1,21}(\tau).
\end{split}
\label{eq:hyz}
\end{equation}
Then we find that
\begin{equation}
\mathbf{U}(\xi;\tau)=\begin{bmatrix}-h&-\xi+h^2-y\\-1 & h\end{bmatrix}
\label{eq:PI-U-matrix}
\end{equation}
and
\begin{equation}
\mathbf{A}(\xi;\tau)=2\xi\mathbf{U}(\xi;\tau) + \begin{bmatrix}z+2hy & -\tau-2y^2-2hz-2h^2y\\
2y & -z-2hy\end{bmatrix},
\label{eq:PI-A-matrix}
\end{equation}
in which $h=h(\tau)$, $y=y(\tau)$, and $z=z(\tau)$ are given by \eqref{eq:hyz}.

Since the mixed partial derivatives of $\mathbf{L}(\xi;\tau)$ must be equal, the Lax pair equations
\begin{equation}
\frac{\partial\mathbf{L}}{\partial\xi}(\xi;\tau)=\mathbf{A}(\xi;\tau)\mathbf{L}(\xi;\tau)\quad\text{and}\quad\frac{\partial\mathbf{L}}{\partial\tau}(\xi;\tau)=\mathbf{U}(\xi;\tau)\mathbf{L}(\xi;\tau)
\label{eq:PI-Lax-pair}
\end{equation}
must be compatible, which implies that the coefficient matrices satisfy the zero curvature condition
\begin{equation}
\frac{\partial\mathbf{A}}{\partial\tau}(\xi;\tau)-\frac{\partial\mathbf{U}}{\partial\xi}(\xi;\tau) + [\mathbf{A}(\xi;\tau),\mathbf{U}(\xi;\tau)]=\mathbf{0}.
\end{equation}
This matrix equation is equivalent, upon separating the coefficients of the different powers of $\xi$, to the first-order system
\begin{equation}
\begin{split}
h'(\tau)&=y(\tau)\\
y'(\tau)&=-z(\tau)\\
z'(\tau)&=-6y(\tau)^2-\tau.
\end{split}
\label{eq:PIsystem}
\end{equation}
Eliminating $z(\tau)$ between the second and third equations yields the Painlev\'e-I equation on $y(\tau)$ in exactly the form \eqref{eq:PI-intro}.
Moreover, taking advantage of the identities $T_{1,11}(\tau)+T_{1,22}(\tau)=0$, $A_{1,11}(\tau)=0$, $A_{1,12}(\tau)=0$, $A_{1,21}(\tau)=0$ (again), and $A_{2,21}(\tau)=0$ we obtain a further algebraic identity among the functions $h$, $y$, and $z$, which amounts to exactly the definition \eqref{eq:Hamiltonian-intro} of $h$ as the Hamiltonian associated with $y$.
It follows that the system \eqref{eq:PIsystem} can be written in Hamiltonian form associated with $y$, in which $y'(\tau)$ is replaced by $-z(\tau)$.
\begin{equation}
\frac{\dd y}{\dd\tau}=\frac{\partial h}{\partial z},\quad\frac{\dd z}{\dd\tau}=-\frac{\partial h}{\partial y},\quad\frac{\dd h}{\dd\tau}=\frac{\partial h}{\partial\tau}.
\end{equation}

\subsubsection{Identification of the real tritronqu\'ee solution}
The paper of Kapaev \cite{Kapaev2004} contains an exhaustive description of the asymptotic behavior for large complex $\tau$ of solutions of the Painlev\'e-I equation \eqref{eq:PI-intro}, and how it depends on the specific choice of solution as encoded in the Stokes multipliers associated to the irregular singular point at $\xi=\infty$ for the linear equation $\mathbf{L}_\xi=\mathbf{A}\mathbf{L}$.  The Stokes multipliers are in turn determined by the jump matrices in the Riemann--Hilbert conditions for $\mathbf{T}(\xi;\tau)$, which are just a formulation of the inverse monodromy problem of determining the coefficients $h,y,z,\tau$ in the matrix $\mathbf{A}$ from the Stokes multipliers.  One can deduce from the paper of Kapaev \cite{Kapaev2004} that the specific solution associated with the jump conditions illustrated in Figure~\ref{fig:PI-Standard-Tritronquee} is the \emph{real tritronqu\'ee} solution, i.e., the unique solution $y=y(\tau)$ of \eqref{eq:PI-intro} that satisfies the asymptotic condition given in \eqref{eq:y-asymp}.
This solution is real-valued for real $\tau$ and has been proven \cite{CostinHuangTanveer2014} to be analytic for all $|\tau|$ in the indicated asymptotically pole-free sector.  What makes the solution a \emph{tritronqu\'ee} is the fact that among the one-parameter family of \emph{tronqu\'ee} solutions having square-root asymptotics in the sector $|\arg(-\tau)|<\tfrac{2}{5}\pi$ it is the unique solution for which the same asymptotic holds in the larger sector indicated in \eqref{eq:y-asymp}.
In fact, the Painlev\'e-I equation \eqref{eq:PI-intro} has exactly five tritronqu\'ee solutions related by the symmetry $(\tau,y)\mapsto (\ee^{2\pi\ii n/5}\tau,\ee^{\ii n\pi/5}y)$ for $n\in\mathbb{Z}\pmod{5}$, and exactly one of which is real-valued for real $\tau$.  Kapaev's paper \cite{Kapaev2004} shows that this discrete symmetry is reflected in a corresponding symmetry of the Stokes multipliers in which the diagram in Figure~\ref{fig:PI-Standard-Tritronquee} is essentially rotated by integer multiples of $\tfrac{2}{5}\pi$ radians.  One of the other four tritronqu\'ee solutions appeared in the work of Bertola and Tovbis \cite{BertolaTovbis2014} for instance.

\subsubsection{Laurent expansions and additional identities}
\label{sec:series}
For later use in Section~\ref{sec:NearThePoles}, we develop some further properties of the functions $h(\tau)$, $y(\tau)$, and $z(\tau)$, and of some other coefficients in the expansion \eqref{eq:T-expansion}.
We first derive from \eqref{eq:PIsystem} the second-order second-degree equation satisfied by the Hamiltonian $h(\tau)$:
\begin{equation}
h''(\tau)^2+2h(\tau)-4h'(\tau)^3-2\tau h'(\tau)=0
\label{eq:PIsigma}
\end{equation}
which is frequently called the ``$\sigma$-form'' of the Painlev\'e-I equation.  Now let $\tau_\mathrm{p}$ be a pole of the real tritronqu\'ee solution $y(\tau)$.  Then also $h(\tau)$ has a pole at $\tau=\tau_\mathrm{p}$, and from \eqref{eq:PIsigma} it is easy to see that the pole is simple, with residue $-1$.  In fact, every solution of \eqref{eq:PIsigma} having a pole at $\tau=\tau_\mathrm{p}$ has a Laurent expansion of the form
\begin{equation}
h(\tau)=-\frac{1}{\tau-\tau_\mathrm{p}} + h_0 + \sum_{n=1}^\infty h_n(\tau-\tau_\mathrm{p})^n,
\label{eq:h-Laurent}
\end{equation}
where the coefficients $\{h_n\}_{n\ge 1}$ are systematically determined in terms of $\tau_\mathrm{p}$ and $h_0$ by substitution into \eqref{eq:PIsigma}.  This series then implies corresponding series for $y(\tau)$ and $z(\tau)$ by the relations (cf., \eqref{eq:PIsystem}) $y(\tau)=h'(\tau)$ and $z(\tau)=-y'(\tau)$.  Hence $y(\tau)$ and $z(\tau)$ have double and triple poles respectively at $\tau=\tau_\mathrm{p}$, but no residue. We will also need an expansion of $T_{2,21}(\tau)$.  To this end,
we can start from the Lax pair equations \eqref{eq:PI-Lax-pair} to define matrices
\begin{equation}
\begin{split}
\mathbf{W}(\xi;\tau)&:=\left(\frac{\partial\mathbf{L}}{\partial\xi}(\xi;\tau)-\mathbf{A}(\xi;\tau)\mathbf{L}(\xi;\tau)\right)\ee^{-(4\xi^{5/2}/5+\tau\xi^{1/2})\sigma_3}\mathbf{M}\xi^{-\sigma_3/4}\\
\mathbf{Z}(\xi;\tau)&:=\left(\frac{\partial\mathbf{L}}{\partial\tau}(\xi;\tau)-\mathbf{A}(\xi;\tau)\mathbf{L}(\xi;\tau)\right)\ee^{-(4\xi^{5/2}/5+\tau\xi^{1/2})\sigma_3}\mathbf{M}\xi^{-\sigma_3/4}
\end{split}
\label{eq:WZ}
\end{equation}
which must vanish identically.  On the other hand, substituting $\mathbf{L}(\xi;\tau)=\mathbf{T}(\xi;\tau)\ee^{(4\xi^{5/2}/5+\tau\xi^{1/2})\sigma_3}$ and \eqref{eq:T-expansion} one sees that $\mathbf{W}(\xi;\tau)$ and $\mathbf{Z}(\xi;\tau)$ both have Laurent expansions about $\xi=\infty$ with coefficients explicitly expressed in terms of $h(\tau)$, $y(\tau)$, $z(\tau)$, and the matrix coefficients $\mathbf{T}_p(\tau)$, $p\ge 1$.  Setting to zero the  $\xi^{-1}$ term in the Laurent expansion of $\mathbf{Z}(\xi;\tau)$ one obtains with the help of the definitions \eqref{eq:hyz} a differential equation for $T_{1,12}(\tau)$:
\begin{equation}
T_{1,12}'(\tau)=h(\tau)z(\tau)+y(\tau)^2+h(\tau)^2y(\tau)+\frac{1}{4}\tau.
\end{equation}
Expanding the right-hand side around $\tau=\tau_\mathrm{p}$ using the computed Laurent expansions of $h(\tau)$, $y(\tau)$, and $z(\tau)$ shows that $T_{1,12}'(\tau)$ generally has a double pole, but has no residue:
\begin{equation}
T_{1,12}'(\tau)=\frac{h_0^2}{(\tau-\tau_\mathrm{p})^2} + \sum_{n=0}^\infty (n+1)t_{n+1}(\tau-\tau_\mathrm{p})^n,
\end{equation}
for some coefficients $\{t_n\}_{n\ge 1}$ that are systematically determined from $h_0$ and $\tau_\mathrm{p}$.
Therefore, the series representation of $T_{1,12}'(\tau)$ can be integrated up to yield
\begin{equation}
T_{1,12}(\tau)=-\frac{h_0^2}{\tau-\tau_\mathrm{p}} +t_0 + \sum_{n=1}^\infty t_n(\tau-\tau_\mathrm{p})^n
\label{eq:T112series}
\end{equation}
where $t_0$ is an integration constant not determined by this procedure.  Next, we set to zero the $\xi^0$ term in the Laurent expansion of $\mathbf{W}(\xi;\tau)$ and obtain that
\begin{equation}
T_{2,21}(\tau)=-\frac{1}{2}h(\tau)^3+T_{1,12}(\tau)+\frac{3}{2}h(\tau)y(\tau)+\frac{1}{2}z(\tau).
\end{equation}
Substituting the already-obtained Laurent series about $\tau=\tau_\mathrm{p}$ for $h(\tau)$, $y(\tau)$, $z(\tau)$, and $T_{1,12}(\tau)$ then gives the result
\begin{equation}
T_{2,21}(\tau)=\frac{1}{2}h_0^2\frac{1}{\tau-\tau_\mathrm{p}}+\left(t_0-\frac{1}{2}h_0^3\right) + \sum_{n=1}^\infty u_n(\tau-\tau_\mathrm{p})^n
\label{eq:T221series}
\end{equation}
where $\{u_n\}_{n\ge 1}$ are determined systematically as well.  

\subsection{Global parametrix for $\mathbf{O}(w)$, error analysis, and expressions for the potentials}
\subsubsection{Global parametrix definition}
To successfully use the phase-linearized outer parametrix we needed the assumption that $(x,t)=(0,t_\mathrm{gc})+\mathcal{O}(\epsilon^{4/5})$ where the difference is measured in the Euclidean norm in the $(x,t)$-plane.  Since $\tau=\epsilon^{-4/5}s(x,t)$ where $s(x,t)$ is a differentiable function $\mathbb{R}^2\to\mathbb{C}$ with $s(0,t_\mathrm{gc})=0$, the same assumption implies that $\tau$ is bounded as $\epsilon\to 0$. 

While it has no poles in the sector $|\arg(-\tau)|<\tfrac{4}{5}\pi$, the real tritronqu\'ee solution $y(\tau)$ of the Painlev\'e-I equation \eqref{eq:PI-intro} has infinitely many (double) poles in the complementary sector, and these values of $\tau$ are precisely those for which Riemann--Hilbert Problem~\ref{rhp:tritronquee} has no solution.  These points have to be avoided, and moreover to ensure that $\mathbf{T}(\xi;\tau)$ remains bounded, 
we use the condition that $y(\epsilon^{-4/5}(\ii ax+b(t-t_\gc)))$ is bounded, or equivalently since $|x|+|t-t_\gc|=\mathcal{O}(\epsilon^{4/5})$, that $y(\tau)$ is bounded where $\tau=\epsilon^{-4/5}s(x,t)$. Thus, $(x,t)$ is such that $\tau$ is bounded away from all poles of $y(\cdot)$ by a fixed $\tau$-distance, or equivalently the minimum distance between $(x,t)$ and the preimage under $\tau=\epsilon^{-4/5}s(x,t)$ of each pole of $y$ is bounded below by a multiple of $\epsilon^{4/5}$. So the size of the exclusions associated with the poles of $y$ is proportional to the allowed Euclidean distance between $(x,t)$ and the catastrophe point $(0,t_\mathrm{gc})$, both being $\mathcal{O}(\epsilon^{4/5})$.  Later in Section~\ref{sec:NearThePoles} we will remove the restriction that $\tau$ cannot lie close to any pole of $y(\tau)$, but this will require a further modification of the parametrix.

To analyze $\mathbf{O}(w)$ under these conditions, we define a global parametrix for $\mathbf{O}(w)$ as follows:
\begin{equation}
\dot{\mathbf{O}}(w):=\begin{cases}\dot{\mathbf{O}}^\mathrm{in}(w),&\quad w\in U\\
\dot{\mathbf{O}}^\mathrm{in}(w^*)^*,&\quad w\in U^*\\
\dot{\mathbf{O}}^\mathrm{out,l}(w),&\quad w\in \mathbb{C}\setminus (\overline{U}\cup\overline{U}^*).
\end{cases}
\label{eq:global-parametrix}
\end{equation}
Here $\dot{\mathbf{O}}^\mathrm{out,l}(w)$ is the phase-linearized outer parametrix described in Section~\ref{sec:phase-linearized-outer-parametrix}, and $\dot{\bfO}^{\mathrm{in}}(w)$ is the inner parametrix defined in terms of $\bfT(\xi;\tau)$ by \eqref{eq:CmatrixDefine} and \eqref{eq:inner-parametrix}.

\subsubsection{Error analysis}
\label{sec:AwayErrorAnalysis}
The error in approximating $\mathbf{O}(w)$ by its global parametrix $\dot{\mathbf{O}}(w)$ is measured by the matrix $\mathbf{E}(w):=\mathbf{O}(w)\dot{\mathbf{O}}(w)^{-1}$.  By construction it satisfies $\mathbf{E}(w)\to\mathbb{I}$ as $w\to\infty$ and it is analytic except (in general) along the union of jump contours of the two factors.  Because $(x,t)$ is close to the gradient catastrophe point, where the phase-linearized outer parametrix agrees with the outer parametrix described in the proof of Proposition~\ref{prop:BeforeCatastrophe} and as shown also in that proof approximates the jump conditions of $\mathbf{O}(w)$ with $\mathcal{O}(\epsilon)$ accuracy outside of $U$ and $U^*$, it is easy to see that under our assumptions on $(x,t)$ we have $\mathbf{E}_+(w)=\mathbf{E}_-(w)(\mathbb{I}+\mathcal{O}(\epsilon^{3/5}))$ on all jump contours for $\mathbf{E}(w)$ in the open upper half-plane outside of $U$ (and so by Schwarz symmetry the same is true in the open lower half-plane outside of $U^*$) because the error term is largest on $\beta$ where the estimate \eqref{eq:comparison-outside} holds and on $\beta^*$ where a Schwarz reflection of the same estimate holds, while on $\mathbb{R}_+$ we have $\mathbf{E}_+(w)=\sigma_2\mathbf{E}_-(w)\sigma_2$ exactly for $0<w<1-\delta$ and for $1+\delta<w$ and we have $\mathbf{E}_+(w)=\sigma_2\mathbf{E}_-(w)\sigma_2(\mathbb{I}+\mathcal{O}(\epsilon))$ for $1-\delta<w<1+\delta$.  The point is that everywhere on the jump contour that the original $g$-function led to jump matrix elements that were uniformly exponentially small, the same is true with the modified $g$-function by continuity of the $\epsilon$-independent factors in the exponents.  Since $\dot{\mathbf{O}}^\mathrm{in}(w)$ is an exact local solution of the jump conditions for $\mathbf{O}(w)$ within $U$, a Morera argument shows that $\mathbf{E}(w)$ is analytic within $U$ and $U^*$.  The remaining information we need about $\mathbf{E}(w)$ is the nature of its jump discontinuity across $\partial U$.  To compute this, we assume clockwise orientation for $\partial U$ and use the fact that $\mathbf{O}(w)$ has no jump across $\partial U$ to find
\begin{equation}
\mathbf{E}_+(w)=\mathbf{O}_+(w)\dot{\mathbf{O}}^\mathrm{out,l}(w)^{-1}=\mathbf{O}_-(w)\dot{\mathbf{O}}^\mathrm{out,l}(w)^{-1}=\mathbf{E}_-(w)\dot{\mathbf{O}}^\mathrm{in}(w)\dot{\mathbf{O}}^\mathrm{out}(w)^{-1},\quad w\in\partial U.
\end{equation}
Combining \eqref{eq:mismatch-simple}--\eqref{eq:mismatch-Y} with the estimates $\mathbf{C}(w)=\mathcal{O}(1)$, $\dot{\mathbf{O}}^{\mathrm{out,l}}(w)=\mathcal{O}(1)$, $\dot{\mathbf{O}}^\mathrm{out,l}(w)^{-1}=\mathcal{O}(1)$, $p(w;x,t)=1+\mathcal{O}(\epsilon^{4/5})$ ($p$ is defined in \eqref{eq:p-ratio}), and $Y_N(w)=1+\mathcal{O}(\epsilon)$ all of which hold uniformly in $w$ on $\partial U$ for $(x,t)$ values under consideration then shows that
\begin{equation}
\mathbf{E}_+(w)=\mathbf{E}_-(w)\mathbf{C}(w)\epsilon^{\sigma_3/10}\mathbf{T}(\xi;\tau)\mathbf{M}\xi^{-\sigma_3/4}\epsilon^{-\sigma_3/10}\mathbf{C}(w)^{-1}(\mathbb{I}+\mathcal{O}(\epsilon^{4/5})),\quad w\in\partial U,
\label{eq:away-jump-on-partial-U-1}
\end{equation}
where $\xi=\epsilon^{-2/5}W(w;x,t)$ and $\tau=\epsilon^{-4/5}s(x,t)$.  Since $\partial U$ is independent of $\epsilon$ and bounded away from $w=\alpha$ which maps to $W=0$ and hence $\xi=0$, we see that $\xi$ is uniformly large, of size proportional to $\epsilon^{-2/5}$ when $w\in\partial U$.  Therefore, we may use the asymptotic expansion \eqref{eq:T-expansion}, which in the present context takes the form
\begin{equation}
\mathbf{T}(\xi;\tau)\mathbf{M}\xi^{-\sigma_3/4}=\mathbb{I} + \frac{\epsilon^{2/5}}{W(w;x,t)}\begin{bmatrix}T_{1,11}(\tau) & T_{1,12}(\tau)\\-h(\tau) & T_{1,22}(\tau)\end{bmatrix} + \mathcal{O}(\epsilon^{4/5}),\quad w\in\partial U,
\end{equation}
which implies also that
\begin{equation}
\epsilon^{\sigma_3/10}\mathbf{T}(\xi;\tau)\mathbf{M}\xi^{-\sigma_3/4}\epsilon^{-\sigma_3/10}=
\mathbb{I}-\frac{\epsilon^{1/5}h(\tau)}{W(w;x,t)}\sigma_- + \mathcal{O}(\epsilon^{2/5}),\quad w\in\partial U.
\end{equation}
Finally, from \eqref{eq:away-jump-on-partial-U-1}, the jump of $\mathbf{E}(w)$ across $\partial U$ may be characterized as
\begin{equation}
\mathbf{E}_+(w)=\mathbf{E}_-(w)\left[\mathbb{I}-\frac{\epsilon^{1/5}h(\tau)}{W(w;x,t)}\mathbf{C}(w)\sigma_-\mathbf{C}(w)^{-1} + \mathcal{O}(\epsilon^{2/5})\right],\quad w\in\partial U
\label{eq:away-E-jump-partial-U-final}
\end{equation}
where the $\mathcal{O}(\epsilon^{2/5})$ error term is uniform on $\partial U$.  Invoking Schwarz symmetry of $\mathbf{E}(w)$ gives the corresponding jump on $\partial U^*$ (with corresponding counter-clockwise orientation):
\begin{equation}
\mathbf{E}_+(w)=\mathbf{E}_-(w)\left[\mathbb{I}+\frac{\epsilon^{1/5}h(\tau)^*}{W(w^*;x,t)^*}\mathbf{C}(w^*)^*\sigma_-\mathbf{C}(w^*)^{*-1}+\mathcal{O}(\epsilon^{2/5})\right],\quad w\in\partial U^*.
\label{eq:away-E-jump-partial-U-star-final}
\end{equation}

This is nearly a small-norm problem for $\mathbf{E}(w)$, with dominant contributions to the jump conditions coming from $\partial U$ and $\partial U^*$.  To make it a standard small-norm problem, it only remains to build in further symmetry to deal with the non-standard jump condition on $\mathbb{R}_+$ by defining
\begin{equation}
\mathbf{F}(z):=\begin{cases}\mathbf{E}(z^2),&\quad \mathrm{Im}\{z\}>0\\
\sigma_2\mathbf{E}(z^2)\sigma_2,&\quad\mathrm{Im}\{z\}<0.
\end{cases}
\label{eq:go-to-z-plane}
\end{equation}
Then $\mathbf{F}(z)$ has only jump conditions of the standard right-multiplication form: $\mathbf{F}_+(z)=\mathbf{F}_-(z)\mathbf{V}_\mathbf{F}(z)$, and the jump contour is compact, disjoint from $z=0$, and the jump matrix satisfies $\mathbf{V}_\mathbf{F}(z)=\mathbb{I}+\mathcal{O}(\epsilon^{3/5})$ uniformly except on four closed contours surrounding the two $z$-plane preimages of each of the points $w=\alpha$ and $w=\alpha^*$.  On these contours, according to \eqref{eq:away-E-jump-partial-U-final}--\eqref{eq:away-E-jump-partial-U-star-final} we have instead $\mathbf{V}_\mathbf{F}(z)=\mathbb{I}+\mathcal{O}(\epsilon^{1/5})$.  Since also $\mathbf{F}(z)\to\mathbb{I}$ as $z\to\infty$ it follows from standard small-norm theory that $\mathbf{F}(z)-\mathbb{I}$ takes boundary values on its jump contour that are $\mathcal{O}(\epsilon^{1/5})$ in the $L^2$ sense.  Therefore, if $\Sigma_\mathbf{F}$ denotes the jump contour for $\mathbf{F}$, by the Plemelj formula it must hold that
\begin{equation}
\begin{split}
\mathbf{F}(z)&=\mathbb{I} + \frac{1}{2\pi\ii}\int_{\Sigma_\mathbf{F}}\frac{\mathbf{F}_-(s)(\mathbf{V}_\mathbf{F}(s)-\mathbb{I})}{s-z}\,\dd s \\
&= \mathbb{I}+\frac{1}{2\pi\ii}\int_{\Sigma_\mathbf{F}}\frac{\mathbf{V}_\mathbf{F}(s)-\mathbb{I}}{s-z}\,\dd s + \frac{1}{2\pi\ii}\int_{\Sigma_\mathbf{F}}\frac{(\mathbf{F}_-(s)-\mathbb{I})(\mathbf{V}_\mathbf{F}(s)-\mathbb{I})}{s-z}\,\dd s
\end{split}
\end{equation} 
for any $z$ disjoint from $\Sigma_\mathbf{F}$.  In particular, since $\Sigma_\mathbf{F}$ does not contain the origin, we have
\begin{equation}
\mathbf{F}(0)=\mathbb{I}+\frac{1}{2\pi\ii}\int_{\Sigma_\mathbf{F}}(\mathbf{V}_\mathbf{F}(s)-\mathbb{I})\frac{\dd s}{s} +\frac{1}{2\pi\ii}\int_{\Sigma_\mathbf{F}}(\mathbf{F}_-(s)-\mathbb{I})(\mathbf{V}_\mathbf{F}(s)-\mathbb{I})\frac{\dd s}{s}.
\end{equation}
Since $\mathbf{F}_--\mathbb{I}=\mathcal{O}(\epsilon^{1/5})$ in $L^2(\Sigma_\mathbf{F})$, $s^{-1}\in L^2(\Sigma_\mathbf{F})$, and $\mathbf{V}_\mathbf{F}(s)-\mathbb{I}=\mathcal{O}(\epsilon^{1/5})$ in $L^\infty(\Sigma_\mathbf{F})$, by Cauchy-Schwarz it follows that
\begin{equation}
\mathbf{F}(0)=\mathbb{I}+\frac{1}{2\pi\ii}\int_{\Sigma_\mathbf{F}}(\mathbf{V}_\mathbf{F}(s)-\mathbb{I})\frac{\dd s}{s} + \mathcal{O}(\epsilon^{2/5}).
\end{equation}
The contour $\Sigma_\mathbf{F}$ can be replaced with the union of $z$-plane preimages of $\partial U$ and $\partial U^*$, and also $\mathbf{V}_\mathbf{F}(s)-\mathbb{I}$ can be replaced with its leading contribution explicitly proportional to $\epsilon^{1/5}$, both at the cost of an error that can be absorbed into the $\mathcal{O}(\epsilon^{2/5})$ term already present.  Thus, we have
\begin{multline}
\mathbf{F}(0)=\mathbb{I} +\frac{\epsilon^{1/5}h(\tau)}{2\pi\ii}\oint_{\ii\sqrt{-\alpha}}\frac{\mathbf{C}(s^2)\sigma_-\mathbf{C}(s^2)^{-1}}{sW(s^2;x,t)}\dd s + 
\frac{\epsilon^{1/5}h(\tau)^*}{2\pi\ii}\oint_{\ii\sqrt{-\alpha^*}}\frac{\mathbf{C}(s^{*2})^*\sigma_-\mathbf{C}(s^{*2})^{*-1}}{sW(s^{*2};x,t)^*}\,\dd s\\
{}+\frac{\epsilon^{1/5}h(\tau)}{2\pi\ii}\oint_{-\ii\sqrt{-\alpha}}\frac{\sigma_2\mathbf{C}(s^2)\sigma_-\mathbf{C}(s^2)^{-1}\sigma_2}{sW(s^2;x,t)}\,\dd s\\
{}+\frac{\epsilon^{1/5}h(\tau)^*}{2\pi\ii}\oint_{-\ii\sqrt{-\alpha^*}}\frac{\sigma_2\mathbf{C}(s^{*2})^*\sigma_-\mathbf{C}(s^{*2})^{*-1}\sigma_2}{sW(s^{*2};x,t)^*}\,\dd s + \mathcal{O}(\epsilon^{2/5}),
\end{multline}
where in each term the integral is (now) a positively-oriented loop surrounding the indicated point. 
Making the substitution $s=\ii\sqrt{-w}$ in the first two integrals, and $s=-\ii\sqrt{-w}$ in the second two integrals yields
\begin{multline}
\mathbf{F}(0)=\mathbb{I}+\frac{\epsilon^{1/5}h(\tau)}{4\pi\ii}\oint_{\alpha}\frac{\mathbf{C}(w)\sigma_-\mathbf{C}(w)^{-1}}{wW(w;x,t)}\,\dd w + \frac{\epsilon^{1/5}h(\tau)^*}{4\pi\ii}\oint_{\alpha^*}\frac{\mathbf{C}(w^*)^*\sigma_-\mathbf{C}(w^*)^{*-1}}{wW(w^*;x,t)^*}\,\dd w\\
{}+\frac{\epsilon^{1/5}h(\tau)}{4\pi\ii}\oint_\alpha\frac{\sigma_2\mathbf{C}(w)\sigma_-\mathbf{C}(w)^{-1}\sigma_2}{wW(w;x,t)}\,\dd w\\
{}+\frac{\epsilon^{1/5}h(\tau)^*}{4\pi\ii}\oint_{\alpha^*}\frac{\sigma_2\mathbf{C}(w^*)^*\sigma_-\mathbf{C}(w^*)^{*-1}\sigma_2}{wW(w^*;x,t)^*}\,\dd w + \mathcal{O}(\epsilon^{2/5}).
\end{multline} 
Since $\mathbf{C}(w)$ is analytic with unit determinant within the contour of integration about $w=\alpha$, which does not also contain $w=0$, while $W(w;x,t)$ has a simple zero at $w=\alpha$ all four explicit integrals can be evaluated by residues:
\begin{multline}
\mathbf{F}(0)=\mathbb{I}+\frac{\epsilon^{1/5}h(\tau)}{2\alpha W'(\alpha;x,t)}\mathbf{C}(\alpha)\sigma_-\mathbf{C}(\alpha)^{-1} + \frac{\epsilon^{1/5}h(\tau)^*}{2\alpha^*W'(\alpha;x,t)^*}\mathbf{C}(\alpha)^*\sigma_-\mathbf{C}(\alpha)^{*-1}\\
{}+\frac{\epsilon^{1/5}h(\tau)}{2\alpha W'(\alpha;x,t)}\sigma_2\mathbf{C}(\alpha)\sigma_-\mathbf{C}(\alpha)^{-1}\sigma_2+\frac{\epsilon^{1/5}h(\tau)^*}{2\alpha^*W'(\alpha;x,t)^*}\sigma_2\mathbf{C}(\alpha)^*\sigma_-\mathbf{C}(\alpha)^{*-1}\sigma_2 +\mathcal{O}(\epsilon^{2/5}).
\end{multline}
Noting that the first and last two terms proportional to $\epsilon^{1/5}$ are complex conjugates, this can also be written as
\begin{equation}
\begin{split}
\mathbf{F}(0)&=\mathbb{I}+\epsilon^{1/5}\mathrm{Re}\left\{\frac{h(\tau)}{\alpha W'(\alpha;x,t)}\left(\mathbf{C}(\alpha)\sigma_-\mathbf{C}(\alpha)^{-1}+\sigma_2\mathbf{C}(\alpha)\sigma_-\mathbf{C}(\alpha)^{-1}\sigma_2\right)\right\}+\mathcal{O}(\epsilon^{2/5})\\
&=\mathbb{I}+\epsilon^{1/5}\mathrm{Re}\left\{\frac{h(\tau)(C_{12}(\alpha)^2+C_{22}(\alpha)^2)}{\alpha W'(\alpha;x,t)}\right\}(\sigma_--\sigma_+) + \mathcal{O}(\epsilon^{2/5}),
\end{split}
\end{equation}
where on the second line we used $\det(\mathbf{C}(w))=1$.  Since $z=0$ if and only if $w=0$, and $\mathbf{E}(0)=\sigma_2\mathbf{E}(0)\sigma_2$ by continuity at $w=0$ (see also \cite{BuckinghamMiller2013}), we have $\mathbf{E}(0)=\mathbf{F}(0)$ given as above.  

\subsubsection{Asymptotic formul\ae\ for $\cos(\tfrac{1}{2}u_N(x,t))$ and $\sin(\tfrac{1}{2}u_N(x,t))$}
It only remains to extract the potentials $\cos(\tfrac{1}{2}u_N(x,t))$ and $\sin(\tfrac{1}{2}u_N(x,t))$, which by definition are obtained from the first column of the matrix $\mathbf{H}(w)$:
\begin{equation}
\cos(\tfrac{1}{2}u_N(x,t))=H_{11}(0)\quad\text{and}\quad\sin(\tfrac{1}{2}u_N(x,t))=H_{21}(0).
\end{equation}
But $\mathbf{H}(0)=\mathbf{M}(0)=\mathbf{N}(0)$ because none of the changes of variable involve the neighborhood of $w=0$.  Since $g(0)=0$ (because $g(w)$ is continuous at $w=0$ and $g_+(w)+g_-(w)=0$ for $w>0$) it then follows that also $\mathbf{N}(0)=\mathbf{O}(0)$.  The last step is to bring in the outer parametrix and the error matrix via $\mathbf{O}(0)=\mathbf{E}(0)\dot{\mathbf{O}}(0)=\mathbf{E}(0)\dot{\mathbf{O}}^\mathrm{out,l}(0)$.  Thus,
\begin{equation}
\cos(\tfrac{1}{2}u_N(x,t))=(\mathbf{E}(0)\dot{\mathbf{O}}^\mathrm{out,l}(0))_{11}\quad\text{and}\quad
\sin(\tfrac{1}{2}u_N(x,t))=(\mathbf{E}(0)\dot{\mathbf{O}}^\mathrm{out,l}(0))_{21}.
\label{eq:AwayFormulaeForCosSin}
\end{equation}
Using Proposition~\ref{prop:PropertiesOfY}, the phase-linearized outer parametrix at the origin can be written in terms of Jacobi elliptic functions and the complete elliptic integral of the first kind $\KK(m_\mathrm{gc})$ as follows:
\begin{equation}
\dot{\mathbf{O}}^\mathrm{out,l}(0)=\begin{bmatrix}\displaystyle\mathrm{dn}\left(\frac{2\KK(m_\mathrm{gc})}{\pi\epsilon}\Phi^l(t);m_\mathrm{gc}\right) & \displaystyle \sqrt{m_\mathrm{gc}}\,\mathrm{sn}\left(\frac{2\KK(m_\mathrm{gc})}{\pi\epsilon}\Phi^l(t);m_\mathrm{gc}\right)\\
\displaystyle -\sqrt{m_\mathrm{gc}}\,\mathrm{sn}\left(\frac{2\KK(m_\mathrm{gc})}{\pi\epsilon}\Phi^l(t);m_\mathrm{gc}\right) & \displaystyle 
\mathrm{dn}\left(\frac{2\KK(m_\mathrm{gc})}{\pi\epsilon}\Phi^l(t);m_\mathrm{gc}\right)\end{bmatrix}=\begin{bmatrix}\dot{C}(t) &-\dot{S}(t)\\\dot{S}(t) & \dot{C}(t)\end{bmatrix}
\label{eq:PhaseLinearizedOuterParametrix-at-zero}
\end{equation}
where $\Phi^l(t)$ is the linearization of $\Phi(x,t)$ defined in \eqref{eq:Phi-l}, $m_\mathrm{gc}\in (0,1)$ is defined in \eqref{eq:mgc-alphagc}, and the second equality is simply a comparison with the definitions \eqref{eq:C-dot-define}--\eqref{eq:S-dot-define}.  Therefore,
\begin{equation}
\begin{split}
\cos(\tfrac{1}{2}u_N(x,t))&=\dot{C}(t)-
\epsilon^{1/5}\mathrm{Re}\left\{\frac{h(\tau)(C_{12}(\alpha(x,t))^2+C_{22}(\alpha(x,t))^2)}{\alpha(x,t) W'(\alpha(x,t);x,t)}\right\}\dot{S}(t) + \mathcal{O}(\epsilon^{2/5}),\\
\sin(\tfrac{1}{2}u_N(x,t))&=\dot{S}(t)+
\epsilon^{1/5}\mathrm{Re}\left\{\frac{h(\tau)(C_{12}(\alpha(x,t))^2+C_{22}(\alpha(x,t))^2)}{\alpha(x,t) W'(\alpha(x,t);x,t)}\right\}
\dot{C}(t)+\mathcal{O}(\epsilon^{2/5}).
\end{split}
\label{eq:cos-sin-before-argument-linearization}
\end{equation}
We further simplify by first noting that the assumption $(x,t)=(0,t_\mathrm{gc})+\mathcal{O}(\epsilon^{4/5})$ implies
\begin{equation}
\frac{C_{12}(\alpha(x,t))^2+C_{22}(\alpha(x,t))^2}{\alpha(x,t)W'(\alpha;x,t);x,t)}=
\frac{C_{12}(\alpha_\mathrm{gc})^2+C_{22}(\alpha_\mathrm{gc})^2}{\alpha_\mathrm{gc}W_\mathrm{gc}'(\alpha_\mathrm{gc})}+\mathcal{O}(\epsilon^{4/5}).
\end{equation}
Similarly, 
\begin{equation}
\begin{split}
\tau &=\epsilon^{-4/5}s(x,t)\\
&=\epsilon^{-4/5}\left(\frac{\partial s}{\partial x}(0,t_\mathrm{gc})x + \frac{\partial s}{\partial t}(0,t_\mathrm{gc})(t-t_\mathrm{gc}) + \mathcal{O}(x^2+(t-t_\mathrm{gc})^2)\right)\\
&=\epsilon^{-4/5}s^l(x,t)+ \mathcal{O}(\epsilon^{4/5}),
\end{split}
\end{equation}
where, using Lemma~\ref{lemma:PartialsOfs},
$s^l(x,t)$ is the linear function
\begin{equation}
s^l(x,t):=\ii a x + b(t-t_\mathrm{gc}).
\end{equation}
Therefore, since all derivatives of the function $h(\tau)$ are uniformly bounded due to the assumption bounding the image of $(x,t)$ in the $\tau$-plane away from all poles of $y(\tau)$, 
\begin{equation}
h(\tau)=h\left(\frac{s^l(x,t)}{\epsilon^{4/5}}\right) + \mathcal{O}(\epsilon^{4/5}).
\end{equation}
These observations allow us to rewrite \eqref{eq:cos-sin-before-argument-linearization} in the form
\begin{equation}
\begin{split}
\cos(\tfrac{1}{2}u_N(x,t))&=\dot{C}(t)-
\epsilon^{1/5}\mathrm{Re}\left\{h\left(\frac{s^l(x,t)}{\epsilon^{4/5}}\right)\frac{C_{12}(\alpha_\mathrm{gc}))^2+C_{22}(\alpha_\mathrm{gc})^2}{\alpha_\mathrm{gc} W_\mathrm{gc}'(\alpha_\mathrm{gc})}\right\}\dot{S}(t)+\mathcal{O}(\epsilon^{2/5})\\
\sin(\tfrac{1}{2}u_N(x,t))&=\dot{S}(t)+\epsilon^{1/5}\mathrm{Re}\left\{h\left(\frac{s^l(x,t)}{\epsilon^{4/5}}\right)\frac{C_{12}(\alpha_\mathrm{gc})^2+C_{22}(\alpha_\mathrm{gc})^2}{\alpha_\mathrm{gc} W_\mathrm{gc}'(\alpha_\mathrm{gc})}\right\}
\dot{C}(t)+\mathcal{O}(\epsilon^{2/5}).
\end{split}
\label{eq:cos-sin-after-argument-linearization}
\end{equation}
Recall that the condition that $\tau$ is bounded away from any poles of $y(\tau)$ is guaranteed by the condition that $y(\epsilon^{-4/5}(iax+b(t-t_\mathrm{gc})))$ is bounded.
Just one more ingredient is needed to complete the proof of Theorem~\ref{thm:AwayFromPoles}:
\begin{lemma}
\begin{equation}
\frac{C_{12}(\alpha_\mathrm{gc})^2+C_{22}(\alpha_\mathrm{gc})^2}{\alpha_\mathrm{gc}W'_\mathrm{gc}(\alpha_\mathrm{gc})}=\frac{2}{|W'_\mathrm{gc}(\alpha_\mathrm{gc})|^{1/2}}\left(\frac{m_\mathrm{gc}}{1-m_\mathrm{gc}}\right)^{1/4}
\mathrm{cn}\left(\frac{2\KK(m_\mathrm{gc})}{\pi\epsilon}\Phi^l(t);m_\mathrm{gc}\right).
\label{eq:real-thing-IV}
\end{equation}
\label{lemma:C-real}
\end{lemma}
The proof of this result is technical, relying on numerous identities involving Riemann theta functions of genus $1$, and will be given in Appendix~\ref{sec:proof-lemma:C-real}.

Using Lemma~\ref{lemma:C-real} in \eqref{eq:cos-sin-after-argument-linearization} 
completes the proof of Theorem~\ref{thm:AwayFromPoles}, in which the constant $\sigma>0$ is given by
\begin{equation}
\sigma:=|W'_\mathrm{gc}(\alpha_\mathrm{gc})|^{1/2}.
\label{eq:SIGMA}
\end{equation}

\section{Proof of Theorem~\ref{thm:NearThePoles}}
\label{sec:NearThePoles}
All of the analysis of Section~\ref{sec:AwayFromPoles} applies, up to the introduction of the matrix $\mathbf{T}(\xi;\tau)$ solving Riemann--Hilbert Problem~\ref{rhp:tritronquee}.  The new ingredients required to prove Theorem~\ref{thm:NearThePoles} all relate to the fact that $\mathbf{T}(\xi;\tau)$ fails to exist whenever $\tau$ is a pole $\tau_\mathrm{p}$ of the real tritronqu\'ee solution $y(\tau)$ of the Painlev\'e-I equation $y''(\tau)=6y(\tau)^2+\tau$, and is unbounded in every deleted neighborhood of $\tau_\mathrm{p}$.  While we can still make use of the phase-linearized outer parametrix $\dot{\mathbf{O}}^\mathrm{out,l}(w)$, to prove the theorem we will need to replace the inner parametrix with one based on some modification of $\mathbf{T}(\xi;\tau)$ that makes sense when $\tau$ is near $\tau_\mathrm{p}$.  Making this replacement means that we will then need to account for the difference when we study the analogue of the error matrix $\mathbf{E}(w)$.  It will turn out that $\mathbf{E}(w)$ can no longer be studied by small norm theory alone, but first it will be necessary to build a \emph{parametrix for the error}.  The leading terms in the solution of the sine-Gordon equation will then be extracted from that parametrix.

\subsection{Modifying the inner parametrix}
Let $\tau_\mathrm{p}$ be one of the infinitely-many poles of the real tritronqu\'ee solution $y(\tau)$ of the Painlev\'e-I equation, all of which are known \cite{CostinHuangTanveer2014} to lie in the sector $|\arg(\tau)|<\tfrac{1}{5}\pi$.  In formulating a replacement for $\mathbf{T}(\xi;\tau)$ for $\tau$ near $\tau_\mathrm{p}$ we wish to maintain the analyticity and jump conditions of Riemann--Hilbert Problem~\ref{rhp:tritronquee} because these are the properties that lead to the inner parametrix being an exact local solution of the Riemann--Hilbert problem for $\mathbf{O}(w)$ for $w\in U$.  However, as we will see, the normalization condition is less essential.  

To this end, we seek a solution $\widehat{\mathbf{T}}(\xi;\tau)$ of the following Riemann--Hilbert problem.
\begin{rhp}[Regularization of $\bfT(\xi;\tau)$]
Given $\tau\in\mathbb{C}$, seek a $2\times 2$ matrix function $\widehat{\mathbf{T}}=\widehat{\mathbf{T}}(\xi;\tau)$ defined in the four sectors $0<\arg(\xi)<\tfrac{2}{5}\pi$, $-\tfrac{2}{5}\pi<\arg(\xi)<0$, $\tfrac{2}{5}\pi<\arg(\xi)<\pi$, and $-\pi<\arg(\xi)<-\tfrac{2}{5}\pi$, that satisfies the following additional conditions:
\begin{itemize}
\item[]\textbf{Analyticity:}  $\widehat{\mathbf{T}}(\xi;\tau)$ is analytic for $\xi$ in each sector of definition, taking continuous boundary values from each sector on the union of rays forming its boundary.
\item[]\textbf{Jump conditions:}  Let each excluded ray be oriented in the direction away from the origin and denote the boundary value taken by $\widehat{\bfT}(\xi;\tau)$ on the ray from the left (resp., right) by $\widehat{\bfT}_+(\xi;\tau)$ (resp., $\widehat{\bfT}_-(\xi;\tau)$). Then the boundary values satisfy $\widehat{\mathbf{T}}_+(\xi;\tau)=\widehat{\mathbf{T}}_-(\xi;\tau)\mathbf{V}_\mathbf{T}(\xi;\tau)$ where the matrix $\mathbf{V}_\mathbf{T}(\xi;\tau)$ is defined on each ray as shown in Figure~\ref{fig:PI-Standard-Tritronquee}.
\item[]\textbf{Normalization:}  The matrix $\widehat{\mathbf{T}}(\xi;\tau)$ satisfies the normalization condition
\begin{equation}
\lim_{\xi\to\infty}\widehat{\mathbf{T}}(\xi;\tau)\mathbf{M}\xi^{3\sigma_3/4}=\mathbb{I}.
\label{eq:Tilde-T-norm}
\end{equation}
\end{itemize}
\label{rhp:Tilde-T}
\end{rhp}
Comparing with Riemann--Hilbert Problem~\ref{rhp:tritronquee}, the only difference is in the normalization condition \eqref{eq:Tilde-T-norm} which replaces the diagonal matrix $\xi^{-\sigma_3/4}$ with $\xi^{3\sigma_3/4}$.

\begin{lemma}
Let $\tau_\mathrm{p}$ be a pole of the real Painlev\'e-I tritronqu\'ee function $y(\tau)$.  Then there is a neighborhood $N$ of $\tau_\mathrm{p}$ such that Riemann--Hilbert Problem~\ref{rhp:Tilde-T} has a unique solution $\widehat{\mathbf{T}}(\xi;\tau)$ that is an analytic function of $\tau\in N$. 
\label{lemma:Ttilde}
\end{lemma}
\begin{proof}
Let $N$ be a neighborhood of $\tau_\mathrm{p}$ containing only the one pole.  Then $\mathbf{T}(\xi;\tau)$ exists for $\tau\in N\setminus\{\tau_\mathrm{p}\}$ satisfying all of the conditions of Riemann--Hilbert Problem~\ref{rhp:tritronquee}, and we consider trying to build $\widehat{\mathbf{T}}(\xi;\tau)$ for such $\tau$ from $\mathbf{T}(\xi;\tau)$ by a Schlesinger gauge transformation (or Darboux transformation):
\begin{equation}
\widehat{\mathbf{T}}(\xi;\tau)=\mathbf{S}(\xi;\tau)\mathbf{T}(\xi;\tau),\quad\tau\in N\setminus\{\tau_\mathrm{p}\},\quad
\mathbf{S}(\xi;\tau):=\mathbf{S}_1(\tau)\xi+\mathbf{S}_0(\tau),
\label{eq:Schlesinger}
\end{equation}
where $\mathbf{S}_1(\tau)$ and $\mathbf{S}_0(\tau)$ are matrix functions to be determined.  We observe that the choice of the gauge matrix $\mathbf{S}(\xi;\tau)$ as a linear function of $\xi$ implies that the analyticity and jump conditions for $\widehat{\mathbf{T}}(\xi;\tau)$ are automatically satisfied, so it only remains to impose the normalization condition \eqref{eq:Tilde-T-norm} which will determine $\mathbf{S}_1(\tau)$ and $\mathbf{S}_0(\tau)$ in terms of the expansion coefficients $\mathbf{T}_j(\tau)$, $j=1,2,3,\dots$.  Indeed, assuming \eqref{eq:Tilde-T-norm} holds in the form
\begin{equation}
\widehat{\mathbf{T}}(\xi;\tau)\mathbf{M}\xi^{3\sigma_3/4}=\mathbb{I}+\widehat{\mathbf{T}}_1(\tau)\xi^{-1}+\widehat{\mathbf{T}}_2(\tau)\xi^{-2}+\mathcal{O}(\xi^{-3}),\quad\xi\to\infty,
\label{eq:Ttilde-expansion}
\end{equation}
we substitute from \eqref{eq:Schlesinger} on the left-hand side and then recall the expansion \eqref{eq:T-expansion} of $\mathbf{T}(\xi;\tau)\mathbf{M}\xi^{-\sigma_3/4}$.
This yields the identity
\begin{equation}
(\mathbf{S}_1(\tau)\xi +\mathbf{S}_0)(\mathbb{I}+\mathbf{T}_1(\tau)\xi^{-1}+\mathbf{T}_2(\tau)\xi^{-2}+\mathcal{O}(\xi^{-3}))=(\mathbb{I}+\widehat{\mathbf{T}}_1(\tau)\xi^{-1}+\widehat{\mathbf{T}}_2(\tau)\xi^{-2}+\mathcal{O}(\xi^{-3}))\xi^{-\sigma_3},\quad\xi\to\infty.
\end{equation}
Separating out the coefficients of different powers of $\xi$ then yields
\begin{equation}
\mathbf{S}_1(\tau)=\begin{bmatrix}0&0\\0&1\end{bmatrix}\quad\text{and}\quad\mathbf{S}_0(\tau)=\begin{bmatrix}0&T_{1,21}(\tau)^{-1}\\-T_{1,21}(\tau) & (T_{1,11}(\tau)T_{1,21}(\tau)-T_{2,21}(\tau))T_{1,21}(\tau)^{-1}\end{bmatrix}
\end{equation}
along with infinitely many relations between the coefficients $\{\widehat{\mathbf{T}}_j(\tau)\}_{j\ge 1}$ and $\{\mathbf{T}_j(\tau)\}_{j\ge 1}$, for instance
\begin{equation}
\begin{split}
\widehat{T}_{1,12}(\tau)&=S_{0,12}(\tau)+S_{1,11}(\tau)T_{1,12}(\tau)+S_{1,12}(\tau)T_{1,22}(\tau)\\
&=T_{1,21}(\tau)^{-1}\\
&=-h(\tau)^{-1},
\end{split}
\label{eq:Ttilde112}
\end{equation}
where $h(\tau)$ is the Hamiltonian associated with the real tritronqu\'ee solution $y(\tau)$.  It has a simple pole at $\tau=\tau_\mathrm{p}$, so $\widehat{T}_{1,12}(\tau)$ can be interpreted as an analytic function on $N$ with a simple zero at $\tau=\tau_\mathrm{p}$.  Recalling also the definition \eqref{eq:hyz} of $y(\tau)$ in terms of the coefficients $\mathbf{T}_j(\tau)$ we can write $\mathbf{S}_0(\tau)$ in the form
\begin{equation}
\mathbf{S}_0(\tau)=\begin{bmatrix}0 & -h(\tau)^{-1}\\h(\tau) & \tfrac{1}{2}(y(\tau)-h(\tau)^2)+T_{2,21}(\tau)h(\tau)^{-1}\end{bmatrix}
\end{equation}
where only $T_{2,21}(\tau)$ is not easily expressed in terms of $y$, $z$, and $h$.

Now both $\mathbf{T}(\xi;\tau)$ and $\mathbf{S}_0(\tau)$ have a pole at $\tau=\tau_\mathrm{p}$ but we can see that the singularities cancel in the matrix multiplication by the following indirect argument.  
For $\tau\in N\setminus\{\tau_\mathrm{p}\}$, the gauge matrix $\mathbf{S}(\xi;\tau)$ and $\mathbf{T}(\xi;\tau)$ can both be differentiated with respect to both $\xi$ and $\tau$, and one checks easily that the matrix $\widehat{\mathbf{L}}(\xi;\tau):=\widehat{\mathbf{T}}(\xi;\tau)\ee^{(4\xi^{5/2}+\tau\xi^{1/2})\sigma_3}$ satisfies simultaneously the differential equations 
\begin{equation}
\frac{\partial\widehat{\mathbf{L}}}{\partial\xi}=\widehat{\mathbf{A}}\widehat{\mathbf{L}}\quad\text{and}\quad
\frac{\partial\widehat{\mathbf{L}}}{\partial\tau}=\widehat{\mathbf{U}}\widehat{\mathbf{L}}
\label{eq:tilde-Lax-pair}
\end{equation}
where $\widehat{\mathbf{A}}(\xi;\tau)$ and $\widehat{\mathbf{U}}(\xi;\tau)$ are given in terms of $\mathbf{A}(\xi;\tau)$ and $\mathbf{U}(\xi;\tau)$ (cf., \eqref{eq:PI-U-matrix}--\eqref{eq:PI-A-matrix}) by
\begin{equation}
\begin{split}
\widehat{\mathbf{A}}(\xi;\tau)&:=\mathbf{S}(\xi;\tau)\mathbf{A}(\xi;\tau)\mathbf{S}(\xi;\tau)^{-1}+\frac{\partial\mathbf{S}}{\partial\xi}(\xi;\tau)\mathbf{S}(\xi;\tau)^{-1}\\
\widehat{\mathbf{U}}(\xi;\tau)&:=\mathbf{S}(\xi;\tau)\mathbf{U}(\xi;\tau)\mathbf{S}(\xi;\tau)^{-1}+
\frac{\partial\mathbf{S}}{\partial\tau}(\xi;\tau)\mathbf{S}(\xi;\tau)^{-1}.
\end{split}
\label{eq:tilde-coefficients}
\end{equation}
It is easy to see that $\widehat{\mathbf{A}}(\xi;\tau)$ is a cubic matrix polynomial in $\xi$ with coefficients that are analytic functions of $\tau\in N\setminus\{\tau_\mathrm{p}\}$ possibly having a pole of some finite order at $\tau=\tau_\mathrm{p}$.  However, more detailed analysis based on the Laurent series for $h(\tau)$, $y(\tau)$, $z(\tau)$, and $T_{2,21}(\tau)$ developed earlier in Section~\ref{sec:series} shows that any singularity in $\widehat{\mathbf{A}}(\xi;\tau)$ at $\tau=\tau_\mathrm{p}$ must be removable.  In fact, one can show that
\begin{equation}
\widehat{\mathbf{A}}(\xi;\tau_\mathrm{p})=\begin{bmatrix}2h_0\xi-2t_0 & -2\\
-2\xi^3+2h_0^2\xi^2-(\tau_\mathrm{p}+4h_0t_0)\xi+2t_0^2+2h_0 & -2h_0\xi+2t_0\end{bmatrix},
\end{equation}
where $h_0$ and $t_0$ are the $(\tau-\tau_\mathrm{p})^0$ terms in the expansions \eqref{eq:h-Laurent} and \eqref{eq:T112series} respectively.  Since $\widehat{\mathbf{A}}(\xi;\tau)$ is regular at $\tau=\tau_\mathrm{p}$, the differential equation $\widehat{\mathbf{L}}_\xi(\xi;\tau)=\widehat{\mathbf{A}}(\xi;\tau)\widehat{\mathbf{L}}(\xi;\tau)$ has a fundamental matrix solution that is an entire function of $\xi$. Then since $\widehat{\mathbf{T}}(\xi;\tau)$ can be constructed from the latter fundamental matrix (this is the direct monodromy construction for the Lax pair \eqref{eq:tilde-Lax-pair}), it follows that Riemann--Hilbert Problem~\ref{rhp:Tilde-T} has a solution that depends analytically on $\tau\in N$ with no need to exclude the pole $\tau_\mathrm{p}$ of $\mathbf{T}(\xi;\tau)$.  This solution is unique by standard Liouville arguments and therefore coincides with the formula $\widehat{\mathbf{T}}(\xi;\tau)=\mathbf{S}(\xi;\tau)\mathbf{T}(\xi;\tau)$ in which $\widehat{\mathbf{T}}(\xi;\tau)$ is obtained from $\mathbf{T}(\xi;\tau)$ via the Schlesinger gauge transformation $\mathbf{S}(\xi;\tau)$.
\end{proof}

To use the matrix $\widehat{\mathbf{T}}(\xi;\tau)$ instead of $\mathbf{T}(\xi;\tau)$ when $\tau\approx \tau_\mathrm{p}$ to build an inner parametrix for $\mathbf{O}(w)$ for $w\in U$, we simply replace $\mathbf{T}(\xi;\tau)$ in the definition \eqref{eq:inner-parametrix} by $\epsilon^{-2\sigma_3/5}\widehat{\mathbf{T}}(\xi;\tau)$.  As desired, since $\widehat{\mathbf{T}}(\xi;\tau)$ and $\mathbf{T}(\xi;\tau)$ satisfy the same jump conditions, it is again clear that $\dot{\mathbf{O}}^\mathrm{in}(w)$ is an exact local solution within $U$ of the Riemann--Hilbert problem for $\mathbf{O}(w)$.  However, due to the modification of the asymptotic normalization condition in the limit $\xi\to\infty$, the comparison to the phase-linearized outer parametrix on $\partial U$ will no longer be near-identity and will therefore require further modeling.  Indeed, recalling that $\xi=\epsilon^{-2/5}W(w;x,t)$ and $\tau=\epsilon^{-4/5}s(x,t)$ as well as the definition \eqref{eq:p-ratio} of $p(w;x,t)$, the analogue of \eqref{eq:mismatch-simple} now takes the form
\begin{multline}
\dot{\mathbf{O}}^\mathrm{in}(w)\dot{\mathbf{O}}^\mathrm{out,l}(w)^{-1}=\mathbf{C}(w)\epsilon^{-3\sigma_3/10}\widehat{\mathbf{T}}(\xi;\tau)\mathbf{M}\xi^{3\sigma_3/4}W_\mathrm{gc}(w)^{-\sigma_3}\epsilon^{3\sigma_3/10}p(w;x,t)^{-3\sigma_3/4}\mathbf{C}(w)^{-1},\\ 
w\in\partial U\cap(\mathrm{II}\cup\mathrm{III}),
\label{eq:mismatch-simple-tilde}
\end{multline}
while instead of \eqref{eq:mismatch-Y} we get
\begin{multline}
\dot{\mathbf{O}}^\mathrm{in}(w)\dot{\mathbf{O}}^\mathrm{out,l}(w)^{-1}=\mathbf{C}(w)\epsilon^{-3\sigma_3/10}\widehat{\mathbf{T}}(\xi;\tau)\mathbf{M}\xi^{3\sigma_3/4}W_\mathrm{gc}(w)^{-\sigma_3}\epsilon^{3\sigma_3/10}p(w;x,t)^{-3\sigma_3/4}\mathbf{C}(w)^{-1}\\
{}\cdot\dot{\mathbf{O}}^\mathrm{out,l}(w)Y_N(w)^{\sigma_3/2}\dot{\mathbf{O}}^\mathrm{out,l}(w)^{-1},\quad
w\in\partial U\cap(\mathrm{I}\cup\mathrm{IV}\cup\mathrm{V}).
\label{eq:mismatch-Y-tilde}
\end{multline}
As before, the product of factors on the second line of \eqref{eq:mismatch-Y-tilde} have the form $\mathbb{I}+\mathcal{O}(\epsilon)$ uniformly on $\partial U$.  Similarly, $p(w;x,t)=1+\mathcal{O}(\epsilon^{4/5})$ uniformly on $\partial U$ because for $(x,t)$ to approach the preimage of a given pole $\tau=\tau_\mathrm{p}$ we in particular have $(x,t)=(0,t_\mathrm{gc})+\mathcal{O}(\epsilon^{4/5})$; hence as $\mathbf{C}(w)^{-1}$ is uniformly bounded we can write $p(w;x,t)^{\sigma_3/4}\mathbf{C}(w)^{-1}=\mathbf{C}(w)^{-1}(\mathbb{I}+\mathcal{O}(\epsilon^{4/5}))$.  This means that the important factors are common in the formul\ae\ \eqref{eq:mismatch-simple-tilde}--\eqref{eq:mismatch-Y-tilde}; using \eqref{eq:Ttilde-expansion} and \eqref{eq:Ttilde112} shows that
\begin{multline}
\mathbf{C}(w)\epsilon^{-3\sigma_3/10}\widehat{\mathbf{T}}(\xi;\tau)\mathbf{M}\xi^{3\sigma_3/4}W_\mathrm{gc}(w)^{-\sigma_3}\epsilon^{3\sigma_3/10}\mathbf{C}(w)^{-1}\\
\begin{aligned}=&\mathbf{C}(w)
\left[W_\mathrm{gc}(w)^{-\sigma_3}+\epsilon^{-1/5}\widehat{T}_{1,12}(\tau)\sigma_+ + \mathcal{O}(\epsilon^{1/5})\right]\mathbf{C}(w)^{-1}\\
=&\mathbf{C}(w)
\left[W_\mathrm{gc}(w)^{-\sigma_3}-\epsilon^{-1/5}h(\tau)^{-1}\sigma_+ + \mathcal{O}(\epsilon^{1/5})\right]\mathbf{C}(w)^{-1},\quad\text{uniformly for $w\in\partial U$}.
\end{aligned}
\label{eq:mismatch-near-pole}
\end{multline}
Since $h(\tau)$ has a simple pole at $\tau=\tau_\mathrm{p}$ with residue $-1$, $h(\tau)^{-1}$ vanishes exactly to first order in $\tau-\tau_\mathrm{p}$, so under the assumptions of Theorem~\ref{thm:NearThePoles}, we see that $\epsilon^{-1/5}h(\tau)^{-1}$ is bounded but is not necessarily small.  Indeed, for $(X,T)$ bounded we have $\tau=\epsilon^{-4/5}s(x,t)=\epsilon^{-4/5}s(x_\mathrm{p}+\epsilon X,t_\mathrm{gc}+(t_\mathrm{p}-t_\mathrm{gc} +\epsilon T)$, and by definition of $(x_\mathrm{p},t_\mathrm{p})$, we also have $\tau_\mathrm{p}=\epsilon^{-4/5}(\ii a x_\mathrm{p}+(t_\mathrm{p}-t_\mathrm{gc})b)$.  Hence $\tau-\tau_\mathrm{p}=\epsilon^{1/5}\ell(X,T) + \mathcal{O}(\epsilon^{4/5})$ where
\begin{equation}
\ell=\ell(X,T):=\ii a X +bT,
\label{eq:ell-def}
\end{equation}
and therefore 
\begin{equation}
\epsilon^{-1/5}h(\tau)^{-1}=-\epsilon^{-1/5}[(\tau-\tau_\mathrm{p})+\mathcal{O}((\tau-\tau_\mathrm{p})^2)]=-\ell(X,T)+\mathcal{O}(\epsilon^{2/5}).
\label{eq:one-over-h-linearized}
\end{equation}

\subsection{Initial global parametrix, corresponding error matrix, and parametrix for the error}
With the inner parametrix for $\mathbf{O}(w)$ redefined within $U$ as indicated above we appeal again to \eqref{eq:global-parametrix} to define an initial global parametrix $\dot{\mathbf{O}}(w)$.  The corresponding error matrix is defined formally exactly as before:  $\mathbf{E}(w):=\mathbf{O}(w)\dot{\mathbf{O}}(w)^{-1}$, only the definition of $\dot{\mathbf{O}}(w)$ has changed within the neighborhoods $U$ and $U^*$.  Therefore, the error $\mathbf{E}(w)$ has exactly the same domain of analyticity as in the previous situation, and its jump conditions aside from those across $\partial U\cup\partial U^*$ are also identical.  Just as before the dominant jump conditions satisfied by $\mathbf{E}(w)$ occur across $\partial U$ and $\partial U^*$, but unlike in the previous setting these are no longer near-identity in general.  Indeed, combining \eqref{eq:mismatch-simple-tilde}, \eqref{eq:mismatch-Y-tilde}, and \eqref{eq:mismatch-near-pole}, along with \eqref{eq:one-over-h-linearized} shows that taking clockwise orientation of $\partial U$,
\begin{equation}
\mathbf{E}_+(w)=\mathbf{E}_-(w)\left[\mathbf{C}(w)\left(W_\mathrm{gc}(w)^{-\sigma_3}+\ell(X,T)\sigma_+\right)\mathbf{C}(w)^{-1}+\mathcal{O}(\epsilon^{1/5})\right],\quad w\in\partial U.
\label{eq:Near-E-Jump-partialU}
\end{equation}
Since the explicit terms are not converging to $\mathbb{I}$ for $w\in\partial U$ as $\epsilon\to 0$, we have not yet reduced the problem to a small-norm setting.  We must first model the error itself with a parametrix that captures the dominant terms in the jump conditions across $\partial U$ and $\partial U^*$.

To this end, let $\dot{\mathbf{E}}(w)$ be a solution of the following Riemann--Hilbert problem.
\begin{rhp}[Parametrix for the error]
Fix constants $\theta\in(0,\pi)$, $a<0$, and $b>0$, as well as a conformal map $W_\mathrm{gc}(\cdot)$ with $W_\mathrm{gc}(\ee^{\ii\theta})=0$ and $\arg(W'_\mathrm{gc}(\ee^{\ii\theta}))=\tfrac{1}{2}\pi-\theta$.
For additional real variable parameters $X,T,\nu$, seek a $2\times 2$ matrix function $\dot{\mathbf{E}}(w)=\dot{\mathbf{E}}(w;X,T,\nu)$ with the following properties:
\begin{itemize}
\item[]\textbf{Analyticity:} $\dot{\mathbf{E}}(w)$ is analytic for $w\in\mathbb{C}\setminus(\mathbb{R}_+\cup\partial U\cup\partial U^*)$ and takes continuous boundary values on the jump contour from its complement.
\item[]\textbf{Jump conditions:} The boundary values are related by the jump conditions:
\begin{equation}
\dot{\mathbf{E}}_+(w)=\dot{\mathbf{E}}_-(w)\mathbf{C}(w)\left(W_\mathrm{gc}(w)^{-\sigma_3}+\ell(X,T)\sigma_+\right)\mathbf{C}(w)^{-1},\quad w\in\partial U,
\label{eq:EdotJumpPartialU}
\end{equation}
\begin{equation}
\dot{\mathbf{E}}_+(w)=\dot{\mathbf{E}}_-(w)\mathbf{C}(w^*)^*\left(W_\mathrm{gc}(w^*)^{*-\sigma_3}+\ell(X,T)^*\sigma_+\right)^{-1}\mathbf{C}(w^*)^{*-1},\quad w\in\partial U^*
\end{equation}
(here $\partial U$ has clockwise orientation and $\partial U^*$ has counter-clockwise orientation), and
\begin{equation}
\dot{\mathbf{E}}_+(w)=\sigma_2\dot{\mathbf{E}}_-(w)\sigma_2,\quad w>0.
\label{eq:EdotJumpRplus}
\end{equation}
Here $\mathbf{C}(w)=\mathbf{C}(w;\theta,\nu,W_\mathrm{gc}(\cdot))$ is defined by \eqref{eq:CmatrixDefine} and the parameters $a$ and $b$ enter via the linear function $\ell(X,T)$ defined by \eqref{eq:ell-def}.
\item[]\textbf{Normalization:} The following normalization condition holds:  $\dot{\mathbf{E}}(w)\to\mathbb{I}$ as $w\to\infty$.  
\end{itemize}
\label{rhp:E-parametrix}
\end{rhp}
\begin{lemma}
Let $W_\mathrm{gc}(\cdot)$, $\theta\in (0,\pi)$, $a<0$, and $b>0$ be fixed.  Then Riemann--Hilbert Problem~\ref{rhp:E-parametrix} has a unique solution $\dot{\mathbf{E}}(w)=\dot{\mathbf{E}}(w;X,T,\nu)$ with unit determinant for all $(X,T,\nu)\in\mathbb{R}^3$, and given a compact set $K\subset\mathbb{R}^2$, $\dot{\mathbf{E}}(w;X,T,\nu)$ is uniformly bounded for $w\in\mathbb{C}\setminus (\mathbb{R}_+\cup\partial U\cup \partial U^*)$, $(X,T)\in K$, and $\nu\in\mathbb{R}$.
\label{lemma:dotE-exists}
\end{lemma}
\begin{proof}
By a standard Liouville argument, it is easy to check that this problem has at most one solution which must satisfy Schwarz symmetry in the form $\dot{\mathbf{E}}(w^*)^*=\dot{\mathbf{E}}(w)$ and which has unit determinant.  Moreover by removing the non-standard jump condition across the positive real axis by mapping to the $z$-plane via a definition analogous to \eqref{eq:go-to-z-plane}, one easily confirms that the corresponding equivalent Riemann--Hilbert problem satisfies the conditions of Zhou's vanishing lemma \cite{Zhou1989} and hence there is a unique solution.  Since the parameters $(X,T,\nu)$ enter the jump conditions analytically, it follows via analytic Fredholm theory that the solution is also real-analytic in these variables.  The desired bound then follows because the jump matrices are $2\pi$-periodic in $\nu$ (this dependence is hidden in $\mathbf{C}(w)$ via the phase-linearized outer parametrix $\dot{\mathbf{O}}^\mathrm{out,l}(w)$), so that $\nu\in S^1$, and of course $S^1\times K$ is a compact subset of $\mathbb{R}^3$.
\end{proof}
In practice, we will replace the angle $\nu$ with $\nu=\epsilon^{-1}\Phi^l(t)$.

\subsection{Corrected error matrix, small-norm problem, and formul\ae\ for the potentials}
\label{sec:corrected-error}
Taking fixed parameters $\theta=\arg(\alpha_\mathrm{gc})$, $a<0$ and $b>0$ from Lemma~\ref{lemma:PartialsOfs}, and conformal map $W_\mathrm{gc}(\cdot)=W(\cdot;0,t_\mathrm{gc})$ from Lemma~\ref{lemma:ConformalMap}, Let $\dot{\mathbf{E}}(w)=\dot{\mathbf{E}}(w;X,T,\epsilon^{-1}\Phi^l(t))$ be the solution of Riemann--Hilbert Problem~\ref{rhp:E-parametrix}.  We compare $\dot{\mathbf{E}}(w)$ with $\mathbf{E}(w)$ itself by defining a \emph{corrected error matrix} $\mathbf{E}^\mathrm{c}(w):=\mathbf{E}(w)\dot{\mathbf{E}}(w)^{-1}$.  

Observe that $\mathbf{E}^\mathrm{c}(w)$ is analytic exactly where $\mathbf{E}(w)$ is, because the jump contour for $\dot{\mathbf{E}}(w)$ is contained within that for $\mathbf{E}(w)$, and like $\mathbf{E}(w)$, $\mathbf{E}^\mathrm{c}(w)$ is required to extend continuously to the jump contour from each component of its complement.  First we calculate the the jump for $\mathbf{E}^\mathrm{c}(w)$ across $\mathbb{R}_+$, where we pointed out  that $\mathbf{E}(w)$ satisfies $\mathbf{E}_+(w)=\sigma_2\mathbf{E}_-(w)\sigma_2(\mathbb{I}+\mathcal{O}(\epsilon))$ in which the error term $\mathcal{O}(\epsilon)$ vanishes identically for $|w-1|>\delta$.  Using the jump condition \eqref{eq:EdotJumpRplus}, it follows that
\begin{equation}
\begin{split}
\mathbf{E}^\mathrm{c}_+(w)&=\sigma_2\mathbf{E}^\mathrm{c}_-(w)\sigma_2\cdot\sigma_2\dot{\mathbf{E}}_-(w)\sigma_2(\mathbb{I}+\mathcal{O}(\epsilon))\sigma_2\dot{\mathbf{E}}_-(w)^{-1}\sigma_2 \\ &= \sigma_2\mathbf{E}^\mathrm{c}_-(w)\sigma_2(\mathbb{I}+\mathcal{O}(\epsilon)),\quad w\in\mathbb{R}_+
\end{split}
\end{equation}
where the second line follows from Lemma~\ref{lemma:dotE-exists} under the assumption $(X,T)\in K$ for $K$ compact.  Note that the error term on the second line is generally different than on the first line, but still has support $[1-\delta,1+\delta]$.  On the remaining parts of the jump contour with the exception of the closed curves $\partial U\cup\partial U^*$, we have $\dot{\mathbf{E}}_+(w)=\dot{\mathbf{E}}_-(w)=\dot{\mathbf{E}}(w)$ and we pointed out in Section~\ref{sec:AwayErrorAnalysis} that $\mathbf{E}_+(w)=\mathbf{E}_-(w)(\mathbb{I}+\mathcal{O}(\epsilon^{3/5}))$ holds uniformly.  Therefore on all such arcs,
\begin{equation}
\begin{split}
\mathbf{E}_+^\mathrm{c}(w)&=\mathbf{E}_-^\mathrm{c}(w)\dot{\mathbf{E}}(w)(\mathbb{I}+\mathcal{O}(\epsilon^{3/5}))\dot{\mathbf{E}}(w)^{-1}\\
&=\mathbf{E}_-^\mathrm{c}(w)(\mathbb{I}+\mathcal{O}(\epsilon^{3/5})),
\end{split}
\end{equation}
where again we used Lemma~\ref{lemma:dotE-exists} and the error term on the second line is generally a different function from that on the first line.  Next we calculate the remaining jump conditions holding for $w\in\partial U\cup\partial U^*$.  For $w\in\partial U$, we combine \eqref{eq:Near-E-Jump-partialU} with \eqref{eq:EdotJumpPartialU} to obtain
\begin{equation}
\begin{split}
\mathbf{E}^\mathrm{c}_+(w)&=\mathbf{E}^\mathrm{c}_-(w)\dot{\mathbf{E}}_-(w)\left[\mathbf{C}(w)(W_\mathrm{gc}(w)^{-\sigma_3}+\ell(X,T)\sigma_+)\mathbf{C}(w)^{-1}+\mathcal{O}(\epsilon^{1/5})\right]\\
&\qquad\qquad\qquad{}\cdot\left[\mathbf{C}(w)(W_\mathrm{gc}(w)^{-\sigma_3}+\ell(X,T)\sigma_+)\mathbf{C}(w)^{-1}\right]^{-1}\dot{\mathbf{E}}_-(w)^{-1}\\
&=\mathbf{E}^\mathrm{c}_-(w)(\mathbb{I}+\mathcal{O}(\epsilon^{1/5})),\quad \text{uniformly for $w\in\partial U$,}
\end{split}
\end{equation}
where we used again Lemma~\ref{lemma:dotE-exists} and the fact that $W_\mathrm{gc}(w)^{-\sigma_3}+\ell(X,T)\sigma_+$ has unit determinant and is uniformly bounded on $\partial U$.  Similar calculations show that in the same uniform sense,
\begin{equation}
\mathbf{E}^\mathrm{c}_+(w)=\mathbf{E}^\mathrm{c}_-(w)(\mathbb{I}+\mathcal{O}(\epsilon^{1/5})),\quad w\in\partial U^*.
\end{equation}
Finally, we note that since $\mathbf{E}(w)\to \mathbb{I}$ (because $\mathbf{O}(w)$ does by hypothesis and $\dot{\mathbf{O}}(w)$ does by explicit construction), the normalization condition on $\dot{\mathbf{E}}(w)$ in Riemann--Hilbert Problem~\ref{rhp:E-parametrix} guarantees that also $\mathbf{E}^\mathrm{c}(w)\to\mathbb{I}$ as $w\to\infty$.  

It follows that upon carrying $\mathbf{E}^\mathrm{c}(w)$ over to the $z$-plane as in \eqref{eq:go-to-z-plane} we arrive at a Riemann--Hilbert problem of small-norm type, and we therefore deduce that $\mathbf{E}^\mathrm{c}(w)=\mathbb{I}+\mathcal{O}(\epsilon^{1/5})$ holds uniformly on the complement of the jump contour.  In particular, $\mathbf{E}^\mathrm{c}(0)=\mathbb{I}+\mathcal{O}(\epsilon^{1/5})$, which implies that $\mathbf{E}(0)=\mathbf{E}^\mathrm{c}(0)\dot{\mathbf{E}}(0)=(\mathbb{I}+\mathcal{O}(\epsilon^{1/5}))\dot{\mathbf{E}}(0)$.  Next we substitute into the exact formul\ae\ \eqref{eq:AwayFormulaeForCosSin} which are also valid in the current situation to find
\begin{equation}
\cos(\tfrac{1}{2}u_N(x,t))=\ddot{C} + \mathcal{O}(\epsilon^{1/5})\quad\text{and}\quad\sin(\tfrac{1}{2}u_N(x,t))=\ddot{S}+\mathcal{O}(\epsilon^{1/5})
\end{equation}
where
\begin{equation}
\ddot{C}:=(\dot{\mathbf{E}}(0)\dot{\mathbf{O}}^\mathrm{out,l}(0))_{11}\quad\text{and}\quad\ddot{S}:=(\dot{\mathbf{E}}(0)\dot{\mathbf{O}}^\mathrm{out,l}(0))_{21}.
\label{eq:ddotCddotS}
\end{equation}

\subsection{Differential equations satisfied by the leading terms}
Here we will show that the leading terms $\ddot{C}$ and $\ddot{S}$ for $\cos(\tfrac{1}{2}u_N(x,t))$ and $\sin(\tfrac{1}{2}u_N(x,t))$ respectively can themselves be written in the form
\begin{equation}
\ddot{C}=\cos(\tfrac{1}{2}U(X,T))\quad\text{and}\quad\ddot{S}=\sin(\tfrac{1}{2}U(X,T))
\end{equation}
where $U(X,T)$ is an exact solution of the sine-Gordon equation in the form
\begin{equation}
\frac{\partial^2U}{\partial T^2}-\frac{\partial^2U}{\partial X^2}+\sin(U)=0.
\label{eq:sG-XT}
\end{equation}
To this end, we first prove the analogous result for $\dot{C}$ and $\dot{S}$, using their connection with the phase-linearized outer parametrix $\dot{\mathbf{O}}^\mathrm{out,l}(w)$ as given in \eqref{eq:PhaseLinearizedOuterParametrix-at-zero}.  Of course a direct proof using instead the formul\ae\ \eqref{eq:C-dot-define}--\eqref{eq:S-dot-define} would be easier, but we will leverage the connection with the phase-linearized outer parametrix to prove the more interesting result for $\ddot{C}$ and $\ddot{S}$.
\subsubsection{Relating $\dot{C}$ and $\dot{S}$ to sine-Gordon}
\label{sec:CdotSdot-solve-sG}
Note that by definition of $\dot{\mathbf{O}}^\mathrm{out,l}(w)$ all dependence on $(x,t)$ actually enters via the combination $t/\epsilon$, so we regard $\dot{\mathbf{O}}^\mathrm{out,l}(w)$ as depending parametrically on $(X,T)=(\epsilon^{-1}(x-x_\mathrm{p}),\epsilon^{-1}(t-t_\mathrm{p}))$.  We wish to apply a dressing argument to $\dot{\mathbf{O}}^\mathrm{out,l}(w)=\dot{\mathbf{O}}^\mathrm{out,l}(w;X,T)$, but such an argument requires that all dependence on $(X,T)$ in the jump matrices can be removed by conjugation by $\ee^{\ii \mathcal{Q}(w)\sigma_3}$ where $\mathcal{Q}(w)=\mathcal{Q}(w;X,T):=E(w)X+D(w)T$, in which $E(w)$ and $D(w)$ are given by \eqref{eq:E-and-D-define}.  To arrive at such jump conditions, we attempt to introduce a scalar function $G(w)=G(w;X,T)$ analytic off all jump contours for $\dot{\mathbf{O}}^\mathrm{out,l}(w;X,T)$, uniformly bounded in $w$ and continuous up to the jump contours, satisfying $G(w)\to 0$ as $w\to\infty$, as well as jump conditions of the form:
\begin{equation}
G_+(w)+G_-(w)=\begin{cases}0,&\quad w>0,\\
2\ii \mathcal{Q}(w;X,T)+\ii \eta_0,&\quad w\in\beta,\\
2\ii \mathcal{Q}(w;X,T)-\ii \eta_0, &\quad w\in\beta^*,
\end{cases}
\end{equation}
where $\eta_0$ is real and independent of $w$.   The value of $\eta_0$ needs to be chosen to ensure the existence of $G(w)$ but it will turn out to be closely related to the linearized phase in the Riemann--Hilbert conditions for $\dot{\mathbf{O}}^\mathrm{out,l}(w)$.  To see this, we write $G(w)$ in the form
\begin{equation}
G(w)=\ii \mathcal{Q}(w;X,T) + \frac{R^\mathrm{gc}(w)}{\sqrt{-w}}\upsilon(w)
\end{equation}
for a function $\upsilon(w)$ to be determined.  Here, $R^\mathrm{gc}(w)$ simply denotes the special case of the square-root function $R(w)$ for which the branch points are the gradient catastrophe values $(\alpha,\alpha^*)=(\alpha_\mathrm{gc},\alpha^*_\mathrm{gc})=(\ee^{\ii\theta},\ee^{-\ii\theta})$ and whose branch cut is $\beta\cup\beta^*=\beta_\mathrm{gc}\cup\beta_\mathrm{gc}^*$.  Since $R^\mathrm{gc}(w)=w+\mathcal{O}(1)$ as $w\to\infty$, in order that $G(w)\to 0$ as $w\to\infty$, we require that $\upsilon(w)\to -(X+T)/4$ as $w\to\infty$.  Taking into account that the ratio $R^\mathrm{gc}(w)/\sqrt{-w}$ changes sign across all jump contours where the sum of boundary values is specified for $G$, we see that the Plemelj formula can be used to solve for $\upsilon(w)$:
\begin{equation}
\upsilon(w)=-\frac{1}{4}(X+T)+\frac{\eta_0}{2\pi}\int_{\beta}\frac{\sqrt{-w'}\,\dd w'}{R^\mathrm{gc}_+(w')(w'-w)}-\frac{\eta_0}{2\pi}\int_{\beta^*}\frac{\sqrt{-w'}\,\dd w'}{R^\mathrm{gc}_+(w')(w'-w)}.
\end{equation}
This formula yields an expression for $G(w)$ that satisfies all conditions required except that $G(w)$ is bounded as $w\to 0$.  For this, we need to choose $\eta_0$ so that $\upsilon(0)=(X-T)/(4R^\mathrm{gc}(0))$.  Note that $R^\mathrm{gc}(0)=-1$ for $\alpha_\mathrm{gc}=\ee^{\ii\theta}$.  The integration contours in the expression for $\upsilon(0)$ can be deformed by Cauchy's Theorem so that 
\begin{equation}
\upsilon(0)=-\frac{1}{4}(X+T)+\frac{\eta_0}{2\pi}\int_{-\infty}^0\frac{\dd w'}{R^\mathrm{gc}(w')\sqrt{-w'}}.
\end{equation}
The latter integral turns out to be a complete elliptic integral of the first kind (see the remark at the end of Section~\ref{sec:Darboux} for details):
\begin{equation}
\upsilon(0)=-\frac{1}{4}(X+T)-\frac{\eta_0}{\pi}\KK(m_\mathrm{gc}).
\end{equation}
Therefore imposing the condition $\upsilon(0)=(X-T)/(4R^\mathrm{gc}(0))$ determines $\eta_0$ as
\begin{equation}
\eta_0=-\frac{\pi T}{2\KK(m_\mathrm{gc})}=\omega_\mathrm{gc} T,\quad \omega_\mathrm{gc}:=\omega(0,t_\mathrm{gc}).
\end{equation}
Note that $G(0)=0$ holds unambiguously.

Now we use the function $G(w)$ to define a related matrix by 
\begin{equation}
\mathbf{K}(w;X,T):=\dot{\mathbf{O}}^\mathrm{out,l}(w;X,T)\ee^{G(w;X,T)\sigma_3}.
\end{equation}
We see that $\mathbf{K}(w;X,T)$ has the following properties:
\begin{itemize}
\item It is analytic for $w\in\mathbb{C}\setminus(\mathbb{R}_+\cup\beta\cup\beta^*)$, is bounded and continuous up to the jump contour except near $w=\alpha_\mathrm{gc}$ and $w=\alpha_\mathrm{gc}^*$ where it blows up no worse than a negative one-fourth power.
\item It tends to $\mathbb{I}$ as $w\to\infty$.
\item It satisfies the following jump conditions:
\begin{equation}
\mathbf{K}_+(w;X,T)=\sigma_2\mathbf{K}_-(w;X,T)\sigma_2,\quad w>0,
\end{equation}
\begin{equation}
\mathbf{K}_+(w;X,T)=\mathbf{K}_-(w;X,T)\ee^{-\ii \mathcal{Q}(w;X,T)\sigma_3}\ii\sigma_1\ee^{\ii (\Phi_\mathrm{gc}+t_\mathrm{gc}\omega_\mathrm{gc})\sigma_3/(2\epsilon)}\ee^{\ii \mathcal{Q}(w;X,T)\sigma_3},\quad w\in\beta,
\end{equation}
and a similar jump condition consistent with the Schwarz symmetry $\mathbf{K}(w^*;X,T)=\mathbf{K}(w;X,T)^*$ for $w\in\beta^*$.
\item $\mathbf{K}(0;X,T)=\dot{\mathbf{O}}^\mathrm{out,l}(0;X,T)$.
\end{itemize}
The jump conditions for this problem now depend upon $(X,T)$ only via conjugation by $\ee^{\ii \mathcal{Q}(w;X,T)\sigma_3}$, so the dressing method applies.  That is, one defines the matrix $\mathbf{L}(w;X,T):=\mathbf{K}(w;X,T)\ee^{-\ii\mathcal{Q}(w;X,T)\sigma_3}$ and checks that its jump conditions are independent of $(X,T)$.  Hence $\partial_X\mathbf{L}(w;X,T)\cdot\mathbf{L}(w;X,T)^{-1}$ and $\partial_T\mathbf{L}(w;X,T)\cdot\mathbf{L}(w;X,T)^{-1}$ are functions analytic in $\mathbb{C}\setminus\mathbb{R}_+$ with boundary behavior along $\mathbb{R}_+$ and asymptotic behavior near $w=0$ and $w=\infty$ sufficient to conclude that these are Laurent polynomials of degree $(1,1)$ in $\sqrt{-w}$ with coefficients extracted from expansions of $\mathbf{K}(w;X,T)$ near $w=0,\infty$.  These matrices can in turn be identified with the coefficient matrices in \eqref{eq:wLaxPair} (in which $\epsilon=1$ and $(x,t)$ is replaced with $(X,T)$).  Thus one derives the Zakharov--Faddeev--Takhtajan Lax pair for \eqref{eq:sG-XT} directly from the Riemann--Hilbert conditions for $\mathbf{K}(w;X,T)$, and this in turn implies that the coefficients in the Lax pair are subject to the compatibility condition, which is the nonlinear partial differential equation \eqref{eq:sG-XT} itself.  Full details can be found in the proof of \cite[Proposition 2.1]{BuckinghamMiller2013}, but in any case this argument shows that 
\begin{equation}
\dot{C}:=K_{11}(0;X,T)
=\dot{O}_{11}^\mathrm{out,l}(0;X,T)\quad\text{and}\quad
\dot{S}:=K_{21}(0;X,T)
=\dot{O}_{21}^\mathrm{out,l}(0;X,T)
\end{equation}
are connected with an exact global solution $U(X,T)$ of the sine-Gordon equation in the form \eqref{eq:sG-XT} by $\dot{C}=\cos(\tfrac{1}{2}U(X,T))$ and $\dot{S}=\sin(\tfrac{1}{2}U(X,T))$.

\subsubsection{$\dot{\mathbf{E}}(w)$ as a Darboux transformation matrix for $\mathbf{K}(w;X,T)$}
\label{sec:Darboux}
Let $\dot{\mathbf{E}}(w)$ denote the solution of Riemann--Hilbert Problem~\ref{rhp:E-parametrix} for the parameter values indicated at the beginning of Section~\ref{sec:corrected-error}.  Again it is clear that this matrix depends on $(x,t)$ via the combinations $(X,T)$ so we write $\dot{\mathbf{E}}(w;X,T)$.   Consider the product
\begin{equation}
\widehat{\mathbf{K}}(w)=\widehat{\bfK}(w;X,T):=\dot{\mathbf{E}}(w;X,T)\mathbf{K}(w;X,T).
\end{equation}
It is straightforward to check that $\widehat{\mathbf{K}}(w;X,T)$ has all of the previously enumerated properties of $\mathbf{K}(w;X,T)$ except that $\widehat{\mathbf{K}}(w;X,T)$ has additional jump discontinuities across the closed contours $\partial U$ and $\partial U^*$.  The jump across $\partial U$ reads:
\begin{multline}
\widehat{\mathbf{K}}_+(w;X,T)=\widehat{\mathbf{K}}_-(w;X,T)
\ee^{-\ii\Omega_{\mathrm{p},N}\sigma_3/2}\ee^{-(G(w;X,T)-\ii\omega_\mathrm{gc}T/2)\sigma_3}\mathbf{M}W_\mathrm{gc}(w)^{-\sigma_3/4}\\
{}\cdot\left(W_\mathrm{gc}(w)^{-\sigma_3}+\ell(X,T)\sigma_+\right)W_\mathrm{gc}(w)^{\sigma_3/4}\mathbf{M}^{-1}\ee^{(G(w;X,T)-\ii\omega_\mathrm{gc}T/2)\sigma_3}\ee^{\ii\Omega_{\mathrm{p},N}\sigma_3/2},\quad w\in\partial U,
\end{multline}
where 
\begin{equation}
\Omega_{\mathrm{p},N}:=\frac{\Phi_\mathrm{gc}-(t_\mathrm{p}-t_\mathrm{gc})\omega_\mathrm{gc}}{\epsilon}.
\label{eq:Omega}
\end{equation}
Here, for simplicity we have written $W_\mathrm{gc}(w)^{\pm 1/4}$ where we mean the analytic continuation of $W_\mathrm{gc}(w)^{\pm 1/4}$ from $U\cap(\mathrm{III}\cup\mathrm{IV})$ to $U\setminus\beta$. Note that a similar interpretation of fractional powers of $W_\mathrm{gc}(w)$ holds throughout the following calculations.  The jump condition across $\partial U^*$ is consistent with the Schwarz symmetry $\widehat{\mathbf{K}}(w^*)=\widehat{\mathbf{K}}(w)^*$.

We have already argued that the sine-Gordon Lax pair equations can be deduced from $\mathbf{K}(w;X,T)$ by a dressing argument.  However we cannot immediately apply the same argument to $\widehat{\mathbf{K}}(w;X,T)$ because the jumps across $\partial U\cup\partial U^*$ do not become independent of $(X,T)$ upon conjugation by $\ee^{\ii \mathcal{Q}(w;X,T)\sigma_3}$.  Nonetheless, we begin the same way, by defining $\widehat{\mathbf{L}}(w;X,T):=\widehat{\mathbf{K}}(w;X,T)\ee^{-\ii \mathcal{Q}(w;X,T)\sigma_3}$.  It is easy to check that $\widehat{\mathbf{L}}(w;X,T)$ has the following properties:
\begin{itemize}
\item It is analytic for $w\in\mathbb{C}\setminus(\mathbb{R}_+\cup\beta\cup\beta^*\cup\partial U\cup\partial U^*)$,  is bounded and continuous up to the jump contour, except near $w=\alpha_\mathrm{gc}$ and $w=\alpha_\mathrm{gc}^*$ where it blows up no worse than a negative one-fourth power, and except near $w=0$ where $\widehat{\mathbf{L}}(w;X,T)\ee^{\ii \mathcal{Q}(w;X,T)\sigma_3}$ is bounded, and except near $w=\infty$.
\item 
$\widehat{\mathbf{L}}(w;X,T)\ee^{\ii \mathcal{Q}(w;X,T)\sigma_3}$ tends to $\mathbb{I}$ as $w\to\infty$.
\item 
It satisfies the following jump conditions:
\begin{equation}
\widehat{\mathbf{L}}_+(w;X,T)=\sigma_2\widehat{\mathbf{L}}_-(w;X,T)\sigma_2,\quad w>0,
\end{equation}
\begin{equation}
\widehat{\mathbf{L}}_+(w;X,T)=\widehat{\mathbf{L}}_-(w;X,T)\ii\sigma_1\ee^{\ii(\Phi_\mathrm{gc}+t_\mathrm{gc}\omega_\mathrm{gc})\sigma_3/(2\epsilon)},\quad w\in\beta,
\end{equation}
\begin{equation}
\widehat{\mathbf{L}}_+(w;X,T)=\widehat{\mathbf{L}}_-(w;X,T)\mathbf{V}(w;X,T),\quad w\in\partial U,
\end{equation}
where
\begin{multline}
\mathbf{V}(w;X,T):=\ee^{-\ii\Omega_{\mathrm{p},N}\sigma_3/2}\ee^{-(G(w;X,T)-\ii \mathcal{Q}(w;X,T)-\ii\omega_\mathrm{gc}T/2)\sigma_3}\mathbf{M}W_\mathrm{gc}(w)^{-\sigma_3/4}\\
{}\cdot\left(W_\mathrm{gc}(w)^{-\sigma_3}+\ell(X,T)\sigma_+\right)W_\mathrm{gc}(w)^{\sigma_3/4}\mathbf{M}^{-1}\ee^{(G(w;X,T)-\ii \mathcal{Q}(w;X,T)-\ii\omega_\mathrm{gc}T/2)\sigma_3}\ee^{\ii\Omega_{\mathrm{p},N}\sigma_3/2}.
\end{multline}
Similar jump conditions consistent with the Schwarz symmetry $\widehat{\mathbf{L}}(w^*;X,T)=\widehat{\mathbf{L}}(w;X,T)^*$ hold on the contours $\beta^*$ and $\partial U^*$.
\end{itemize}
We now consider the matrix $\mathbf{X}(w;X,T):=\partial\widehat{\mathbf{L}}(w;X,T)\cdot\widehat{\mathbf{L}}(w;X,T)^{-1}$ where $\partial=\partial_X$ or $\partial=\partial_T$.  Since the jump conditions for $\widehat{\mathbf{L}}(w;X,T)$ on $\mathbb{R}_+\cup\beta\cup\beta^*$ are independent of $(X,T)$, it follows that $\mathbf{X}(w;X,T)$ is analytic on $\beta\cup\beta^*$ and satisfies the jump condition
$\mathbf{X}_+(w;X,T)=\sigma_2\mathbf{X}_-(w;X,T)\sigma_2$ for $w>0$.
It is also easy to check that $\mathbf{X}(w;X,T)$ has removable singularities at $w=\alpha_\mathrm{gc}$ and $w=\alpha_\mathrm{gc}^*$ and takes continuous boundary values on $\mathbb{R}_+$ except at the origin.  Now the jump matrix $\mathbf{V}(w;X,T)$ depends on $(X,T)$, so the jump condition for $\mathbf{X}(w;X,T)$ on $\partial U$ reads
\begin{equation}
\mathbf{X}_+(w;X,T)=\mathbf{X}_-(w;X,T) + \widehat{\mathbf{L}}_-(w;X,T)\partial \mathbf{V}(w;X,T)\cdot\mathbf{V}(w;X,T)^{-1}\widehat{\mathbf{L}}_-(w;X,T)^{-1},\quad w\in\partial U.
\end{equation}
Since $\mathbf{X}_-(w;X,T)$ is the boundary value taken on $\partial U$ by the function $\mathbf{X}(w;X,T)$ analytic in $U$ from the interior, this formula shows that $\mathbf{X}(w;X,T)$ will admit an analytic continuation into $U$ through $\partial U$ from the exterior if $\widehat{\mathbf{L}}_-(w;X,T)\partial\mathbf{V}(w;X,T)\cdot\mathbf{V}(w;X,T)^{-1}\widehat{\mathbf{L}}_-(w;X,T)^{-1}$ is analytic in $U$.  

This condition is not so convenient to check, because although $\mathbf{V}(w;X,T)$ is known and has a simple form, $\widehat{\mathbf{L}}_-(w;X,T)$ is more implicitly determined (writing it down in any form requires Riemann theta functions of genus $1$).  However, we may observe that 
\begin{equation}
W_\mathrm{gc}(w)^{\sigma_3/4}\mathbf{M}^{-1}\ee^{(G(w;X,T)-\ii \mathcal{Q}(w;X,T)-\ii\omega_\mathrm{gc}T/2)\sigma_3}\ee^{\ii\Omega_{\mathrm{p},N}\sigma_3/2}
\end{equation}
is an exact local solution of the Riemann--Hilbert problem for $\widehat{\mathbf{L}}(w;X,T)$ within $U$ (i.e., it is analytic except on $\beta\cap U$ where it satisfies the jump condition satisfied by $\widehat{\mathbf{L}}(w;X,T)$, and it blows up to the same order at $w=\alpha_\mathrm{gc}$), and so 
\begin{equation}
\mathbf{F}_-(w;X,T):=\widehat{\mathbf{L}}_-(w;X,T)\ee^{-\ii\Omega_{\mathrm{p},N}\sigma_3/2}\ee^{-(G(w;X,T)-\ii \mathcal{Q}(w;X,T)-\ii\omega_\mathrm{gc}T/2)\sigma_3}\mathbf{M}W_\mathrm{gc}(w)^{-\sigma_3/4}
\end{equation} 
admits analytic continuation from $\partial U$ into $U$.  Thus $\widehat{\mathbf{L}}_-(w;X,T)\partial\mathbf{V}(w;X,T)\cdot\mathbf{V}(w;X,T)^{-1}\widehat{\mathbf{L}}_-(w;X,T)^{-1}=\mathbf{F}_-(w;X,T)\mathbf{R}(w;X,T)\mathbf{F}_-(w;X,T)^{-1}$ admits analytic continuation into $U$ if and only if $\mathbf{R}(w;X,T)$ does, where $\mathbf{R}(w;X,T)$ is explicitly given by
\begin{multline}
\mathbf{R}(w;X,T):=\\ W_\mathrm{gc}(w)^{\sigma_3/4}\mathbf{M}^{-1}\ee^{(G(w;X,T)-\ii \mathcal{Q}(w;X,T)-\ii\omega_\mathrm{gc}T/2)\sigma_3}\ee^{\ii\Omega_{\mathrm{p},N}\sigma_3/2}\cdot\partial\mathbf{V}(w;X,T)\cdot\mathbf{V}(w;X,T)^{-1}\\
{}\cdot\ee^{-\ii\Omega_{\mathrm{p},N}\sigma_3/2}\ee^{-(G(w;X,T)-\ii \mathcal{Q}(w;X,T)-\ii\omega_\mathrm{gc}T/2)\sigma_3}\mathbf{M}W_\mathrm{gc}(w)^{-\sigma_3/4}.
\end{multline}
Now, we calculate the derivatives with respect to $X$ and $T$, which enter $\mathbf{V}(w;X,T)$ only via $G(w;X,T)-\ii \mathcal{Q}(w;X,T)-\tfrac{1}{2}\ii\omega_\mathrm{gc}T$ and $\ell(X,T)$.  Thus:
\begin{multline}
\mathbf{R}(w;X,T)=\partial\left(G(w;X,T)-\ii \mathcal{Q}(w;X,T)-\frac{1}{2}\ii\omega_\mathrm{gc} T\right)\\
{}\cdot
\begin{bmatrix}-\ell(X,T)W_\mathrm{gc}(w)^{1/2} & W_\mathrm{gc}(w)^{1/2}-W_\mathrm{gc}(w)^{-3/2}+\ell(X,T)^2W_\mathrm{gc}(w)^{-1/2}\\W_\mathrm{gc}(w)^{-1/2}-W_\mathrm{gc}(w)^{3/2} & \ell(X,T)W_\mathrm{gc}(w)^{1/2}\end{bmatrix}
\\
{}+ \partial \ell(X,T) W_\mathrm{gc}(w)^{-1}\sigma_+.
\end{multline}
The second term on the right-hand side (on the third line) is clearly meromorphic in $U$ with a simple pole only at $w=\alpha_\mathrm{gc}$.  A simple contour deformation shows that 
\begin{equation}
G(w;X,T)-\ii \mathcal{Q}(w;X,T)-\frac{1}{2}\ii\omega_\mathrm{gc}T = \frac{R^\mathrm{gc}(w)}{\sqrt{-w}}j(w;X,T),
\end{equation}
where $j(w;X,T)$ is the function analytic at $w=\alpha_\mathrm{gc}$ defined by
\begin{equation}
j(w;X,T):=-\frac{1}{4}(X+T)+\frac{\omega_\mathrm{gc}T}{4\pi}\int_\lambda\frac{\sqrt{-w'}\,\dd w'}{R^\mathrm{gc}(w')(w'-w)}-\frac{\omega_\mathrm{gc}T}{2\pi}\int_{\beta^*}\frac{\sqrt{-w'}\,\dd w'}{R_+^\mathrm{gc}(w')(w'-w)},\quad w\in U
\label{eq:j-of-w}
\end{equation}
in which $\lambda$ denotes a clockwise-oriented loop beginning and terminating at $w=1$ on opposite sides of $\beta$ and enclosing $\beta$ and $U$.  Since $R^\mathrm{gc}(w)$ is independent of $(X,T)$, all dependence on $(X,T)$ is explicit and linear.  Moreover, $R^\mathrm{gc}(w)$ is locally proportional to $W_\mathrm{gc}(w)^{1/2}$, which implies that $\partial (G(w;X,T)-\ii \mathcal{Q}(w;X,T)-\tfrac{1}{2}\ii\omega_\mathrm{gc}T)$ has the form of the product of $W_\mathrm{gc}(w)^{1/2}$ with a function analytic and non-vanishing within $U$.  It follows that all elements of $\mathbf{R}(w;X,T)$ except possibly $R_{12}(w;X,T)$ are analytic within $U$.  The element $R_{12}(w;X,T)$ has an apparent simple pole at $w=\alpha_\mathrm{gc}$ only, but this singularity is removable if and only if
\begin{multline}
\partial \ell(X,T)=\left[\lim_{w\to \alpha_\mathrm{gc}}\frac{R^\mathrm{gc}(w)}{\sqrt{-w}W_\mathrm{gc}(w)^{1/2}}\right]\\
{}\cdot\partial\left[-\frac{1}{4}(X+T)+\frac{\omega_\mathrm{gc}T}{4\pi}
\int_\lambda\frac{\sqrt{-w'}\,\dd w'}{R^\mathrm{gc}(w')(w'-\alpha_\mathrm{gc})} -
\frac{\omega_\mathrm{gc}T}{2\pi}\int_{\beta^*}\frac{\sqrt{-w'}\,\dd w'}{R_+^\mathrm{gc}(w')(w'-\alpha_\mathrm{gc})}\right].
\end{multline}
Taking $\partial=\partial_X$ and $\partial=\partial_T$ and recalling \eqref{eq:ell-def} yields two conditions:
\begin{equation}
\ii a =   \left[\lim_{w\to \alpha_\mathrm{gc}}\frac{R^\mathrm{gc}(w)}{\sqrt{-w}W_\mathrm{gc}(w)^{1/2}}\right]\left(-\frac{1}{4}\right)
\label{eq:x-derivative-condition}
\end{equation}
and
\begin{equation}
b=\left[\lim_{w\to \alpha_\mathrm{gc}}\frac{R^\mathrm{gc}(w)}{\sqrt{-w}W_\mathrm{gc}(w)^{1/2}}\right]
\left(-\frac{1}{4} +\frac{\omega_\mathrm{gc}}{4\pi}\int_\lambda\frac{\sqrt{-w'}\,\dd w'}{R^\mathrm{gc}(w')(w'-\alpha_\mathrm{gc})} -
\frac{\omega_\mathrm{gc}}{2\pi}\int_{\beta^*}\frac{\sqrt{-w'}\,\dd w'}{R_+^\mathrm{gc}(w')(w'-\alpha_\mathrm{gc})}\right),
\label{eq:t-derivative-condition}
\end{equation}
where the coefficients $a<0$ and $b>0$ in $\ell(X,T)$ relate to partial derivatives of $s$ at the gradient catastrophe point as described in Lemma~\ref{lemma:PartialsOfs}.
\begin{lemma}
The conditions \eqref{eq:x-derivative-condition}--\eqref{eq:t-derivative-condition} hold, and therefore $R_{12}(w;X,T)$ may be considered to be analytic within $U$.
\label{lemma:X-T-identities}
\end{lemma}
\begin{proof}
Recall the expression for $\ii a = \partial s/\partial x (0,t_\mathrm{gc})$ obtained in Lemma~\ref{lemma:PartialsOfs}:
\begin{equation}
\ii a=\frac{\partial s}{\partial x}(0,t_\mathrm{gc})=-\ii\frac{(m_\mathrm{gc}(1-m_\mathrm{gc}))^{1/4}}{2|W_\mathrm{gc}'(\alpha_\mathrm{gc})|^{1/2}}.
\label{eq:x-derivative-old-formula}
\end{equation}
Now we compute the limit in \eqref{eq:x-derivative-condition} assuming that $w\in(\mathrm{III}\cup\mathrm{IV})\cap U$ so that $W_\mathrm{gc}(w)^{1/2}$ may be taken to have its original meaning with its branch cut on the negative real axis in the $W$-plane.  Assuming that $(w-\alpha_\mathrm{gc})^{1/2}$ is taken to agree with the principal branch in the distant right half $w$-plane and to have its branch cut being the union of an arc of the unit circle between $\alpha_\mathrm{gc}=\ee^{\ii\theta}$ and $-1$ with the ray $(-\infty,-1)$, we can then see easily that 
\begin{equation}
\mathop{\lim_{w\to\alpha_\mathrm{gc}}}_{w\in\mathrm{III}\cup\mathrm{IV}}\frac{R^\mathrm{gc}(w)}{(w-\alpha_\mathrm{gc})^{1/2}}=(\alpha_\mathrm{gc}-\alpha_\mathrm{gc}^*)^{1/2}=\ee^{\ii\pi/4}\sqrt{2\sin(\theta)},
\end{equation}
a number with argument equal to $\tfrac{1}{4}\pi$.  Also, 
\begin{equation}
\mathop{\lim_{w\to\alpha_\mathrm{gc}}}_{w\in\mathrm{III}\cup\mathrm{IV}}\frac{W_\mathrm{gc}(w)^{1/2}}{(w-\alpha_\mathrm{gc})^{1/2}}=Q'_\mathrm{gc}(0)=|W_\mathrm{gc}'(\alpha_\mathrm{gc})|^{1/2}\ee^{\ii\pi/4}\ee^{-\ii\theta/2}.
\end{equation}
Combining these proves \eqref{eq:x-derivative-condition} where \eqref{eq:x-derivative-old-formula} is used on the left-hand side.

In the proof of Lemma~\ref{lemma:PartialsOfs} we also obtained the formula
\begin{equation}
b=\frac{\partial s}{\partial t}(0,t_\mathrm{gc})=
-\ii\frac{(m_\mathrm{gc}(1-m_\mathrm{gc}))^{1/4}}{2|W'_\mathrm{gc}(\alpha_\mathrm{gc})|^{1/2}}\left[-\ii\frac{A_\mathrm{gc}}{B_\mathrm{gc}\sin(\theta)}-\ii\cot(\theta)\right],
\label{eq:t-derivative-old-formula}
\end{equation}
in which (cf., \eqref{eq:AB})
\begin{equation}
A_\mathrm{gc}:=\int_{\beta}\frac{\sqrt{-\xi}\,\dd\xi}{R^\mathrm{gc}_+(\xi)}+
\int_{\beta^*}\frac{\sqrt{-\xi}\,\dd\xi}{R^\mathrm{gc}_-(\xi)}\quad\text{and}\quad
B_\mathrm{gc}:=\int_{\beta}\frac{\dd\xi}{R_+^\mathrm{gc}(\xi)\sqrt{-\xi}} + 
\int_{\beta^*}\frac{\dd\xi}{R_-^\mathrm{gc}(\xi)\sqrt{-\xi}}.
\end{equation}
The latter integrals can be rewritten with the help of the loop contour $\lambda$ appearing in \eqref{eq:j-of-w} as follows:
\begin{equation}
A_\mathrm{gc}=\frac{1}{2}\int_\lambda\frac{\sqrt{-w}\,\dd w}{R^\mathrm{gc}(w)} +\frac{1}{2}
\int_{\lambda^*}\frac{\sqrt{-w}\,\dd w}{R^\mathrm{gc}(w)}\quad\text{and}\quad
B_\mathrm{gc}=\frac{1}{2}\int_\lambda\frac{\dd w}{R^\mathrm{gc}(w)\sqrt{-w}}+\frac{1}{2}\int_{\lambda^*}\frac{\dd w}{R^\mathrm{gc}(w)\sqrt{-w}}.
\label{eq:Agc-Bgc-loops}
\end{equation}
In light of the established consistency of \eqref{eq:x-derivative-condition} and \eqref{eq:x-derivative-old-formula}, to prove consistency of \eqref{eq:t-derivative-condition} and \eqref{eq:t-derivative-old-formula} it suffices to show that 
\begin{equation}
\sin(\theta)B_\mathrm{gc}+\sin(\theta)\int_\lambda\frac{\sqrt{-w}\,\dd w}{R^\mathrm{gc}(w)(w-\alpha_\mathrm{gc})}+\sin(\theta)\int_{\lambda^*}\frac{\sqrt{-w}\,\dd w}{R^\mathrm{gc}(w)(w-\alpha_\mathrm{gc})}+\ii A_\mathrm{gc}+\ii\cos(\theta)B_\mathrm{gc}=0
\label{eq:t-consistency}
\end{equation}
where we have also rewritten the last integral in the parentheses on the right-hand side of \eqref{eq:t-derivative-condition} in terms of a loop integral over $\lambda^*$ and taken into account the identity $\omega_\mathrm{gc}B_\mathrm{gc}=-\pi$.  Combining the integrals over $\lambda$ and $\lambda^*$ and using $\alpha_\mathrm{gc}=\ee^{\ii\theta}$, the left-hand side above reads
\begin{equation}
-\frac{\ii}{2}\int_{\lambda\cup\lambda^*}\frac{w^2-2\alpha_\mathrm{gc}w+\alpha_\mathrm{gc}\alpha_\mathrm{gc}^*}{\sqrt{-w}R^\mathrm{gc}(w)(w-\alpha_\mathrm{gc})}\,\dd w = \ii\int_{\lambda\cup\lambda^*}\frac{\dd}{\dd w}\left(\frac{\sqrt{-w}(w-\alpha_\mathrm{gc}^*)}{R^\mathrm{gc}(w)}\right)\,\dd w
\end{equation}
which vanishes because on $\lambda$ we can write $\sqrt{-w}=-\ii\sqrt{w}$ while on $\lambda^*$ we can write $\sqrt{-w}=\ii\sqrt{w}$ and because $\lambda\cup (-\lambda^*)$ is a closed clockwise-oriented loop around $\beta$ on which $\sqrt{w}(w-\alpha_\mathrm{gc}^*)/R^\mathrm{gc}(w)$ is single-valued.
\end{proof}
\begin{remark}
Taking the imaginary part of the identity \eqref{eq:t-consistency}, some contour deformations along the lines indicated above show that
\begin{equation}
\frac{A_\mathrm{gc}}{B_\mathrm{gc}\sin(\theta)}+\cot(\theta)=\frac{2\sin(\theta)}{B_\mathrm{gc}}\int_{-\infty}^0\frac{\sqrt{-w}\,\dd w}{R^\mathrm{gc}(w)^3}.
\label{eq:Agc-etc-integral}
\end{equation}
Similar contour deformations applied to $B_\mathrm{gc}$ given by \eqref{eq:Agc-Bgc-loops} can be written as
\begin{equation}
B_\mathrm{gc}=-\int_{-\infty}^0\frac{\dd w}{R^\mathrm{gc}(w)\sqrt{-w}}.
\end{equation}
Since $R^\mathrm{gc}(w)<0$ for real $w<1$, it follows that $A_\mathrm{gc}/(B_\mathrm{gc}\sin(\theta))+\cot(\theta)<0$.  In \cite[Pg.\@ 56]{BuckinghamMiller2013} explicit substitutions are given that show that $B_\mathrm{gc}$ is related to the complete elliptic integral of the first kind:  $B_\mathrm{gc}=2\KK(m_\mathrm{gc})$ where $m_\mathrm{gc}$ is given by \eqref{eq:mgc-alphagc}.  Similar substitutions in the integral in \eqref{eq:Agc-etc-integral} show that it involves elliptic integrals of both the first and second kinds, and indeed that $A_\mathrm{gc}/(B_\mathrm{gc}\sin(\theta))+\cot(\theta)$ can be written in the form $-\rho(m_\mathrm{gc})$ (see \eqref{eq:rho-func-define}) and that the quantity $\rho(m_\mathrm{gc})$ is strictly positive for $0<m_\mathrm{gc}<1$.
\end{remark}

It follows that $\mathbf{X}(w;X,T)$ can be re-defined within $U$ so as to be analytic across $\partial U$ and within $U$ (this function will not necessarily take the boundary value $\mathbf{X}_-(w;X,T)$ on $\partial U$ from within $U$).  Appealing to Schwarz symmetry shows that a similar redefinition can be made within $U^*$.  Thus, $\mathbf{X}(w;X,T)$ can be regarded as a function analytic in $\mathbb{C}$ except for along $\mathbb{R}_+$ where it satisfies the jump condition $\mathbf{X}_+(w;X,T)=\sigma_2\mathbf{X}_-(w;X,T)\sigma_2$ and having computable behavior as $w\to 0$ and $w\to\infty$.  Following now the same sort of arguments described above in Section~\ref{sec:CdotSdot-solve-sG}, one identifies $\mathbf{X}(w;X,T)$ with a Laurent polynomial in $\sqrt{-w}$ of degree $(1,1)$.  The asymptotics of $\widehat{\mathbf{L}}(w;X,T)$ as $w\to 0$ and $w\to\infty$ allow the identification of $\mathbf{X}(w;X,T)$ with the coefficient matrix in the $X$ or $T$ equation (depending on the interpretation of $\partial$) of the Zakharov--Faddeev--Takhtajan Lax pair for the sine-Gordon equation in the form \eqref{eq:sG-XT}, for potentials 
\begin{equation}
\begin{split}
\ddot{C}(X,T):=\widehat{K}_{11}(0;X,T)
&=\dot{E}_{11}(0;X,T)K_{11}(0;X,T) + \dot{E}_{12}(0;X,T)K_{21}(0;X,T)\\
&=\dot{E}_{11}(0;X,T)\dot{O}^\mathrm{out,l}_{11}(0;X,T)+\dot{E}_{12}(0;X,T)\dot{O}^\mathrm{out,l}_{21}(0;X,T)
\quad\text{and}\\
\ddot{S}(X,T):=\widehat{K}_{21}(0;X,T)
&=\dot{E}_{21}(0;X,T)K_{11}(0;X,T) + \dot{E}_{22}(0;X,T)K_{21}(0;X,T)\\&=\dot{E}_{21}(0;X,T)\dot{O}^\mathrm{out,l}_{11}(0;X,T)+\dot{E}_{22}(0;X,T)\dot{O}^\mathrm{out,l}_{21}(0;X,T)
\end{split}
\label{eq:cos-sin-linearized}
\end{equation}
that can be written in the form $\ddot{C}(X,T)=\cos(\tfrac{1}{2}U(X,T))$ and $\ddot{S}(X,T)=\sin(\tfrac{1}{2}U(X,T))$ for some global solution $U(X,T)$ of the sine-Gordon equation in the form \eqref{eq:sG-XT}.  The solution $U(X,T)$ related to $\ddot{C}$ and $\ddot{S}$ is a \emph{different solution} than that related to $\dot{C}$ and $\dot{S}$ in Section~\ref{sec:CdotSdot-solve-sG}.  In particular the former solution has nontrivial $X$-dependence while the latter is independent of $X$.  The link between the two solutions is given by a B\"acklund transformation of \eqref{eq:sG-XT} implied by the Darboux transformation defined by the matrix $\dot{\mathbf{E}}(w;X,T)$.

\subsection{Explicit solution of Riemann--Hilbert Problem~\ref{rhp:E-parametrix}}
To solve for $\dot{\mathbf{E}}(w)$, we assume that in the domain exterior to the neighborhoods $U$ and $U^*$, $\dot{\mathbf{E}}(w)$ is a rational function of $\sqrt{-w}$ given by:
\begin{multline}
\dot{\mathbf{E}}(w)=\mathbb{I}+\frac{\mathbf{G}}{\sqrt{-w}-\sqrt{-\alpha_\mathrm{gc}}}-\frac{\sigma_2\mathbf{G}\sigma_2}{\sqrt{-w}+\sqrt{-\alpha_\mathrm{gc}}}+\frac{\mathbf{G}^*}{\sqrt{-w}-\sqrt{-\alpha_\mathrm{gc}^*}}-
\frac{\sigma_2\mathbf{G}^*\sigma_2}{\sqrt{-w}+\sqrt{-\alpha_\mathrm{gc}^*}},\\
w\in\mathbb{C}\setminus(U\cup U^*\cup\mathbb{R}_+),
\end{multline}
where $\mathbf{G}$ is a $2\times 2$ matrix independent of $w$ to be determined.  Being discontinuous only along the branch cut $\mathbb{R}_+$ of $\sqrt{-w}$, this ``exterior'' ansatz is clearly consistent with the analyticity condition of Riemann--Hilbert Problem~\ref{rhp:E-parametrix}, and it automatically satisfies the needed jump condition for $w>0$.  It also satisfies the normalization condition as $w\to\infty$, and is Schwarz-symmetric.  Note that by finding two common denominators, we can also write the exterior ansatz in the form
\begin{multline}
\dot{\mathbf{E}}(w)=\mathbb{I} + \frac{(\sigma_2\mathbf{G}\sigma_2-\mathbf{G})\sqrt{-w}-(\sigma_2\mathbf{G}\sigma_2+\mathbf{G})\sqrt{-\alpha_\mathrm{gc}}}{w-\alpha_\mathrm{gc}} \\{}+
\frac{(\sigma_2\mathbf{G}^*\sigma_2-\mathbf{G}^*)\sqrt{-w}-(\sigma_2\mathbf{G}^*\sigma_2+\mathbf{G}^*)\sqrt{-\alpha_\mathrm{gc}^*}}{w-\alpha_\mathrm{gc}^*},\quad w\in\mathbb{C}\setminus(U\cup U^*\cup\mathbb{R}_+).
\label{eq:exterior-ansatz}
\end{multline}
It is furthermore convenient to represent the matrices $\sigma_2\mathbf{G}\sigma_2\pm\mathbf{G}$ equivalently in the form
\begin{equation}
\sigma_2\mathbf{G}\sigma_2-\mathbf{G}=\ii A_3\sigma_3+\ii A_1\sigma_1\quad\text{and}\quad
\sigma_2\mathbf{G}\sigma_2+\mathbf{G}=\frac{A_0}{\sqrt{-\alpha_\mathrm{gc}}}\mathbb{I} +\frac{\ii A_2}{\sqrt{-\alpha_\mathrm{gc}}}\sigma_2.
\label{eq:A's}
\end{equation}
where $A_j$, $j=0,1,2,3$, are complex scalars that parametrize the four elements of $\mathbf{G}$.
These scalars are determined by the condition that if $\dot{\mathbf{E}}_+(w)$ is the exterior boundary value of $\dot{\mathbf{E}}(w)$ given by \eqref{eq:exterior-ansatz} for $w\in\partial U$, then as $\dot{\mathbf{E}}_-(w)\mathbf{C}(w)$ represents the boundary value taken on $\partial U$ of a function analytic in $U$, the principal part of the Laurent series of
\begin{equation}
\dot{\mathbf{E}}_-(w)\mathbf{C}(w)=\dot{\mathbf{E}}_+(w)\left[\mathbf{C}(w)\left(W_\mathrm{gc}(w)^{-\sigma_3}+\ell(X,T)\sigma_+\right)\mathbf{C}(w)^{-1}\right]^{-1}\mathbf{C}(w)
\end{equation}
about $w=\alpha_\mathrm{gc}$ must vanish.  Since the right-hand side can be viewed as a product of two factors each having a simple pole at $w=\alpha_\mathrm{gc}$, the coefficients of $(w-\alpha_\mathrm{gc})^{-2}$ and $(w-\alpha_\mathrm{gc})^{-1}$ in the Laurent expansion of the product must both be set to zero.  These conditions, along with their complex conjugates (or equivalently, imposing similar conditions for the jump across $\partial U^*$), will uniquely determine $A_j$, $j=0,\dots,3,$ and hence $\mathbf{G}$.

Using \eqref{eq:A's} in \eqref{eq:exterior-ansatz}, the Laurent expansion of $\dot{\mathbf{E}}_+(w)$ about $w=\alpha_\mathrm{gc}$ is
\begin{equation}
\dot{\mathbf{E}}_+(w)=\frac{\mathbf{E}_1}{w-\alpha_\mathrm{gc}}+\mathbf{E}_0+\mathcal{O}(w-\alpha_\mathrm{gc}),\quad w\to\alpha_\mathrm{gc}
\end{equation}
where  
\begin{equation}
\mathbf{E}_1:=\ii\sqrt{-\alpha_\mathrm{gc}}(A_3\sigma_3+A_1\sigma_1)-(A_0\mathbb{I}+\ii A_2\sigma_2)
\end{equation}
and 
\begin{equation}
\mathbf{E}_0:=\mathbb{I}-\frac{\ii}{2\sqrt{-\alpha_\mathrm{gc}}}(A_3\sigma_3+A_1\sigma_1)-\frac{1}{2\ii\mathrm{Im}\{\alpha_\mathrm{gc}\}}\left[\ii\sqrt{-\alpha_\mathrm{gc}}(A_3^*\sigma_3+A_1^*\sigma_1)+(A_0^*\mathbb{I}+\ii A_2^*\sigma_2)\right].
\end{equation}
We then similarly obtain the Laurent expansion
\begin{multline}
\left[\mathbf{C}(w)\left(W_\mathrm{gc}(w)^{-\sigma_3}+\ell(X,T)\sigma_+\right)\mathbf{C}(w)^{-1}\right]^{-1}\mathbf{C}(w)=\\
\begin{bmatrix}0 & M_1\\
0 & M_2\end{bmatrix}\frac{1}{w-\alpha_\mathrm{gc}} +\begin{bmatrix}
0 & \delta  M_1 +m_1\\
0 & \delta M_2 +m_2 \end{bmatrix} + \mathcal{O}(w-\alpha_\mathrm{gc}),\quad
w\to\alpha_\mathrm{gc}
\end{multline}
where
\begin{equation}
M_j:=\frac{C_{j2}(\alpha_\mathrm{gc})}{W_\mathrm{gc}'(\alpha_\mathrm{gc})} \quad\text{and}\quad m_j:=\frac{C_{j2}'(\alpha_\mathrm{gc})}{W_\mathrm{gc}'(\alpha_\mathrm{gc})}-\ell(X,T) C_{j1}(\alpha_\mathrm{gc}),\quad j=1,2
\end{equation}
and
\begin{equation}
\delta:=-\frac{W_\mathrm{gc}''(\alpha_\mathrm{gc})}{2W_\mathrm{gc}'(\alpha_\mathrm{gc})}.
\end{equation}
Therefore, the condition that the coefficient of $(w-\alpha_\mathrm{gc})^{-2}$ in the Laurent expansion of $\dot{\mathbf{E}}_-(w)$ about $w=\alpha_\mathrm{gc}$ should vanish is the vector equation
\begin{equation}
\mathbf{E}_1\begin{bmatrix}M_1\\M_2\end{bmatrix}=\mathbf{0},
\label{eq:double-pole}
\end{equation}
and the condition that the residue of $\dot{\mathbf{E}}_-(w)$ at $w=\alpha_\mathrm{gc}$ should vanish is the vector equation
\begin{equation}
\mathbf{E}_0\begin{bmatrix}M_1\\M_2\end{bmatrix} + \delta\mathbf{E}_1\begin{bmatrix}M_1\\M_2\end{bmatrix} +\mathbf{E}_1\begin{bmatrix}m_1\\m_2\end{bmatrix}=\mathbf{0}.
\label{eq:simple-pole}
\end{equation}
(Note that taking into account \eqref{eq:double-pole} we may omit the term proportional to $\delta$.)
The condition \eqref{eq:double-pole} indicates that $\mathbf{E}_1$ has rank at most one, and given that $[M_1;M_2]^\top$ is in its kernel it may be written in the form
\begin{equation}
\mathbf{E}_1=\begin{bmatrix}u_1\\u_2\end{bmatrix}\begin{bmatrix}C_{22}(\alpha_\mathrm{gc});&-C_{12}(\alpha_\mathrm{gc})\end{bmatrix}.
\end{equation}
The scalars $A_j$ for $j=0,\dots,3$ can then be expressed in terms of $u_1$ and $u_2$ by 
\begin{equation}
\begin{split}
A_0&=-\frac{1}{2}(C_{22}(\alpha_\mathrm{gc})u_1-C_{12}(\alpha_\mathrm{gc})u_2)\\
A_1&=\frac{1}{2\ii\sqrt{-\alpha_\mathrm{gc}}}(C_{22}(\alpha_\mathrm{gc})u_2-C_{12}(\alpha_\mathrm{gc})u_1)\\
A_2&=\frac{1}{2}(C_{22}(\alpha_\mathrm{gc})u_2+C_{12}(\alpha_\mathrm{gc})u_1)\\
A_3&=\frac{1}{2\ii\sqrt{-\alpha_\mathrm{gc}}}(C_{22}(\alpha_\mathrm{gc})u_1+C_{12}(\alpha_\mathrm{gc})u_2).
\end{split}
\label{eq:As-us}
\end{equation}
So, it remains to determine $u_j$, $j=1,2$.  With the help of \eqref{eq:As-us}, we first express the elements of $\mathbf{E}_0$ explicitly in terms of $u_j$, $j=1,2$.  Then we observe that the terms in \eqref{eq:simple-pole} proportional to $\ell(X,T)$ simplify due to the fact that $\det(\mathbf{C}(\alpha_\mathrm{gc}))=1$, so ultimately \eqref{eq:simple-pole} takes the form
\begin{equation}
\mathbf{P}\mathbf{u} +\mathbf{Q}\mathbf{u}^*=-\begin{bmatrix}C_{12}(\alpha_\mathrm{gc})\\C_{22}(\alpha_\mathrm{gc})\end{bmatrix},
\label{eq:u-ustar}
\end{equation}
where $\mathbf{u}=[u_1; u_2]^\top$ and where 
\begin{equation}
\mathbf{P}:=(C_{22}(\alpha_\mathrm{gc})C_{12}'(\alpha_\mathrm{gc})-C_{22}'(\alpha_\mathrm{gc})C_{12}(\alpha_\mathrm{gc})-\ell(X,T) W_\mathrm{gc}'(\alpha_\mathrm{gc}))\mathbb{I} + \frac{1}{4\alpha_\mathrm{gc}}(C_{12}(\alpha_\mathrm{gc})^2+C_{22}(\alpha_\mathrm{gc})^2)\ii\sigma_2
\label{eq:P-matrix-define}
\end{equation}
and 
\begin{multline}
\mathbf{Q}:=\frac{1}{4\ii\mathrm{Im}\{\alpha_\mathrm{gc}\}}\left(\frac{\sqrt{-\alpha_\mathrm{gc}}}{\sqrt{-\alpha_\mathrm{gc}^*}}+1\right)(C_{12}(\alpha_\mathrm{gc})C_{22}(\alpha_\mathrm{gc})^*-C_{12}(\alpha_\mathrm{gc})^*C_{22}(\alpha_\mathrm{gc}))\mathbb{I} \\{}+ 
\frac{1}{4\ii\mathrm{Im}\{\alpha_\mathrm{gc}\}}\left(\frac{\sqrt{-\alpha_\mathrm{gc}}}{\sqrt{-\alpha_\mathrm{gc}^*}}-1\right)(|C_{12}(\alpha_\mathrm{gc})|^2 + |C_{22}(\alpha_\mathrm{gc})|^2)\ii\sigma_2.
\label{eq:Q-matrix-define}
\end{multline}
Observe that \eqref{eq:u-ustar} combines with its complex conjugate to form a $4\times 4$ square linear system to determine $u_1$, $u_2$, $u_1^*$ and $u_2^*$.  This system is necessarily uniquely solvable because it is equivalent to the solution of  Riemann--Hilbert Problem~\ref{rhp:E-parametrix} for which Lemma~\ref{lemma:dotE-exists} asserts unique solvability.   

To further analyze the matrix $\dot{\mathbf{E}}(0)$ appearing in the definitions \eqref{eq:ddotCddotS} of $\ddot{C}$ and $\ddot{S}$, we first use \eqref{eq:exterior-ansatz} which is valid for $w=0$, and combine with \eqref{eq:A's} to get
\begin{equation}
\dot{\mathbf{E}}(0)=\begin{bmatrix}1+2\mathrm{Re}\{\alpha_\mathrm{gc}^{-1}A_0\} & 2\mathrm{Re}\{\alpha_\mathrm{gc}^{-1}A_2\}\\
-2\mathrm{Re}\{\alpha_\mathrm{gc}^{-1}A_2\} & 1+2\mathrm{Re}\{\alpha_\mathrm{gc}^{-1}A_0\}\end{bmatrix}.
\label{eq:dotEzero-As}
\end{equation}
Starting from \eqref{eq:cos-sin-linearized} with $\dot{C}:=\dot{O}^\mathrm{out,l}_{11}(0)$ and $\dot{S}:=\dot{O}^\mathrm{out,l}_{21}(0)$, and using \eqref{eq:dotEzero-As}, the defect potentials $\ddot{C}$ and $\ddot{S}$ can be written as
\begin{equation}
\begin{split}
\ddot{C}&=\dot{C} + \mathrm{Re}\{2\ee^{-\ii\theta}A_0\}\dot{C} + \mathrm{Re}\{2\ee^{-\ii\theta}A_2\}\dot{S}\\
\ddot{S}&=\dot{S} -\mathrm{Re}\{2\ee^{-\ii\theta}A_2\}\dot{C} + \mathrm{Re}\{2\ee^{-\ii\theta}A_0\}\dot{S}.
\end{split}
\label{eq:ddotCS-simple}
\end{equation}
From \eqref{eq:As-us} we can express the vector $\mathbf{v}:=\left[2A_0;\;2A_2\right]^\top$ in terms of $\mathbf{u}$ by
\begin{equation}
\mathbf{v}=-\sigma_3\mathbf{B}\mathbf{u},\quad\mathbf{B}:=\begin{bmatrix}C_{22}(\alpha_\mathrm{gc}) & -C_{12}(\alpha_\mathrm{gc})\\
C_{12}(\alpha_\mathrm{gc}) & C_{22}(\alpha_\mathrm{gc})\end{bmatrix}.
\end{equation}
According to Lemma~\ref{lemma:C-real}, the matrix $\mathbf{B}$ is invertible provided that  $\epsilon^{-1}\Phi^l(t)=\Omega_{\mathrm{p},N}-\omega_\mathrm{gc}T$ is not a half-integer multiple of $\pi$; regarding these isolated values of $T$ as eventually removable singularities, we can invert $\mathbf{B}$ explicitly:
\begin{equation}
\mathbf{B}^{-1}=\frac{1}{C_{12}(\alpha_\mathrm{gc})^2+C_{22}(\alpha_\mathrm{gc})^2}\begin{bmatrix}C_{22}(\alpha_\mathrm{gc}) &C_{12}(\alpha_\mathrm{gc})\\
-C_{12}(\alpha_\mathrm{gc}) & C_{22}(\alpha_\mathrm{gc})\end{bmatrix}\quad\implies\quad
\mathbf{u}=-\mathbf{B}^{-1}\sigma_3\mathbf{v}.
\end{equation}
Note that as $\mathbf{P}$, $\mathbf{Q}$, and $\mathbf{B}$ are all scalar linear combinations of $\mathbb{I}$ and $\ii\sigma_2$, they and their conjugates all mutually commute.
The right-hand side of \eqref{eq:u-ustar} can be written in the form
\begin{equation}
-\begin{bmatrix}C_{12}(\alpha_\mathrm{gc})\\C_{22}(\alpha_\mathrm{gc})\end{bmatrix}=\sigma_3\mathbf{B}\begin{bmatrix}0\\1\end{bmatrix}.
\end{equation}
Hence, eliminating $\mathbf{u}$ in favor of $\mathbf{v}$ and including the complex conjugate as a second equation gives the $4\times 4$ system on $\mathbf{v}$ and $\mathbf{v}^*$
\begin{equation}
\begin{split}
-\mathbf{PB}^{-1}\sigma_3\mathbf{v} - \mathbf{QB}^{-*}\sigma_3\mathbf{v}^* &=\sigma_3\mathbf{B}\begin{bmatrix}0\\1\end{bmatrix}\\
-\mathbf{Q}^*\mathbf{B}^{-1}\sigma_3\mathbf{v}-\mathbf{P}^*\mathbf{B}^{-*}\sigma_3\mathbf{v}^*&=\sigma_3\mathbf{B}^*\begin{bmatrix}0\\1\end{bmatrix},
\end{split}
\end{equation}
from which $\mathbf{v}^*$ is easily eliminated giving
a $2\times 2$ system on $\mathbf{v}$ alone:
\begin{equation}
(\mathbf{QQ}^*-\mathbf{P}^*\mathbf{P})\mathbf{B}^{-1}\sigma_3\mathbf{v}=(\mathbf{P}^*\sigma_3\mathbf{B}-\mathbf{Q}\sigma_3\mathbf{B}^*)\begin{bmatrix}0\\1\end{bmatrix}.
\end{equation}
Multiplying on the left by $\mathbf{B}$ which commutes with everything except $\sigma_3$ gives
\begin{equation}
(\mathbf{QQ}^*-\mathbf{P}^*\mathbf{P})\sigma_3\mathbf{v}=(\mathbf{P}^*\mathbf{B}\sigma_3\mathbf{B}-\mathbf{Q}\mathbf{B}\sigma_3\mathbf{B}^*)\begin{bmatrix}0\\1\end{bmatrix}.
\end{equation}
Now we pass to the variables appearing in the formul\ae\ \eqref{eq:ddotCS-simple} by noting that $\mathbf{Q}\mathbf{Q}^*-\mathbf{P}^*\mathbf{P}$ is a real matrix, multiplying by $\ee^{-\ii\theta}\sigma_3$ and taking the real part:
\begin{equation}
\sigma_3(\mathbf{QQ}^*-\mathbf{P}^*\mathbf{P})\sigma_3\mathbf{w}=\mathrm{Re}\left\{\ee^{-\ii\theta}\sigma_3\mathbf{P}^*\mathbf{B}\sigma_3\mathbf{B}-\ee^{-\ii\theta}\sigma_3\mathbf{QB}\sigma_3\mathbf{B}^*\right\}\begin{bmatrix}0\\1\end{bmatrix},
\label{eq:w-system}
\end{equation}
where
\begin{equation}
\mathbf{w}=\mathrm{Re}\{\ee^{-\ii\theta}\mathbf{v}\}=\begin{bmatrix}\mathrm{Re}\{2\ee^{-\ii\theta}A_0\}\\\mathrm{Re}\{2\ee^{-\ii\theta}A_2\}\end{bmatrix}.
\end{equation}
\begin{lemma}
We have the following explicit representations, in which $\rho(\cdot)$ is defined by \eqref{eq:rho-func-define} and we use the abbreviated notation \eqref{eq:abbreviated-notation-1}--\eqref{eq:abbreviated-notation-3} for $m=m_\mathrm{gc}$ and $\Omega=\Omega_{\mathrm{p},N}$:
\begin{equation}
|W_\mathrm{gc}'(\alpha_\mathrm{gc})|^{-1}\sigma_3\mathbf{QQ}^*\sigma_3=\frac{1-2\mathrm{dn}^2}{4\sqrt{m_\mathrm{gc}(1-m_\mathrm{gc})}}\cdot\mathbb{I}+\frac{\mathrm{sn}\,\mathrm{dn}}{2\sqrt{1-m_\mathrm{gc}}}\cdot\ii\sigma_2,
\label{eq:system-LHS-Q}
\end{equation}
and
\begin{multline}
|W_\mathrm{gc}'(\alpha_\mathrm{gc})|^{-1}\sigma_3\mathbf{P}^*\mathbf{P}\sigma_3 = 
\frac{1}{4}\left[\left(\frac{(1-m_\mathrm{gc})^{3/4}}{m_\mathrm{gc}^{1/4}}\mathrm{p}-(m_\mathrm{gc}(1-m_\mathrm{gc}))^{1/4}\rho(m_\mathrm{gc})T\right)^2\right.\\
\left.+\sqrt{m_\mathrm{gc}(1-m_\mathrm{gc})}X^2 -
\sqrt{\frac{m_\mathrm{gc}}{1-m_\mathrm{gc}}}\mathrm{cn}^2\right]\cdot \mathbb{I}\\
-\frac{1}{2}\left[\sqrt{1-m_\mathrm{gc}}\mathrm{p}-\sqrt{m_\mathrm{gc}}\rho(m_\mathrm{gc})T\right]\mathrm{cn}\cdot\ii\sigma_2,
\label{eq:system-LHS-P}
\end{multline}
and
\begin{multline}
|W'_\mathrm{gc}(\alpha_\mathrm{gc})|^{-1}\mathrm{Re}\left\{\ee^{-\ii\theta}\sigma_3\mathbf{P}^*\mathbf{B}\sigma_3\mathbf{B}-\ee^{-\ii\theta}\sigma_3\mathbf{QB}\sigma_3\mathbf{B}^*\right\} = \\
\left(\left(\sqrt{1-m_\mathrm{gc}}\mathrm{p}-\sqrt{m_\mathrm{gc}}\rho(m_\mathrm{gc})T\right)\mathrm{cn}+\frac{\mathrm{sn}\,\mathrm{dn}}{\sqrt{1-m_\mathrm{gc}}}\right)\cdot\mathbb{I}.
\label{eq:system-RHS}
\end{multline}
\label{lemma:system-simplify}
\end{lemma}
The proof requires many identities involving functions on elliptic curves and will be given in Appendix~\ref{sec:system-simplify}.
Solving the system \eqref{eq:w-system} by Cramer's rule, one observes that
\begin{equation}
\frac{\det(\sigma_3(\mathbf{QQ}^*-\mathbf{P}^*\mathbf{P})\sigma_3)}{|W_\mathrm{gc}'(\alpha_\mathrm{gc})|^2}=
q(X,T;m_\mathrm{gc},\Omega_{\mathrm{p},N})^2+r(X,T;m_\mathrm{gc},\Omega_{\mathrm{p},N})^2.
\end{equation}
and also that $|W'_\mathrm{gc}(\alpha_\mathrm{gc})|$ cancels out of the solution, which is given explicitly by
\begin{equation}
\mathrm{Re}\{2\ee^{-\ii\theta}A_0\}=w_1=-\frac{2q(X,T;m_\mathrm{gc},\Omega_{\mathrm{p},N})^2}{q(X,T;m_\mathrm{gc},\Omega_{\mathrm{p},N})^2+r(X,T;m_\mathrm{gc},\Omega_{\mathrm{p},N})^2}
\end{equation}
and
\begin{equation}
\mathrm{Re}\{2\ee^{-\ii\theta}A_2\}=w_2=-\frac{2q(X,T;m_\mathrm{gc},\Omega_{\mathrm{p},N})r(X,T;m_\mathrm{gc},\Omega_{\mathrm{p},N})}{q(X,T;m_\mathrm{gc},\Omega_{\mathrm{p},N})^2+r(X,T;m_\mathrm{gc},\Omega_{\mathrm{p},N})^2},
\end{equation}
in which $q(X,T;m,\Omega)$ and $r(X,T;m,\Omega)$ are defined by \eqref{eq:q-and-r-define}.  Using these results in \eqref{eq:ddotCS-simple} then produces the formula \eqref{eq:ddotCS-compact}, in which $m=m_\mathrm{gc}$ and $\Omega=\Omega_{\mathrm{p},N}$.

This completes the proof of Theorem~\ref{thm:NearThePoles}.

\appendix
\section{An Outer Parametrix Built from Elliptic Functions}
\label{app:theta}
Consider the following Riemann--Hilbert problem.
\begin{rhp}
Let a real number $\nu$ be given along with a contour $\beta$ lying in $\mathbb{C}_+$ with initial endpoint $w=\alpha\in\mathbb{C}_+$ and terminal endpoint $w=1$. Seek a $2\times 2$ matrix function $\mathbf{Y}(w)=\mathbf{Y}(w;\beta,\nu)$, $\mathbf{Y}:\mathbb{C}\setminus (\beta\cup\beta^*\cup\mathbb{R}_+)\to\mathbb{C}^{2\times 2}$ with the following properties.
\begin{itemize}
\item[]\textbf{Analyticity:}  $\mathbf{Y}(\cdot)$ is analytic in its domain of definition, taking continuous boundary values from the domain at each point of $\beta\cup\beta^*\cup\mathbb{R}_+$ except for the endpoints of $\beta\cup\beta^*$, namely $\alpha$ and $\alpha^*$, where we require that $\mathbf{Y}(w)=\mathcal{O}(|w-\alpha|^{-1/4})$ and $\mathbf{Y}(w)=\mathcal{O}(|w-\alpha^*|^{-1/4})$ respectively.
\item[]\textbf{Jump Conditions:}  The boundary values taken by $\mathbf{Y}(w)$ from each side along the arcs $\beta$ and $\beta^*$, and along $\mathbb{R}_+$ (oriented from $w=0$ to $w=+\infty$) are related by 
\begin{equation}
\mathbf{Y}_+(w;\beta,\nu)=\mathbf{Y}_-(w;\beta,\nu)(-\ii\sigma_1)\ee^{\ii\nu\sigma_3},\quad w\in\beta\quad\text{and}\quad
\mathbf{Y}_+(w;\beta,\nu)=\mathbf{Y}_-(w;\beta,\nu)(-\ii\sigma_1)\ee^{-\ii\nu\sigma_3},\quad w\in\beta^*,
\end{equation}
and
\begin{equation}
\mathbf{Y}_+(w;\beta,\nu)=\sigma_2\mathbf{Y}_-(w;\beta,\nu)\sigma_2,\quad w\in\mathbb{R}_+.
\end{equation}
\item[]\textbf{Normalization:}  $\mathbf{Y}(w;\beta,\nu)\to\mathbb{I}$ as $w\to\infty$ in all directions, including tangent to $\mathbb{R}_+$.
\end{itemize}
\label{rhp:OuterParametrixGeneral}
\end{rhp}
This problem is solved in \cite[Appendix B.1]{BuckinghamMiller2013}.  The following proposition characterizes the most important properties of the solution.
\begin{proposition}[Properties of $\mathbf{Y}$]
Riemann--Hilbert Problem~\ref{rhp:OuterParametrixGeneral} has a unique solution for every $(\beta,\nu)$ as indicated in the problem statement.  The solution $\mathbf{Y}(w;\beta,\nu)$ has unit determinant, is $2\pi$-periodic in $\nu$, and $(1+|w-\alpha|^{-1/4}+|w-\alpha^*|^{-1/4})^{-1}\mathbf{Y}(w;\beta,\nu)$ is uniformly bounded for $(w,\nu)\in(\mathbb{C}\setminus(\beta\cup\mathbb{R}_+))\times (-\pi,\pi]$.  The bound is also uniform with respect to curves $\beta$ whose free endpoint varies in a small neighborhood of a fixed point $\alpha\in\mathbb{C}_+$.  Finally, the solution has a well-defined value at $w=0$, namely
\begin{equation}
\mathbf{Y}(0;\beta,\nu)=\begin{bmatrix}\displaystyle\mathrm{dn}\left(\frac{2\KK(m)}{\pi}\nu;m\right) & \displaystyle \sqrt{m}\mathrm{sn}\left(\frac{2\KK(m)}{\pi}\nu;m\right)\\
\displaystyle -\sqrt{m}\mathrm{sn}\left(\frac{2\KK(m)}{\pi}\nu;m\right) & \displaystyle 
\mathrm{dn}\left(\frac{2\KK(m)}{\pi}\nu;m\right)\end{bmatrix},\quad m:=\sin^2(\tfrac{1}{2}\arg(\alpha)).
\label{eq:general-outer-parametrix-at-origin}
\end{equation}
\label{prop:PropertiesOfY}
\end{proposition}

In \eqref{eq:general-outer-parametrix-at-origin}, $\KK(\cdot)$ denotes the complete elliptic integral of the first kind \eqref{eq:KE}, while $\mathrm{dn}(z;m)$ and $\mathrm{sn}(z;m)$ are Jacobi elliptic functions.

We will need also an explicit formula for the solution $\mathbf{Y}(w;\beta,\nu)$ valid for $w$ near $\alpha$, and we take the opportunity to simplify the formula in transcribing it from \cite[Appendix B.1]{BuckinghamMiller2013}.  Let $q(w)$ denote the function uniquely determined by the properties that 
\begin{itemize}
\item $q(\cdot)$ is analytic for $w\in\mathbb{C}\setminus(\beta\cup\beta^*)$,
\item $q(w)^4 = (w-\alpha)/(w-\alpha^*)$,
\item $q(\infty)=1$.
\end{itemize}
Likewise, let $S(w)$ denote the function uniquely determined by the properties that
\begin{itemize}
\item $S(\cdot)$ is analytic for $w\in\mathbb{C}\setminus(\beta\cup\beta^*\cup\mathbb{R}_-)$,
\item $S(w)^2 = w(w-\alpha)(w-\alpha^*)$,
\item $S(w) = w^{3/2}(1+o(1))$ as $w\to\infty$ (principal branch of $w^{3/2}$ intended).
\end{itemize}
Letting $a$ denote a positively-oriented Jordan curve with $\beta\cup\beta^*$ in its interior and $\mathbb{R}_-$ in its exterior, define a constant $c$ by
\begin{equation}
c:=2\pi\ii\left[\oint_a\frac{\dd w}{S(w)}\right]^{-1}
\label{eq:Abel-c-define}
\end{equation}
and a constant $\mathcal{H}$ by
\begin{equation}
\mathcal{H}:=2c\int_0^{\alpha^*}\frac{\dd w}{S(w)},
\label{eq:general-H-define}
\end{equation}
where the path of integration is arbitrary in the simply-connected domain $\mathbb{C}_-\setminus\beta^*$.  It can be shown that $\mathrm{Re}\{\mathcal{H}\}<0$ holds.  Given such a constant $\mathcal{H}$, the Riemann theta function $\Theta(z;\mathcal{H})$ is defined by the convergent series
\begin{equation}
\Theta(z;\mathcal{H}):=\sum_{n=-\infty}^{\infty}\ee^{\tfrac{1}{2}\mathcal{H}n^2}\ee^{nz},\quad \mathrm{Re}\{\mathcal{H}\}<0,\quad z\in\mathbb{C}.
\label{eq:Riemann-Theta}
\end{equation}
For $w\in\mathbb{C}_+\setminus\beta$ we define a branch of the Abel map by the formula
\begin{equation}
A(w):=\int_\alpha^w\frac{c\,\dd\xi}{S(\xi)}
\label{eq:Abel-branch}
\end{equation}
for a path of integration from $\alpha$ to $w$ that is arbitrary in the simply-connected domain $\mathbb{C}_+\setminus\beta$.  Defining the \emph{Riemann constant} $\mathcal{K}$ by
\begin{equation}
\mathcal{K}:=\ii\pi +\frac{1}{2}\mathcal{H}
\label{eq:Riemann-K-define}
\end{equation}
and shifted phases $\varphi^\pm$ by
\begin{equation}
\varphi^\pm:=\nu\pm\frac{1}{2}\pi,
\label{eq:shifted-phases-define}
\end{equation}
we have the following explicit formula for $\mathbf{Y}(w)$ when $\mathrm{Im}\{w\}>0$:
\begin{equation}
\mathbf{Y}(w)=\frac{q(w)}{2}\cdot\frac{\Theta(\ii\pi;\mathcal{H})}{\Theta(A(w)+\mathcal{K};\mathcal{H})}
\mathbf{Z}(w),\quad \mathrm{Im}\{w\}>0,
\label{eq:OuterParametrixGeneral-Solution-1}
\end{equation}
where $\mathbf{Z}(w)$ is the $2\times 2$ matrix function with elements
\begin{equation}
\begin{split}
Z_{11}(w)&:=
\frac{\Theta(A(w)+\mathcal{K}-\ii\varphi^-;\mathcal{H})}{\Theta(\ii\pi-\ii\varphi^-;\mathcal{H})} +
\frac{\Theta(A(w)+\mathcal{K}-\ii\varphi^+;\mathcal{H})}{\Theta(\ii\pi-\ii\varphi^+;\mathcal{H})} \\
Z_{12}(w)&:=
\ii\left(\frac{\Theta(A(w)+\mathcal{K}+\ii\varphi^+;\mathcal{H})}{\Theta(\ii\pi+\ii\varphi^+;\mathcal{H})}-
\frac{\Theta(A(w)+\mathcal{K}+\ii\varphi^-;\mathcal{H})}{\Theta(\ii\pi+\ii\varphi^-;\mathcal{H})}\right)\\
Z_{21}(w)&:=
\ii\left(\frac{\Theta(A(w)+\mathcal{K}-\ii\varphi^-;\mathcal{H})}{\Theta(\ii\pi-\ii\varphi^-;\mathcal{H})}-
\frac{\Theta(A(w)+\mathcal{K}-\ii\varphi^+;\mathcal{H})}{\Theta(\ii\pi-\ii\varphi^+;\mathcal{H})}\right)\\
Z_{22}(w)&:=
\frac{\Theta(A(w)+\mathcal{K}+\ii\varphi^-;\mathcal{H})}{\Theta(\ii\pi+\ii\varphi^-;\mathcal{H})} +
\frac{\Theta(A(w)+\mathcal{K}+\ii\varphi^+;\mathcal{H})}{\Theta(\ii\pi+\ii\varphi^+;\mathcal{H})}.
\end{split}
\label{eq:OuterParametrixGeneral-Solution-2}
\end{equation}

In general, the significance of $\mathcal{K}$ given by \eqref{eq:Riemann-K-define} is that it is the only zero, modulo integer multiples of $2\pi\ii$ and $\mathcal{H}$, of the entire function $z\mapsto\Theta(z;\mathcal{H})$, and it is a simple zero.  In addition to 
\begin{equation}
\Theta(\mathcal{K};\mathcal{H})=0\quad\text{and}\quad\Theta'(\mathcal{K};\mathcal{H})\neq 0,
\label{eq:Theta-zero}
\end{equation}
the series given by \eqref{eq:Riemann-Theta} has the following well-known automorphic properties
\begin{equation}
\Theta(-z;\mathcal{H})=\Theta(z;\mathcal{H}),\quad\Theta(z+2\pi\ii;\mathcal{H})=\Theta(z;\mathcal{H}),\quad\text{and}\quad
\Theta(z+\mathcal{H};\mathcal{H})=\ee^{-\tfrac{1}{2}\mathcal{H}}\ee^{-z}\Theta(z;\mathcal{H}).
\label{eq:Theta-identities}
\end{equation}
Some additional identities for theta functions will be useful shortly, so we remark that all of the identities we will need can be found in the Digital Library of Mathematical Functions \cite{NIST:DLMF}, which however uses a different classical notation.  The key relation one needs to translate is $\Theta(z;\mathcal{H})=\theta_3(w|\tau)=\theta_3(w,q)$ where $z=2\ii w$, $\mathcal{H}=2\pi\ii\tau$, and $q=\ee^{\ii\pi\tau}$.  

For instance, using ``Jacobi's identity'' and the definitions of the other three theta functions $\theta_j$, $j=1,2,4$, one can extract from \cite[Chap.\@ 20]{NIST:DLMF} that
\begin{equation}
\Theta'(\mathcal{K};\mathcal{H})=\frac{1}{2}\Theta(\tfrac{1}{2}\mathcal{H};\mathcal{H})\Theta(0;\mathcal{H})\Theta(\ii\pi;\mathcal{H}).
\label{eq:Theta-prime-K}
\end{equation}
Also, using the addition formula \cite[Eqn.\@ 20.7.9]{NIST:DLMF} with $w=z$ gives the double-argument identity
\begin{equation}
\Theta(z;\mathcal{H})^4-\ee^{\mathcal{H}/2}\ee^{2z}\Theta(z+\tfrac{1}{2}\mathcal{H};\mathcal{H})^4 = \Theta(\ii\pi;\mathcal{H})^3\Theta(2z;\mathcal{H}).
\label{eq:theta-identity-0}
\end{equation}
We will also need some general identities that allow us to transform the parameter $\mathcal{H}$.
First, by combining the identities \cite[Eqn.\@ 20.7.28]{NIST:DLMF} and \cite[Eqn.\@ 20.2.14]{NIST:DLMF} we obtain
\begin{equation}
\Theta(z;\mathcal{H}+2\pi\ii)=\Theta(z+\ii\pi;\mathcal{H}),\quad\mathrm{Re}\{\mathcal{H}\}<0,\quad z\in\mathbb{C}.
\label{eq:theta-identity-1}
\end{equation}
Second, combining ``Watson's identity'' in the form \cite[Eqn.\@ 20.7.14]{NIST:DLMF} with \cite[Eqns.\@ 20.2.10 and 20.2.12]{NIST:DLMF} we obtain the identity
\begin{equation}
\Theta(z_1;\mathcal{H})\Theta(z_2;\mathcal{H})=\Theta(z_1+z_2;2\mathcal{H})\Theta(z_1-z_2;2\mathcal{H})+\ee^{\mathcal{H}/2}\ee^{z_1}\Theta(z_1+z_2+\mathcal{H};2\mathcal{H})\Theta(z_1-z_2+\mathcal{H};2\mathcal{H}),
\label{eq:theta-identity-2}
\end{equation}
which is a kind of addition formula for $\Theta$ that also rescales the parameter $\mathcal{H}$.

\section{Proof of Lemma~\ref{lemma:C-real}}
\label{sec:proof-lemma:C-real}
To evaluate $C_{j2}(\alpha_\mathrm{gc})$, we use the fact that analyticity of $\mathbf{C}(w)$ within $U$ means we may choose to approach the point $w=\alpha_\mathrm{gc}$ from the sectors $\mathrm{III}\cup\mathrm{IV}$ in which according to \eqref{eq:Odot-Otilde} $\widetilde{\mathbf{O}}^\mathrm{out,l}(w)$ coincides with $\dot{\mathbf{O}}^\mathrm{out,l}(w)\ee^{-\ii\Phi^l(t)\sigma_3/(2\epsilon)}$, in which $\dot{\mathbf{O}}^\mathrm{out,l}(w):=\mathbf{Y}(w;\beta_\mathrm{gc},\epsilon^{-1}\Phi^l(t))$ where $\mathbf{Y}(w;\beta,\nu)$ denotes the solution of Riemann--Hilbert Problem~\ref{rhp:OuterParametrixGeneral} and $\beta_\mathrm{gc}$ has endpoint $\alpha_\mathrm{gc}$.  Assuming throughout the calculation below that $w\in(\mathrm{III}\cup\mathrm{IV})\cap U$, we therefore have from \eqref{eq:CmatrixDefine} that
\begin{equation}
\mathbf{C}(w)=\dot{\mathbf{O}}^\mathrm{out,l}(w)\ee^{-\ii\Phi^l(t)\sigma_3/(2\epsilon)}\mathbf{M}W_\mathrm{gc}(w)^{-\sigma_3/4}.
\end{equation}
Referring to \eqref{eq:OuterParametrixGeneral-Solution-1}--\eqref{eq:OuterParametrixGeneral-Solution-2}
we obtain
\begin{multline}
C_{12}(w)=\frac{q(w)W_\mathrm{gc}(w)^{1/4}}{\Theta(A+\mathcal{K})}\frac{\Theta(\ii\pi)}{2\sqrt{2}}
\left[-\ii\ee^{\ii\Phi^l(t)/(2\epsilon)}\frac{\Theta(A+\mathcal{K}+\ii\varphi^-)}{\Theta(\ii\pi +\ii\varphi^-)}+\ii\ee^{\ii\Phi^l(t)/(2\epsilon)}\frac{\Theta(A+\mathcal{K}+\ii\varphi^+)}{\Theta(\ii\pi+\ii\varphi^+)}\right.\\
\left.{}-\ee^{-\ii\Phi^l(t)/(2\epsilon)}\frac{\Theta(A+\mathcal{K}-\ii\varphi^-)}{\Theta(\ii\pi-\ii\varphi^-)}-\ee^{-\ii\Phi^l(t)/(2\epsilon)}\frac{\Theta(A+\mathcal{K}-\ii\varphi^+)}{\Theta(\ii\pi-\ii\varphi^+)}\right]
\label{eq:C12}
\end{multline}
and
\begin{multline}
C_{22}(w)=\frac{q(w)W_\mathrm{gc}(w)^{1/4}}{\Theta(A+\mathcal{K})}\frac{\Theta(\ii\pi)}{2\sqrt{2}}
\left[\ee^{\ii\Phi^l(t)/(2\epsilon)}\frac{\Theta(A+\mathcal{K}+\ii\varphi^-)}{\Theta(\ii\pi +\ii\varphi^-)}+\ee^{\ii\Phi^l(t)/(2\epsilon)}\frac{\Theta(A+\mathcal{K}+\ii\varphi^+)}{\Theta(\ii\pi+\ii\varphi^+)}\right.\\
\left.{}-\ii\ee^{-\ii\Phi^l(t)/(2\epsilon)}\frac{\Theta(A+\mathcal{K}-\ii\varphi^-)}{\Theta(\ii\pi-\ii\varphi^-)}+\ii\ee^{-\ii\Phi^l(t)/(2\epsilon)}\frac{\Theta(A+\mathcal{K}-\ii\varphi^+)}{\Theta(\ii\pi-\ii\varphi^+)}\right],
\label{eq:C22}
\end{multline}
in which we are using the simplified notation 
\begin{equation}
\Theta(z):=\Theta(z;\mathcal{H}^\mathrm{gc})
\label{eq:Theta-simple}
\end{equation}
where the right-hand side is defined for general $\mathcal{H}$ in the left half-plane by 
\eqref{eq:Riemann-Theta}, and $A$ is shorthand for the branch $A(w)$ of the Abel map defined by \eqref{eq:Abel-branch}.  The Riemann constant $\mathcal{K}$ and the shifted phases $\varphi^\pm$ are given in terms of $\mathcal{H}=\mathcal{H}^\mathrm{gc}$ and $\nu=\epsilon^{-1}\Phi^l(t)$ respectively by \eqref{eq:Riemann-K-define} and \eqref{eq:shifted-phases-define}.  $\mathcal{H}^\mathrm{gc}$ is in turn defined from the contour $\beta=\beta_\mathrm{gc}$ via \eqref{eq:Abel-c-define}--\eqref{eq:general-H-define}.
Recall from Appendix~\ref{app:theta} that $q(w)$ is a function that, since $\beta=\beta_\mathrm{gc}$,
vanishes like $(w-\alpha_\mathrm{gc})^{1/4}$ at $w=\alpha_\mathrm{gc}$, and likewise as a function of $w$, $A$ vanishes like $(w-\alpha_\mathrm{gc})^{1/2}$ at $w=\alpha_\mathrm{gc}$.

It follows from these expressions that the sum of squares that we need to calculate is 
\begin{multline}
C_{12}(w)^2+C_{22}(w)^2=\frac{q(w)^2W_\mathrm{gc}(w)^{1/2}}{\Theta(A+\mathcal{K})^2}\frac{\Theta(\ii\pi)^2}{2}
\left[\ee^{\ii\Phi^l(t)/\epsilon}\frac{\Theta(A+\mathcal{K}+\ii\varphi^-)\Theta(A+\mathcal{K}+\ii\varphi^+)}{\Theta(\ii\pi+\ii\varphi^-)\Theta(\ii\pi+\ii\varphi^+)}\right.\\
{}+\ee^{-\ii\Phi^l(t)/\epsilon}\frac{\Theta(A+\mathcal{K}-\ii\varphi^-)\Theta(A+\mathcal{K}-\ii\varphi^+)}{\Theta(\ii\pi-\ii\varphi^-)\Theta(\ii\pi-\ii\varphi^+)}\\
\left.{}+\ii\frac{\Theta(A+\mathcal{K}+\ii\varphi^-)\Theta(A+\mathcal{K}-\ii\varphi^+)}{\Theta(\ii\pi+\ii\varphi^-)\Theta(\ii\pi-\ii\varphi^+)}-\ii
\frac{\Theta(A+\mathcal{K}-\ii\varphi^-)\Theta(A+\mathcal{K}+\ii\varphi^+)}{\Theta(\ii\pi-\ii\varphi^-)\Theta(\ii\pi+\ii\varphi^+)}\right].
\end{multline}
Now we consider letting $w\to\alpha_\mathrm{gc}$ from sectors $\mathrm{III}\cup\mathrm{IV}$ within $U$.  As mentioned above, $A\to 0$ in this limit.  The four terms in square brackets above all therefore have well-defined limiting values obtained simply by replacing $A$ by $0$ in the arguments of the theta functions in the numerators.  
Using the identities \eqref{eq:Theta-identities} one checks furthermore that all four terms in brackets above have \emph{exactly the same limiting value}.  The numerator and denominator of the initial fraction on the right-hand side both vanish in the limit, so this is the only delicate part of the computation.  
Thus,
\begin{multline}
C_{12}(\alpha_\mathrm{gc})^2+C_{22}(\alpha_\mathrm{gc})^2=\\
2\Theta(\ii\pi)^2 \ee^{\ii\Phi^l(t)/\epsilon}\frac{\Theta(\mathcal{K}+\ii\epsilon^{-1}\Phi^l(t)-\tfrac{1}{2}\ii\pi)\Theta(\mathcal{K}+\ii\epsilon^{-1}\Phi^l(t)+\tfrac{1}{2}\ii\pi)}{\Theta(\ii\pi+\ii\epsilon^{-1}\Phi^l(t)-\tfrac{1}{2}\ii\pi)\Theta(\ii\pi+\ii\epsilon^{-1}\Phi^l(t)+\tfrac{1}{2}\ii\pi)}\lim_{w\to\alpha_\mathrm{gc}}\left[\frac{q(w)^2W_\mathrm{gc}(w)^{1/2}}{\Theta(A+\mathcal{K})^2}\right],
\label{eq:sum-of-squares-1}
\end{multline}
where the limit is taken from sector $\mathrm{III}$ or $\mathrm{IV}$ of $U$.    
Now, using the defnition \eqref{eq:Abel-branch} of $A=A(w)$ it is easy to show that
\begin{equation}
A^2=\frac{4c_\mathrm{gc}^2}{\alpha_\mathrm{gc}(\alpha_\mathrm{gc}-\alpha_\mathrm{gc}^*)}(w-\alpha_\mathrm{gc})+\mathcal{O}((w-\alpha_\mathrm{gc})^2),\quad w\to\alpha_\mathrm{gc},
\label{eq:A1-squared}
\end{equation}
in which the constant $c_\mathrm{gc}$ is defined from $\beta=\beta_\mathrm{gc}$ via \eqref{eq:Abel-c-define}.
This calculation is useful because $\Theta(\mathcal{K})=0$, so by Taylor expansion,
\begin{equation}
\Theta(A+\mathcal{K})^2=\Theta'(\mathcal{K})^2A^2 + o(A^2) = \frac{4c_\mathrm{gc}^2\Theta'(\mathcal{K})^2}{\alpha_\mathrm{gc}(\alpha_\mathrm{gc}-\alpha_\mathrm{gc}^*)}(w-\alpha_\mathrm{gc})(1+o(1)),\quad w\to\alpha_\mathrm{gc}.
\label{eq:Theta-expand}
\end{equation}
In view of this calculation, we may rewrite \eqref{eq:sum-of-squares-1} as
\begin{equation}
C_{12}(\alpha_\mathrm{gc})^2+C_{22}(\alpha_\mathrm{gc})^2=
\frac{\alpha_\mathrm{gc}(\alpha_\mathrm{gc}-\alpha_\mathrm{gc}^*)\Theta(\ii\pi)^2L}{2c_\mathrm{gc}^2\Theta'(\mathcal{K})^2}\ee^{\ii\Phi^l(t)/\epsilon}\frac{\Theta(\mathcal{K}+\ii\epsilon^{-1}\Phi^l(t)-\tfrac{1}{2}\ii\pi)\Theta(\mathcal{K}+\ii\epsilon^{-1}\Phi^l(t)+\tfrac{1}{2}\ii\pi)}{\Theta(\ii\pi+\ii\epsilon^{-1}\Phi^l(t)-\tfrac{1}{2}\ii\pi)\Theta(\ii\pi+\ii\epsilon^{-1}\Phi^l(t)+\tfrac{1}{2}\ii\pi)}.
\label{eq:sum-of-squares-2}
\end{equation}
where $L$ is a constant defined by the limit taken from $(\mathrm{III}\cup\mathrm{IV})\cap U$
\begin{equation}
L:=\lim_{w\to\alpha_\mathrm{gc}}\left[\frac{q(w)^2W_\mathrm{gc}(w)^{1/2}}{w-\alpha_\mathrm{gc}}\right].
\label{eq:L-define}
\end{equation}
The hypothesis that the gradient catastrophe point is simple ensures $W_\mathrm{gc}(w)$ is conformal at $w=\alpha_\mathrm{gc}$ implying that $q(w)^2W_\mathrm{gc}(w)^{1/2}$ vanishes there precisely to first order, and hence $L$ is a nonzero constant with modulus
\begin{equation}
|L|=\sqrt{\frac{|W'_\mathrm{gc}(\alpha_\mathrm{gc})|}{|\alpha_\mathrm{gc}-\alpha_\mathrm{gc}^*|}}=\frac{|W_\mathrm{gc}'(\alpha_\mathrm{gc})|^{1/2}}{\sqrt{2\sin(\theta)}}=\frac{|W_\mathrm{gc}'(\alpha_\mathrm{gc})|^{1/2}}{2\sqrt{\sin(\tfrac{1}{2}\theta)\cos(\tfrac{1}{2}\theta)}}=\frac{|W_\mathrm{gc}'(\alpha_\mathrm{gc})|^{1/2}}{2(m_\mathrm{gc}(1-m_\mathrm{gc}))^{1/4}}\neq 0.
\label{eq:modulus-of-L}
\end{equation}
The argument of $L$ can be evaluated as follows:
\begin{itemize}
\item Since the limit $w\to\alpha_\mathrm{gc}$ is taken from sectors $\mathrm{III}\cup\mathrm{IV}$, we may represent $q(w)$ in terms of principal branches as 
\begin{equation}
q(w)=\frac{\ee^{\ii\pi/8}(-\ii(w-\alpha_\mathrm{gc}))^{1/4}}{\ee^{\ii\pi/8}(-\ii(w-\alpha_\mathrm{gc}^*))^{1/4}}=\frac{(-\ii(w-\alpha_\mathrm{gc}))^{1/4}}{(-\ii(\alpha_\mathrm{gc}-\alpha_\mathrm{gc}^*))^{1/4}}(1+\mathcal{O}(w-\alpha_\mathrm{gc})),\quad w\in (\mathrm{III}\cup\mathrm{IV})\cap U.
\end{equation}
Note that the denominator is positive real:  $(-\ii(\alpha_\mathrm{gc}-\alpha_\mathrm{gc}^*))^{1/4}=(2\sin(\theta))^{1/4}>0$.  Assuming further that $w$ approaches $\alpha_\mathrm{gc}$ along the boundary of sector $\mathrm{IV}$ where it meets sector $\mathrm{V}$, some simple geometry shows that $\arg(q(w))=\tfrac{1}{4}\theta$.
\item Again assuming that $w$ lies along the boundary of sector $\mathrm{IV}$ where $W_\mathrm{gc}(w)$ is negative real, it follows easily that $\arg(W_\mathrm{gc}(w)^{1/4})=\tfrac{1}{4}\pi$.
\item Similarly, when $w$ approaches $\alpha_\mathrm{gc}$ along the common boundary of sectors $\mathrm{IV}$ and $\mathrm{V}$, $\arg(w-\alpha_\mathrm{gc})\to \theta+\tfrac{1}{2}\pi$.
\end{itemize}
Therefore combining these results we see that $L=|L|\ee^{-\ii\theta/2}$.

Now, using \eqref{eq:Theta-prime-K} for $\mathcal{H}=\mathcal{H}^\mathrm{gc}$ to eliminate $\Theta'(\mathcal{K})$ we get
\begin{equation}
C_{12}(\alpha_\mathrm{gc})^2+C_{22}(\alpha_\mathrm{gc})^2=
\frac{2\alpha_\mathrm{gc}(\alpha_\mathrm{gc}-\alpha_\mathrm{gc}^*)L}{c_\mathrm{gc}^2\Theta(\tfrac{1}{2}\mathcal{H}^\mathrm{gc})^2\Theta(0)^2}\ee^{\ii\Phi^l(t)/\epsilon}\frac{\Theta(\mathcal{K}+\ii\epsilon^{-1}\Phi^l(t)-\tfrac{1}{2}\ii\pi)\Theta(\mathcal{K}+\ii\epsilon^{-1}\Phi^l(t)+\tfrac{1}{2}\ii\pi)}{\Theta(\ii\pi+\ii\epsilon^{-1}\Phi^l(t)-\tfrac{1}{2}\ii\pi)\Theta(\ii\pi+\ii\epsilon^{-1}\Phi^l(t)+\tfrac{1}{2}\ii\pi)}.
\label{eq:sum-of-squares-3}
\end{equation}
Next, we note that (as in \cite[Pgs.\@ 123--124]{BuckinghamMiller2013}) while the theta parameter $\mathcal{H}^\mathrm{gc}$ is complex, it is related to a real negative parameter $\mathcal{H}_0^\mathrm{gc}$ in a simple way:
\begin{equation}
\mathcal{H}^\mathrm{gc}=\frac{1}{2}(\mathcal{H}^\mathrm{gc}_0+2\pi\ii)
\label{eq:H-decomposition}
\end{equation}
where $\mathcal{H}^\mathrm{gc}_0$ is given explicitly in terms of the complete elliptic integral of the first kind $\KK(\cdot)$ by
\begin{equation}
\mathcal{H}^\mathrm{gc}_0=-2\pi\frac{\KK(1-m_\mathrm{gc})}{\KK(m_\mathrm{gc})}<0
\label{eq:H0}
\end{equation}
and $m_\mathrm{gc}$ is defined in terms of $\alpha_\mathrm{gc}=\ee^{\ii\theta}$ by \eqref{eq:mgc-alphagc}.  To express the theta functions in terms of Jacobi elliptic functions, we need to first rewrite them in terms of theta functions with parameter $\mathcal{H}^\mathrm{gc}_0$ in place of $\mathcal{H}^\mathrm{gc}$.  
Having already suppressed the parameter $\mathcal{H}^\mathrm{gc}$, we now expand on our abbreviated notation with the definition
\begin{equation}
\Theta_0(z):=\Theta(z;\mathcal{H}_0^\mathrm{gc}).
\label{eq:Theta-0}
\end{equation}
Combining \eqref{eq:theta-identity-1} and \eqref{eq:theta-identity-2} and using \eqref{eq:H-decomposition} gives
\begin{multline}
\Theta(z_1)\Theta(z_2)=\Theta_0(z_1+z_2+\ii\pi)\Theta_0(z_1-z_2+\ii\pi)\\{}+\ii\ee^{\mathcal{H}^\mathrm{gc}_0/4}\ee^{z_1}\Theta_0(z_1+z_2+\tfrac{1}{2}\mathcal{H}^\mathrm{gc}_0)
\Theta_0(z_1-z_2+\tfrac{1}{2}\mathcal{H}^\mathrm{gc}_0)
\label{eq:theta-last-identity}
\end{multline}
where we have also used the periodic property of $\Theta$ written in \eqref{eq:Theta-identities} to simplify the result. We use \eqref{eq:theta-last-identity} to deal with the products of theta functions appearing in \eqref{eq:sum-of-squares-3} as follows:  first taking $z_1=0$ and $z_2=\tfrac{1}{2}\mathcal{H}^\mathrm{gc}=\tfrac{1}{4}\mathcal{H}^\mathrm{gc}_0+\tfrac{1}{2}\ii\pi$ gives
\begin{equation}
\Theta(0)^2\Theta(\tfrac{1}{2}\mathcal{H}^\mathrm{gc})^2=
4\Theta_0(\tfrac{1}{4}\mathcal{H}^\mathrm{gc}_0-\tfrac{1}{2}\ii\pi)^4,
\label{eq:Theta-fourth-identity}
\end{equation}
where we also made liberal use of the automorphic identities \eqref{eq:Theta-identities}.
Next, taking $z_1=\ii\pi+\ii\epsilon^{-1}\Phi^l(t)-\tfrac{1}{2}\ii\pi$ and $z_2=\ii\pi+\ii\epsilon^{-1}\Phi^l(t)+\tfrac{1}{2}\ii\pi$ gives
\begin{equation}
\Theta(\ii\pi+\ii\epsilon^{-1}\Phi^l(t)-\tfrac{1}{2}\ii\pi)\Theta(\ii\pi+\ii\epsilon^{-1}\Phi^l(t)+\tfrac{1}{2}\ii\pi)=\Theta_0(2\ii\epsilon^{-1}\Phi^l(t)+\ii\pi)\Theta_0(0),
\end{equation}
where we also used \eqref{eq:Theta-zero} and the definition \eqref{eq:Riemann-K-define} to eliminate some terms.
Similarly, using $\mathcal{K}=\tfrac{1}{2}\mathcal{H}^\mathrm{gc}+\ii\pi = \tfrac{1}{4}\mathcal{H}^\mathrm{gc}_0+\tfrac{3}{2}\ii\pi$ and taking $z_1=\mathcal{K}+\ii\epsilon^{-1}\Phi^l(t)-\tfrac{1}{2}\ii\pi$ and $z_2=\mathcal{K}+\ii\epsilon^{-1}\Phi^l(t)+\tfrac{1}{2}\ii\pi$ gives
\begin{equation}
\Theta(\mathcal{K}+\ii\epsilon^{-1}\Phi^l(t)-\tfrac{1}{2}\ii\pi)\Theta(\mathcal{K}+\ii\epsilon^{-1}\Phi^l(t)+\tfrac{1}{2}\ii\pi)=\Theta_0(2\ii\epsilon^{-1}\Phi^l(t)+\tfrac{1}{2}\mathcal{H}^\mathrm{gc}_0)\Theta_0(0).
\end{equation}
These results allow us to rewrite \eqref{eq:sum-of-squares-3} as follows:
\begin{equation}
C_{12}(\alpha_\mathrm{gc})^2+C_{22}(\alpha_\mathrm{gc})^2=\frac{\alpha_\mathrm{gc}(\alpha_\mathrm{gc}-\alpha_\mathrm{gc}^*)L}{2c_\mathrm{gc}^2\Theta_0(\tfrac{1}{4}\mathcal{H}^\mathrm{gc}_0-\tfrac{1}{2}\ii\pi)^4}\ee^{\ii\Phi^l(t)/\epsilon}\frac{\Theta_0(2\ii\epsilon^{-1}\Phi^l(t)+\tfrac{1}{2}\mathcal{H}^\mathrm{gc}_0)}{\Theta_0(2\ii\epsilon^{-1}\Phi^l(t)+\ii\pi)}.
\end{equation}
Finally, we bring in Jacobi elliptic functions; combining \cite[Eqn.\@ 22.2.5]{NIST:DLMF} with the results in \cite[\S 20.2(iii)]{NIST:DLMF} gives the identity
\begin{equation}
\mathrm{cn}\left(\frac{\KK(m_\mathrm{gc})}{\pi}z;m_\mathrm{gc}\right)=\ee^{\ii z/2}\frac{\Theta_0(\ii\pi)\Theta_0(\ii z+\tfrac{1}{2}\mathcal{H}^\mathrm{gc}_0)}{\Theta_0(\tfrac{1}{2}\mathcal{H}^\mathrm{gc}_0)\Theta_0(\ii z+\ii\pi)}.
\end{equation}
Therefore, also dividing by $\alpha_\mathrm{gc} W_\mathrm{gc}'(\alpha_\mathrm{gc})$ we obtain
\begin{equation}
\frac{C_{12}(\alpha_\mathrm{gc})^2+C_{22}(\alpha_\mathrm{gc})^2}{\alpha_\mathrm{gc} W_\mathrm{gc}'(\alpha_\mathrm{gc})}=\frac{(\alpha_\mathrm{gc}-\alpha_\mathrm{gc}^*)\Theta_0(\tfrac{1}{2}\mathcal{H}^\mathrm{gc}_0)L}{2c_\mathrm{gc}^2W_\mathrm{gc}'(\alpha_\mathrm{gc})\Theta_0(\ii\pi)\Theta_0(\tfrac{1}{4}\mathcal{H}^\mathrm{gc}_0-\tfrac{1}{2}\ii\pi)^4}\cdot\mathrm{cn}\left(\frac{2\KK(m_\mathrm{gc})}{\pi\epsilon}\Phi^l(t);m_\mathrm{gc}\right).
\label{eq:real-thing}
\end{equation}
Now since $\alpha_\mathrm{gc}=\ee^{\ii\theta}$ with $0<\theta<\pi$, we have $\alpha_\mathrm{gc}-\alpha_\mathrm{gc}^*=2\ii\sin(\theta)=4\ii\sin(\tfrac{1}{2}\theta)\cos(\tfrac{1}{2}\theta)=4\ii\sqrt{m_\mathrm{gc}(1-m_\mathrm{gc})}$.  Also, 
according to \cite[Proposition 4.2 and Pg.\@ 125]{BuckinghamMiller2013}, the normalization constant $c_\mathrm{gc}$ for the Abel map with branch points $(w_0,w_1)=(\alpha_\mathrm{gc},\alpha_\mathrm{gc}^*)$ can be identified with $\partial_t\Phi(0,t_\mathrm{gc})=-\omega_\mathrm{gc}$.  Hence using \eqref{eq:omega0} we find that $c_\mathrm{gc}^2=\pi^2/(4\KK(m_\mathrm{gc})^2)$.  Furthermore, according to \eqref{eq:arg-Wprime-gc} we have $W_\mathrm{gc}'(\alpha_\mathrm{gc})=\ii\ee^{-\ii\theta}|W_\mathrm{gc}'(\alpha_\mathrm{gc})|$, so using \eqref{eq:modulus-of-L} together with $L=|L|\ee^{-\ii\theta/2}$ shows that
\begin{multline}
\frac{C_{12}(\alpha_\mathrm{gc})^2+C_{22}(\alpha_\mathrm{gc})^2}{\alpha_\mathrm{gc} W_\mathrm{gc}'(\alpha_\mathrm{gc})}\\
{}=
\frac{4(m_\mathrm{gc}(1-m_\mathrm{gc}))^{1/4} \KK(m_\mathrm{gc})^2}{\pi^2|W_\mathrm{gc}'(\alpha_\mathrm{gc})|^{1/2}}
\frac{\Theta_0(\tfrac{1}{2}\mathcal{H}_0^\mathrm{gc})}{\Theta_0(\ii\pi)\ee^{-\ii\theta/2}\Theta_0(\tfrac{1}{4}\mathcal{H}_0^\mathrm{gc}-\tfrac{1}{2}\ii\pi)^4}
\cdot\mathrm{cn}\left(\frac{2\KK(m_\mathrm{gc})}{\pi\epsilon}\Phi^l(t);m_\mathrm{gc}\right).
\label{eq:real-thing-II}
\end{multline}
Since the zeros of $\Theta_0(z)$ are precisely the lattice points $(2m+1)\ii\pi + (2n+1)\tfrac{1}{2}\mathcal{H}_0^\mathrm{gc}$ for $(m,n)\in\mathbb{Z}^2$, none of the three theta functions on the right-hand side vanishes.  Moreover, since $\mathcal{H}_0^\mathrm{gc}<0$, the Fourier series \eqref{eq:Riemann-Theta} shows that $\Theta_0(\tfrac{1}{2}H_0^\mathrm{gc})$ is positive 
and that $\Theta_0(\ii\pi)$ is real.  Furthermore, since $\Theta_0(\ii\pi)\neq 0$ for all $\mathcal{H}_0^\mathrm{gc}<0$ has the same sign for all such $\mathcal{H}_0^\mathrm{gc}$, by taking the limit $\mathcal{H}_0^\mathrm{gc}\to -\infty$ one sees that in fact $\Theta_0(\ii\pi)$ is positive as well.  Now we have the following identities (cf., \cite[Eqns.\@ 20.9.1--20.9.2]{NIST:DLMF}):
\begin{equation}
m_\mathrm{gc}=\ee^{\mathcal{H}_0^\mathrm{gc}/2}\frac{\Theta_0(\tfrac{1}{2}\mathcal{H}_0^\mathrm{gc})^4}{\Theta_0(0)^4}\quad\text{and}\quad
\KK(m_\mathrm{gc})=\frac{\pi}{2}\Theta_0(0)^2.
\end{equation}
Combining \cite[Eqn.\@ 20.7.5]{NIST:DLMF} with the first equation one gets that
\begin{equation}
1-m_\mathrm{gc}=\frac{\Theta_0(\ii\pi)^4}{\Theta_0(0)^4}.
\end{equation}
Since it is easily shown that $\mathcal{H}_0^\mathrm{gc}<0$ implies that $\Theta_0(0)>0$, we may take roots and obtain the identities
\begin{equation}
\begin{split}
\Theta_0(0)&=\sqrt{\frac{2}{\pi}}\sqrt{\KK(m_\mathrm{gc})}\\
\Theta_0(\ii\pi)&=(1-m_\mathrm{gc})^{1/4}\sqrt{\frac{2}{\pi}}\sqrt{\KK(m_\mathrm{gc})}\\
\Theta_0(\tfrac{1}{2}\mathcal{H}_0^\mathrm{gc})&=m_\mathrm{gc}^{1/4}\ee^{-\mathcal{H}_0^\mathrm{gc}/8}\sqrt{\frac{2}{\pi}}\sqrt{\KK(m_\mathrm{gc})}.
\end{split}
\label{eq:eliminate-thetas}
\end{equation}
Hence, \eqref{eq:real-thing-II} becomes
\begin{equation}
\frac{C_{12}(\alpha_\mathrm{gc})^2+C_{22}(\alpha_\mathrm{gc})^2}{\alpha_\mathrm{gc}W_\mathrm{gc}'(\alpha_\mathrm{gc})}=\frac{4\sqrt{m_\mathrm{gc}}\ee^{-\mathcal{H}_0^\mathrm{gc}/8}\KK(m_\mathrm{gc})^2}{\pi^2|W_\mathrm{gc}'(\alpha_\mathrm{gc})|^{1/2}\ee^{-\ii\theta/2}\Theta_0(\tfrac{1}{4}\mathcal{H}_0^\mathrm{gc}-\tfrac{1}{2}\ii\pi)^4}\cdot
\mathrm{cn}\left(\frac{2\KK(m_\mathrm{gc})}{\pi\epsilon}\Phi^l(t);m_\mathrm{gc}\right).
\label{eq:real-thing-III}
\end{equation}
Again since $\tfrac{1}{4}\mathcal{H}_0^\mathrm{gc}-\tfrac{1}{2}\ii\pi$ does not belong to the lattice of zeros of $\Theta_0(\cdot)$, the product $\ee^{-\ii\theta/2}\Theta_0(\tfrac{1}{4}\mathcal{H}_0^\mathrm{gc}-\tfrac{1}{2}\ii\pi)^4$ is nonzero.  Moreover, we argue as follows: noting how the expression on the left-hand side of \eqref{eq:real-thing-III} enters into the expressions \eqref{eq:cos-sin-after-argument-linearization}, after taking into account that $h(\tau^*)=h(\tau)^*$ one can let $\epsilon\to 0$ and conclude that, due to the evenness in $x$ of the quantities on the left-hand side (cf., \eqref{eq:condensate-even}), it must be the case that $\ee^{-\ii\theta}\Theta(\tfrac{1}{4}\mathcal{H}_0^\mathrm{gc}-\tfrac{1}{2}\ii\pi)^4$ is real-valued.  Using this fact we can in turn determine both the sign and the modulus of the latter quantity.  Indeed, to this end we take $z=\tfrac{1}{4}\mathcal{H}_0^\mathrm{gc}-\tfrac{1}{2}\ii\pi$ and $\mathcal{H}=\mathcal{H}_0^\mathrm{gc}$ in \eqref{eq:theta-identity-0} after which we use \eqref{eq:Theta-identities} to shift an argument by $\mathcal{H}_0^\mathrm{gc}$, and obtain the identity
\begin{equation}
\begin{split}
\Theta_0(\tfrac{1}{4}\mathcal{H}_0^\mathrm{gc}-\tfrac{1}{2}\ii\pi)^4+\Theta_0(\tfrac{1}{4}\mathcal{H}_0^\mathrm{gc}+\tfrac{1}{2}\ii\pi)^4&=\Theta_0(\ii\pi)^3\Theta_0(\tfrac{1}{2}\mathcal{H}_0^\mathrm{gc})\\
&= m_\mathrm{gc}^{1/4}(1-m_\mathrm{gc})^{3/4}\ee^{-\mathcal{H}_0^\mathrm{gc}/8}\frac{4\KK(m_\mathrm{gc})^2}{\pi^2},
\end{split}
\end{equation}
where on the second line we used \eqref{eq:eliminate-thetas}.
Since $\mathcal{H}_0^\mathrm{gc}<0$, it follows that the left-hand side is $2\kappa_0$ where $\kappa_0:=\mathrm{Re}\{\Theta_0(\tfrac{1}{4}\mathcal{H}_0^\mathrm{gc}-\tfrac{1}{2}\ii\pi)^4\}$, which is clearly positive.  Now let $\lambda_0:=\mathrm{Im}\{\Theta_0(\tfrac{1}{4}\mathcal{H}_0^\mathrm{gc}-\tfrac{1}{2}\ii\pi)^4\}$.  Then writing the established condition $\mathrm{Im}\{\ee^{-\ii\theta/2}\Theta_0(\tfrac{1}{4}\mathcal{H}_0^\mathrm{gc}-\tfrac{1}{2}\ii\pi)^4\}=0$ in the form
\begin{equation}
0=\mathrm{Im}\{\ee^{-\ii\theta/2}(\kappa_0+\ii\lambda_0)\}=\mathrm{Im}\left\{\left(\sqrt{1-m_\mathrm{gc}}-\ii\sqrt{m_\mathrm{gc}}\right)(\kappa_0+\ii\lambda_0)\right\}=\sqrt{1-m_\mathrm{gc}}\lambda_0-\sqrt{m_\mathrm{gc}}\kappa_0
\end{equation}
we get that 
\begin{equation}
\lambda_0=\kappa_0\sqrt{\frac{m_\mathrm{gc}}{1-m_\mathrm{gc}}}.
\end{equation}
Therefore,
\begin{equation}
\begin{split}
\ee^{-\ii\theta/2}\Theta_0(\tfrac{1}{4}\mathcal{H}_0^\mathrm{gc}-\tfrac{1}{2}\ii\pi)^4 &= \mathrm{Re}\{\ee^{-\ii\theta/2}(\kappa_0+\ii\lambda_0)\}\\ &=\mathrm{Re}\left\{\left(\sqrt{1-m_\mathrm{gc}}-\ii\sqrt{m_\mathrm{gc}}\right)(\kappa_0+\ii\lambda_0)\right\}\\ &=\sqrt{1-m_\mathrm{gc}}\kappa_0+\sqrt{m_\mathrm{gc}}\lambda_0\\ &=
\frac{\kappa_0}{\sqrt{1-m_\mathrm{gc}}}\\
&=\frac{2(m_\mathrm{gc}(1-m_\mathrm{gc}))^{1/4}\ee^{-\mathcal{H}_0^\mathrm{gc}/8}\KK(m_\mathrm{gc})^2}{\pi^2}.
\end{split}
\label{eq:Theta-fourth-final-identity}
\end{equation}
Using this in \eqref{eq:real-thing-III} finally gives \eqref{eq:real-thing-IV} and completes the proof.

\begin{remark}
We stress that the derivation of the simple formula \eqref{eq:real-thing-IV} hinges in part on an indirect argument that $\ee^{-\ii\theta/2}\Theta(\tfrac{1}{4}\mathcal{H}_0^\mathrm{gc}-\tfrac{1}{2}\ii\pi;\mathcal{H}_0^\mathrm{gc})$ is real-valued, where $\mathcal{H}_0^\mathrm{gc}<0$ is given by \eqref{eq:H0} in which $m_\mathrm{gc}$ is related to $\theta\in (0,\pi)$ by \eqref{eq:mgc-alphagc}.  We have not found any other more direct proof of this fact.
\end{remark}

\section{Proof of Lemma~\ref{lemma:system-simplify}}
\label{sec:system-simplify}
To write $\mathbf{P}$ and $\mathbf{Q}$ given by \eqref{eq:P-matrix-define} and \eqref{eq:Q-matrix-define} respectively in a compact form, let us introduce
\begin{equation}
\begin{split}
G&:=C_{22}(\alpha_\mathrm{gc})C_{12}'(\alpha_\mathrm{gc})-C_{22}'(\alpha_\mathrm{gc})C_{12}(\alpha_\mathrm{gc})\\
H&:=C_{12}(\alpha_\mathrm{gc})^2+C_{22}(\alpha_\mathrm{gc})^2\\
I&:=C_{12}(\alpha_\mathrm{gc})C_{22}(\alpha_\mathrm{gc})^*-C_{12}(\alpha_\mathrm{gc})^*C_{22}(\alpha_\mathrm{gc})\\
J&:=|C_{12}(\alpha_\mathrm{gc})|^2+|C_{22}(\alpha_\mathrm{gc})|^2,
\end{split}
\end{equation}
so that, using $\alpha_\mathrm{gc}=\ee^{\ii\theta}$, we get
\begin{equation}
\mathbf{P}=(G-\ell(X,T)W'_\mathrm{gc}(\alpha_\mathrm{gc}))\mathbb{I}+\frac{\ee^{-\ii\theta}H}{4}i\sigma_2
\quad\text{and}\quad
\mathbf{Q}=\frac{\ii (\ee^{\ii\theta}-1)I}{4\sin(\theta)}\mathbb{I}+\frac{\ii (\ee^{\ii\theta}+1)J}{4\sin(\theta)} \ii\sigma_2,
\end{equation}
and also
\begin{equation}
\mathbf{B}\sigma_3\mathbf{B}= H\sigma_3\quad\text{and}\quad \mathbf{B}\sigma_3\mathbf{B}^*=\sigma_3\left(J\mathbb{I}+I\ii\sigma_2\right).
\end{equation}
Therefore, looking at the terms on the right-hand side of \eqref{eq:w-system} we see that
\begin{equation}
\ee^{-\ii\theta}\sigma_3\mathbf{P}^*\mathbf{B}\sigma_3\mathbf{B}=\ee^{-\ii\theta}H(G^*-\ell(X,T)^*W_\mathrm{gc}'(\alpha_\mathrm{gc})^*)\mathbb{I}-\frac{|H|^2}{4}\ii\sigma_2
\label{eq:RHS-1}
\end{equation}
and
\begin{equation}
\ee^{-\ii\theta}\sigma_3\mathbf{QB}\sigma_3\mathbf{B}^*=\frac{\ii IJ}{2\sin(\theta)}\mathbb{I}+\frac{\ii((1-\ee^{-\ii\theta})I^2-(1+\ee^{-\ii\theta})J^2)}{4\sin(\theta)}\ii\sigma_2.
\label{eq:RHS-2}
\end{equation}
Now, taking into account that $\alpha_\mathrm{gc}W'_\mathrm{gc}(\alpha_\mathrm{gc})=\ii |W'_\mathrm{gc}(\alpha_\mathrm{gc})|$ due to \eqref{eq:arg-Wprime-gc}, the quantity $H$ is completely characterized by Lemma~\ref{lemma:C-real}.  
Indeed, using $\epsilon^{-1}\Phi^l(t)=\Omega_{\mathrm{p},N}-\omega_\mathrm{gc}T$ and \eqref{eq:omega0} we have
\begin{equation}
H=2\ii |W'_\mathrm{gc}(\alpha_\mathrm{gc})|^{1/2}\left(\frac{m_\mathrm{gc}}{1-m_\mathrm{gc}}\right)^{1/4}
\mathrm{cn}\left(\frac{2\KK(m_\mathrm{gc})\Omega_{\mathrm{p},N}}{\pi}+T;m_\mathrm{gc}\right).
\label{eq:H-simple}
\end{equation}
The quantities $I$ and $J$ are obviously purely imaginary and positive real, respectively.  By following similar reasoning as in Appendix~\ref{sec:proof-lemma:C-real} (we leave the details to the reader) one can show that
\begin{equation}
I=2\ii |W'_\mathrm{gc}(\alpha_\mathrm{gc})|^{1/2}(m_\mathrm{gc}(1-m_\mathrm{gc}))^{1/4}\mathrm{sn}\left(\frac{2\KK(m_\mathrm{gc})\Omega_{\mathrm{p},N}}{\pi}+T;m_\mathrm{gc}\right)
\label{eq:I-simple}
\end{equation}
and
\begin{equation}
J=2|W'_\mathrm{gc}(\alpha_\mathrm{gc})|^{1/2}\left(\frac{m_\mathrm{gc}}{1-m_\mathrm{gc}}\right)^{1/4}\mathrm{dn}\left(\frac{2\KK(m_\mathrm{gc})\Omega_{\mathrm{p},N}}{\pi}+T;m_\mathrm{gc}\right).
\label{eq:J-simple}
\end{equation}

To express $G$ in terms of elliptic functions is more difficult; in fact $G$ is not an elliptic function of $T$ at all, although it is periodic in the real direction, as we will now show. 
First, we write
\begin{equation}
G=C_{22}(\alpha_\mathrm{gc})^2\frac{\dd}{\dd w}\left[\frac{C_{12}(w)}{C_{22}(w)}\right]_{w=\alpha_\mathrm{gc}}.
\label{eq:coefficient-derivative-fraction}
\end{equation}
Now, from \eqref{eq:C12}--\eqref{eq:C22} we see that some factors cancel in the fraction $C_{12}(w)/C_{22}(w)$ and in fact we can write the latter in the form
\begin{equation}
\frac{C_{12}(w)}{C_{22}(w)}=-\ii\frac{\mathcal{F}(A)-\mathcal{G}(A)}{\mathcal{F}(A)+\mathcal{G}(A)},
\end{equation}
where $A=A(w)$ is the branch of the Abel map defined in \eqref{eq:Abel-branch}, and where entire functions $\mathcal{F}(\cdot)$ and $\mathcal{G}(\cdot)$ are defined by
\begin{equation}
\mathcal{F}(A):=\frac{\ee^{\ii\Phi^l(t)/(2\epsilon)}\ee^{A/2}\Theta(A+\mathcal{K}+\ii\varphi^-)-\ii\ee^{-\ii\Phi^l(t)/(2\epsilon)}\ee^{A/2}\Theta(A+\mathcal{K}-\ii\varphi^-)}{\Theta(\ii\pi+\ii\varphi^-)}
\end{equation}
and
\begin{equation}
\mathcal{G}(A):=\frac{\ee^{\ii\Phi^l(t)/(2\epsilon)}\ee^{A/2}\Theta(A+\mathcal{K}+\ii\varphi^+)+\ii\ee^{-\ii\Phi^l(t)/(2\epsilon)}\ee^{A/2}\Theta(A+\mathcal{K}-\ii\varphi^+)}{\Theta(\ii\pi+\ii\varphi^+)}.
\end{equation}
The definition \eqref{eq:shifted-phases-define} with $\nu=\epsilon^{-1}\Phi^l(t)$ and the identities \eqref{eq:Theta-identities} can then be used to show that $\mathcal{F}(\cdot)$ and $\mathcal{G}(\cdot)$ are both even functions.  In fact, Taylor expansion about $A=0$ shows that
\begin{multline}
\Theta(\ii\pi+\ii\varphi^-)\mathcal{F}(A)=\ee^{\ii\Phi^l(t)/(2\epsilon)}\Theta(\mathcal{K}+\ii\varphi^-)-\ii\ee^{-\ii\Phi^l(t)/(2\epsilon)}\Theta(\mathcal{K}-\ii\varphi^-) \\
+\left[\ee^{\ii\Phi^l(t)/(2\epsilon)}\left(\Theta'(\mathcal{K}+\ii\varphi^-) + \frac{1}{2}\Theta(\mathcal{K}+\ii\varphi^-)\right)-\ii\ee^{-\ii\Phi^l(t)/(2\epsilon)}\left(\Theta'(\mathcal{K}-\ii\varphi^-)+\frac{1}{2}\Theta(\mathcal{K}-\ii\varphi^-)\right)\right]A\\
+
\left[\frac{1}{2}\ee^{\ii\Phi^l(t)/(2\epsilon)}\left(\Theta''(\mathcal{K}+\ii\varphi^-)+\Theta'(\mathcal{K}+\ii\varphi^-)+\frac{1}{4}\Theta(\mathcal{K}+\ii\varphi^-)\right) \right.\\
\left.{}-\frac{1}{2}\ii\ee^{-\ii\Phi^l(t)/(2\epsilon)}\left(\Theta''(\mathcal{K}-\ii\varphi^-)+\Theta'(\mathcal{K}-\ii\varphi^-)+\frac{1}{4}\Theta(\mathcal{K}-\ii\varphi^-)\right)\right]A^2+\mathcal{O}(A^3),
\end{multline}
and using the fact that $\mathcal{F}$ is an even function of $A$ the coefficient of $A$ necessarily vanishes (this can also be checked directly), and the error term is proportional to $A^4$.  The fact that the coefficient of $A$ vanishes can then be used to eliminate the first derivatives of $\Theta$ in the coefficient of $A^2$ as follows:
\begin{multline}
\Theta(\ii\pi+\ii\varphi^-)\mathcal{F}(A)=\ee^{\ii\Phi^l(t)/(2\epsilon)}\Theta(\mathcal{K}+\ii\varphi^-)-\ii\ee^{-\ii\Phi^l(t)/(2\epsilon)}\Theta(\mathcal{K}-\ii\varphi^-) \\
+
\left[\frac{1}{2}\ee^{\ii\Phi^l(t)/(2\epsilon)}\left(\Theta''(\mathcal{K}+\ii\varphi^-)-\frac{1}{4}\Theta(\mathcal{K}+\ii\varphi^-)\right) \right.\\
\left.{}-\frac{1}{2}\ii\ee^{-\ii\Phi^l(t)/(2\epsilon)}\left(\Theta''(\mathcal{K}-\ii\varphi^-)-\frac{1}{4}\Theta(\mathcal{K}-\ii\varphi^-)\right)\right]A^2+\mathcal{O}(A^4).
\label{eq:F-expansion}
\end{multline}
In exactly the same way, we find the expansion
\begin{multline}
\Theta(\ii\pi+\ii\varphi^+)\mathcal{G}(A)=\ee^{\ii\Phi^l(t)/(2\epsilon)}\Theta(\mathcal{K}+\ii\varphi^+)+\ii\ee^{-\ii\Phi^l(t)/(2\epsilon)}\Theta(\mathcal{K}-\ii\varphi^+)\\
+\left[\frac{1}{2}\ee^{\ii\Phi^l(t)/(2\epsilon)}\left(\Theta''(\mathcal{K}+\ii\varphi^+)-\frac{1}{4}\Theta(\mathcal{K}+\ii\varphi^+)\right)\right.\\
\left.{}+\frac{1}{2}\ii\ee^{-\ii\Phi^l(t)/(2\epsilon)}\left(\Theta''(\mathcal{K}-\ii\varphi^+)-\frac{1}{4}\Theta(\mathcal{K}-\ii\varphi^+)\right)\right]A^2 + \mathcal{O}(A^4).
\label{eq:G-expansion}
\end{multline}
To abbreviate the above Taylor coefficients we write $\mathcal{F}(A)=\mathcal{F}_0+\mathcal{F}_1A^2+\mathcal{O}(A^4)$ and $\mathcal{G}(A)=\mathcal{G}_0+\mathcal{G}_1A^2+\mathcal{O}(A^4)$.  Then we have the corresponding Taylor expansion
\begin{equation}
\frac{C_{12}(w)}{C_{22}(w)}=-\ii\frac{\mathcal{F}_0-\mathcal{G}_0}{\mathcal{F}_0+\mathcal{G}_0} + 2\ii\frac{\mathcal{F}_0\mathcal{G}_1-\mathcal{F}_1\mathcal{G}_0}{(\mathcal{F}_0+\mathcal{G}_0)^2}A^2+\mathcal{O}(A^4).
\end{equation}
Therefore, using \eqref{eq:A1-squared},
\begin{equation}
\frac{C_{12}(w)}{C_{22}(w)}=-\ii\frac{\mathcal{F}_0-\mathcal{G}_0}{\mathcal{F}_0+\mathcal{G}_0}+\frac{8\ii c_\mathrm{gc}^2}{\alpha_\mathrm{gc}(\alpha_\mathrm{gc}-\alpha_\mathrm{gc}^*)}\cdot\frac{\mathcal{F}_0\mathcal{G}_1-\mathcal{F}_1\mathcal{G}_0}{(\mathcal{F}_0+\mathcal{G}_0)^2}(w-\alpha_\mathrm{gc})+\mathcal{O}((w-\alpha_\mathrm{gc})^2),
\end{equation}
so that
\begin{equation}
\frac{\dd}{\dd w}\left[\frac{C_{12}(w)}{C_{22}(w)}\right]_{w=\alpha_\mathrm{gc}}=\frac{8\ii c_\mathrm{gc}^2}{\alpha_\mathrm{gc}(\alpha_\mathrm{gc}-\alpha_\mathrm{gc}^*)}\cdot\frac{\mathcal{F}_0\mathcal{G}_1-\mathcal{F}_1\mathcal{G}_0}{(\mathcal{F}_0+\mathcal{G}_0)^2}.
\end{equation}
Next, observing that in terms of $\mathcal{F}_0$ and $\mathcal{G}_0$, $C_{22}(\alpha_\mathrm{gc})^2$ as given by \eqref{eq:C22} can be rewritten in the form
\begin{equation}
C_{22}(\alpha_\mathrm{gc})^2=\left[\frac{q(w)^2W_\mathrm{gc}(w)^{1/2}}{\Theta(A+\mathcal{K})^2}\right]_{w=\alpha_\mathrm{gc}}\frac{\Theta(\ii\pi)^2}{8}(\mathcal{F}_0+\mathcal{G}_0)^2,
\end{equation}
in which the evaluation at $w=\alpha_\mathrm{gc}$ is understood as a limit taken from the sectors $\mathrm{III}$ or $\mathrm{IV}$ of $U$, we arrive at
\begin{equation}
G=\left[\frac{q(w)^2W_\mathrm{gc}(w)^{1/2}}{\Theta(A+\mathcal{K})^2}\right]_{w=\alpha_\mathrm{gc}}\frac{\ii c_\mathrm{gc}^2\Theta(\ii\pi)^2}{\alpha_\mathrm{gc}(\alpha_\mathrm{gc}-\alpha_\mathrm{gc}^*)}(\mathcal{F}_0\mathcal{G}_1-\mathcal{F}_1\mathcal{G}_0).
\label{eq:C12primeC22-minus-swap}
\end{equation}

Using \eqref{eq:Theta-expand} with \eqref{eq:L-define} and then \eqref{eq:Theta-prime-K} and \eqref{eq:Theta-fourth-identity} followed by \eqref{eq:Theta-fourth-final-identity} and using the facts that $\arg(L)=-\tfrac{1}{2}\theta$ and $\sin(\theta)=2m_\mathrm{gc}^{1/2}(1-m_\mathrm{gc})^{1/2}$ puts this in the simpler form
\begin{equation}
G=\ee^{-\ii\theta}\ii |W_\mathrm{gc}(\alpha_\mathrm{gc})|^{1/2}
\frac{\pi^2\ee^{\mathcal{H}_0^\mathrm{gc}/8}}{16 \KK(m_\mathrm{gc})^2\sqrt{m_\mathrm{gc}(1-m_\mathrm{gc})}}(\mathcal{F}_0\mathcal{G}_1-\mathcal{F}_1\mathcal{G}_0).
\label{eq:G-again}
\end{equation}
So it remains to calculate $\mathcal{F}_0\mathcal{G}_1-\mathcal{F}_1\mathcal{G}_0$; some terms cancel immediately and we obtain the expression
\begin{multline}
2\Theta(\ii\pi+\ii\varphi^-)\Theta(\ii\pi+\ii\varphi^+)(\mathcal{F}_0\mathcal{G}_1-\mathcal{F}_1\mathcal{G}_0)=\\
\ee^{\ii\Phi^l(t)/\epsilon}\left(\Theta(\mathcal{K}+\ii\varphi^-)\Theta''(\mathcal{K}+\ii\varphi^+)-
\Theta''(\mathcal{K}+\ii\varphi^-)\Theta(\mathcal{K}+\ii\varphi^+)\right)\\
{}+\ee^{-\ii\Phi^l(t)/\epsilon}\left(\Theta(\mathcal{K}-\ii\varphi^-)\Theta''(\mathcal{K}-\ii\varphi^+)-
\Theta''(\mathcal{K}-\ii\varphi^-)\Theta(\mathcal{K}-\ii\varphi^+)\right)\\
{}+\ii\left(\Theta(\mathcal{K}+\ii\varphi^-)\Theta''(\mathcal{K}-\ii\varphi^+)-\Theta''(\mathcal{K}+\ii\varphi^-)\Theta(\mathcal{K}-\ii\varphi^+)\right)\\
{}-\ii\left(\Theta(\mathcal{K}-\ii\varphi^-)\Theta''(\mathcal{K}+\ii\varphi^+)-\Theta''(\mathcal{K}-\ii\varphi^-)\Theta(\mathcal{K}+\ii\varphi^+)\right),
\end{multline}
where we recall \eqref{eq:Theta-simple} and $\varphi^\pm=\epsilon^{-1}\Phi^l(t)\pm\tfrac{1}{2}\ii \pi$.
We now wish to convert expressions involving $\Theta(z)=\Theta(z;\mathcal{H}^\mathrm{gc})$ into equivalent expressions involving $\Theta_0(z):=\Theta(z;\mathcal{H}^\mathrm{gc}_0)$.  This can be accomplished by the use of \eqref{eq:theta-last-identity} and its second derivatives with respect to $z_1$ and $z_2$ which can be combined to yield the identity
\begin{multline}
\Theta(z_1)\Theta''(z_2)-\Theta''(z_1)\Theta(z_2)=-4\Theta_0'(z_1+z_2+\ii\pi)\Theta_0'(z_1-z_2+\ii\pi)\\{}-
\ii\ee^{\mathcal{H}_0^\mathrm{gc}/4}\ee^{z_1}\Theta_0(z_1+z_2+\tfrac{1}{2}\mathcal{H}_0^\mathrm{gc})\Theta_0(z_1-z_2+\tfrac{1}{2}\mathcal{H}_0^\mathrm{gc})
-4\ii\ee^{\mathcal{H}_0^\mathrm{gc}/4}\ee^{z_1}\Theta_0'(z_1+z_2+\tfrac{1}{2}\mathcal{H}_0^\mathrm{gc})\Theta_0'(z_1-z_2+\tfrac{1}{2}\mathcal{H}_0^\mathrm{gc})\\{} - 2\ii\ee^{\mathcal{H}_0^\mathrm{gc}/4}\ee^{z_1}
\Theta_0'(z_1+z_2+\tfrac{1}{2}\mathcal{H}_0^\mathrm{gc})\Theta_0(z_1-z_2+\tfrac{1}{2}\mathcal{H}_0^\mathrm{gc})
-2\ii\ee^{\mathcal{H}_0^\mathrm{gc}/4}\ee^{z_1}\Theta_0(z_1+z_2+\tfrac{1}{2}\mathcal{H}_0^\mathrm{gc})\Theta_0'(z_1-z_2+\tfrac{1}{2}\mathcal{H}_0^\mathrm{gc}).
\end{multline}
Therefore all second derivatives of theta functions cancel, and after much use of the identities \eqref{eq:Theta-identities} we arrive at
\begin{equation}
\mathcal{F}_0\mathcal{G}_1-\mathcal{F}_1\mathcal{G}_0=-8\ii\frac{\Theta_0'(\mathcal{K}_0)}{\Theta_0(0)}\frac{\Theta_0'(2\ii\epsilon^{-1}\Phi^l(t)+\ii\pi)}{\Theta_0(2\ii\epsilon^{-1}\Phi^l(t)+\ii\pi)},
\end{equation}
in which $\mathcal{K}_0:=\tfrac{1}{2}\mathcal{H}_0^\mathrm{gc}+\ii\pi$ is the Riemann constant for $\Theta_0(z)$, i.e., all zeros of $\Theta_0(z)$ lie at the points $\mathcal{K}_0+2\pi\ii m + \mathcal{H}_0^\mathrm{gc}n$ for $(m,n)\in\mathbb{Z}^2$.
Using \eqref{eq:Theta-prime-K} and \eqref{eq:eliminate-thetas} to eliminate $\Theta_0'(\mathcal{K}_0)/\Theta_0(0)$ and substituting into \eqref{eq:G-again} gives
\begin{equation}
G=\ee^{-\ii\theta}|W_\mathrm{gc}'(\alpha_\mathrm{gc})|^{1/2}\frac{\pi}{2\KK(m_\mathrm{gc})(m_\mathrm{gc}(1-m_\mathrm{gc}))^{1/4}}
\cdot\frac{\Theta_0'(2\ii\epsilon^{-1}\Phi^l(t)+\ii\pi)}{\Theta_0(2\ii\epsilon^{-1}\Phi^l(t)+\ii\pi)}.
\label{eq:G-latest}
\end{equation}
Since $\epsilon^{-1}\Phi^l(t)=\Omega_{\mathrm{p},N}-\omega_\mathrm{gc}T$ with $\omega_\mathrm{gc}<0$ defined by \eqref{eq:omega0}, it is obvious from the $2\pi\ii$-periodicity of $\Theta_0(z)$ that $G$ is a periodic function of $T$.  Now $\Theta_0'(z)/\Theta_0(z)=\dd\log(\Theta_0(z))/\dd z$ is not an elliptic function, since according to \eqref{eq:Theta-identities} it grows linearly in the $\mathcal{H}_0^\mathrm{gc}$-direction.  But the same formula shows that its derivative is an elliptic function because
\begin{equation}
\left.\frac{\dd^2}{\dd z^2}\log(\Theta_0(z))\right|_{z=w+2\pi\ii}=\left.\frac{\dd^2}{\dd z^2}\log(\Theta_0(z))\right|_{z=w+\mathcal{H}_\mathrm{gc}^0}=\left.\frac{\dd^2}{\dd z^2}\log(\Theta_0(z))\right|_{z=w}.
\end{equation}
This elliptic function has only one singularity in its period rectangle $(\mathrm{Re}\{z\},\mathrm{Im}\{z\})\in (\mathcal{H}_0^\mathrm{gc},0]\times (0,2\pi]$, namely a double pole at the center  $z=\mathcal{K}_0$ of the rectangle.  Since $\mathcal{K}_0$ is a simple zero of the entire function $\Theta_0(z)$, it follows that
in a neighborhood of $z=\mathcal{K}_0$,
\begin{equation}
\frac{\Theta_0'(z)}{\Theta(z)}=\frac{\dd}{\dd z}\log(\Theta_0(z))=\frac{1}{z-\mathcal{K}_0}+\text{holomorphic,}
\end{equation}
so the corresponding Laurent expansion of its derivative is
\begin{equation}
\frac{\dd^2}{\dd z^2}\log(\Theta_0(z))=\frac{-1}{(z-\mathcal{K}_0)^2}+\text{holomorphic}.
\label{eq:Laurent-doubleprime}
\end{equation}
Recalling the expression \eqref{eq:H0} for $\mathcal{H}_0^\mathrm{gc}$, it is convenient to introduce the rescaled independent variable $\zeta:=-\ii \KK(m_\mathrm{gc})z/\pi$.  In terms of $\zeta$, the period rectangle becomes $(\mathrm{Re}\{\zeta\},\mathrm{Im}\{\zeta\})\in (0,2\KK(m_\mathrm{gc})]\times (0,2\KK(1-m_\mathrm{gc})]$, and the Laurent expansion \eqref{eq:Laurent-doubleprime} becomes
\begin{equation}
\frac{\dd^2}{\dd z^2}\log(\Theta_0(z))=\frac{\KK(m_\mathrm{gc})^2}{\pi^2}\frac{1}{(\zeta-(\KK(m_\mathrm{gc})+\ii \KK(1-m_\mathrm{gc})))^2} + \text{holomorphic.}
\end{equation}
Now, according to \cite[\S 22.4]{NIST:DLMF}, $\mathrm{dn}(\zeta;m_\mathrm{gc})^{-2}$ is an elliptic function with the same periods, and 
\begin{equation}
\frac{1}{\mathrm{dn}(\zeta;m_\mathrm{gc})^2}=-\frac{1}{1-m_\mathrm{gc}}\cdot\frac{1}{(\zeta-(\KK(m_\mathrm{gc})+\ii \KK(1-m_\mathrm{gc})))^2} + \text{holomorphic.}
\end{equation}
From this it follows that 
\begin{equation}
F(z):=\frac{\dd^2}{\dd z^2}\log(\Theta_0(z)) + \frac{(1-m_\mathrm{gc})\KK(m_\mathrm{gc})^2}{\pi^2}\cdot\frac{1}{\mathrm{dn}(\zeta;m_\mathrm{gc})^2}
\label{eq:constant-function}
\end{equation}
is an entire elliptic function of $z$, hence a constant.  The constant may be computed by evaluating at $z=0$ and $\zeta=0$:
\begin{equation}
F(z)=F(0)=\frac{\Theta_0''(0)}{\Theta_0(0)}+\frac{(1-m_\mathrm{gc})\KK(m_\mathrm{gc})^2}{\pi^2}.
\end{equation}
However, an alternative way to evaluate the constant that seems more useful is to use the fact that the logarithmic derivative $\dd\log(\Theta_0(z))/\dd z$ is $2\pi\ii$-periodic, so
\begin{equation}
\begin{split}
F(z)=\frac{1}{2\pi\ii}\int_0^{2\pi\ii}F(z)\,\dd z &= \frac{(1-m_\mathrm{gc})\KK(m_\mathrm{gc})^2}{2\pi^3\ii}\int_0^{2\pi\ii}\frac{\dd z}{\mathrm{dn}(-\ii \KK(m_\mathrm{gc})z/\pi;m_\mathrm{gc})^2}\\ & = \frac{(1-m_\mathrm{gc})\KK(m_\mathrm{gc})^2}{\pi^2}\frac{1}{2\KK(m_\mathrm{gc})}\int_0^{2\KK(m_\mathrm{gc})}\frac{\dd\zeta}{\mathrm{dn}(\zeta;m_\mathrm{gc})^2}\\
&=\frac{(1-m_\mathrm{gc})\KK(m_\mathrm{gc})^2}{\pi^2}\left\langle\frac{1}{\mathrm{dn}(\cdot;m_\mathrm{gc})^2}\right\rangle.
\end{split}
\end{equation}
where for a periodic function $f$ of a real variable $\langle f(\cdot)\rangle$ denotes the average value.
Using the latter expression for the constant $F(z)$ on the left-hand side of \eqref{eq:constant-function}
gives
\begin{equation}
\frac{\dd^2}{\dd z^2}\log(\Theta_0(z))=\frac{(1-m_\mathrm{gc})\KK(m_\mathrm{gc})^2}{\pi^2}\left(\left\langle
\frac{1}{\mathrm{dn}(\cdot;m_\mathrm{gc})^2}\right\rangle-\frac{1}{\mathrm{dn}(-\ii \KK(m_\mathrm{gc})z/\pi;m_\mathrm{gc})^2}\right).
\end{equation}
Finally, since $\Theta_0(z)$ is an even function of $z$, $\Theta_0'(0)/\Theta_0(0)=0$, so we obtain the expression
\begin{equation}
\begin{split}
\frac{\Theta_0'(z)}{\Theta_0(z)}&=\frac{(1-m_\mathrm{gc})\KK(m_\mathrm{gc})^2}{\pi^2}\int_0^z
\left(\left\langle
\frac{1}{\mathrm{dn}(\cdot;m_\mathrm{gc})^2}\right\rangle-\frac{1}{\mathrm{dn}(-\ii \KK(m_\mathrm{gc})z'/\pi;m_\mathrm{gc})^2}\right)\,\dd z'\\
&=\ii\frac{(1-m_\mathrm{gc})\KK(m_\mathrm{gc})}{\pi}\int_0^{-\ii \KK(m_\mathrm{gc})z/\pi}\left(
\left\langle\frac{1}{\mathrm{dn}(\cdot;m_\mathrm{gc})^2}\right\rangle-\frac{1}{\mathrm{dn}(\zeta;m_\mathrm{gc})^2}\right)\,\dd\zeta.
\end{split}
\end{equation}
So, returning to \eqref{eq:G-latest}, we arrive at the formula
\begin{equation}
G=\ii\ee^{-\ii\theta}
|W'_\mathrm{gc}(\alpha_\mathrm{gc})|^{1/2}
\frac{(1-m_\mathrm{gc})^{3/4}}{2m_\mathrm{gc}^{1/4}}
\mathrm{p}\left(\frac{2\KK(m_\mathrm{gc})\Omega_{\mathrm{p},N}}{\pi}+T;m_\mathrm{gc}\right),
\label{eq:G-simple}
\end{equation}
in which the $2\pi$-periodic function $\mathrm{p}(w;m_\mathrm{gc})$ is defined by \eqref{eq:abbreviated-notation-2}--\eqref{eq:abbreviated-notation-3} (in particular, the compact formula \eqref{eq:abbreviated-notation-3} for the average can be found by making the substitution $t=\mathrm{cn}(\zeta;m_\mathrm{gc})/\mathrm{dn}(\zeta;m_\mathrm{gc})$ and comparing with \eqref{eq:KE}).

Using the results \eqref{eq:H-simple}--\eqref{eq:J-simple} and \eqref{eq:G-simple} in \eqref{eq:RHS-1}, recalling also \eqref{eq:a-and-b-formulae} and the abbreviated notation \eqref{eq:abbreviated-notation-1}--\eqref{eq:abbreviated-notation-3}, we easily establish \eqref{eq:system-LHS-Q}--\eqref{eq:system-LHS-P}, and we also get that
\begin{equation}
|W'_\mathrm{gc}(\alpha_\mathrm{gc})|^{-1}\mathrm{Re}\left\{\ee^{-\ii\theta}\sigma_3\mathbf{P}^*\mathbf{B}\sigma_3\mathbf{B}\right\}=
\left(\sqrt{1-m_\mathrm{gc}}\mathrm{p}-\sqrt{m_\mathrm{gc}}\rho(m_\mathrm{gc})T\right)
\mathrm{cn}\cdot\mathbb{I} 
-
\sqrt{\frac{m_\mathrm{gc}}{1-m_\mathrm{gc}}}\mathrm{cn}^2\ii\sigma_2
\end{equation}
where $\rho(m_\mathrm{gc})$ is the explicit expression \eqref{eq:rho-func-define}.
Likewise, from \eqref{eq:RHS-2} we get
\begin{equation}
|W'_\mathrm{gc}(\alpha_\mathrm{gc})|^{-1}\mathrm{Re}\left\{\ee^{-\ii\theta}\sigma_3\mathbf{QB}\sigma_3\mathbf{B}^*\right\}=-\frac{1}{\sqrt{1-m_\mathrm{gc}}}\mathrm{sn}\,\mathrm{dn}\cdot\mathbb{I}-\sqrt{\frac{m_\mathrm{gc}}{1-m_\mathrm{gc}}}\mathrm{cn}^2\ii\sigma_2.
\end{equation}
Combining these latter two results, one observes the cancellation of the terms proportional to $\ii\sigma_2$ and recovers \eqref{eq:system-RHS}, completing the proof of the lemma.  

\section{A Catalog of Defects}
\label{sec:Catalog}
Since the defect solutions $U(X,T)=U(X,T;m,\Omega)$ of the sine-Gordon equation in the form $U_{TT}-U_{XX}+\sin(U)=0$ as described in the introduction depend on two real parameters, an elliptic modulus $m\in (0,1)$ and a phase $\Omega\in\mathbb{R}\pmod{2\pi}$, there is a wide variety of possible behavior of the solutions.  In this appendix, we use the exact solution formul\ae\ \eqref{eq:ddotCS-compact}--\eqref{eq:abbreviated-notation-3} to plot $\cos(U(X,T;m,\Omega))$ and $\sin(U(X,T;m,\Omega))$ for various choices of the parameters.  The plots in each of the below figures correspond to the same value of $m$, which would be fixed for all defects in the neighborhood of a given simple gradient catastrophe point.  Given the value of $m=m_\mathrm{gc}$ associated with a given catastrophe point, different defects in the neighborhood of the same point and different values of $N$ would yield different values of the phase parameter $\Omega$.
\begin{figure}[h!]
\begin{center}
\includegraphics[height=.24\linewidth]{fig/Legend-SMALL.pdf}%
\includegraphics[height=.24\linewidth]{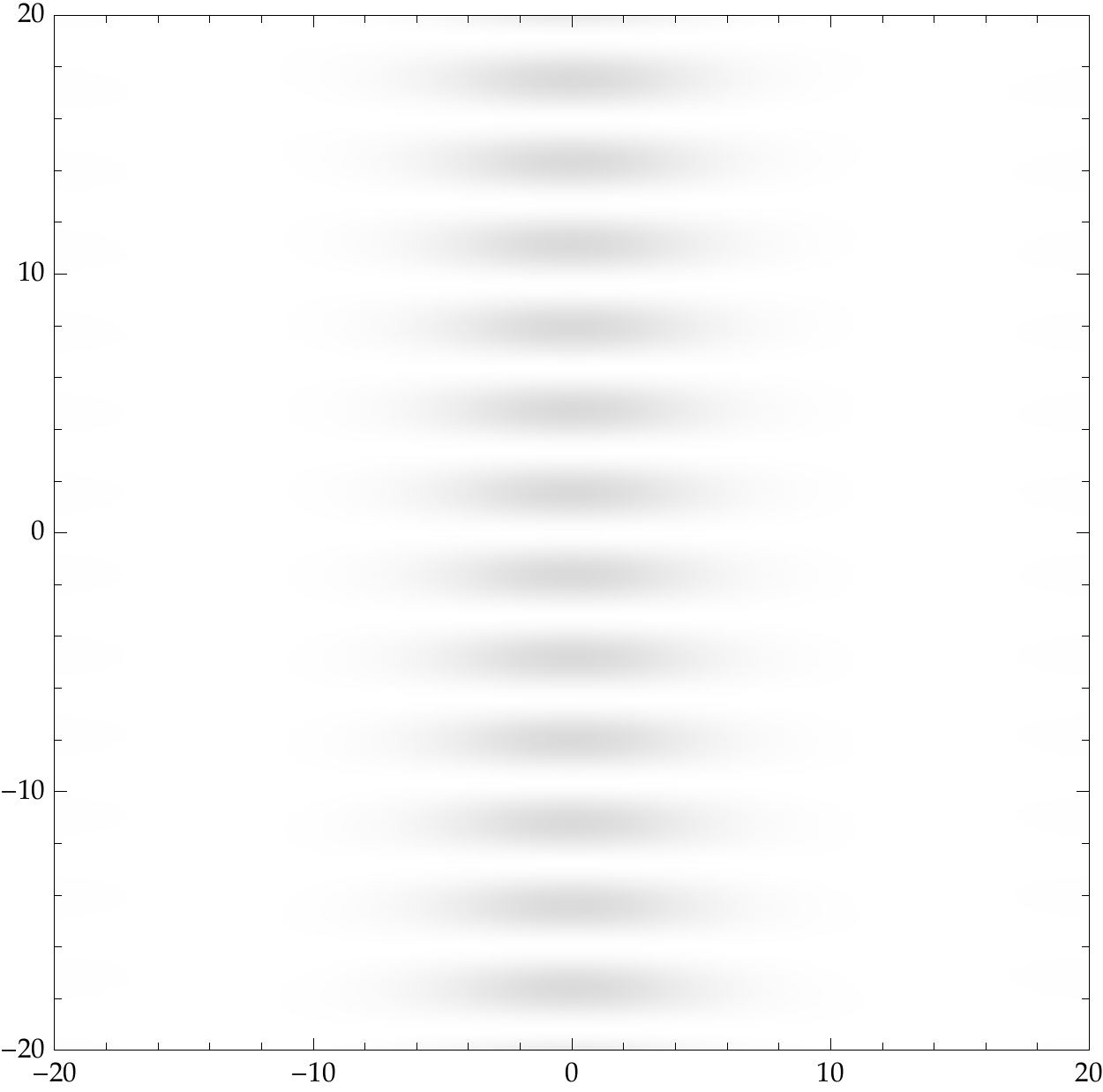}%
\includegraphics[height=.24\linewidth]{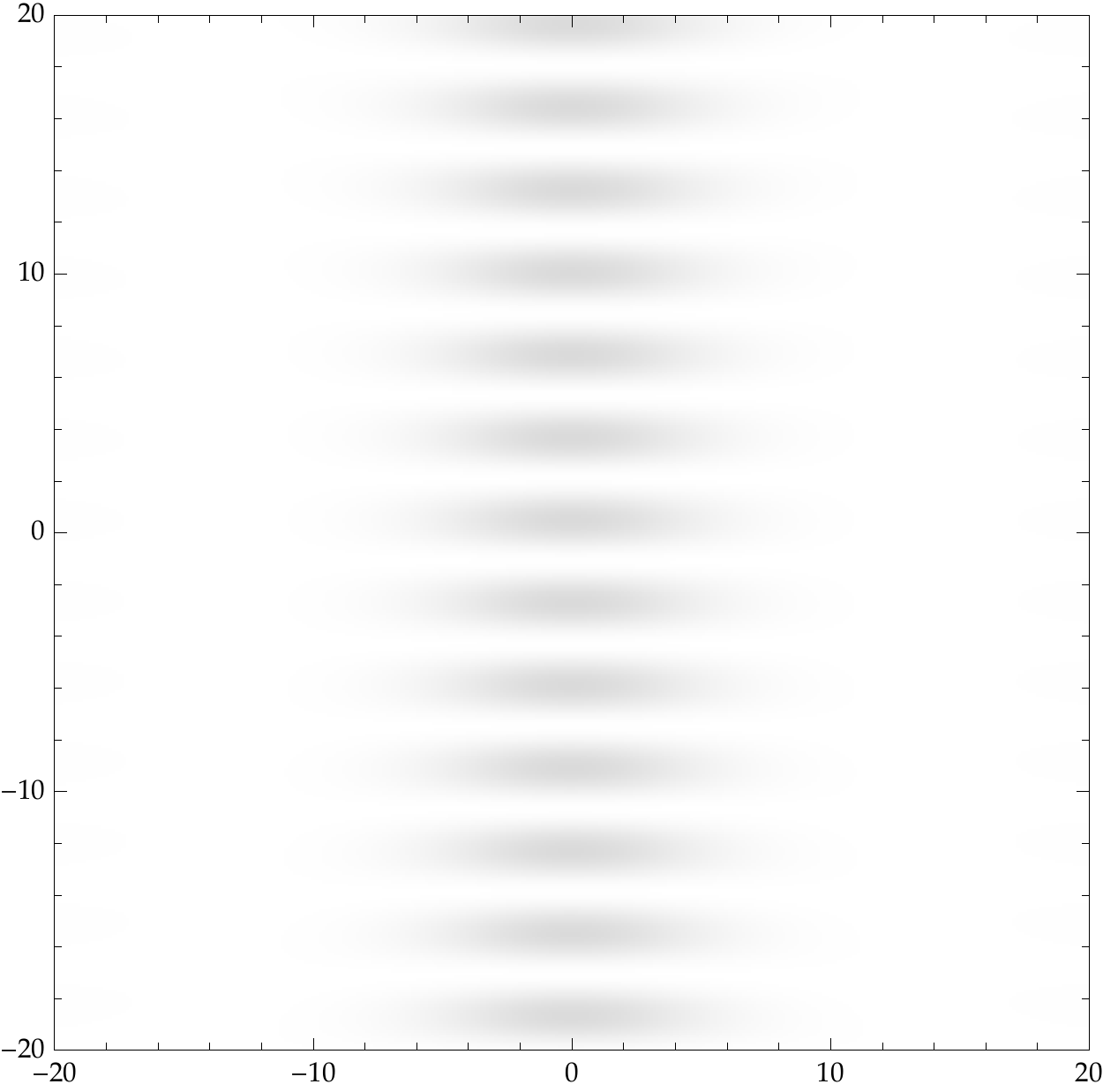}%
\includegraphics[height=.24\linewidth]{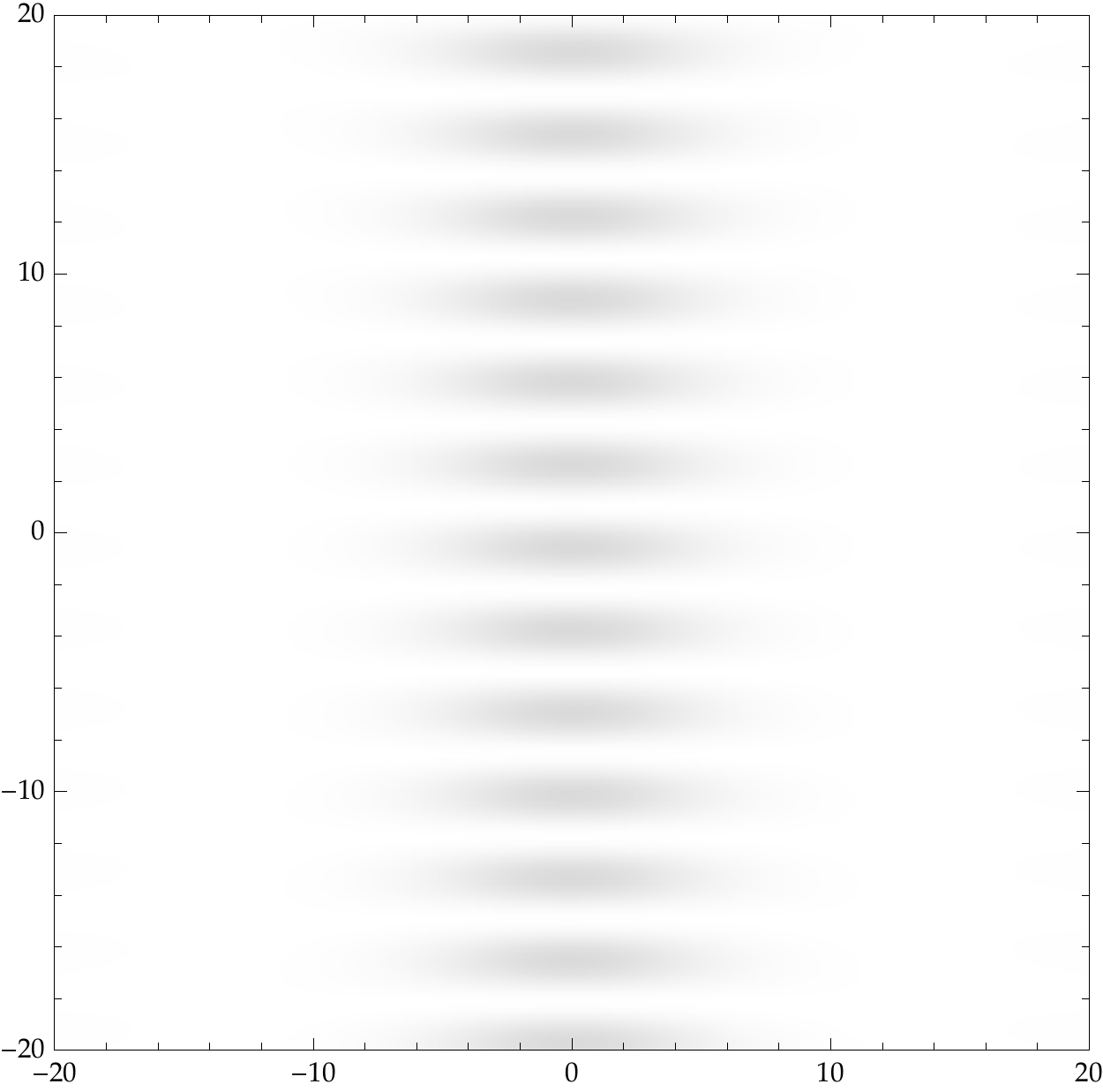}%
\includegraphics[height=.24\linewidth]{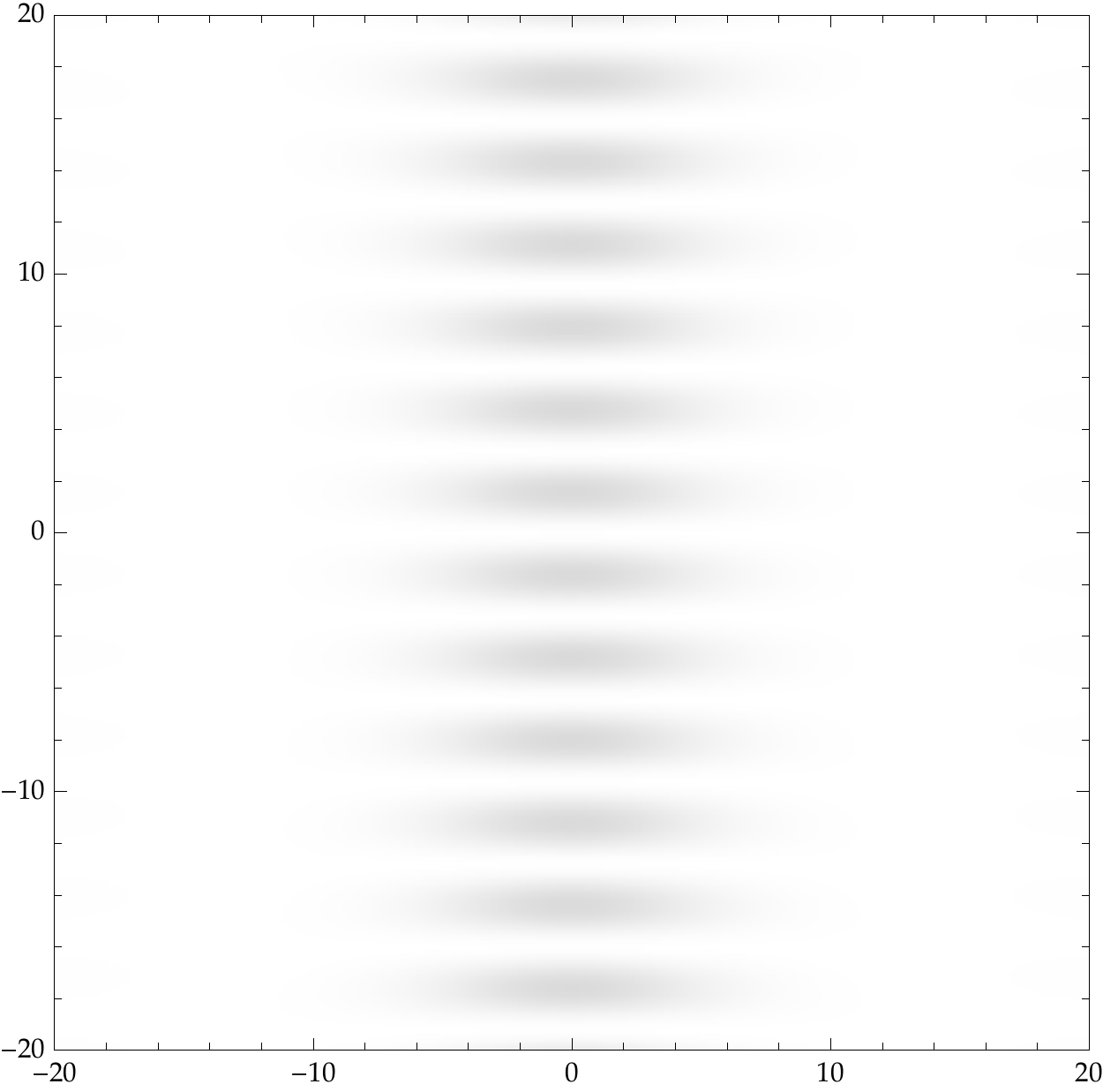}\\
\includegraphics[height=.24\linewidth]{fig/Legend-SMALL.pdf}%
\includegraphics[height=.24\linewidth]{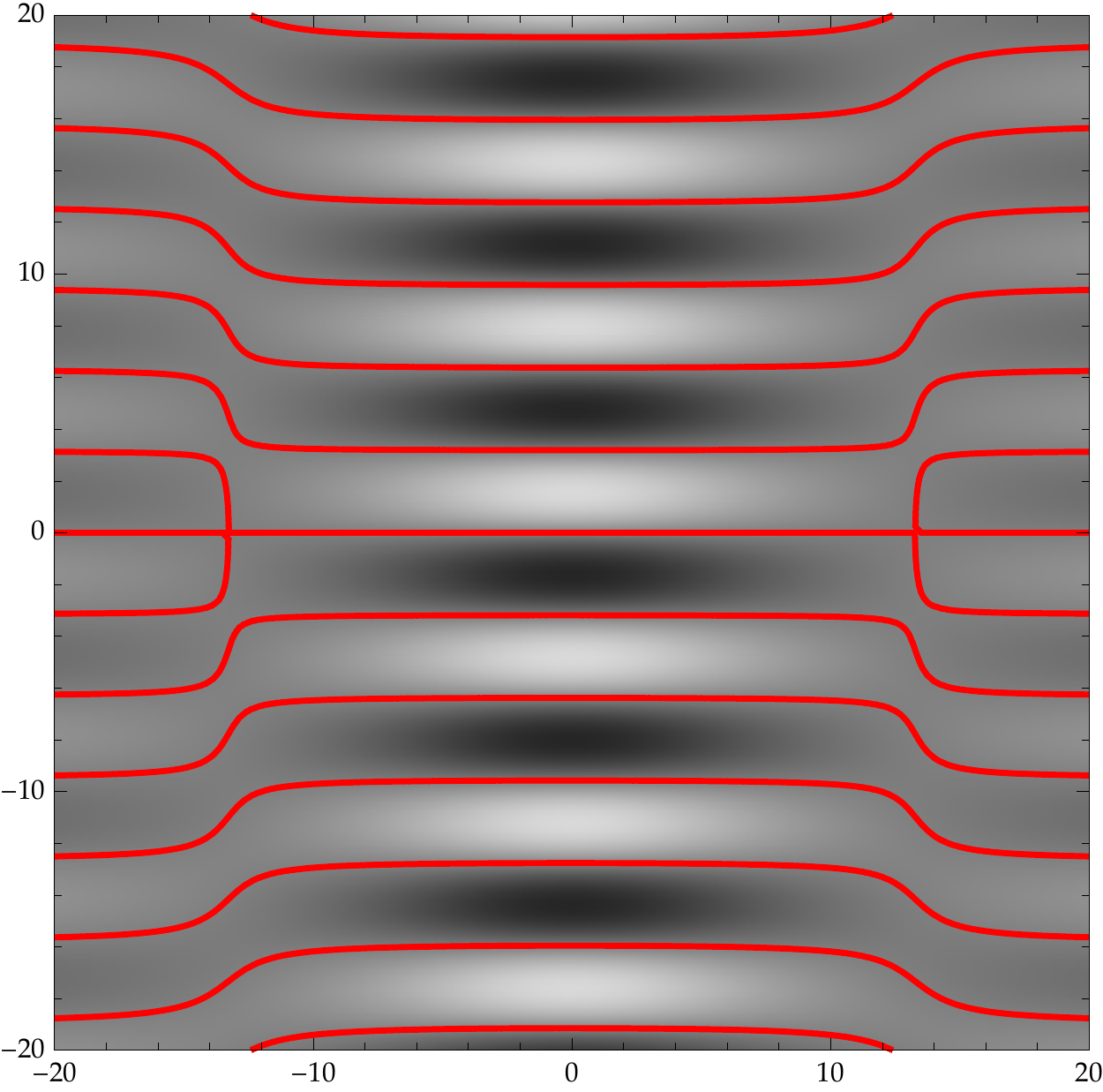}%
\includegraphics[height=.24\linewidth]{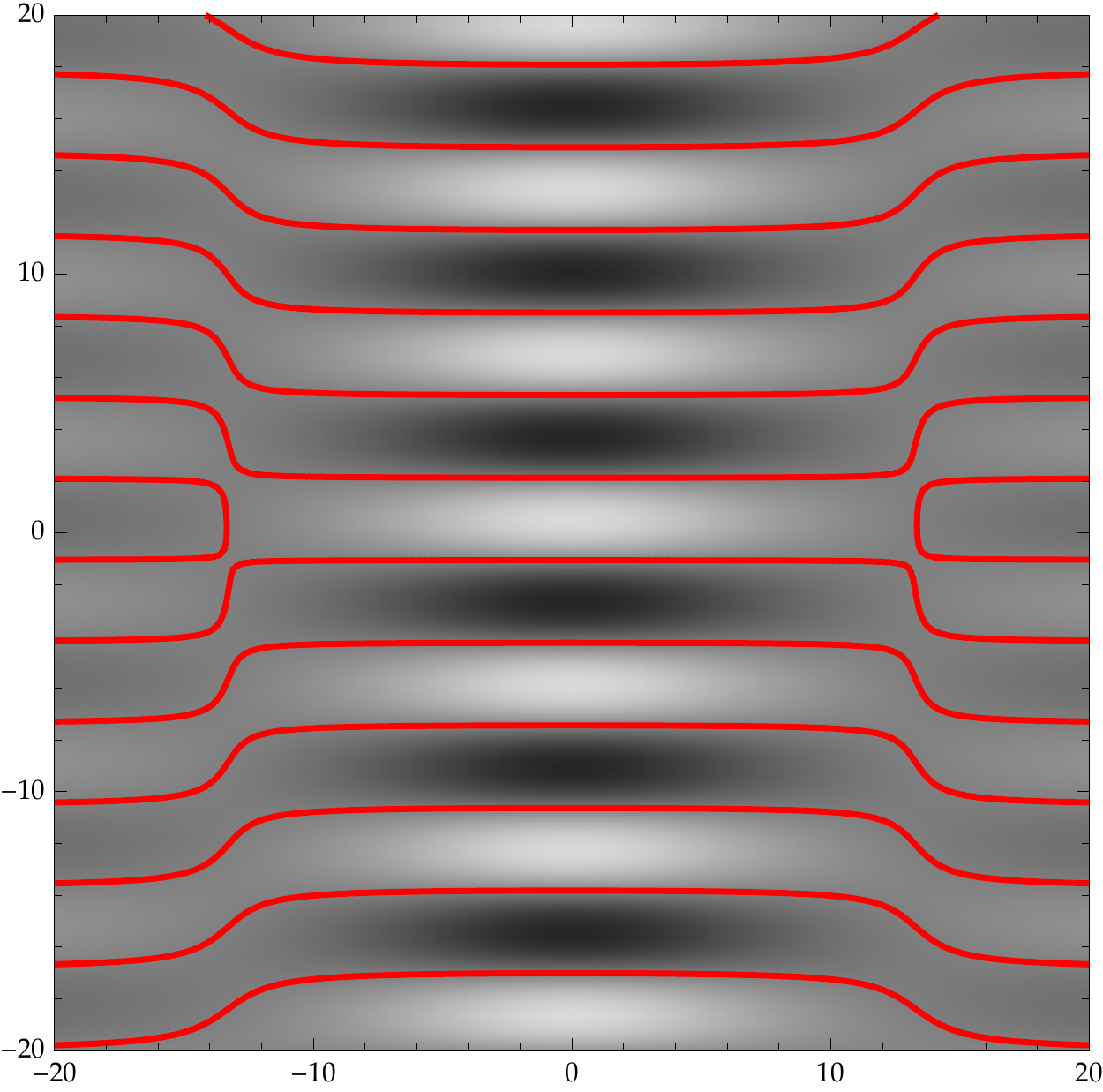}%
\includegraphics[height=.24\linewidth]{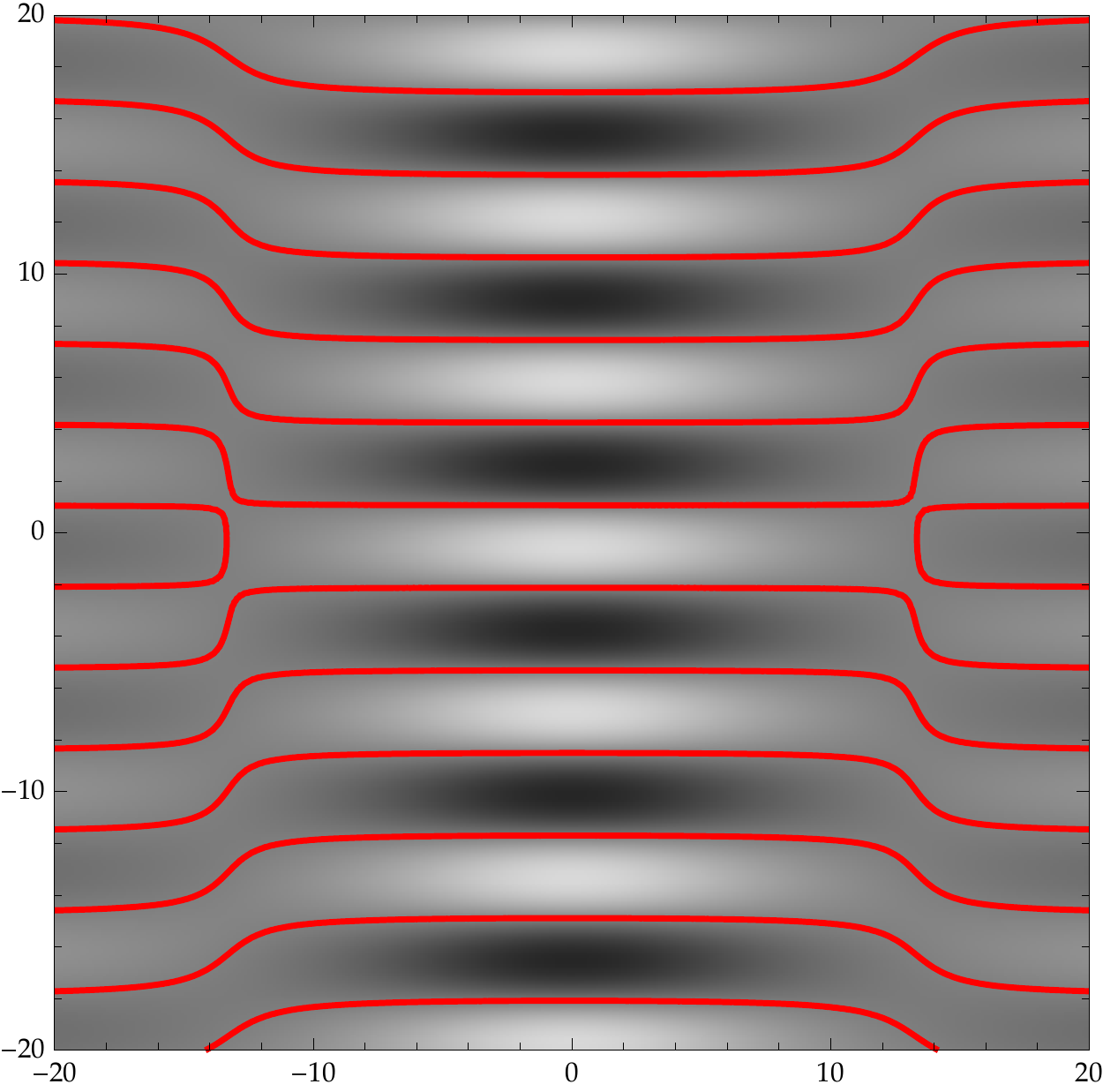}%
\includegraphics[height=.24\linewidth]{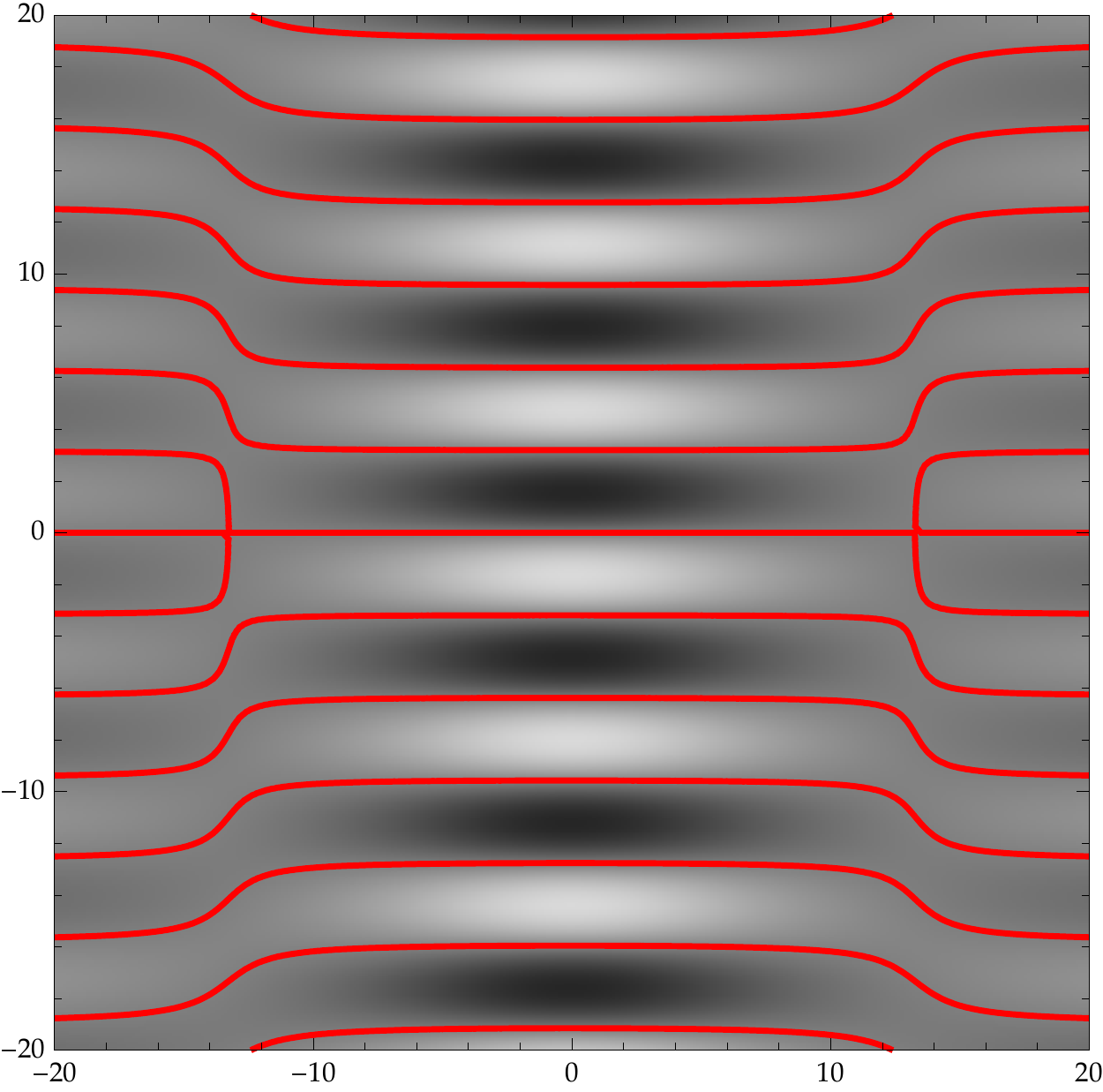}
\end{center}
\caption{Exact solutions corresponding to $m=\sin^2(\tfrac{1}{24}\pi)$ plotted for $-20<X<20$ (horizontal coordinate) and $-20<T<20$ (vertical coordinate), with zero level curves shown in red and grayscale map as indicated in the legends on the left.  Left-to-right:  $\Omega=0,\pi/3,2\pi/3,\pi$.  Top row:  $\cos(U(X,T;m,\Omega))$.  Bottom row:  $\sin(U(X,T;m,\Omega)$.}
\label{fig:exact-solutions-first}
\end{figure}
\begin{figure}[h!]
\begin{center}
\includegraphics[height=.24\linewidth]{fig/Legend-SMALL.pdf}%
\includegraphics[height=.24\linewidth]{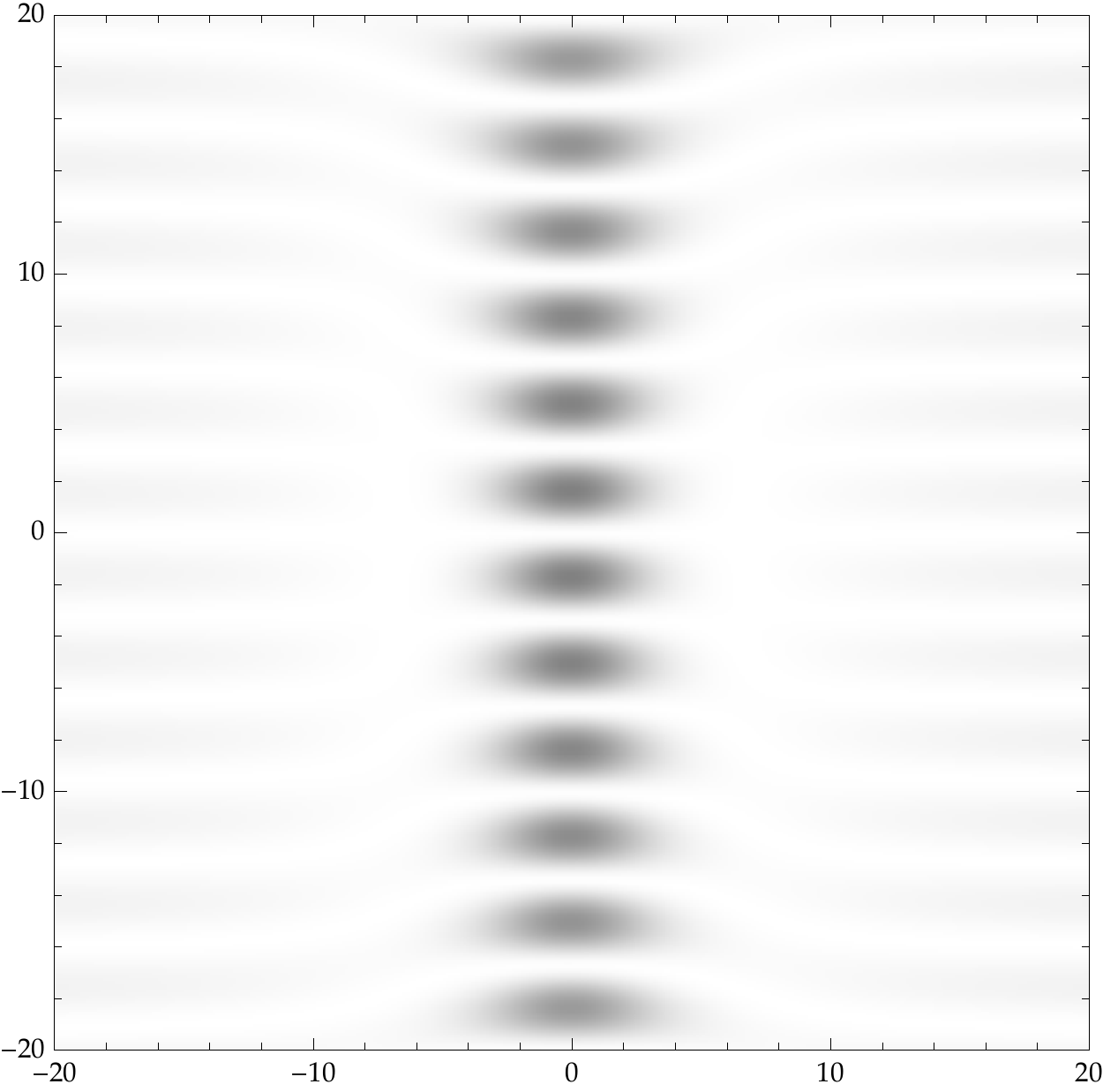}%
\includegraphics[height=.24\linewidth]{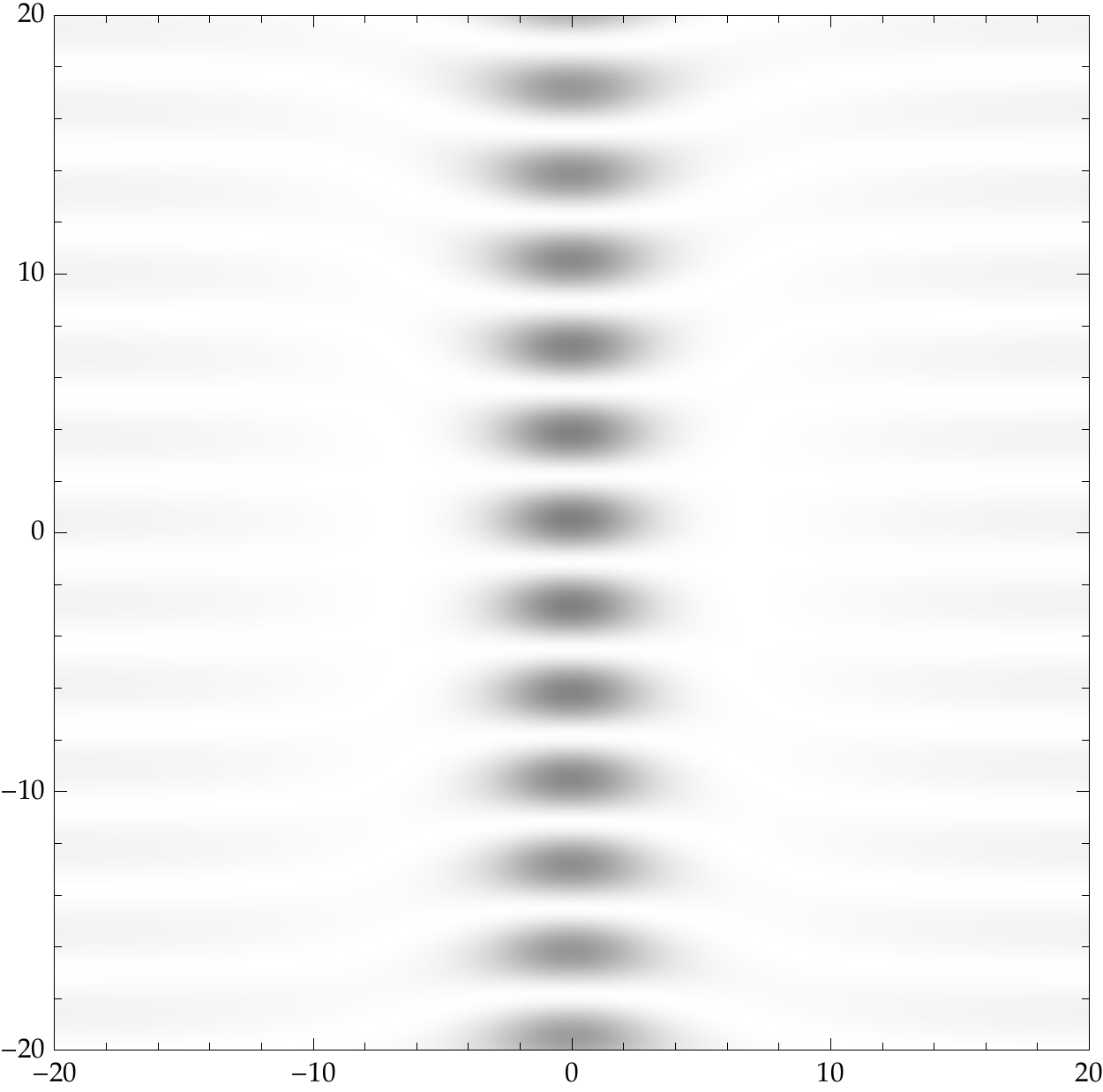}%
\includegraphics[height=.24\linewidth]{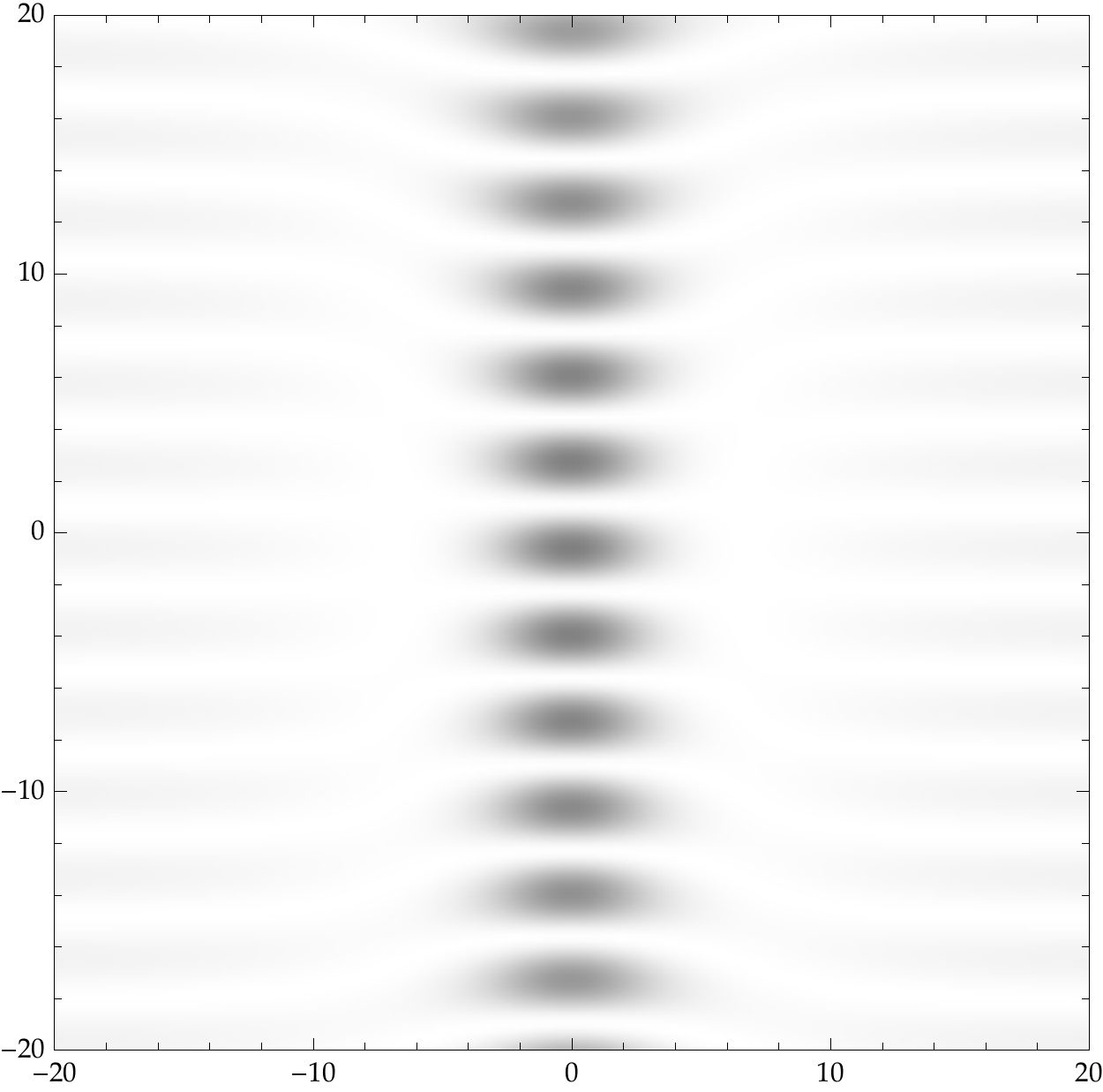}%
\includegraphics[height=.24\linewidth]{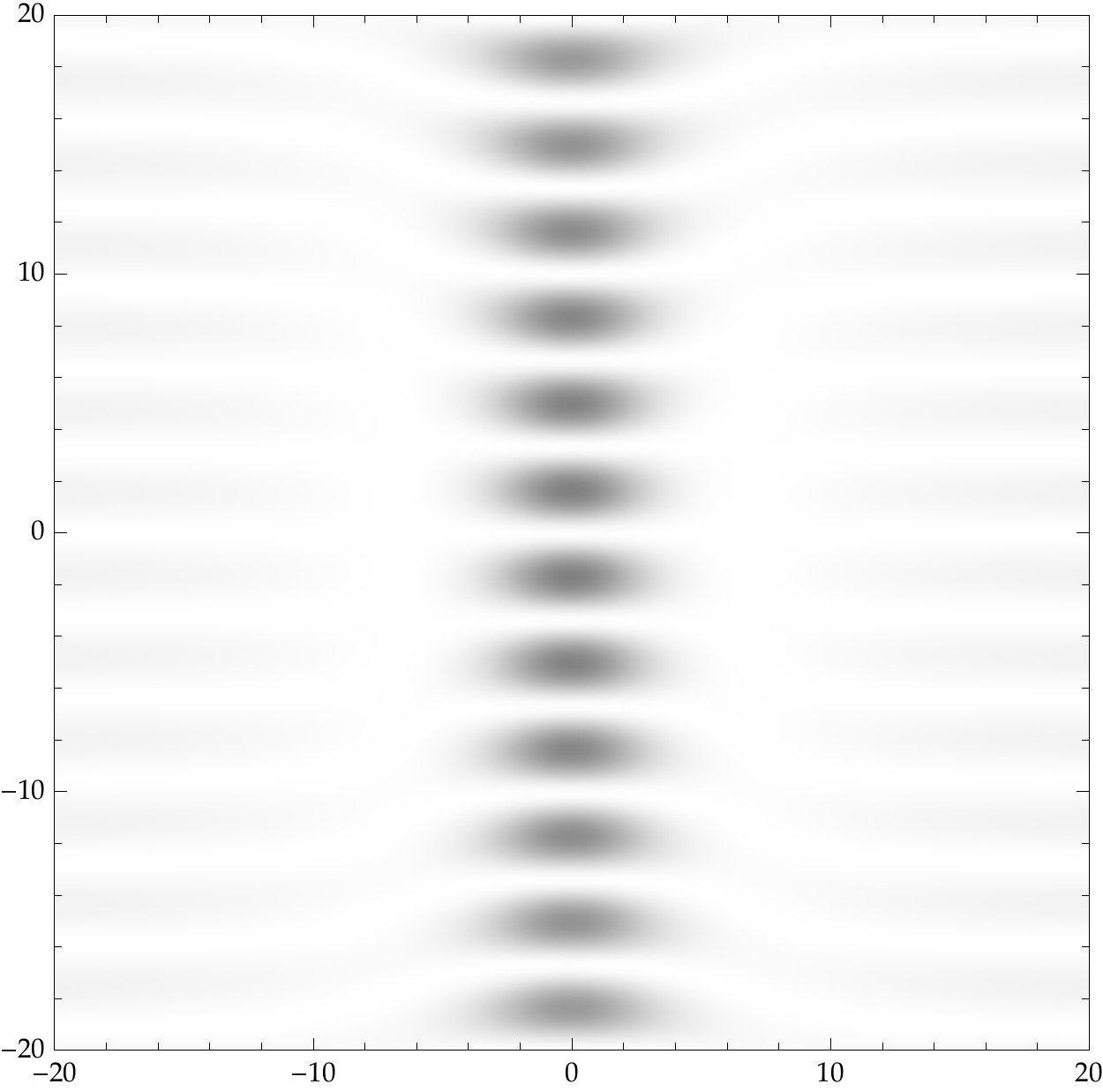}\\
\includegraphics[height=.24\linewidth]{fig/Legend-SMALL.pdf}%
\includegraphics[height=.24\linewidth]{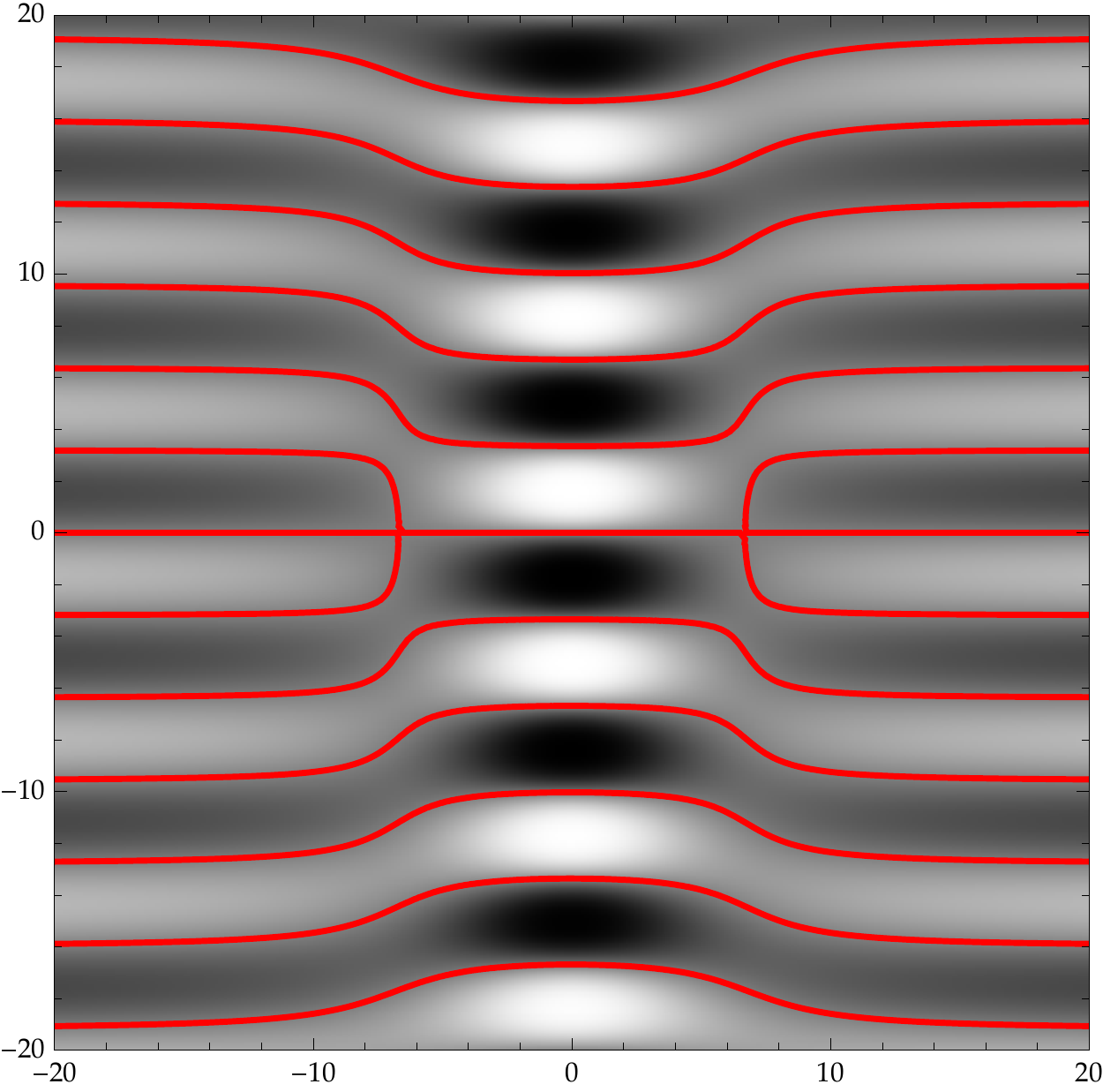}%
\includegraphics[height=.24\linewidth]{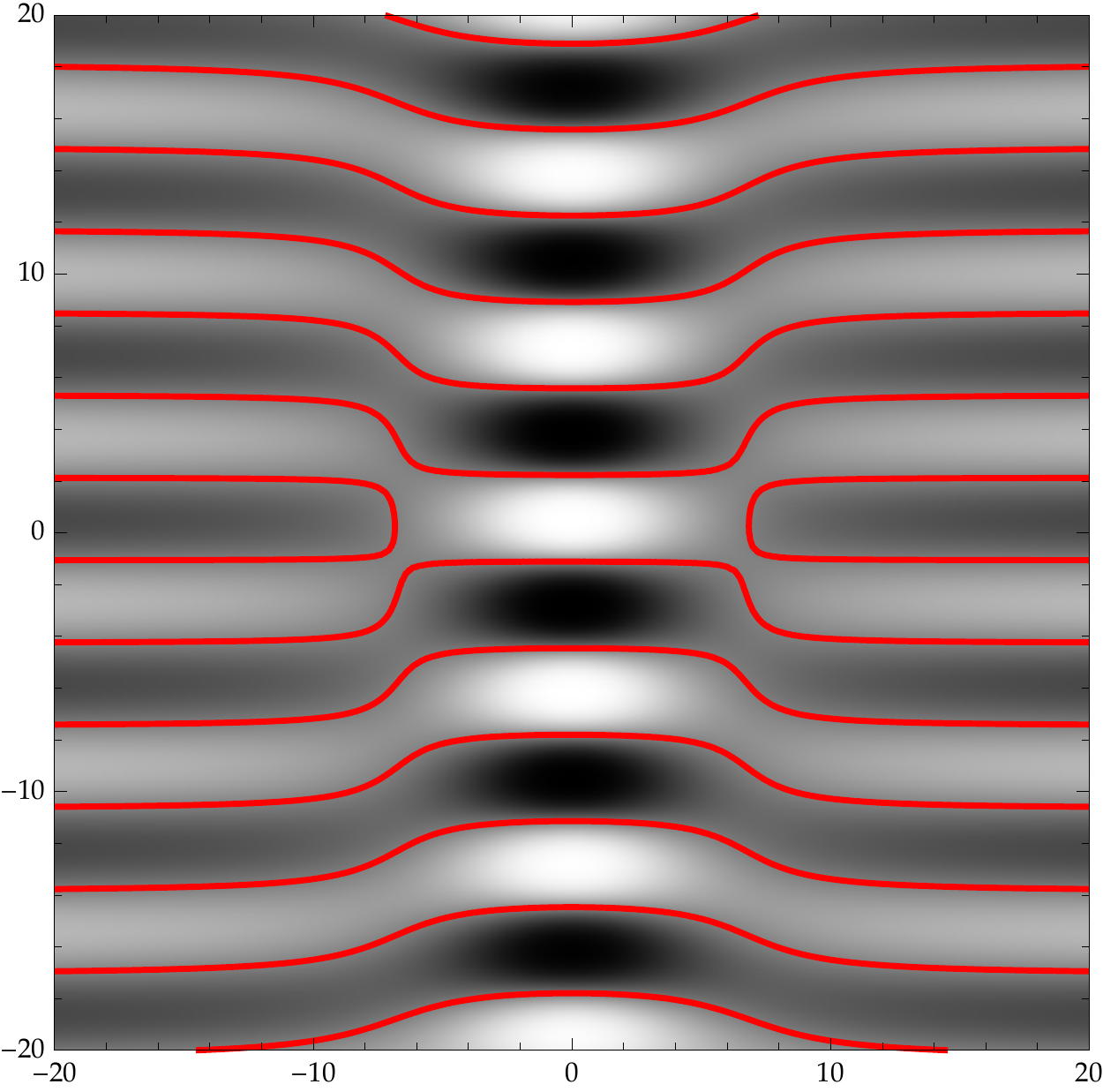}%
\includegraphics[height=.24\linewidth]{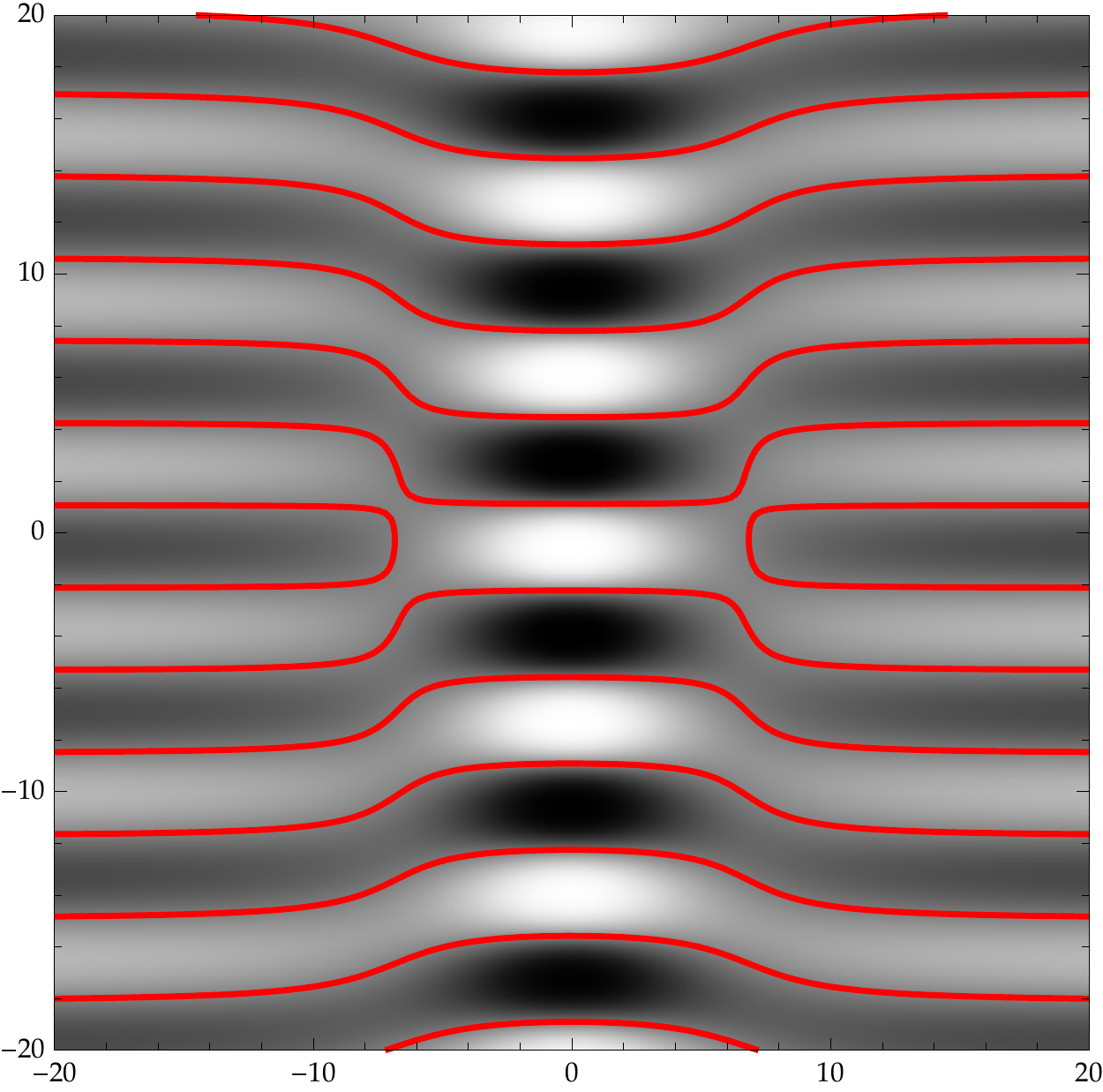}%
\includegraphics[height=.24\linewidth]{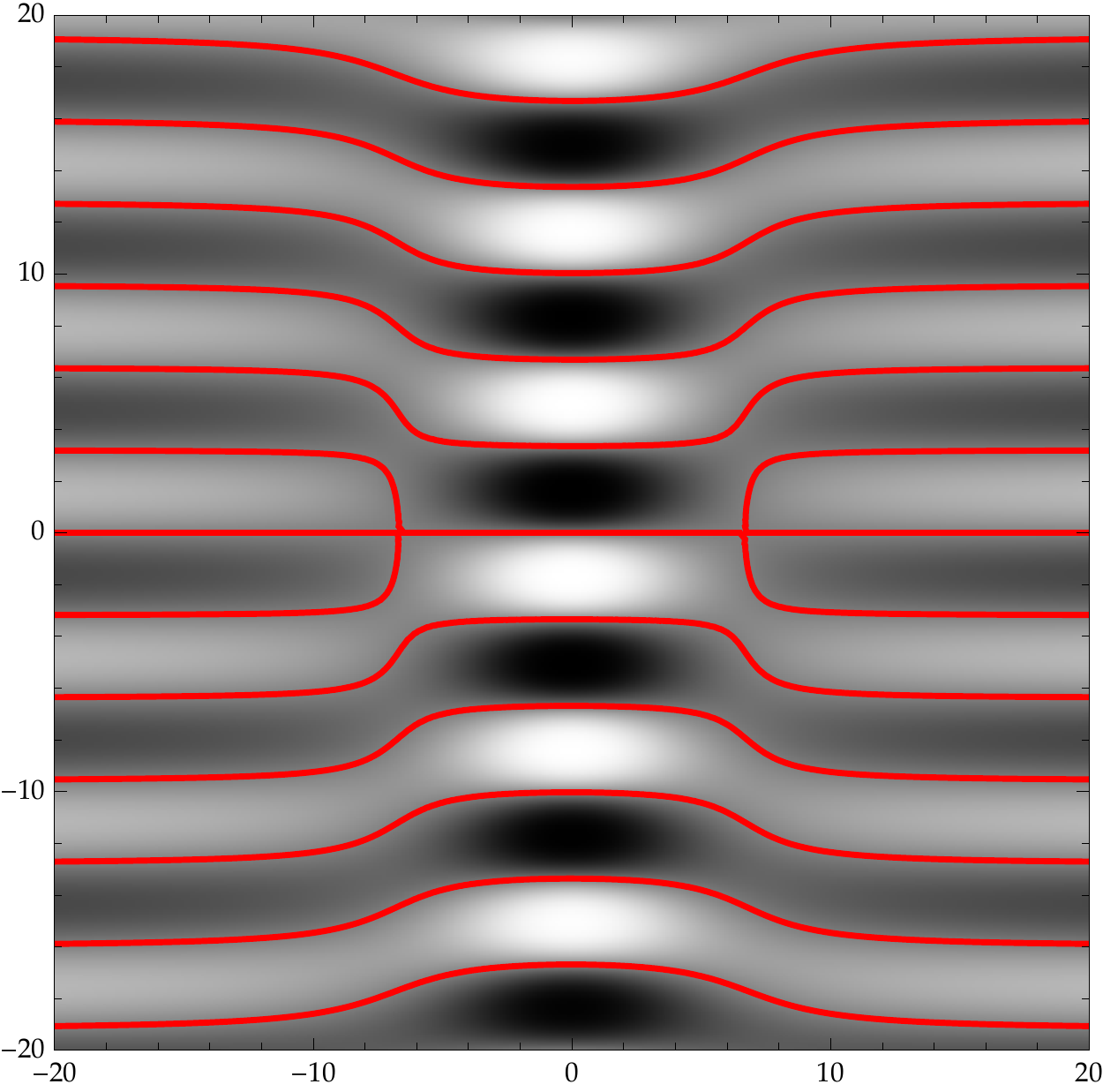}
\end{center}
\caption{As in Figure~\ref{fig:exact-solutions-first} but for $m=\sin^2(\tfrac{1}{12}\pi)$.}
\end{figure}
\begin{figure}[h!]
\begin{center}
\includegraphics[height=.24\linewidth]{fig/Legend-SMALL.pdf}%
\includegraphics[height=.24\linewidth]{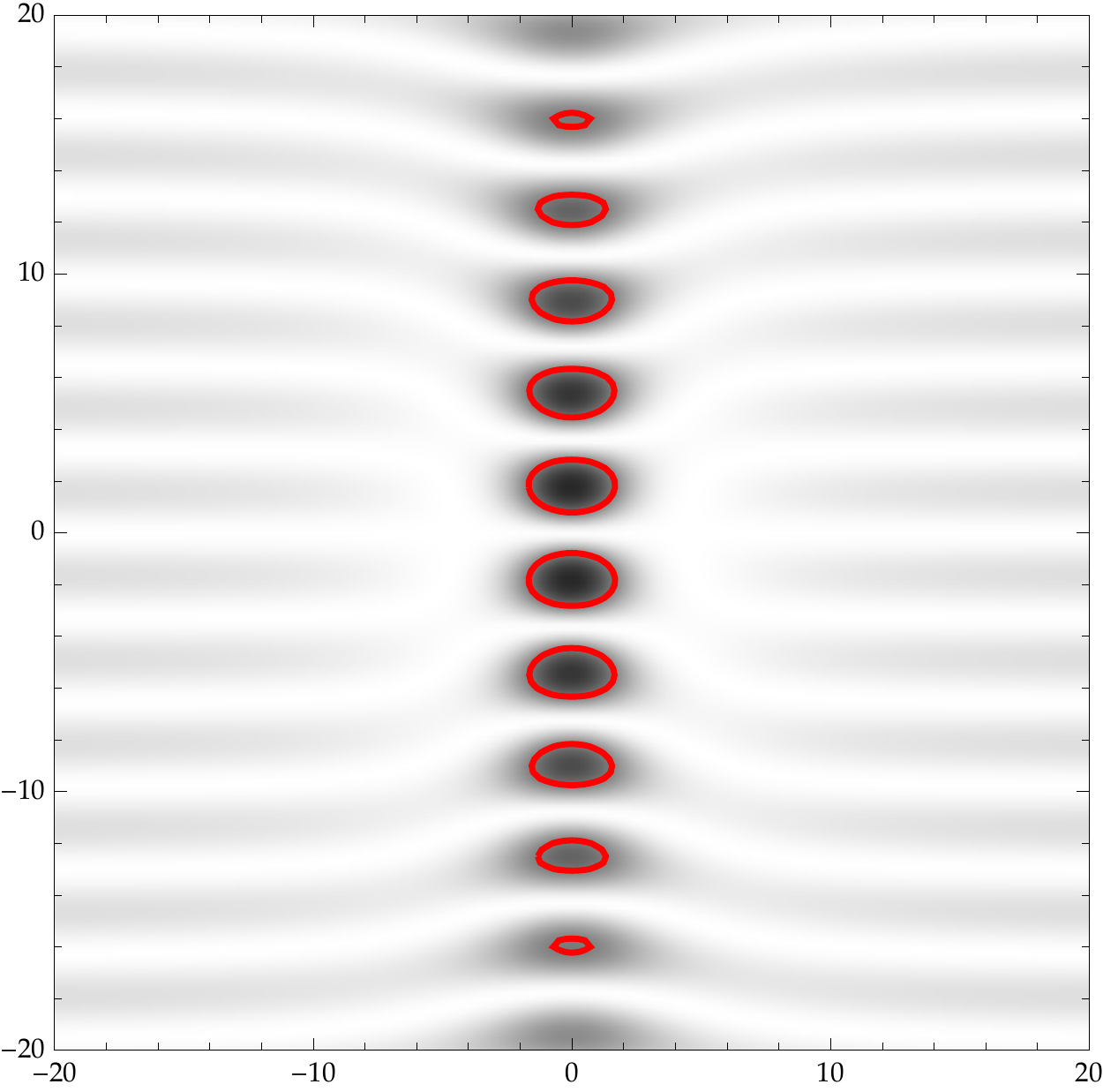}%
\includegraphics[height=.24\linewidth]{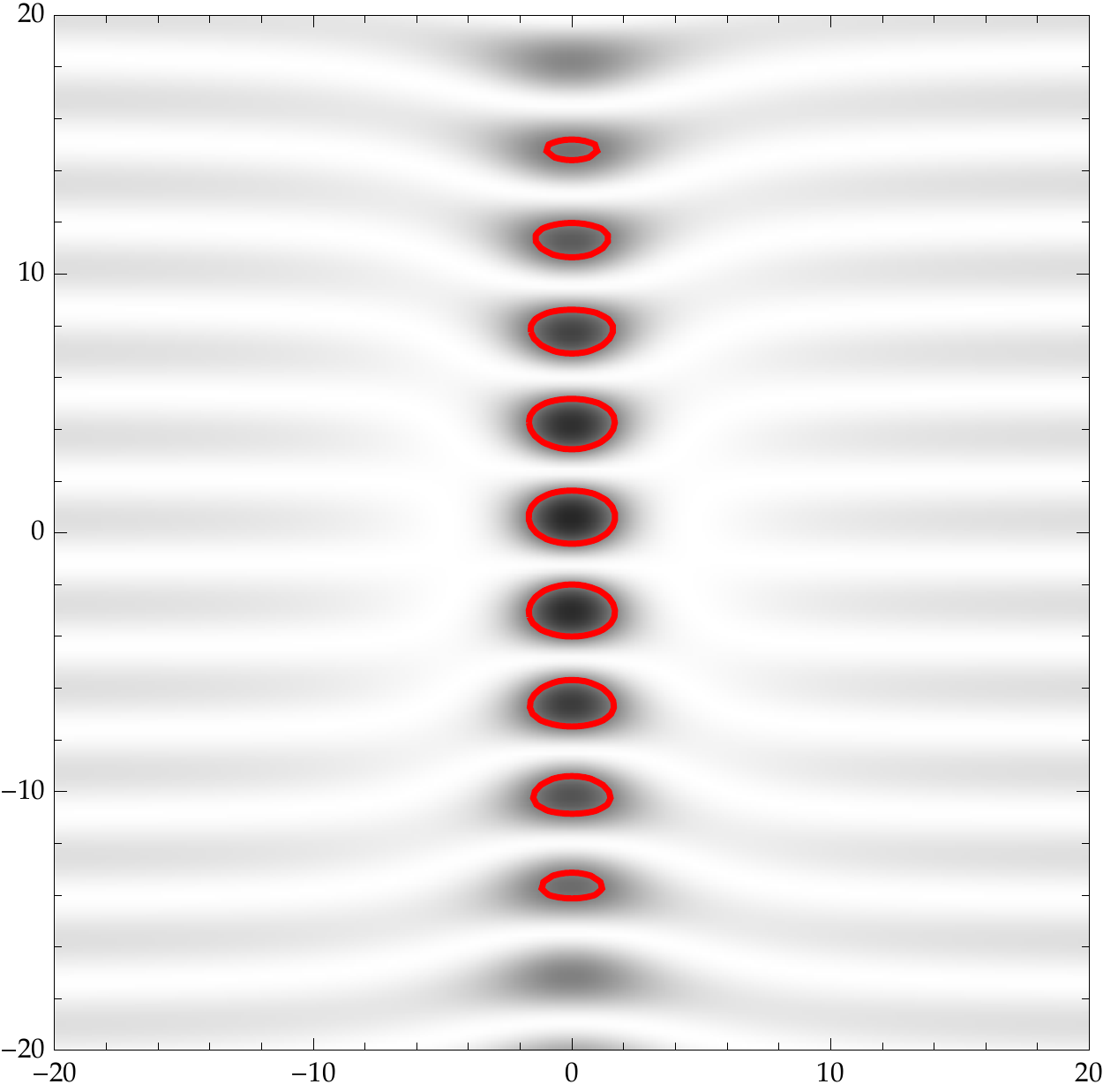}%
\includegraphics[height=.24\linewidth]{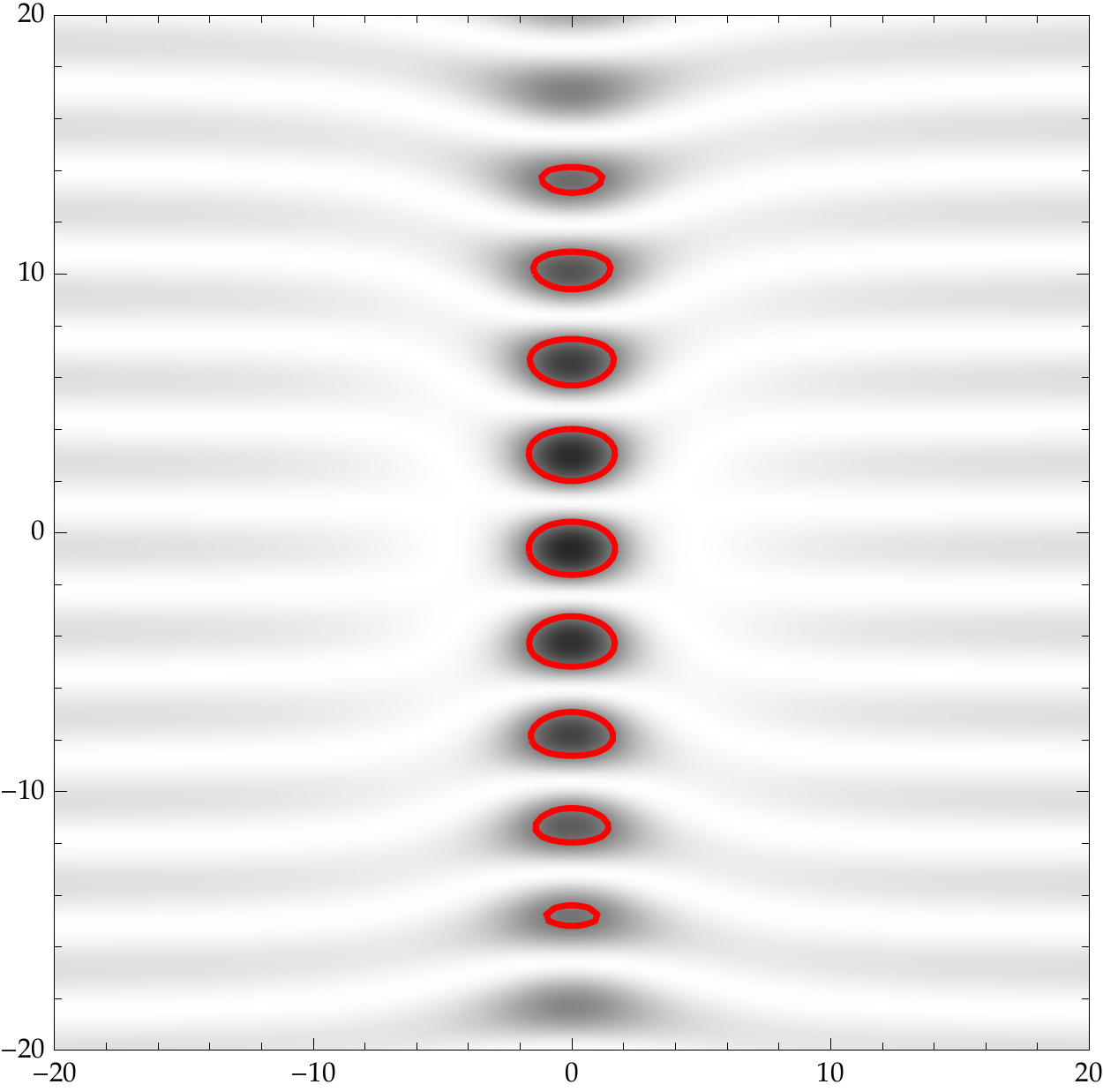}%
\includegraphics[height=.24\linewidth]{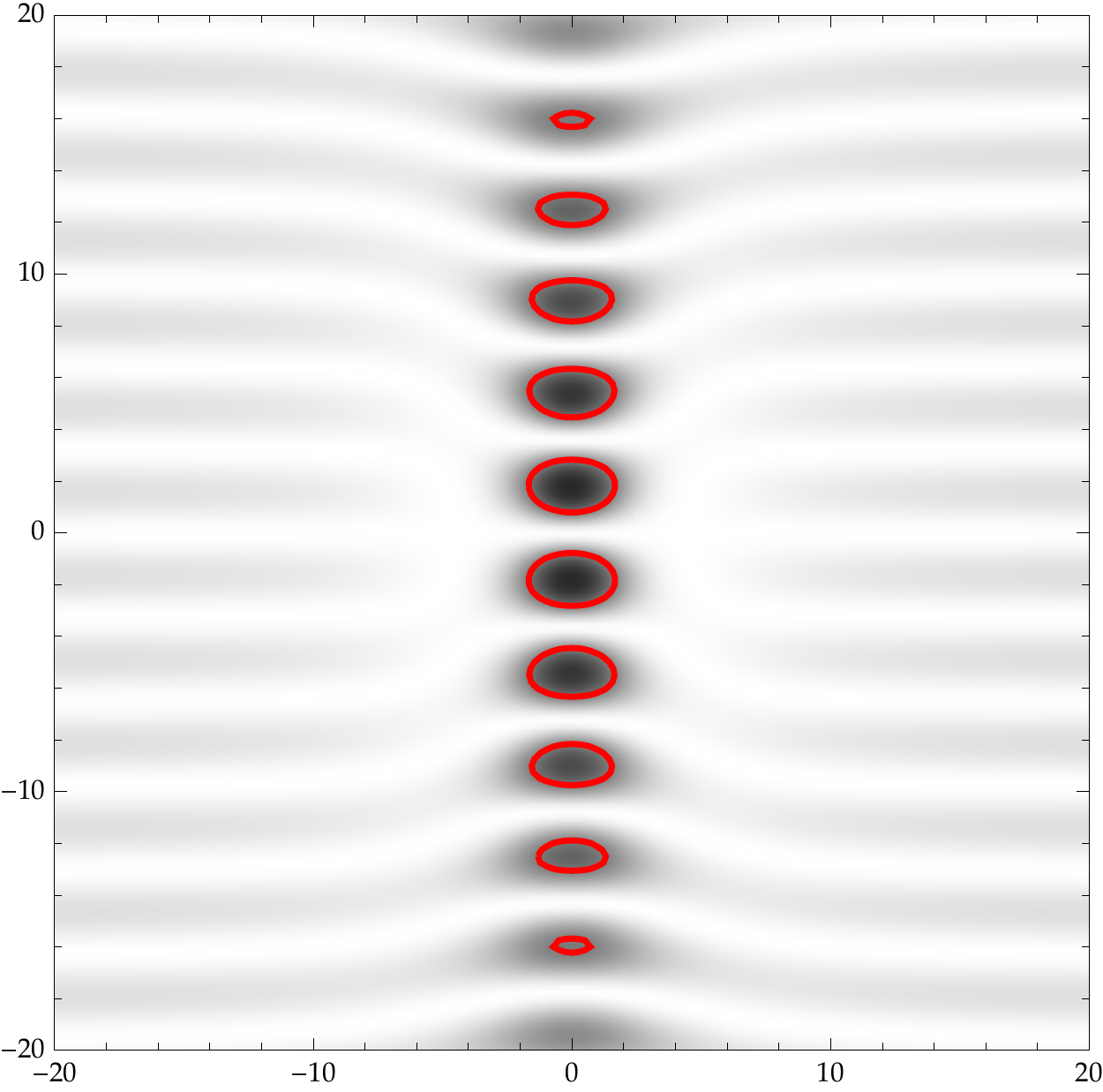}\\
\includegraphics[height=.24\linewidth]{fig/Legend-SMALL.pdf}%
\includegraphics[height=.24\linewidth]{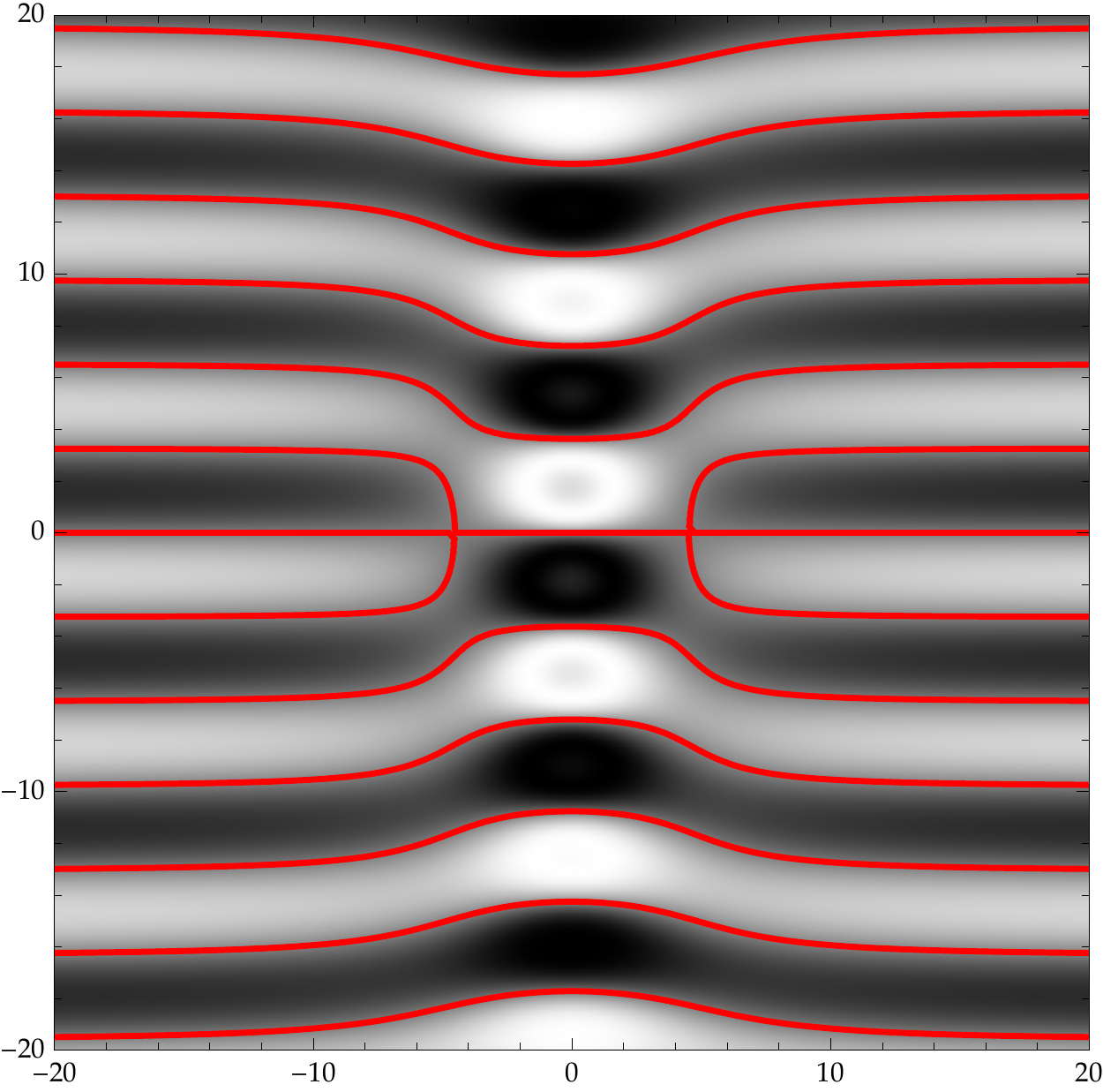}%
\includegraphics[height=.24\linewidth]{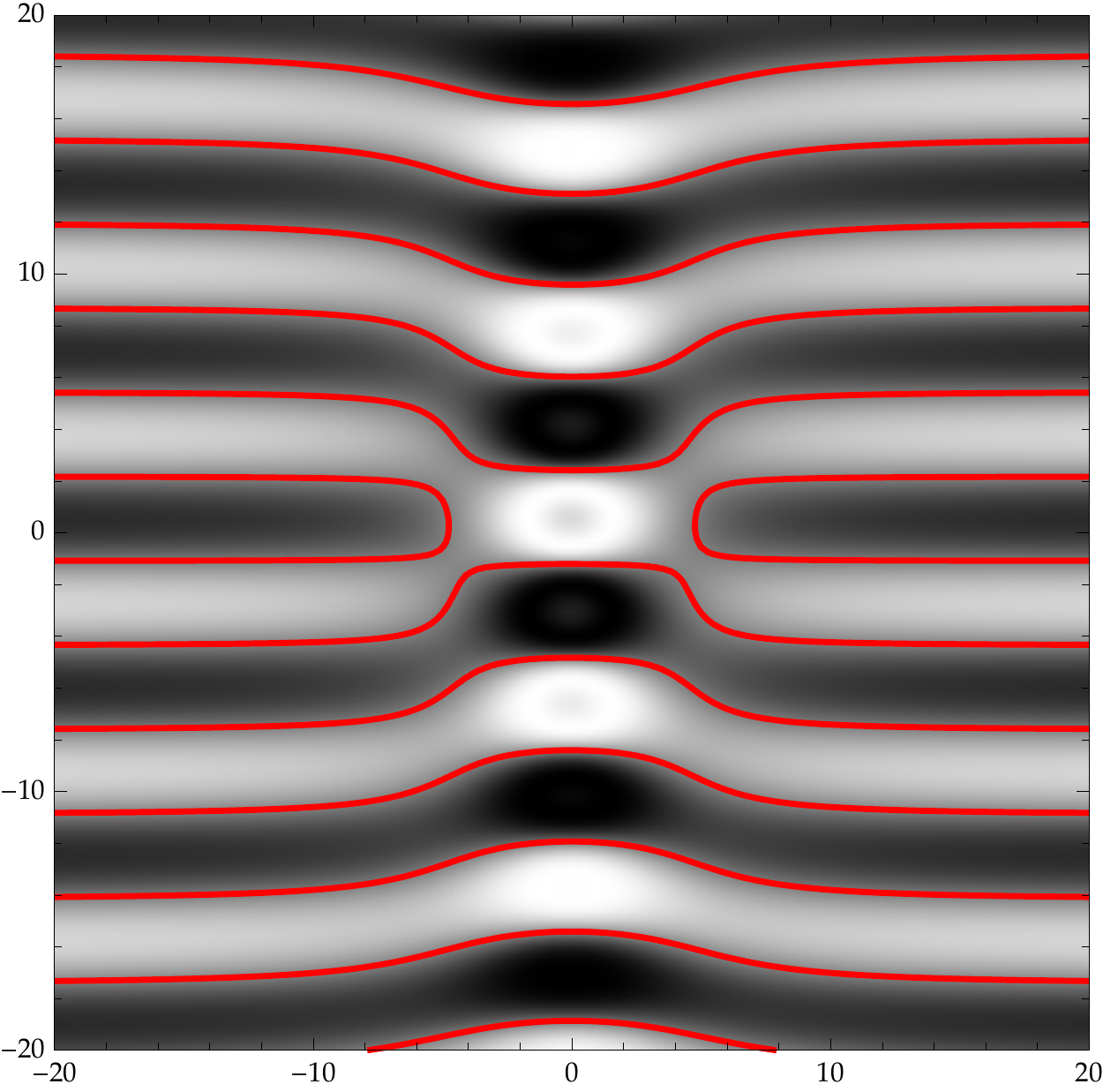}%
\includegraphics[height=.24\linewidth]{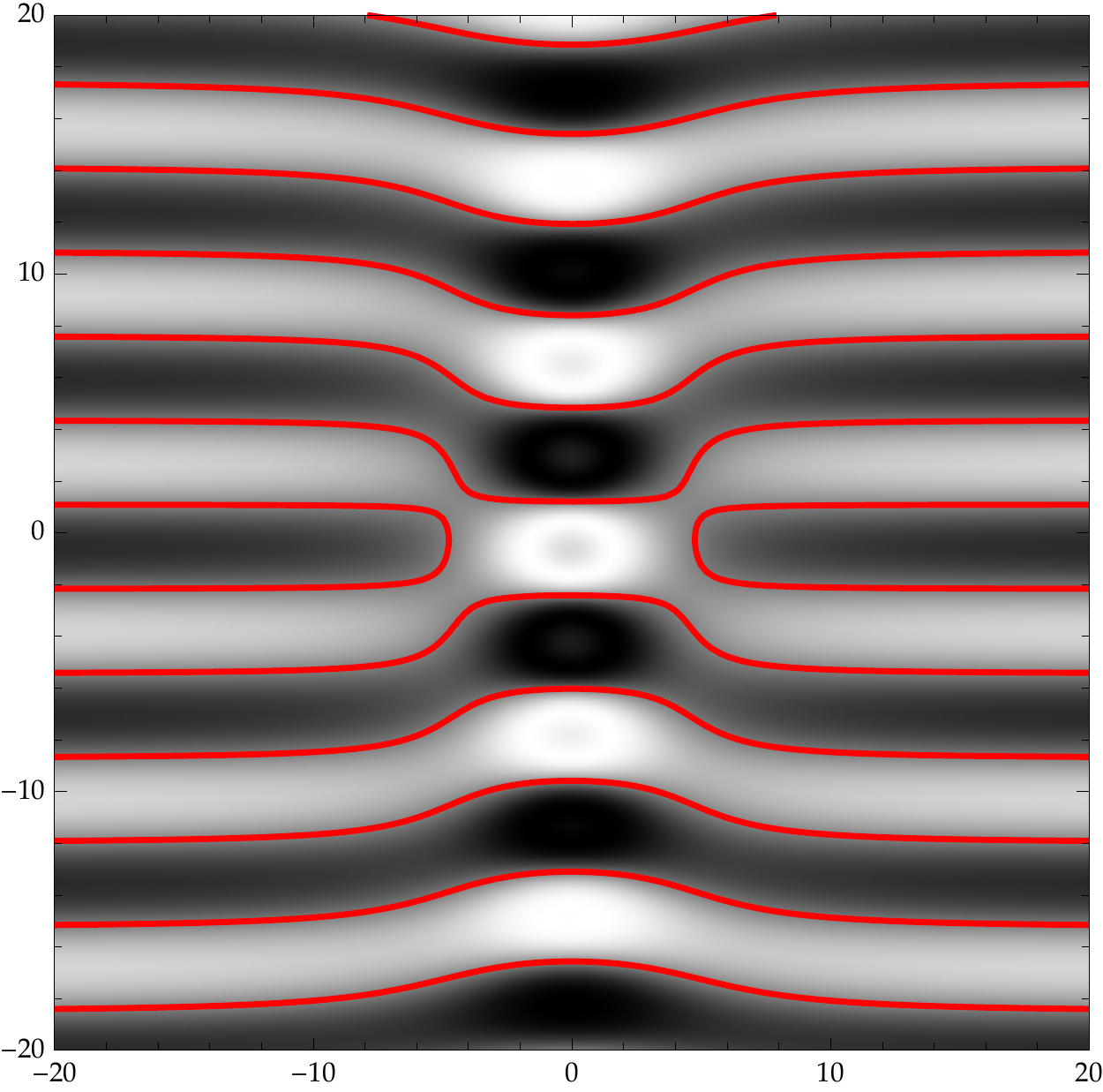}%
\includegraphics[height=.24\linewidth]{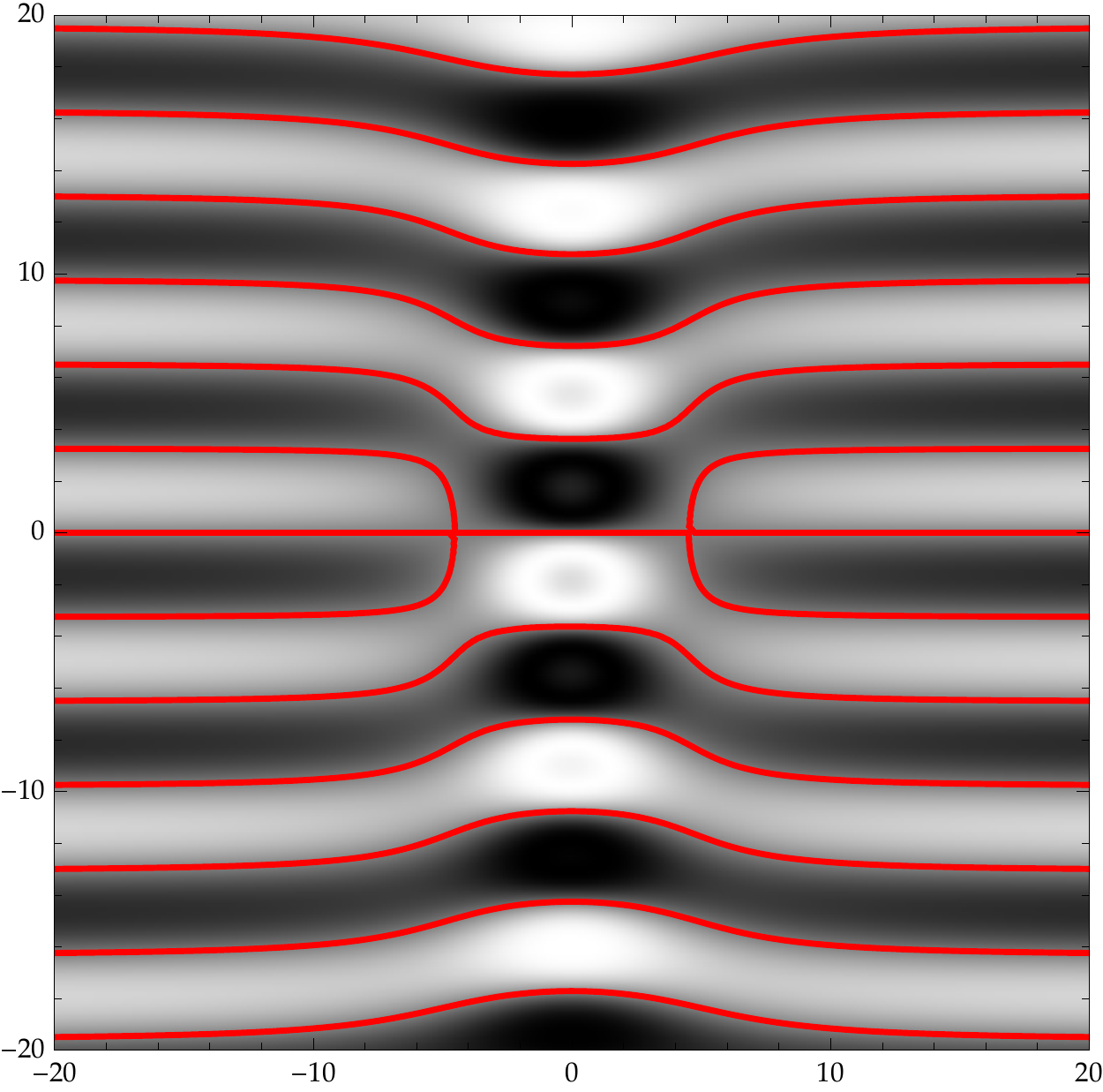}
\end{center}
\caption{As in Figure~\ref{fig:exact-solutions-first} but for $m=\sin^2(\tfrac{1}{8}\pi)$.}
\end{figure}
\begin{figure}[h!]
\begin{center}
\includegraphics[height=.24\linewidth]{fig/Legend-SMALL.pdf}%
\includegraphics[height=.24\linewidth]{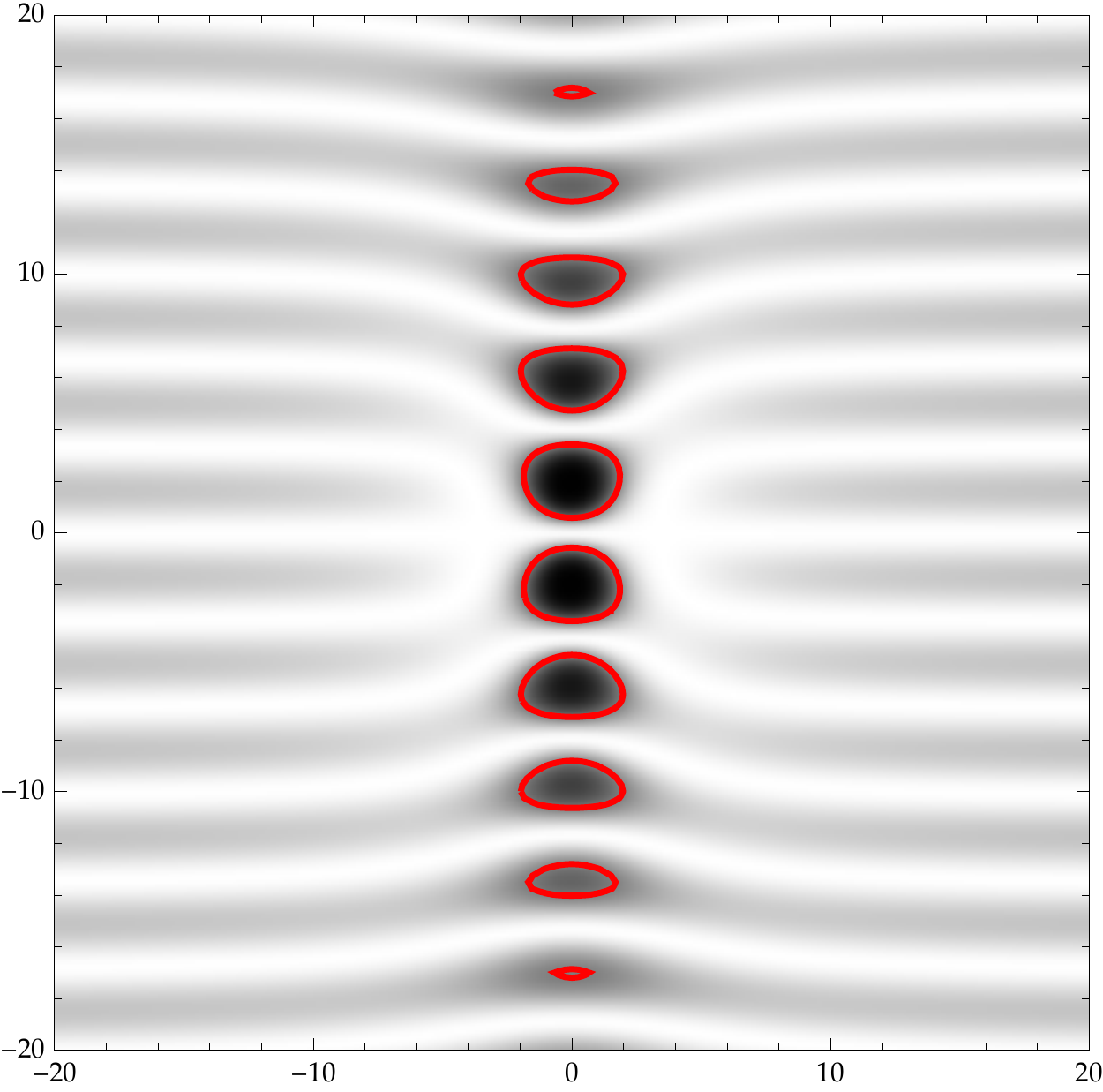}%
\includegraphics[height=.24\linewidth]{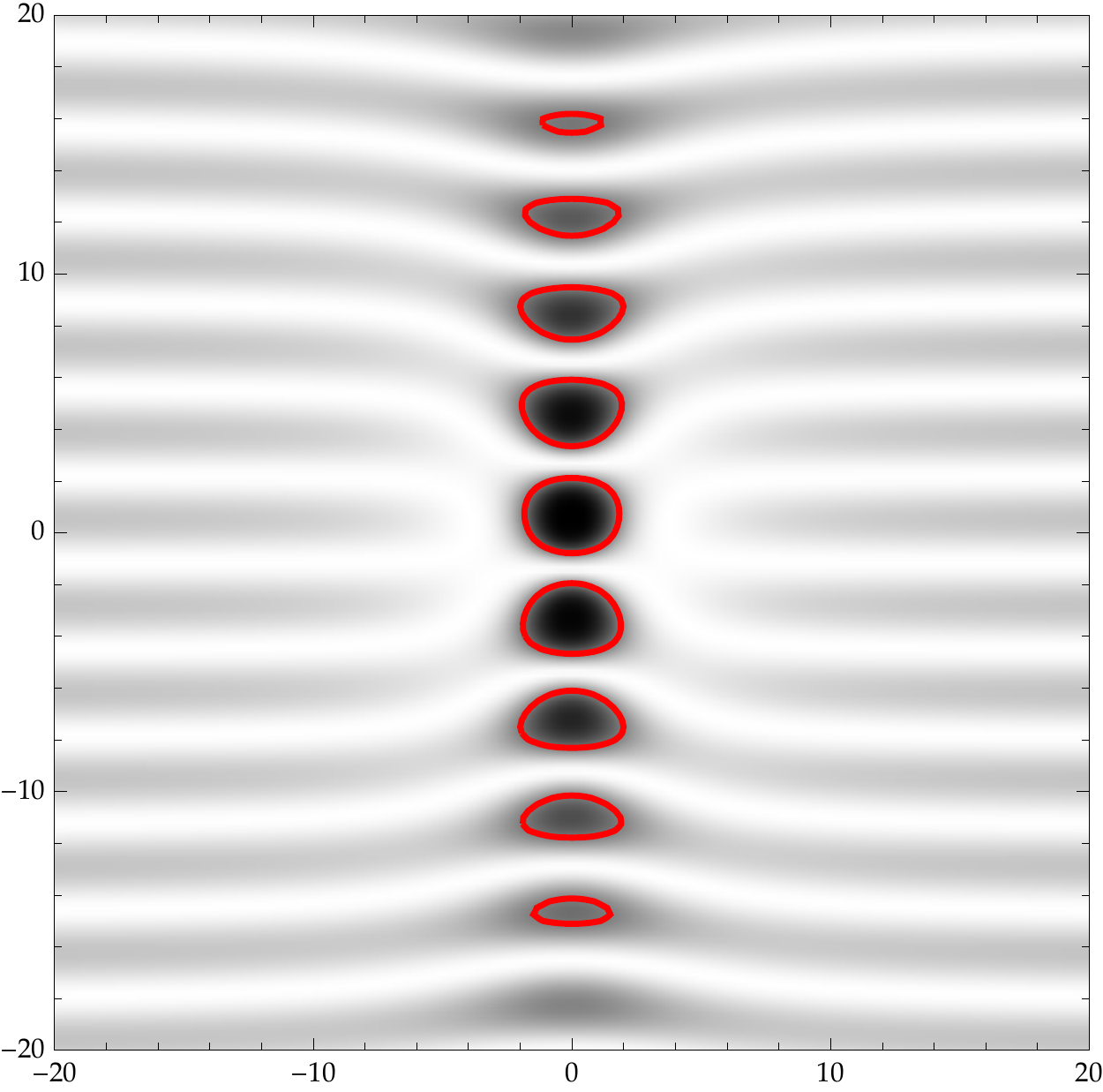}%
\includegraphics[height=.24\linewidth]{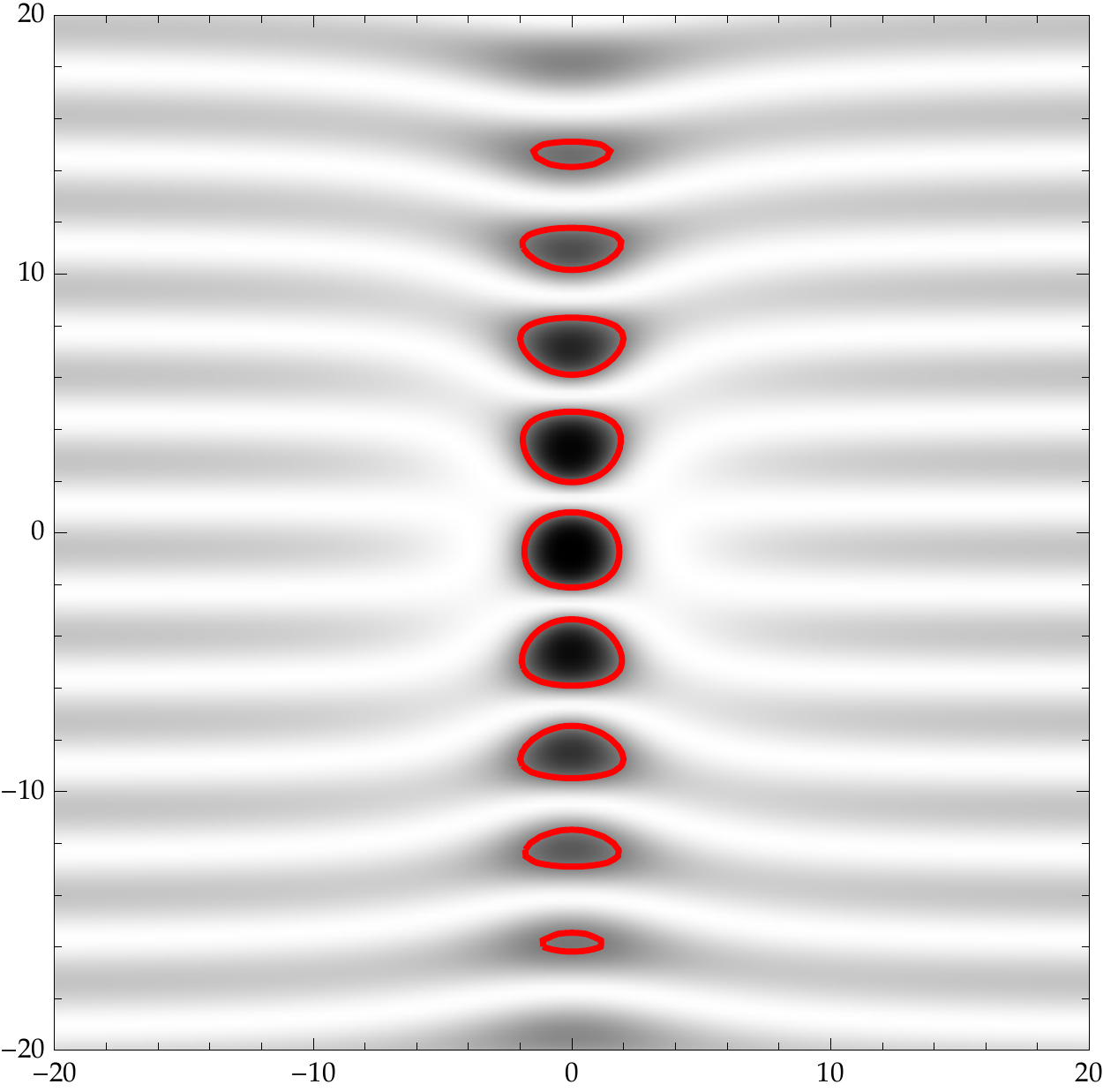}%
\includegraphics[height=.24\linewidth]{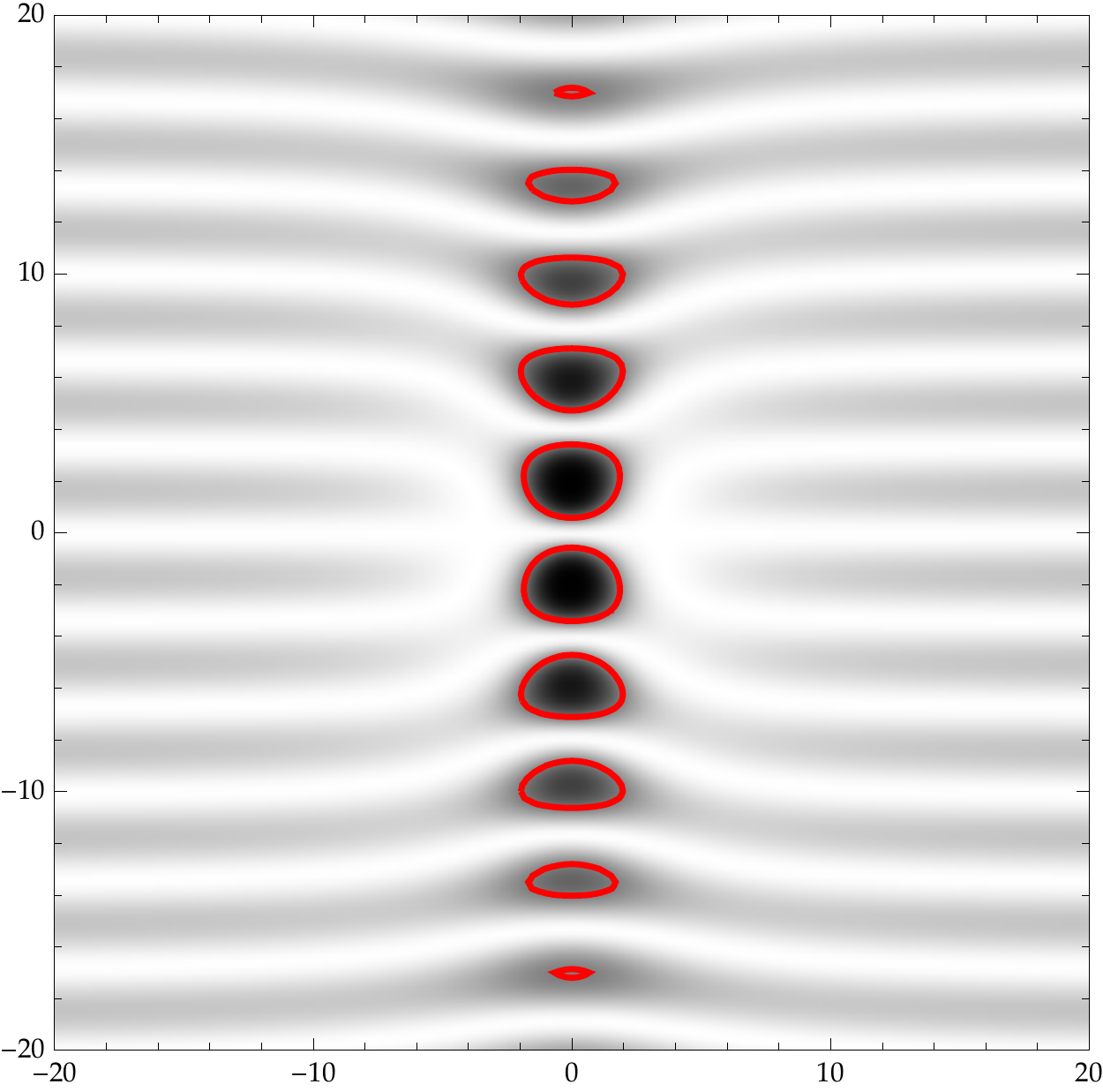}\\
\includegraphics[height=.24\linewidth]{fig/Legend-SMALL.pdf}%
\includegraphics[height=.24\linewidth]{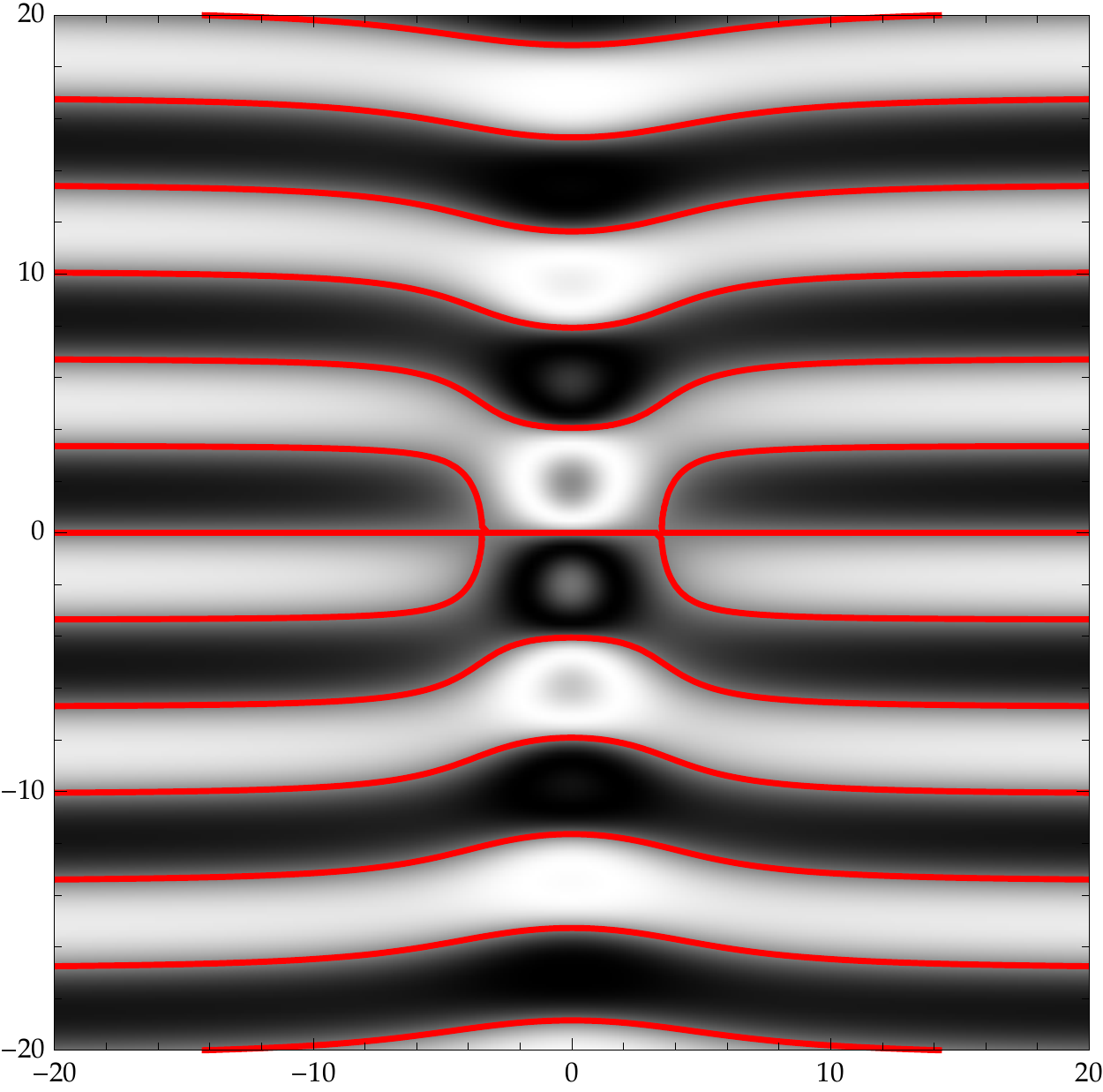}%
\includegraphics[height=.24\linewidth]{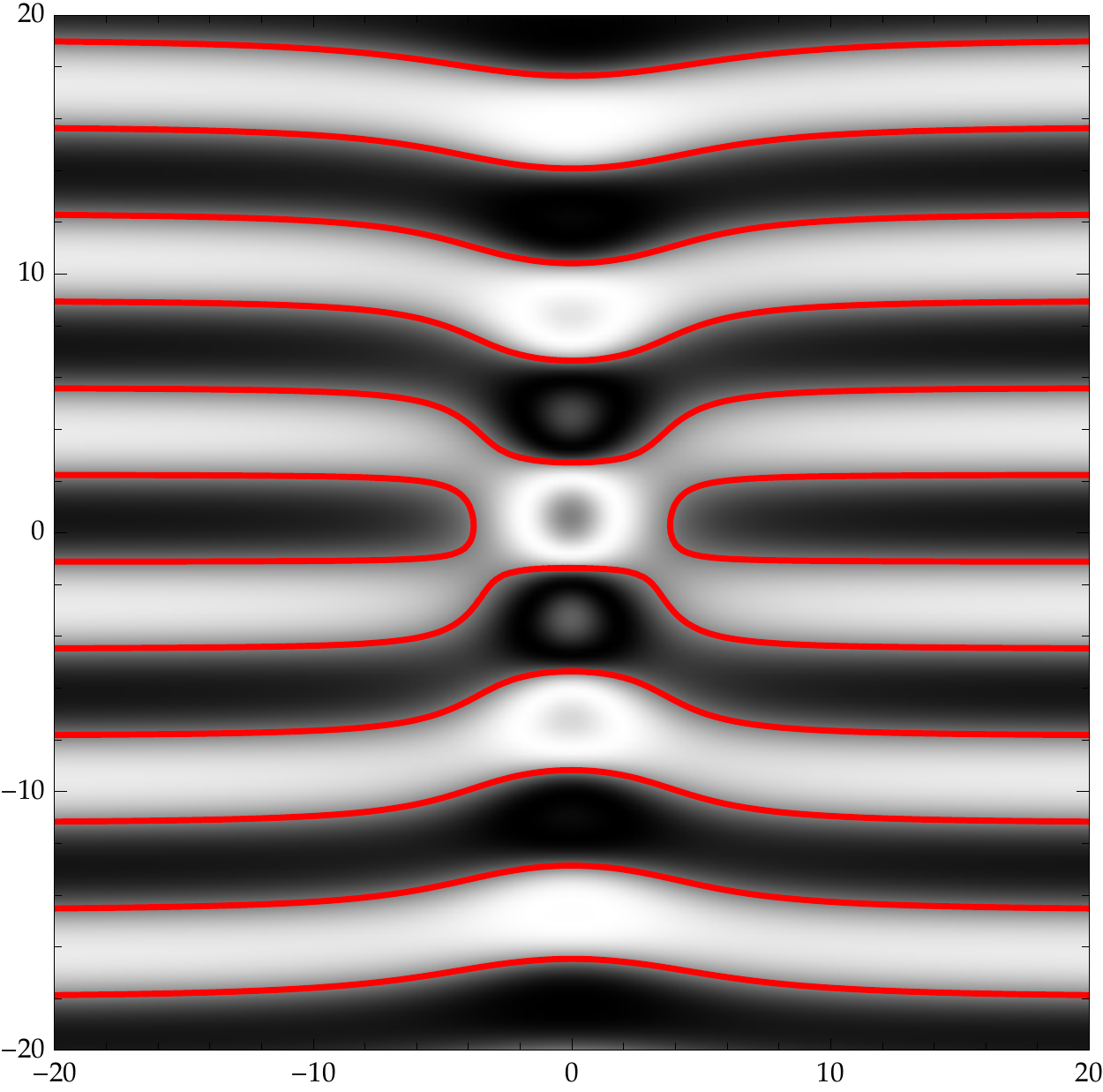}%
\includegraphics[height=.24\linewidth]{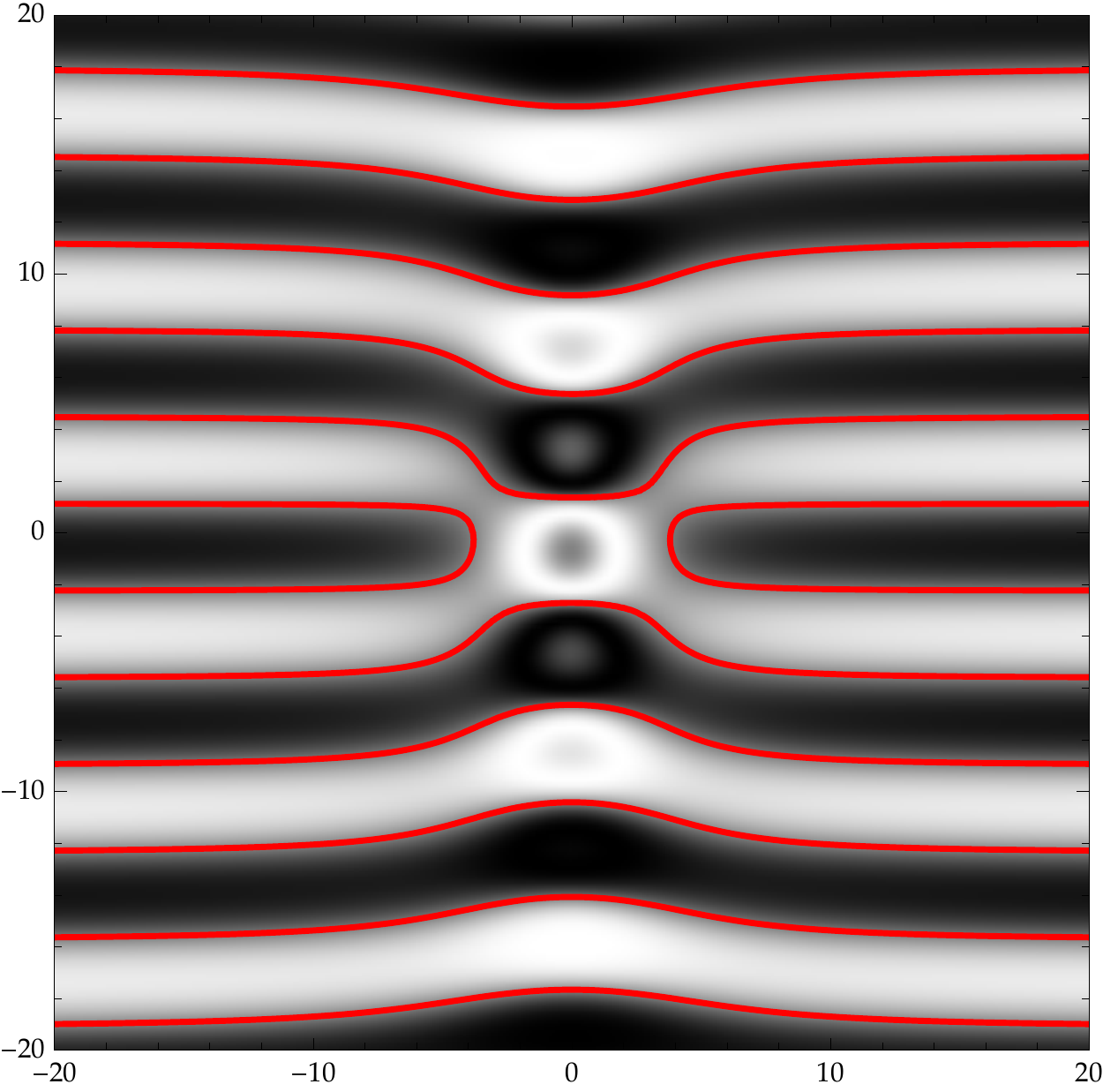}%
\includegraphics[height=.24\linewidth]{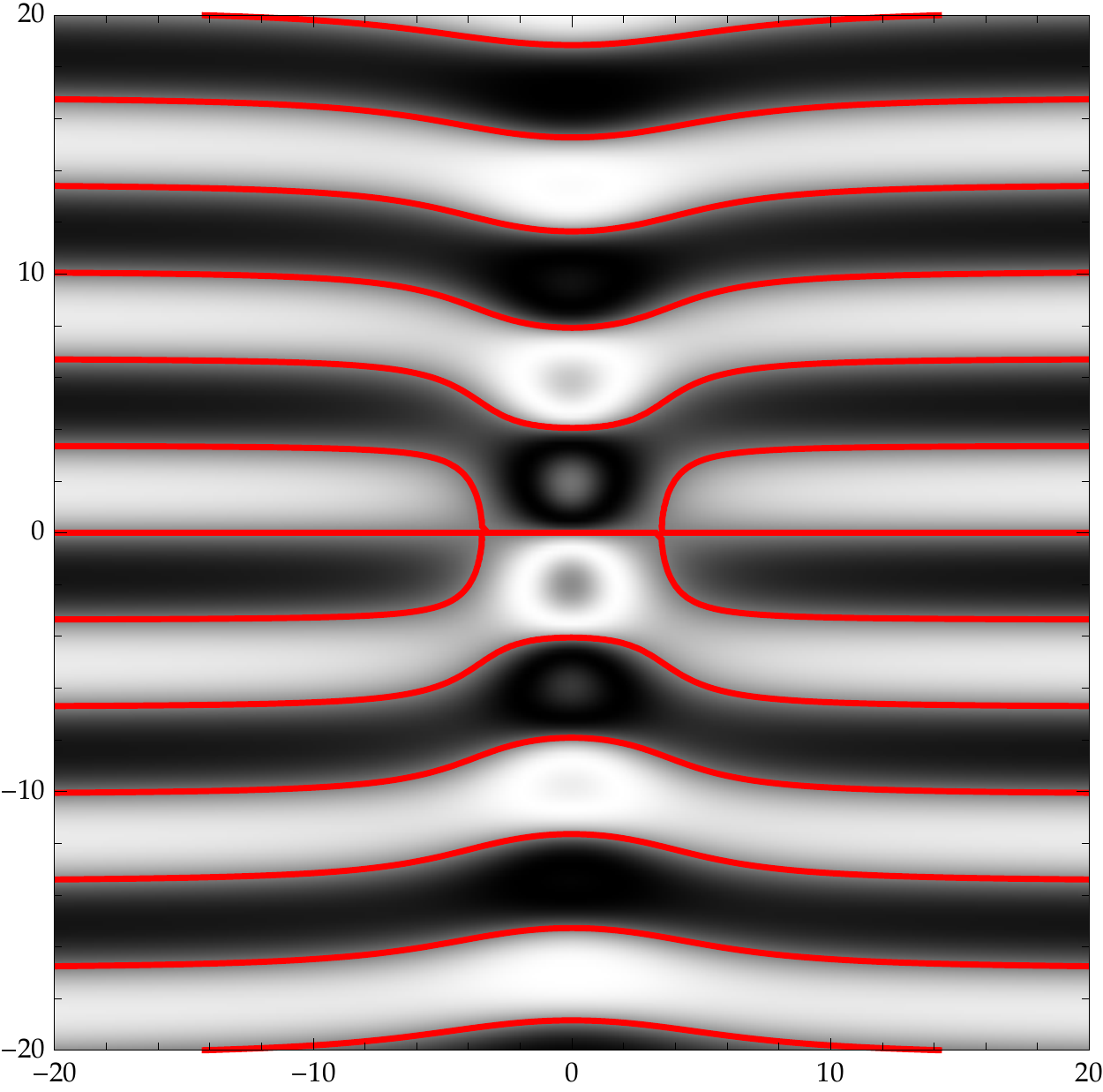}
\end{center}
\caption{As in Figure~\ref{fig:exact-solutions-first} but for $m=\sin^2(\tfrac{1}{6}\pi)$.}
\end{figure}
\begin{figure}[h!]
\begin{center}
\includegraphics[height=.24\linewidth]{fig/Legend-SMALL.pdf}%
\includegraphics[height=.24\linewidth]{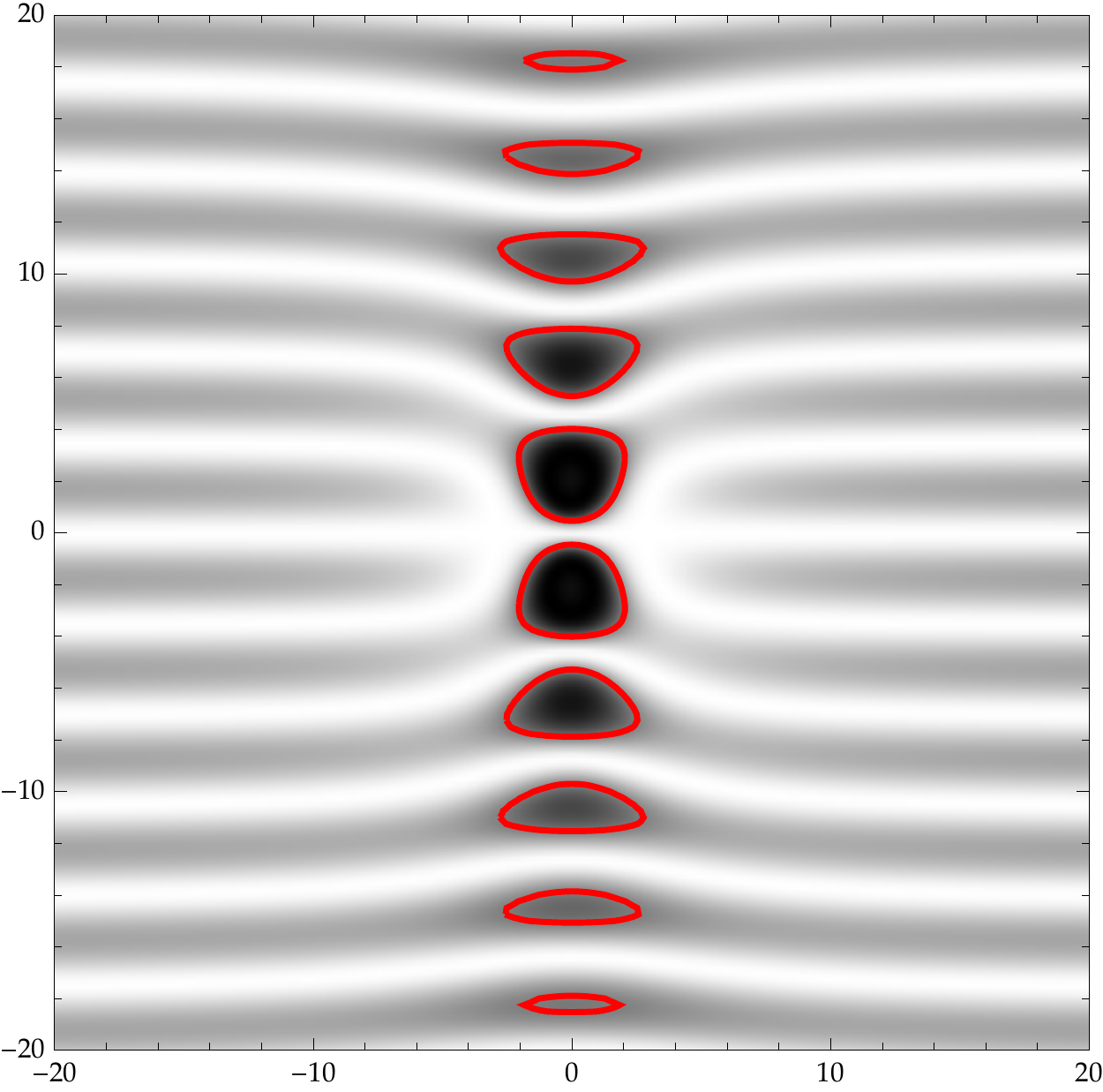}%
\includegraphics[height=.24\linewidth]{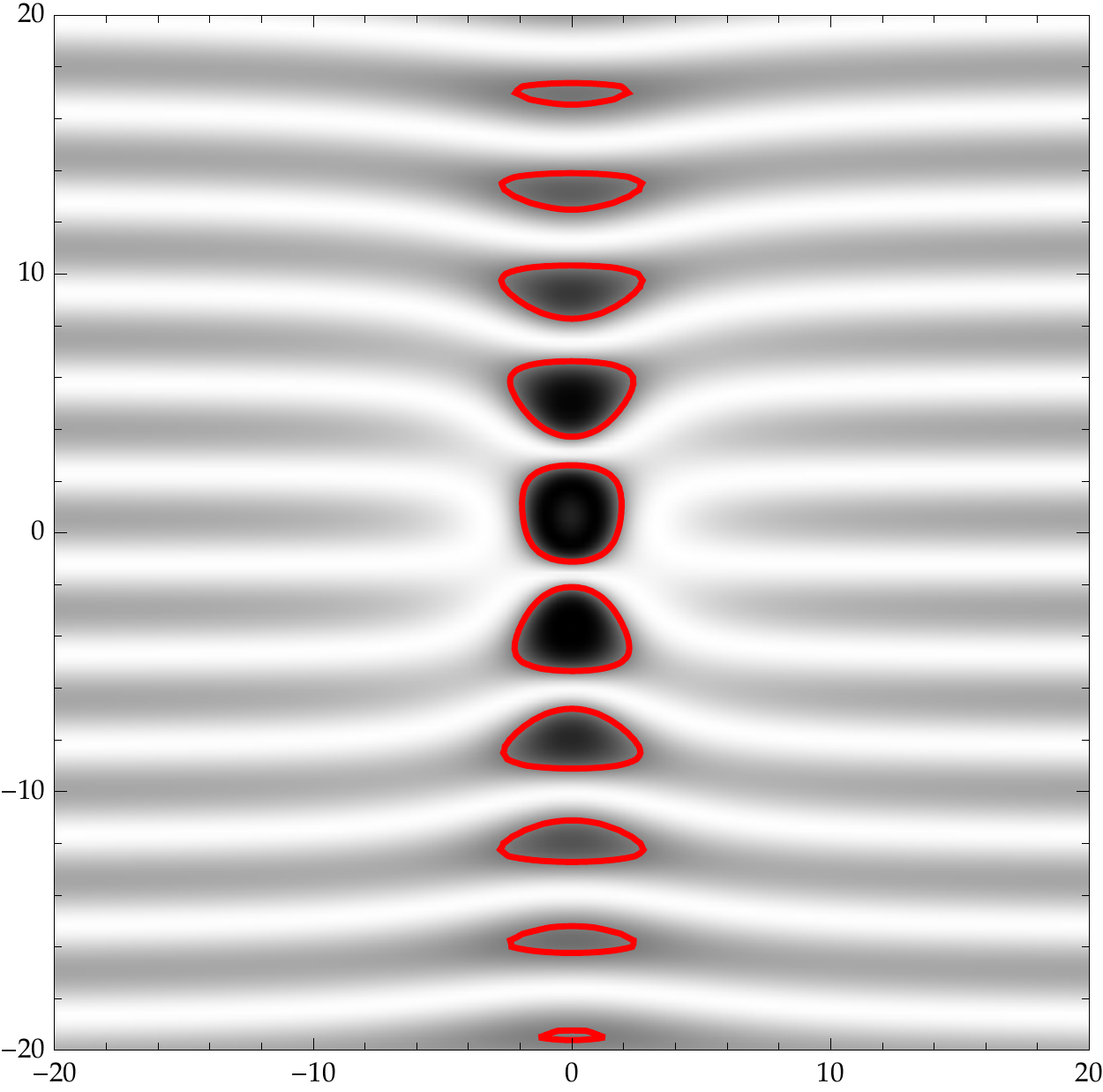}%
\includegraphics[height=.24\linewidth]{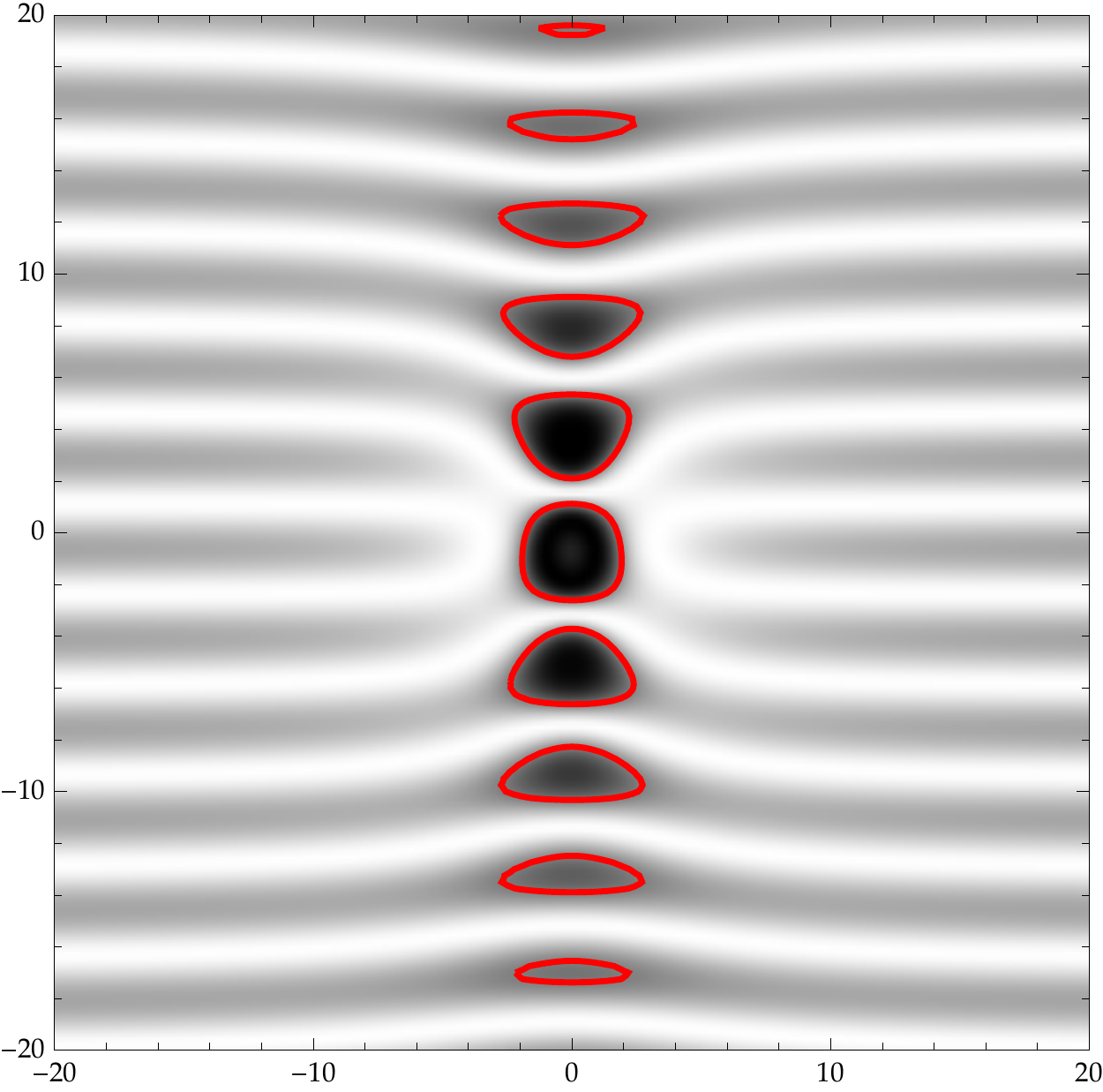}%
\includegraphics[height=.24\linewidth]{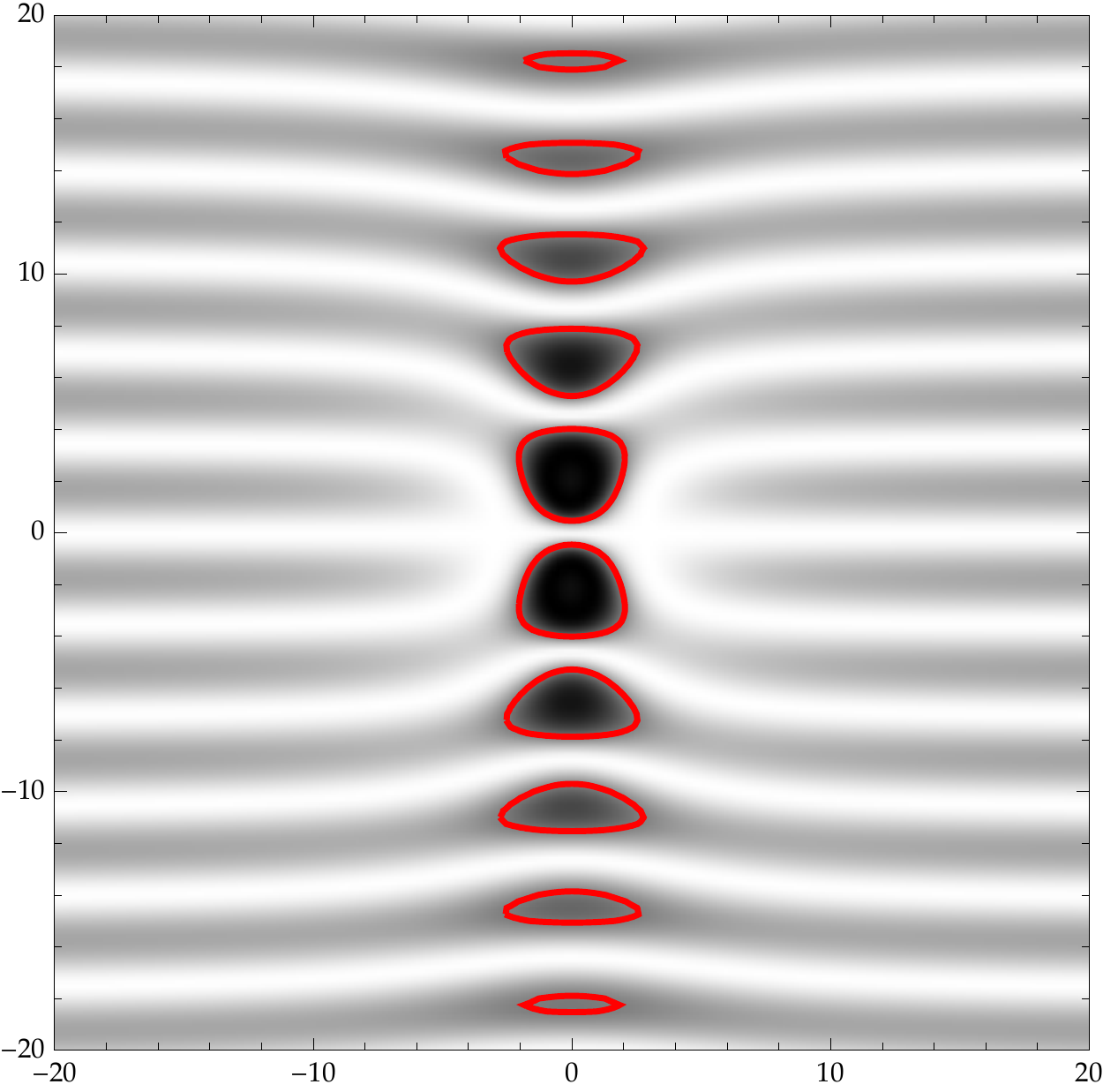}\\
\includegraphics[height=.24\linewidth]{fig/Legend-SMALL.pdf}%
\includegraphics[height=.24\linewidth]{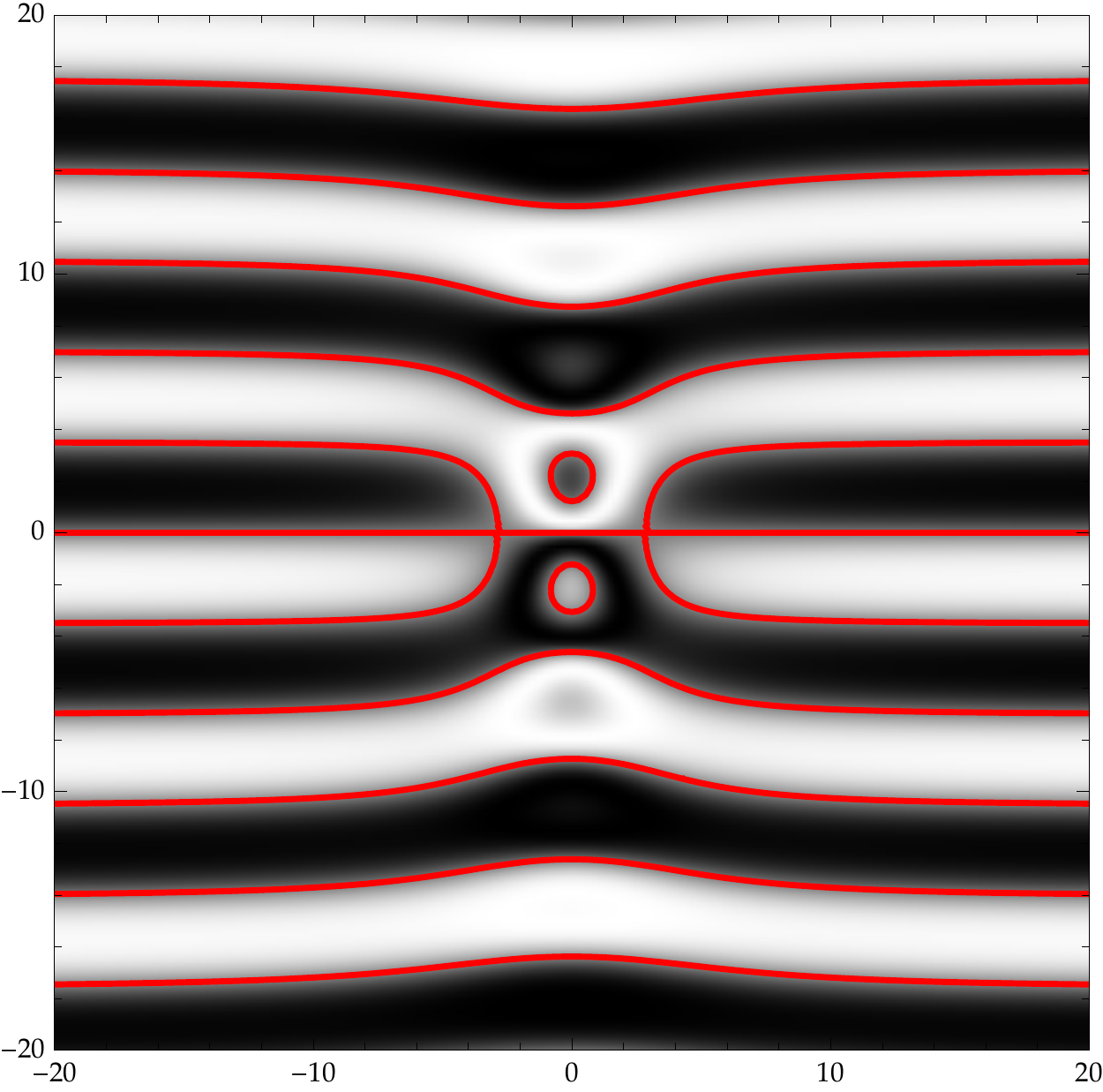}%
\includegraphics[height=.24\linewidth]{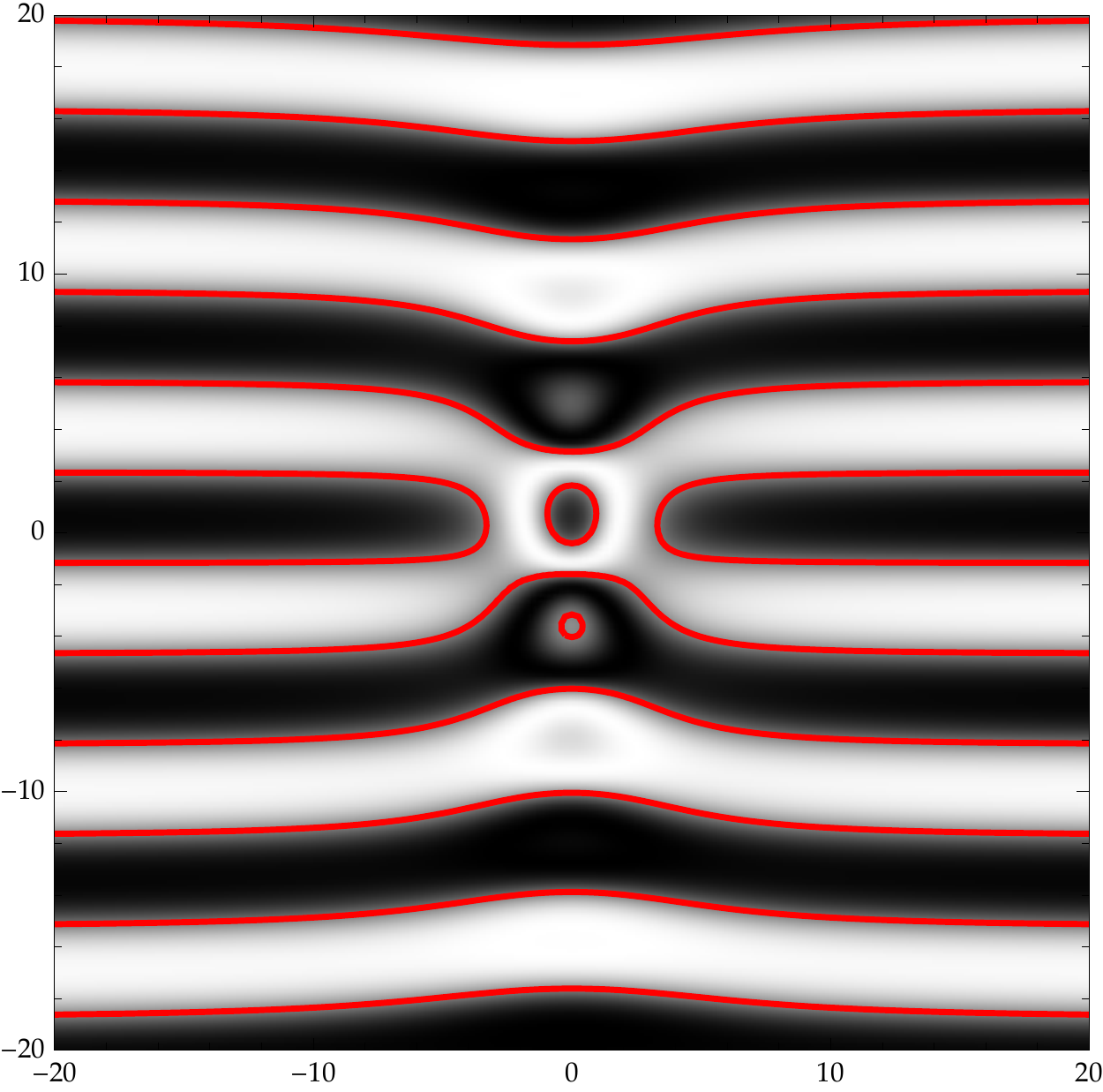}%
\includegraphics[height=.24\linewidth]{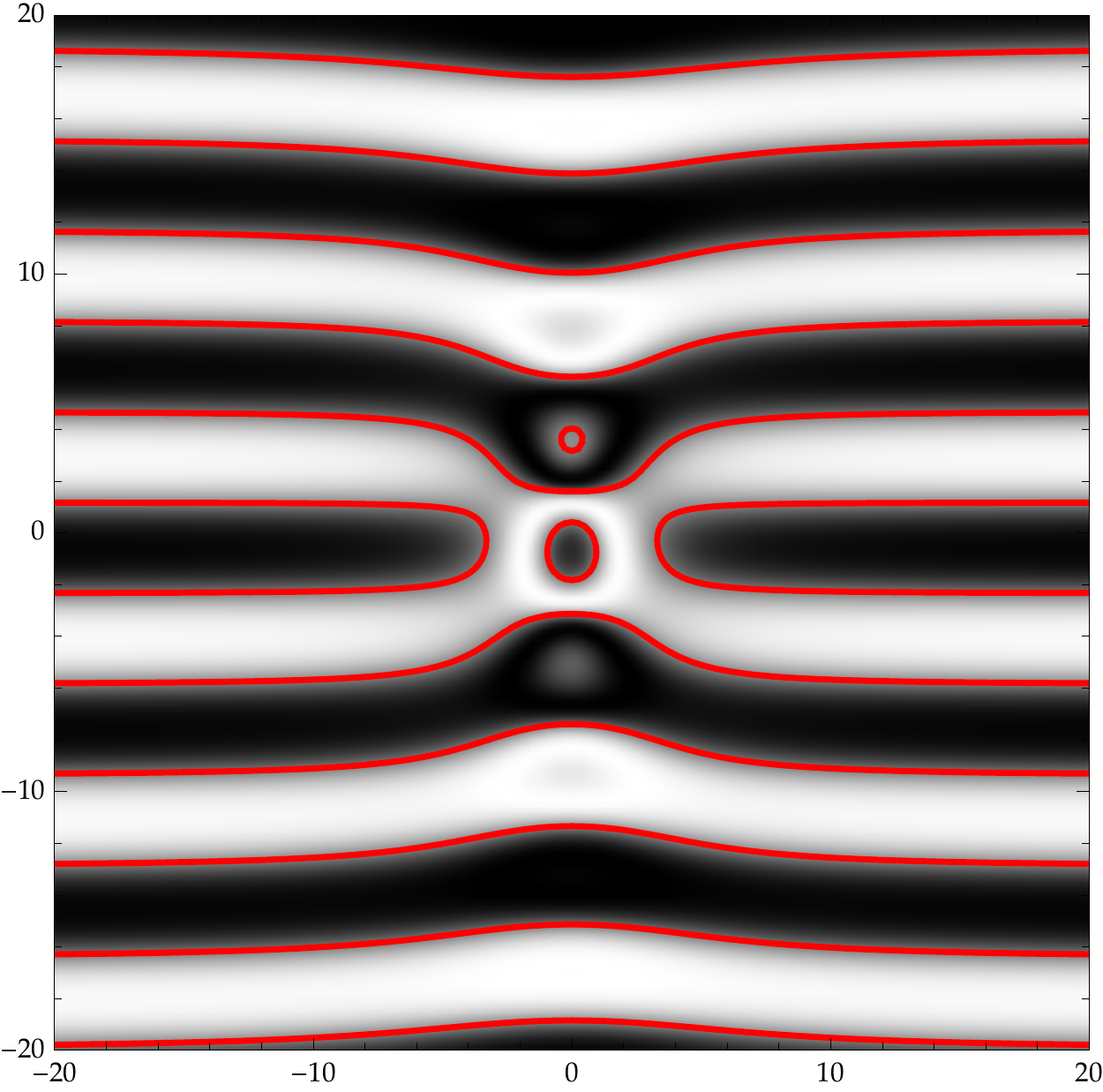}%
\includegraphics[height=.24\linewidth]{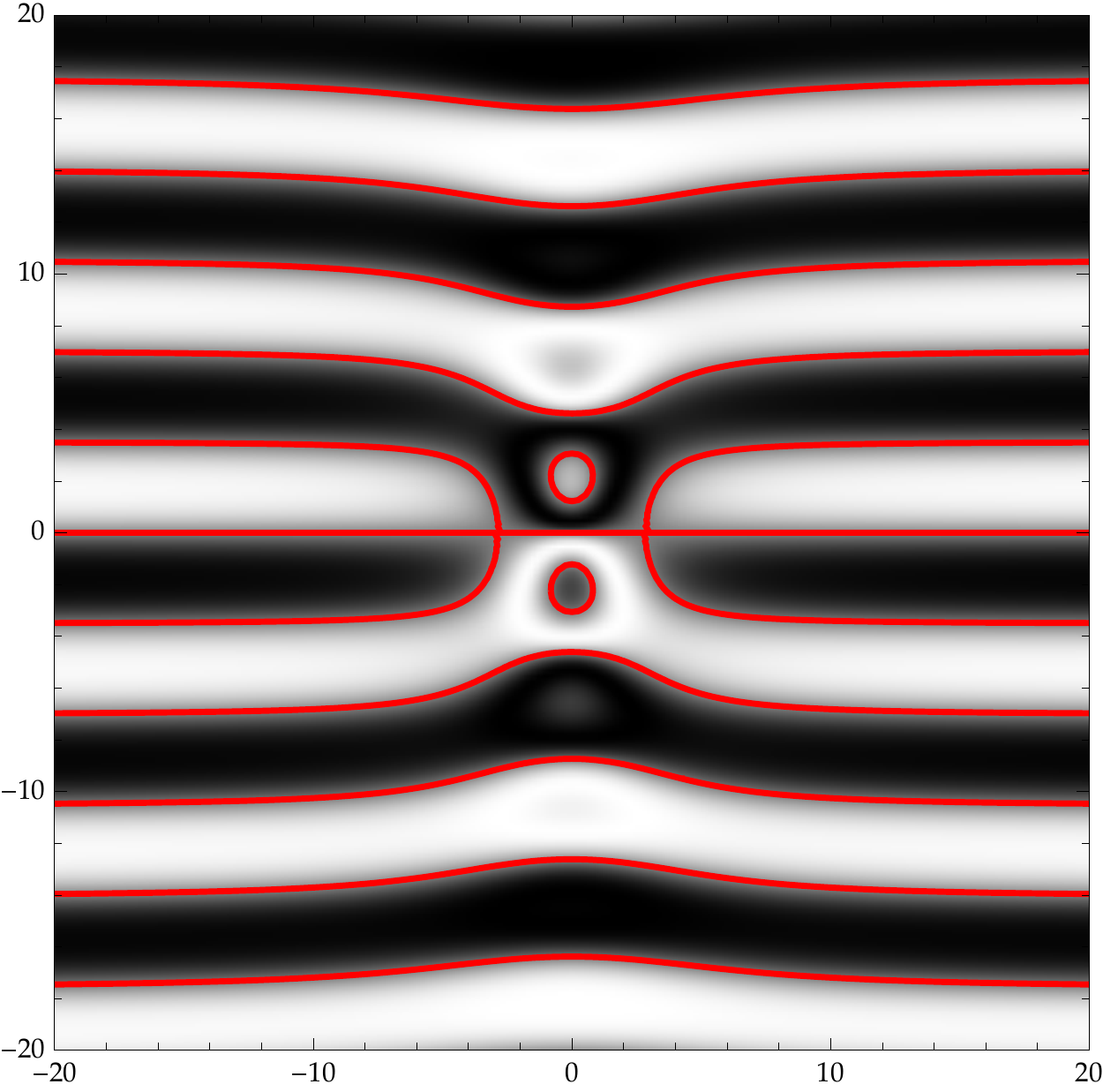}
\end{center}
\caption{As in Figure~\ref{fig:exact-solutions-first} but for $m=\sin^2(\tfrac{5}{24}\pi)$.}
\label{fig:exact-solutions-closest}
\end{figure}
\begin{figure}[h!]
\begin{center}
\includegraphics[height=.24\linewidth]{fig/Legend-SMALL.pdf}%
\includegraphics[height=.24\linewidth]{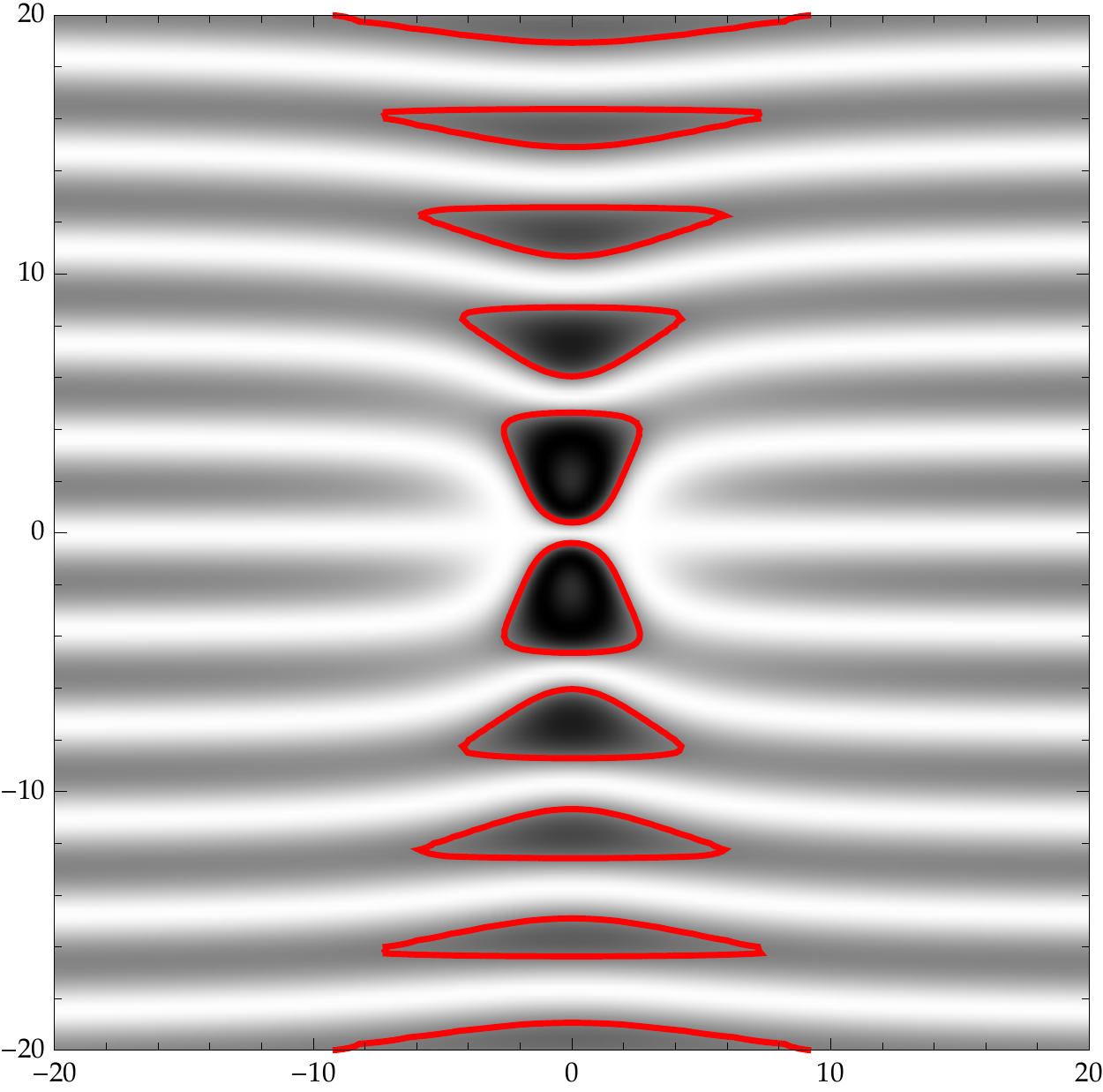}%
\includegraphics[height=.24\linewidth]{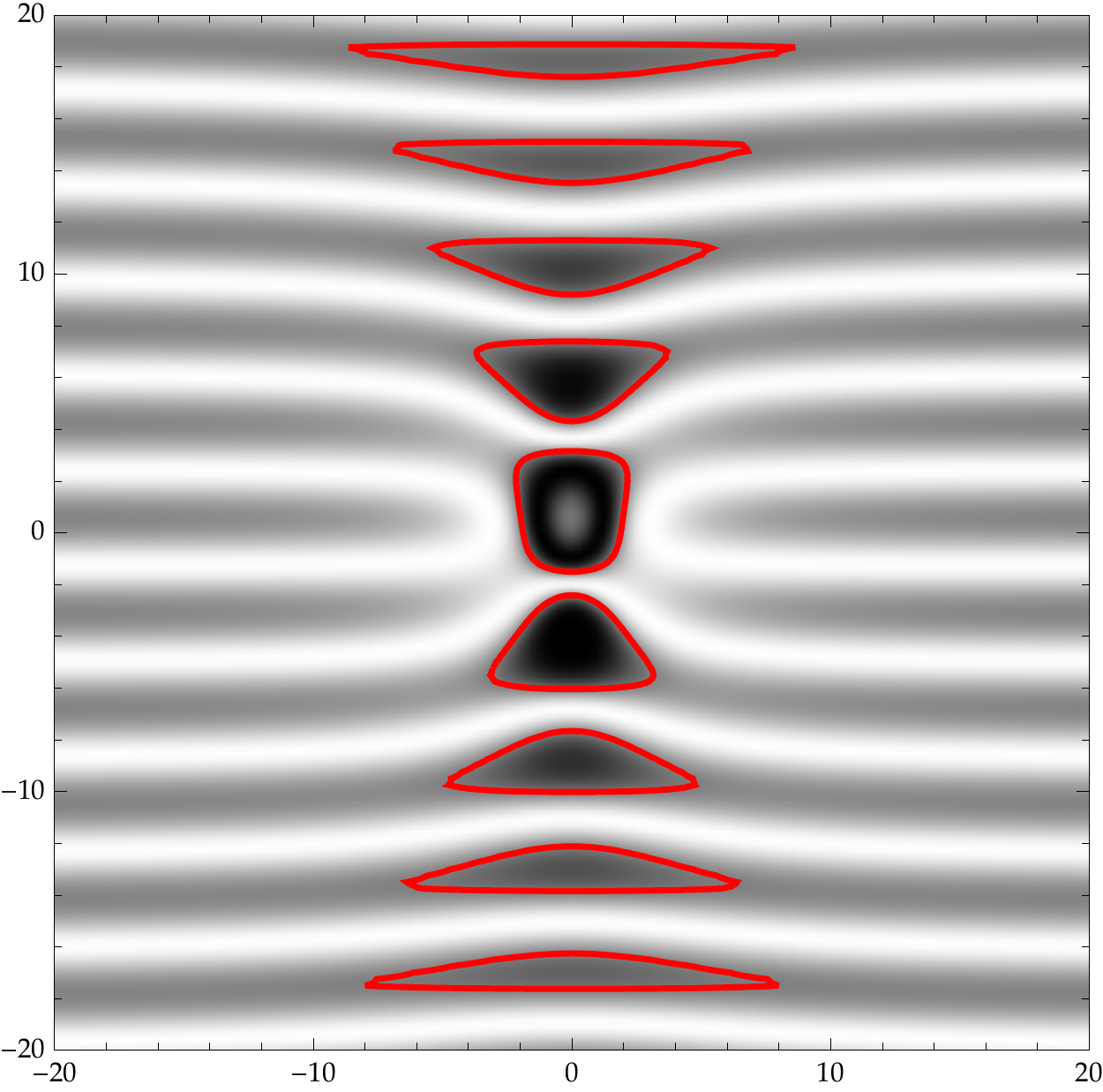}%
\includegraphics[height=.24\linewidth]{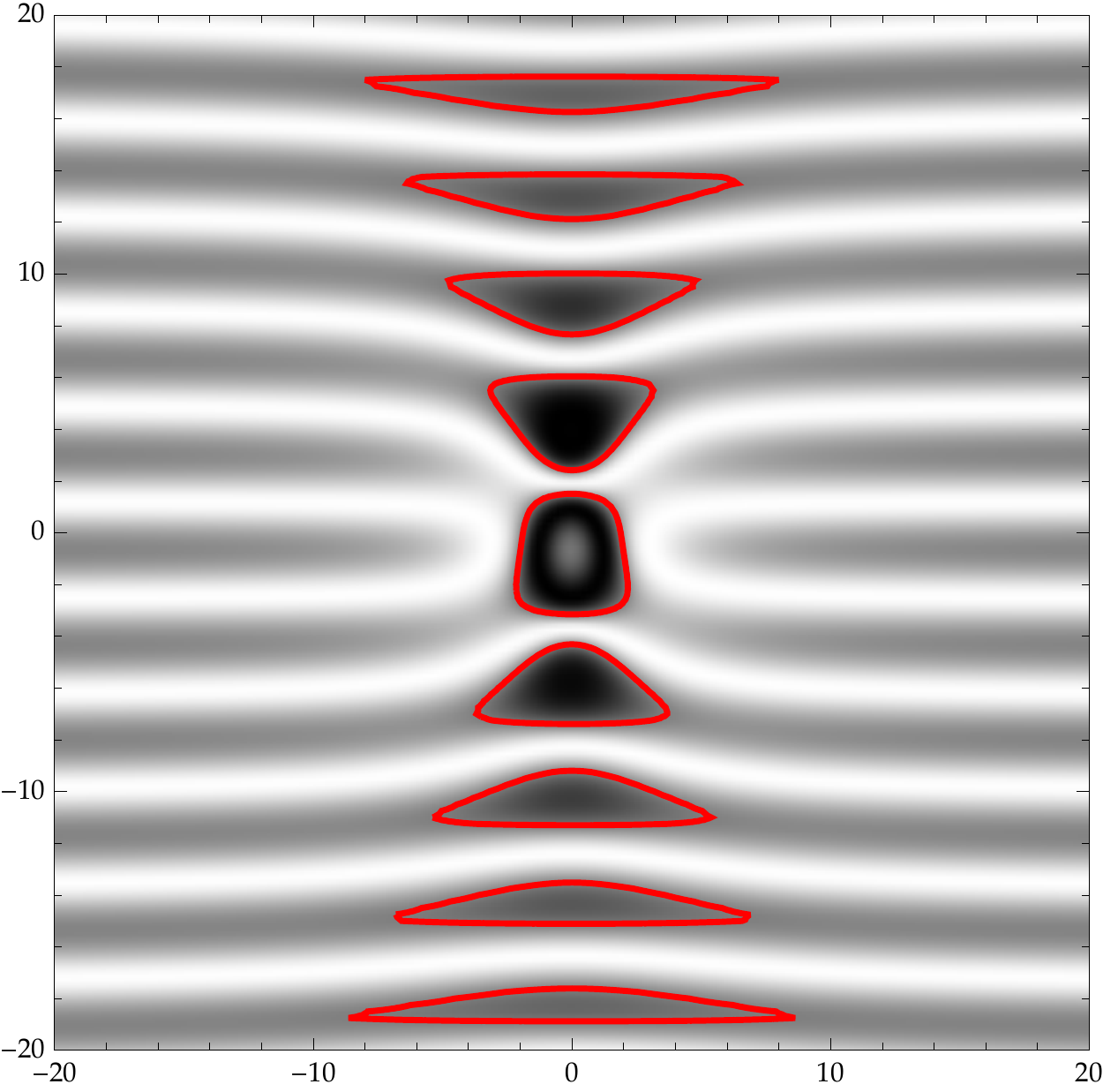}%
\includegraphics[height=.24\linewidth]{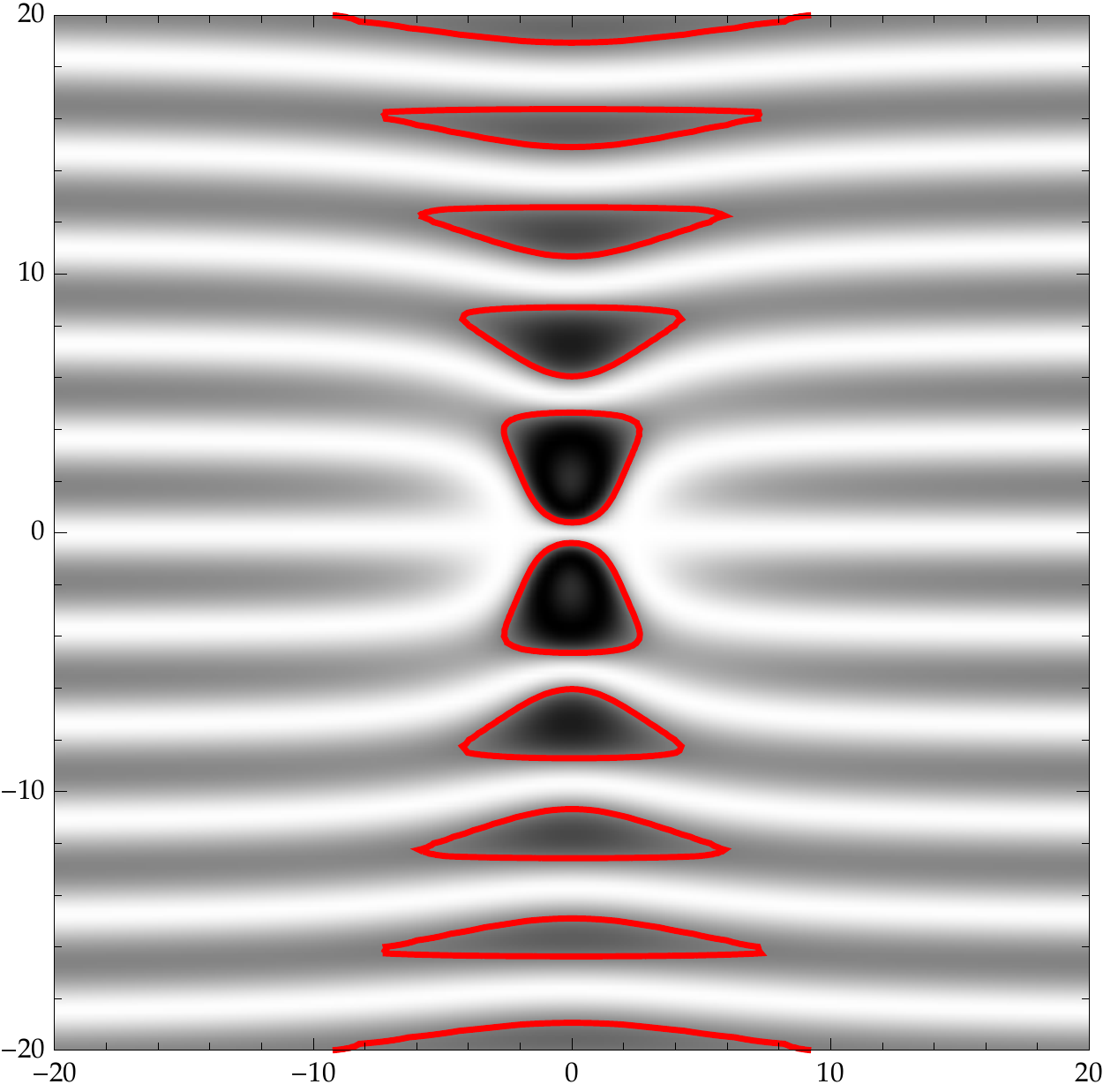}\\
\includegraphics[height=.24\linewidth]{fig/Legend-SMALL.pdf}%
\includegraphics[height=.24\linewidth]{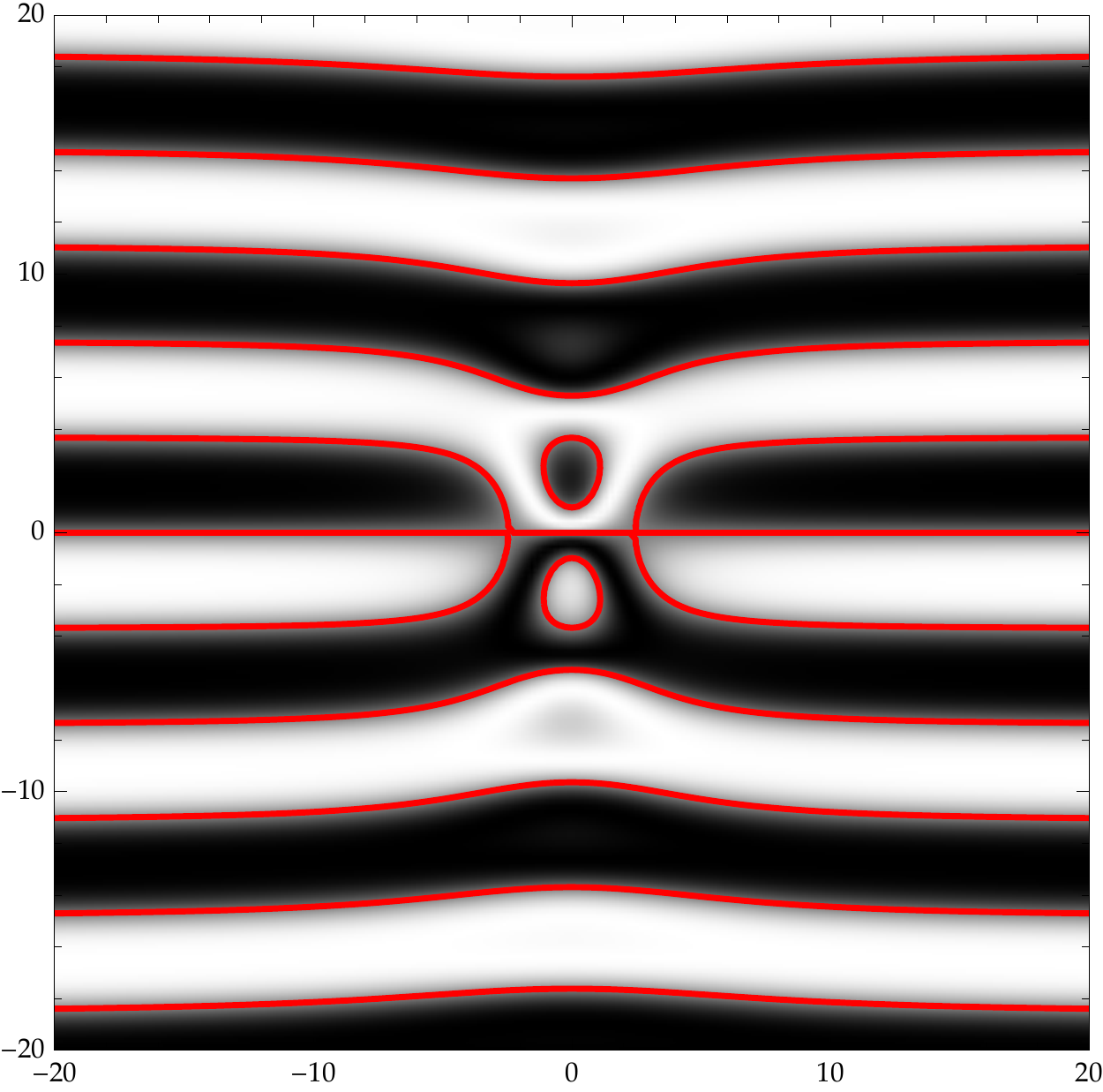}%
\includegraphics[height=.24\linewidth]{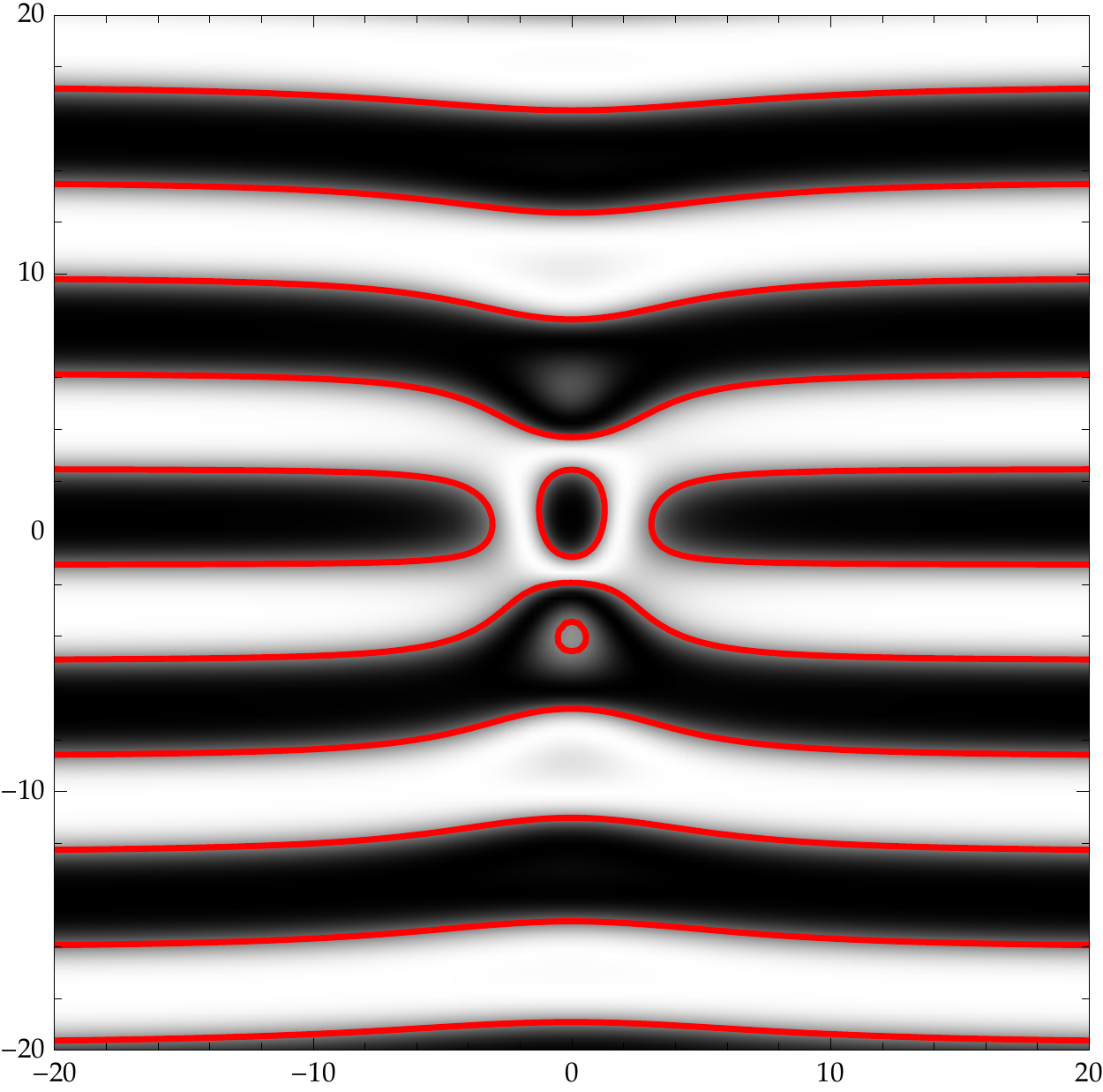}%
\includegraphics[height=.24\linewidth]{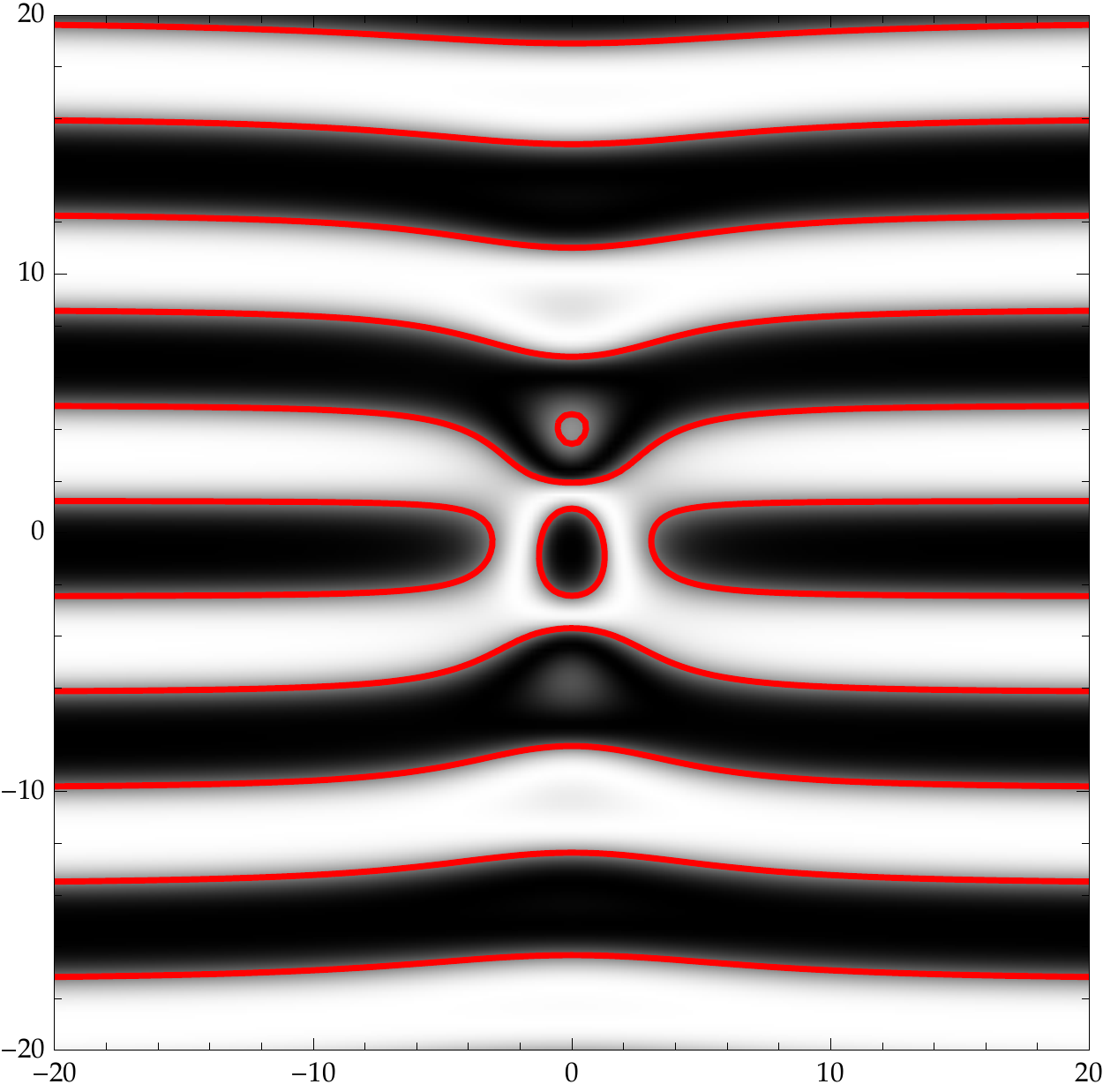}%
\includegraphics[height=.24\linewidth]{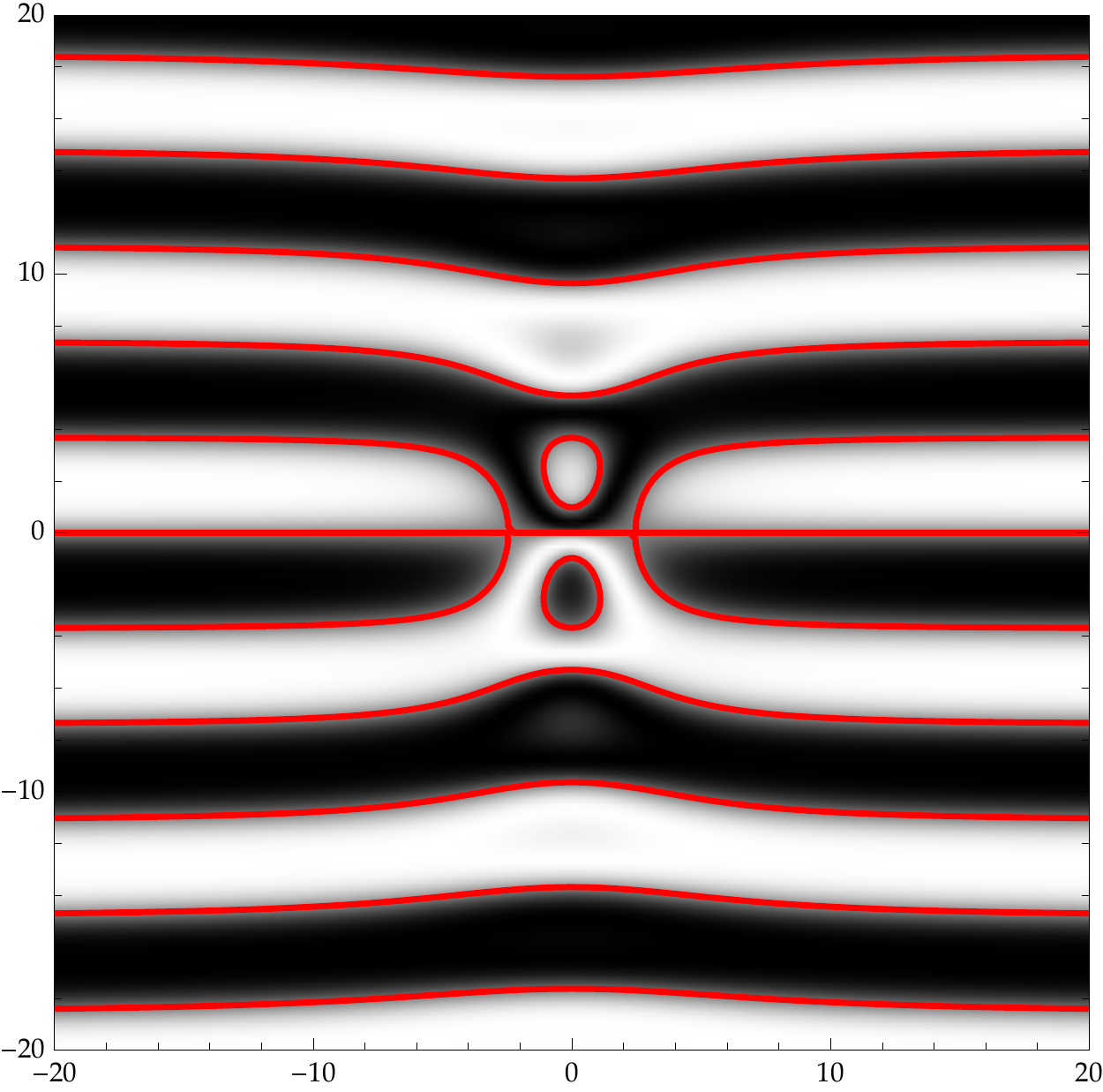}
\end{center}
\caption{As in Figure~\ref{fig:exact-solutions-first} but for $m=\sin^2(\tfrac{1}{4}\pi)$.  Note the horizontal expansion of the red curves in the top row of plots, indicating $\cos(U(X,T;m,\Omega))=0$.  This is occurring because the amplitude of the background wave is increasing with $m$, and in a neighborhood of $(X,T)\to\infty$ $U$ first exits the interval $[-\tfrac{1}{2}\pi,\tfrac{1}{2}\pi]$ for $m>\sin^2(\tfrac{1}{4}\pi)=\tfrac{1}{2}$.  Compare with the first row of Figure~\ref{fig:next-plot}.}
\end{figure}
\begin{figure}[h!]
\begin{center}
\includegraphics[height=.24\linewidth]{fig/Legend-SMALL.pdf}%
\includegraphics[height=.24\linewidth]{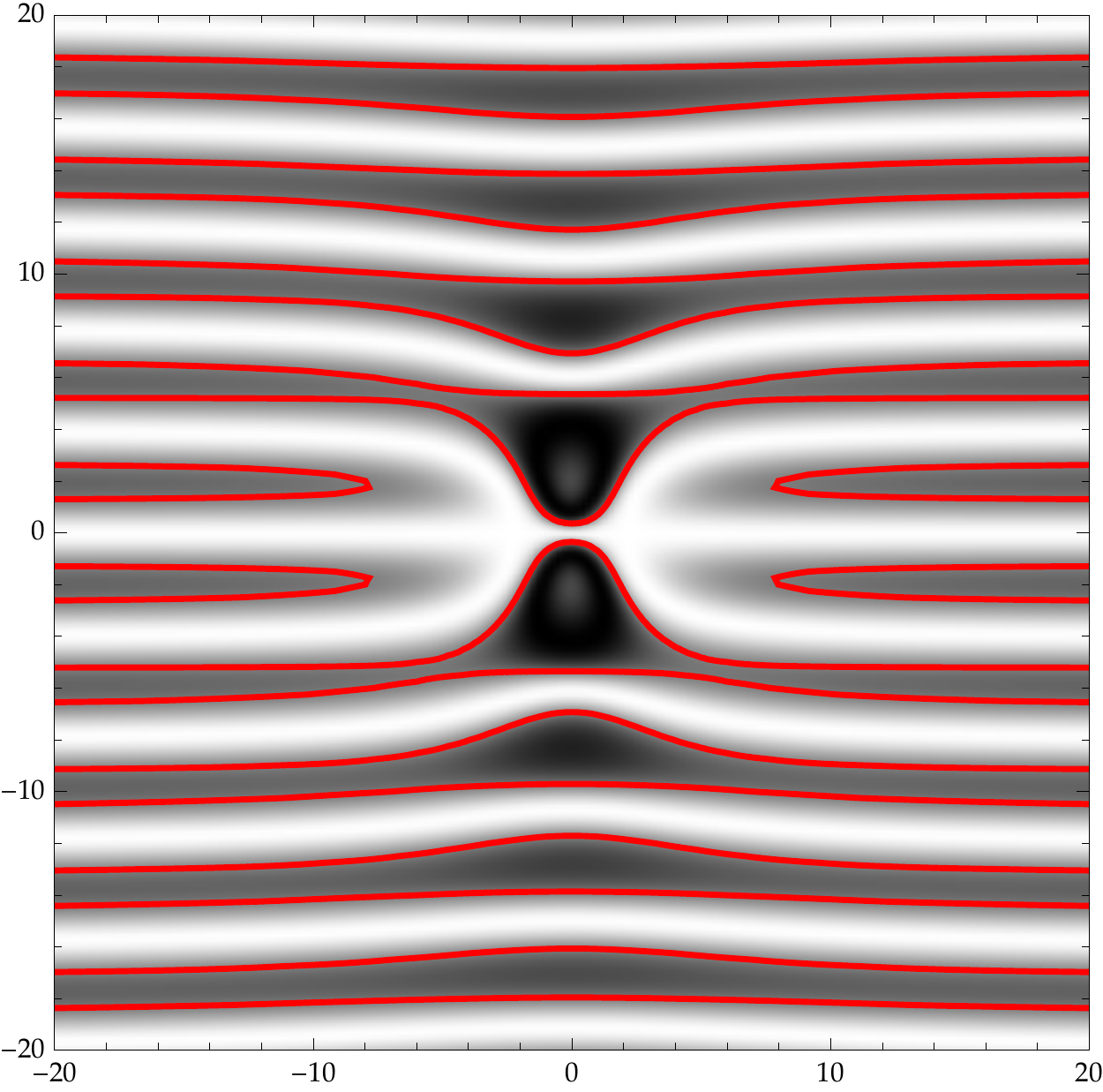}%
\includegraphics[height=.24\linewidth]{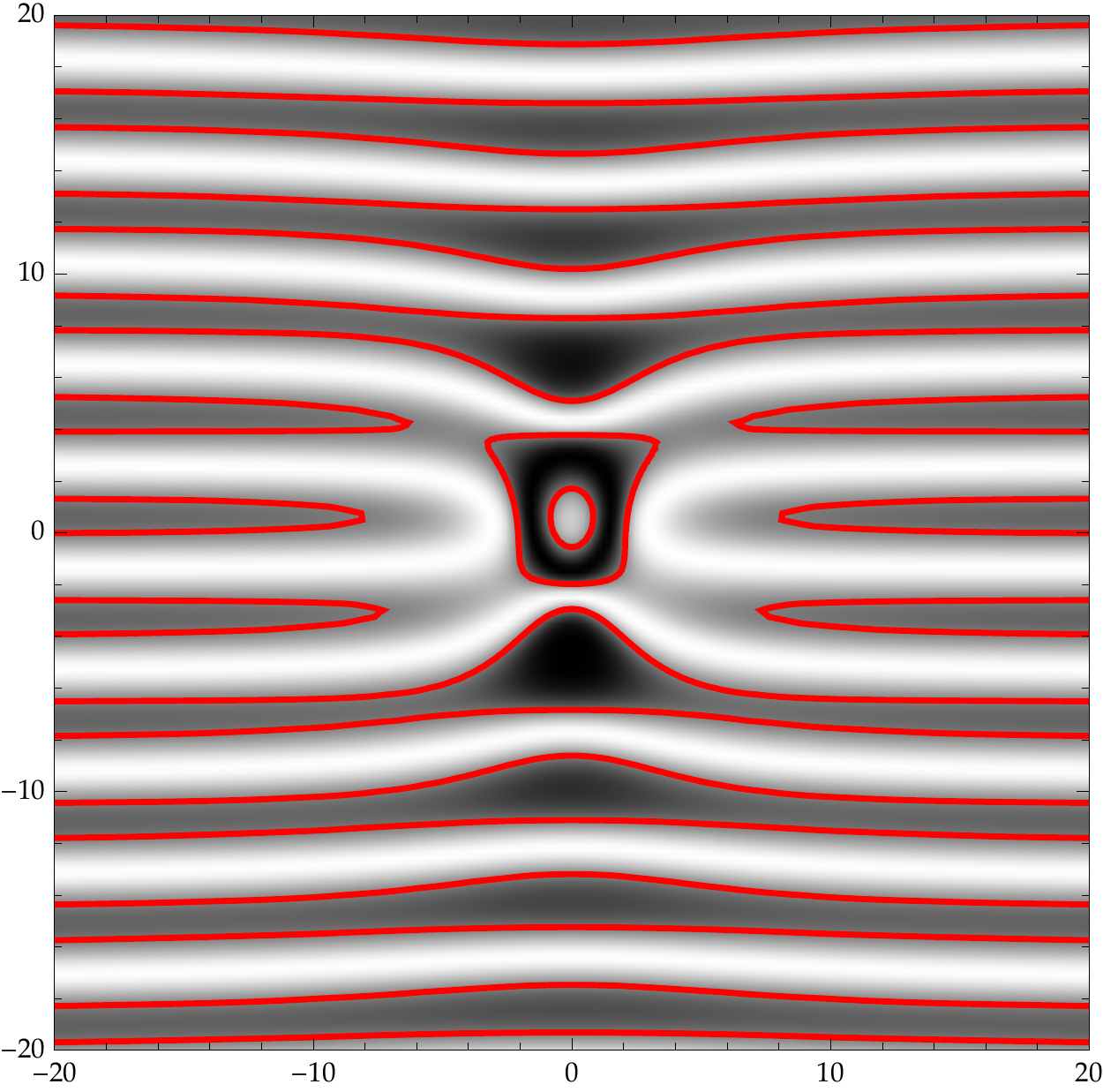}%
\includegraphics[height=.24\linewidth]{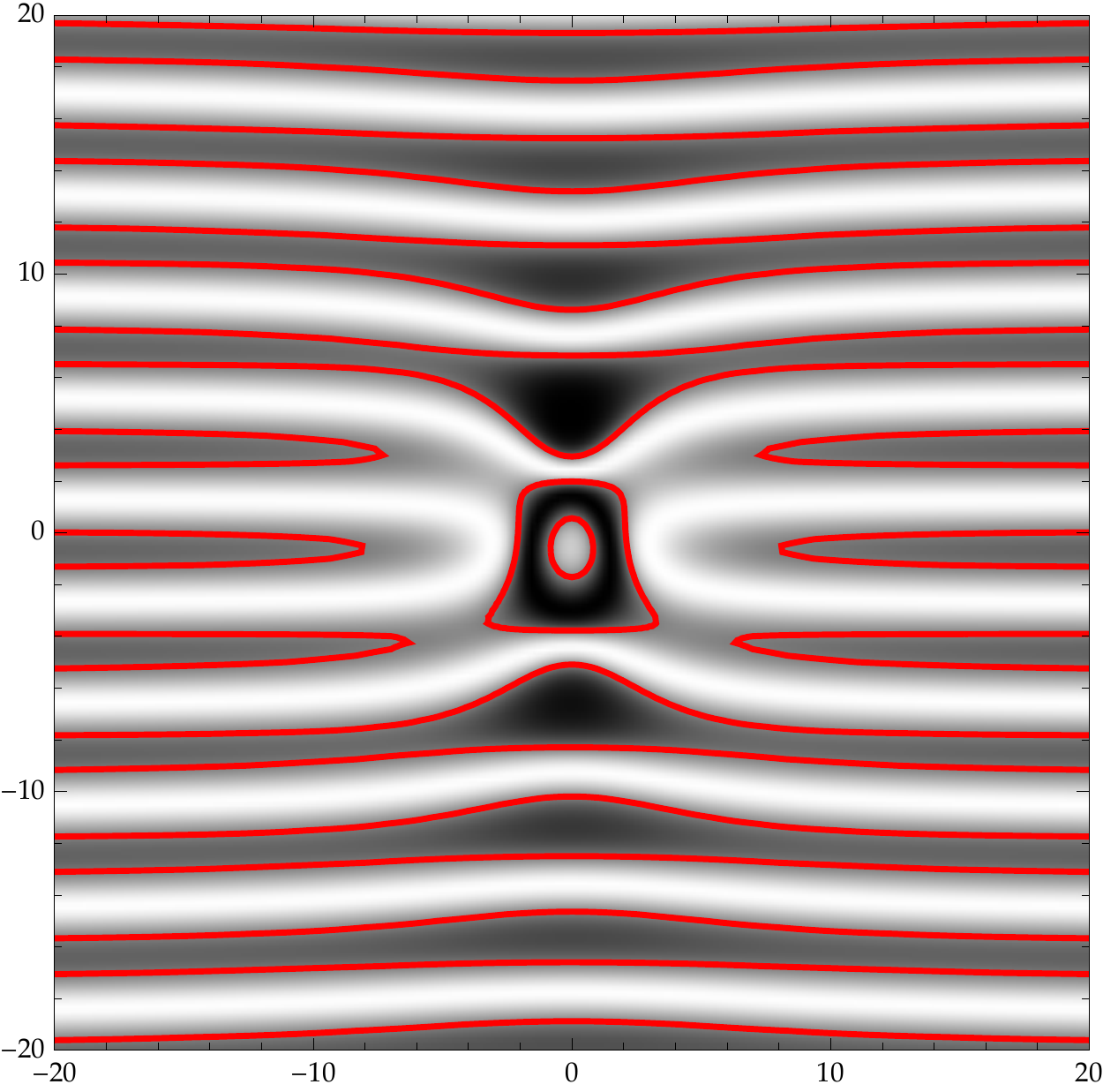}%
\includegraphics[height=.24\linewidth]{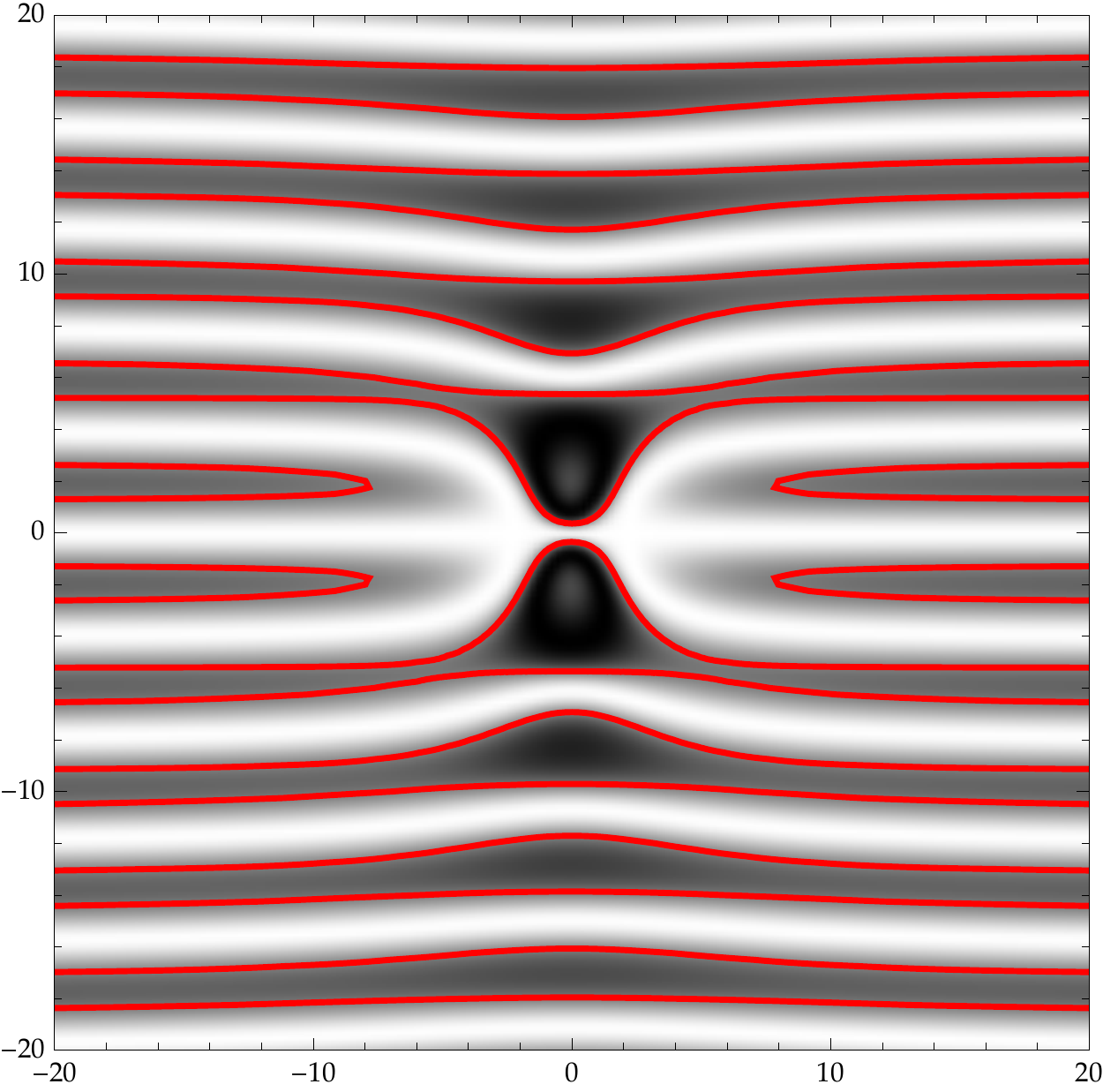}\\
\includegraphics[height=.24\linewidth]{fig/Legend-SMALL.pdf}%
\includegraphics[height=.24\linewidth]{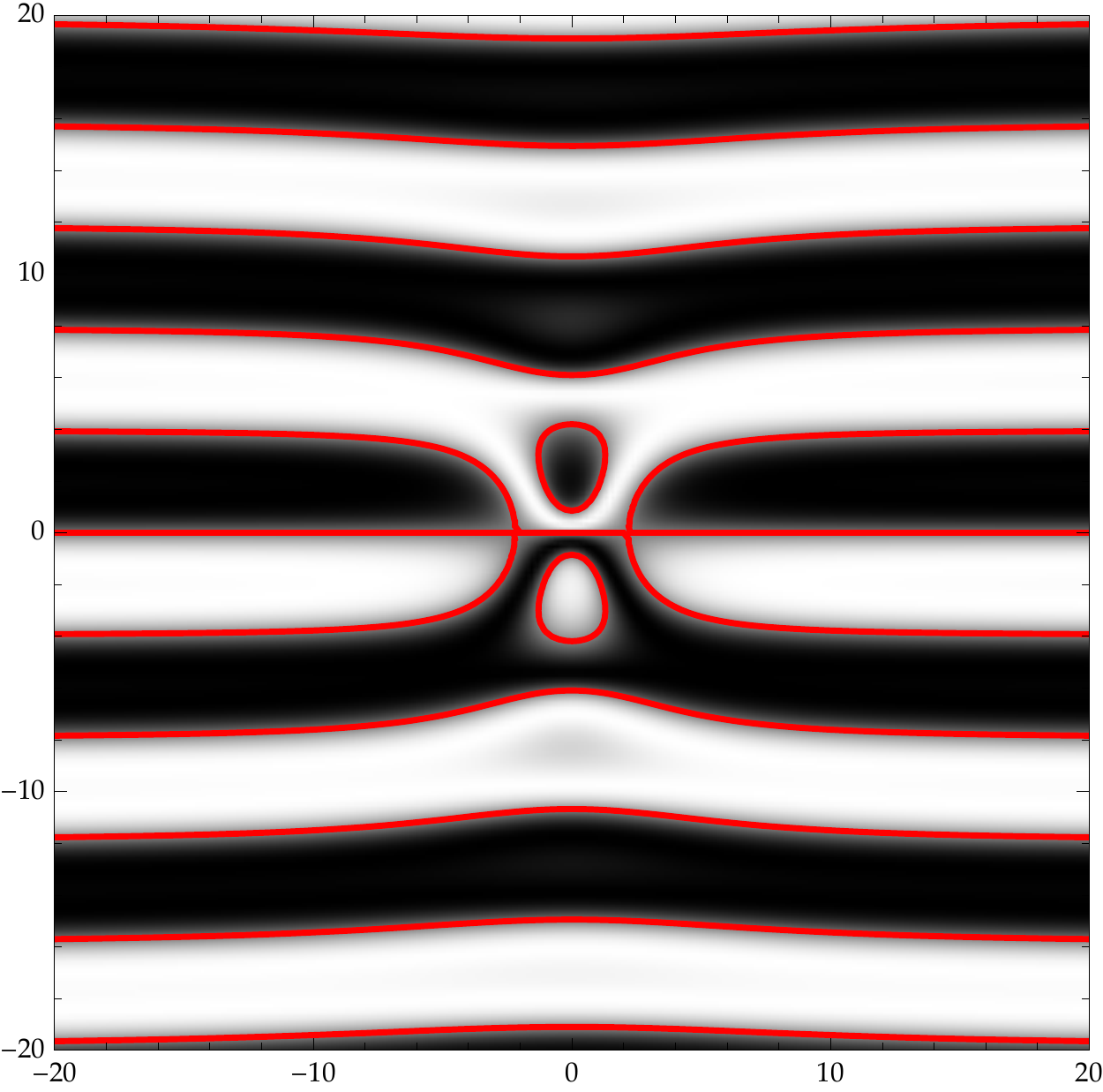}%
\includegraphics[height=.24\linewidth]{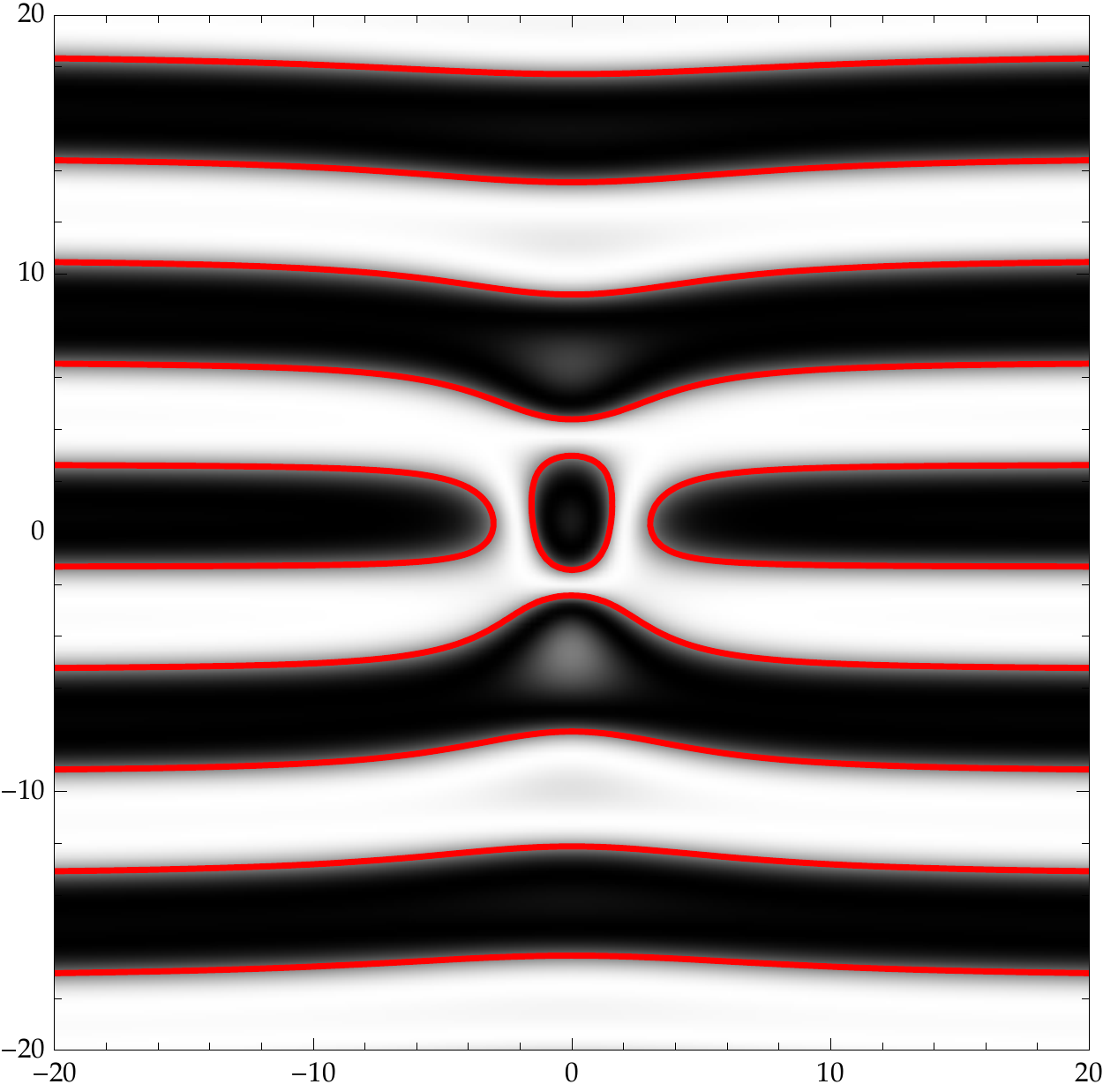}%
\includegraphics[height=.24\linewidth]{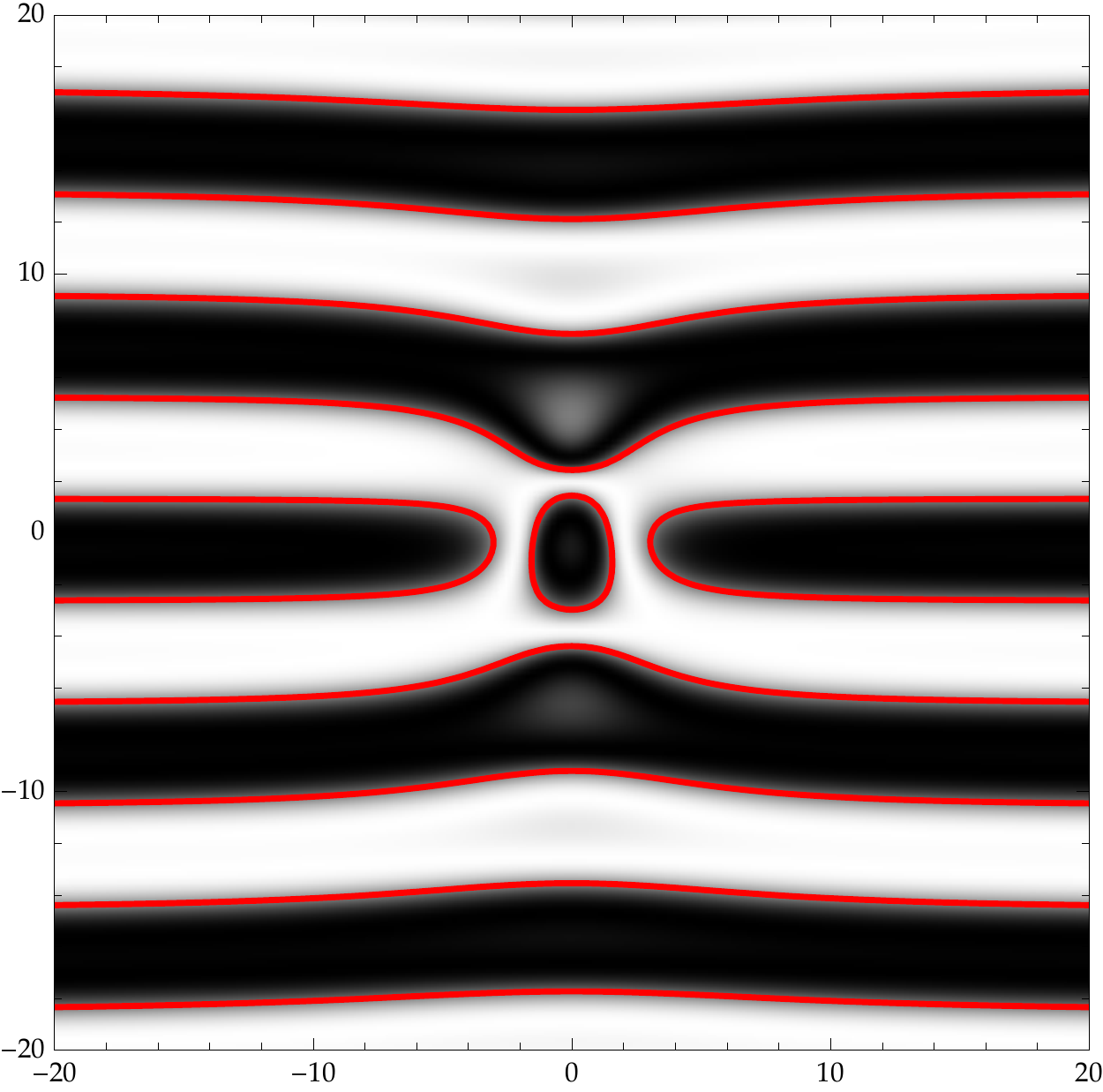}%
\includegraphics[height=.24\linewidth]{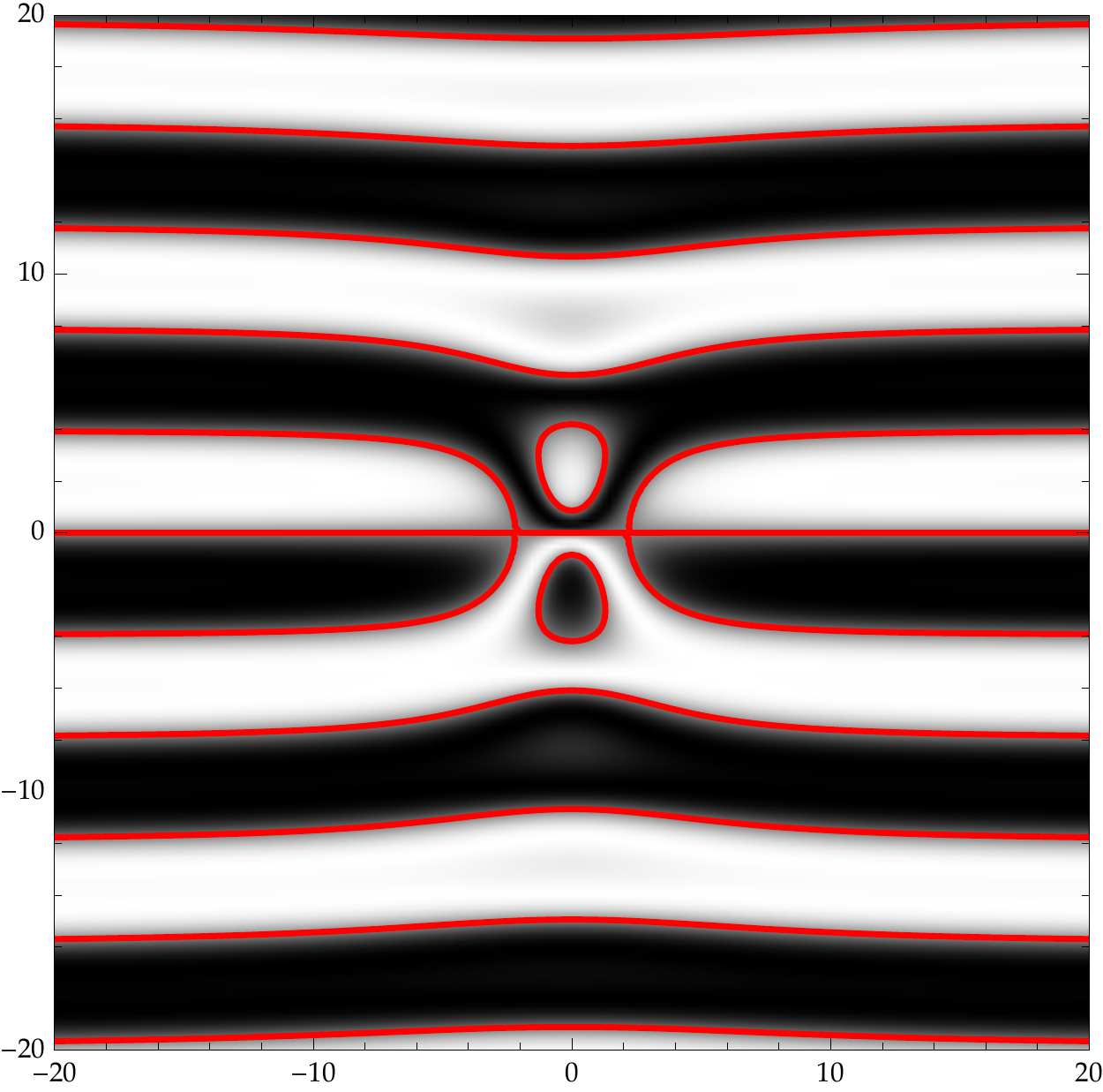}
\end{center}
\caption{As in Figure~\ref{fig:exact-solutions-first} but for $m=\sin^2(\tfrac{7}{24}\pi)$.}
\label{fig:next-plot}
\end{figure}
\begin{figure}[h!]
\begin{center}
\includegraphics[height=.24\linewidth]{fig/Legend-SMALL.pdf}%
\includegraphics[height=.24\linewidth]{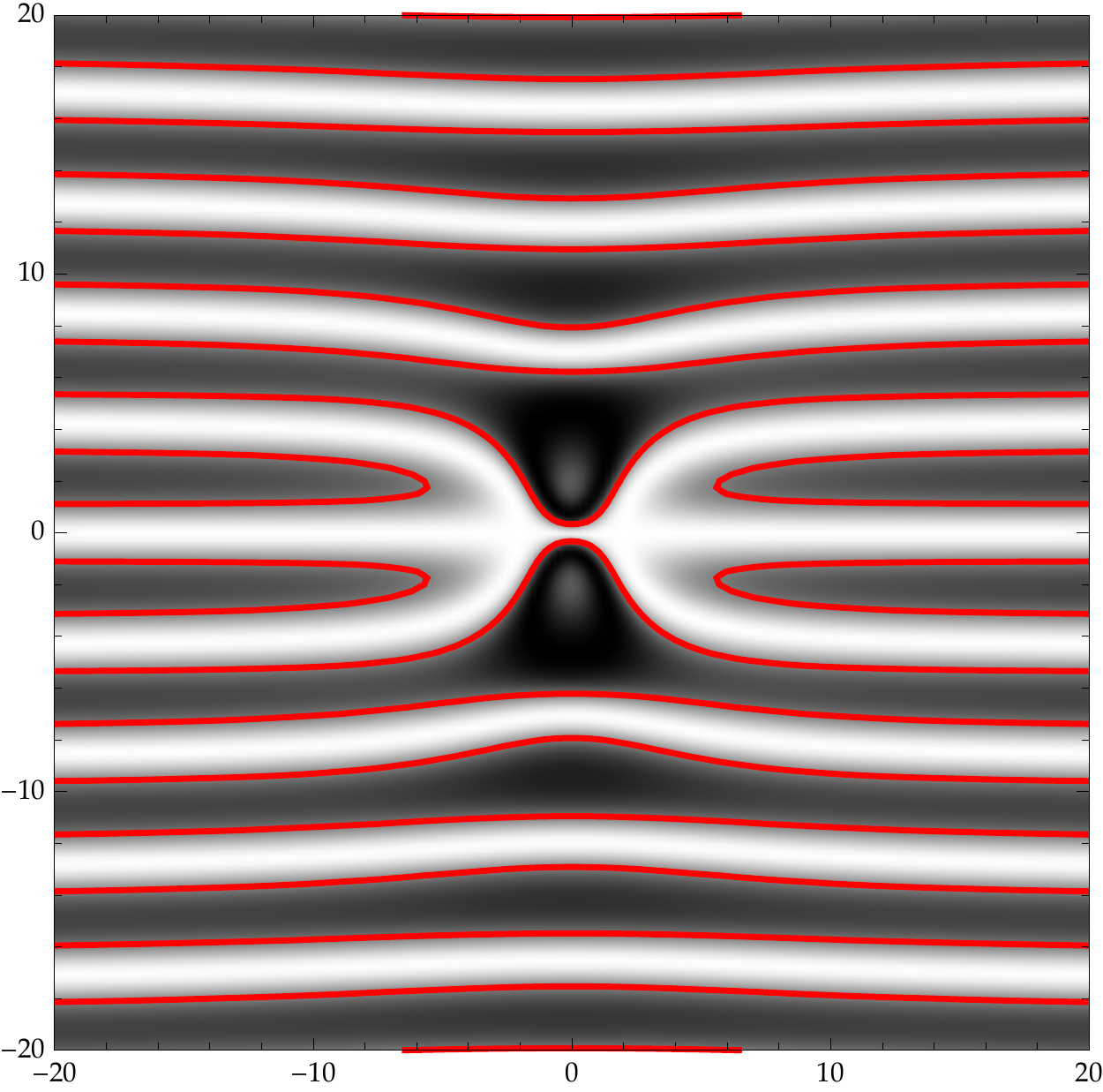}%
\includegraphics[height=.24\linewidth]{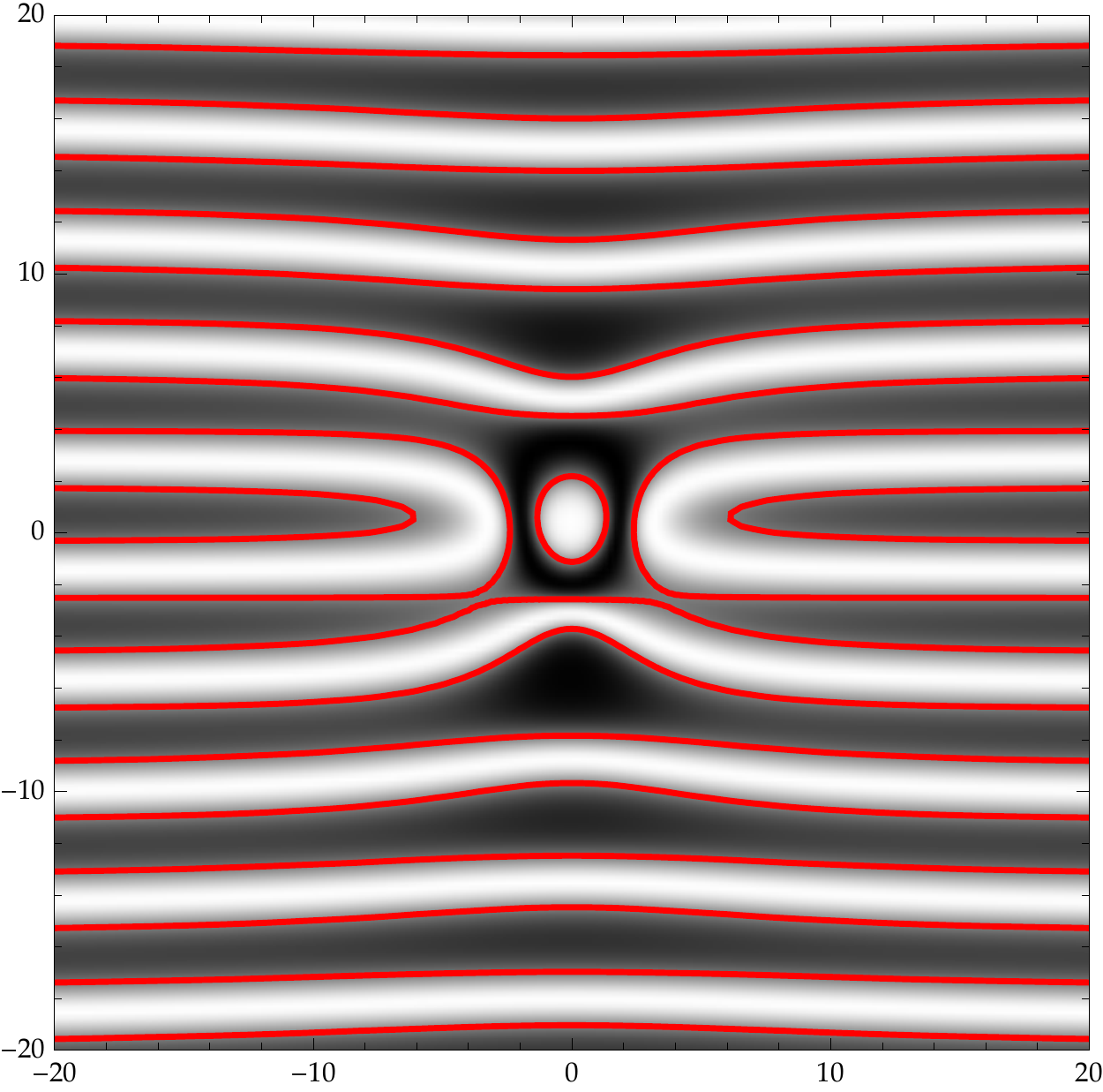}%
\includegraphics[height=.24\linewidth]{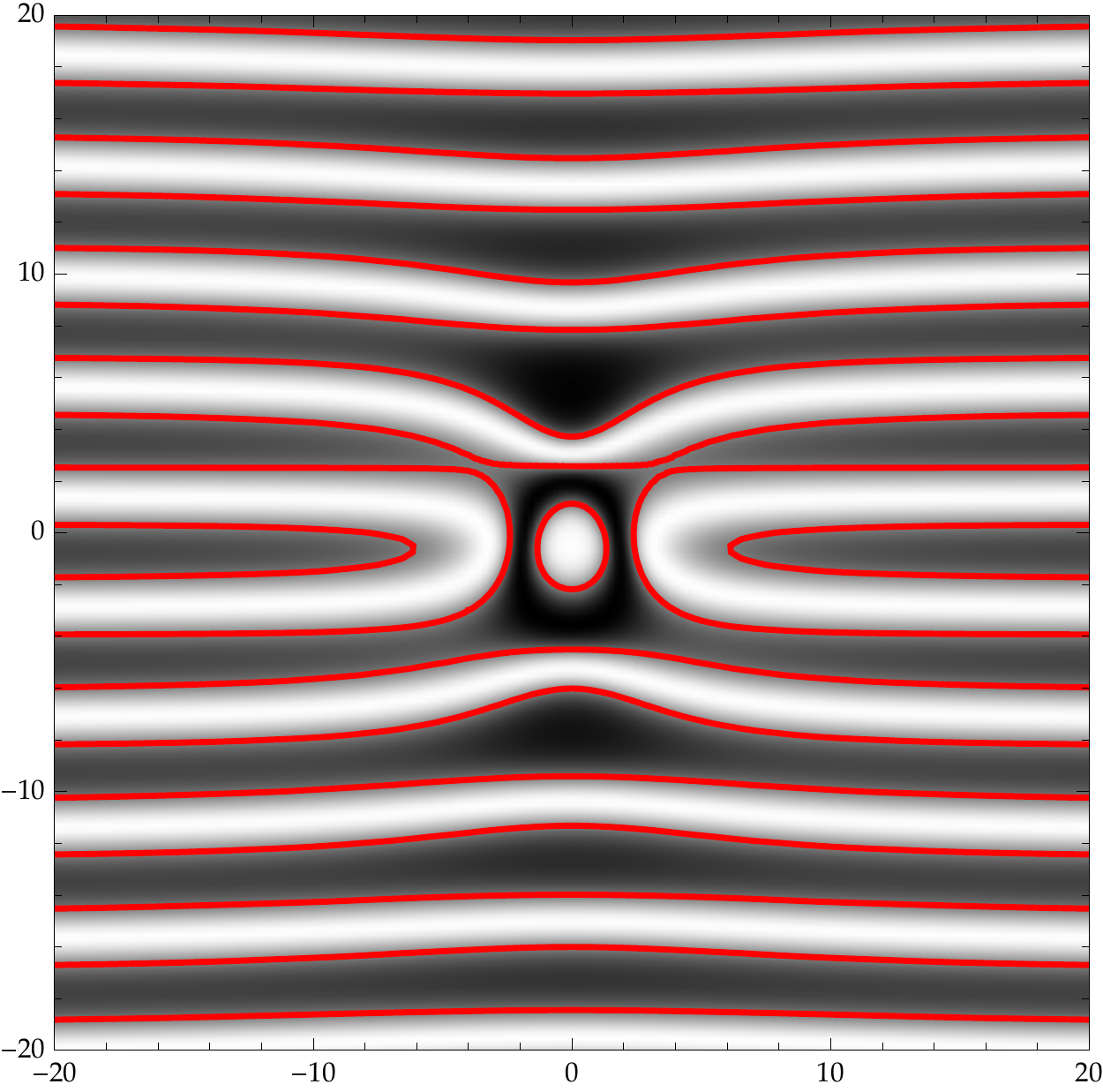}%
\includegraphics[height=.24\linewidth]{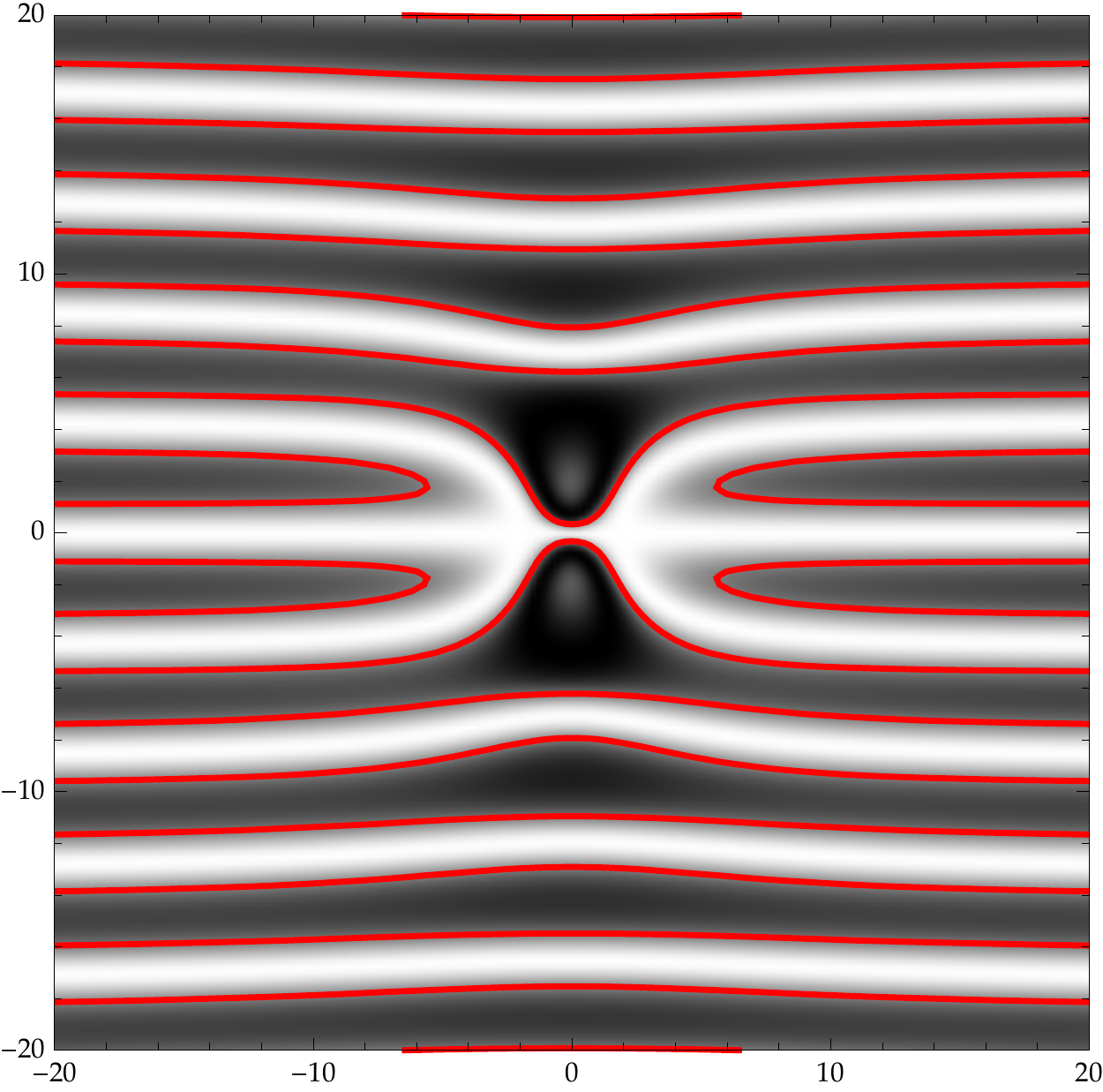}\\
\includegraphics[height=.24\linewidth]{fig/Legend-SMALL.pdf}%
\includegraphics[height=.24\linewidth]{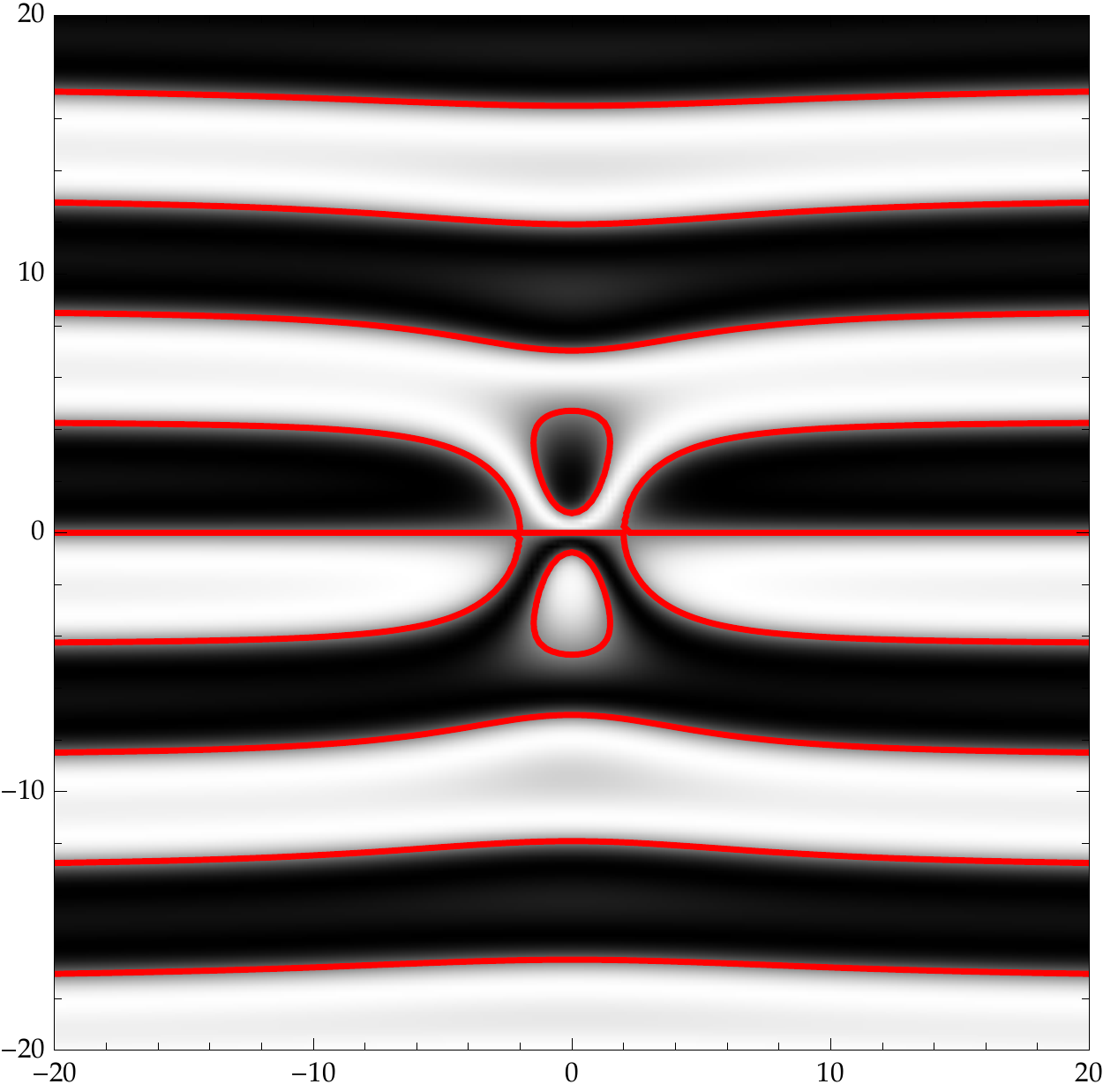}%
\includegraphics[height=.24\linewidth]{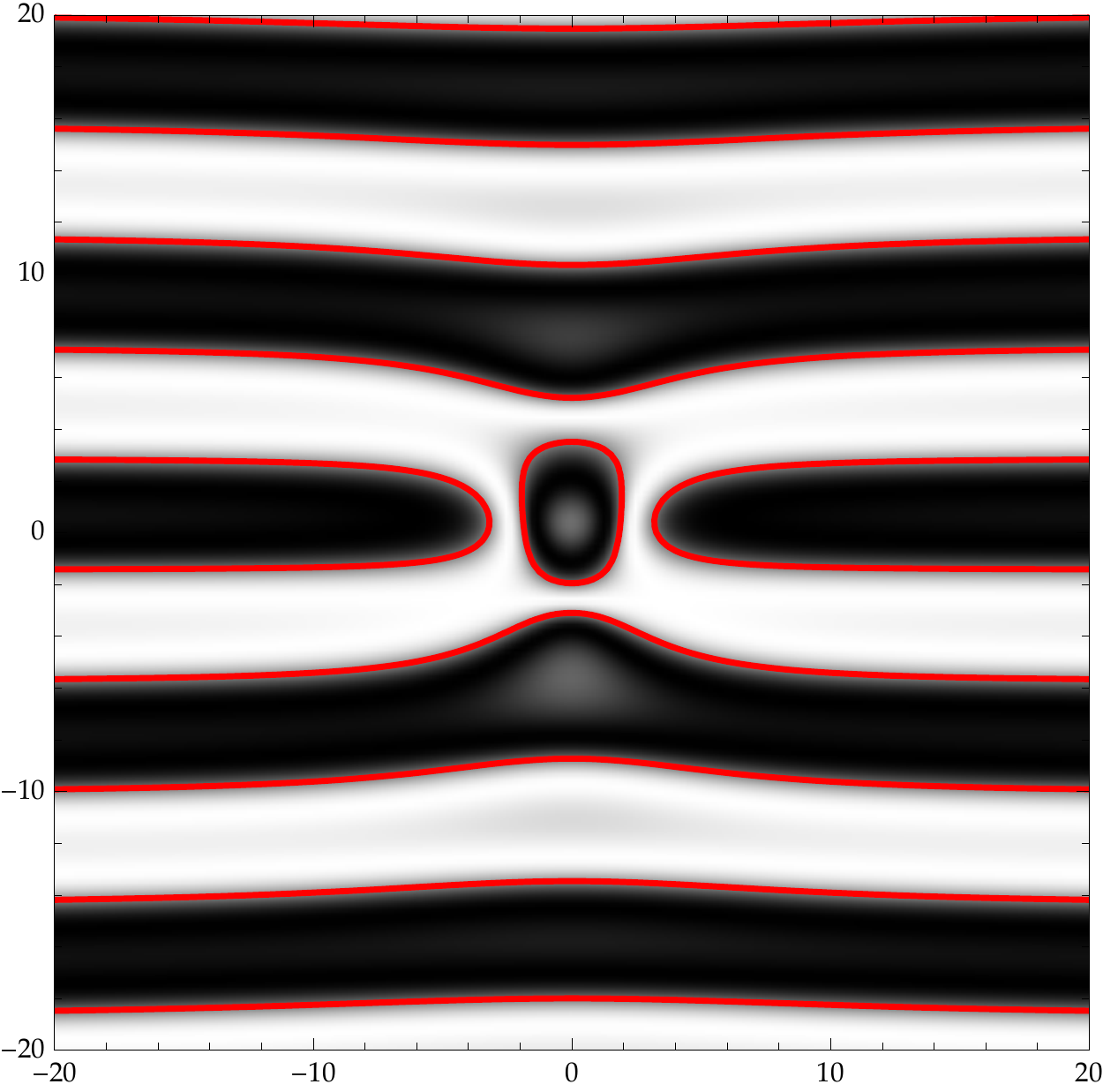}%
\includegraphics[height=.24\linewidth]{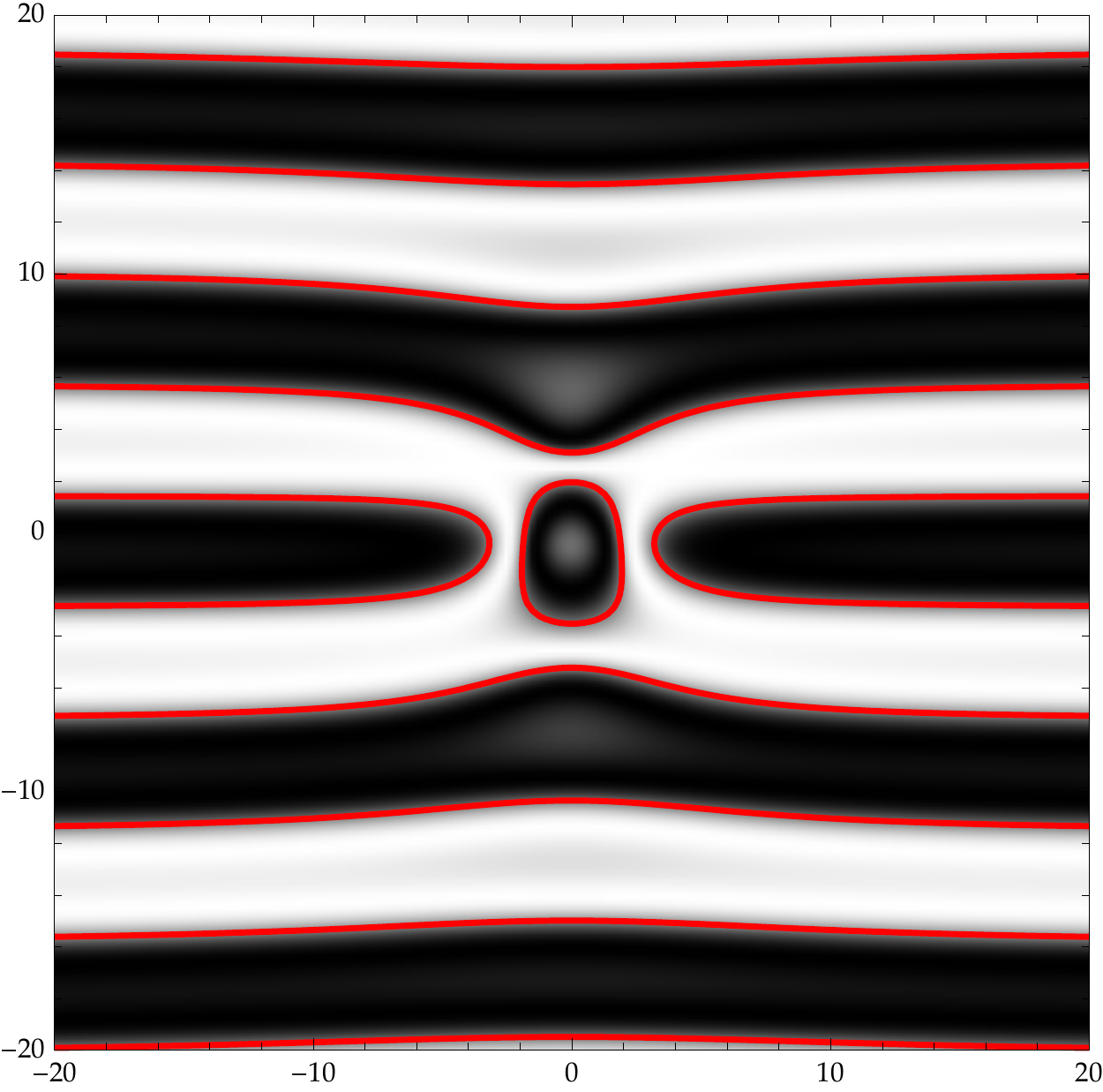}%
\includegraphics[height=.24\linewidth]{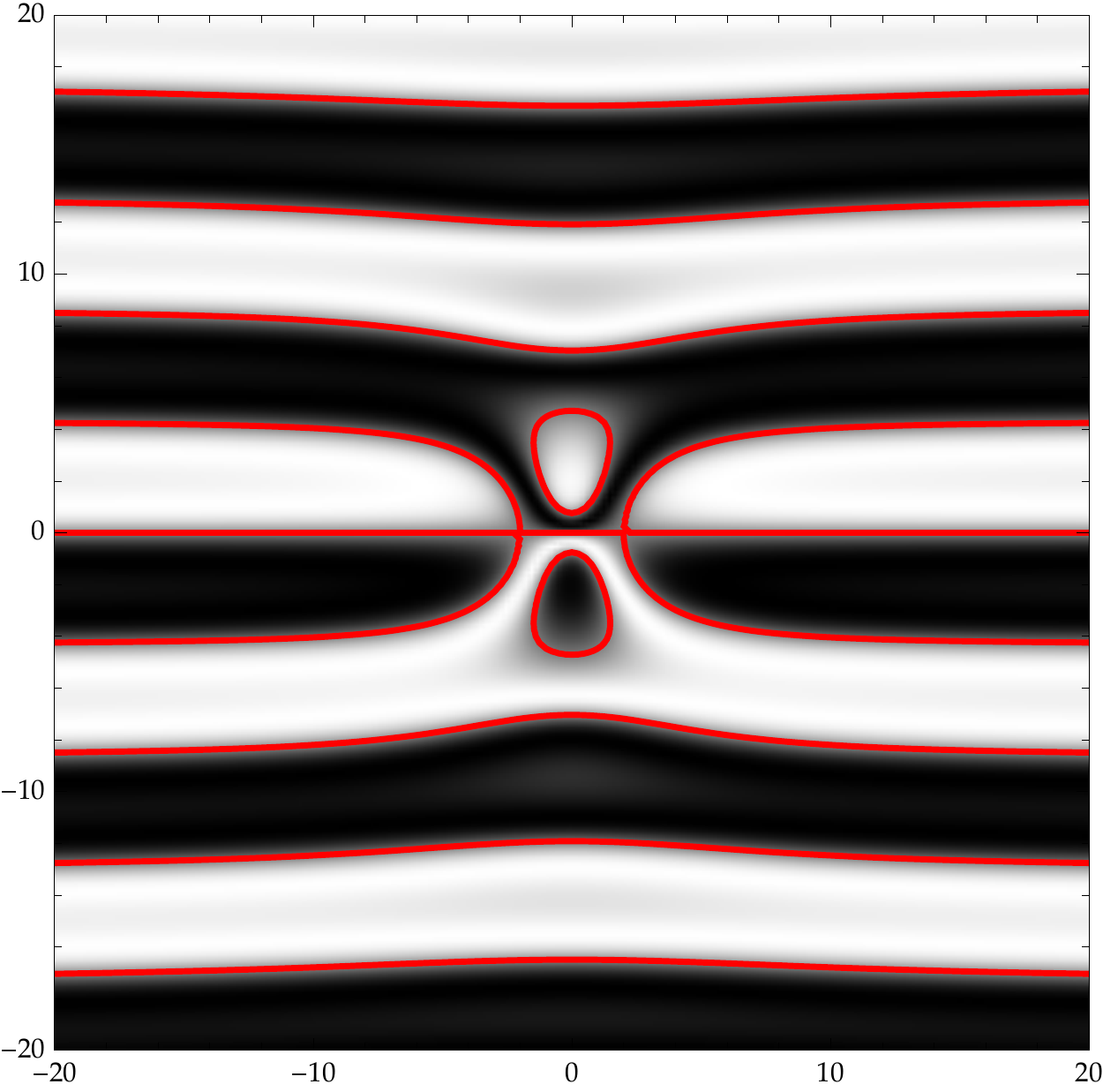}
\end{center}
\caption{As in Figure~\ref{fig:exact-solutions-first} but for $m=\sin^2(\tfrac{1}{3}\pi)$.}
\end{figure}
\begin{figure}[h!]
\begin{center}
\includegraphics[height=.24\linewidth]{fig/Legend-SMALL.pdf}%
\includegraphics[height=.24\linewidth]{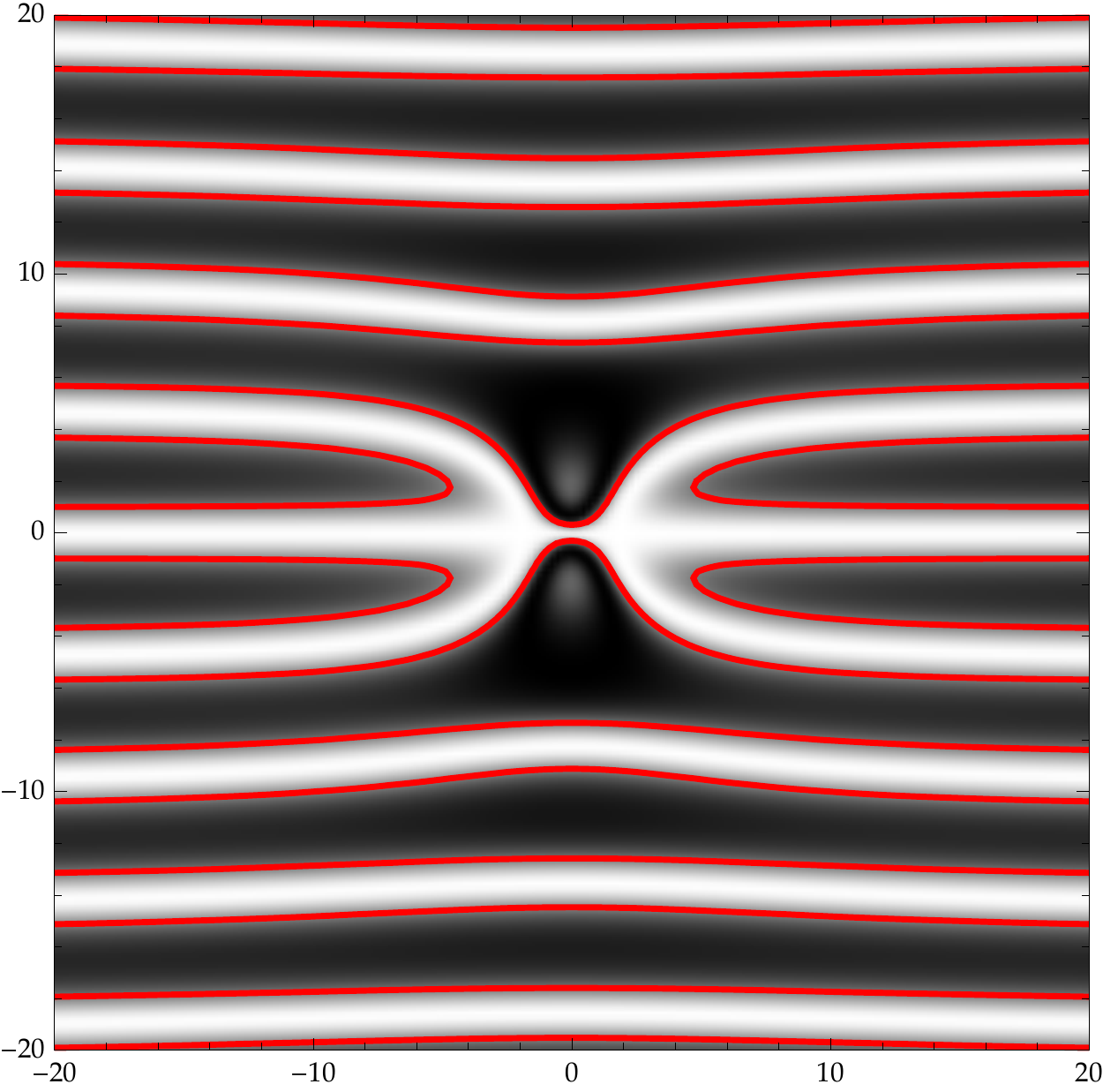}%
\includegraphics[height=.24\linewidth]{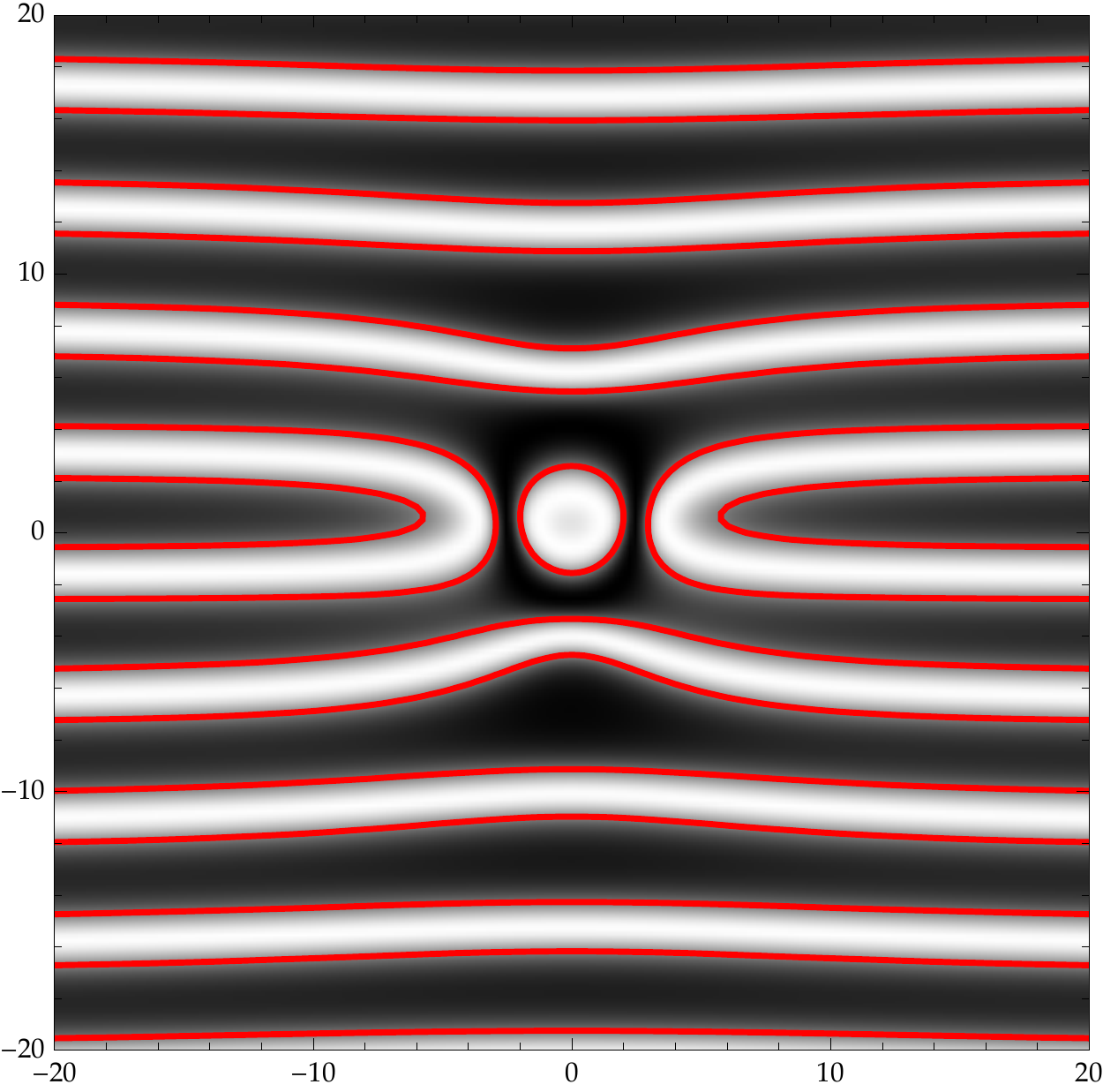}%
\includegraphics[height=.24\linewidth]{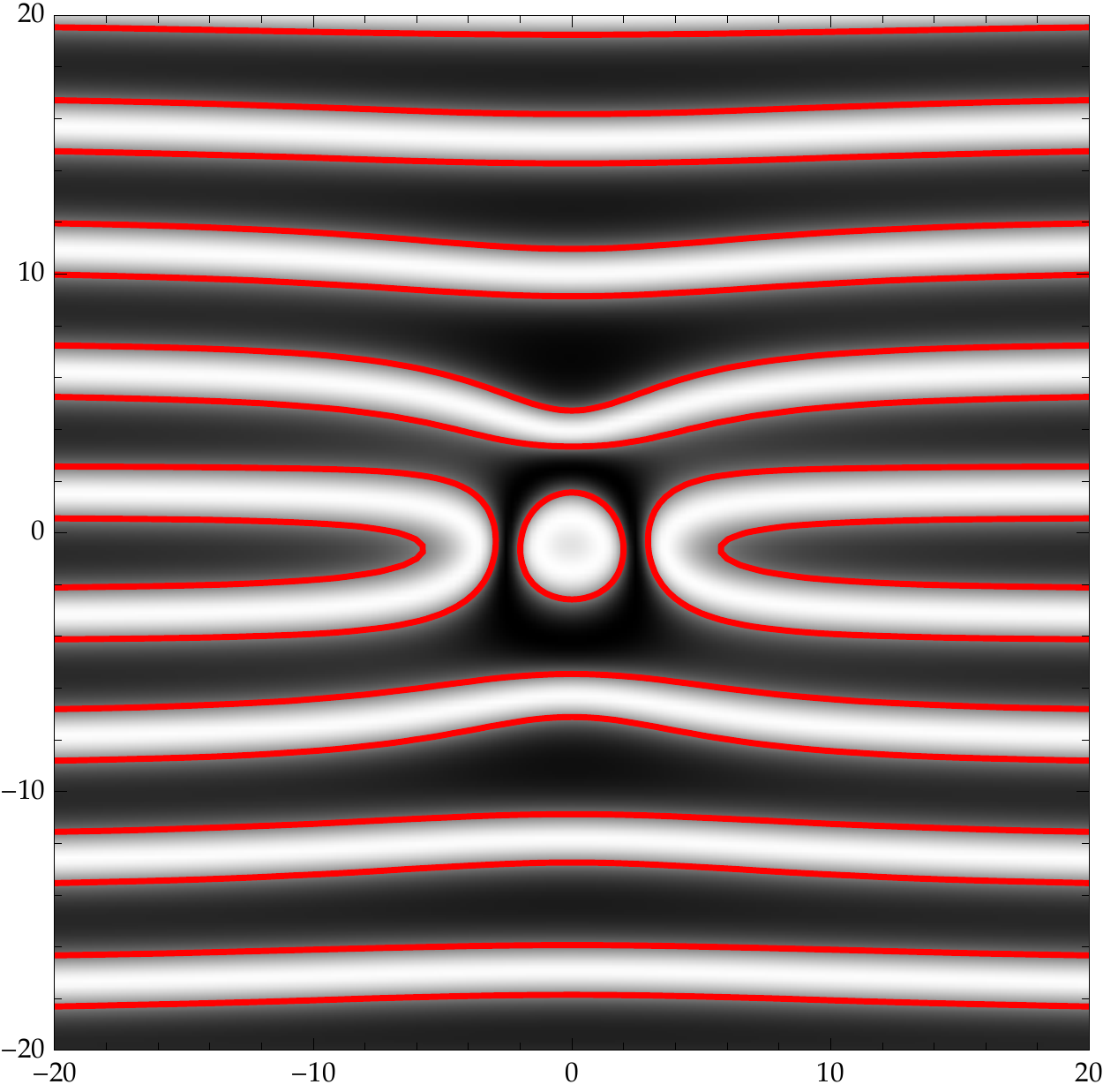}%
\includegraphics[height=.24\linewidth]{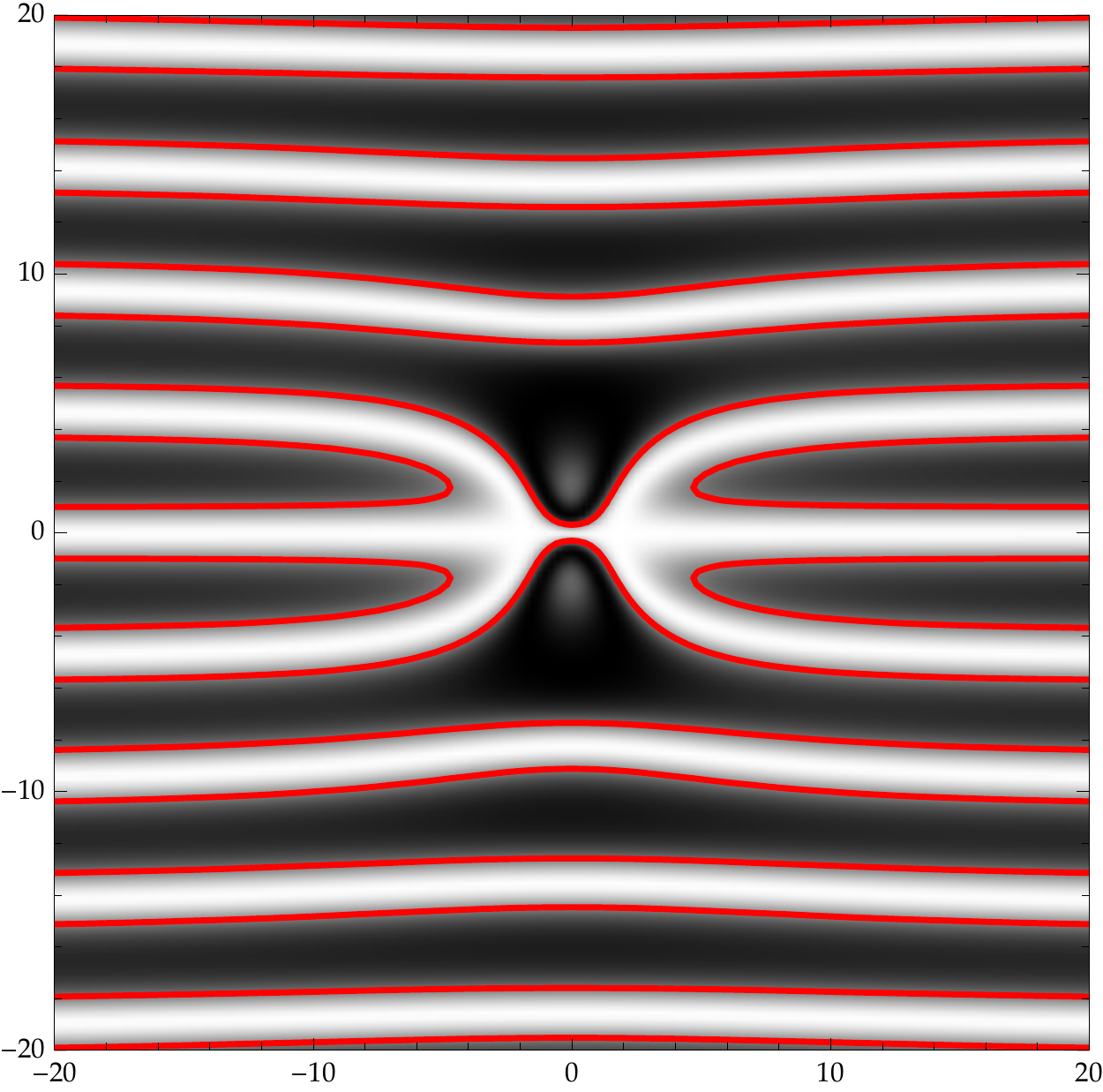}\\
\includegraphics[height=.24\linewidth]{fig/Legend-SMALL.pdf}%
\includegraphics[height=.24\linewidth]{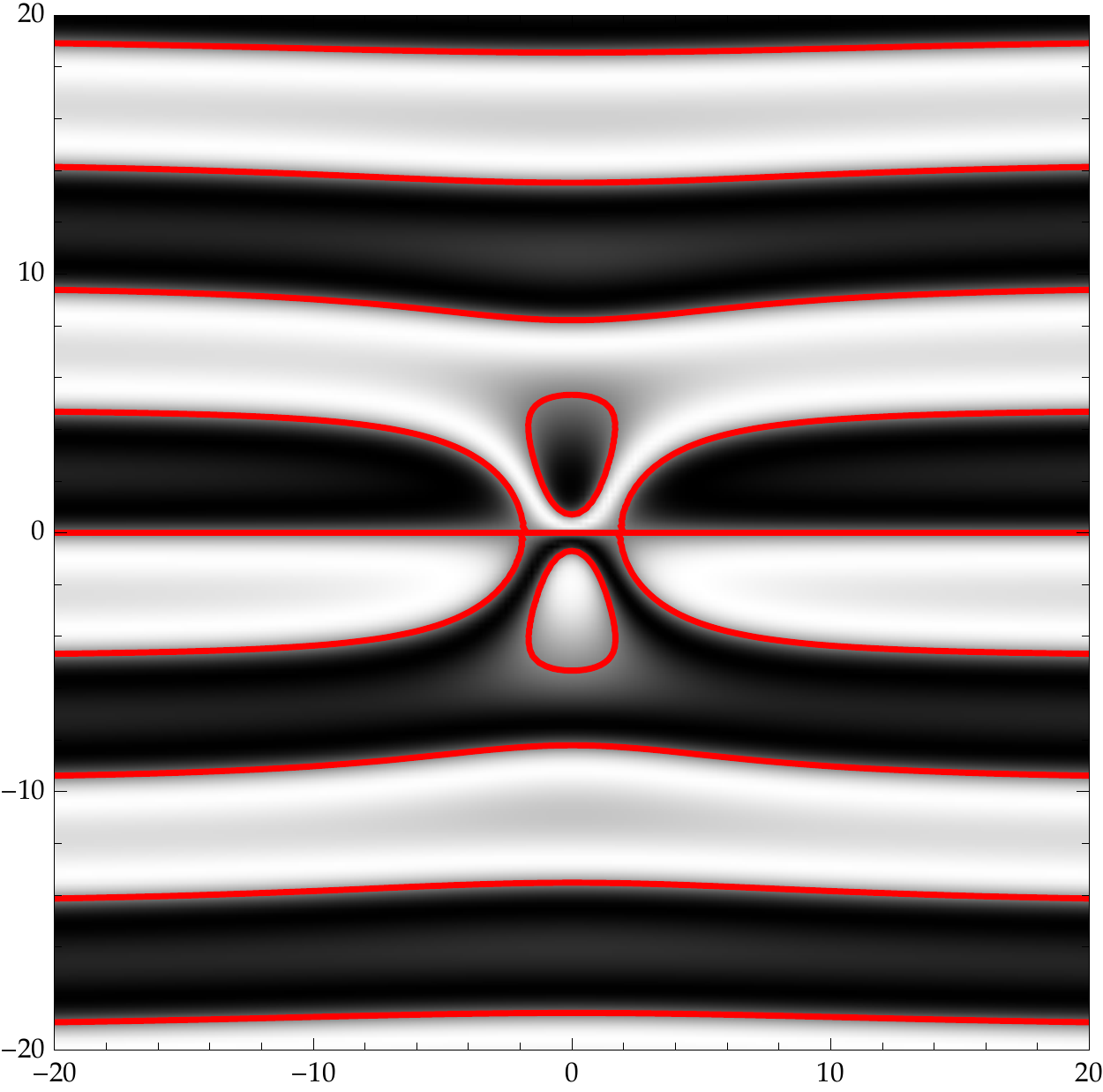}%
\includegraphics[height=.24\linewidth]{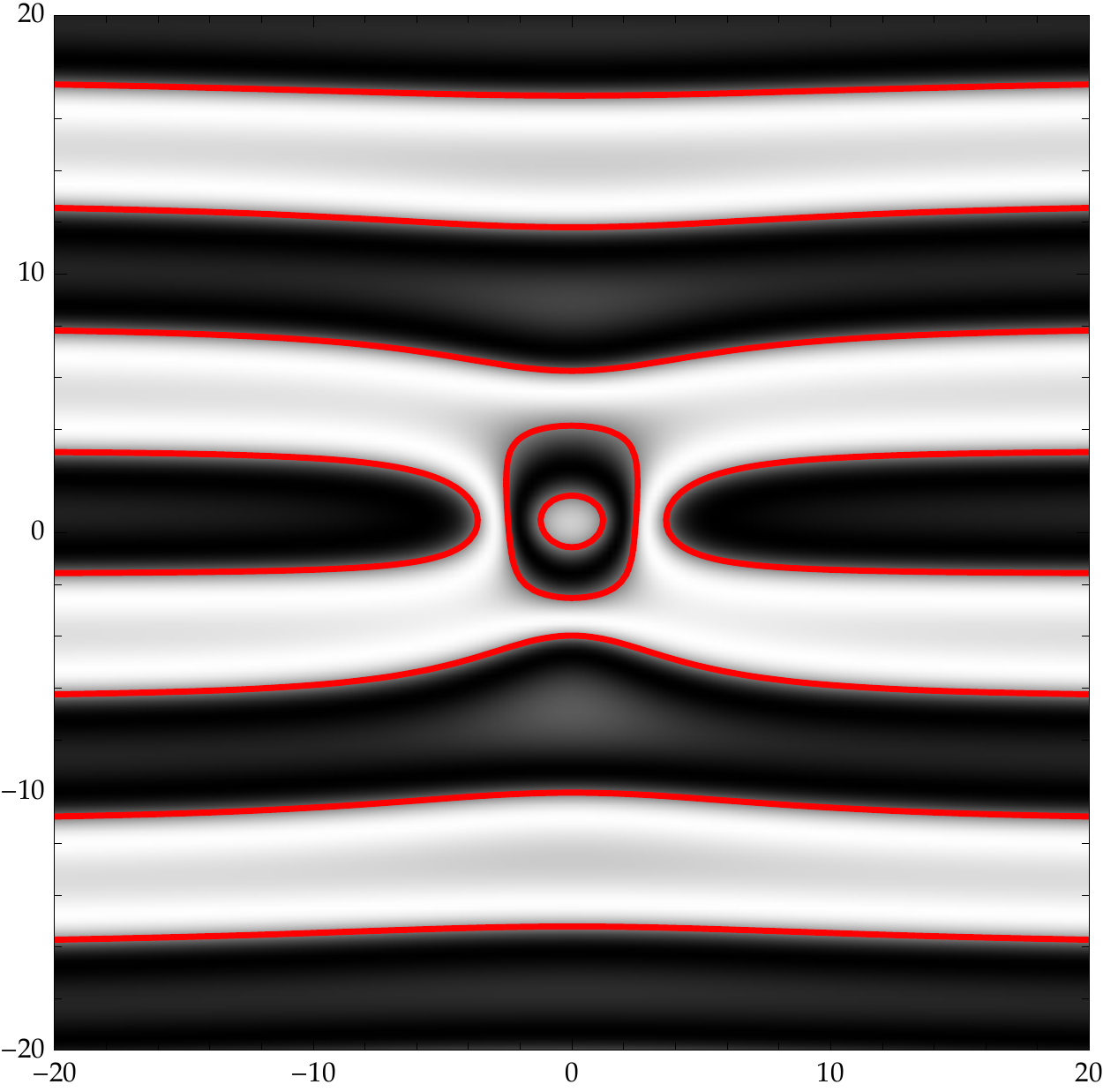}%
\includegraphics[height=.24\linewidth]{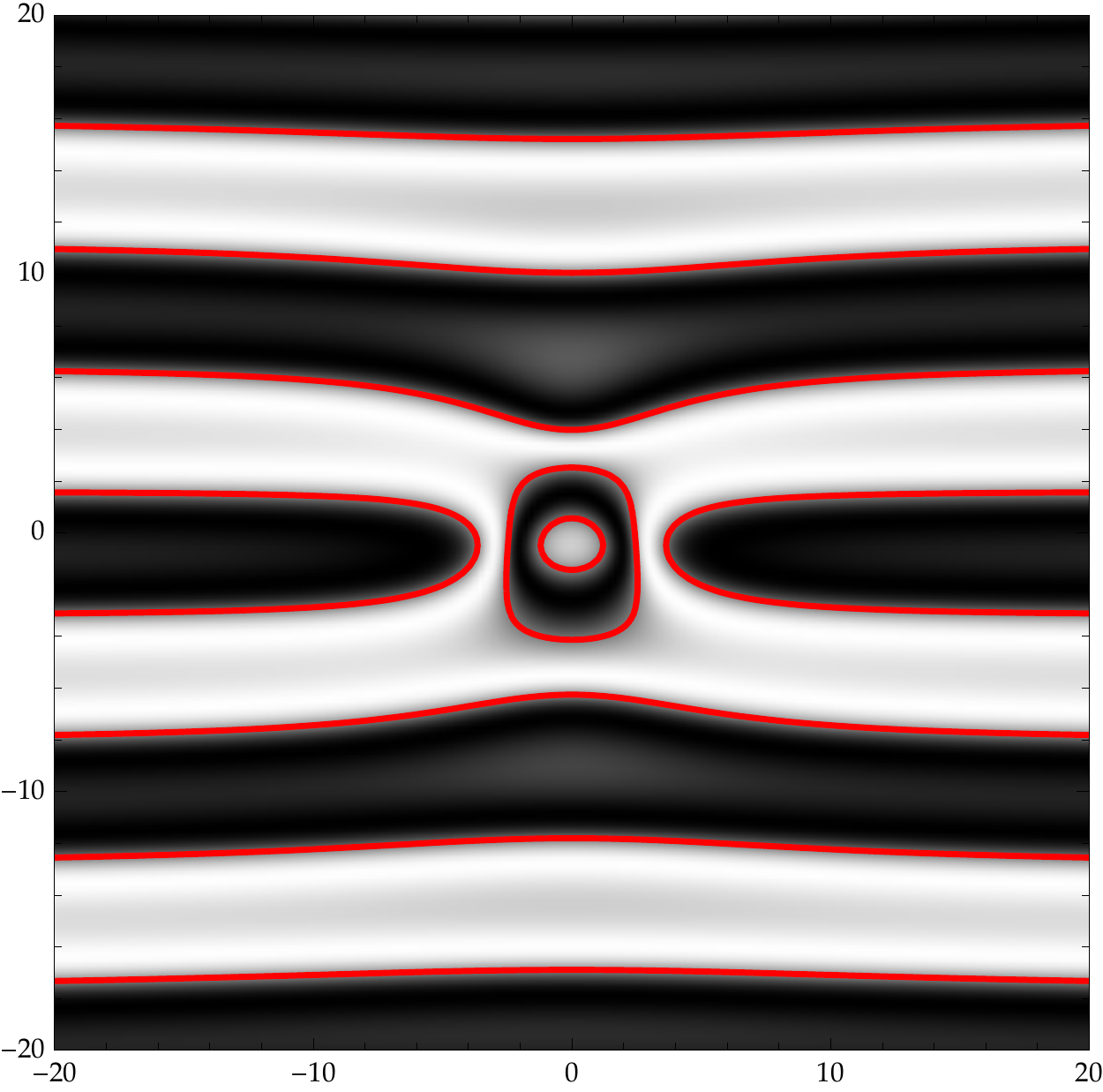}%
\includegraphics[height=.24\linewidth]{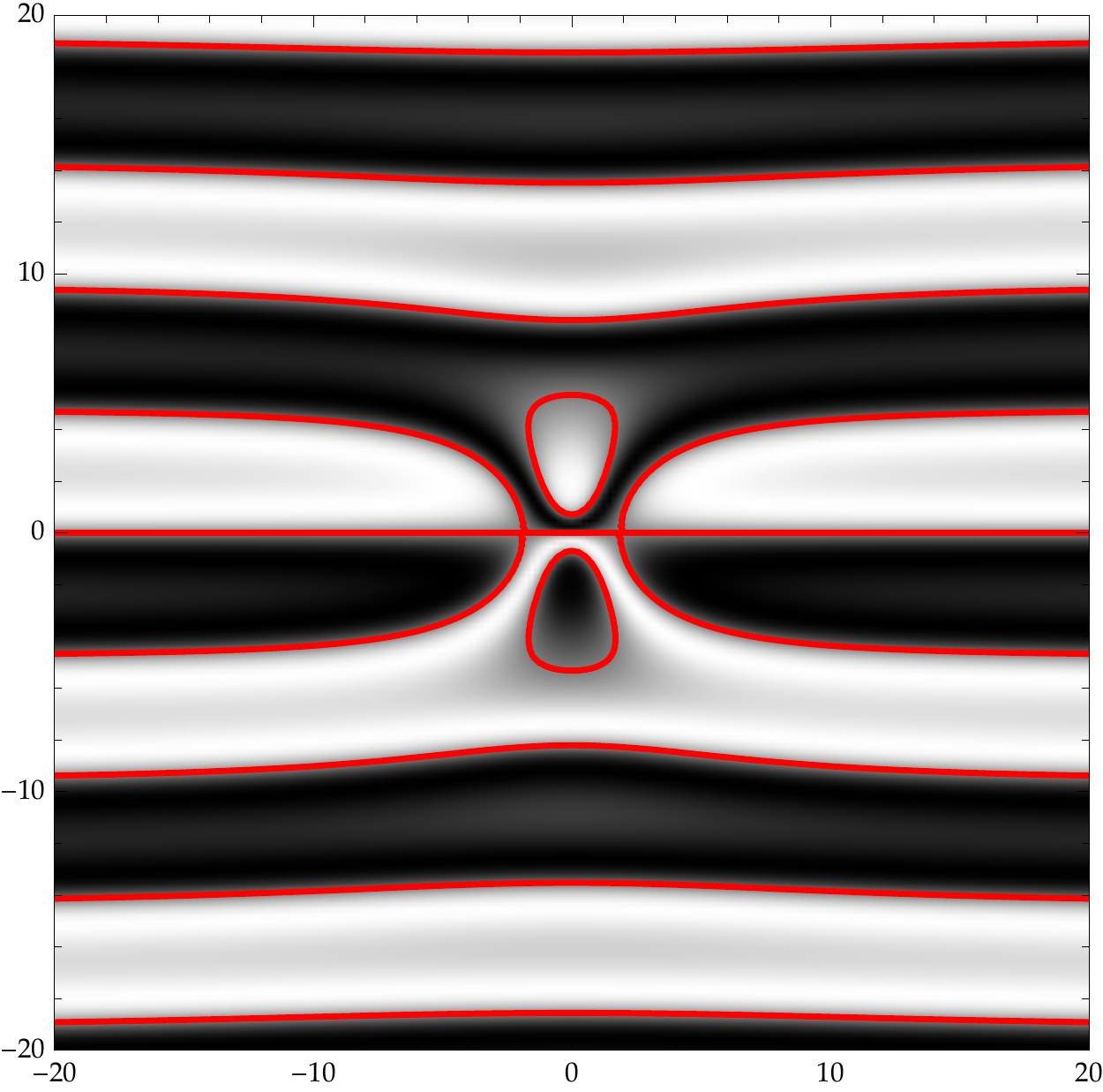}
\end{center}
\caption{As in Figure~\ref{fig:exact-solutions-first} but for $m=\sin^2(\tfrac{3}{8}\pi)$.}
\end{figure}
\begin{figure}[h!]
\begin{center}
\includegraphics[height=.24\linewidth]{fig/Legend-SMALL.pdf}%
\includegraphics[height=.24\linewidth]{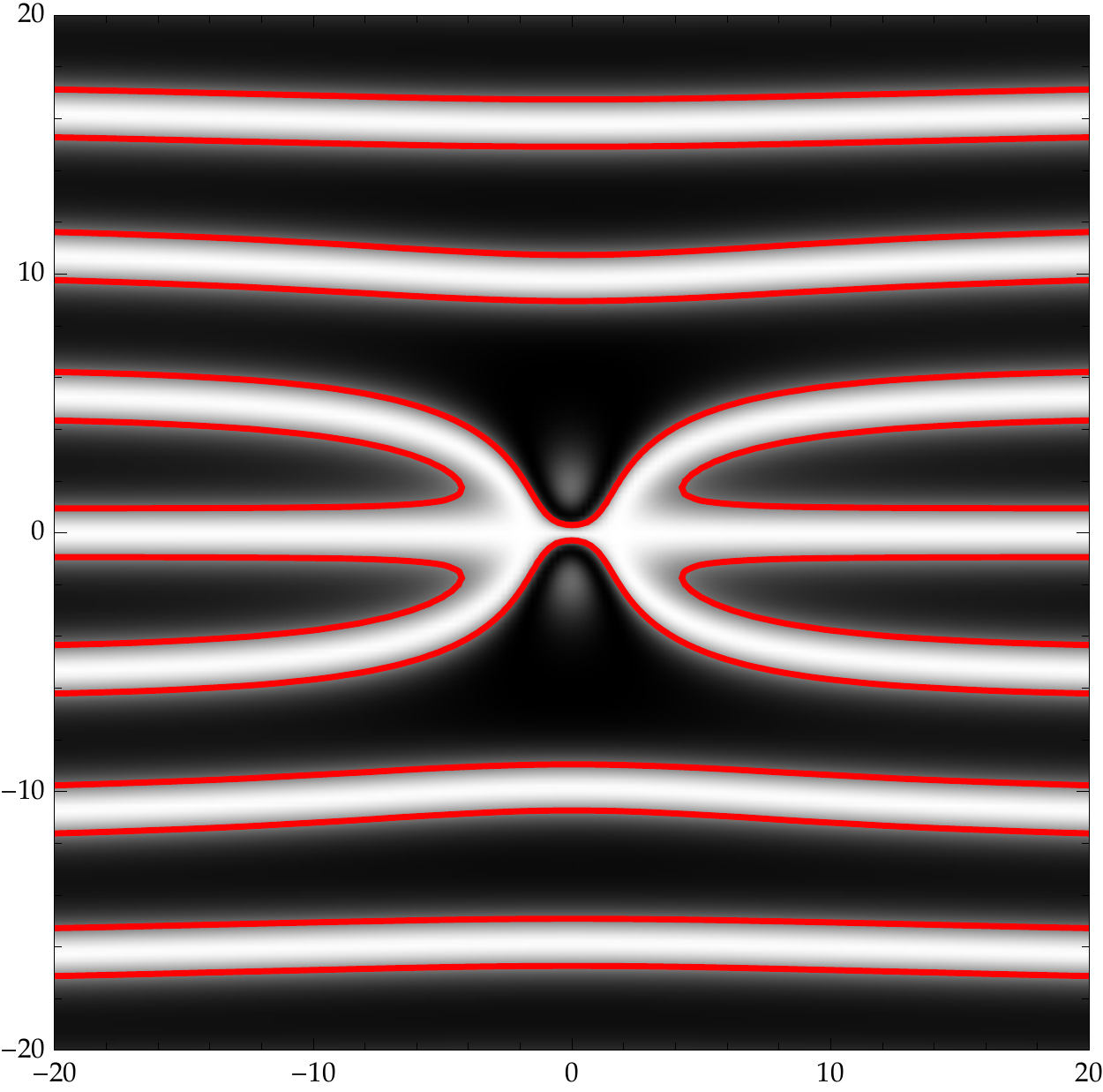}%
\includegraphics[height=.24\linewidth]{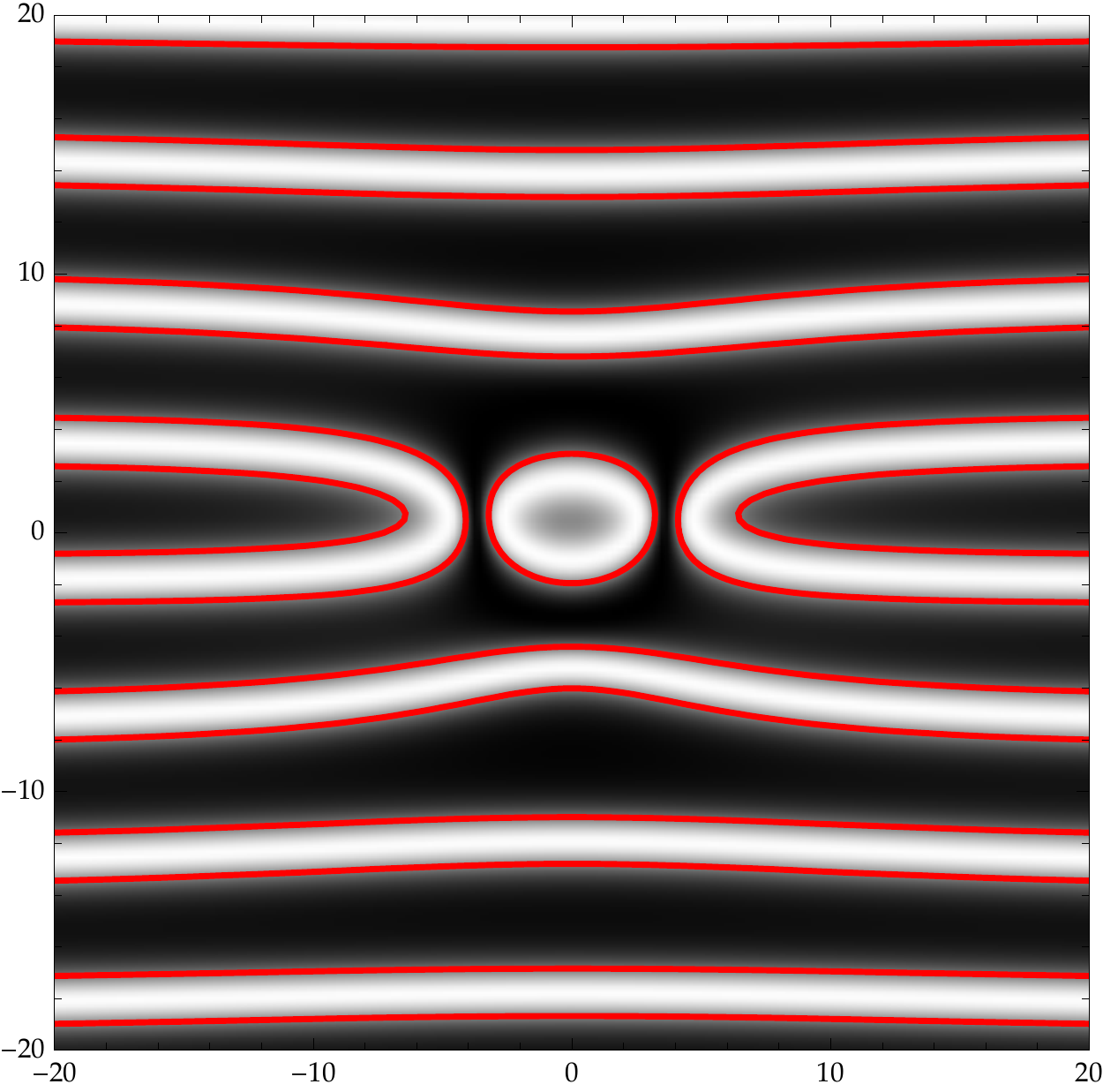}%
\includegraphics[height=.24\linewidth]{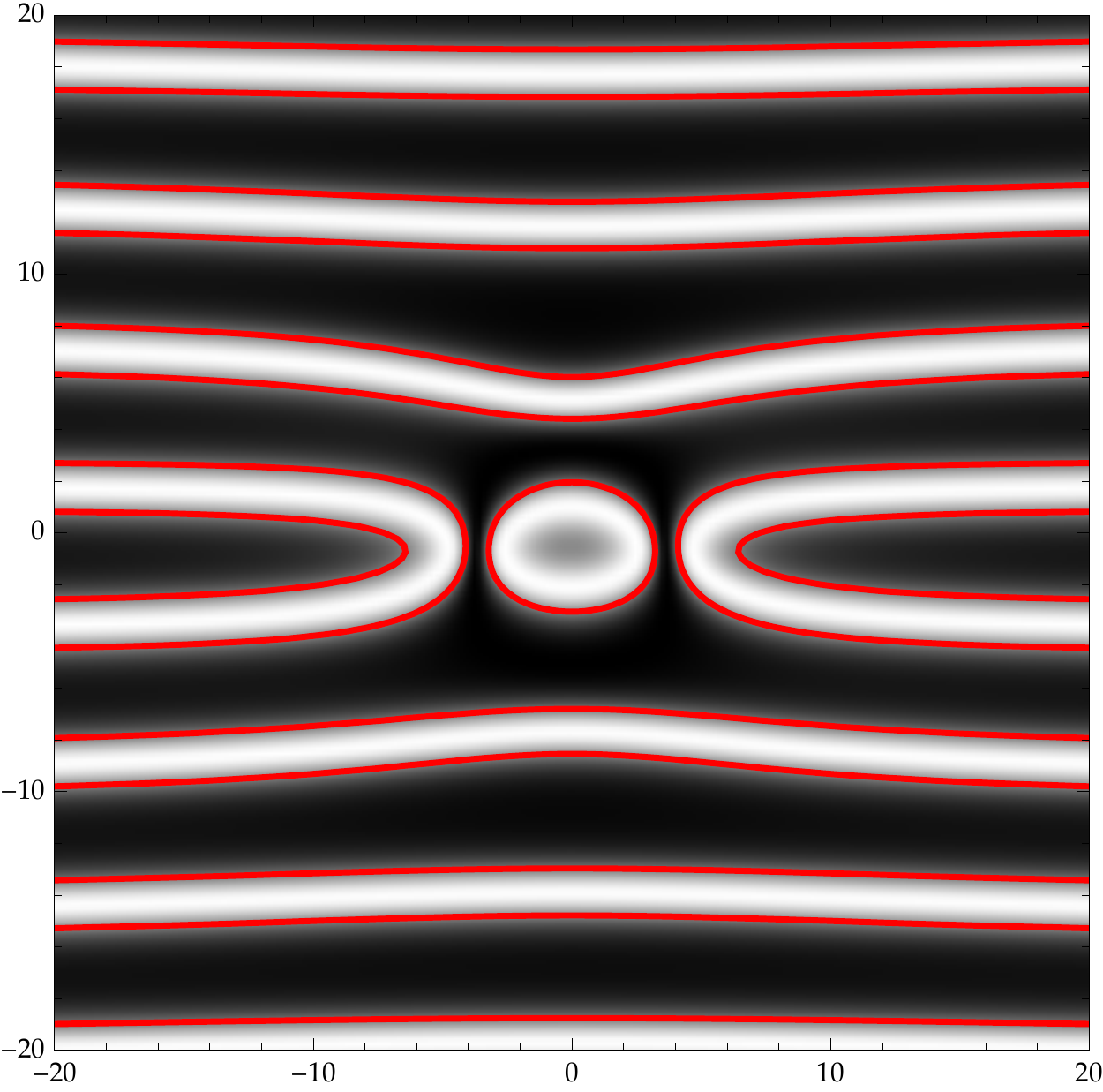}%
\includegraphics[height=.24\linewidth]{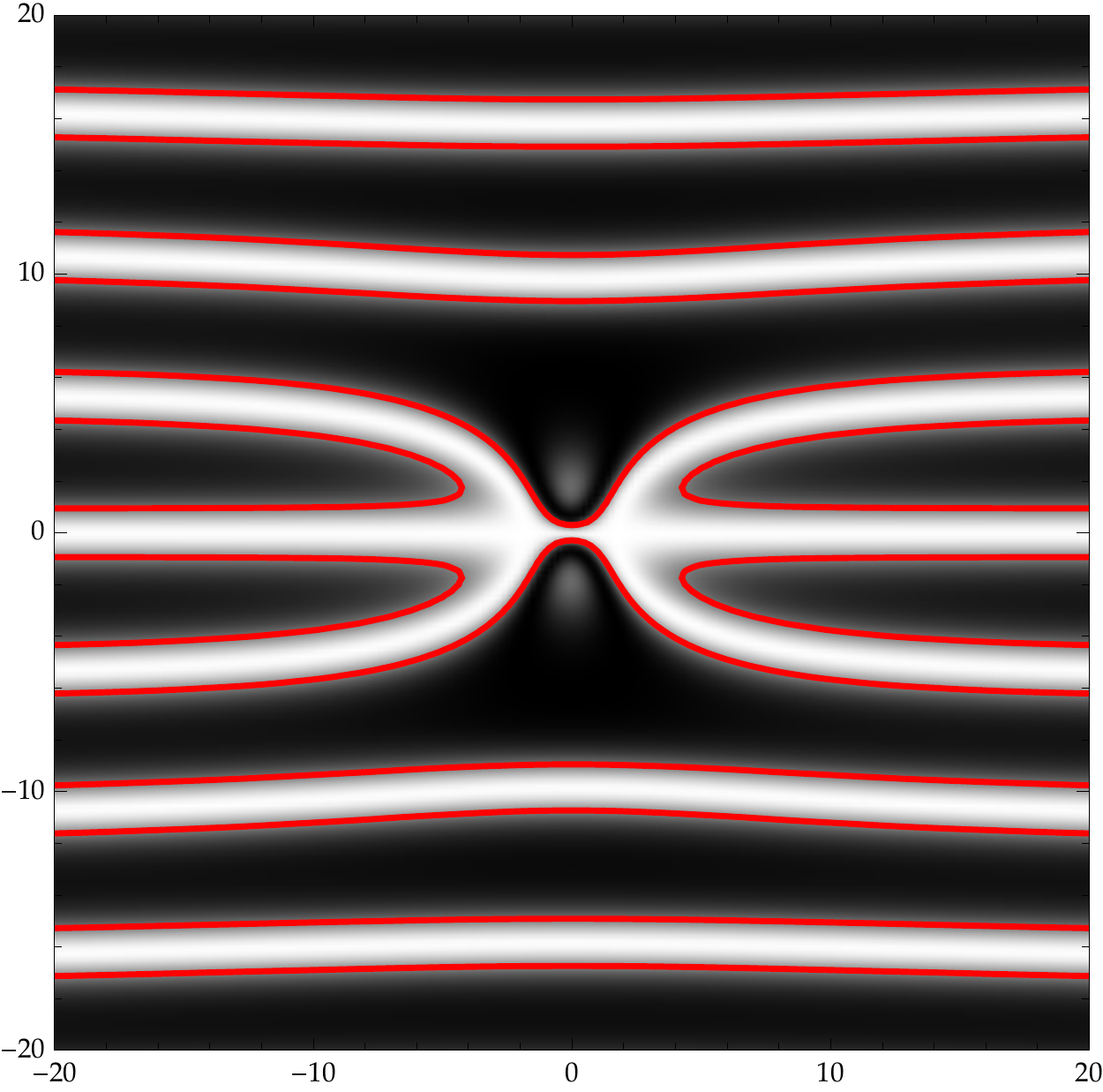}\\
\includegraphics[height=.24\linewidth]{fig/Legend-SMALL.pdf}%
\includegraphics[height=.24\linewidth]{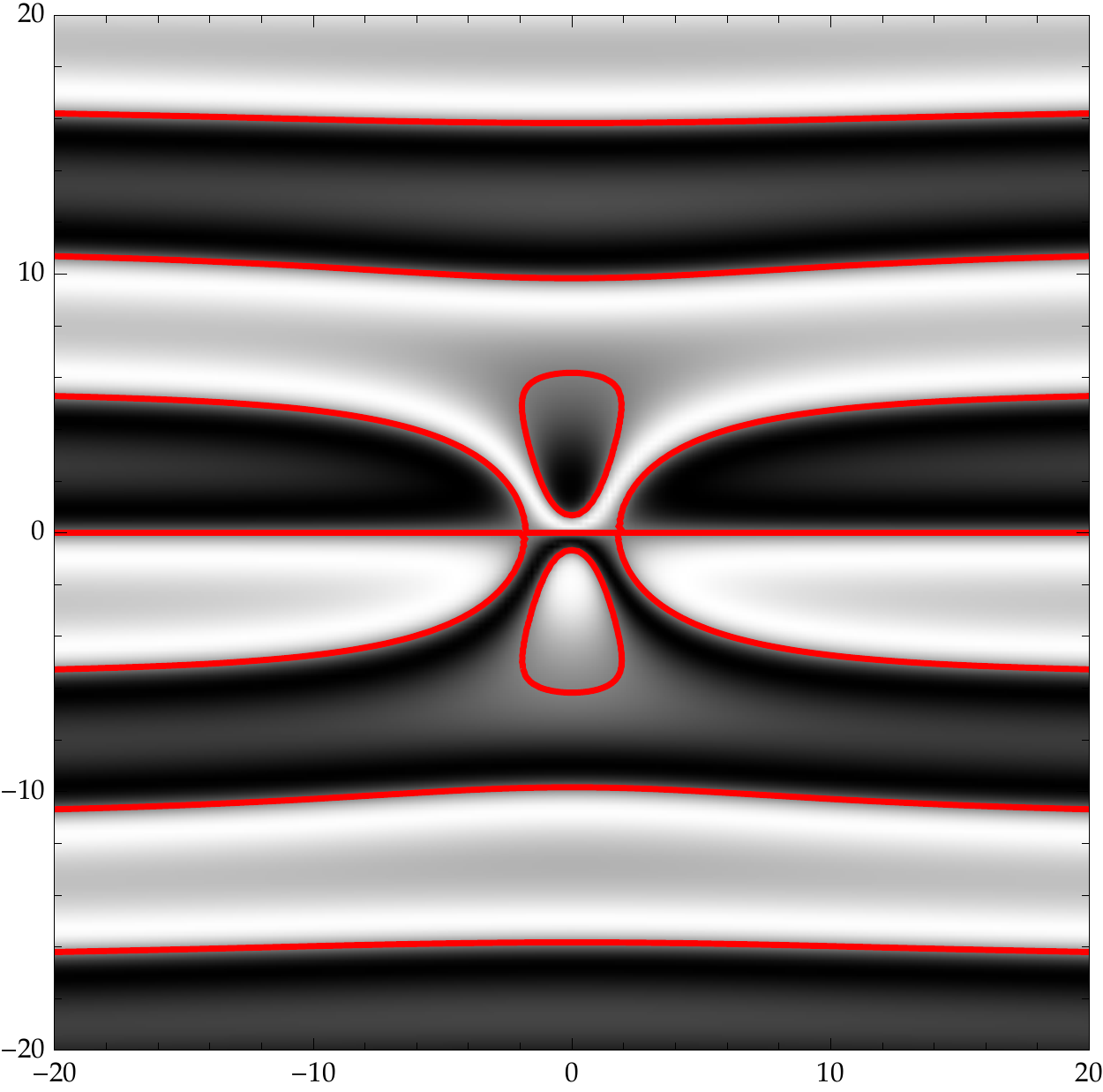}%
\includegraphics[height=.24\linewidth]{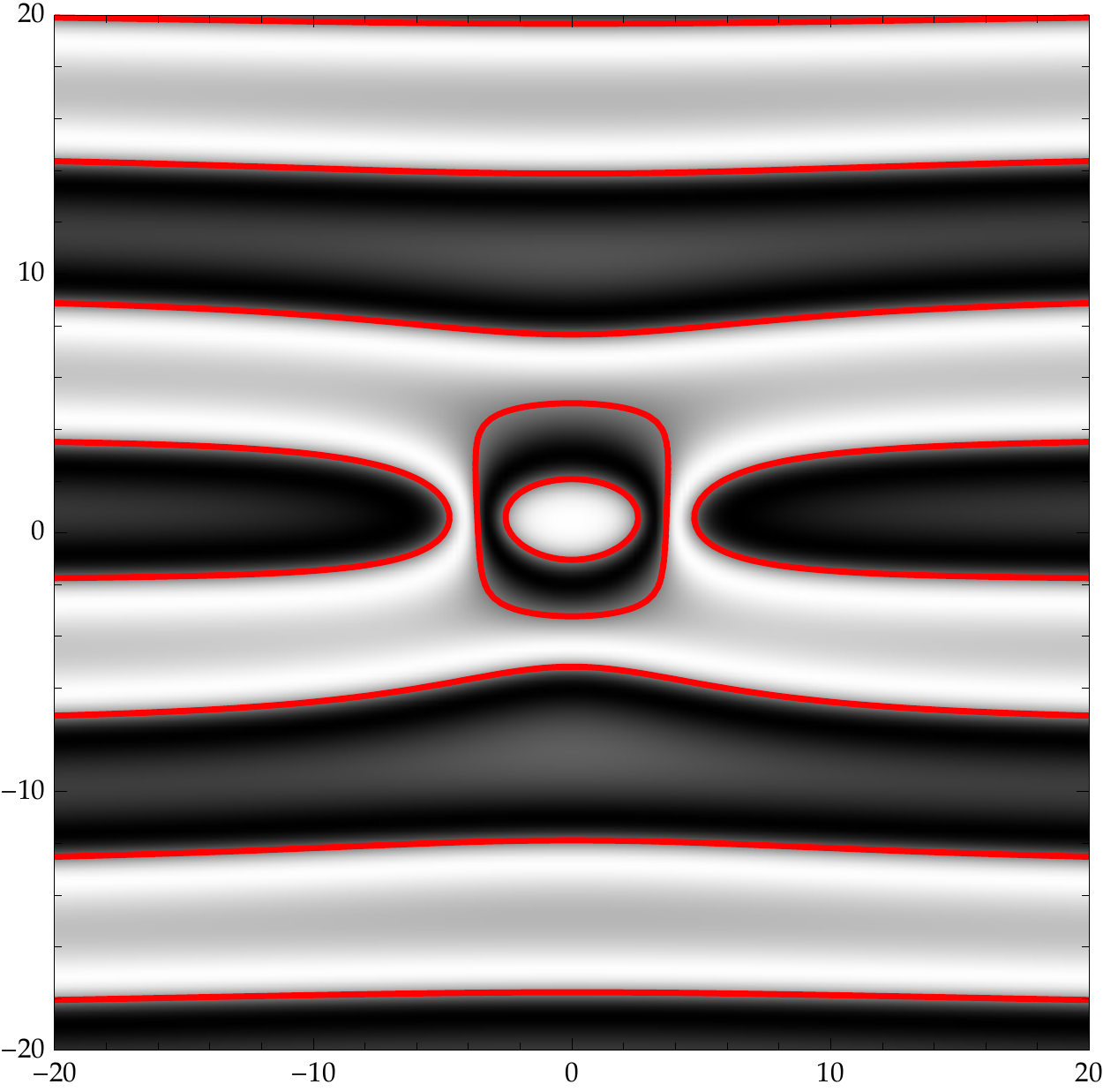}%
\includegraphics[height=.24\linewidth]{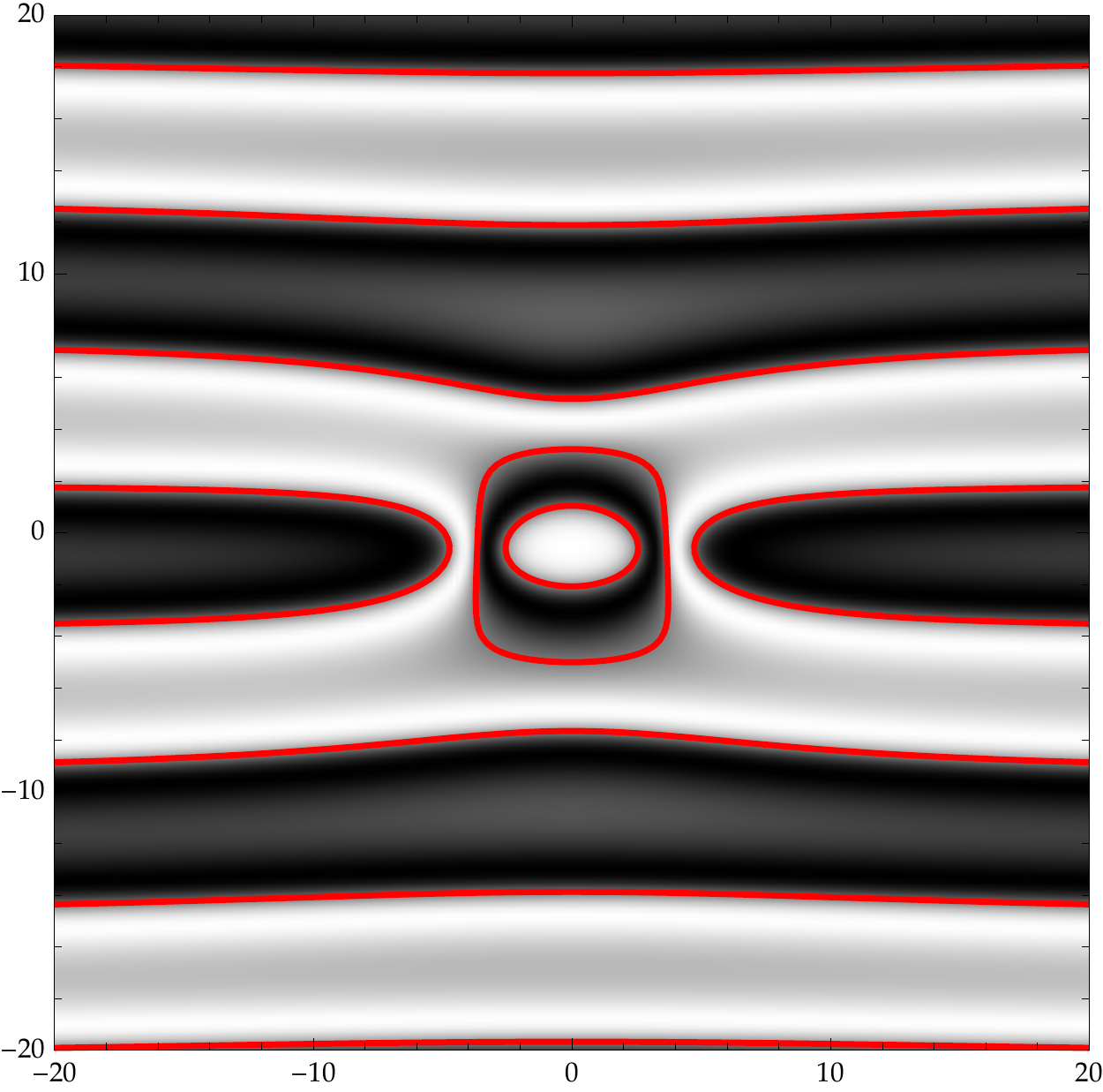}%
\includegraphics[height=.24\linewidth]{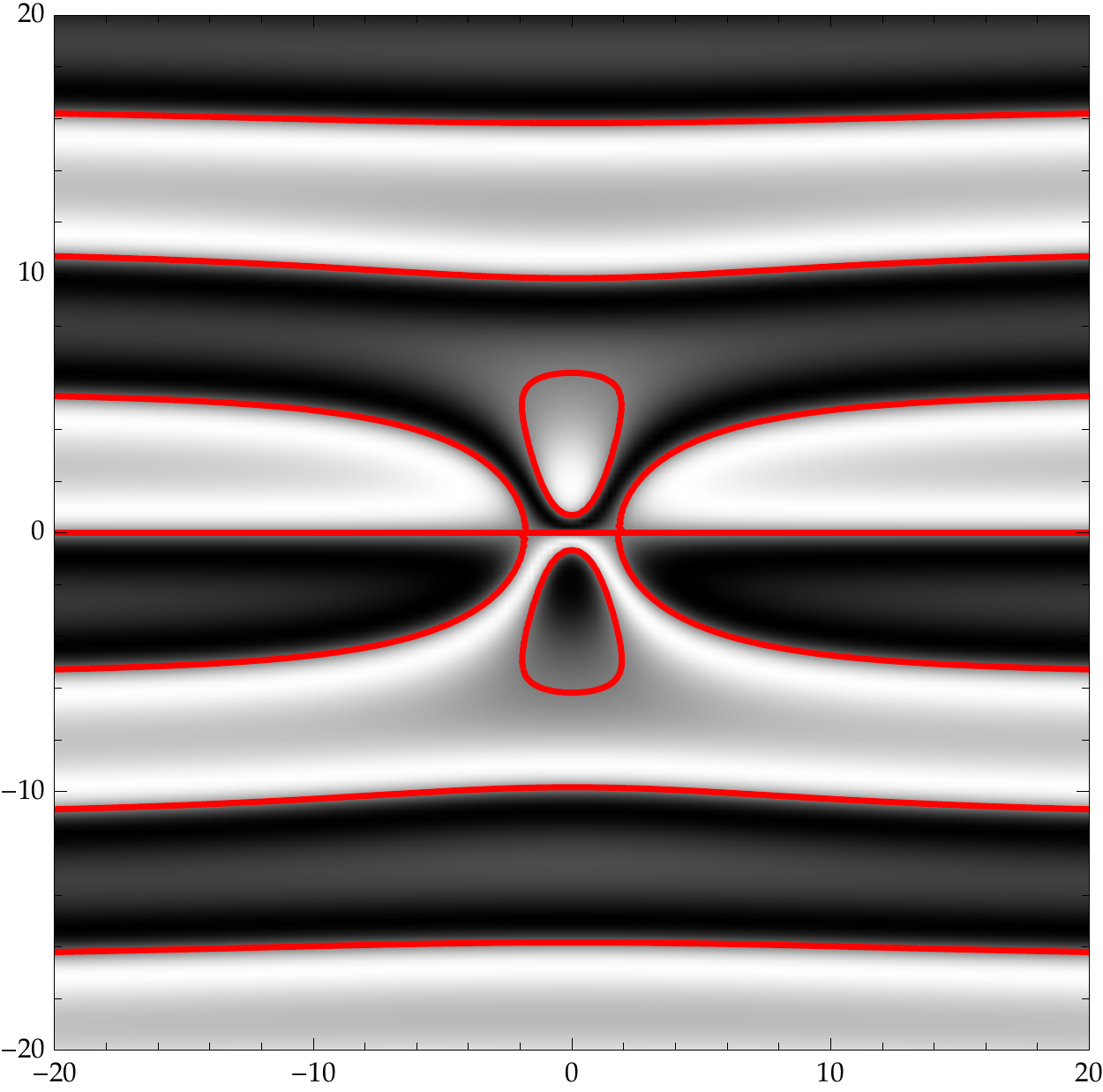}
\end{center}
\caption{As in Figure~\ref{fig:exact-solutions-first} but for $m=\sin^2(\tfrac{5}{12}\pi)$.}
\end{figure}
\begin{figure}[h!]
\begin{center}
\includegraphics[height=.24\linewidth]{fig/Legend-SMALL.pdf}%
\includegraphics[height=.24\linewidth]{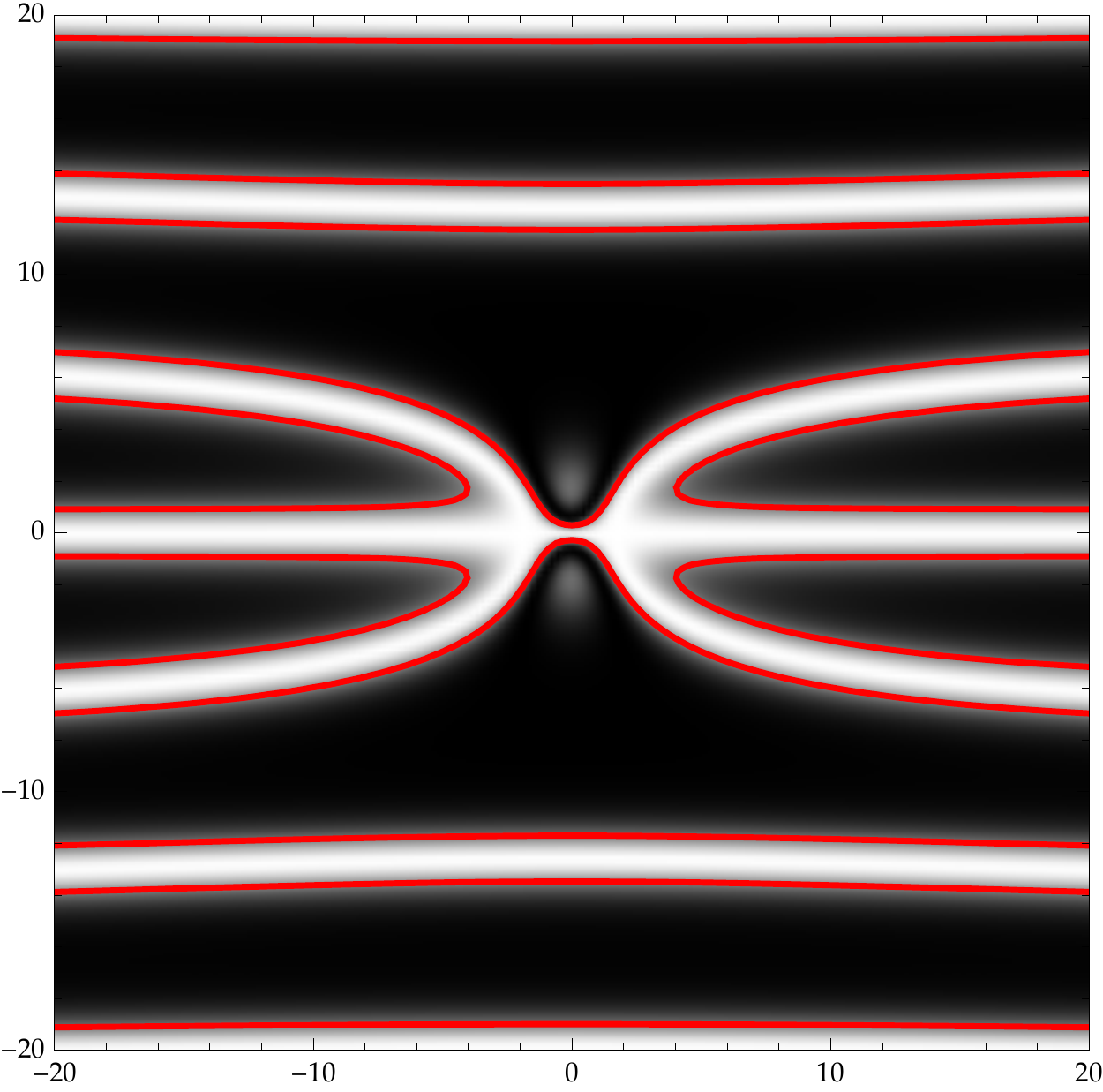}%
\includegraphics[height=.24\linewidth]{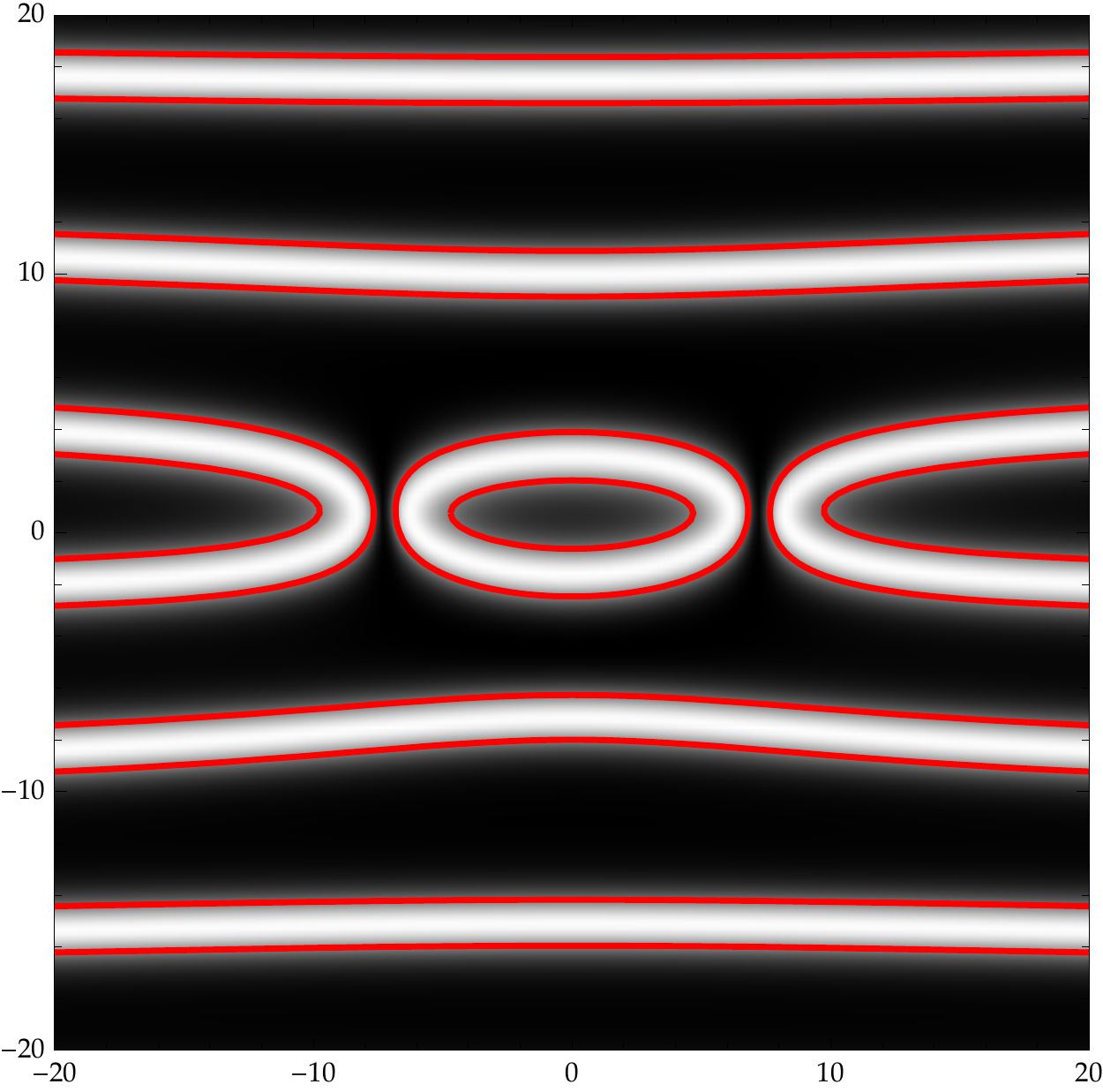}%
\includegraphics[height=.24\linewidth]{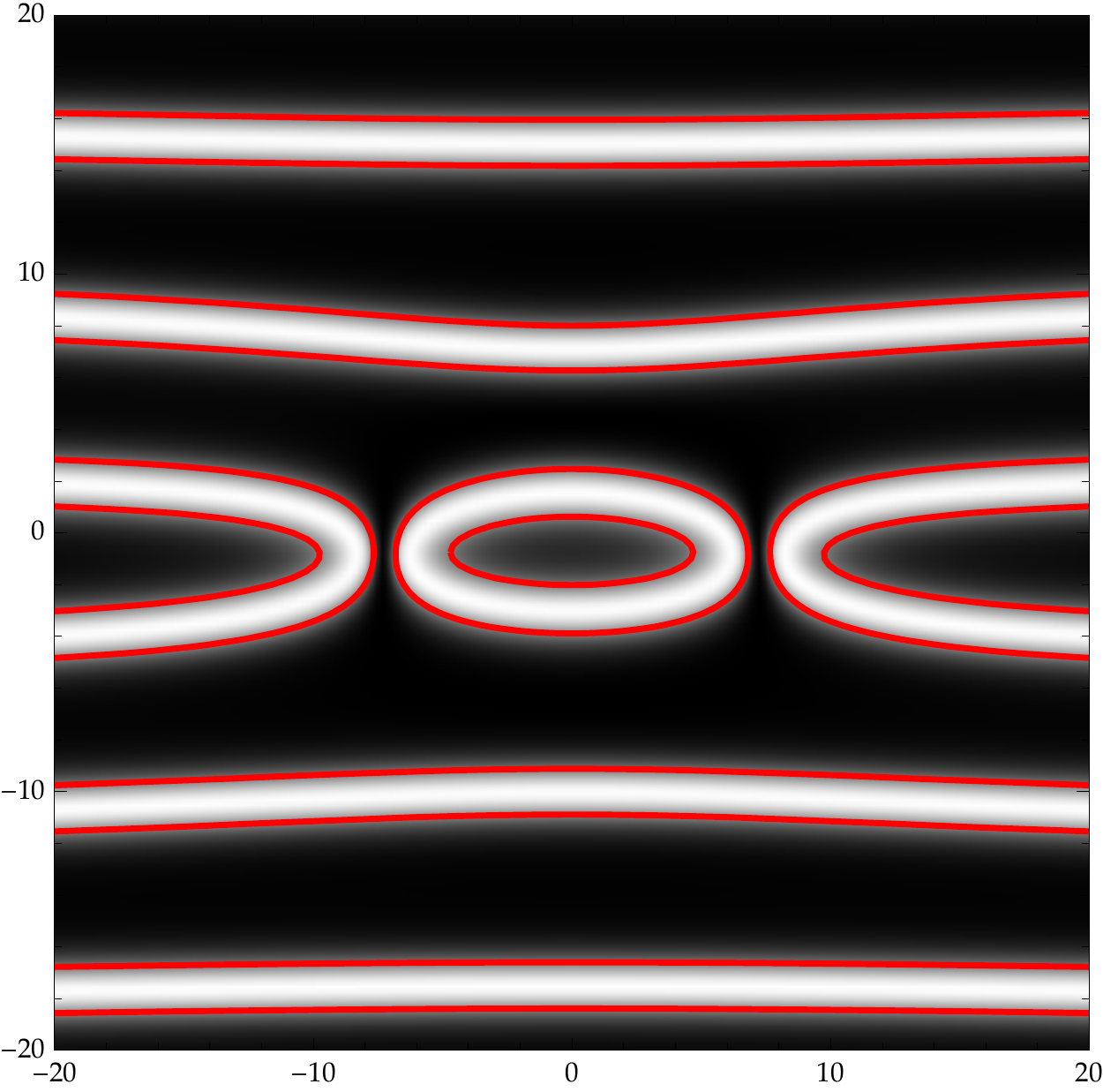}%
\includegraphics[height=.24\linewidth]{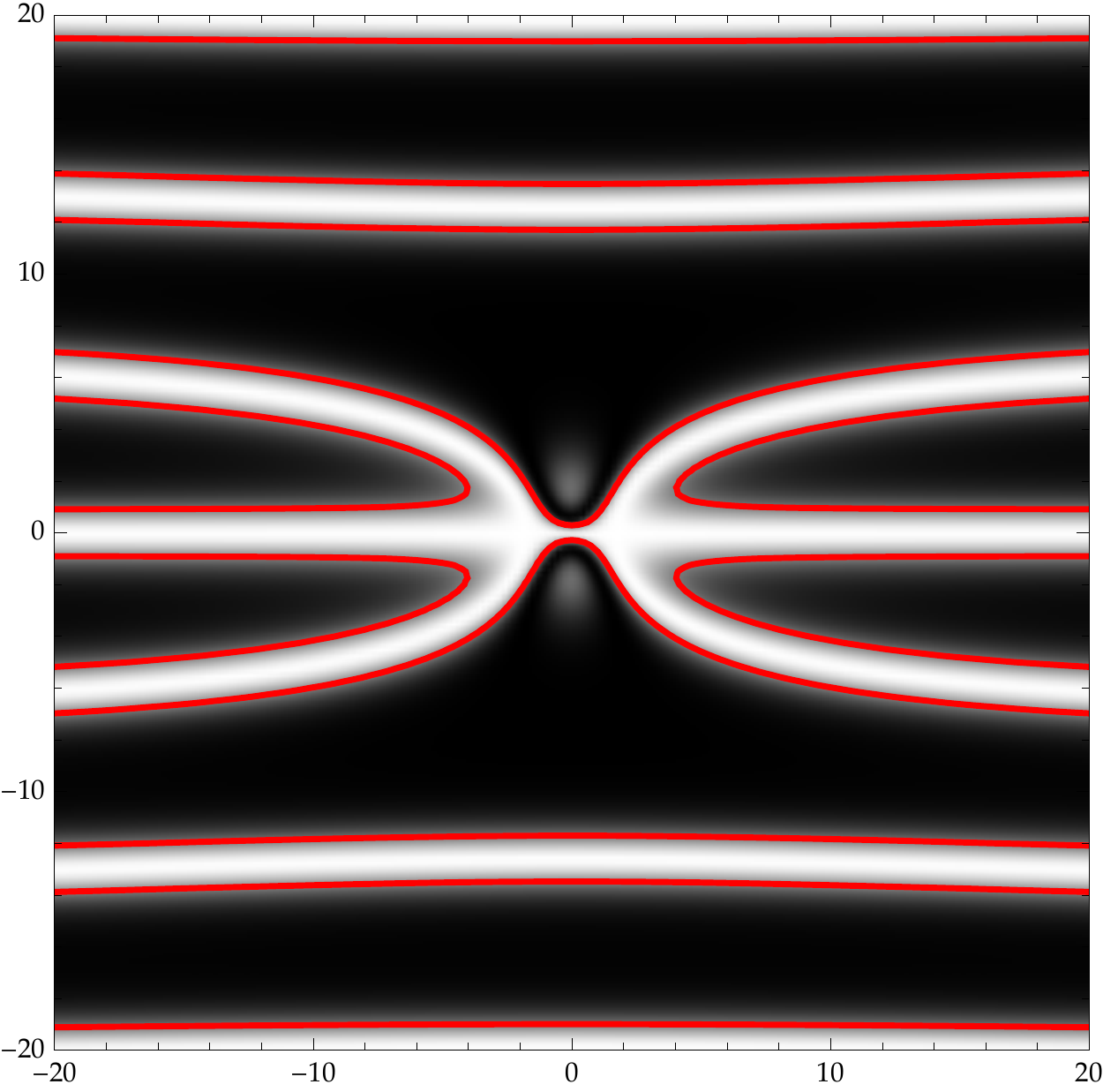}\\
\includegraphics[height=.24\linewidth]{fig/Legend-SMALL.pdf}%
\includegraphics[height=.24\linewidth]{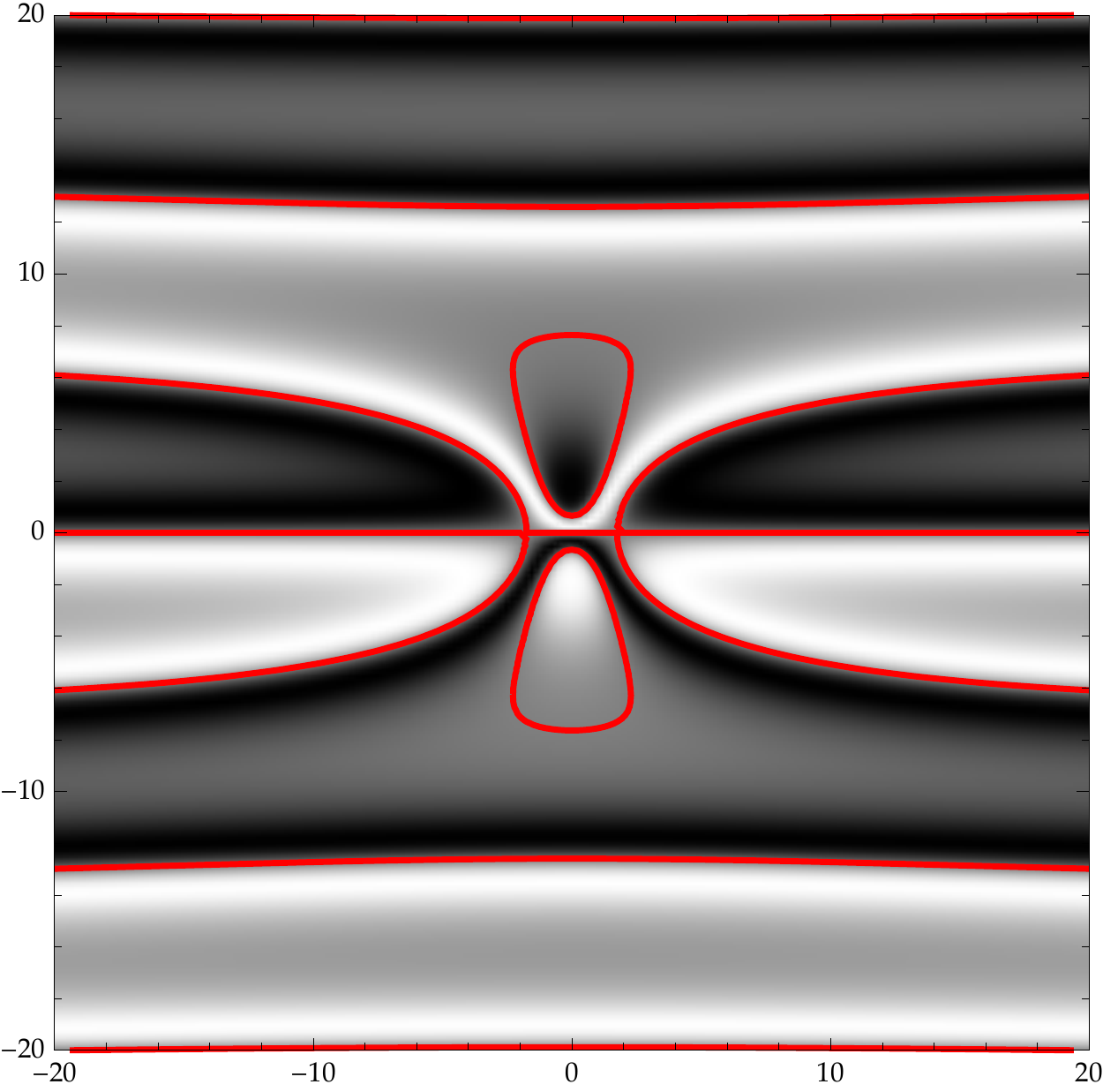}%
\includegraphics[height=.24\linewidth]{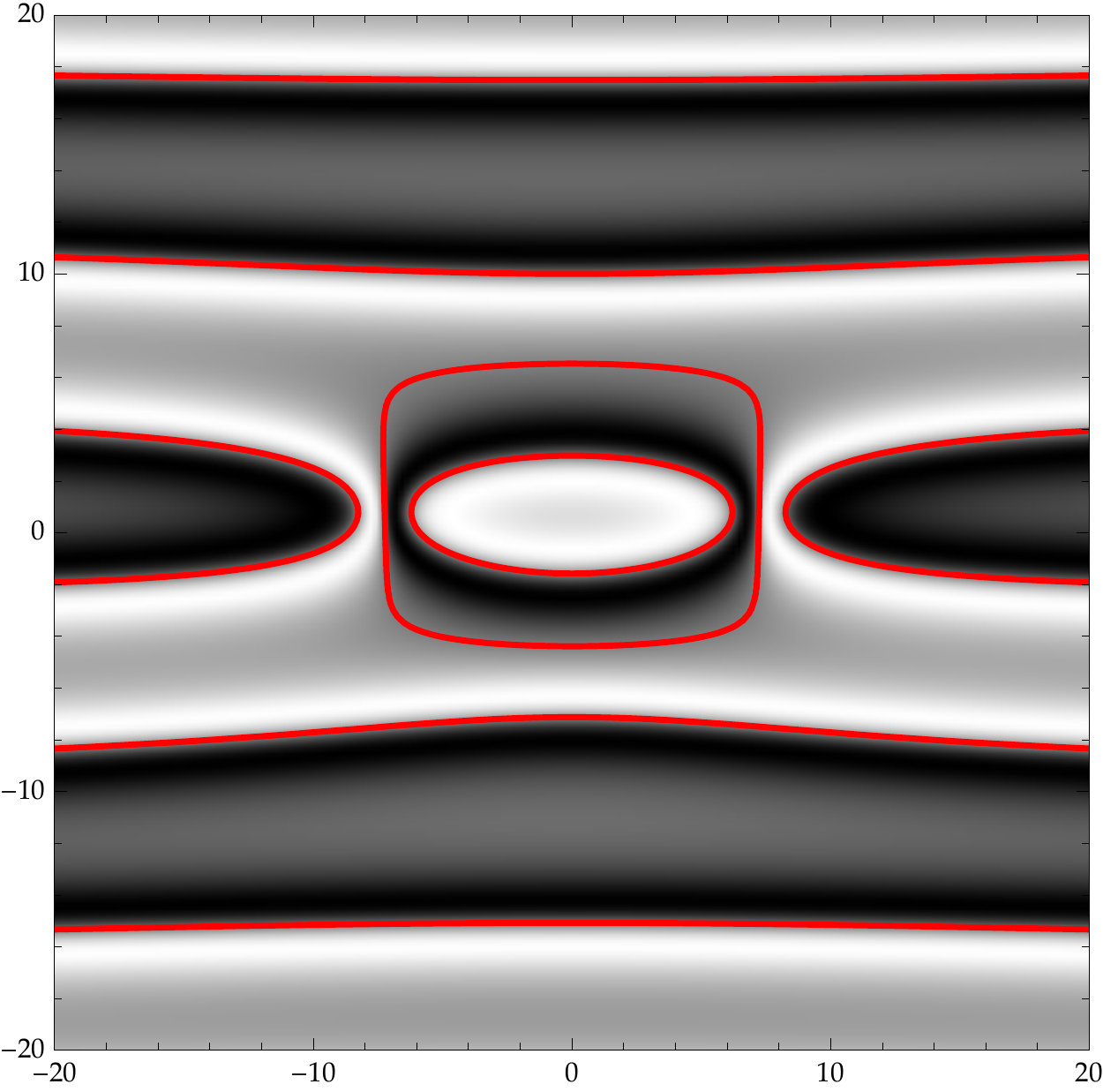}%
\includegraphics[height=.24\linewidth]{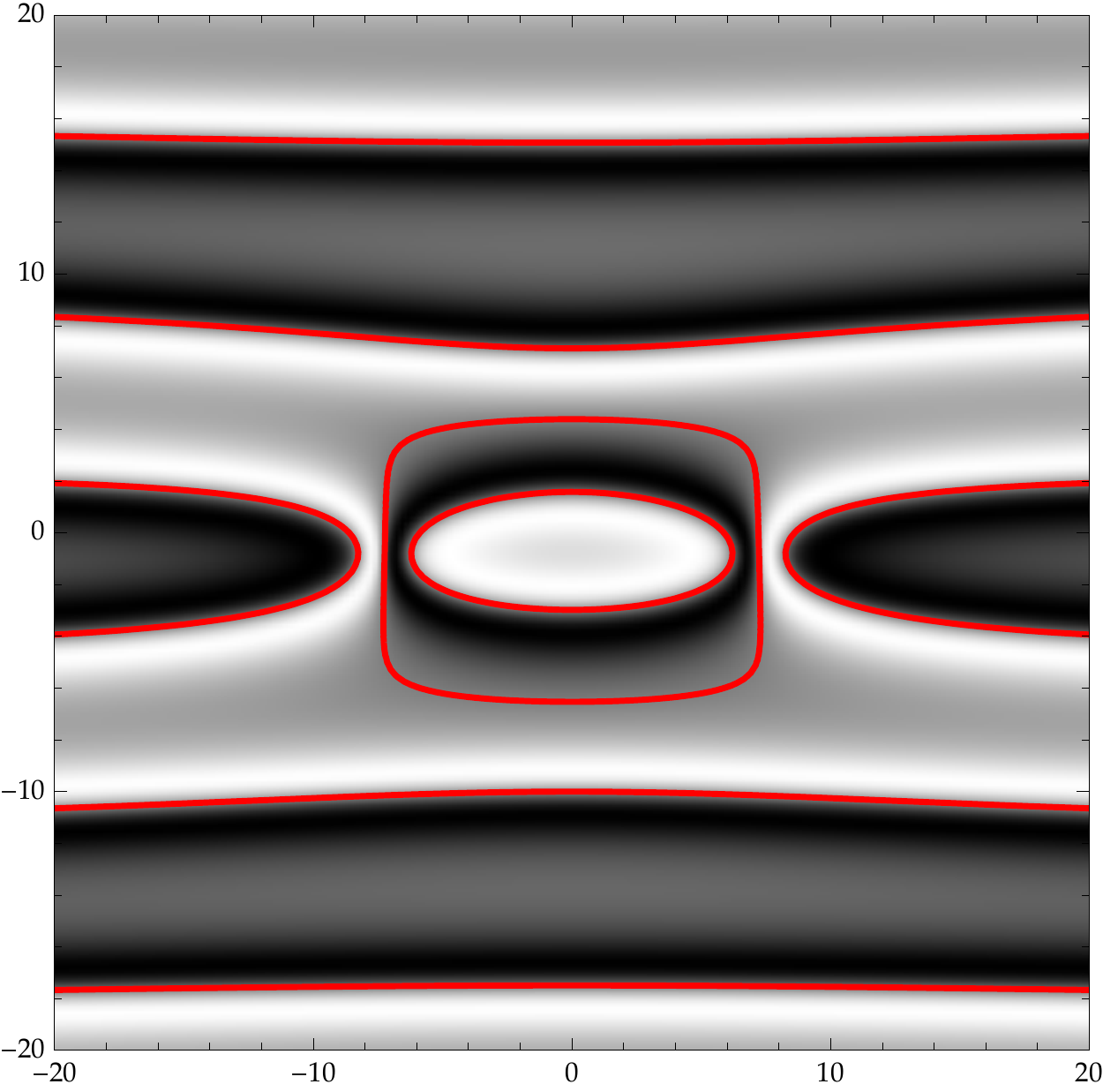}%
\includegraphics[height=.24\linewidth]{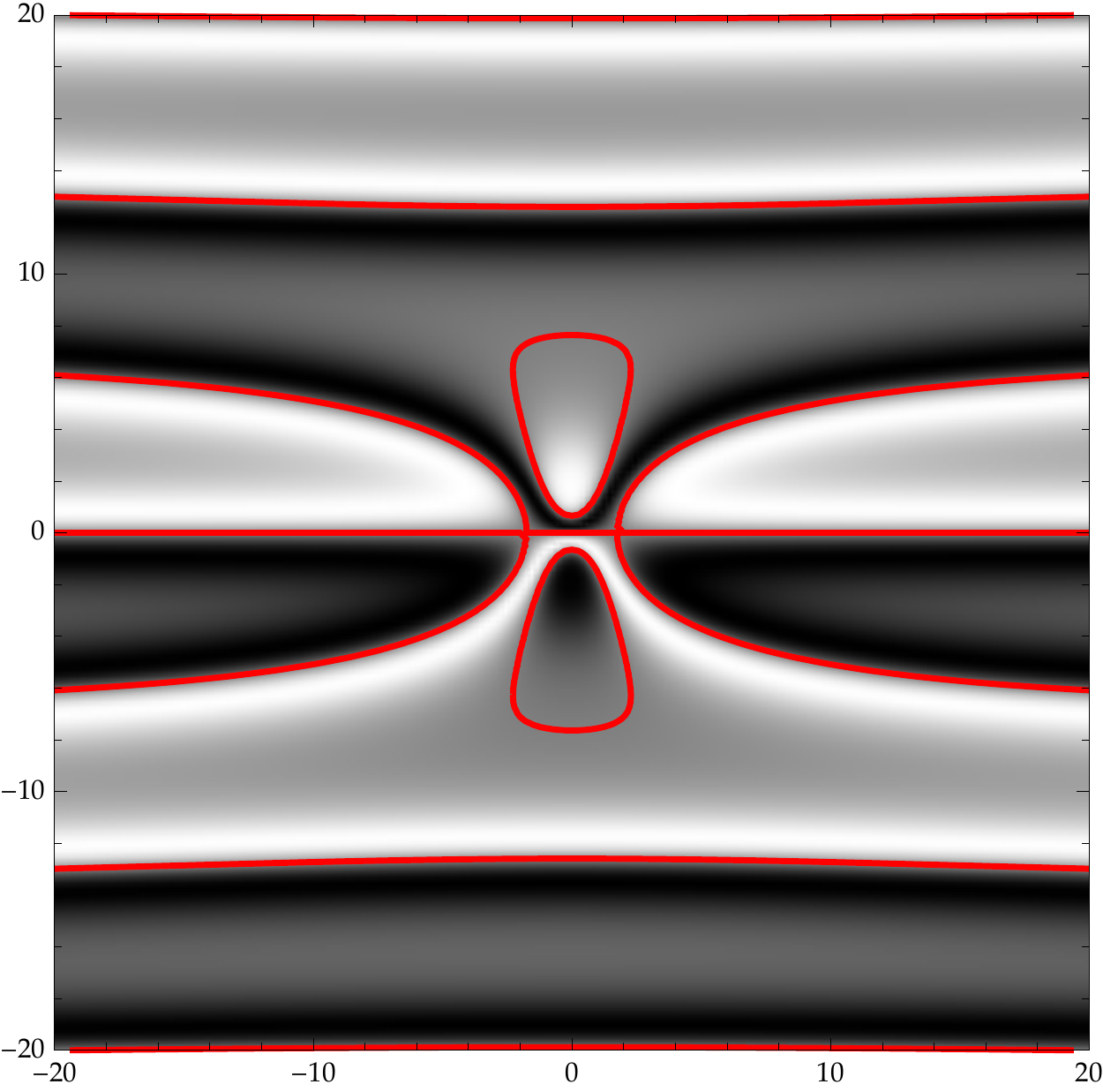}
\end{center}
\caption{As in Figure~\ref{fig:exact-solutions-first} but for $m=\sin^2(\tfrac{11}{24}\pi)$.}
\label{fig:exact-solutions-last}
\end{figure}

\clearpage

\begin{thebibliography}{10}

\bibitem{AblowitzKaupNewellSegur1973}
M.~J. Ablowitz, D.~J. Kaup, A.~C. Newell, and H.~Segur.
\newblock Nonlinear-evolution equations of physical significance.
\newblock {\em Physcal Review Letters}, 31:125--127, Jul 1973.

\bibitem{BaikKriecherbauerMcLaughlinMiller2007}
J.~Baik, T.~Kriecherbauer, K.~T.-R. McLaughlin, and P.~D. Miller.
\newblock {\em Discrete Orthogonal Polynomials: Asymptotics and Applications},
  volume 164 of {\em Annals of Mathematics Studies}.
\newblock Princeton University Press, Princeton, New Jersey, 2007.

\bibitem{BaroneEspositoMageeScott1971}
A.~Barone, F.~Esposito, C.~J. Magee, and A.~C. Scott.
\newblock Theory and applications of the sine-{G}ordon equation.
\newblock {\em La Rivista del Nuovo Cimento (1971-1977)}, 1(2):227--267, Apr
  1971.

\bibitem{BertolaTovbis2014}
M.~Bertola and A.~Tovbis.
\newblock Universality for the focusing nonlinear {Schr\"odinger} equation at
  the gradient catastrophe point: Rational breathers and poles of the
  tritronqu\'ee solution to {Painlev\'e~I}.
\newblock {\em Communications on Pure and Applied Mathematics}, 66(5):678--752,
  2014.

\bibitem{Bour1862}
E.~Bour.
\newblock Th{\'e}orie de la d{\'e}formation des surfaces.
\newblock {\em Journal de L'{\'E}cole Imperiale Polytechnique}, 22:1–148,
  1862.

\bibitem{BuckinghamMiller2008}
R.~Buckingham and P.~D. Miller.
\newblock Exact solutions of semiclassical non-characteristic {C}auchy problems
  for the sine-{G}ordon equation.
\newblock {\em Physica D: Nonlinear Phenomena}, 237:2296--2341, 9 2008.

\bibitem{BuckinghamJenkinsMiller2017}
R.~J. Buckingham, R.~M. Jenkins, and P.~D. Miller.
\newblock Semiclassical soliton ensembles for the three-wave resonant
  interaction equations.
\newblock {\em Communications in Mathematical Physics}, 354(3):1015--1100, Sep
  2017.

\bibitem{BuckinghamMiller2012}
R.~J. Buckingham and P.~D. Miller.
\newblock The {sine-Gordon} equation in the semiclassical limit: Critical
  behavior near a separatrix.
\newblock {\em Journal d'Analyse Math{\'e}matique}, 118(2):397--492, Nov 2012.

\bibitem{BuckinghamMiller2013}
R.~J. Buckingham and P.~D. Miller.
\newblock The {sine-Gordon} equation in the semiclassical limit: dynamics of
  fluxon condensates.
\newblock {\em Memoirs of the AMS}, 225(1059):(i)--(v) and 1--136, 2013.

\bibitem{BuckinghamMiller2015}
R.~J. Buckingham and P.~D. Miller.
\newblock Large-degree asymptotics of rational {Painlev\'e}-{II} functions:
  critical behaviour.
\newblock {\em Nonlinearity}, 28(6):1539, 2015.

\bibitem{ChenPelinovsky2018b}
J.~Chen and D.~E. Pelinovsky.
\newblock Rogue periodic waves in the focusing nonlinear {S}chr\"odinger
  equation.
\newblock {\em Proceedings of the Royal Society of London. Series A},
  474:20170814 (18 pages), 2018.

\bibitem{ChenPelinovsky2018a}
J.~Chen and D.~E. Pelinovsky.
\newblock Rogue periodic waves in the modified {K}d{V} equation.
\newblock {\em Nonlinearity}, 31:1955--1980, 2018.

\bibitem{ChenPelinovsky2019}
J.~Chen and D.~E. Pelinovsky.
\newblock Periodic travelling waves of the modified {K}d{V} equation and rogue
  waves on the periodic background.
\newblock {\em Journal of Nonlinear Science}, 29:2797--2843, 2018.

\bibitem{Coleman1975}
S.~Coleman.
\newblock Quantum sine-{G}ordon equation as the massive {T}hirring model.
\newblock {\em Physical Review D}, 11:2088--2097, Apr 1975.

\bibitem{CostinHuangTanveer2014}
O.~Costin, M.~Huang, and S.~Tanveer.
\newblock Proof of the {D}ubrovin conjecture and analysis of the tritronquée
  solutions of {$P_{I}$}.
\newblock {\em Duke Mathematical Journal}, 163(4):665--704, 03 2014.

\bibitem{Cuevas-MaraverKevrekidisWilliams2014}
J.~Cuevas-Maraver, P.~Kevrekidis, and F.~Williams, editors.
\newblock {\em The Sine-Gordon Model and Its Applications}, volume~10 of {\em
  Nonlinear Systems and Complexity}.
\newblock Springer International Publishing, Cham, Germany, 2014.

\bibitem{DeiftZhou1995}
P.~A. Deift and X.~Zhou.
\newblock Asymptotics for the {P}ainlev\'e {II} equation.
\newblock {\em Communications on Pure and Applied Mathematics}, 48(3):277--337,
  1995.

\bibitem{NIST:DLMF}
{\it NIST Digital Library of Mathematical Functions}.
\newblock http://dlmf.nist.gov/, Release 1.0.19 of 2018-06-22.
\newblock F.~W.~J. Olver, A.~B. {Olde Daalhuis}, D.~W. Lozier, B.~I. Schneider,
  R.~F. Boisvert, C.~W. Clark, B.~R. Miller and B.~V. Saunders, eds.

\bibitem{DubrovinGravaKlein2009}
B.~Dubrovin, T.~Grava, and C.~Klein.
\newblock On universality of critical behavior in the focusing nonlinear
  {Schr\"odinger} equation, elliptic umbilic catastrophe and the
  tritronqu{\'e}e solution to the {Painlev\'e-I} equation.
\newblock {\em Journal of Nonlinear Science}, 19(1):57--94, Feb 2009.

\bibitem{DubrovinGravaKleinMoro2015}
B.~Dubrovin, T.~Grava, C.~Klein, and A.~Moro.
\newblock On critical behaviour in systems of {Hamiltonian} partial
  differential equations.
\newblock {\em Journal of Nonlinear Science}, 25(3):631--707, 2015.

\bibitem{ErcolaniJinLevermoreMacEvoy2003}
N.~M. Ercolani, S.~Jin, C.~D. Levermore, and W.~D. MacEvoy, Jr.
\newblock The zero-dispersion limit for the odd flows in the focusing
  {Z}akharov-{S}habat hierarchy.
\newblock {\em International Mathematics Research Notices}, 2003:2529--2564,
  2003.

\bibitem{FaddeevTakhtajanZakharov1974}
L.~D. Faddeev, L.~A. Takhtajan, and V.~E. Zakharov.
\newblock {Complete description of solutions of the sine-Gordon equation}.
\newblock {\em Doklady Akademii Nauk Series Fiziki}, 219:1334--1337, 1974.
\newblock [Sov. Phys. Dokl.19,824(1975)].

\bibitem{FornbergWeideman2011}
B.~Fornberg and J.~A.~C.~Weideman.
\newblock A numerical methodology for the {Painlev\'e} equations.
\newblock {\em Journal of Computational Physics}, 230(15):5957 -- 5973, 2011.

\bibitem{FrenkelKontorova1939}
J.~Frenkel and T.~Kontorova.
\newblock On the theory of plastic deformation and twinning.
\newblock {\em Journal of Physics USSR}, l:137, 1939.

\bibitem{JinLevermoreMcLaughlin1999}
S.~Jin, C.~D. Levermore, and D.~W. McLaughlin.
\newblock The semiclassical limit of the defocusing {NLS} hierarchy.
\newblock {\em Communications on Pure and Applied Mathematics}, 52:613--654,
  1999.

\bibitem{OrangeBook}
S.~Kamvissis, K.~McLaughlin, and P.~Miller.
\newblock {\em Semiclassical Soliton Ensembles for the Focusing Nonlinear
  {Schr\"odinger} Equation}, volume 154 of {\em Annals of Mathematics Studies}.
\newblock Princeton University Press, Princeton, New Jersey, 2003.

\bibitem{Kapaev2004}
A.~A. Kapaev.
\newblock Quasi-linear {Stokes} phenomenon for the {Painlev\'e} first equation.
\newblock {\em Journal of Physics A: Mathematical and General}, 37(46):11149,
  2004.

\bibitem{Kaup1975}
D.~J. Kaup.
\newblock Method for solving the {sine-Gordon} equation in laboratory
  coordinates.
\newblock {\em Studies in Applied Mathematics}, 54(2):165--179, June 1975.

\bibitem{KlausShaw2002}
M.~Klaus and J.~K. Shaw.
\newblock Purely imaginary eigenvalues of {Zakharov-Shabat} systems.
\newblock {\em Physical Review E}, 65:36607--36611, February 2002.

\bibitem{LaxLevermore1979}
P.~D. Lax and C.~D. Levermore.
\newblock The small dispersion limit for the {Korteweg-de Vries} equation. {I.,
  II., III.}
\newblock {\em Communications on Pure and Applied Mathematics}, 36:253--290,
  571--593, 809--829, 1983.

\bibitem{Lu2018}
B.-Y. Lu.
\newblock {\em The Semi-Classical Sine-{G}ordon Equation, Universality at the
  Gradient Catastrophe and the {Painlev\'e-I} Equation}.
\newblock PhD thesis, University of Michigan, 2018.

\bibitem{LyngMiller2007}
G.~D. Lyng and P.~D. Miller.
\newblock The $N$-soliton of the focusing nonlinear {Schr\"odinger} equation
  for $N$ large.
\newblock {\em Communications on Pure and Applied Mathematics},
  60(7):951--1026, 2007.

\bibitem{MillerXu2011}
P.~D. Miller and Z.~Xu.
\newblock On the zero-dispersion limit of the {Benjamin-Ono Cauchy} problem for
  positive initial data.
\newblock {\em Communications on Pure and Applied Mathematics}, 64:205--270,
  2011.

\bibitem{SatsumaYajima1974}
J.~Satsuma and N.~Yajima.
\newblock Initial value problems of one-dimensional self-modulation of
  nonlinear waves in dispersive media.
\newblock {\em Progress of Theoretical Physics Supplement}, 55:284--306, 1974.

\bibitem{ScottChuReible1976}
A.~C. Scott, F.~Y.~F. Chu, and S.~A. Reible.
\newblock Magnetic-flux propagation on a {J}osephson transmission line.
\newblock {\em Journal of Applied Physics}, 47(7):3272--3286, 1976.

\bibitem{Whitham1965a}
G.~B. Whitham.
\newblock A general approach to linear and non-linear dispersive waves using a
  {L}agrangian.
\newblock {\em Journal of Fluid Mechanics}, 22:273--283, 1965.

\bibitem{Whitham1965b}
G.~B. Whitham.
\newblock Non-linear dispersive waves.
\newblock {\em Proceedings of the Royal Society of London. Series A},
  283:238--261, 1965.

\bibitem{Yakushevich2004}
L.~V. Yakushevich.
\newblock {\em Nonlinear Physics of DNA, Second Edition}.
\newblock Wiley-VCH, Weinheim, 2004.

\bibitem{ZakharovShabat1971}
V.~E. Zakharov and A.~B. Shabat.
\newblock Exact theory of two-dimensional self-focusing and one-dimensional
  self-modulation of waves in nonlinear media.
\newblock {\em Soviet Physics---Journal of Experimental and Theoretical
  Physics}, 34(1):62--69, 1972.
\newblock [translated from Russian, \textit{Zh.\@ Eksper.\@ Teoret.\@ Fiz.\@}
  \textbf{61}, 118--134, 1971.].

\bibitem{Zhou1989}
X.~Zhou.
\newblock The {Riemann--Hilbert} problem and inverse scattering.
\newblock {\em SIAM Journal on Mathematical Analysis}, 20(4):966--986, 1989.

\end{thebibliography}

\end{document}